\documentclass[twoside,online]{report}

\usepackage{suthesis-2e}
\dept{Mathematics}
\usepackage{graphicx}

\usepackage{import}

\usepackage{pdfpages}

\usepackage{xcolor}

\usepackage{enumitem}

\usepackage[labelfont=bf,font=sf]{caption}

\usepackage[font={sf,small},labelfont=md]{subfig}

\usepackage[noadjust,nosort]{cite}

\usepackage{amsmath, amsthm, amssymb, tikz, mathtools, array, mathrsfs, tensor}

\usepackage{polynom}

\usepackage{esint}

\usepackage{multicol}

\usepackage{dsfont}
	\newcommand{\one}{\mathds{1}}

\usepackage{framed}

\usepackage{upref}

\usepackage{hyperref}
	\hypersetup{colorlinks=true,linkcolor=blue,citecolor=blue,urlcolor=blue}

\allowdisplaybreaks


\newcommand{\eq}[1]{\begin{align*} #1 \end{align*}}
\newcommand{\eeq}[1]{\begin{align} \begin{split} #1 \end{split} \end{align}}

\newcommand{\stackref}[2]{\stackrel{\mbox{\footnotesize{\eqref{#1}}}}{#2}}
\newcommand{\stackrefp}[2]{\stackrel{\phantom{\mbox{\footnotesize{\eqref{#1}}}}}{#2}}

\def\eps{\varepsilon}
\def\vphi{\varphi}

\newcommand{\E}{\mathbb{E}}

\newcommand{\N}{\mathbb{N}}
\renewcommand{\P}{\mathbb{P}}

\newcommand{\R}{\mathbb{R}}
\renewcommand{\S}{\mathbb{S}}

\newcommand{\Z}{\mathbb{Z}}

\renewcommand{\AA}{\mathcal{A}}
\newcommand{\BB}{\mathcal{B}}
\newcommand{\CC}{\mathcal{C}}
\newcommand{\DD}{\mathcal{D}}

\newcommand{\FF}{\mathcal{F}}
\newcommand{\GG}{\mathcal{G}}

\newcommand{\II}{\mathcal{I}}

\newcommand{\KK}{\mathcal{K}}
\newcommand{\LL}{\mathcal{L}}
\newcommand{\MM}{\mathcal{M}}
\newcommand{\NN}{\mathcal{N}}

\newcommand{\PP}{\mathcal{P}}

\newcommand{\RR}{\mathcal{R}}
\renewcommand{\SS}{\mathcal{S}}
\newcommand{\TT}{\mathcal{T}}
\newcommand{\UU}{\mathcal{U}}
\newcommand{\VV}{\mathcal{V}}
\newcommand{\WW}{\mathcal{W}}
\newcommand{\XX}{\mathcal{X}}
\newcommand{\YY}{\mathcal{Y}}

\newcommand{\EEE}{\mathscr{E}}
\newcommand{\FFF}{\mathscr{F}}

\newcommand{\ff}{\mathfrak{f}}

\newcommand{\iprod}[2]{\langle #1,\, #2\rangle}

\newcommand{\vc}[1]{{\boldsymbol #1}}

\newcommand{\wt}[1]{\widetilde{#1}}
\newcommand{\wh}[1]{\widehat{#1}}

\DeclareMathOperator{\Var}{Var}
\DeclareMathOperator{\Cov}{Cov}

\DeclareMathOperator{\diam}{diam}

\DeclareMathOperator{\sech}{sech}

\DeclareMathOperator*{\esssup}{ess\,sup}
\DeclareMathOperator*{\essinf}{ess\,inf}
\DeclareMathOperator*{\argmin}{arg\,min} 

\newcommand{\givenk}[3][]{#1[ #2 \: #1| \: #3 #1]} 
\newcommand{\givenp}[3][]{#1( #2 \: #1| \: #3 #1)} 
\newcommand{\givena}[3][]{#1\langle #2 \: #1| \: #3 #1\rangle} 

\DeclareMathOperator{\e}{e} 
\newcommand{\cc}{\mathrm{c}} 
\newcommand{\dd}{\mathrm{d}} 
\DeclareMathOperator{\tv}{TV} 

\DeclarePairedDelimiter\ceil{\lceil}{\rceil}
\DeclarePairedDelimiter\floor{\lfloor}{\rfloor}

            \makeatletter
            \DeclareFontFamily{OMX}{MnSymbolE}{}
            \DeclareSymbolFont{MnLargeSymbols}{OMX}{MnSymbolE}{m}{n}
            \SetSymbolFont{MnLargeSymbols}{bold}{OMX}{MnSymbolE}{b}{n}
            \DeclareFontShape{OMX}{MnSymbolE}{m}{n}{
                <-6>  MnSymbolE5
               <6-7>  MnSymbolE6
               <7-8>  MnSymbolE7
               <8-9>  MnSymbolE8
               <9-10> MnSymbolE9
              <10-12> MnSymbolE10
              <12->   MnSymbolE12
            }{}
            \DeclareFontShape{OMX}{MnSymbolE}{b}{n}{
                <-6>  MnSymbolE-Bold5
               <6-7>  MnSymbolE-Bold6
               <7-8>  MnSymbolE-Bold7
               <8-9>  MnSymbolE-Bold8
               <9-10> MnSymbolE-Bold9
              <10-12> MnSymbolE-Bold10
              <12->   MnSymbolE-Bold12
            }{}
            
            \let\llangle\@undefined
            \let\rrangle\@undefined
            \DeclareMathDelimiter{\llangle}{\mathopen}%
                                 {MnLargeSymbols}{'164}{MnLargeSymbols}{'164}
            \DeclareMathDelimiter{\rrangle}{\mathclose}%
                                 {MnLargeSymbols}{'171}{MnLargeSymbols}{'171}
            \makeatother
            
\DeclareRobustCommand*{\ASref}{\ref{averages_squared}}
\DeclareRobustCommand*{\EOref}{\ref{expected_overlap_thm}}
\DeclareRobustCommand*{\ECref}{\ref{easy_cor}}
\DeclareRobustCommand*{\RefVLemma}{\ref{V_calculations}}

\newtheorem{thm}{Theorem}[section]
\newtheorem{prop}[thm]{Proposition}
\newtheorem{cor}[thm]{Corollary}
\newtheorem{lemma}[thm]{Lemma}
\newtheorem{claim}[thm]{Claim}
\newtheorem{theirthm}{Theorem}[chapter] 

\theoremstyle{definition}
\newtheorem{defn}[thm]{Definition}

\newtheorem{remark}[thm]{Remark}

    \author{Erik Bates}
    \submitdate{June 2019}
    
\begin{document}
    \title{Localization and free energy asymptotics in disordered statistical mechanics and random growth models}
    \author{Erik Bates}
    \principaladvisor{Sourav Chatterjee}
    \firstreader{Amir Dembo}
    \secondreader{Persi Diaconis}
    
 
    \beforepreface
    \prefacesection{Abstract}
    
    This dissertation develops, for several families of statistical mechanical and random growth models, techniques for analyzing infinite-volume asymptotics.
    In the statistical mechanical setting, we focus on the low-temperature phases of spin glasses and directed polymers, wherein the ensembles exhibit localization which is physically phenomenological.
    We quantify this behavior in several ways and establish connections to properties of the limiting free energy.
    We also consider two popular zero-temperature polymer models, namely first- and last-passage percolation.
    For these random growth models, we investigate the order of fluctuations in their growth rates, which are analogous to free energy.

        
    \prefacesection{Acknowledgments}

Of the many intellectual opportunities I have enjoyed as a student, working with Sourav Chatterjee has been by far the most rewarding.
I am profoundly grateful for our collaborations and his kind and thoughtful advising.  His mentorship provides me inspiration on a daily basis, which is perhaps the best gift an advisor can give his student.

Much gratitude is also owed to Persi Diaconis and Amir Dembo, whose instruction and guidance early in my graduate career were pivotal in my nascent studies of probability theory.  Their suggestions and encouragement have been influential throughout my time at Stanford; I am thankful for the great many things they have taught me.

For the wonderful experience I have had in the Stanford Department of Mathematics, I thank the entirety of its faculty and staff.  Particularly impactful have been Gretchen Lantz, who helped shape every step (from visit weekend to filing for graduation) for the better, and Rose Stauder, whose friendship and wholehearted service  I have missed since her retirement.
I would also like to thank Lenya Ryzhik and Eleny Ionel for their leadership and counsel as directors of graduate studies, as well as Emmanuel Cand\`es, Jacob Fox, Andrea Montanari, George Papanicolaou, George Schaeffer, Tadashi Tokieda, Andr\'as Vasy, Ravi Vakil, Jan Vondr\'ak, and Lexing Ying for excellent learning experiences.
Most of all, I thank my classmates with whom I have had the privilege of sharing my graduate education.
Their friendship has brought great joy.  I also owe special appreciation to my current and former office mates---Francisco Arana Herrera, Hannah Larson, Nhi Troung, Weiluo Ren, and Gurbir Dhillon---both for their unfailing cheer and for putting up with my clothes, shoes, plants, and office hours.

My education at Stanford has been significantly influenced by several others elsewhere in the university.
I am deeply thankful for Carole Pertofsky, who inspired me (and many others) to learn how to pursue wellness.  In the journey that ensued, I was fortunate to learn from Donnovan Yisreal, Aneel Chima, and Orgyen Chowang Rinpoche.  Their influence on my life vastly outsizes the small amount of class time their courses are allotted.
I am especially grateful to Suzi Weersing, Stephanie Kalfayan, and Persis Drell for not only their role in making this university great, but also the unique opportunities I have had to learn from their leadership on campus.

I am privileged to have been preceded by a number of terrific role models, to whom I am indebted for their example and advice: Megan Bernstein, Evita Nestoridi, Graham White, Chris Janjigian, Arjun Krishnan, Riddhi Basu, Nick Cook, Naomi Feldheim, Ohad Feldheim, Julian Gold, Moumanti Podder, and Subhabrata Sen.
I also thank the many others who played a role in my job application process: Antonio Auffinger, Alan Hammond, Sam Kimport, Tom Church, Shirshendu Ganguly, Peter Hintz, Sarah Peluse, Jens Reinhold, Kate Coppess, Irene Li, James Miller, and Maureen Bragger.

Before graduate school I was influenced by many other very important educators.  It would be near impossible to overstate the importance of Jeanne Wald in shaping my undergraduate experience at Michigan State University. 
I benefitted tremendously from the exceptional opportunities she cultivated for students in the Advanced Track program, the inaugural class of which I was lucky to be a part of.
I am also especially thankful to Anne Gelb, Bruce Sagan, Jeff Schenker, and Yang Wang for their commitment and enthusiasm in helping me grow mathematically.

I enjoyed truly every one of my high school and middle school mathematics teachers, whose instruction was second to none: Don Tolly, Rebecca Achterhof, Andy Evans, Robin Teliczan, Mike Michaud, and Dave Friday.
Just as important, however, have been the writing skills I learned from my english teachers: Linda McCarthy, Lynn Cvengros, Kathleen Devarenne, and Kelly Kermode.
And to Cynthia Girodat, Erik Cliff, and Sara Ahmicasaube, I will be always be grateful for the love of learning they inspired.

Special thanks are due to four of my closest friends. 
Without Melvin, my passion for mathematics would not be the same.  Without Chad, my passion for non-mathematics would not be the same. 
Without Trevor, problem sets on Friday and Saturday nights would not have been the same.
And without Spencer, my website's domain name would not be the same.

Finally, I owe my deepest thanks to my parents, grandparents, siblings, and partner Laura, for their unconditional support.
They are the greatest part of life.
I dedicate this dissertation to my grandfather, Clifford Weil, who inspired me to be a full-time mathematician.

\vspace{2\baselineskip}
\noindent The research comprising this dissertation was partially supported by a National Science Foundation Graduate Research Fellowship, grant DGE-114747.

\DeclareRobustCommand*{\ASref}{\ref{averages_squared}}
\DeclareRobustCommand*{\EOref}{\ref{expected_overlap_thm}}
\DeclareRobustCommand*{\ECref}{\ref{easy_cor}}
\DeclareRobustCommand*{\RefVLemma}{\ref{V_calculations}}

    \tablespagefalse
    \afterpreface
    
 
    \chapter{Introduction}

In statistical mechanics, physical systems are considered to exist in a random state, where the probability of  realizing a given state is determined by its energetic favorability.
For instance, if the total energy of the system is fixed, then all possible states are equally likely.
That is, if there are $N$ possible states, then the system at equilibrium has chance $1/N$ of being in any particular state.
This is called a \textit{microcanonical ensemble} and is a consequence of the second law of thermodynamics,  which states that the entropy of an isolated system can never decrease. 
In particular, entropy must be maximized at equilibrium; indeed, regarding the case above, the unique probability measure on $N$ elements with maximum entropy is the uniform measure.

In real-world systems, however, different states very often have different energies.
The system can change between these states by exchanging energy with its surroundings.
This results in a \textit{canonical ensemble}, wherein the probability of realizing state $\sigma$ (from among a finite collection $\Sigma$ of possibilities) is
\eeq{ \label{canonical_ensemble}
\frac{1}{Z}\e^{-H(\sigma)/(kT)},
}
where $H(\sigma)$ is the energy of state $\sigma$, $T$ is temperature, $k$ is Boltzmann's constant, and $Z$ is a normalizing constant known as the \textit{partition function}:
\eq{
Z = \sum_{\sigma\in\Sigma} \e^{-H(\sigma)/(kT)}.
} 
The function $H(\cdot)$ is called the \textit{Hamiltonian}.
The reason for \eqref{canonical_ensemble} is again the second law: among all probability measures $\mu$ on $\Sigma$ yielding the same average energy,
\eq{
\sum_{\sigma\in\Sigma} H(\sigma)\mu(\sigma) = \sum_{\sigma\in\Sigma} H(\sigma)\frac{1}{Z}\e^{-H(\sigma)/(kT)},
}
it is \eqref{canonical_ensemble} that maximizes entropy.
To understand equilibria of thermodynamic systems, then, one must be able to analyze canonical ensembles.

Typically one is interested in the asymptotic properties of these ensembles in the so-called \textit{thermodynamic limit}, in which the volume and number of particles are simultaneously taken to infinity in such a way that the particle density remains fixed.\footnote{There are also \textit{grand canonical ensembles}, in which a system exchanges not just energy but also mass with its surroundings.  That is, the number of particles is itself a random quantity.  Since we will focus on thermodynamic limits, this distinction is not important.
Also, here ``volume" and ``particle" are generic terms whose physical meaning changes with the system under consideration.}
This is how we can begin to relate our models to real-world systems.
Indeed, it is only in the infinite-volume setting that the models will exhibit phase transitions.
To be precise, let us set some notation which will appear throughout the text.

Let $(\Sigma_n,\mathscr{F}_n,P_n)_{n\geq1}$ be a sequence of probability spaces, each one equipped with a Hamiltonian $H_n : \Sigma_n \to \R$.
We will write $E_n(\cdot)$ for expectation with respect to $P_n$.
For convenience, we set $\beta \coloneqq 1/(kT)$, which (in a slight abuse of terminology) we call the \textit{inverse temperature}.
The canonical ensemble or \textit{Gibbs measure} is then the probability measure $\mu_n^\beta$ on $(\Sigma_n,\mathscr{F}_n)$ given by\footnote{\label{minus_footnote}Notice that we have adopted the mathematician's convention of replacing the Hamiltonian by its negative.  While this means the energetically favorable states are now those with \textit{high} energy, it will save us several hundred minus signs throughout the text.}
\eeq{ \label{gibbs_definition}
\mu_n^\beta(\dd\sigma) \coloneqq \frac{\e^{\beta H_n(\sigma)}}{Z_n(\beta)}\ P_n(\dd\sigma), \qquad \text{where} \qquad
Z_n(\beta) \coloneqq E_n(\e^{\beta H_n(\sigma)}).
}
That is, $\mu_n^\beta$ is absolutely continuous with respect to $P_n$, having density $\e^{\beta H_n(\sigma)}/Z_n(\beta)$.
To denote expectation with respect to $\mu_n^\beta$, we use angled brackets:
\eq{
\langle f(\sigma)\rangle_\beta \coloneqq \int_{\Sigma_n} f(\sigma)\ \mu_n^\beta(\dd\sigma) = \frac{E_n(f(\sigma)\e^{\beta H_n(\sigma)})}{E_n(\e^{\beta H_n(\sigma)})}.
}
Typically the scaling will be chosen so that $Z_n(\beta)$ is of exponential order in $n$, and so the relevant intensive quantity is the \textit{free energy density} (which we will simply call the \textit{free energy}),\footnote{It would be more proper to call $p(\beta)/\beta$ the free energy, although it will be easier to work with $p(\beta)$ directly.}
\eq{
F_n(\beta) \coloneqq \frac{1}{n}\log Z_n(\beta).
}
We will say that a thermodynamic limit exists if there exists a limiting free energy,
\eq{
p(\beta) \coloneqq \lim_{n\to\infty} F_n(\beta) \quad \text{for all $\beta\geq0$.}
}
The fundamental importance of the function $p(\cdot)$ can be understood from at least two perspectives.

On one hand, $F_n(\cdot)$ is the logarithmic moment generating function (also known as the cumulant generating function) of the random variable $H_n(\sigma)$ with respect to $P_n$.
As such, it completely determines the distribution of said random variable, and key thermodynamic quantities can be expressed in terms of its derivatives.
For example, the expected energy density of a state is simply a first derivative,
\eq{
\frac{\langle H_n(\sigma)\rangle_\beta}{n} = F_n'(\beta),
}
while the density of energy variance is the second derivative:
\eq{
\frac{\langle{H_n(\sigma)^2}\rangle_\beta - \langle{H_n(\sigma)}\rangle_\beta^2}{n} = F_n''(\beta).
}
In turn, the asymptotics of these statistics are recorded in the limit function $p(\cdot)$.
Moreover, phase transitions can be identified as points of non-analyticity.\footnote{Assuming $F_n(\beta)$ is finite for all $\beta$, the function $F_n(\cdot)$ is everywhere analytic.  Mathematically, this is why phase transitions only appear in infinite-volume settings.}
These generally occur because one of the derivatives of $p(\cdot)$ has a discontinuity, which corresponds to some thermodynamic quantity being discontinuous with the temperature.

On the other hand, from a more probabilistic perspective, $p(\cdot)$ is the instrument by which one can obtain large deviations principles for the canonical ensemble from those for the reference sequence $(P_n)_{n\geq1}$.
In this sense, the free energy encodes information on the statistics of not only typical states, but also rare ones.
For example, suppose that
$H_n/n$ obeys
a large deviations principle with respect to $P_n$, which we express in condensed form by
\eq{
-\lim_{n\to\infty} \frac{1}{n}\log P_n\Big(\frac{H_n(\sigma)}{n} \in \dd h\Big) = I(h), \quad h\in\R.
}
At least formally, it follows that
\eq{
-\lim_{n\to\infty} \frac{1}{n}\log \mu_n^\beta\Big(\frac{H_n(\sigma)}{n}\in \dd h\Big) = I(h) - \beta h  + p(\beta) \eqqcolon I^\beta(h).
}
Furthermore, knowing that $\inf_h I^\beta(h) = 0$, we can write $p(\cdot)$ as the Legendre--Fenchel transform of the original rate function:
\eeq{ \label{min_energy_principle}
p(\beta) = \sup_{h}\{\beta h - I(h)\}.
}
This is called the \textit{minimum energy principle} and specifies the optimal tradeoff between energy (i.e.~the value of $h$) and entropy (the density of states having that energy), occurring at equilibrium.\footnote{It may seem strange that a variational formula with a supremum is labeled as a \textit{minimum} principle.
The reason for this is that we replaced $-H_n(\sigma)$ with $H_n(\sigma)$ in \eqref{gibbs_definition}; see footnote \ref{minus_footnote}.
Also notice that if $I(\cdot)$ is convex, then duality allows us to conclude from \eqref{min_energy_principle} that
$I(h) = \sup_\beta \{\beta h - p(\beta)\}$.
}

The above discussion amounts to the following premise: tantamount to equilibrium statistical mechanics are the connections between a system's physical properties and a single function: the free energy.
Historically, the range of mathematical tools developed to understand these connections has been immense.
Many of the most advanced techniques, however, are model-specific and not robust to small parameter changes.
This is particularly true for \textit{disordered} models, which are the focus of this dissertation.

In disordered models, the Hamiltonian $H_n(\cdot)$ is itself random.
Physically, this is meant to model random irregularities in either the system's ``particles" or the surroundings in which they exist.
The total energy of a configuration thus depends on the random realization of the disorder from which the Hamiltonian is defined, and so statements like \eqref{min_energy_principle} become more difficult to formulate.  
In general, the challenge is to determine if, and in what ways, this disorder changes the behavior of the system.
That is, how does the \textit{quenched} system (i.e.~the ensemble given the random disorder) compare to the prior distribution $P_n$?
Moreover, to what degree do the differences depend on the specific realization of the disorder?
Given a specific model, common objectives include:
\begin{itemize}
\item[(a)] establishing the existence of a thermodynamic limit and (ideally) constructing an infinite-volume ensemble to which the finite-volume ensembles converge;
\item[(b)] identifying phase transitions that determine when qualitative differences appear between the disordered and non-disordered systems; 
\item[(c)] demonstrating that the qualitative behavior of the quenched system does not depend on the realization of the disorder (in an almost sure sense); and
\item[(d)] determining, within a single qualitative phase, the quantitative variability that \textit{does} depend on the realization of the disorder.
\end{itemize}
The last item above refers to the \textit{averaged quenched} ensemble, which is the ensemble obtained by averaging over all possible quenched ensembles.

In the following chapters, we will make progress on each these goals for various disordered systems.
Our emphasis will be on methodology, namely the development of new mathematical tools that might be applied in other contexts, and so we will seek generality whenever possible.
With that said, two families of models will feature prominently throughout, namely \textit{directed polymers in random environment} and \textit{spin glasses}.
These models will be defined in Chapters \ref{endpoint} and \ref{replica}, respectively.
Very briefly, the contents of the dissertation can be summarized as follows:
\begin{itemize}
\item Chapter \ref{endpoint} (includes content from \cite{bates-chatterjee20I,bates18}): 
A series of abstract tools are developed in order to study the endpoint localization phenomenon exhibited by directed polymers at low temperatures.
As a first application of these tools, we prove that in an environment with finite exponential moment, the endpoint distribution is asymptotically purely atomic if and only if the system is in the low temperature phase. 
The analogous result for a heavy-tailed environment was proved by Vargas in~2007.
As a second application, we prove a subsequential version of the longstanding conjecture that in the low temperature phase, the endpoint distribution is asymptotically localized in a region of stochastically bounded diameter. 
All results hold in arbitrary dimensions and make no use of integrability. \\
Main results: Theorems \ref{intro_result1}, \ref{intro_result2}, and \ref{upper_bound}

\item Chapter \ref{replica} (includes content from \cite{bates-chatterjee20II}): We consider a broad class of disordered systems in which the disorder is assumed to be Gaussian distributed, including directed polymers in Gaussian environment and also classical mean-field spin glasses.
For this collection of models, we identify a sufficient condition (in terms of the free energy) for the ensemble to concentrate on a small collection of states.
This proves an alternate version of the localization phenomenon from Chapter \ref{endpoint}, a version which in fact generalizes very naturally to the broader class of models. \\
Main results: Theorems \ref{easy_cor}, \ref{expected_overlap_thm}, and \ref{averages_squared}

\item Chapter \ref{rsb} (includes content from \cite{bates-sloman-sohn19}): Here we focus on the \textit{multi-species Sherrington--Kirkpatrick (SK) spin glass}, which provides a mild departure from the mean-field setting (i.e.~the classical SK model) toward a model which captures the communities and inhomogeneities of real-world networks.
By an analysis of the Parisi variational formula for the free energy, we prove the existence of a phase transition and identify a sufficient condition for subcriticality that is conjecturally sharp.
This gives the first non-asymptotic proof of so-called replica symmetry breaking for a multi-species spin glass. \\
Main results: Theorems \ref{uniquenessof2speciesrRSsol} and \ref{2speciessymmetrybreaking}

\item Chapter \ref{fluctuations} (includes content from \cite{bates-chatterjee20III}):
The final chapter specializes to planar random growth models, including not only directed polymers but also first- and last-passage percolation, which can be thought of as zero-temperature polymer models.
We prove $\sqrt{\log n}$ lower bounds on the order of growth fluctuations under no assumptions on the distribution of vertex or edge weights other than the minimum conditions required for avoiding pathologies. 
Such bounds were previously known only for certain restrictive classes of distributions.
In addition, the first-passage shape fluctuation exponent is shown to be at least $1/8$, extending previous results to more general distributions.  \\
Main results: Theorems \ref{fpp_thm}, \ref{fpp_thm_1}, \ref{fpp_thm_2}, \ref{lpp_thm}, and \ref{dp_thm}

\end{itemize}

Each chapter can be read independently of the others.

\section{Notation}
Throughout the manuscript, $\R$ will denote the real numbers, $\Z$ the integers, and $\N$ the positive integers.
The symbol $\P$ will denote a probability measure defined on some probability space---which we often do not mention explicitly---and $\E$ will denote expectation with respect to $\P$.

Euler's number will be written in roman font as $\e$, so as to not be confused with $e$, which in Chapter \ref{fluctuations} will denote an edge of $\Z^2$.

When defining key words or phrases, we will use italics to indicate which word or phrase is being defined.


%
%

    
    \chapter{Endpoint localization of directed polymers} \label{endpoint}
    
\section{Introduction}
The model of directed polymers in random environment was introduced in the physics literature by Huse and Henley \cite{huse-henley85} to represent the phase boundary of the Ising model in the presence of random impurities. It was later mathematically reformulated as a model of random walk in random potential by Imbrie and Spencer~\cite{imbrie-spencer88}. 
In its classical form, the model considers a simple random walk (SRW) on the integer lattice $\Z^d$, whose paths---considered the ``polymer"---are reweighted according to a random environment that refreshes at each time step.
Large values in the environment tend to attract the random walker and possibly force localization phenomena; this attraction grows more effective in lower dimensions and at lower temperatures.
On the other hand, the random walk's natural dynamics favor diffusivity.
Which of these competing features dominates asymptotically is a central question in the study of directed polymers.  

Over the last thirty years, the directed polymer model has played an important role as a source of many fascinating problems in the probability literature, culminating in the amazing recent developments in integrable polymer models. However, in spite of the wealth of information now available for integrable models, our knowledge about the general case is fairly limited, especially in spatial dimension greater than one. The goal of this chapter is to introduce an abstract theory that allows computations for polymer models that are not integrable. 
Before we discuss our approach in detail, let us very briefly communicate the main consequences:

\begin{itemize}
\item The probabilistic model of $(d+1)$-dimensional directed polymers of length $n$ assigns a random probability measure to the set of random walk paths of length $n$ in $\Z^d$ that start at the origin. The precise mathematical model is defined in Section \ref{model_1} below. The ``endpoint distribution'' is the probability distribution of the final vertex reached by the path. Note that this is a random probability measure on $\Z^d$, which we will denote $f_n$. Understanding the  behavior of $f_n$ is the main goal of this chapter. Although a number of results about $f_n$ were known prior to this work (reviewed in  Section \ref{disorder_background}), this is the first work that puts forward a comprehensive theoretical framework for analyzing the asymptotic properties of~$f_n$.
\item 
To understand the asymptotic behavior of $f_n$, we 
consider the {\it empirical measure} of the endpoint distribution,
\eq{
\rho_n \coloneqq  \frac{1}{n}\sum_{i=0}^{n-1} \delta_{f_i},
}
where $\delta_{f_i}$ is the Dirac point mass at $f_i$, considering $f_i$ as a random element of the space of probability measures on $\Z^d$. Note that $\rho_n$ is therefore a random probability measure on the space of probability measures on $\Z^d$. 
\item To obtain an adequate understanding of $\rho_n$, it turns out that we need to move to a larger space. The reason for this will be explained later, in Section \ref{endpoint_results}. We first define a non-standard compactification $\SS$ of the space of probability measures on $\Z^d$, and then consider $\PP(\SS)$, the space of probability measures on $\SS$ with the Wasserstein metric. We consider $\rho_n$ as an element of $\PP(\SS)$. The elements of $\SS$ are called ``partitioned subprobability measures''.

\item We define a certain continuous function $\TT:\PP(\SS)\to \PP(\SS)$, called the ``update map'', and let $\KK$ be the set of fixed points of $\TT$. Furthermore, we define a linear functional $\RR$ on $\PP(\SS)$, and let $\MM$ be the subset of $\KK$ where $\RR$ is minimized. Our first main result is that $\MM$ is nonempty, and the distance between $\rho_n$ and the set $\MM$ converges almost surely to zero as $n\to \infty$. 

\item As a consequence of the above result, we obtain a new variational formula for the limiting free energy of directed polymers; namely, the limiting free energy equals the value of $\RR$ on $\MM$. 
Like many statistical mechanical models, directed polymers have a high temperature phase and a low temperature phase.
These are characterized respectively by whether $\KK$ is just a single trivial object or instead contains non-trivial elements.

\item  One of the most striking features of directed polymers is that in the low temperature phase, $f_n$ has ``atoms'' whose weights do not decay to zero as $n\to \infty$. This is in stark contrast with the endpoint distribution of, say, simple random walk, where the most likely site has mass of order $n^{-d/2}$. The precise statement of this well-known localization phenomenon will be discussed in Section~\ref{disorder_background}. One of the main applications of the abstract machinery developed in this chapter is to show that in the low temperature regime, the atoms account for {\it all} of the mass---that is, there is no part of the mass that diffuses out. Such a result was proved earlier for directed polymers in heavy-tailed environment, and also for a particular $(1+1)$-dimensional integrable model. 
We prove it under finite exponential moments, to which previously known techniques do not apply.

\item Our second main application is to show that in the low temperature regime, there is almost surely a subsequence of positive density along which the endpoint distribution concentrates mass $> 1-\delta$ on a set of diameter $\leq K(\delta)$, where $\delta$ is arbitrary and $K(\delta)$ is a deterministic constant that depends only on $\delta$ and some features of the model. In other words, not only does the mass localize on atoms, but the atoms themselves localize in a set of bounded diameter. This proves a subsequential version of a longstanding conjecture about the endpoint distribution. Prior to this work, the only case where a similar statement could be proved was for a $(1+1)$-dimensional integrable model.

\end{itemize}
We will now begin a more detailed presentation by defining the model below. This is followed by a discussion of the known results about polymer models, and then a general overview of the results proved in this chapter and the ideas involved in the proofs. 

\subsection{The model of directed polymers in random environment} \label{model_1}
Let $d$ be a positive integer, to be called the \textit{spatial} or \textit{transverse dimension}. 
Let $P$ denote the law of the \textit{reference walk}, which is a homogeneous random walk $(\sigma_i)_{i\geq0}$ on $\Z^d$.
To be precise, we state the assumptions on $P$:
\eeq{
P(\sigma_0 = 0) = 1, \qquad P\givenp{\sigma_{i+1} = x}{\sigma_i = y} = P(\sigma_1 = x - y) \eqqcolon P(y,x) < 1. \label{walk_assumption_1}
}
Next we introduce a collection of i.i.d.~random variables $\vc \omega = (\omega(i,x) : i \geq 1, x \in \Z^d)$, called the \textit{random environment}, supported on some probability space $(\Omega,\FF,\P)$.
We will write $E$ and $\E$ for expectation with respect to $P$ and $\P$, respectively.
Finally, let $\beta\geq0$ be a parameter representing \textit{inverse temperature}.
Then for $n\geq0$, the \textit{quenched polymer measure} of length $n $ is the Gibbs measure $\mu_n^\beta$ defined by
\eeq{ \label{polymer_measure_def}
\mu_n^\beta(\dd \sigma) = \frac{1}{Z_n(\beta)} \e^{\beta H_n(\sigma)}\ P(\dd \sigma), \quad \text{where} \quad H_n(\sigma) \coloneqq \sum_{i = 1}^n \omega(i,\sigma_i).
}
The normalizing constant
\eq{
Z_n(\beta) \coloneqq E(\e^{\beta H_n(\sigma)}) = \sum_{\sigma_1,\dots,\, \sigma_n \in \Z^d} \exp\bigg(\beta \sum_{i = 1}^n \omega(i,\sigma_i)\bigg)\prod_{i=1}^n P(\sigma_{i-1},\sigma_i), \quad \sigma_0 \coloneqq 0,
}
is called the \textit{quenched partition function}.
A fundamental quantity of the system is calculated from this constant, namely the \textit{quenched free energy},
\eq{
F_n(\beta) \coloneqq \frac{\log Z_n(\beta)}{n}.
}
We specify ``quenched" to indicate that the randomness from the environment has not been averaged out.
That is, $\mu_n^\beta$, $Z_n(\beta)$, and $F_n(\beta)$ are each random processes with respect to the filtration
\eeq{ \label{fndef}
\FF_n \coloneqq \sigma(\omega(i,x) : 1 \leq i \leq n, x \in \Z^d), \quad n \geq 0.
}
It is this quenched setting in which we are interested, and so unless stated otherwise, ``almost sure" statements are made with respect to $\P$.

When $Z_n(\beta)$ is replaced by its expectation $\E Z_n(\beta)$, one obtains the \textit{annealed} free energy,
\eeq{ \label{mgf_def}
\frac{\log \E Z_n(\beta)}{n} = \frac{\log (\E \e^{\beta \omega})^n}{n} = \log \E(\e^{\beta \omega}) \eqqcolon \lambda(\beta),
}
where $\omega$ denotes (here and henceforth) a generic copy of $\omega(1,0)$.
Notice that $\lambda(\cdot)$ depends only on the law of the environment, which we denote $\mathfrak{L}_\omega$.
We will assume finite ($1+\eps$)-exponential moments,
\eeq{ \label{mgf_assumption}
0 \leq \beta < \beta_{\max} \coloneqq \sup\{t \geq 0: \lambda(\pm t) < \infty\},
}
so that $Z_n(\beta)$ has finite $\pm(1+\eps)$-moments at the given inverse temperature.
Otherwise $\mathfrak{L}_\omega$ is completely general, although to avoid trivialities, we will always assume that $\omega$ is not an almost sure constant.

We can equivalently write
\eq{
Z_n(\beta) = \sum_{x \in \Z^d} Z_n(\beta,x),
}
where
\eeq{ \label{p2p}
Z_n(\beta,x) \coloneqq  E(\e^{\beta H_n(\sigma)}; \sigma_n = x)
}
is often referred to as a \textit{point-to-point} partition function.
Correspondingly, $Z_n(\beta)$ is distinguished as the \textit{point-to-level} partition function.
In the sequel, we will neither obtain nor mention results specific to the point-to-point setting, although there is an inherent duality with the point-to-level setting if the random environment is paired with an external field; see \cite[Section 4]{georgiou-rassoul-seppalainen16}.
This relationship is essential in the study of (1+1)-dimensional integrable models, which we highlight in Section \ref{solvability_background} below.
First, though, we review results for the general setting.

\section{An overview of known results}  \label{disorder_background}
In this section we briefly survey results from the literature that will be relevant to the ones proved in this chapter.  
Since we will work with arbitrary dimension and arbitrary environment, our main focus will be on results proved in this general setting; these are presented in Sections \ref{SRW_background} and \ref{long_range_review}.
In Section \ref{solvability_background}, we will highlight related work for so-called ``exactly solvable" or "integrable" models that are limited to the one-dimensional case.

\subsection{Simple random walk case} \label{SRW_background}
We first specialize to the classical (and standard) case when $P$ is SRW.
One can find a more detailed review in the book of Comets \cite{comets17}.

This directed polymer model first appeared in physics literature \cite{huse-henley85,kardar85,huse-henley-fisher85,kardar-nelson85,kardar-zhang87}.
A typical physical interpretation of the setup explains the ``polymer" nomenclature.
Each vertex in $\Z^d$ is a possible location of a monomer, and a polymer is a chain of monomers occupying neighboring vertices.
A random walk $(\sigma_i)_{i \geq 0}$ chosen according to $P$ represents a polymer growing without preference for particular directions; that is, growing in a purely uniform environment.
When random impurities are introduced, the polymer may exhibit markedly different behavior because of attraction to or repulsion from certain impurities.
A simple choice of environment to model this scenario is a Bernoulli random environment:
\eq{
\omega_u = \begin{cases}
-1 &\text{with probability $p > 0$,} \\
+1 &\text{with probability $1 - p > 0$}.
\end{cases}
}
In this case, the energy of a sample path would increase by traversing a $+1$ site, thus decreasing its likelihood under $\mu_n^\beta$.
Another standard choice is a Gaussian random environment, where each $\omega_u$ is a standard normal random variable.


\subsubsection{High and low temperature phases}
The qualitative behavior of directed polymers depends on the disorder distribution $\mathfrak{L}_\omega$, the inverse temperature $\beta$, and the dimension $d$.
Moreover, this dependence can be detected through the free energy $F_n(\beta)$ as follows.
When the randomness of the environment is averaged out, one obtains the \textit{expected quenched free energy},
\eq{
\E F_n(\beta) = \frac{1}{n}\, \E \log Z_n(\beta).
}
In general, this quantity is distinct from $\lambda(\beta)$, the annealed free energy.
Indeed, Jensen's inequality 
gives the comparison
\eeq{ \label{jensen_applied_finite}
\E F_n(\beta) < \lambda(\beta),
}
where the inequality is strict because $F_n(\beta)$ is not an almost sure constant, and $x\mapsto\log x$ is not linear.
A superadditivity argument  (see the proof of Lemma \ref{means_converge}) 
shows that $\E F_n(\beta)$ converges to $\sup_{n \geq 0} \E F_n(\beta)$ as $n\to\infty$.
It has also been shown in \cite{vargas07}, under hypotheses much weaker than \eqref{mgf_assumption}, that $F_n(\beta)-\E F_n(\beta)$ tends to $0$ almost surely (see also Lemma \ref{concentration}).
Therefore, there is a deterministic limit\footnote{For $P =$ SRW, \eqref{Fn_lim} was initially shown by Carmona and Hu \cite[Proposition 1.4]{carmona-hu02} for a Gaussian environment and by Comets, Shiga, and Yoshida \cite[Proposition 2.5]{comets-shiga-yoshida03} for a general environment.
It was later proved by Vargas \cite[Theorem 3.1]{vargas07} under weaker assumptions.
In Proposition \ref{free_energy_converges}, we give a proof for general $P$ under the assumption \eqref{mgf_assumption}.}
\eeq{
p(\beta) \coloneqq  \lim_{n \to \infty} \E F_n(\beta) = \lim_{n\to\infty} F_n(\beta) \quad \mathrm{a.s.} \label{Fn_lim}
}
Using the FKG inequality, Comets and Yoshida~\cite{comets-yoshida06} identified a phase transition:

\begin{theirthm}[{\cite[Theorem 3.2]{comets-yoshida06}}]\label{critical_temperature}
Let $P$ be SRW and assume $\lambda(t) < \infty$ for all $t \in \R$.
Then there exists a critical value $\beta_\cc =  \beta_{\mathrm{c}}(\mathfrak{L}_\omega,d) \in [0,\infty]$ such that
\eq{
0 \le \beta \leq \beta_{\mathrm{c}} \quad &\implies \quad p(\beta) = \lambda(\beta), \\
\beta > \beta_{\mathrm{c}} \quad &\implies \quad p(\beta) < \lambda(\beta).
}
\end{theirthm}

We refer to the region $0 \le \beta < \beta_{\mathrm{c}}$ as the ``high temperature phase'', while $\beta > \beta_{\mathrm{c}}$ defines the ``low temperature phase".\footnote{The difference $\lambda(\beta)-p(\beta)$ is sometimes called the \textit{Lyapunov exponent} of the system.
In this language, the low temperature phase is characterized by a nonzero Lyapunov exponent.}
Roughly speaking, high temperatures reduce the influence of the random environment, and so polymer growth resembles the original reference walk, while at low temperatures the random impurities force a much different behavior.
This distinction has been most frequently made in terms of the \textit{endpoint distribution} $\mu_n^\beta(\sigma_n \in \cdot) $, which is a random probability measure on $\Z^d$.
For instance, one striking result first proved by Carmona and Hu \cite{carmona-hu02} (for a Gaussian environment) 
and then by Comets, Shiga, and Yoshida \cite{comets-shiga-yoshida03} (in the general case) 
is the following: 
If $\beta > \beta_{\mathrm{c}}$, then the polymer endpoint observes so-called \textit{strong localization}:
\begin{align} \label{strong_localization}
\exists\ c > 0, \quad \limsup_{n \to \infty} \max_{x \in \Z^d} \mu_n^\beta(\sigma_n = x) \geq c \quad \text{a.s.}
\tag{SL}
\end{align}
That is, infinitely often the polymer has ``favorite sites" at which its endpoint distribution concentrates.

\subsubsection{Weak and strong disorder regimes}
While examining the high and low temperature regimes is very natural, the mathematical development of directed polymers in random environment has often followed an ostensibly different route.
Since the work of Bolthausen \cite{bolthausen89}, analysis of the directed polymer model has frequently focused on the \textit{normalized partition function},
\eq{
W_n(\beta) \coloneqq  Z_n(\beta) \e^{-n\lambda(\beta)}, 
}
with $W_0 = Z_0 = 1$. 
It is not difficult to check that $(W_n(\beta))_{n\geq0}$ is a positive martingale adapted to the filtration $(\FF_n)_{n \geq 0}$ defined in \eqref{fndef}.
In particular, $\E[W_n(\beta)] = 1$ for all $n$, and the martingale convergence theorem 
implies that there is an $\FF$-measurable random variable $W_\infty(\beta)$ such that
\eeq{
\lim_{n \to \infty} W_n(\beta) = W_\infty(\beta) \quad \text{a.s.} \label{Ztilde_def}
}
Furthermore, we necessarily have $W_\infty(\beta) \geq 0$ almost surely, and the event of positivity
$\{W_\infty(\beta) > 0\}$ is measurable with respect to the tail $\sigma$-algebra,
\eq{
\bigcap_{n = 1}^\infty \sigma(\omega(i,x) : i \geq n,\, x\in \Z^d).
}
By Kolmogorov's zero-one law 
we have either \textit{weak} or \textit{strong disorder},\footnote{Related work on the (multiplicative) stochastic heat equation has been done for $d\geq3$ \cite{mukherjee-shamov-zeitouni16}.}
\begin{align}
W_\infty(\beta) &> 0 \quad \text{a.s.}, \tag{WD} \label{weak_disorder} \\
W_\infty(\beta) &= 0 \quad \text{a.s.} \tag{SD} \label{strong_disorder}
\end{align}
Carmona and Hu \cite{carmona-hu02} 
and Comets, Shiga, and Yoshida \cite{comets-shiga-yoshida03} 
gave the following characterization of the two phases.
It says that there is strong disorder exactly when the endpoint overlap of two independent polymers has infinite expectation.

\begin{theirthm}[{\cite[Proposition 5.1]{carmona-hu02} and \cite[Theorem 2.1]{comets-shiga-yoshida03}}]\label{disorder_equiv} 
\label{sd_equals_wl}
Let $P$ be SRW and assume $\lambda(t) < \infty$ for all $t \in \R$.
Then strong disorder \eqref{strong_disorder} is equivalent to \textit{weak localization},
\begin{align}  \label{weak_localization}
\sum_{n = 0}^\infty (\mu_n^\beta)^{\otimes 2}(\sigma_n = \sigma_n') = \infty \quad \mathrm{a.s.}, \tag{WL}
\end{align}
where $\sigma$ and $\sigma'$ are independent samples from $\mu_n^\beta$.
\end{theirthm}

Despite the unfortunate clash of terminology in the above theorem, one should not mistake weak localization to be associated with weak disorder.
The use of ``weak" in the first case is only to distinguish this notion of localization from the other \eqref{strong_localization} notion defined before.
Indeed, it is apparent from the inequality
\eq{
(\mu_n^\beta)^{\otimes 2}(\sigma_n = \sigma_n') \geq \Big(\max_{x \in \Z^d} \mu_n^\beta(\sigma_n = x)\Big)^2
}
that $\eqref{strong_localization} \implies \eqref{weak_localization}$.
In fact, Carmona and Hu \cite{carmona-hu06} proved that in a continuous-time model with an environment of i.i.d.~Brownian motions---the parabolic Anderson model---the converse is also true: $\eqref{strong_localization} \Leftarrow \eqref{weak_localization}$. 
It is believed that the two notions are equivalent in general.

As with the high and low temperature regimes, there is a phase transition between weak and strong disorder. 

\begin{theirthm}[{\cite[Theorem 12.1]{comets-yoshida06}}]
Let $P$ be SRW and assume $\lambda(t) < \infty$ for all $t \in \R$.
There exists a critical inverse temperature $\bar{\beta}_{\mathrm{c}} = \bar{\beta}_{\mathrm{c}}(\mathfrak{L}_\omega,d) \in [0,\infty]$ such that
\eq{
0 \le \beta < \bar{\beta}_{\mathrm{c}} \quad &\implies \quad \eqref{weak_disorder}, \\
\beta > \bar{\beta}_{\mathrm{c}} \quad &\implies \quad \eqref{strong_disorder}.
}
\end{theirthm}

Determining the behavior of $W_\infty(\beta)$ at $\beta = \bar{\beta}_{\mathrm{c}}$ is an open problem. 
For analogous models on $b$-ary trees, it is known from work of Kahane and Peyri\`ere \cite{kahane-peyriere76} that strong disorder occurs at $\bar{\beta}_{\mathrm{c}}$.
Interestingly, in dimensions $d = 1$ and $d = 2$, there is strong disorder at all finite temperatures (i.e.~$\bar{\beta}_{\mathrm{c}} = 0$), while in higher dimensions weak disorder occurs at sufficiently high temperatures ($\bar{\beta}_{\mathrm{c}} > 0$).
Precise conditions on $\lambda(\beta)$ guaranteeing either behavior can be found in Theorem 2.3.2 of the review \cite{comets-shiga-yoshida04}.
These conditions are in fact a culmination of results from \cite{imbrie-spencer88,bolthausen89,sinai95,albeverio-zhou96,song-zhou96,kifer97,carmona-hu02,comets-shiga-yoshida03}.

Whereas Theorem \ref{disorder_equiv} characterizes the disorder regimes in terms of endpoint localization, one can also attempt to give a characterization based on endpoint diffusion. 
This is usually described by some exponent $\xi = \xi(\mathfrak{L},\beta,d)$ such that ``typical" polymer endpoints are distance $O(n^\xi)$ from the origin.
One way to make this precise is to insist that $\xi$ satisfy
\eq{
\lim_{C\to\infty}\liminf_{n\to\infty} \PP\bigg(C^{-1}n^{\xi} \leq \int \|\omega_n\|_2\ \rho_n(\dd\omega) \leq Cn^{\xi}\bigg) = 1.
}
For instance, Comets and Yoshida  \cite{comets-yoshida06} showed, as part of a more general Brownian central limit theorem, that $\xi = 1/2$ in weak disorder.
That is, the location of the polymer endpoint asymptotically matches that of simple random walk on $\Z^d$.
On the other hand, it is believed (at least in low dimensions, see \cite{piza97}) that the polymer endpoint is superdiffusive in strong disorder.
Specifically, when $d = 1$ (for which any $\beta>0$ yields strong disorder), it is conjectured that $\xi = 2/3$.
Apart from the integrable models discussed in Section \ref{solvability_background}, the only result in this direction is an upper bound due to Piza \cite{piza97} and is conditional on a curvature assumption.
In some related models on $\R$, both upper and lower bounds are known: $3/5\leq \xi \leq 3/4$ \cite{wuthrich98,petermann00,mejane04,comets-yoshida05,comets-yoshida04,bezerra-tindel-viens08}.



A major challenge is to reconcile the notion of high and low temperature with the notion of weak and strong disorder.
The first is more natural from the perspective of statistical physics, 
because phase transitions, in the standard sense, are defined in terms of non-analyticities of limiting free energies. Since the annealed free energy is typically analytic everywhere, the transition of the 
difference $\lambda(\beta)-p(\beta)$ from zero to nonzero denotes the first point of non-analyticity of the limiting free energy $p(\beta)$.
On the other hand, the more abstract conditions of \eqref{weak_disorder} and \eqref{strong_disorder} have lead to a wealth of mathematical results.
It is conjectured (see \cite{comets-yoshida06,carmona-hu06}) that $\beta_{\mathrm{c}} = \bar{\beta}_{\mathrm{c}}$,\footnote{The result $\beta_\cc = \bar\beta_\cc$ is well-known for analogous models on $b$-ary trees.
A nice exposition can be found in \cite[Chapter 4]{comets17}.} 
which would mean
\eq{
0 \leq \beta < \beta_{\mathrm{c}} \quad \implies \quad \eqref{weak_disorder} \qquad \text{and} \qquad
\beta > \beta_{\mathrm{c}} \quad \implies \quad  \eqref{strong_disorder}.
}
Evidence for this belief includes the result by Comets and Vargas \cite[Theorem 1.1]{comets-vargas06} that $\beta_{\mathrm{c}} = 0$ universally in $d = 1$, and the subsequent proof by Lacoin  \cite[Theorem 1.6]{lacoin10} that $\beta_{\mathrm{c}} = 0$ in $d = 2$.\footnote{The $d=1$ result can also be seen from \cite{lacoin10} as a corollary of estimates on the asymptotics of $p(\beta)$ as $\beta\to0$.
These estimates have been subsequently sharpened in \cite{watbled12,nakashima14,alexander-yildirim15,berger-lacoin17}.}
For $d \geq 3$, only the trivial second implication above is known (i.e.~$\beta_{\mathrm{c}} \geq \bar{\beta}_{\mathrm{c}}$).
Indeed, it is clear from \eqref{Fn_lim} and \eqref{Ztilde_def} that $p(\beta)<\lambda(\beta) \implies  \eqref{strong_disorder}$, since
\eq{
 \lambda(\beta) - p(\beta) &= \lambda(\beta) - \lim_{n \to \infty}\frac{1}{n} \log Z_n(\beta)
 = \lambda(\beta) - \lim_{n \to \infty} \frac{1}{n} \log(\e^{n\lambda(\beta)}\, W_n(\beta))
= - \lim_{n \to \infty} \frac{\log W_n(\beta)}{n},
}
and so $p(\beta) < \lambda(\beta)$ implies $\lim_{n\to\infty} W_n(\beta) = 0$.
Therefore, it has become common to say that in the low temperature phase, there is \textit{very strong disorder},
\begin{align} \label{VSD}
 p(\beta) < \lambda(\beta). \tag{VSD}
\end{align}

In recent work, Rassoul-Agha, Sepp\"al\"ainen and Yilmaz \cite{rassoul-seppalainen-yilmaz13,rassoul-seppalainen-yilmaz17} 
expressed $p(\beta)$ in terms of several variational formulas.\footnote{In fact, they did the same for the annealed free energy and then gave a variational characterization of weak disorder, namely the existence of a minimizer to their expression. The minimizer was shown to be unique, equal to $W_\infty(\beta)$, and interestingly \textit{not} bounded away from 0.}
A method for finding minimizers to these formulas was proposed in \cite{georgiou-rassoul-seppalainen16}, and it was noted in \cite{rassoul-seppalainen-yilmaz17} 
that they may not admit minimizers when $\beta > \beta_{\mathrm{c}}$.
These results suggest the possibility of a more general correspondence between the disorder regimes and the existence of variational minimizers (see \cite[Conjecture 2.13]{rassoul-seppalainen-yilmaz17}).

When $d = 1$, a so-called ``intermediate" disorder regime was discovered in \cite{alberts-khanin-quastel14I} by scaling the inverse temperature to $0$ as $n\to\infty$, in such a way that features of both weak and strong disorder are observed in the limit. 
More specifically, when $\beta_n = \beta n^{-1/4}$, the polymer measure exhibits diffusivity ($\xi=1/2$), and the fluctuations of $\log Z_n(\beta)$ are order $1$ as in weak disorder.
The actual fluctuations, however, are not Gaussian, but rather depend on $\beta$ and the random environment, interpolating between Gaussian at $\beta = 0$ and Tracy--Widom at $\beta = \infty$.

\subsubsection{Further versions of localization}

Another central task in the theory of directed polymers is determining if and when weak localization results imply stronger ones.
The following analog to Theorem \ref{disorder_equiv} shows that $\eqref{VSD} \implies \eqref{strong_localization}$, which means the strongest type of localization occurs throughout the low temperature regime.
Therefore, if $\beta_{\mathrm{c}} = \bar{\beta}_{\mathrm{c}}$, then all notions of disorder and localization are equivalent, except possibly at the critical temperature.

\begin{theirthm}[{\cite[Corollary 2.2 and Theorem 2.3(a)]{comets-shiga-yoshida03}}]\label{characterization0} 
Let $P$ be SRW and assume $\lambda(t) < \infty$ for all $t \in \R$.
Then \eqref{VSD} holds if and only if there exists $\eps > 0$ such that
\eq{
\liminf_{n \to \infty} \frac{1}{n} \sum_{i = 0}^{n-1} (\mu_{i}^\beta)^{\otimes 2}(\sigma_i = \sigma_i') \geq \eps \quad \mathrm{a.s.}
}
Equivalently, there exists $\eps > 0$ such that
\eeq{
\liminf_{n \to \infty} \frac{1}{n} \sum_{i = 0}^{n-1}  \max_{x \in \Z^d} \mu_{i}^\beta(\sigma_i = x) \geq \eps \quad \mathrm{a.s.} \label{similar_apa}
}
\end{theirthm}
The above result was generalized to the case of $0 < \beta < B\coloneqq \sup\{t \geq 0 : \lambda(t) < \infty\}$ by Vargas \cite[Theorem 3.6]{vargas07}, who also proved a stronger localization \eqref{eps_atom_delta} for parameters sufficiently close to $B$.
In the spirit of \eqref{similar_apa}, we consider the set of ``$\eps$-atoms",
\eq{
\AA_i^\eps \coloneqq  \{x \in \Z^d : \mu_{i}^\beta(\sigma_i = x) > \eps\}, \quad i \geq 0,\ \eps > 0.
}

\begin{theirthm}[{\cite[Theorem 3.7]{vargas07}}] \label{partial_apa_background}
Let $P$ be SRW.
Assume $B > 0$ and $\lim_{t \to B} \lambda(t)/t = \infty$.
Then for every $\delta < 1$, there exists $\eps > 0$ and $\beta_0 \in (0,B)$ such that
\eeq{ \label{eps_atom_delta}
\beta \in (\beta_0,B) \quad \implies \quad \liminf_{n\to\infty}\frac{1}{n}\sum_{i = 0}^{n-1} \mu_i^\beta(\sigma_i \in \AA_i^\eps)  \geq \delta \quad \mathrm{a.s.}
}
\end{theirthm}

Strong localization \eqref{strong_localization} or the stronger property \eqref{similar_apa} captures the tendency of the endpoint distribution to localize mass when $\beta > \beta_{\mathrm{c}}$. 
In this scenario, it is natural to ask whether the entire mass localizes, or if some positive proportion of the mass remains delocalized.  
Theorem \ref{partial_apa_background} says that the former holds as $\beta\nearrow B$, but we would like to ask the question for \textit{fixed} $\beta$.
In \cite{vargas07}, Vargas proposed the following definition for complete localization, or as Vargas called it, ``asymptotic pure atomicity".
We say that the sequence $(\mu_{i}^\beta(\sigma_i \in \cdot))_{i \geq 0}$ is \textit{asymptotically purely atomic} if for every sequence $(\eps_i)_{i \geq 0}$ tending to 0 as $i \to \infty$, we have
\eeq{ \label{apa_definition}
\frac{1}{n} \sum_{i = 0}^{n-1} \mu_{i}^\beta(\sigma_i \in \AA_i^{\eps_i}) \to 1 \quad \text{in probability, as $n \to \infty$.}
}
The following result was obtained in \cite{vargas07}. 

\begin{theirthm}[{\cite[Theorem 3.2 and Corollary 3.3]{vargas07}}] \label{vargas_apa}
Let $P$ be SRW.
If $\lambda(\beta) = \infty$, then $(\mu_{i}^\beta(\sigma_i \in \cdot))_{i \geq 0}$ is asymptotically purely atomic.
\end{theirthm}

Actually, Vargas showed that $(\mu_{i-1}^\beta(\sigma_i \in \cdot))_{i \geq 1}$ is asymptotically purely atomic, with the set $\AA_i^\eps$ replaced by $\{x \in \Z^d : \mu_{i-1}^\beta(\sigma_i = x) > \eps\}$.  
From observation \eqref{next_step}, it is not difficult to check that the two notions are equivalent; see Corollary \eqref{apa_either_way}.
One of the main results of this chapter, Theorem \ref{total_mass}, asserts that the conclusion of Theorem \ref{vargas_apa} continues to hold even if $\lambda(\beta)$ is finite, as long as $\beta > \beta_{\mathrm{c}}$ and~\eqref{mgf_assumption} holds.\footnote{Furthermore, we show that ``in probability" from \eqref{apa_definition} can be replaced with ``almost surely".} 
In other words, one can take $\beta_0 = \beta_\cc$ in Theorem \ref{partial_apa_background}.
This is discussed in greater detail in Section \ref{endpoint_results} below.

\subsection{Polymers with other reference walks} \label{long_range_review}
A smaller portion of the literature allows for $P$ to be a more general walk.
Let us begin with a motivating family of examples.

\subsubsection{Long-range models}
In \cite{comets07}, Comets initiated the study of \textit{long-range} directed polymers.
In this model, the simple random walk is replaced by a general random walk capable of superdiffusive motion.
In this way, the long-range polymer can model the behavior of 
heavy-tailed walks, such as L\'evy flights, when placed in an inhomogeneous random environment.
Indeed, L\'evy flights in random potentials have been used to study chemical reactions \cite{chen-deem01} and particle dispersions \cite{sokolov-mai-blumen97,brockmann-geisel03}.
Moreover, their continuous-time analogs, L\'evy processes, appear in a variety of disciplines including fluid mechanics, solid state physics, polymer chemistry, and mathematical finance \cite{nielsen-mikosch-resnick01}.

Let us now state the model precisely. 
It is assumed that the walk $P$ belongs to the domain of attraction of an $\alpha$-stable law for some $\alpha \in (0,2]$.
That is, there is $\alpha \in (0,2]$ and deterministic sequences $a_n > 0$, $b_n \in \R^d$ such that
\eeq{ \label{alpha-stable}
\lim_{n \to \infty} E \exp\Big(iz \cdot \frac{\sigma_n - b_n}{a_n}\Big) = E(\e^{iz \cdot S_\alpha}) \quad \text{for all $z \in \R^d$},
}
where $S_\alpha$ is a random variable with a law satisfying the following property.
If $X_1,X_2,\dots,X_n$ are i.i.d.~and share this law, then $(X_1+\cdots+X_n)/n^{1/\alpha}$ also has this law.
When $d = 1$, the condition \eqref{alpha-stable} has a simpler description in terms of the tails of the random walk increment.
Namely, when $\alpha \in (0,2)$ it is assumed that
\eq{
P(|\sigma_1| \geq r) = r^{-\alpha} L(r) \quad \text{for all $r \geq 1$},
}
and for $\alpha = 2$,
\eq{
E\big[|\sigma_1|^2\, \one_{\{|\sigma_1|\leq r\}}\big] = L(r) \quad \text{for all $r \geq 1$},
}
where $L$ is a positive function slowly varying at infinity.\footnote{``Slowly varying" is in the sense of Karamata: For all $s>0$, $L(sx)/L(x)\to1$ as $x\to\infty$.}
Then $a_n = L'(n)n^{1/\alpha}$ for some slowly varying function $L'$, and $b_n$ can be taken equal to zero if $\alpha\in(0,1)$, equal to $n E(\sigma_1)$ if $\alpha\in(1,2]$, or computed if $\alpha=1$ (see \cite[Chapter 7]{gnedenko-kolmogorov54}).
For example, any walk having increments with a finite second moment belongs to the $\alpha=2$ case, where $S_2$ is Gaussian.

\subsubsection{Behaviors at high and low temperature}

For sufficiently high temperatures, the long-range setting described above admits a CLT involving the $\alpha$-stable law (\cite[Theorem 4.2]{comets07} and \cite[Theorem 1.9]{wei16}), which generalizes the Brownian CLT proved in \cite[Theorem 1.2]{comets-yoshida06} when the reference walk is SRW.
Long-range polymers raise further intrigue because these CLTs can hold even in low dimensions (when $\alpha\leq d$), for instance when two independent walks with law $P$ intersect only finitely many times \cite[Theorems 4.1 and 4.2]{comets07}.
This is in contrast to the SRW case, for which $\beta_\cc = 0$ in $d = 1$ \cite[Theorem 1.1]{comets-vargas06} and $d = 2$ \cite[Theorem 1.6]{lacoin10}.

On the other hand, \cite[Corollary 6.3]{comets07} and its generalization \cite[Theorem 1.17]{wei16} show that Theorem \ref{characterization0} continues to hold for long-range polymers, suggesting that universal behaviors also appear at low temperatures.
Indeed, part of the work in \cite{comets07,wei16} was to extend localization results known for polymers constructed from a SRW to those constructed with $\alpha$-stable reference walks.
This chapter continues the progress in this direction by proving that in the low temperature regime, certain qualitative behaviors of the polymer's endpoint distribution are the same for \textit{any} reference walk.

We conclude this section by mentioning a simple sufficient condition for the low temperature phase.
Notice that the condition depends only on the entropy of the random walk increment, not its actual distribution.
One can therefore interpret the result as providing a temperature threshold below which concentration of the polymer measure is strong enough to overcome any superdiffusivity of the random walk.

\begin{theirthm}[{\cite[Proposition 5.1]{comets07}}] \label{localization_sufficient_background}
Assume $\lambda(\beta) < \infty$, and also that $q(x) \coloneqq P(\sigma_1 = x)$ has finite entropy.
If
\eq{
\beta \lambda'(\beta) - \lambda(\beta) > -\sum_{x \in \Z^d} q(x)\log q(x),
}
then $p(\beta) < \lambda(\beta)$.
\end{theirthm}

\subsubsection{Further analogies with SRW case}
Here we list several other properties of directed polymers constructed from SRW that have been proved to extend to long-range cases:
\begin{itemize}
\item The method of identifying a KPZ regime initiated in \cite{alberts-khanin-quastel14I,alberts-khanin-quastel14II} was extended to certain long-range cases in \cite{caravenna-sun-zygouras17I,caravenna-sun-zygouras17II}.
In fact, the authors of \cite{caravenna-sun-zygouras17II} are able to identify a universal limit for the point-to-point log partition functions, in critical cases, for both $d= 1$ and $d=2$.
\item In \cite{comets-fukushima-nakajima-yoshida15}, the asymptotics of $p(\beta)$ as $\beta\to\pm\infty$ are derived when $P$ has a stretched exponential tail, and the environment consists of Bernoulli random variables.
\item Recall that for the SRW case, it was proved that $\beta_\cc = 0$ by Comets and Vargas for $d = 1$ \cite{comets-vargas06} and by Lacoin for $d = 2$ \cite{lacoin10}.
For the long-range $\alpha$-stable case with $\alpha \in (1,2]$ and $d = 1$, Wei showed the analogous result: $\beta_\cc = \bar \beta_\cc = 0$ \cite[Theorem 1.11(ii)]{wei16}.\footnote{On the other hand, for $d \geq 3$ and any truly $d$-dimensional random walk $P$, the transience of $P$ forces $\bar\beta_\cc > 0$.
Also when $d = 1$ and $\alpha \in (0,1)$, or $d = 2$ and $\alpha \in (0,2)$, transience implies $\bar\beta_\cc > 0$, see \cite[Theorem 4.1]{comets07}.}
Later, Wei showed \cite[Theorem 1.3]{wei18} that in the critical case $\alpha = d = 1$ and under some regularity assumptions on $P$, $\bar\beta_\cc = 0$ again implies $\beta_\cc = 0$.
\end{itemize}

\subsubsection{Continuous versions of long-range polymers}
In \cite{miura-tawara-tsuchida08}, a continuous version of $\alpha$-stable long-range polymers is considered.
Specifically, a phase transition was shown for the normalized partition function associated to a L\'evy process subjected to a Poissonian random environment. 
Sufficient conditions were given for either side of the transition.
In the same way that the $\alpha$-stable polymer introduced in \cite{comets07} generalized classical lattice polymers, this L\'evy process model generalized a Brownian motion in Poissonian environment, which was introduced in \cite{comets-yoshida05} and also considered in \cite{lacoin11,comets-yoshida13,comets-cosco??,fukushima-junk19}.

\subsection{Related results in integrable models} \label{solvability_background}
Until recently, no directed polymer model offered the possibility of exact asymptotic calculations.
This was in contrast to the integrable last passage percolation (LPP) models in $1+1$ dimensions, for which specific choices of the passage time distribution (namely geometric or exponential) allow one to use representation theory to derive exact formulas for the distribution of passage times, which are the LPP analogs of partition functions.
Such was the approach advanced in the seminal work of Johansson \cite{johansson00}, which showed that asymptotic fluctuations of passage times follow the Tracy--Widom distributions at scale $n^{1/3}$---GUE for point-to-point passage times, and GOE for the point-to-line passage time.
The connection to Tracy--Widom laws can be seen more explicitly in a model known as Brownian LPP \cite{oconnell03}, which also exhibits the $n^{1/3}$ scaling.
These results place the LPP models within the KPZ universality class (see \cite{corwin12, borodin-petrov14}). 
For a review of related works, including results verifying spatial fluctuations at scale $n^{2/3}$, we refer the reader to \cite{quastel-remenik14} and references therein.

In the zero-temperature limit $\beta \to \infty$, the polymer measure $\mu_n^\beta$ concentrates on the path that is most likely given the random environment.
When $\beta = \infty$, the directed polymer model is equivalent to LPP.
The natural conjecture, therefore, is that directed polymers in strong disorder obey the same KPZ scaling relations.
In particular, models in $1+1$ dimensions should have energy fluctuations of order $n^{1/3}$, and endpoint fluctuations of order $n^{2/3}$. 
The first case permitting exact calculations was the integrable log-gamma model introduced by Sepp{\"a}l{\"a}inen \cite{seppalainen12}, whose breakthrough work proved the conjectured exponents.
Subsequent studies showed that free energy fluctuations were again of Tracy--Widom type \cite{corwin-oconnell-seppalainen-zygouras14,borodin-corwin-remenik13}, gave a large deviation rate function \cite{georgiou-seppalainen13}, and computed the limiting value $p(\beta)$ \cite{georgiou-rassoul-seppalainen-yilmaz15}.

Regarding spatial fluctuations, Comets and Nguyen \cite{comets-nguyen16} found an explicit limiting endpoint distribution in the equilibrium case of the log-gamma model.
More precisely, they showed that the endpoint distribution, when re-centered around the most likely site, converges in law to a certain random distribution on $\Z$. 
While the boundary conditions they enforce prevent the correct exponent of $2/3$ from appearing, which would mean fluctuations of the averaged quenched distribution occur on the order of $n^{2/3}$, their result implies that the \textit{quenched} endpoint distributions concentrate all their mass in a microscopic region of $O(1)$ diameter around a single favorite site.
In the notation of this chapter, if $x_n = \arg \max \mu_n^\beta(\sigma_n = \cdot)$, then the result of \cite{comets-nguyen16} says
\eeq{
\lim_{K \to \infty} \liminf_{n \to \infty} \mu_n^\beta(\|\sigma_n - x_n\|_1 \leq K) = 1 \quad \text{in probability.}\label{mode}
}
This provides an affirmative case of the so-called ``favorite region" conjecture, which speculates that in strong disorder, directed polymers on the lattice localize their endpoint to a region of fixed diameter. 
We emphasize once more that this is a statement about quenched endpoint distributions.
While little is known on annealed distributions in the general case, 
there has been considerable progress for the integrable LPP models. 
Johansson \cite{johansson03} proved that the rescaled location of the endpoint converges in distribution to the maximal point of an Airy$_2$ process minus a parabola.
More recent work has expressed the density of this limit in terms of Fredholm determinants \cite{flores-quastel-remenik13}, or from a different approach, in terms of the Hastings-McLeod solution to the Painlev\'{e} II equation \cite{schehr12}; the two formulas were shown to be equivalent in \cite{baik-liechty-schehr12}.
Tail estimates for this density can be found in \cite{bothner-liechty13,quastel-remenik15}.

For further literature on integrable directed polymers, see \cite{corwin-seppalainen-shen15,barraquand-corwin17,thiery-doussal15} and references therein.
We also mention the semi-discrete model introduced by O'Connell and Yor \cite{oconnell-yor01}, which is a positive-temperature version of Brownian LPP and has offered another provable case of $n^{1/3}$ energy fluctuations \cite{seppalainen-valko10,borodin-corwin14,borodin-corwin-ferrari14,talyigas-veto19}. 


\subsection{A remark about the definition of endpoint distribution}
All the results of Section \ref{disorder_background} are normally stated in the literature using the measure $\mu_{n-1}^\beta(\sigma_n \in \cdot)$, as opposed to $\mu_n^\beta(\sigma_n \in \cdot)$. 
The reason is that the former arises naturally out of the martingale analysis for $W_n(\beta)$ (for instance, see the proof of \cite[Theorem~3.3.1]{comets-shiga-yoshida04}), but in all cases, the statements are equivalent when the latter is used instead.
This equivalence can be seen by writing
\eeq{
\mu_{n-1}^\beta(\sigma_n = x) =  \sum_{y\in\Z^d} \mu_{n-1}^\beta(\sigma_{n-1} = y)P(y,x), \label{next_step}
}
(where $\|\cdot\|_1$ denotes the $\ell^1$ norm on $\Z^d$), which implies
\eq{
\Big(\max_{y\in\Z^d} P(y,x) \Big) \max_{y \in \Z^d} \mu_{n-1}^\beta(\sigma_{n-1} = y) \leq \max_{x \in \Z^d} \mu_{n-1}^\beta(\sigma_n = x) \leq \max_{y \in \Z^d} \mu_{n-1}^\beta(\sigma_{n-1} = y).
}
Similarly, we have
\eq{
\Big(\max_{y\in\Z^d} P(y,x)^2 \Big) (\mu_{n-1}^\beta)^{\otimes 2}(\sigma_{n-1} = \sigma_{n-1}') \leq (\mu_{n-1}^\beta)^{\otimes 2}(\sigma_{n} = \sigma_{n}') \leq (\mu_{n-1}^\beta)^{\otimes 2}(\sigma_{n-1} = \sigma_{n-1}'),
}
where the first inequality follows from the observation that if $\sigma_{n-1} = \sigma_{n-1}' = y$, then $\sigma_n = \sigma_n'$ with probability $\sum_{y\in\Z^d} P(y,x)^2$ under $\mu_{n-1}^\beta$.
The second inequality is due to Cauchy--Schwarz applied to \eqref{next_step}:
\eq{
(\mu_{n-1}^\beta)^{\otimes 2}(\sigma_n = \sigma_n') 
= \sum_{x \in \Z^d} \mu_{n-1}^\beta(\sigma_n = x)^2
&= \sum_{x\in\Z^d}\bigg(\sum_{y\in\Z^d}\mu_{n-1}^\beta(\sigma_{n-1}=y)P(y,x)\bigg)^2 \\
&\leq \sum_{x \in \Z^d} \sum_{y\in\Z^d} \mu_{n-1}^\beta(\sigma_{n-1} = y)^2P(y,x)\\
&= \sum_{y \in \Z^d} \mu_{n-1}^\beta(\sigma_{n-1} = y)^2
= (\mu_{n-1}^\beta)^{\otimes 2}(\sigma_{n-1}=\sigma_{n-1}').
}
As $\mu_n^\beta(\sigma_n \in \cdot)$ will be the more natural object for our purposes, we reserve the term ``endpoint distribution" for this measure.

\subsection{An overview of results proved in this chapter} \label{endpoint_results}
The endpoint distribution $\mu_n^\beta(\sigma_n \in \cdot)$ of length-$n$ polymers is a random probability measure on $\mathbb{Z}^d$. The main goal of this chapter is to understand the behavior of this object as $n\to\infty$.
The majority of our work is in 
developing methods
to compute limits of endpoint distributions, or more to the point, of \textit{functionals} of the endpoint distribution.
While proving the existence of these limits is presently out of reach, by instead considering empirical averages we are able to establish results in Ces\`aro form.
Broadly speaking, the construction consists of three components: (i) a compactification of measures on $\Z^d$; (ii) a Markov kernel whose invariant measures are the possible limits of the endpoint distribution; and (iii) a functional that distinguishes those invariant measures with minimal energy.
In fact, our approach has inspired similar works by Br{\"o}ker and Mukherjee for the stochastic heat equation and Gaussian multiplicative chaos \cite{broker-mukherjee19II}, and by Bakhtin and Seo \cite{bakhtin-seo20} for directed polymers with a continuous-space reference walk.

\subsubsection{Compactification of measures on $\Z^d$}
One key obstacle to studying the endpoint distribution is that the standard topology of weak convergence of probability measures is inadequate for understanding its limiting behavior. Even the weaker topology of vague convergence does not provide an adequate description of the localization phenomenon that is of central interest.
In the recent work of Mukherjee and Varadhan \cite{mukherjee-varadhan16}, this issue was addressed by exploiting translation invariance inherent in their problem to pass to a certain compactification of probability measures on $\R^d$.
We have drawn inspiration from their methods, using a similar construction we will soon describe.
Ultimately, this approach is enabled by the ``concentration compactness" phenomenon.
Before we discuss this topic further, let us motivate the discussion with two elementary examples.

Roughly speaking, the difficulty of using standard weak or vague convergence is that the endpoint distribution may manifest itself as multiple localized ``blobs'' that escape to infinity in different directions as $n\to \infty$. Additionally, some part of the mass may just ``diffuse to zero''. As a first example, consider the following sequence of probability mass functions on $\mathbb{Z}$:
\[
q_n(x) \coloneqq 
\begin{cases}
1/2 &\text{if $x=n$,}\\
1/(2n) &\text{if $0\le x< n$,}\\
0 &\text{otherwise.}
\end{cases}
\]
This sequence fails to converge weakly, and its vague limit is the zero measure. 
However, to understand the true limiting behavior of $q_n$, it seems more appropriate to take the vague limit after translating the measure by $n$; that is, taking $\wt{q}_n(x) \coloneqq  q_n(x-n)$. 
The limit of the sequence $\wt{q}_n$ is the subprobability measure that puts mass $1/2$ at the origin and zero elsewhere, which gives a better picture of the true limiting behavior of $q_n$. 

The situation is more complicated when multiple blobs escape to infinity in different directions. For example, consider the following sequence of probability mass functions on $\mathbb{Z}$:
\[
r_n(x) \coloneqq 
\begin{cases}
1/5 &\text{if $x \in \{-2n, -n, n, n+1\}$,}\\
1/(5n) &\text{if $0\le x<n$,}\\
0 &\text{otherwise.}
\end{cases}
\]
Like $q_n$, the sequence $r_n$ also converges vaguely 
to the subprobability measure of mass zero. But in this case, no sequence of translates of $r_n$ can fully capture its limiting behavior, because a translate can only capture the mass at one of the three blobs but not all of them simultaneously.  The only recourse, it seems, is to express the limit not as a single subprobability measure, but a union of three subprobability measures on three distinct copies of $\mathbb{Z}$. Two of these will put mass $1/5$ at the origin in their respective copies of $\Z$, and the third will put mass $1/5$ at $0$ and $1/5$ at $1$ in its copy of $\Z$. 
The first two are the vague 
limits of $r_n(\cdot +2n)$ and $r_n(\cdot + n)$, and the third is the vague 
limit of $r_n(\cdot - n)$. One can view this object jointly as a subprobability measure of total mass $4/5$ on the set $\{1,2,3\}\times \Z$. Of course, it is not important which copy of $\Z$ gets which part of the measure, nor does it matter if a translation is applied to the subprobability measure on one of the copies. Thus, the limit object is an equivalence class of subprobability measures on $\{1,2,3\}\times \Z$ rather than a single subprobability measure.

Generalizing the above idea, we can consider equivalence classes of subprobability measures on $\N \times \Z^d$. First, define the $\N$-support of a subprobability measure $f$ on $\N\times \Z^d$ to be the set of all $n\in \N$ such that $f$ puts nonzero mass on $\{n\}\times \Z^d$. Next, declare two subprobability measures $f$ and $g$ on $\N \times \Z^d$ to be equivalent if there is a bijection $\tau$ between their $\N$-supports and a number $x_n\in \Z^d$ for each $n$ in the $\N$-support of $f$ such that 
\[
g(\tau(n), \cdot) = f(n, x_n +\cdot).
\]
It is easy to verify that this is indeed an equivalence relation. 
We call the equivalence classes \textit{partitioned subprobability measures} on $\Z^d$, and the set of all equivalence classes is denoted by $\SS$. The set $\SS$ is formally defined and studied in Section \ref{topology}.

We will view probability measures on $\Z^d$ (or more precisely, their mass functions) as elements of $\SS$ supported on a single copy of $\Z^d$.
In particular, we have a natural embedding of endpoint distributions into $\SS$.
In Section \ref{topology-definitions}, we define a metric $d_\alpha$ on $\SS$ (not to be confused with the dimension $d$).
Due to the somewhat complicated nature of the metric, we will not reproduce the definition in this overview section, but it is constructed to ensure that certain functionals on $\SS$ are continuous or at least semi-continuous.
In this study, these functionals are related to either localization or the free energy of the system.
One of the first nontrivial results of this chapter is that $(\SS, d_\alpha)$ is a compact metric space, in analogy with the construction in \cite{mukherjee-varadhan16}. 
This is Theorem~\ref{compactness_result} in Section~\ref{compactness}.

In summary, the metric $d_\alpha$ imposes a translation invariance that prevents so-called blobs from escaping to infinity.
Moreover, such escape turns out to be the only obstruction to compactness.
This is the fundamental observation of a more general theory known as ``concentration compactness", which builds on the notion of the concentration functions of probability measures, defined by L\'evy~\cite{levy54}.  The idea of using L\'evy's concentration functions for constructing compactifications of function spaces was introduced by Parthasarathy, Ranga Rao, and Varadhan~\cite{parthasarathy-rangarao-varadhan62}, and  developed into a powerful tool by Lions \cite{lions84I,lions84II,lions85I,lions85II}. It has been applied broadly in the study of calculus of variations and PDEs. 
In the language of those settings, the equivalence relation described above plays the role of a ``profile decomposition" (for a friendly discussion, see \cite{tao08,tao10}).
In the present setting, this decomposition states that 
In other words, within any sequence of probability mass functions on $\Z^d$, one can find a subsequence $s_n$ for which there exists a countable collection $(p_j)_{j \in \N}$ of fixed \textit{sub}probability mass functions, and sequences $x_{j,n}$ in $\Z^d$ such that the following three statements hold:
\eq{
\sum_{j \in \N} \sum_{x \in \Z^d} p_j(x) \leq 1, \quad
\lim_{n \to \infty} \|x_{j,n} - x_{k,n}\|_1 = \infty \quad \text{if $j \neq k$}, \quad
&\lim_{n \to \infty} \Big\|s_n(\cdot) - \sum_{j \in \N} p_j(\cdot - x_{j,n})\Big\|_\infty = 0.
}
The $p_j$ are called \textit{profiles} and play the role of the vague limits in the examples given above.
In the first example, there was only one nonzero profile, while in the second there were three.
Our motivation for defining the set $\SS$, then, is to record these profiles as a single function.

As mentioned before, a similar construction in a continuous setting has appeared recently in the paper~\cite{mukherjee-varadhan16}.
Specifically, its authors first identify probability distributions on $\R^d$ that are translations of one another (that is, they define an equivalence relation), and then embed the resulting quotient space within a set (which is analogous to $\SS$) of countable sequences of subprobability measures on $\R^d$, again identified by translations.
Upon defining a metric on this set under which the image of the embedding is dense, they prove that the resulting topology is compact.
This fact is used to establish a large deviation result for occupation measures on $\R^d$ induced by Brownian motion.
In comparison, our construction is built to study the endpoint measures on $\Z^d$ induced by directed polymers, and we follow a similar program to obtain a compact topology on the space of interest: the set $\SS$, in our case.
One key difference is that a measure on $\Z^d$ can be thought of as a function on $\Z^d$, and so we can define a metric on $\SS$ permitting explicit calculations. 
This bypasses the use of an abstract family of test functions and allows us to directly check the continuity of certain functionals on (equivalence classes of) measures.

After defining the metric space $(\SS,d_\alpha)$ and proving its compactness, we present in Section \ref{compare_topologies} a formal comparison between our construction and the one of Mukherjee and Varadhan.
In particular, we first introduce the relevant ideas from their paper \cite{mukherjee-varadhan16} for measures on $\Z^d$, and check that in this discrete setting one can still define a metric $D$ on $\SS$ by using translation invariant test functions.
Heuristics and motivating examples suggest correctly that this alternative metric $D$ is equivalent to $d_\alpha$;
the equivalence is stated as Proposition \ref{equal_topologies}, in terms of convergent sequences.
This result is satisfying in that it shows a consistency of theory that might be generalized.
Nevertheless, its proof is quite involved.
In fact, our argument requires the compactness of $d_\alpha$ to even establish the equivalence of topologies, meaning a separate proof that $(\SS,D)$ is compact would not be sufficient to show that $(\SS,d_\alpha)$ is compact.
In any case, the entirety of Section \ref{compare_topologies} is supplementary and tangential to the subject of directed polymers.
The reader may choose to skip this part or return to it after reading later sections.

\subsubsection{The update map}
In Section \ref{transformation} we resume developing our abstract machinery by defining 
an ``update'' map $T:\SS \to \PP(\SS)$,  where $\PP(\SS)$ is the set of all probability measures on $\SS$ equipped with the Kantorovich--Rubinstein--Wasserstein  distance  $\WW$ (the $L^1$ version). 
The map $T$ has the property that if $f_n(\cdot) \coloneqq  \mu_n^\beta(\sigma_n = \cdot)$ is the endpoint mass function of the length-$n$ polymers (considered as an element of $\SS$ supported on a single copy of $\Z^d$), then $Tf_n$ is the law of $f_{n+1}$ given $\FF_n$ (recall the definition \eqref{fndef} of $\FF_n$).
In fact, it is not difficult to see that $(f_n)_{n\geq0}$ is Markovian, and thus $T$ is the Markov kernel generating this chain.
Considerable work is done in Section \ref{endpoint_distributions} to show that $T$ is a continuous map (i.e.~has the Feller property).  
This is the conclusion of Proposition \ref{continuous1}. The explicit nature of the metric $d_\alpha$ is particularly important in the proof of this result.
Finally, the map $T$ lifts to a map $\TT: \PP(\SS) \to \PP(\SS)$, defined as
\eq{
\TT\rho(\dd g) \coloneqq  \int Tf(\dd g)\ \rho(\dd f).
}
The continuity of $T$ implies that $\TT$ is also continuous.

In Section \ref{free_energy} we study the following random element of $\PP(\SS)$:
\eq{
\rho_n \coloneqq  \frac{1}{n}\sum_{i = 0}^{n-1} \delta_{f_i}.
}
Here $\delta_{f_i}$ is the unit point mass at the $i^{\text{th}}$ endpoint mass function $f_i$, considered as an element of $\SS$ as before. In words, $\rho_n$ is the empirical measure of the endpoint distributions up to time $n$. Let 
\[
\KK \coloneqq  \{\rho \in \PP(\SS) : \TT\rho = \rho\}
\]
be the set of fixed points of $\TT: \PP(\SS) \to \PP(\SS)$.
The first main result of Section \ref{free_energy} is Corollary \ref{close_probability}, which says that
\eeq{
\lim_{n \to \infty} \inf_{\rho \in \KK} \WW(\rho_n,\rho) = 0 \quad \mathrm{a.s.} \label{going_to_K}
}
This result provides a heuristic connection to $(1+1)$-dimensional integrable models (for instance, see \cite{seppalainen12,corwin-seppalainen-shen15,barraquand-corwin17,thiery-doussal15}), which work in part by identifying a disorder distribution and boundary conditions such that the system is stationary under spatial translations.
Loosely speaking, our approach similarly recovers a stationarity property \textit{in the limit}, even without explicit calculations.
In this way, our methods replace this key feature of integrable models with a much weaker, but more general, abstract framework.
This is enabled by a topology on endpoint distributions that is rich enough to capture the desired localization, yet sufficiently ``compressed" to be compact.

\subsubsection{The energy functional}
Given the convergence \eqref{going_to_K}, the next key observation is that the free energy $F_n(\beta)$ can be expressed in terms of the empirical measure $\rho_n$, see \eqref{Fn_ito_empirical}.
We thus define a functional $\RR : \PP(\SS) \to \R$ so that we may concisely write
$\E F_n(\beta) = \E \RR(\rho_{n})$.
Proposition \ref{free_energy_converges} and Corollary \ref{close_probability} then lead to the following variational formula for the limiting free energy, given in Theorem \ref{upper_bound}. For any $\beta$ such that \eqref{mgf_assumption} holds,
\eq{
p(\beta) = \lim_{n\to\infty} F_n(\beta) = \lim_{n \to \infty} \E F_n(\beta) = \inf_{\rho \in \KK} \RR(\rho).
}
The connections between this formula and those of \cite{rassoul-seppalainen-yilmaz13,rassoul-seppalainen-yilmaz17} are unclear.

Nevertheless, this computation allows us to improve Corollary \ref{close_probability} in an important way to yield Theorem \ref{close_M}, which says that if 
\[
\MM \coloneqq  \Big\{\rho_0 \in \KK : \RR(\rho_0) = \inf_{\rho \in \KK} \RR(\rho)\Big\},
\]
then
\eq{
\lim_{n \to \infty} \inf_{\rho \in \MM} \WW(\rho_n,\rho) = 0 \quad \mathrm{a.s.}
}
In Section \ref{empirical_limits} we study the minimizing set $\MM$. In particular, Theorem \ref{characterization} says that either $\MM$ consists of the single element of total mass zero (which happens when $0\le \beta\le \beta_{\mathrm{c}}$, where $\beta_{\mathrm{c}}$ is the critical inverse temperature of Theorem \ref{critical_temperature}), or every element of $\MM$ has total mass one (which happens when $\beta>\beta_{\mathrm{c}}$).
This result is similar to the technique of identifying a phase transition as the critical point at which a recursive distributional equation begins to have a nontrivial solution; for an account of this method, we refer the reader to the survey of Aldous and Bandyopadhyay~\cite{aldous-bandyopadhyay05} and references therein.
This idea is also present in work of Yoshida \cite{yoshida08} on more general linear stochastic evolutions, although there the nontrivial solutions exist in the high temperature regime rather than the low temperature regime.

From a different perspective, the limit law of the empirical measure can be viewed as an ``order parameter'' for the model, whose behavior distinguishes between the high and low temperature regimes. Such order parameters arise frequently in the study of disordered systems. A prominent example is the Sherrington--Kirkpatrick (SK) model of spin glasses, where the limiting distribution of the overlap serves as the order parameter (see Panchenko~\cite{panchenko13}). Interestingly, the limiting free energy of the SK model can also be expressed as the solution of a variational problem involving the order parameter. This is the famous Parisi formula proved by Talagrand~\cite{talagrand06}. 
In this way, the partitioned subprobability measures might be seen as counterparts to the random overlap structures introduced by Aizenman, Sims and Starr~\cite{ass07}, 
and the update map as the analog to similar stabilizing maps arising out of the cavity method for spin glasses and related competing particle systems, that were studied by Aizenman and Contucci~\cite{aizenman-contucci98}, Ruzmaikina and Aizenman~\cite{ruzmaikina-aizenman05} and Arguin and Chatterjee~\cite{arguin-chatterjee13}.
In other ways, however, the analogy is quite distant.
For instance, the spin glass systems we speak of are mean-field models lacking any geometry from the lattice.
Also, our variational formula relies on the directed nature of the problem; that is, the random environment refreshes at each time step, allowing us to exploit Markovian structure.

\subsection{Main applications}
Theorem \ref{characterization} yields the following concrete application of our abstract theory of partitioned subprobability measures; it is later stated as Theorem \ref{total_mass}.
Recall the notations and terminologies related to Theorem~\ref{critical_temperature} and Theorem~\ref{vargas_apa}. 

\begin{thm} \label{intro_result1}
Assume \eqref{walk_assumption_1} and \eqref{mgf_assumption}.
\begin{itemize}
\item[(a)] If $p(\beta) < \lambda(\beta)$, then for every sequence $(\eps_i)_{i \geq 0}$ tending to 0 as $i \to \infty$,
\eq{
\lim_{n \to \infty} \frac{1}{n} \sum_{i = 0}^{n-1} \mu_{i}^\beta(\sigma_i \in \AA_i^{\eps_i}) = 1 \quad \mathrm{a.s.}
}
\item[(b)] If $p(\beta) = \lambda(\beta)$, then there exists a sequence $(\eps_i)_{i \geq 0}$ tending to 0 as $i \to \infty$ such that
\eq{
\lim_{n \to \infty} \frac{1}{n} \sum_{i = 0}^{n-1} \mu_{i}^\beta(\sigma_i \in \AA_i^{\eps_i}) = 0 \quad \mathrm{a.s.}
}
\end{itemize}
\end{thm}

This generalizes Theorem \ref{vargas_apa}, where an ``in probability" version of (a) was proved under the condition $\lambda(\beta)=\infty$. 

In Section \ref{main_thm2} we apply our techniques to go further than atomic localization by considering ``geometric localization".
In the low temperature phase, the endpoint distribution can not only concentrate mass on a few likely sites, but moreover have those sites close together.
We make this phenomenon precise in the following manner.
For $\delta \in (0,1)$ and a nonnegative number $K$, let $\GG_{\delta,K}$ denote the set of probability mass functions on $\Z^d$ that assign measure greater than $1-\delta$ to some subset of $\Z^d$ having diameter at most $K$; see \eqref{g_set_def} for a symbolic definition.
We will say that the sequence $(\mu_i^\beta(\sigma_i \in \cdot))_{i \geq 0}$ is \textit{geometrically localized with positive density} if for every $\delta$, there is $K< \infty$ and $\theta > 0$ such that
\eq{
\liminf_{n \to \infty} \frac{1}{n} \sum_{i = 0}^{n-1} \one_{\{\mu_i^\beta(\sigma_i = \cdot)\in \GG_{\delta,K}\}} \geq \theta \quad \mathrm{a.s.,}
}
where $\one_{A}$ denotes the indicator of the event $A$.
That is, there are endpoint distributions with limiting density at least $\theta$ that place mass greater than $1 - \delta$ on a set of bounded diameter.
We will say $(\mu_i^\beta(\sigma_i \in \cdot))_{i \geq 0}$ is \textit{geometrically localized with full density} if for every $\delta$, there is $K < \infty$ such that
\eeq{
\liminf_{n \to \infty} \frac{1}{n} \sum_{i = 0}^{n-1} \one_{\{\mu_i^\beta(\sigma_i = \cdot)\in \GG_{\delta,K}\}} \geq 1-\delta \quad \mathrm{a.s.} \label{geometric_localization}
}
The main result of Section \ref{main_thm2} is contained in Theorem \ref{localized_subsequence} and says the following.

\begin{thm} \label{intro_result2}
Assume \eqref{walk_assumption_1} and \eqref{mgf_assumption}. 
\begin{itemize} 
\item[(a)] If $p(\beta) < \lambda(\beta)$, then there is geometric localization with positive density. 
Moreover, the numbers $K$ and $\theta$ are deterministic quantities that depend only on the choice of $\delta$, the disorder distribution $\mathfrak{L}_\omega$, the parameter $\beta$, and the dimension $d$.
\item[(b)] If $p(\beta) = \lambda(\beta)$, then for any $K$ and any $\delta \in (0,1)$,
\eq{
\lim_{n\to\infty} \frac{1}{n} \sum_{i=0}^{n-1} \one_{\{\mu_i^\beta(\sigma_i = \cdot) \in \GG_{\delta,K}\}} = 0 \quad \mathrm{a.s.}
}
\end{itemize}
\end{thm}

As mentioned in Section \ref{solvability_background}, the only case where a version of geometric localization has been proved is an integrable $(1+1)$-dimensional model, for which Comets and Nguyen \cite{comets-nguyen16} proved localization and moreover computed the limit distribution of the endpoint. Similar results for one-dimensional random walk in random environment were proved by Sinai~\cite{sinai82}, Golosov~\cite{golosov84}, and Gantert, Peres and Shi~\cite{gantert-peres-shi10}.

\subsubsection*{The single-copy condition}
In addition to the above unconditional results, we also prove a few conditional statements, which hold under the condition that every $\rho \in \MM$ puts all its mass on those $f \in \SS$ that are supported on a single copy of $\Z^d$. We call this the ``single-copy condition''.
One consequence of the single-copy condition is geometric localization with full density, as defined in \eqref{geometric_localization}. Part (b) of Theorem \ref{localized_subsequence} proves this conditional claim. 

A second consequence of the single-copy condition is Proposition \ref{localization_thm}, which gives the following Ces\`aro form of \eqref{mode}. For each $i\ge 0$ and $K\ge 0$, let $\CC_{i}^K$ be the set of all $x\in \Z^d$ that are at distance $\le K$ from {\it every} mode of the endpoint mass function $\mu_i^\beta(\sigma_i = \cdot)$.\footnote{To relate this definition back to geometric localization, note that $\CC_i^K$ has diameter $\le 2K$.} 
Then, assuming \eqref{walk_assumption_1}, \eqref{mgf_assumption}, and the single-copy condition, Proposition~\ref{localization_thm} asserts that 
\eq{
\lim_{K \to \infty} \liminf_{n \to \infty} \frac{1}{n} \sum_{i = 0}^{n-1} \mu_i^\beta(\sigma_i\in \CC_{i}^K) = 1 \quad \mathrm{a.s.}
}
That is, endpoint localization is tight after suitable translations. So if the additional hypothesis mentioned in the previous paragraph is always true in the low temperature phase, this would solve in a Ces\`aro sense the longstanding ``favorite region" conjecture.

In view of the result \eqref{mode} of Comets and Nguyen~\cite{comets-nguyen16}, it seems plausible that the single-copy condition holds for the log-gamma polymer in $1+1$ dimensions. 
Unfortunately, we have been unable to determine whether or not the single-copy condition holds in general. 
Furthermore, we are not aware of any conjectures on what is true in higher dimensions.
The results of \cite{barral-rhodes-vargas12} suggest that at least for directed polymers on $b$-ary trees, the single-copy condition does not hold, and full geometric localization is not valid. 
This may be related to the fluctuations of $\log Z_n(\beta)$, which are known to be order $1$ on the tree (see \cite{derrida-spohn88}) but conjectured to be order $n^{1/3}$ when $d=1$.


\section{Free energy and phase transition} \label{free_energy_background}
In order to give context for the main results, which concern the behavior of polymer measures above and below a phase transition, we must first check that such a phase transition exists!
To do so, we need to prove \eqref{Fn_lim} and Theorem \ref{critical_temperature} in the general setting.
First, in Section \ref{convergence_of_free_energy} we show that the quenched free energy has a deterministic limit. 
The arguments used here are standard, and the expert reader may skip them; nevertheless, the details are included to verify that no essential facts are lost when working with an arbitrary reference walk.
Next, a proof of the phase transition is given in Section \ref{existence_of_phase_transition}.
In particular, the methods initiated in \cite{comets-yoshida06} must be refined to account for general $P$ and weaker assumptions on the logarithmic moment generating function $\lambda(\beta)$.

\subsection{Convergence of free energy} \label{convergence_of_free_energy}
We begin by showing the existence of a limiting free energy.
Throughout this section one may assume a condition just slightly weaker than \eqref{mgf_assumption}, namely
\eeq{ \label{weaker_mgf_assumption}
\lambda(\pm\beta) < \infty.
}

\begin{prop} \label{free_energy_converges}
Assume \eqref{weaker_mgf_assumption}.
Then the limiting free energy exists and is deterministic:
\eeq{ \label{Fn_lim_general}
\lim_{n \to \infty} F_n(\beta) = p(\beta) \quad \mathrm{a.s.} \text{ and in } L^p,\, p \in [1,\infty).
}
\end{prop}
The proof follows the usual program of showing first that $\E F_n(\beta)$ converges, and second that $F_n(\beta)$ concentrates around its mean.

\begin{lemma} \label{means_converge}
Assume \eqref{weaker_mgf_assumption}.
Then
\eeq{
\lim_{n \to \infty} \E F_n(\beta) = \sup_{n \geq 0} \E F_n(\beta) \eqqcolon p(\beta) \leq \lambda(\beta). \label{p_def}
}
\end{lemma}

\begin{lemma}[{cf.~\cite[Theorem 1.4]{liu-watbled09}}] \label{concentration}
Assume \eqref{weaker_mgf_assumption}.
Then
\eeq{ \label{Fn_tail_estimates}
\P(|F_n(\beta) - \E F_n(\beta)| > x) \leq \begin{cases}
2\e^{-ncx^2} &\mathrm{if~} x \in [0,1], \\
2\e^{-ncx} &\mathrm{if~} x > 1, \end{cases}
}
where $c > 0$ is a constant depending only on the value of $K \coloneqq 2\e^{\lambda(\beta)+\lambda(-\beta)}$.
In particular,
\eq{ 
\lim_{n \to \infty} |F_n(\beta) - \E F_n(\beta)| = 0  \quad \mathrm{a.s.}\text{ and in } L^p,\, p \in [1,\infty).
}
\end{lemma}

Proposition \ref{free_energy_converges} is immediate given Lemmas \ref{means_converge} and \ref{concentration}.
Using the estimates from \eqref{Fn_tail_estimates}, one simply integrates the tail to obtain the $L^p$ part of \eqref{Fn_lim_general}, and appeals to Borel--Cantelli for the almost sure part.

First we prove Lemma \ref{means_converge}.

\begin{proof}[Proof of Lemma \ref{means_converge}]
The equality in \eqref{p_def} follows from Fekete's lemma, once we show that the sequence $(\E \log Z_n(\beta))_{n\geq0}$ is superadditive.
Indeed, we will simply generalize the argument seen, for instance, in \cite[p.~440]{carmona-hu02}.
But first we check that $\E\log Z_n(\beta)$ is finite for each $n$.
To obtain an upper bound, we use Jensen's inequality applied to $t \mapsto \log t$.
This just yields the annealed bound,
\eq{
\E\log Z_n(\beta) \leq \log \E Z_n(\beta) = \log \E\bigg[\sum_{\sigma_1,\dots,\sigma_n} \exp\bigg(\beta\sum_{i=1}^n\omega(i,\sigma_i)\bigg)\prod_{i=1}^nP(\sigma_{i-1},\sigma_i)\bigg].
}
The nonnegativity of all summands allows us, by Tonelli's theorem, to pass the expectation through the sum.
That is,
\eq{
\E Z_n(\beta) = \sum_{\sigma_1,\dots,\sigma_n} \E\exp\bigg(\beta\sum_{i=1}^n\omega(i,\sigma_i)\bigg)\prod_{i=1}^nP(\sigma_{i-1},\sigma_i)
= \sum_{\sigma_1,\dots,\sigma_n} \e^{n\lambda(\beta)}\prod_{i=1}^nP(\sigma_{i-1},\sigma_i) = \e^{n\lambda(\beta)},
}
and so
\eq{
\E \log Z_n(\beta) \leq n\lambda(\beta) < \infty.
}
In particular, the inequality in \eqref{p_def} follows from the above display.

On the other hand, a lower bound is found by again using Jensen's inequality, but applied to $t \mapsto \e^t$ and with respect to $P$:
\eq{
\E \log Z_n(\beta) &= \E\bigg[\log \sum_{\sigma_1,\dots,\sigma_n} \exp\bigg(\beta\sum_{i=1}^n\omega(i,\sigma_i)\bigg)\prod_{i=1}^nP(\sigma_{i-1},\sigma_i)\bigg] \\
&\geq \E\bigg[\sum_{\sigma_1,\dots,\sigma_n} \beta \sum_{i = 1}^n \omega(i,\sigma_i) \prod_{i = 1}^n P(\sigma_{i-1},\sigma_i)\bigg]
= \beta\E(\omega),
}
where the final equality is a consequence of Fubini's theorem.
Indeed, we may apply Fubini's theorem because
\eq{
\E|\omega| 
= \int_0^\infty \P(|\omega|\geq t)\ \dd t
\leq \int_0^\infty \e^{-\beta t} \E(\e^{\beta|\omega|})\ \dd t
< \infty.
}
Therefore, we obtain the desired lower bound:
\eq{
\E\log Z_n(\beta) \geq \beta\E(\omega) > -\infty.
}

Now we can prove superadditivity.
For a given integer $k \geq 0$ and $y \in \Z^d$, let $\theta_{k,y}$ be the associated time-space translation of the environment:
\eq{
(\theta_{k,y}\, \omega)(i,x) = \omega(i+k,x+y).
}
Because the collection $(\omega(i,x): i \geq 1, x\in\Z^d)$ is i.i.d., the random variables $Z_n(\beta)$ and $Z_n(\beta) \circ \theta_{k,y}$ have the same law.
Furthermore, for any $0 \leq k \leq n$, we have the identity
\eeq{
Z_n(\beta) = \sum_{y \in \Z^d} Z_k(\beta,y)\cdot(Z_{n-k}(\beta) \circ \theta_{k,y}), \label{Z_identity}
}
where $Z_k(\beta,y) \coloneqq E(\e^{\beta H_k(\sigma)}; \sigma_k = y)$ is the contribution to $Z_k(\beta)$ coming from the endpoint $y$.
By Jensen's inequality,
\eq{
\log Z_n(\beta) = \log \sum_{y \in \Z^d} Z_k(\beta,y)\cdot(Z_{n-k}(\beta) \circ \theta_{k,y})
&= \log Z_k(\beta) + \log \sum_{y \in \Z^d} \frac{Z_k(\beta,y)}{Z_k(\beta)}\cdot (Z_{n-k}(\beta) \circ \theta_{k,y}) \\
&\geq \log Z_k(\beta) + \sum_{y \in \Z^d} \frac{Z_k(\beta,y)}{Z_k(\beta)}\log(Z_{n-k}(\beta) \circ \theta_{k,y}).
}
Notice that $Z_{n-k}(\beta) \circ \theta_{k,y}$ depends only on the environment after time $k$, meaning it is independent of $\FF_k$. 
We thus have
\eq{
\E\givenk[\bigg]{\sum_{y \in \Z^d} \frac{Z_k(\beta,y)}{Z_k(\beta)}\log(Z_{n-k}(\beta) \circ \theta_{k,y})}{\FF_k} 
&= \sum_{y \in \Z^d} \frac{Z_k(\beta,y)}{Z_k(\beta)}\E(\log Z_{n-k}(\beta) \circ \theta_{k,y})
= \E \log Z_{n-k}(\beta).
}
Using this observation in the previous display, we arrive at the desired superadditive inequality:
\eq{
\E\log Z_n(\beta) \geq \E \log Z_k(\beta) + \E \log Z_{n-k}(\beta).
}
\end{proof}

Next we show concentration of $F_n(\beta)$.
The proof of Lemma \ref{concentration} comes directly from \cite{liu-watbled09}, in which Liu and Watbled prove the martingale inequality \eqref{mg_ineq}.
One interesting feature of Lemma \ref{concentration} is that the constant $c$ can be chosen independently of the reference walk.
To verify that this is the case, we will recall the relevant arguments from \cite{liu-watbled09} that utilize the following proposition.

\begin{prop}[{\cite[Theorem 2.1]{liu-watbled09}}] \label{mg_inequality}
Let $(X_j)_{1 \leq j\leq n}$ be a sequence of supermartingale differences, adapted to the filtration $(\FF_j)_{1\leq j \leq n}$.
Let $S_n \coloneqq \sum_{j = 1}^n X_i$.
If for some constant $K > 0$ and all $j = 1,2,\dots,n$,
\eq{
\E\givenp{\e^{|X_j|}}{\FF_{j-1}} \leq K \quad \mathrm{a.s.},
}
then
\eeq{ \label{mg_ineq}
\P\Big(\frac{S_n}{n} > x\Big) \leq \begin{cases}
\exp\big(-\frac{nx^2}{K(1+\sqrt{2})^2}\big) &\text{if }0\leq x\leq K, \\
\exp\big(-\frac{nx}{(1+\sqrt{2})^2}\big) &\text{if }x > K.
\end{cases}
}
\end{prop}

\begin{proof}[Proof of Lemma \ref{concentration}]
Fix $n$ and consider the normalized partition function
\eq{
W_n(\beta) \coloneqq \frac{Z_n(\beta)}{\E Z_n(\beta)} = \frac{Z_n(\beta)}{\e^{n\lambda(\beta)}} 
= E \exp\bigg(\sum_{i = 1}^n [\beta\eta(i,\omega_i) - \lambda(\beta)]\bigg).
}
For each $1 \leq j \leq n$, define the $\sigma$-algebra
\eq{
\wh F_j = \sigma(\eta(i,x) ; 1 \leq i \leq n, i \neq j, x \in \Z^d).
}
By writing
\eq{
\log W_n(\beta) - \E\log W_n(\beta) = \sum_{j = 1}^n \E\givenp{\log W_n(\beta)}{\FF_j} - \E\givenp{\log W_n(\beta)}{\FF_{j-1}},
}
we identity the martingale differences
\eq{
X_{j} \coloneqq \E\givenp{\log W_n(\beta)}{\FF_j} - \E\givenp{\log W_n(\beta)}{\FF_{j-1}}, \quad 1 \leq j \leq n.
}
If we define the random variable
\eq{
\wh W_j(\beta) \coloneqq E(e_{j}), \quad \text{where} \quad e_{j} \coloneqq \exp\bigg(\sum_{\substack{1\leq i \leq n \\ i \neq j}} [\beta\eta(i,\omega_i) - \lambda(\beta)]\bigg),
}
then the fact that $\wh W_j(\beta)$ is $\wh F_j$-measurable implies $\E\givenp{\wh W_j(\beta)}{\FF_{j}} = \E\givenp{\wh W_j(\beta)}{\FF_{j-1}}$.
Therefore, we can instead write
\eq{
X_j = \E\givenp[\Big]{\log \frac{W_n(\beta)}{\wh W_j(\beta)}}{\FF_j} - \E\givenp[\Big]{\log \frac{W_n(\beta)}{\wh W_j(\beta)}}{\FF_{j-1}}.
}
Now, for any $t \in \R$, conditional Jensen's inequality applied to $x \mapsto \e^{x}$ gives
\eeq{ \label{t_and_minus_t}
\E\givenp{\e^{tX_j}}{\FF_{j-1}} &= \E\givenp{\e^{t\E\givenp{\log W_n(\beta)/\wh W_j(\beta)}{\FF_j}}}{\FF_{j-1}}\e^{-t\E\givenp{\log W_n(\beta)/\wh W_j(\beta)}{\FF_{j-1}}} \\
&\leq \E\givenp[\big]{(W_n(\beta)/\wh W_j(\beta))^t}{\FF_{j-1}}\E\givenp[\big]{(W_n(\beta)/\wh W_j(\beta))^{-t}}{\FF_{j-1}}.
}
Now observe that
\eq{
\frac{W_n(\beta)}{\wh W_j(\beta)} = \sum_{x \in \Z^d} \frac{E(e_{j}\, ;\, \omega_j = x)}{\wh W_j(\beta)}\e^{\beta\eta(j,x) - \lambda(\beta)},
}
from which Jensen's inequality shows
\eq{
\Big(\frac{W_n(\beta)}{\wh W_j(\beta)}\Big)^t \leq \sum_{x \in \Z^d} \frac{E(e_{j}\, ;\, \omega_j = x)}{\wh W_j(\beta)}\e^{t\beta\eta(j,x) - t\lambda(\beta)}
}
whenever $t \leq 0$ or $t \geq 1$.
Now the independence of $\eta(j,x)$ from $\wh \FF_j \supset \FF_{j-1}$ yields
\eeq{ \label{at_least_1}
\E\givenk[\bigg]{\Big(\frac{W_n(\beta)}{\wh W_j(\beta)}\Big)^t}{\FF_{j-1}}
&= \E\givenk[\bigg]{\E\givenk[\Big]{\Big(\frac{W_n(\beta)}{\wh W_j(\beta)}\Big)^t}{\wh \FF_j}}{\FF_{j-1}} \\
&\leq  \E\givenk[\bigg]{\sum_{x \in \Z^d}\frac{E(e_{j}\, ;\, \omega_j = x)}{\wh W_j(\beta)}}{\FF_{j-1}}\e^{\lambda(t\beta)-t\lambda(\beta)} 
= \e^{\lambda(t\beta)-t\lambda(\beta)}.
}
Therefore, if $|t| \geq 1$, then we can deduce from \eqref{t_and_minus_t} and two applications of \eqref{at_least_1} that
\eq{
\E\givenp{\e^{tX_j}}{\FF_{j-1}} \leq \e^{\lambda(t\beta)-t\lambda(\beta)}\e^{\lambda(-t\beta)+t\lambda(\beta)} = \e^{\lambda(t\beta)+\lambda(-t\beta)}.
}
In particular,
\eq{
\E\givenp{\e^{|X_j|}}{\FF_{j-1}} \leq \E\givenp{\e^{X_j}}{\FF_{j-1}} + \E\givenp{\e^{-X_j}}{\FF_{j-1}}
\leq 2\e^{\lambda(\beta)+\lambda(-\beta)} \eqqcolon K(\beta).
}
Applying Proposition \ref{mg_inequality} to
\eq{
\frac{S_n}{n} = \frac{\log W_n(\beta)}{n} - \frac{\E\log W_n(\beta)}{n}
= \frac{\log Z_n(\beta)}{n} - \lambda(\beta) - \frac{\E \log Z_n(\beta)}{n} + \lambda(\beta)
= F_n(\beta) - \E F_n(\beta),
}
we see that that \eqref{Fn_tail_estimates} is a weaker form of \eqref{mg_ineq}, with
\eq{
c \coloneqq \min\Big(\frac{1}{K(\beta)(1+\sqrt{2})^2},\frac{1}{(1+\sqrt{2})^2}\Big).
}
\end{proof}

\subsection{Existence of critical temperature} \label{existence_of_phase_transition}
Now we prove a phase transition between the high temperature and low temperature regimes.

\begin{prop} \label{phase_transition}
Assume
\eq{
\beta_{\max} \coloneqq \sup\{\beta \geq 0 : \lambda(\pm\beta) < \infty\} \in (0,\infty].
}
Then $\lambda - p \geq 0$ is non-decreasing on the interval $[0,\beta_{\max})$.
In particular, there exists a critical value $\beta_\cc = \beta_\cc(d,\mathfrak{L}_\omega,P) \in [0,\beta_{\max}]$ such that for every $\beta \in (0,\beta_{\max})$,
\begin{align}
0 \leq \beta \leq \beta_\cc \quad &\implies \quad p(\beta) = \lambda(\beta), \label{below_phase_transition}\\
\beta > \beta_\cc \quad &\implies \quad p(\beta) < \lambda(\beta). \label{above_phase_transition}
\end{align}
\end{prop}

\begin{remark}
Before the proof, some comments are in order: \begin{itemize}
\item This phase transition was proved by Comets and Yoshida \cite[Theorem 3.2(b)]{comets-yoshida06} for the SRW case under the hypothesis $\beta_{\max} = \infty$, and we adopt their general proof strategy.
\item Two generalizations were claimed to follow easily from the same methods.
Vargas \cite[Lemma 3.4]{vargas07} suggests the hypothesis $\beta_{\max} = \infty$ can be dropped, while Comets \cite[Theorem 6.1]{comets07} allows $P$ to be $\alpha$-stable, $\alpha \in [0,2)$.
We do both by assuming only $\beta_{\max} > 0$ and allowing $P$ to be a general random walk.
It seems the resulting difficulties are only technical, but how to resolve them is not obvious, and so we provide a full proof.
\item If $\lim_{\beta\to\beta_{\max}}\beta\lambda'(\beta) - \lambda(\beta) = \infty$, then Theorem \ref{localization_sufficient_background} guarantees $\beta_\cc < \beta_{\max}$ whenever the entropy of $P(\sigma_1 = \cdot)$ is finite.
\end{itemize}
\end{remark}

\begin{lemma} \label{decomposition}
For any fixed $\beta > 0$, there is a decomposition
\eq{
|x|\e^{\beta x} = g(x) - h(x),
}
where $g : \R \to \R$ is non-decreasing, and $0 \leq h(x) \leq (\beta \e)^{-1}$ for all $x \in \R$.
\end{lemma}

\begin{proof}
Observe that $\vphi(x) \coloneqq |x|\e^{\beta x}$ is an increasing function on $(-\infty,-\beta^{-1}]$ and on $[0,\infty)$, with
$\vphi(-\beta^{-1}) = (\beta \e)^{-1}$ and $\vphi(x_0) = (\beta \e)^{-1}$ for some (unique) $x_0 > 0$.
Therefore, we define
\eq{
g(x) \coloneqq \begin{cases} (\beta \e)^{-1} &\text{if }x \in [-\beta^{-1},x_0], \\
|x|\e^{\beta x} &\text{otherwise}, \end{cases} \qquad
h(x) \coloneqq g(x) - |x|\e^{\beta x}.
}
Then $g$ is non-decreasing with $g \geq \vphi$, and so
$0 \leq h(x) \leq (\beta \e)^{-1}$.
\end{proof}

\begin{proof}[Proof of Proposition \ref{phase_transition}]
Note that
$\log Z_n(\beta) = E(\e^{\beta H_n(\sigma)})$
is the (random) logarithmic moment generating function for $H_n(\sigma)$ with respect to $P$, and is finite for all $\beta \in [0,\beta_{\max})$ almost surely by the proof of Lemma \ref{means_converge}.
Therefore, $\E\log Z_n(\beta)$ is convex on $(0,\beta_{\max})$.
Now seen to be the limit of convex functions, $p$ must be convex on $(0,\beta_{\max})$.
It follows from general convex function theory that $p$ is differentiable almost everywhere on $(0,\beta_{\max})$, and
\eq{
p'(\beta) = \lim_{n \to \infty} \frac{\partial}{\partial\beta} \E F_n(\beta) \quad \text{whenever $p'(\beta)$ exists}.
}
Furthermore, convexity implies $p$ is absolutely continuous on any closed subinterval of $[0,\beta_{\max})$, and thus
\eq{
p(\beta_0) = \int_0^{\beta_0} p'(\beta)\ \dd\beta \quad \text{for all } \beta_0 \in (0,\beta_{\max}).
}
Suppose we can show
\eeq{ \label{derivative_ineq_to_show}
\frac{\partial}{\partial \beta}\E F_n(\beta) \leq \lambda'(\beta) \quad \text{for all $\beta \in (0,\beta_{\max})$.}
}
Then we could conclude that
\eq{
\lambda(\beta_0) - p(\beta_0) = \int_0^{\beta_0} [\lambda'(\beta)-p'(\beta)]\ \dd\beta
= \int_0^{\beta_0} \Big[\lambda'(\beta) - \lim_{n\to\infty} \frac{\partial}{\partial\beta}\E F_n(\beta)\Big]\ \dd\beta
}
is an increasing function of $\beta_0 \in [0,\beta_{\max})$, with
$p(0) = \lambda(0) = 0.$
In particular, the existence of $\beta_\cc$ will be proved.
Therefore, we need only to show \eqref{derivative_ineq_to_show}.
To do so, we would like to write
\eeq{ \label{would_like}
\frac{\partial}{\partial\beta} \E\log Z_n(\beta)
\stackrel{\text{(a)}}{=} \E\Big[\frac{\partial}{\partial\beta}\log Z_n(\beta)\Big]
= \E\bigg[\frac{\frac{\partial}{\partial\beta}Z_n(\beta)}{Z_n(\beta)}\bigg] 
&\stackrel{\phantom{\text{(b)}}}{=} \E\bigg[\frac{\frac{\partial}{\partial\beta}E(\e^{\beta H_n(\sigma)})}{Z_n(\beta)}\bigg] \\
&\stackrel{\text{(b)}}{=} \E\bigg[E\bigg(\frac{H_n(\sigma)\e^{\beta H_n(\sigma)}}{Z_n(\beta)}\bigg)\bigg] \\
&\stackrel{\text{(c)}}{=} E\bigg[\E\bigg(\frac{H_n(\sigma)\e^{\beta H_n(\sigma)}}{Z_n(\beta)}\bigg)\bigg], \\
}
but each of (a), (b), and (c) require justification.
Postponing these technical verifications for the moment, we complete the proof of \eqref{derivative_ineq_to_show} assuming \eqref{would_like}.

For a fixed $\sigma \in \Omega_p$, we can write
\eq{
\E\bigg(\frac{H_n(\sigma)\e^{\beta H_n(\sigma)}}{Z_n(\beta)}\bigg) = \e^{n\lambda(\beta)}\wt\E\bigg(\frac{H_n(\sigma)}{Z_n(\beta)}\bigg),
}
where $\wt\E$ denotes expectation with respect to the probability measure $\wt\P$ given by
\eq{
\dd\wt\P = \frac{\e^{\beta H_n(\sigma)}}{\e^{n\lambda(\beta)}}\ \dd\P
= \frac{\e^{\beta H_n(\sigma)}}{\E(\e^{\beta H_n(\sigma)})}\ \dd\P.
}
Since the Radon-Nikodym derivative
\eq{
\frac{\dd \wt\P}{\dd\P} = \e^{\beta H_n(\sigma)-n\lambda(\beta)} = \prod_{i = 1}^n \e^{\beta \omega(i,\sigma_i)-\lambda(\beta)}
}
is a product of independent quantities (with respect to $\P$), the probability measure $\wt\P$ remains a product measure.
Therefore, we can apply the Harris-FKG inequality (e.g.~see \cite[Theorem 2.4]{grimmett99}):
$H_n(\sigma)$ is non-decreasing in all $\omega(i,x)$, while $Z_n(\beta)^{-1}$ is decreasing, which implies
\eq{
\wt\E\bigg(\frac{H_n(\sigma)}{Z_n(\beta)}\bigg) &\leq 
\wt\E(H_n(\sigma))\wt\E(Z_n(\beta)^{-1}),
}
where
\eq{
\wt\E(H_n(\sigma)) = \e^{-n\lambda(\beta)} \E(H_n(\sigma)\e^{\beta H_n(\sigma)}) 
&= \e^{-n\lambda(\beta)} \sum_{i = 1}^n \E(\omega(i,\sigma_i)\e^{\beta H_n(\sigma)}) \\
&= \e^{-n\lambda(\beta)} \sum_{i = 1}^n \underbrace{\E(\omega(i,\sigma_i)\e^{\beta \omega(i,\sigma_i)})}_{\lambda'(\beta)\e^{\lambda(\beta)}}\prod_{j \neq i} \underbrace{\E(\e^{\beta\omega(j,\sigma_j)})}_{\e^{\lambda(\beta)}}
= n\lambda'(\beta).
}
Therefore,
\eeq{ \label{last_inequality}
\E\bigg(\frac{H_n(\sigma)\e^{\beta H_n(\sigma)}}{Z_n(\beta)}\bigg) \leq \e^{n\lambda(\beta)}\wt\E(H_n(\sigma))\wt\E(Z_n(\beta)^{-1}) = n\lambda'(\beta)\cdot\E(Z_n(\beta)^{-1}\e^{\beta H_n(\sigma)}).
}
We now have
\eq{
\frac{\partial}{\partial \beta} \E\log Z_n(\beta) 
\stackrel{\mbox{\scriptsize\eqref{would_like}}}{=} E\bigg[\E\bigg(\frac{H_n(\sigma)\e^{\beta H_n(\sigma)}}{Z_n(\beta)}\bigg)\bigg]
&\stackrel{\mbox{\scriptsize\eqref{last_inequality}}}{\leq} n\lambda'(\beta)\cdot E[\E(Z_n(\beta)^{-1}\e^{\beta H_n(\sigma)})] \\
&\stackrel{\phantom{\eqref{last_inequality}}}{=} n\lambda'(\beta)\cdot\E[Z_n(\beta)^{-1}E(\e^{\beta H_n(\sigma)})] = n\lambda'(\beta),
}
where the penultimate equality is a consequence of Tonelli's theorem, since $Z_n(\beta)^{-1}\e^{\beta H_n(\sigma)} > 0$.
The inequality \eqref{derivative_ineq_to_show} now follows by dividing by $n$.

\subsubsection{Justification of $\mathrm{(c)}$ in \eqref{would_like}}
Fix $\beta \in (0,\beta_{\max})$.
Choose $q > 1$ such that $q\beta < \beta_{\max}$, and let $q'$ be its H\"{o}lder conjugate:
$1/q + 1/q' = 1$.
Step (c) in \eqref{would_like} will follow from Fubini's theorem once we verify that
\eeq{ \label{before_expectation}
E\bigg(\E\bigg|\frac{H_n(\sigma)\e^{\beta H_n(\sigma)}}{Z_n(\beta)}\bigg|\bigg) < \infty.
}
Let $g$ and $h$ be as in Lemma \ref{decomposition}, so that we can write
\eeq{ \label{g_and_h}
\bigg|\frac{H_n(\sigma)\e^{\beta H_n(\sigma)}}{Z_n(\beta)}\bigg|
= \frac{|H_n(\sigma)|\e^{\beta H_n(\sigma)}}{Z_n(\beta)}
= \frac{g(H_n(\sigma))}{Z_n(\beta)} - \frac{h(H_n(\sigma))}{Z_n(\beta)} \leq \frac{g(H_n(\sigma))}{Z_n(\beta)}.
}
Temporarily fix a path $\sigma \in \Omega_p$.
Since $Z_n(\beta)$ and $H_n$, and therefore $g(H_n)$, are non-decreasing functions of all $\omega(i,x)$, the Harris-FKG inequality shows
\eeq{ \label{fkg_1}
\E\Big(\frac{g(H_n(\sigma))}{Z_n(\beta)}\Big) \leq \E\big[g(H_n(\sigma))\big]\E(Z_n(\beta)^{-1}).
}
The first factor satisfies
\eeq{ \label{first_factor}
\E\big[g(H_n(\sigma))\big]
&= \E\big[|H_n(\sigma)|\e^{\beta H_n(\sigma)}\big] + \E\big[h(H_n(\sigma))\big] \\
&\leq \E\bigg[\sum_{i = 1}^n |\omega(i,\sigma_i)|\exp\bigg(\beta\sum_{i = 1}^n \omega(i,\sigma_i)\bigg)\bigg] + (\beta \e)^{-1} \\
&= \E\bigg[\sum_{i=1}^n |\omega(i,\sigma_i)|\e^{\beta \omega(i,\sigma_i)}\prod_{j\neq i} \e^{\beta\omega(j,\sigma_j)}\bigg] + (\beta \e)^{-1} \\
&\leq n\E(|\omega| \e^{\beta \omega})\e^{(n-1)\lambda(\beta)} + (\beta \e)^{-1} \\
&\leq n(\E|\omega|^{q'})^{1/q'}\e^{\lambda(q\beta)/q}\e^{(n-1)\lambda(\beta)} + (\beta \e)^{-1} < \infty.
}
The second factor satisfies
\eeq{ \label{second_factor}
\E(Z_n(\beta)^{-1}) = \E\big(E(\e^{\beta H_n(\sigma)})^{-1}\big) \leq \E\big(E(\e^{\beta H_n(\sigma)})\big) = E\big(\E(\e^{\beta H_n(\sigma)})\big) = n\lambda(-\beta) < \infty,
}
where we have used Tonelli's theorem to exchange the order of integration.
We have thus shown
\eeq{ \label{fubini_hypothesis}
E\bigg(\E\bigg|\frac{H_n(\sigma)\e^{\beta H_n(\sigma)}}{Z_n(\beta)}\bigg|\bigg)
\stackrel{\mbox{\scriptsize\eqref{g_and_h}}}{\leq} E\bigg(\E\Big[\frac{g(H_n(\sigma))}{Z_n(\beta)}\Big]\bigg)
\stackrel{\text{\eqref{fkg_1}--\eqref{second_factor}}}{<} \infty,
}
as desired.

\subsubsection{Justification of $\mathrm{(b)}$ in \eqref{would_like}} \label{justification_b}
By simple differentiation rules,
\eq{
\frac{\partial}{\partial\beta}\log Z_n(\beta) = \frac{\frac{\partial}{\partial\beta}Z_n(\beta)}{Z_n(\beta)}
= \frac{\frac{\partial}{\partial \beta} E(\e^{\beta H_n(\sigma)})}{Z_n(\beta)}.
}
We would like to pass the derivative through the expectation and write
\eq{
\frac{\partial}{\partial \beta} E(\e^{\beta H_n(\sigma)}) = E(H_n(\sigma)\e^{\beta H_n(\sigma)}).
}
That is, we claim
\eeq{ \label{derivative_through_EE}
&\frac{\partial}{\partial \beta}\bigg[\sum_{\sigma_1,\sigma_2,\dots,\sigma_n \in \Z^d} \exp\bigg(\beta\sum_{i = 1}^n \omega(i,\sigma_i)\bigg)\prod_{i=1}^n P(\sigma_{i-1},\sigma_i)\bigg] \\
&= \sum_{\sigma_1,\sigma_2,\dots,\sigma_n \in \Z^d} \bigg(\sum_{i = 1}^n \omega(i,\sigma_i)\bigg)\exp\bigg(\beta\sum_{i = 1}^n \omega(i,\sigma_i)\bigg)\prod_{i=1}^n P(\sigma_{i-1},\sigma_i) \quad \mathrm{a.s.}
}
To show \eqref{derivative_through_EE} at a particular $\beta_0 \in (0,\beta_{\max})$, it suffices to exhibit $\eps > 0$ and a constant $C_{\vc\omega} < \infty$ depending only on the quenched environment $\vc\omega$, such that
\eeq{ \label{suffices_for_need}
\sum_{\sigma_1,\sigma_2,\dots,\sigma_n\in\Z^d} \bigg|\sum_{i = 1}^n \omega(i,\sigma_i)\bigg|\exp\bigg(\beta\sum_{i = 1}^n \omega(i,\sigma_i)\bigg)\prod_{i=1}^n P(\sigma_{i-1},\sigma_i) < C_{\vc\omega} \quad \forall\, \beta \in [\beta_0-\eps,\beta_0+\eps].
}
Indeed, if \eqref{suffices_for_need} holds, then for $\beta \in [\beta_0-\eps,\beta_0+\eps]$,
\eq{
\sum_{\substack{\sigma_1,\sigma_2,\dots,\sigma_n \in \Z^d \\ \|\sigma_i\|_1 \leq M\, \forall\, i}} \bigg(\sum_{i = 1}^n \omega(i,\sigma_i)\bigg)\exp\bigg(\beta\sum_{i = 1}^n \omega(i,\sigma_i)\bigg)\prod_{i=1}^n P(\sigma_{i-1},\sigma_i)
\xrightarrow[{M\to\infty}]{\mathrm{uniformly~in~}\beta} E(H_n(\sigma)\e^{\beta H_n(\sigma)}).
}
In summary, we know by definition that
\eq{
\sum_{\substack{\sigma_1,\sigma_2,\dots,\sigma_n \in \Z^d \\ \|\sigma_i\|_1 \leq M\, \forall\, i}} \exp\bigg(\beta\sum_{i = 1}^n \omega(i,\sigma_i)\bigg)\prod_{i=1}^n P(\sigma_{i-1},\sigma_i) \xrightarrow[M\to\infty]{} E(\e^{\beta H_n(\sigma)}) \quad
\text{for all $\beta \in [0,\beta_{\max})$},
}
and we know by the above argument that the derivative of the left-hand side converges uniformly to $E(H_n(\sigma)\e^{\beta H_n(\sigma)})$ near $\beta_0$.
It follows that
\eeq{ \label{derivative_on_interval}
\frac{\partial}{\partial\beta}E(\e^{\beta H_n(\sigma)})
= E(H_n(\sigma)\e^{\beta H_n(\sigma)}) \quad \text{for all $\beta \in (\beta_0-\eps,\beta_0+\eps)$,}
}
in particular when $\beta = \beta_0$.
We are thus left only with the task of establishing \eqref{suffices_for_need} for almost every $\vc\omega$.

Fix $\beta_0 \in (0,\beta_{\max})$, and let $\eps > 0$ be such that $\beta_0 + \eps < \beta_{\max}$ and $\beta_0 - \eps > 0$.
Choose a number $q > 1$ such that $q(\beta_0+\eps) < \beta_{\max}$, and let $q' > 1$ denote its H\"{o}lder conjugate: $1/q + 1/q' = 1$.
For all $\beta \in [\beta_0-\eps,\beta_0+\eps]$, we have the uniform upper bound
\eq{
&\sum_{\sigma_1,\sigma_2,\dots,\sigma_n\in\Z^d} \bigg|\sum_{i = 1}^n \omega(i,\sigma_i)\bigg|\exp\bigg(\beta\sum_{i = 1}^n \omega(i,\sigma_i)\bigg)\prod_{i=1}^n P(\sigma_{i-1},\sigma_i) \\
&\leq \sum_{\sigma_1,\sigma_2,\dots,\sigma_n\in\Z^d} \sum_{i = 1}^n |\omega(i,\sigma_i)|\exp\bigg((\beta_0+\eps)\sum_{i = 1}^n \omega(i,\sigma_i)_+\bigg) \prod_{i=1}^n P(\sigma_{i-1},\sigma_i) \\
&\leq \sum_{\sigma_1,\sigma_2,\dots,\sigma_n\in\Z^d} \bigg[\sum_{i = 1}^n |\omega(i,\sigma_i)|(\e^{(\beta_0+\eps)\omega(i,\sigma_i)}+1)\prod_{j \neq i} (\e^{(\beta_0+\eps)\omega(j,\sigma_j)}+1)\bigg] \prod_{i=1}^n P(\sigma_{i-1},\sigma_i),
}
where $x_+ \coloneqq \max(0,x)$.
Furthermore, this upper bound is finite for almost every $\vc\omega$, since
\eq{
&\sum_{\sigma_1,\sigma_2,\dots,\sigma_n\in\Z^d} \E\bigg[\sum_{i = 1}^n |\omega(i,\sigma_i)|(\e^{(\beta_0+\eps)\omega(i,\sigma_i)}+1)\prod_{j \neq i} (\e^{(\beta_0+\eps)\omega(j,\sigma_j)}+1)\bigg] \prod_{i=1}^n P(\sigma_{i-1},\sigma_i) \\
&= n\E\big[|\omega|(\e^{(\beta_0+\eps)\omega}+1)\big]\big(\E[\e^{(\beta_0+\eps)\omega}+1]\big)^{n-1} \\
&\leq n(\E|\omega|^{q'})^{1/q'}\Big(\E\big[(\e^{(\beta_0+\eps)\omega}+1)^q\big]\Big)^{1/q}\big(\E[\e^{(\beta_0+\eps)\omega}+1]\big)^{n-1} \\
&\leq n(\E|\omega|^{q'})^{1/q'}\Big(2^q\E\big[\e^{q(\beta_0+\eps)\omega}+1\big]\Big)^{1/q}\big(\E[\e^{(\beta_0+\eps)\omega}+1]\big)^{n-1} \\
&= 2n(\E|\omega|^{q'})^{1/q'}\big(\e^{\lambda(q(\beta_0+\eps))}+1\big)^{1/q}(\e^{\lambda(\beta_0+\eps)}+1)^{n-1} < \infty.
}
In particular, some $C_{\vc\omega}$ satisfying \eqref{suffices_for_need} exists almost surely, and so
\eqref{derivative_through_EE} holds almost surely.
Therefore, (b) holds in \eqref{would_like}:
\eq{
\E\bigg[\frac{\frac{\partial}{\partial\beta}E(\e^{\beta H_n(\sigma)})}{Z_n(\beta)}\bigg]
= \E\bigg[\frac{E(H_n(\sigma)\e^{\beta H_n(\sigma)})}{Z_n(\beta)}\bigg]
=  \E\bigg[E\bigg(\frac{H_n(\sigma)\e^{\beta H_n(\sigma)}}{Z_n(\beta)}\bigg)\bigg].
}

\subsubsection{Justification of $\mathrm{(a)}$ in \eqref{would_like}}
Fix $\beta_0 \in (0,\beta_{\max})$ and $\eps > 0$ such that $\beta_0+\eps < \beta_{\max}$ and $\beta_0-\eps > 0$.
We begin by writing
\eq{
\Big[\frac{\partial}{\partial \beta} \E\log Z_n(\beta)\Big]_{\beta=\beta_0} = 
\lim_{h \to 0} \E\Big[\frac{\log Z_n(\beta_0+h) - \log Z_n(\beta_0)}{h}\Big].
}
We will ultimately invoke dominated convergence to pull the limit through the expectation.
Consider any fixed $h$ satisfying $|h| \leq \eps$.
Now, $\log Z_n(\cdot)$ is almost surely continuously differentiable, and
so by the mean value theorem,
\eq{
\frac{\log Z_n(\beta_0+h) - \log Z_n(\beta_0)}{h} = \Big[\frac{\partial}{\partial\beta}\log Z_n(\beta)\Big]_{\beta = \beta_0+\eps_{\vc\omega}} \quad \mathrm{a.s.}
}
for some $\eps_{\vc\omega}$ depending on the random environment $\vc\omega$ and satisfying $|\eps_{\vc\omega}| \leq |h| \leq \eps$.
But then by convexity of $\log Z_n(\beta)$, we can bound the difference quotient by consider the endpoints of the interval $[\beta_0-\eps,\beta_0+\eps]$:
\eq{
\Big|\frac{\log Z_n(\beta_0+h) - \log Z_n(\beta_0)}{h}\Big| &= \Big|\Big[\frac{\partial}{\partial\beta}\log Z_n(\beta)\Big]_{\beta=\beta_0+\eps_{\vc\omega}}\Big| \quad \mathrm{a.s.} \\
&\leq \max\bigg(\Big|\Big[\frac{\partial}{\partial\beta}\log Z_n(\beta)\Big]_{\beta=\beta_0-\eps}\Big|,\Big|\Big[\frac{\partial}{\partial\beta}\log Z_n(\beta)\Big]_{\beta=\beta_{0}+\eps}\Big|\bigg) \quad \mathrm{a.s.},
}
where the maximum is now a dominating function independent of $h$.
We showed \eqref{derivative_on_interval} holds almost surely, and so
\eq{
\frac{\partial}{\partial\beta}\log Z_n(\beta) = E\bigg(\frac{H_n(\sigma)\e^{\beta H_n(\sigma)}}{Z_n(\beta)}\bigg)  \quad \text{for all $\beta \in (0,\beta_{\max})$} \quad \mathrm{a.s.}
}
Therefore, for any $\beta_1 \in (0,\beta_{\max})$ we have
\eq{
\E\Big|\Big[\frac{\partial}{\partial\beta}\log Z_n(\beta)\Big]_{\beta=\beta_1}\Big| = \E\bigg|E\bigg(\frac{H_n(\sigma)\e^{-\beta_1 H_n(\sigma)}}{Z_n(\beta)}\bigg)\bigg|
\leq \E\bigg(E\bigg|\frac{H_n(\sigma)\e^{-\beta_1 H_n(\sigma)}}{Z_n(\beta)}\bigg|\bigg)
\stackrel{\mbox{\scriptsize\eqref{fubini_hypothesis}}}{<} \infty.
}
In particular, the dominating function is integrable:
\eq{
&\E\bigg[\max\bigg(\Big|\Big[\frac{\partial}{\partial\beta}\log Z_n(\beta)\Big]_{\beta=\beta_0-\eps}\Big|,\Big|\Big[\frac{\partial}{\partial\beta}\log Z_n(\beta)\Big]_{\beta=\beta_{0}+\eps}\Big|\bigg)\bigg] \\
&\leq \E\Big|\Big[\frac{\partial}{\partial\beta}\log Z_n(\beta)\Big]_{\beta=\beta_0 - \eps}\Big| + \E\Big|\Big[\frac{\partial}{\partial\beta}\log Z_n(\beta)\Big]_{\beta=\beta_{0}+\eps}\Big|
< \infty.
}
Dominated convergence now proves (a) in \eqref{would_like}:
\eq{
\Big[\frac{\partial}{\partial \beta} \E\log Z_n(\beta)\Big]_{\beta=\beta_0}
&= \lim_{h \to 0} \E\Big[\frac{\log Z_n(\beta_0+h) - \log Z_n(\beta_0)}{h}\Big] \\
&= \E\Big[\lim_{h \to 0}\frac{\log Z_n(\beta_0+h) - \log Z_n(\beta_0)}{h}\Big] 
= \E\Big(\Big[\frac{\partial}{\partial\beta} \log Z_n(\beta)\Big]_{\beta = \beta_0}\Big).
}
\end{proof}

\section{Partitioned subprobability measures} \label{topology}
In this section and in the remainder of this chapter, we will  always assume \eqref{mgf_assumption}. Also, throughout, $f_n(\cdot)$ will denote the endpoint probability mass function $\mu_n^\beta(\sigma_n=\cdot)$. 
 The goals of this section are to (i) define a space of functions which contains endpoint distributions (i.e.~probability measures on $\Z^d$) and their localization limits (subprobability measures on $\N \times \Z^d$); (ii) equip said space with a metric; and (iii) prove that the induced metric topology is compact.

\subsection{Definition and properties} \label{topology-definitions}
Let us restrict our attention to the set of functions
\eq{
\SS_0 \coloneqq  \{f : \N \times \Z^d \to [0,1] : \|f\| \leq 1\},
}
where $\N = \{1,2,\dots\}$, and
\eeq{
\|f\| \coloneqq  \sum_{u \in \N \times \Z^d} f(u). \label{norm_def}
}
Since we regard distant point masses as nearly existing on separate copies of $\Z^d$, we consider
differences between elements of $\N \times \Z^d$ in the following non-standard way:
\eq{
(n,x) - (m,y) \coloneqq  \begin{cases}
x - y &\text{if } n = m, \\
\infty &\text{otherwise}.
\end{cases}
}
It then makes sense to write, for $u = (n,x)$ and $v = (m,y)$ in $\N \times \Z^d$,
\eq{
\|u - v\|_1 \coloneqq  \begin{cases}
\|x - y\|_1 &\text{if } n = m, \\
\infty &\text{otherwise},
\end{cases}
}
although $\|\cdot\|_1$ is not to be thought of as a norm.
If $u = (n,x)$ and $y \in \Z^d$, then $u \pm y$ will be understood to mean $(n,x\pm y)$.
Our main tool for enabling compactness will be certain injective functions on $\N \times \Z^d$.
Given a set $A \subset \N \times \Z^d$, we call a map $\phi: A \to \N \times \Z^d$ an \textit{isometry} of degree $m \geq 1$ if for all $u,v\in A$,
\eeq{
\|u - v\|_1 < m
\quad \text{or} \quad
\|\phi(u) - \phi(v)\|_1 < m
\quad \implies \quad 
\phi(u) - \phi(v) = u - v.
 \label{isometry}
}
The maximum $m$ for which \eqref{isometry} holds is called the \textit{maximum degree} of $\phi$, denoted by $\deg(\phi)$.
To say that $\phi$ is an isometry of degree $1$ simply means $\phi$ is injective.
If \eqref{isometry} holds for every $m \in \N$, then $\deg(\phi) = \infty$, meaning $\phi$ acts by translations. That is, each copy of $\Z^d$ intersecting the domain $A$ is translated and moved to some copy of $\Z^d$ in the range, with no two copies in the domain going to the same copy in the range.
Note that an isometry is necessarily injective and thus has an inverse defined on its image.
Because the hypothesis in \eqref{isometry} is symmetric between domain and range, it is clear that
$\deg(\phi) = \deg(\phi^{-1})$.

%

Another useful property of isometries is closure under composition, which is defined in the next lemma.
It is important, especially for the proof below, to note that an isometry is permitted to have empty domain.
That is, there exists the empty isometry $\phi: \varnothing \to \N \times \Z^d$, which is its own inverse and has $\deg(\phi) = \infty$.

\begin{lemma} \label{composition}
Let $\phi : A \to \N \times \Z^d$ and $\psi : B \to \N \times \Z^d$ be isometries.
Define $A' \coloneqq  \{a \in A : \phi(a) \in B\}$.  Then $\theta: A' \to \N \times \Z^d$ defined by $\theta(u) = \psi(\phi(u))$ is an isometry with $\deg(\theta) \geq \min(\deg(\phi),\deg(\psi))$.
\end{lemma}

\begin{proof}
Let $m \coloneqq  \min(\deg(\phi),\deg(\psi))$ so that $\phi$ and $\psi$ are each isometries of degree $m$.
For any $a_1,a_2 \in A'$,
\eq{
\|a_1 - a_2\|_1 < m \quad &\implies \quad \phi(a_1) - \phi(a_2) = a_1 - a_2 \\
&\implies \quad \|\phi(a_1) - \phi(a_2)\|_1 < m \\
&\implies \quad \psi(\phi(a_1)) - \psi(\phi(a_2)) = \phi(a_1) - \phi(a_2) = a_1 - a_2,
}
as well as
\eq{
\|\theta(a_1) - \theta(a_2)\|_1 < m \quad &\implies \quad \phi(a_1) - \phi(a_2) = \psi(\phi(a_1)) - \psi(\phi(a_2)) \\
&\implies \quad \|\phi(a_1) - \phi(a_2)\|_1 < m \\
&\implies \quad a_1 - a_2 = \phi(a_1) - \phi(a_2) = \psi(\phi(a_1)) - \psi(\phi(a_2)).
}
Indeed, $\theta$ is an isometry of degree $m$.
\end{proof}

A final observation about isometries is that they obey a certain extension property, which is proved below:
By allowing the maximum degree of an isometry to be lowered by at most 2, we can expand its domain  by one unit in every direction.
If the maximum degree is infinite (i.e.~$\phi$ acts by translations), then the extension can be repeated \textit{ad infinitum} to recover the translation on all of $\Z^d$, for any copy of $\Z^d$ intersecting the domain.

\begin{lemma} \label{extension}
Suppose that $\phi : A \to \N \times \Z^d$ is an isometry of degree $m \geq 3$.
Then $\phi$ can be extended to an isometry $\Phi: A^{(1)} \to \N \times \Z^d$ of degree $m-2$, where
\eq{
A^{(1)} \coloneqq \{v \in \N \times \Z^d : \|u - v\|_1 \leq 1 \text{ for some $u \in A$}\} \supset A.
}
By induction, if $\phi$ has $\deg(\phi) \geq 2k + m$, then $\phi$ can be extended to an isometry $\Phi : A^{(k)} \to \N \times \Z^d$ of degree $m$, where
\eq{
A^{(k)} \coloneqq \{v \in \N \times \Z^d : \|u - v\|_1 \leq k \text{ for some $u \in A$}\} \supset A.
}
\end{lemma}

\begin{proof}
For $v \in A^{(1)}$ such that $\|u - v\|_1 \leq 1$ with $u \in A$, define
\eq{
\Phi(v) \coloneqq  \phi(u) + (v - u).
}
If $u' \in A$ also satisfies $\|u' - v\|_1 \leq 1$, then
\eq{
\|u - u'\|_1 \leq 2 \quad &\implies \quad \phi(u) - \phi(u') = u - u'
\quad \\
&\implies \quad \phi(u) + (v - u) = \phi(u') + (v - u').
}
So $\Phi$ is well-defined; in particular, $\Phi(u) = \phi(u)$ for all $u \in A$.
To see that $\Phi$ is an isometry of degree $m-2$, consider any $v,v' \in A^{(1)}$ and take $u,u' \in A$ such that $\|u - v\|_1 \leq 1$ and $\|u' - v'\|_1 \leq 1$.
If $\|v - v'\|_1 < m - 2$, then
\eq{
\|u - u'\|_1 < m \quad &\implies \quad \phi(u) - \phi(u') \hspace{0.3ex} = u - u' \\
&\implies \quad \Phi(v) - \Phi(v') = \phi(u) + (v - u)  - \phi(u') - (v' - u') = v - v'.
}
Alternatively, if $\|\Phi(v) - \Phi(v')\|_1 < m -2$, then
\eq{
\|\phi(u) - \phi(u')\| < m
\quad &\implies \quad u - u' = \phi(u) - \phi(u') \\
&\phantom{\implies \quad u - u' }\hspace{0.7ex}  = \Phi(v) - (v-u) - \Phi(v') + (v'-u') \\
&\implies \quad v - v' \hspace{0.3ex} = \Phi(v) - \Phi(v').
}
Indeed, $\deg(\Phi) \geq m-2$.
\end{proof}

We can now define the desired metric on functions.
Roughly speaking, an isometry allows for the comparison of the large values of two functions. 
The size of the isometry's domain reflects how many of their large values are similar, while the degree of the isometry captures how similar their relative positioning is.
The metric is constructed in stages.

First, given an isometry $\phi: A \to \N \times \Z^d$ 
and $\alpha > 1$, we define the $\alpha$-distance function according to $\phi$:
\eeq{ \label{d_phi_def}
d_{\alpha,\phi}(f,g) \coloneqq \alpha\sum_{u \in A} |f(u) - g(\phi(u))| + \sum_{u \notin A} f(u)^\alpha + \sum_{u \notin \phi(A)} g(u)^\alpha + 2^{-\deg(\phi)}, \quad f,g \in \SS_0.
}
Next we take the optimal $\alpha$-distance over all isometries with finite\footnote{Finiteness is needed for measurability reasons; see Lemma \ref{S_meas}.} domains:
\eq{
d_\alpha(f,g) \coloneqq \inf_{\substack{\phi : A \to \N\times\Z^d \\ \deg(\phi) \geq 1 \\ A \text{ finite}}} d_{\alpha,\phi}(f,g).
}
Since $\deg(\phi^{-1}) = \deg(\phi)$, it is easy to see that $d_{\alpha,\phi^{-1}}(g,f) = d_{\alpha,\phi}(f,g)$, and so the function $d_\alpha$ is symmetric:
\eeq{
d_\alpha(f,g) = d_\alpha(g,f) \quad \text{for all } f,g \in \SS_0. \label{symmetry}
}
It also satisfies the triangle inequality.

\begin{lemma}
For any $f,g,h \in \SS_0$,
\eeq{
d_\alpha(f,h) \leq d_\alpha(f,g) + d_\alpha(g,h). \label{triangle}
}
\end{lemma}

\begin{proof}
Fix $\eps > 0$, and choose isometries $\phi : A \to \N \times \Z^d$ and $\psi : B \to \N \times \Z^d$ such that
\eq{
d_{\alpha,\phi}(f,g) < d_\alpha(f,g) + \eps \quad \text{and} \quad
d_{\alpha,\psi}(g,h) < d_\alpha(g,h) + \eps.
}
Define $\theta : A' \to \N \times \Z^d$ as in Lemma \ref{composition}.
We have
\eeq{
d_{\alpha,\theta}(f,h) = \alpha\sum_{u \in A'} |f(u) - h(\theta(u))| + \sum_{u \notin A'} f(u)^\alpha + \sum_{u \notin \theta(A')} h(u)^\alpha + 2^{-\deg(\theta)}. \label{theta_distance}
}
The first sum above can be bounded as
\eeq{
\alpha\sum_{u \in A'} |f(u) - h(\theta(u))| &\leq \alpha\sum_{u \in A'} \bigl(|f(u) - g(\phi(u))| + |g(\phi(u)) - h(\psi(\phi(u)))|\bigr)  \\
&= \alpha\sum_{u \in A'} |f(u) - g(\phi(u))| + \alpha\sum_{u \in B \cap \phi(A)} |g(u) - h(\psi(u))|. \label{tri_bound_1}
}
Now, the Lipschitz norm of the function $t \mapsto t^\alpha$ on $[0,1]$ is $\alpha$, meaning
\eq{
f(u)^\alpha \leq \alpha|f(u)-g(v)| + g(v)^\alpha \quad \text{for any $u,v \in \N \times \Z^d$}.
}
As a result, the second sum in \eqref{theta_distance} satisfies
\eeq{
\sum_{u \notin A'} f(u)^{\alpha} &= \sum_{u \in A \setminus A'} f(u)^\alpha + \sum_{u \notin A} f(u)^\alpha  \\
&\leq \sum_{u \in A \setminus A'}\bigl( \alpha|f(u) - g(\phi(u))| + g(\phi(u))^\alpha\bigr) + \sum_{u \notin A} f(u)^\alpha  \\
&\leq \alpha\sum_{u \in A \setminus A'} |f(u) - g(\phi(u))| + \sum_{u \notin B} g(u)^\alpha + \sum_{u \notin A} f(u)^\alpha. \label{tri_bound_2}
}
Similarly, the third sum satisfies
\eeq{
\sum_{u \notin \theta(A')} h(u)^\alpha &= \sum_{u \in \psi(B) \setminus \theta(A')} h(u)^\alpha + \sum_{u \notin \psi(B)} h(u)^{\alpha}  \\
&\leq \sum_{u \in B \setminus \phi(A)} \bigl(\alpha|h(\psi(u)) - g(u)| + g(u)^\alpha\bigr) + \sum_{u \notin \psi(B)} h(u)^{\alpha}  \\
&\leq \alpha\sum_{u \in B \setminus \phi(A)} |g(u) - h(\psi(u))| + \sum_{u \notin \phi(A)} g(u)^{\alpha}+ \sum_{u \notin \psi(B)} h(u)^{\alpha}. \label{tri_bound_3}
}
Finally, Lemma \ref{composition} guarantees 
\eeq{
\deg(\theta) &\geq \min(\deg(\phi),\deg(\psi))  \\
\implies \quad 2^{-\deg(\theta)} &\leq 2^{-\deg(\phi)} + 2^{-\deg(\psi)}. \label{tri_bound_4}
}
Using \eqref{tri_bound_1}--\eqref{tri_bound_4} in \eqref{theta_distance}, we find
\eq{
d_\alpha(f,h) \leq d_{\alpha,\theta}(f,h) \leq d_{\alpha,\phi}(f,g) + d_{\alpha,\psi}(g,h) < d_\alpha(f,g) + d_\alpha(g,h) + 2\eps.
}
As $\eps$ is arbitrary, \eqref{triangle} follows.
\end{proof}

From \eqref{symmetry} and \eqref{triangle}, we see that $d_\alpha$ is a pseudometric on $\SS_0$.
It does not, however, separate points.
Nevertheless, one can construct a metric space $(\SS,d_\alpha)$ by taking the quotient of $\SS_0$ with respect to the equivalence relation
\eq{
f \equiv g \quad \iff \quad d_\alpha(f,g) = 0.
}
We call $\SS$ the space of \textit{partitioned subprobability measures}, and it naturally inherits the metric $d_\alpha$.
Corollary \ref{better_def_cor} below shows that for distinct $\alpha,\alpha' > 1$, we have $d_\alpha(f,g) = 0$ if and only if $d_{\alpha'}(f,g) = 0$.
Therefore, we are justified in not decorating the space $\SS$ with an $\alpha$ parameter, since $\SS_0/(d_\alpha = 0)$ is always the same set.

We shall write without confusion the symbol $f$ for both the equivalence class in $\SS$ and the representative chosen from $\SS_0$.
When $f$ is evaluated at a particular element $u \in \N \times \Z^d$, an explicit representative has been chosen.
In the sequel, it will be important that certain properties of elements of $\SS_0$ are invariant within the equivalence classes and thus well-defined in $\SS$.
To identify the equivalence classes, we have the following result.

\begin{prop} \label{superisometry}
Two functions $f,g \in \SS_0$ satisfy $d_\alpha(f,g) = 0$ if and only if
there is a set $B\subset \N \times \Z^d$ and a map $\psi : B \to \N \times \Z^d$ such that 
\begin{itemize}
\item[(i)] $f(u) = g(\psi(u))$ for all $u \in B$,
\item[(ii)] $f(u) = 0$ for all $u \notin B$,
\item[(iii)] $g(u) = 0$ for all $u \notin \psi(B)$, and
\item[(iv)] $\psi(u) - \psi(v) = u - v$ for all $u,v \in B$.
\end{itemize}
\end{prop}

\begin{proof}
The ``if" direction is straightforward to prove.
Suppose such a map $\psi$ exists.
For any $\eps > 0$, properties (ii) and (iii) allow us to take $A \subset B$ to be finite but large enough that
\eq{
\sum_{u \notin A} f(u)^\alpha + \sum_{u \notin \psi(A)} g(u)^\alpha < \eps. 
}
Let $\phi: A \to \N \times \Z^d$ be the restriction of $\psi$ to $A$.
Then (i) says $f(u) - g(\phi(u)) = 0$ for all $u \in A$, and (iv) shows $\deg(\phi) = \infty$.
Therefore, 
\eq{
d_\alpha(f,g) \leq d_{\alpha,\phi}(f,g) = \sum_{u \notin A} f(u)^\alpha + \sum_{u \notin \phi(A)} g(u)^\alpha < \eps.
}
As $\eps$ is arbitrary, we must have $d_\alpha(f,g) = 0$.

Now we prove the converse.
Assume $d_\alpha(f,g) = 0$.
Let $u_0$ be any element of $\N\times\Z^d$ with $f(u_0)>0$. 
With $\kappa \coloneqq  f(u_0)$, consider the sets $A \coloneqq  \{u \in\N\times\Z^d:f(u)=\kappa\}$ and $A' \coloneqq  \{u \in\N\times\Z^d:g(u)=\kappa\}$.
Because $\|g\|\le 1$, there is some $\delta \in(0,\kappa)$ such that 
if $g(u)\neq\kappa$, 
then $|\kappa-g(u)| \ge\delta$. Let $\eps = \min\{\kappa^\alpha,\alpha\delta\}$ and choose any isometry $\phi$ with finite domain $D$ such that $d_{\alpha,\phi}(f,g)<\eps$. 
Note that $A \subset D$, since otherwise there would exist some $u \in A\setminus D$ so that $\eps >d_{\alpha,\phi}(f,g) \geq f(u)^\alpha=\kappa^\alpha\ge \eps$. 
Moreover, every $u \in A$ must satisfy $g(\phi(u))=f(u)$, since otherwise
$\eps >d_{\alpha,\phi}(f,g)\ge \alpha|f(u)-g(\phi(u))| \ge \alpha\delta\ge\eps$. 
Hence $\phi(A) \subset A'$, implying by injectivity of $\phi$ that the cardinality of $A$ is at most that of $C$.
Reversing the roles of $f$ and $g$ shows that these two sets actually have the same (finite) cardinality, meaning $\phi(A)=A'$. 
Since $u_0$ is arbitrary, it follows that the set of positive values of $f$ and those of $g$ are identical.
Denote this set by $V$.

First suppose $V$ is a finite set and let $\kappa$ denote the smallest element of $V$.
Take $\delta\in(0,\kappa)$ such that the distance between any two distinct elements of $V$ is at least $\delta$. 
Because $\|f\| \le 1$ and $\|g\|\le 1$, the sets
$B\coloneqq \{u : f(u)\in V\}$ and $B' \coloneqq  \{u : g(u)\in V\}$ are also finite. 
Therefore, there is an $m\in \N$ large enough that if $u,v\in B \cup B'$ with $\|u-v\|_1<\infty$, then $\|u-v\|_1<m$. 
Let $\eps = \min\{\kappa^\alpha, \alpha\delta,2^{-m}\}$ and
choose any isometry $\phi$ with finite domain $D$ such that $d_{\alpha,\phi}(f,g)<\eps$. 
By the same arguments as in the first paragraph, it follows that $B \subset D$, $g(\phi(u)) = f(u)$ for every $u \in B$, and $\phi(B) = B'$.
Hence $\psi \coloneqq  \phi\big|_{B}$ satisfies (i)--(iii).
Furthermore, it is immediate from the choice of $\eps$ that $\deg(\psi)\geq\deg(\phi)>m$. 
Consequently, whenever $u,v\in B$ satisfy $\|u-v\|_1<\infty$ (and thus $\|u-v\|_1<m$) or $\|\phi(u)-\phi(v)\|_1 < \infty$ (and thus $\|\phi(u)-\phi(v)\|_1 < m$), we have $\phi(u)-\phi(v) = u-v$. 
So $\psi$ also satisfies (iv).

Now suppose $V = \{\kappa_1 > \kappa_2 > \cdots\}$ is an infinite set.
For each $n\in \N$, let $V_n=\{\kappa_i : i=1,\dots, n\}$. 
Arguing as in the previous paragraph, we can find an isometry $\psi_n$ with domain $B_n\coloneqq \{u: f(u)\in V_n\}$ satisfying (i) and (iv) with $B$ replaced by $B_n$, and also satisfying $\psi_n(B_n) = B_n' \coloneqq \{u : g(u)\in V_n\}$. 
To complete the proof, we will identify a subsequence of $(\psi_n)_{n=1}^\infty$ that converges to the desired $\psi$. 
For each $n\ge1$, the function $\psi_n\big|_{B_1}$ maps $B_1$ bijectively onto $B_1'$.
Because there are only finitely many such mappings, there must be infinitely many $n_k$ that produce the same mapping.
Let $n_1$ denote the first such index (that is at least $2$) and take $\psi\big|_{B_1}$ to be equal to $\psi_{n_1}\big|_{B_1}$. 
Then repeat the argument along the subsequence $(\psi_{n_k})_{k=1}^\infty$ to define $\psi\big|_{B_2}$ as an extension of $\psi\big|_{B_2}$. 
Continued indefinitely, this procedure yields a map $\psi$ defined on all of $B \coloneqq  \bigcup_{n=1}^\infty B_n$, for which (i) and (iv) hold by construction.
Considering that $B=\{u : f(u)>0\}$ and 
\eq{
\psi(B)=\bigcup_{n=1}^\infty \psi(B_n) = \bigcup_{n=1}^\infty B_n' = \{u : g(u) > 0\},
}
(ii) and (iii) obviously hold as well.
\end{proof}

A more transparent description of equivalence under $d_\alpha$ was given in the introductory Section~\ref{endpoint_results}.
We restate it here, and prove equivalence to the conditions of Proposition \ref{superisometry}.
Recall that the $\N$-support of $f \in \SS_0$ is the set
\eeq{ \label{Nsupport_def}
H_f \coloneqq  \{n \in \N : f(n,x) > 0 \text{ for some $x \in \Z^d$}\}.
}

\begin{cor} \label{better_def_cor}
Let $f,g \in \SS_0$ have $\N$-supports denoted by $H_f$ and $H_g$, respectively.
Then $d_\alpha(f,g) = 0$ if and only if there is a bijection $\tau : H_f \to H_g$ and vectors $(x_n)_{n \in H_f} \subset \Z^d$ such that 
\eeq{
g(\tau(n),x) = f(n,x+x_n) \quad \text{for all $n \in H_f$, $x \in \Z^d$.} \label{better_def}
}
\end{cor}

\begin{proof}
First assume $d_\alpha(f,g) = 0$, and take $\psi : B \to \N \times \Z^d$ as in Proposition \ref{superisometry} so that properties (i)--(iv) hold.
By repeatedly applying Lemma \ref{extension},
one may assume that the domain $B$ is a union of copies of $\Z^d$; that is, $B = H \times \Z^d$.
By property (ii), we can take $H = H_f$, while property (iv) shows first that for every $n \in H_f$, there is $\tau(n) \in \N$ and $x_n \in \Z^d$ satisfying
\eq{
\psi(n,x) = (\tau(n),x-x_n) \quad \text{for all $x \in \Z^d$,}
}
and second that $n \mapsto \tau(n)$ is injective.
Then (i) leads to \eqref{better_def}, while (iii) guarantees that $\tau(H_f) = H_g$. Conversely, assume $\tau : H_f \to H_g$ and $(x_n)_{n \in H_f}$ satisfy \eqref{better_def}.
Then define $\psi : H_f \times \Z^d \to H_g \times \Z^d$ by $\psi(n,x) \coloneqq  (\tau(n),x-x_n)$.
Since the domain of $\psi$ is $H_f \times \Z^d$, and its range is $H_g \times \Z^d$, this map satisfies (ii) and (iii).
At the same time, (iv) holds by construction, and (i) follows from~\eqref{better_def}.
\end{proof}

We can now state a clear geometric condition for a function on $\SS_0$ to be well-defined on $\SS$.
For the functions that will be of interest to us, it will be obvious that the hypotheses of the following corollary are satisfied.
For instance, the norm functional $\|\cdot\|$ defined in \eqref{norm_def} does not depend on the choice of representative, and we can thus safely write $\|f\|$ without reference to a particular representative.

\begin{cor} \label{defined_pspm}
Suppose $L$ is a function on $\SS_0$ satisfying the following:
\begin{itemize}
\item[(i)] L is invariant under shifts of $\Z^d$: If there are vectors $(x_n)_{n \in \N}$ in $\Z^d$ such that $f(n,x) = g(n,x-x_n)$, then $L(f) = L(g)$.
\item[(ii)] L is invariant under permutations of $\N$: If there is a bijection $\tau : \N \to \N$ such that $f(n,x) = g(\tau(n),x)$, then $L(f) = L(g)$.
\item[(iii)] L is invariant under zero-padding: If there is an increasing sequence $(n_k)_{k \geq 1}$ in $\N$ such that
\eq{
f(n,x) = \begin{cases}
g(k,x) &\text{if } n = n_k, \\
0 &\text{otherwise},
\end{cases}
}
then $L(f) = L(g)$.
\end{itemize}
Then $L$ is well-defined on $\SS$ by evaluating at any representative in $\SS_0$.
\end{cor}

\begin{proof}
Suppose $f,g \in \SS_0$ are such that $d_\alpha(f,g) = 0$.
We wish to show that $L(f) = L(g)$.
Let $H_f$ and $H_g$ be the $\N$-supports of $f$ and $g$ respectively, and let $\tau : H_f \to H_g$ be a bijection satisfying \eqref{better_def} with vectors $(x_n)_{n \in H_f}$.
If $H_f$ is finite, then $\tau$ can be trivially extended to a bijection on $\N$, and then invariance properties (i) and (ii) are sufficient to show $L(f) = L(g)$.
If $H$ is infinite, then we enumerate the sets
\eq{
H_f = \{n_1 < n_2 < \cdots\}, \qquad H_g = \{m_1 < m_2 < \cdots\},
}
and define the functions $h_f \in \SS_0$ by $h_f(k,x) = f(n_k,x)$ and $h_g(\ell,x) = g(m_\ell,x)$.
By (iii), we have $L(f) = L(h_f)$ and $L(g) = L(h_g)$.
Furthermore, $\tau$ induces a bijection $\tau' : \N \to \N$ by
\eq{
\tau : n_k \mapsto m_\ell \quad \iff \quad \tau' : k \mapsto \ell.
}
We now have
\eq{
h_f(k,x) = f(n_k,x) = g(\tau(n_k),x - x_n) = h_g(\tau'(k),x - x_n),
}
so that (i) and (ii) give $L(h_f) = L(h_g)$.
So $L(f) = L(g)$ in this case as well.
\end{proof}

In the next lemma we discuss our first example of a function $L$ satisfying the hypotheses of Corollary \ref{defined_pspm}.
This function will be important in defining the ``update procedure" of Section~\ref{transformation}. 

\begin{lemma} \label{norm_equivalence}
The map $\|\cdot\| : \SS \to [0,1]$ defined by \eqref{norm_def} is lower semi-continuous and thus measurable.
\end{lemma}

It is worth noting to the reader that the proof below 
is similar to the arguments for several later results. 

\begin{proof}[Proof of Lemma \ref{norm_equivalence}]
It is clear that $\|\cdot\| : \SS_0 \to [0,1]$ satisfies (i)--(iii) in Corollary \ref{defined_pspm}, and so the map $f \mapsto \|f\|$ is well-defined on $\SS$.
To prove lower semi-continuity, it suffices to fix $f \in \SS$, let $\eps > 0$ be arbitrary, and find $\delta > 0$ sufficiently small that
\eq{
d_\alpha(f,g) < \delta \quad \implies \quad \|g\| > \|f\| - \eps.
}
Upon selecting a representative $f \in \SS_0$, we can find $A \subset \N \times \Z^d$ finite but large enough that
$\sum_{u \notin A} f(u) < \frac{\eps}{2}$.
By possibly omitting some elements of $A$, we may assume that $f$ is strictly positive on $A$ (if $f$ is the constant zero function, this results in $A = \varnothing$).

Now take $0 < \delta < \inf_{u \in A} f(u)^\alpha$, where an infimum over the empty set is $\infty$.
We will further assume $\delta/\alpha < \eps/2$.
If $d_\alpha(f,g) < \delta$, then there is a representative $g \in \SS_0$ and an isometry $\phi : C \to \N \times \Z^d$ such that $d_{\alpha,\phi}(f,g) < \delta$.
It follows that $A \subset C$, since otherwise we would have $d_{\alpha,\phi}(f,g) \geq f(u)^\alpha > \delta$ for some $u \in A \setminus C$.
Hence
\eq{
\|g\| \geq \sum_{u \in \phi(A)} g(u)
&\geq \sum_{u \in A} f(u) - \sum_{u \in A} |f(u) - g(\phi(u))| + \sum_{u \notin A} f(u) - \sum_{u \notin A} f(u) \\
&> \|f\| - \frac{d_{\alpha,\phi}(f,g)}{\alpha} - \frac{\eps}{2} > \|f\| - \frac{\delta}{\alpha} - \frac{\eps}{2} > \|f\| - \eps.
}
\end{proof}

\subsection{Compactness} \label{compactness}
We now state 
the key compactness result for the metric space $(\SS,d_\alpha)$.

\begin{thm} \label{compactness_result}
$(\SS,d_\alpha)$ is a compact metric space.
\end{thm}

A continuum version of this result is \cite[Theorem 3.2]{mukherjee-varadhan16}. 
The proof of Theorem \ref{compactness_result}, however, strongly uses discreteness of $\Z^d$ to deal with the explicit metric $d_\alpha$.

\begin{proof}
Since compactness is equivalent to sequential compactness in metric spaces, it suffices to prove that every sequence $(f_n)_{n \geq 1}$ in $(\SS,d_\alpha)$ contains a converging subsequence.  
That is, there is $f \in \SS$ and a subsequence $(f_{n_k})_{k \geq 1}$ such that
\eq{
\lim_{k \to \infty} d_\alpha(f_{n_k},f) = 0.
}

For each $n$, fix a representative $f_n \in \SS_0$. 
The elements of $\N \times \Z^d$ can be ordered so that $f_n$ attains its $k^{\text{th}}$ largest value at the $k^{\text{th}}$ element on the list.
Ties are broken according to some arbitrary but fixed rule (e.g.~assigning an order to $\Z^d$ and following the lexicographic ordering induced on $\N \times \Z^d$).
This information is recorded by defining
\eq{
u_{n, k} \coloneqq  u \in \N \times \Z^d \text{ at which $f_n$ attains its $k^{\text{th}}$ largest value}.
}
Of interest will be the quantities
\eq{
t_n(k,\ell) \coloneqq  u_{n, k} - u_{n, \ell} \in \Z^d \cup \{\infty\}.
}
Here we consider $\Z^d \cup \{\infty\}$ as the one-point compactification of $\Z^d$ (cf.~Steen and Seebach \cite[page 63]{seebach-steen78}).\footnote{The $\ell^1$ metric on $\Z^d$ induces the discrete topology, so every subset of $\Z^d$ is open.  
In addition, if $K \subset \Z^d$ is compact (here equivalent to bounded), then $\{\infty\} \cup (\Z^d \setminus K)$ is defined to be open.  Any open cover of $\Z^d \cup \{\infty\}$ must contain a set of the form $\{\infty\} \cup (\Z^d \setminus K)$. 
The complement $\Z^d \setminus K$ is a finite set.
The open cover, therefore, must have a finite subcover.}
Since this space is second countable, compactness is equivalent to sequential compactness (see \cite[page 22]{seebach-steen78}).
Consequently, by passing to a subsequence, we may assume that for every $k,\ell \in \N$, there is $t(k,\ell) \in \Z^d \cup \{\infty\}$ so that
\eq{
t_n(k,\ell) \to t(k,\ell) \quad \text{as } n \to \infty.
}
Because convergence takes place in the discrete space $\Z^d \cup \{\infty\}$, either $t(k,\ell) \in \Z^d$ or $t(k,\ell) = \infty$.
In the former case, $t_n(k,\ell) = t(k,\ell)$ for all $n$ sufficiently large.
On the other hand, $t(k,\ell) = \infty$ when either $t_n(k,\ell)$ is finite infinitely often but $\liminf_n \|t_n(k,\ell)\|_1 = \infty$, or $t_n(k,\ell) = \infty$ for all $n$ sufficiently large.

Since $\|f_n\| \leq 1$, we necessarily have
\eeq{
0 \leq f_n(u_{n, k}) \leq \frac{1}{k} \quad \text{for all } n,k. \label{f_n_bound}
}
So by passing to a further subsequence, it may be assumed that
\eeq{
\eps_k \coloneqq  \lim_{n\to\infty} f_n(u_{n, k}) \quad \text{exists for all } k, \label{eps_convergence}
}
where $\eps_k$ satisfies
\eeq{
0 \leq \eps_k \leq \frac{1}{k}. \label{eps_bound}
}
Next, inductively define
\eq{
\ell(1) &\coloneqq  1, \\
\ell(r) &\coloneqq  \min\{k > \ell(r-1)\ :\ t(k,\ell(s)) = \infty \text{ for all } s \leq r - 1\}, \quad r \geq 2.
}
It may be the case that only finitely many $\ell(r)$ can be defined, if the set considered in the definition is empty for some $r$.
In any case, let
\eq{
R \coloneqq  \text{number of $r$ for which $\ell(r)$ is defined} \leq \infty.
}
Clearly each $\ell(r)$ is distinct.
In words, $\ell(r)$ is the next integer past $\ell(r-1)$ such that the distance between $u_{n, \ell(r)}$ and each of $u_{n, \ell(1)}, u_{n, \ell(2)}, \dots ,u_{n, \ell(r-1)}$ is tending to $\infty$ with $n$.

Consider any fixed $k \in \N$.
There is some $r$ such that $t(k,\ell(r))$ is finite.
For instance, if $k = \ell(r)$, then $t(k,\ell(r)) = t(\ell(r),\ell(r)) = 0$.
If $\ell(r) < k < \ell(r+1)$, then $t(k,\ell(s))$ is finite for some $s \leq r$, by the definition of $\ell(r+1)$.
Similarly, if $R < \infty$ and $k > \ell(R)$, then $t(k,\ell(s))$ is finite for some $s \leq R$.
Now given the existence of $r$ with $t(k,\ell(r))$ finite, we claim there is a unique such $r$.
Indeed, if $t(k,\ell(r))$ and $t(k,\ell(r'))$ are both finite, then
\eq{
t(\ell(r),\ell(r')) 
&= \lim_{n \to \infty} (u_{n, \ell(r)} - u_{n, \ell(r')})\\
&= \lim_{n \to \infty} \bigl(-(u_{n, k} - u_{n, \ell(r)}) + (u_{n, k} - u_{n, \ell(r')}) \bigr)\\
&= -t(k,\ell(r)) + t(k,\ell(r'))
\ne \infty,
}
which forces $r = r'$.
This discussion shows it is possible to define
\eq{
r_k \coloneqq  \text{unique $r$ such that } t(k,\ell(r)) \text{ is finite.}
}
Let us summarize the construction thus far.
For each $k \in \N$, consider the point
\eq{
v_k \coloneqq  \Big(r_k, t\big(k,\ell(r_k)\big)\Big) \in \N \times \Z^d.
}
That is, $v_k$ is in the $r_k^{\text{th}}$ copy of $\Z^d$, at the finite limit point of $u_{n, k} - u_{n, \ell(r_k)}$.
Moreover, $r_k$ is the unique $r$ such that $u_{n, k} - u_{n, \ell(r)}$ converges to a finite limit.
So there are $R$ copies of $\Z^d$ populated by the $v_k$, the $r^{\text{th}}$ copy containing those $v_k$ for which $u_{n, k}$ and $u_{n, \ell(r)}$ remain close as $n$ grows large.
The actual point in that copy of $\Z^d$, at which $v_k$ is located, encodes the limiting difference $t(k,\ell(r))$, prescribing exactly how those two locations are separated for large $n$.
More precisely, there is some $N_k$ such that
\eeq{
u_{n, k} - u_{n, \ell(r_k)} = t\big(k,\ell(r_k)\big) \in \Z^d \quad \text{for all $n \geq N_k$}. \label{relative_finite}
}
Finally, recall that these locations give the order statistics on the values of $f_n$.
The first coordinate of $v_k$ indicates that the location of the $k^{\text{th}}$ largest value of $f_n$ remains close to the location of the $\ell(r_k)^{\text{th}}$ largest value of $f_n$, as $n$ tends to infinity.
The $R$ different locations of the $\ell(r)^{\text{th}}$ largest values of $f_n$ may not converge, but they serve as moving reference points to which all other locations remain close.

With the definitions made above, it is now possible to construct the limit function $f$.
Set
\eq{
f(u) \coloneqq  \begin{cases} 
\eps_k &\text{if } u = v_k, \\
0 &\text{otherwise.}
\end{cases}
}
In order for $f$ to be well-defined, it must be checked that the $v_k$ are distinct.
Suppose $v_k = v_{k'}$.
That is, $r_{k} = r_{k'}$ and $t\big(k,\ell(r_k)\big) = t\big(k',\ell(r_{k'})\big)$, meaning that for all $n \geq \max\{N_k,N_{k'}\}$,
\eq{
u_{n, k} - u_{n, \ell(r_k)} 
= t\big(k,\ell(r_k)\big)
&= t\big(k',\ell(r_{k'})\big) 
= u_{n, k'} - u_{n, \ell(r_{k'})} 
= u_{n, k'} - u_{n, \ell(r_k)}.
}
It follows that $u_{n, k} = u_{n, k'}$ for all sufficiently large $n$.
That this holds for even one value of $n$ implies $k = k'$.

In order for $f$ to be an element of $\SS_0$, it must be checked that $\|f\| \leq 1$
($f$ is nonnegative due to \eqref{eps_bound}), which is a straightforward task:
Let $K \in \N$ be fixed.
Given $\eps > 0$, from \eqref{eps_convergence} we may choose $N$ large enough that
\eq{
|f_N(u_{N, k}) - \eps_k| \leq\frac{\eps}{K} \quad \text{for } k = 1,2,\dots,K.
}
Then
\eq{
\sum_{k = 1}^K \eps_k \leq \sum_{k = 1}^K \bigg(f_N(u_{N, k}) + \frac{\eps}{K}\bigg) \leq \|f_N\| + \eps \leq 1 + \eps.
}
As $\eps$ is arbitrary, it follows that
\eq{
\sum_{k = 1}^K f(v_k) = \sum_{k=1}^K \eps_k \leq 1.
}
Letting $K$ tend to infinity yields the bound $\|f\| \leq 1$.

In some sense, the remaining goal of the proof is to show that $f_n(v_k) \to \eps_k$.
Of course, this is not true.
Rather, the pointwise convergence of $(f_n)_{n \geq 1}$ is given by \eqref{eps_convergence}.
Nevertheless, we have the convergence \eqref{relative_finite} of \textit{relative} coordinates.
Furthermore, for any $k$ and $k'$ such that $r_k \neq r_{k'}$, we have $t(k,k') = \infty$. 
Therefore, for any $M > 0$, there is $N_{k, k'}$ such that
\eeq{
\|u_{n, k} - u_{n, k'}\|_1 \geq M \quad \text{for all $n \geq N_{k, k'}$.} \label{relative_infinite}
}
So the desired convergence will hold after pre-composing $f$ with a suitable isometry.
The correct choice of isometry is described below.

Temporarily fix $K \in \N$.
For $n \in \N$, define $\phi_{n, K} : \{u_{n, 1},u_{n, 2},\dots,u_{n, K}\} \to \N \times \Z^d$ by
\eq{
\phi_{n, K}(u_{n, k}) \coloneqq  v_k, \quad k = 1,2,\dots,K.
}
Consider the case when $n \geq N_k$ for all $k = 1,2,\dots,K$.
Then \eqref{relative_finite} guarantees that $u_{n, k} - u_{n, \ell(r_k)} = t\big(k,\ell(r_k)\big)$ for each $k = 1,2,\dots,K$.
Consequently, for $1 \leq k,k' \leq K$ we have
\eq{
r_k = r_{k'} \quad &\implies \quad u_{n, k} - u_{n, k'} = t\big(k,\ell(r_k)\big) - t\big(k',\ell(r_{k'})\big)
= v_k - v_{k'} 
= \phi_{n, K}(u_{n, k}) - \phi_{n, K}(u_{n, k'}).
}
If it is further the case that $n \geq N_{k, k'}$ for all $1 \leq k,k' \leq K$ with $r_k \neq r_{k'}$, then \eqref{relative_infinite} shows
\eq{
r_k \neq r_{k'} \quad &\implies \quad 
\|u_{n, k} - u_{n, k'}\|_1 \geq M \quad \text{and} \quad \|\phi_{n, K}(u_{n, k}) - \phi_{n, K}(u_{n, k'})\|_1 = \|v_k - v_{k'}\|_1 = \infty.
}
Together, the previous two displays ensure $\deg(\phi_{n, K}) \geq M$ for large $n$.
Moreover, since $M$ in \eqref{relative_infinite} can be chosen arbitrarily large, we conclude that for fixed $K$,
\eeq{
\lim_{n \to \infty} \deg(\phi_{n, K}) = \infty. \label{degtoinf}
}
One can now verify convergence of $f_n$ to $f$ under the metric $d_\alpha$.
To do so, one must exhibit, for every $\eps > 0$, some $N \in \N$ satisfying the following condition:
For every $n \geq N$, there is some finite $A_n \subset \N \times \Z^d$ and some isometry $\phi_n : A_n\to \N \times \Z^d$ satisfying $d_{\alpha,\phi_n}(f_n,f) < \eps$.

So fix $\eps > 0$.
For each $K \in \N$, write $A_{n, K} = \{u_{n, 1},u_{n, 2},\dots,u_{n, K}\}$ for the domain of $\phi_{n, K}$.
Choose $K$ large enough that
\eq{
\sum_{k > K} \frac{1}{k^\alpha} < \frac{\eps}{4}.
}
It follows from \eqref{f_n_bound} that
\eeq{
\sum_{k > K} f_n(u_{n, k})^\alpha
< \frac{\eps}{4} \quad \text{for all $n$.} \label{final_bound1}
}
Similarly, from \eqref{eps_bound} we have
\eeq{
\sum_{k > K} f(v_k)^\alpha
= \sum_{k > K} \eps_k^\alpha
< \frac{\eps}{4}. \label{final_bound2}
}
Finally, in light of \eqref{eps_convergence} and \eqref{degtoinf}, it is possible to choose $N$ large enough that
\eeq{
\alpha \sum_{k = 1}^K |f_n(u_{n, k}) - \eps_k| < \frac{\eps}{4} \quad \text{for all } n \geq N,\label{final_bound3}
}
and that
\eeq{
2^{-\deg(\phi_{n, K})} < \frac{\eps}{4} \quad \text{for all $n \geq N$}. \label{final_bound4}
}
Using \eqref{final_bound1}--\eqref{final_bound4}, one arrives at
\eq{
d_{\alpha,\phi_{n, K}}(f_n,f) &=
\alpha\sum_{u \in A_{n,K}} |f_n(u) - f(\phi_{n, K}(u))|
+ \sum_{u \notin A_{n, K}} f_n(u)^\alpha 
+ \sum_{v \notin \phi_{n, K}(A_{n, K})} f(v)^\alpha + 2^{-\deg(\phi_{n, K})} \\
&= \alpha\sum_{k=1}^K |f_n(u_{n, k}) - \underbrace{f(v_k)}_{= \eps_k}|
+ \sum_{k > K} f_n(u_{n, k})^\alpha 
+ \sum_{k > K} f(v_k)^\alpha + 2^{-\deg(\phi_{n, K})}  \\
&< \frac{\eps}{4} + \frac{\eps}{4} + \frac{\eps}{4} + \frac{\eps}{4} = \eps
}
for all $n \geq N$.
In particular,
\eq{
d_\alpha(f_n,f) \leq d_{\alpha,\phi_{n, K}}(f_n,f) < \eps \quad \text{for all } n \geq N.
}
One concludes $d_\alpha(f_n,f) \to 0$ as $n \to \infty$, as desired.
\end{proof}

The compactness guaranteed by Theorem \ref{compactness_result} 
will be used to establish results in the following metric space of probability measures.
Denote by $\PP(\SS)$ the set of Borel probability measures on $\SS$, and equip this space with the \textit{Wasserstein metric} (see \cite[Definition 6.1]{villani09})\footnote{Note that we are always using the Wasserstein-$1$ metric, and the subscript $\alpha$ refers only to the subscript of $d_\alpha$.
For instance, $\WW_2$ would be the Wasserstein-$1$ metric associated with the metric $d_2$, \textit{not} a Wasserstein-$2$ metric.}:
\eq{
\WW_\alpha(\nu,\rho) \coloneqq  \inf_{\pi \in \Pi(\nu,\rho) }\int_{\SS \times \SS} d_\alpha(f,g)\, \pi(\dd f,\dd g),
}
where $\Pi(\nu,\rho)$ denotes the set of probability measures on $\SS \times \SS$ having $\nu$ and $\rho$ as marginals.
Note that we may use dual representation of $\WW$ due to Kantorovich~\cite{kantorovich42} when it is convenient:
\eeq{
\WW_\alpha(\nu,\rho) = \sup_{\vphi} \biggl(\int_\SS \vphi(f)\ \nu(\dd f) - \int_\SS \vphi(f)\ \rho(\dd f)\biggr),\label{kantorovich}
}
where the supremum is over $1$-Lipschitz functions $\vphi : (\SS,d_\alpha) \to\R$.

Now we recall some general facts on convergence in $\PP(\XX)$, where $\XX$ is any Polish space.
It is a standard result (e.g.~see \cite[Theorem 6.9]{villani09}) that 
the Wasserstein distance metrizes the topology of weak convergence on $\PP(\XX)$.
Furthermore, if $\XX$ is compact, then weak convergence is equivalent to convergence of continuous test functions, and the topology on $\PP(\XX)$ is also compact (\cite[Remark 6.19]{villani09}).
In the coming sections, we will employ weak convergence to prove convergence of not only continuous test functions, but also semi-continuous test functions.
Therefore, we will repeatedly apply the Portmanteau lemma, which we record here so that it can be properly quoted later in the chapter.

\begin{lemma}[{Portmanteau, cf.~\cite[Theorem 2.1]{billingsley99} and \cite[Theorem 1.3.4]{vandervaart-wellner96}}] \label{portmanteau}
Given a function $L: \SS \to \R$, define the map $\LL : \PP(\SS) \to \R$ by
\eq{
\LL(\rho) \coloneqq  \int_\SS L(f)\ \rho(\dd f).
}
If $L$ is lower (resp.~upper) semi-continuous, then $\LL$ is lower (resp.~upper) semi-continuous with respect to $\WW_\alpha$.
\end{lemma}

\subsection{Equivalence of metrics} \label{generalized_topology}
We have introduced a family of metrics $(d_\alpha)_{\alpha>1}$ on $\SS$, where the flexibility of choosing $\alpha$ sufficiently close to $1$ will allow us to accommodate $\beta$ arbitrarily close to $\beta_\mathrm{max}$ in \eqref{mgf_assumption}.
It is useful to know that each metric induces the same topology.
The next proposition verifies this fact.

\begin{prop} \label{metrics_equivalent}
For any $\alpha,\alpha' > 1$, $f \in \SS$, and sequence $(f_n)_{n\geq1}$ in $\SS$, we have $d_{\alpha}(f,f_n) \to 0$ as $n \to \infty$ if and only if $d_{\alpha'}(f,f_n) \to 0$.
\end{prop}

\begin{proof}
Since $\alpha$ and $\alpha'$ are interchangeable in the claim, it suffices prove the ``only if" direction.
That is, we assume $d_\alpha(f,f_n) \to 0$ as $n \to \infty$.
Fix representatives $f,f_n \in \SS_0$.
Given $\eps > 0$, set
\eq{
\delta \coloneqq \min\Big\{\Big(\frac{\eps}{4}\Big)^{\alpha/(\alpha'-1)},\frac{\alpha}{\alpha'}\Big(\frac{\eps}{4}\Big),\frac{\eps}{4}\Big\}.
}
Then choose $N$ sufficiently large that $d_\alpha(f,f_n) < \delta$ for all $n \geq N$.
In particular, for any such $n$, there is an isometry $\phi_n : A_n \to \N \times \Z^d$ satisfying
$d_{\alpha,\phi_n}(f,f_n) < \delta$.
In particular,
\eq{
\alpha'\sum_{u \in A_n} |f(u)-f_n(\phi_n(u))| \leq \frac{\alpha'}{\alpha} d_{\alpha,\phi_n}(f,f_n) < \frac{\alpha'}{\alpha}\delta < \frac{\eps}{4}.
}
Also,
\eq{
\sum_{u \notin A_n} f(u)^{\alpha'} \leq \max_{u \notin A_n} f(u)^{\alpha'-1} \sum_{u \notin A_n} f(u)
\leq \Big(\max_{u \notin A_n} f(u)\Big)^{\alpha'-1} \leq d_{\alpha,\phi_n}(f,f_n)^{(\alpha'-1)/\alpha} 
< \delta^{(\alpha'-1)/\alpha} < \frac{\eps}{4},
}
and similarly
\eq{
\sum_{u \notin \phi_n(A_n)} f_n(u)^{\alpha'} < \delta^{(\alpha'-1)/\alpha} < \frac{\eps}{4}.
}
Finally,
\eq{
2^{-\deg(\phi_n)} \leq d_{\alpha,\phi_n}(f,f_n) < \delta < \frac{\eps}{4}.
}
These four inequalities together show
\eq{
d_{\alpha'}(f,f_n) \leq d_{\alpha',\phi_n}(f,f_n) < \eps \quad \text{for all $n \geq N$.}
}
As $\eps > 0$ is arbitrary, it follows that $d_{\alpha'}(f,f_n) \to 0$.

\end{proof}

We conclude this section with some observations that will be needed in later arguments.

\begin{lemma} \label{trivial_bound}
For any $f,g \in \SS$, $d_\alpha(f,g) \leq 2$.
\end{lemma}

\begin{proof}
Pick representatives $f,g \in \SS_0$ and let $\phi : \varnothing \to \N \times \Z^d$ be the empty isometry.
Then
\eq{
d_\alpha(f,g) \leq d_{\alpha,\phi}(f,g) = \sum_{u \in \N \times \Z^d} f(u)^\alpha + \sum_{u \in \N \times \Z^d} g(u)^\alpha 
\leq \sum_{u \in \N \times \Z^d} f(u) + \sum_{u \in \N \times \Z^d} g(u) = \|f\| + \|g\| \leq 2.
}
\end{proof}

The next lemma concerns measure-theoretic properties of the spaces $\SS_0$ and $\SS$.
In particular, $\SS_0$ is considered as a subset of
\eq{
\ell^1(\N \times \Z^d) = \{f : \N \times \Z^d \to \R : \|f\| < \infty\},
}
on which there is the standard $\ell^1$ norm $\|\cdot\|$ that extends \eqref{norm_def}:
\eq{
\|f\| \coloneqq  \sum_{u \in \N \times \Z^d} |f(u)|, \quad f \in \ell^1(\N \times \Z^d).
}

\begin{lemma} \label{S_meas}
Consider the metric spaces $\ell^1(\N \times \Z^d)$ and $\SS$ with their Borel $\sigma$-algebras.
Then the following statements hold: 
\begin{itemize}
\item[(a)] $\SS_0$ is a closed (in particular, measurable) subset of $\ell^1(\N \times \Z^d)$ and is thus itself a measurable space with the subspace $\sigma$-algebra.
\item[(b)] The quotient map $\iota: \SS_0 \to \SS$ that sends $f \in \SS_0$ to its equivalence class in $\SS$ is measurable.
\end{itemize}
\end{lemma}

\begin{proof}
To show $\SS_0 \subset \ell^1(\N \times \Z^d)$ is closed, we express $\SS_0$ as the 
intersection of closed sets:
\eq{
\SS_0 = \{f \in \ell^1(\N \times \Z^d) : \|f\| \leq 1\} \cap \bigg(\bigcap_{u \in \N \times \Z^d} \{f \in \ell^1(\N \times \Z^d) : f(u) \geq 0\}\bigg).
}
To next show $\iota$ is measurable, it suffices to verify that the inverse image of any open ball is measurable.
For $f \in \SS$ and $r > 0$, we write
\eq{
B_r(f) \coloneqq  \{g \in \SS : d_\alpha(f,g) < r\}.
}
Notice that
\eq{
\iota^{-1}(B_r(f)) = \bigcup_{\phi}\ \{g \in \SS_0 : d_{\alpha,\phi}(f,g) < r\},
}
where the union is over isometries with \textit{finite} domains.
The union occurs, therefore, over a countable set.
For each $\phi$, it is clear from \eqref{d_phi_def} that $d_{\alpha,\phi}(f,\, \cdot\, )$ is a measurable function on $\SS_0$, and so each set in the union is measurable.
Being a countable union of measurable sets, $\iota^{-1}(B_r(f))$ is measurable.
\end{proof}

\section{Comparison to the Mukherjee--Varadhan topology} 
\label{compare_topologies}
This section accomplishes two goals: (i) adapt the compactification technique of Mukherjee and Varadhan \cite{mukherjee-varadhan16} to measures on $\Z^d$, and (ii) prove that the metric in this adaptation, which is defined in terms of suitable test functions, is equivalent to the metric $d_\alpha$.
None of the facts proved here are needed in the rest of our study, and the reader will not encounter any lapse of presentation by skipping this section entirely. 
Rather, the results of this section are included to verify that our methods may achieve the same effect as those initiated in \cite{mukherjee-varadhan16}, while offering a more tractable metric with which to work.
Indeed, this is one way our approach capitalizes on the countability of $\Z^d$, although
the discussion that follows underscores the possibility that the abstract machinery can be made more general.

\subsection{Adaptation to the lattice}

In this preliminary section, we convert the Mukherjee--Varadhan setup to the discrete setting.
Aside from the proof of Proposition \ref{not_pseudo}, the construction is completely parallel to the one in \cite{mukherjee-varadhan16}.
For the sake of the ambitious reader, we mirror the notation of that article as closely as possible.
We will use boldface $\vc{x} = (x_1,x_2,\dots,x_k)$ to denote a vector in $(\Z^d)^k$.
For $\vc{x} \in (\Z^d)^k$ and $z \in \Z^d$, we use the notation
\eq{
\vc{x}+z \coloneqq  (x_1+z,x_2+z,\ldots,x_k+z).
}
For an integer $k \ge 2$, call a function $W : (\Z^d)^k \to \R$ \textit{translation invariant} if
\eeq{ \label{trans_invariant}
W(\vc{x}+z) = W(\vc{x}) \quad \text{for all $\vc{x} \in (\Z^d)^k,\ z \in \Z^d$}.
}
We will say such a function $W$ \textit{vanishes at infinity} if
\eeq{ \label{vanish_inf}
\lim_{\max_{i \neq j} \|x_i - x_j\|_1 \to \infty} W(\vc{x}) = 0.
}
Now let $\II_k$ denote the set of functions $W : (\Z^d)^k \to \R$ that are both translation invariant and vanishing at infinity.
The space $\II_k$ is naturally equipped with a metric by the uniform norm,
\eq{
\|W\|_\infty \coloneqq  \sup_{\vc{x} \in (\Z^d)^k} |W(\vc{x})|.
}
The condition \eqref{trans_invariant} means that $W$ depends only on the $k-1$ variables $x_2-x_1,x_3-x_1,\dots,x_{k}-x_{1}$.
That is,
\eq{
W(x_1,x_2,\dots,x_k) = w(x_2-x_1,x_3-x_1,\dots,x_{k}-x_{1}),
}
where $w : (\Z^d)^{k-1} \to \R$ is given by
\eq{
w(y_1,y_2,\dots,y_{k-1}) = W(0,y_1,y_2,\dots,y_{k-1}).
}
Then \eqref{vanish_inf} is equivalent to
\eeq{ \label{vanish_inf_2}
\lim_{\max_{1 \le i \le k-1} \|y_i\|_1 \to \infty} w(\vc{y}) = 0.
}
Since the space of functions $w : (\Z^d)^{k-1} \to \R$ satisfying \eqref{vanish_inf_2} is separable, it follows that $\II_k$ is separable.

The space of test functions will be 
\eq{
\II \coloneqq  \bigcup_{k \ge 2} \II_k.
}
As each $\II_k$ is separable, we can find
a countable dense subset $(W_r)_{r \in \N}$ of $\II$,
where $W_r \in \II_{k_r}$.
Notice that for any $W \in \II_k$ and any $f \in \SS$, the quantity
\eq{
I(W,f) \coloneqq  \sum_{n \in \N} \sum_{\vc{x} \in (\Z^d)^k} W(\vc{x}) \prod_{i = 1}^k f(n,x_i)
}
does not depend on the representative $f$ chosen from $\SS_0$.
(For the sake of exposition, we note that $I(W,f)$ is simply the sum of countably many integrals of $W$, the $n^{\text{th}}$ integral occurring  on the product space $((\Z^d)^k,f^{\otimes k}(n,\cdot))$,
where $f^{\otimes k}(n,\cdot)$ is the product measure whose every marginal has $f(n,\cdot)$ as its probability mass function.)
Indeed, if $f, g \in \SS_0$ are such that $d_\alpha(f,g) = 0$, then by Lemma \ref{better_def_cor} there exists a bijection $\tau$ between their $\N$-supports (denoted $H_f$ and $H_g$, respectively) and a collection $(z_n)_{n \in H_f}$ in $\Z^d$ such that
\eq{
f(n,x) = g(\tau(n),x-z_n), \quad x \in \Z^d.
}
In this case, \eqref{trans_invariant} gives
\eq{
\sum_{n \in \N} \sum_{\vc{x} \in (\Z^d)^k} W(\vc{x}) \prod_{i = 1}^k f(n,x_i)
&= \sum_{n \in H_f} \sum_{\vc{x} \in (\Z^d)^k} W(\vc{x}) \prod_{i = 1}^k g(\tau(n),x_i-z_n) \\
&= \sum_{n \in H_f} \sum_{\vc{x} \in (\Z^d)^k} W(\vc{x}+z_n) \prod_{i = 1}^k g(\tau(n),x_i)\\
&= \sum_{n \in \N} \sum_{\vc{x} \in (\Z^d)^k} W(\vc{x}) \prod_{i = 1}^k g(n,x_i).
}
So $I(W,\cdot)$ is well-defined on $\SS$.
We can thus define the following metric on $\SS$:
\eq{
D(f,g) \coloneqq  \sum_{r = 1}^\infty \frac{1}{2^r}\frac{1}{1+\|W_r\|_\infty} |I(W_r,f) - I(W_r,g)|.
}
Since the family $\{I(\cdot,f): f \in \SS\}$ is uniformly equicontinuous,
\eq{
|I(W_1,f) - I(W_2,f)| &= \bigg| \sum_{n \in \N} \sum_{\vc{x} \in (\Z^d)^k} (W_1(\vc{x}) - W_2(\vc{x})) \prod_{i = 1}^k f(n,x_i)\bigg| \\
&\leq \|W_1-W_2\|_\infty \sum_{n \in \N} \sum_{\vc{x} \in (\Z^d)^k} \prod_{i = 1}^k f(n,x_i) \\
&= \|W_1-W_2\|_\infty \sum_{n \in \N} \bigg(\sum_{x \in \Z^d} f(n,x)\bigg)^k \\
&\leq \|W_1-W_2\|_\infty \sum_{n \in \N} \sum_{x \in \Z^d} f(n,x) 
\leq \|W_1-W_2\|_\infty,
}
and $(W_r)_{r \in \N}$ is dense in $\II$, convergence in this metric implies convergence for \textit{all} test functions:
\eeq{ \label{convergence_criterion}
\lim_{j \to \infty} D(f_j,f) = 0 \quad \iff \quad
\lim_{j \to \infty} I(W,f_j) = I(W,f) \quad \text{for all $W \in \II$.}
}
It is clear that $D$ is reflexive and satisfies the triangle inequality.
It is nontrivial, however, that $D$ separates points.

\begin{prop} \label{not_pseudo}
For $f,g \in \SS_0$, $D(f,g) = 0$ if and only if $d_\alpha(f,g) = 0$.
\end{prop}

Therefore, $D$ is indeed a metric on $\SS$. The topology induced by $D$ on $\SS$ is the lattice analog of the topology defined in Mukherjee and Varadhan~\cite{mukherjee-varadhan16}.

To prove Proposition \ref{not_pseudo}, we will use the lemma below.
Recall that a sequence of real numbers $(a_j)$ \textit{lexicographically dominates} $(a_j')$ 
(which we denote by $(a_j) \succeq (a_j')$) 
if $a_j > a_j'$ for the smallest $j$ for which $a_j \neq a_j'$. 
Of course, if there is no such $j$, then the two sequences are equal. If the two sequences are not equal, then we will say that $(a_j)$ strictly dominates $(a_j')$, and write $(a_j) \succ (a_j')$.

We will say a collection of sequences $\{(a_{i,\, j})\}$ is lexicographically descending ``in $i$" if
\eq{
i \leq i' \quad \implies \quad (a_{i,\, j}) \succeq (a_{i',\, j}).
}
Given an infinite collection of sequences $\{(a_{i,\, j}) : i \in \N\}$, it is not always possible to rearrange the $i$-indices so that the collection is lexicographically descending.
One can easily check, however, that rearrangement  is possible for \textit{nonnegative} sequences satisfying the following condition:
\eeq{
|\{i : a_{i,\, j} > \eps\}|,\,  |\{i : b_{i,\, j} > \eps\}| < \infty \quad \text{for all $\eps > 0$, $j \geq 1$.} \label{all_j_zero}
}

\begin{lemma} \label{powers_lemma}
Let $\{(a_{i,\, j})_{j = 1}^\infty : 1 \leq i \leq N_1\}$ and $\{(b_{i,\, j})_{j = 1}^\infty : 1 \leq i \leq N_2\}$ be two collections of sequences in $[0,1]$, where $N_1,N_2 \in \N \cup \{\infty\}$.
Suppose that \eqref{all_j_zero} holds and
\eeq{
a_{i,\, 1},\, b_{i,\, 1} > 0 \quad \text{for all $i$,} \label{first_terms_pos}
}
so that we may assume each collection is lexicographically descending in $i$.
If, for every $\ell \in \N$,
\eeq{
&\sum_{i = 1}^{N_1} a_{i,\, 1}\prod_{j = 1}^\ell a_{i,\, j}^{p_j} = \sum_{i = 1}^{N_2} b_{i,\, 1}\prod_{j = 1}^\ell b_{i,\, j}^{p_j} < \infty
\quad \text{for all integers $p_1,\dots,p_\ell \geq 0$ with $\sum_{j=1}^\ell p_j \geq 1$,} \label{tests_agree}
}
then $N_1 = N_2$, and $a_{i,\, j} = b_{i,\, j}$ for every $i$ and $j$.
\end{lemma}

\begin{proof}
We will show by induction that for each finite $\ell$, $a_{i,\, \ell} = b_{i,\, \ell}$ for all $i$.
First consider the case when $\ell = 1$.
Since $a_{1,\, 1} = \max_i a_{i,\, 1}$ and $b_1 = \max_i b_{i,\, 1}$, we have
\eq{
a_{1,\, 1} = \lim_{p \to \infty} \bigg(\sum_{i = 1}^{N_1} a_{i,\, 1}^p\bigg)^{1/p}
\stackrel{\eqref{tests_agree}}{=} \lim_{p \to \infty} \bigg(\sum_{i = 1}^{N_2} b_{i,\, 1}^p\bigg)^{1/p}
= b_{1,\, 1}.
}
But then this argument can be repeated with the sequences $\{a_{i,\, 1} : 2 \leq i \leq N_1\}$ and $\{b_{i,\, 1} : 2 \leq i \leq N_2\}$ to obtain $a_{2,\, 1} = b_{2,\, 1}$.
Continuing in this way, one exhaustively determines that $N_1 = N_2 = N$, and $a_{i,\, 1} = b_{i,\, 1}$ for every $i$.

For $\ell \geq 2$, 
assume by induction that for each $j \leq \ell-1$, we have $a_{i,\, j} = b_{i,\, j}$ for every $i$.
By hypothesis \eqref{tests_agree}, for any nonnegative integers $q_1,q_2,\dots,q_{\ell-1}$ with $q_1 \geq 1$,
\eeq{ \label{max_agree}
\max_{i} \bigg(a_{i,\, \ell} \prod_{j = 1}^{\ell-1} a_{i,\, j}^{q_j}\bigg)
&= \lim_{p \to \infty} \bigg[\sum_{i = 1}^N \bigg( a_{i,\, \ell} \prod_{j = 1}^{\ell-1} a_{i,\, j}^{q_j}\bigg)^p\bigg]^{1/p} \\
&= \lim_{p \to \infty} \bigg[\sum_{i = 1}^N \bigg( b_{i,\, \ell} \prod_{j = 1}^{\ell-1} b_{i,\, j}^{q_j}\bigg)^p\bigg]^{1/p}
= \max_{i} \bigg(b_{i,\, \ell} \prod_{j = 1}^{\ell-1} b_{i,\, j}^{q_j}\bigg).
}
We now specify $q_1,q_2,\dots,q_{\ell-1}$ via the following backward induction:
\begin{itemize}
\item If $a_{1,\, \ell-1} = 0$, then set $q_{\ell-1} = 0$.
Otherwise, take $q_{\ell-1} = 1$.
\item Given $q_{\ell-1},q_{\ell-2},\dots,q_{k+1}$, choose $q_k$ as follows:
\begin{itemize}
\item If $a_{1,\, k} = 0$, then set $q_{k} = 0$.
\item Otherwise, consider all $i$ such that the sequences $(a_{i,\, j})_{j = 1}^{\ell-1}$ and $(a_{1,\, j})_{j=1}^{\ell-1}$ first differ when $j = k$.
If it exists, the smallest such $i$, call it $i_k$, will maximize $a_{i,\, k} < a_{1,\, k}$  (in particular, $a_{1,\, k} > 0$).
We then take $q_k$ sufficiently large that
\eeq{
\Big(\frac{a_{i_k,\,k}}{a_{1,\, k}}\Big)^{q_k} < \prod_{j = k+1}^{\ell-1} a_{1,\, j}^{q_j}.
\label{prod_construction}
}
If no such $i$ exists, then set $q_k = 0$.
Notice that \eqref{first_terms_pos} and \eqref{all_j_zero} force $q_1 \geq 1$.
\end{itemize}
\end{itemize}
Having defined $q_1,q_2,\dots,q_{\ell-1}$ to satisfy \eqref{prod_construction}, we obtain the following implication:
\begin{align}
(a_{1,\, j})_{j = 1}^{\ell-1} \succ (a_{i,\, j})_{j = 1}^{\ell-1} \quad
&\implies \quad
\exists\ k,\text{ $a_{1,\, j}$ and $a_{i,\, j}$ first differ at $j = k$} \label{differ}  \\
&\implies \quad
 \prod_{j = 1}^{\ell-1} \frac{a_{i,\, j}^{q_j}}{a_{1,\, j}^{q_j}}
 =  \prod_{j = k}^{\ell-1} \frac{a_{i,\, j}^{q_j}}{a_{1,\, j}^{q_j}}
\leq \Big(\frac{a_{i_k,\, k}}{a_{1,\, k}}\Big)^{q_k} \prod_{j = k+1}^{\ell-1} \frac{1}{a_{1,\, j}^{q_j}} < 1 \label{k_bound} \\
&\implies \quad \lim_{p \to \infty} \bigg(\prod_{j = 1}^{\ell-1} \frac{a_{i,\, j}^{q_j}}{a_{1,\, j}^{q_j}}\bigg)^p = 0. \label{ratio_to_0}
\end{align}
Moreover, because the first inequality in \eqref{k_bound} holds for all $i$ for which \eqref{differ} is true, the convergence in \eqref{ratio_to_0} is uniform over such $i$.
It follows that
\eq{
\lim_{p \to \infty} \max_{i} \bigg[a_{i,\, \ell} \bigg(\prod_{j = 1}^{\ell-1} \frac{a_{i,\, j}^{q_j}}{a_{1,\, j}^{q_j}}\bigg)^p\bigg] = a_{1,\, \ell}.
}
Since the choice of $q_1,q_2,\dots,q_{\ell-1}$ depended only on the $a_{i,\, j}$ with $j \leq \ell-1$, the induction hypothesis gives the same result for the $b$-collection:
\eq{
\lim_{p \to \infty} \max_{i} \bigg[b_{i,\, \ell}\bigg(\prod_{j = 1}^{\ell-1} \frac{b_{i,\, j}^{q_j}}{b_{1,\, j}^{q_j}}\bigg)^p\bigg] = b_{1,\, \ell}.
}
Now \eqref{max_agree}, with the fact that $a_{1,\, j} = b_{1,\, j}$ for $j \leq \ell-1$, allows us to conclude $a_{1,\, \ell} = b_{1,\, \ell}$.
As in the $\ell = 1$ case, we can repeat the above argument with the collections $\{(a_{i,\, j})_{j = 1}^\ell : 2 \leq i \leq N\}$ and $\{(b_{i,\, j})_{j = 1}^\ell : 2 \leq i \leq N\}$ to determine $a_{2,\, \ell} = b_{2,\, \ell}$.
Indeed, induction gives $a_{i,\, \ell} = b_{i,\, \ell}$ for every $i$.
\end{proof}

\begin{proof}[Proof of Proposition \ref{not_pseudo}]
The ``if" direction follows from the fact that $D$ is well-defined.
For the converse, we assume $D(f,g) = 0$ and prove $d_\alpha(f,g) = 0$.
Let $\{z_1,z_2,\dots\}$ be any enumeration of $\Z^d$.
Given integers $\ell \geq 1$ and $p_1,p_2,\dots,p_\ell \geq 0$ with $k\coloneqq  1 + \sum_{j=1}^\ell p_j \geq 2$, consider the following member of $\II_k$:
\eq{
W(\vc{x}) \coloneqq  \begin{cases}
1 &\text{if }x_{i} - x_1 = z_j  \text{ for all $1+\sum_{t = 1}^{j-1} p_t < i \leq 1+ \sum_{t = 1}^{j} p_t$, $1 \leq j \leq \ell$,}\\
0 &\text{otherwise}.
\end{cases}
}
Since $D(f,g) = 0$, we have
\begin{subequations} \label{tests_agree_2}
\begin{align}
I(W,f) &= \sum_{u \in \N \times \Z^d} f(u) \prod_{j = 1}^\ell f(u+z_j)^{p_j}
= \sum_{v \in \N \times \Z^d} g(v) \prod_{j = 1}^\ell g(v+z_j)^{p_j}
= I(W,g). 
\end{align}
Furthermore, these quantities are finite:
\begin{align}
\sum_{u \in \N \times \Z^d} f(u) \prod_{j = 1}^\ell f(u+z_j)^{p_j}
\leq \sum_{u \in \N \times \Z^d} f(u) \leq 1.
\end{align}
\end{subequations}
Now we lexicographically order (descending in $i$) the sequences given by
\eq{
a_{i,\, j} \coloneqq  f(u_i+z_j), \quad f(u_i) > 0, \qquad
b_{i,\, j} \coloneqq  g(v_i+z_j), \quad g(v_i) > 0,
}
for which \eqref{all_j_zero} is true because $\|f\|, \|g\| \leq 1$, and \eqref{tests_agree} is equivalent to \eqref{tests_agree_2}.
Therefore, Lemma \ref{powers_lemma} shows
\eeq{
f(u_i+z_j) = g(v_i+z_j) \quad \text{for all $i$ and $j$}. \label{lemma_consequence}
}
Let $H_f$ and $H_g$ denote the $\N$-supports of $f$ and $g$, respectively.
Since $z_j$ ranges over all of $\Z^d$, \eqref{lemma_consequence} implies that for every $n \in N_f$, there is $m \in H_g$ such that $f(n,\cdot)$ and $g(m,\cdot)$ are translates of each another.
The proof that $d_\alpha(f,g) = 0$ will be complete if we can show that for each $n \in H_f$, a \textit{distinct} $m \in H_g$ can be chosen, since then we would have an injection $\tau : H_f \to H_g$ such that $f(n,\cdot)$ and $g(\tau(n),\cdot)$ are always translates.
By interchanging $f$ and $g$, we would then see that $\tau$ is necessarily a bijection, and so Lemma \ref{better_def_cor} gives $d_\alpha(f,g) = 0$.

We now verify the final fact needed from above: $m \in H_g$ can be chosen distinctly for each $n \in H_f$.
Suppose that $f(n_1,\cdot),f(n_2,\cdot),\dots,f(n_K,\cdot)$ are translates of one another, where $n_1,\dots,n_K$ are distinct elements of $H_f$.
Even when $K$ is chosen maximally, $\|f\| \leq 1$ forces $K$ to be finite.
We can choose indices $i_1,i_2,\dots,i_K$ to simultaneously ``align" all these translates:
\eq{
u_{i_k} \in \{n_k\} \times \Z^d  \text{ and } 
f(u_{i_k}+z_j) = f(u_{i_1}+z_j) \quad \text{for all $j \geq 1$, $k = 1,2,\dots,K$.} \nonumber
}
By \eqref{lemma_consequence}, we then have
\begin{align}
g(v_{i_k}+z_j) &= g(v_{i_1}+z_j)  \quad \text{for all $j \geq 1$, $k = 1,2,\dots,K$.} \label{duplicate_copies}
\end{align}
Since $n_1,\dots,n_K$ are distinct, so too are $i_1,\dots,i_K$, and therefore $v_{i_1},\dots,v_{i_K}$ are distinct.
Upon writing $v_{i_k} = (m_k,y_k)$, we claim that $m_k \neq m_\ell$ for $k \neq \ell$.
Indeed, if $m_k = m_\ell$, then distinctness forces $y_k \neq y_\ell$.
Therefore, $z\coloneqq y_\ell-y_k$ is nonzero and satisfies $v_{i_\ell} = v_{i_k}+z$.
In particular,
\eq{
g(v_{i_\ell}) = g(v_{i_k}+z) \stackrel{\eqref{duplicate_copies}}{=} g(v_{i_\ell}+z).
}
More generally, for any integer $q$,
\eq{
g(v_{i_\ell}+qz) = g(v_{i_k} + (q+1)z) \stackrel{\eqref{duplicate_copies}}{=} g(v_{i_\ell}+(q+1)z).
}
Since $g(v_{i_\ell}) > 0$, it follows that
\eq{
\sum_{q = 0}^\infty g(v_{i_\ell}+qz) = \sum_{q = 0}^\infty g(v_{i_\ell}) = \infty,
}
an obvious contradiction to $\|g\| \leq 1$.
Now each $f(n_k,\cdot)$ is a translate of $g(m_k,\cdot)$, and $m_1,m_2,\dots,m_K$ are all distinct, as desired.
\end{proof}

\subsection{Equivalence of the metrics}
The metric $D$ does, in fact, give rise to a compact topology on $\SS$.
Rather than prove this directly, though, we first show that $d_\alpha$ induces a topology at least as fine as the one induced by $D$.
As the continuous image of a compact set is compact, this immediately implies $D$ also induces a compact topology.  
But the result actually implies more: The topologies are necessarily equivalent.

\begin{prop} \label{equal_topologies}
$D(f_j,f) \to 0$ if and only if $d_\alpha(f_j,f) \to 0$.
\end{prop}

The following topological fact reduces the proof of Proposition \ref{equal_topologies} to showing only one direction of the equivalence.

\begin{lemma}[{see \cite[Theorem 26.6]{munkres00}}] \label{topology_fact}
Suppose $\XX$ is a compact topological space, and $\YY$ is a Hausdorff topological space. 
If $F : \XX \to \YY$ is bijective and continuous, then $F$ must be a homeomorphism.
\end{lemma}

In the present setting, we consider $\XX = (\SS,d_\alpha)$ and $\YY = (\SS,D)$.
With $F$ equal to the identity map, Lemma \ref{topology_fact} says the following: If $d_\alpha(f_j,f) \to 0$ implies $D(f_j,f) \to 0$, then the converse is also true.

\begin{proof}[Proof of Proposition \ref{equal_topologies}]
By Lemma \ref{topology_fact}, it suffices to show that if $d_\alpha(f_j,f)$ converges to $0$, then $D(f_j,f) \to 0$.
And by \eqref{convergence_criterion}, it suffices to check that given any $W \in \II$, we have
$I(W,f_j) \to I(W,f)$.
So consider any $W \in \II_k$, and let $\eps > 0$ be given.
We seek a number $\delta > 0$ such that for any $g \in \SS$,
\eq{
d_\alpha(f,g) < \delta \quad \implies \quad |I(W,f) - I(W,g)| < \eps. 
}
This is trivial if $W$ is constant zero, and so we will henceforth assume $\|W\|_\infty > 0$.
By \eqref{vanish_inf}, there is $K$ large enough that
\eeq{ \label{K_choice_1}
\max_{i \neq j} \|x_i - x_j\|_1 \geq K \quad \implies \quad |W(\vc{x})| < \frac{\eps}{8}.
}
Upon fixing a representative $f \in \SS_0$, we can take $N \in \N$ such that
\eeq{ \label{small_after_N}
\sum_{n = N+1}^\infty \sum_{x \in \Z^d} f(n,x) < \frac{\eps}{8k\|W\|_\infty}.
} 
Next, for each $1 \leq n \leq N$, we choose $A_n \subset \Z^d$ finite but large enough that
\eeq{ \label{small_before_N}
\sum_{x \notin A_n} f(n,x) < \frac{\eps}{8kN\|W\|_\infty}.
}
Now define $A \coloneqq  \bigcup_{n = 1}^N (\{n\} \times A_n)$, so that \eqref{small_after_N} and \eqref{small_before_N} together show
\eeq{ \label{small_total}
\sum_{u \notin A} f(u) < \frac{\eps}{4k\|W\|_\infty}.
}
By possibly omitting some elements of $A$ and/or taking $N$ smaller, we may assume $f$ is strictly positive on $A$.
We may also assume 
\eq{
K \geq \sup_{1 \leq n \leq N} \diam(A_n),
}
since \eqref{K_choice_1} still holds if $K$ is made larger.
Now we choose $\delta > 0$ satisfying
\begin{subequations}
\begin{align}
\delta &< \inf_{u \in A} f(u)^\alpha, \label{delta_condition_1} \\
\delta &< 2^{-K}, \label{delta_condition_2} \\
\delta &< \frac{\eps}{8k\max\{1,|A|^{k-1}\}\|W\|_\infty}, \label{delta_condition_3} \\
{\delta}^{1/\alpha} &< \frac{\eps}{8k(2K)^{(k-1)d}\|W\|_\infty} \label{delta_condition_4}.
\end{align}
\end{subequations}
If $d_\alpha(f,g) < \delta$, then there is a representative $g \in \SS_0$ and an isometry $\phi : C \to \N \times \Z^d$ such that $d_{\alpha,\phi}(f,g) < \delta$.
The condition \eqref{delta_condition_1} implies $A \subset C$, while \eqref{delta_condition_2} guarantees that 
\eeq{
\deg(\phi) > K \geq \diam(A_n) \quad \text{for any $n \leq N$.} \label{deg_big_enough}
}
Consequently, $\phi$ acts by translation on $A_n$. 
That is, for each $n = 1,2,\dots,N$, there is $\tau(n) \in \N$ and $z_n \in \Z^d$ so that
\eeq{
\phi(n,x) = (\tau(n),x + z_n) \quad \text{for all $x \in A_n$.} \label{An_translate}
}
(Here $n \mapsto \tau(n)$ may not be injective.) We thus have
\eeq{ \label{triple_sum}
|I(W,f) - I(W,g)| &=
\bigg|\sum_{n \in \N} \sum_{\vc{x} \in (\Z^d)^k} W(\vc{x}) \prod_{i = 1}^k f(n,x_i) - \sum_{n \in \N} \sum_{\vc{x} \in (\Z^d)^k} W(\vc{x}) \prod_{i = 1}^k g(n,x_i)\bigg| \\
&\leq D_1 + D_2 + D_3,
}
where
\begin{align}
D_1 &\coloneqq  \bigg|\sum_{n \in \N} \sum_{\vc{x} \in (\Z^d)^k} W(\vc{x}) \prod_{i = 1}^k f(n,x_i)
- \sum_{n = 1}^N \sum_{\vc{x} \in A_n^k} W(\vc{x}) \prod_{i = 1}^k f(n,x_i)\bigg|, \label{D1} \\
D_2 &\coloneqq  \bigg|\sum_{n = 1}^N \sum_{\vc{x} \in A_n^k} W(\vc{x}) \prod_{i = 1}^k f(n,x_i)
- \sum_{n = 1}^N \sum_{\vc{x} \in A_n^k} W(\vc{x}+z_n) \prod_{i = 1}^k g(\tau(n),x_i+z_n)\bigg|, \nonumber \\
D_3 &\coloneqq  \bigg|\sum_{n = 1}^N \sum_{\vc{x} \in A_n^k} W(\vc{x}+z_n) \prod_{i = 1}^k g(\tau(n),x_i+z_n) - \sum_{n \in \N} \sum_{\vc{x} \in (\Z^d)^k} W(\vc{x}) \prod_{i = 1}^k g(n,x_i)\bigg|. \nonumber
\end{align}
We shall produce an upper bound for each of $D_1$, $D_2$, and $D_3$, and then use \eqref{triple_sum} to yield the desired result.

First, for $D_1$, notice that the summand  $W(\vc{x}) \prod_{i = 1}^k f(n,x_i)$ will appear in the first sum of \eqref{D1} but not the second if and only if $\vc{x} \notin A_{n}^k$ or $n > N$.
Considering the first case, we observe that $\vc{x} \notin A_n^k$ if and only if some $x_j$ does not belong to $A_n$. 
Hence
\eeq{ \label{bound_1_1}
\bigg| \sum_{n = 1}^N \sum_{\vc{x} \notin A_n^k} W(\vc{x}) \prod_{i = 1}^k f(n,x_i) \bigg|
&\leq \|W\|_\infty \sum_{n = 1}^N \sum_{\vc{x} \notin A_n^k} \prod_{i = 1}^k f(n,x_i) \\
&= \|W\|_\infty \sum_{n = 1}^N \sum_{j = 1}^k \sum_{x_j \notin A_n} f(n,x_j) \sum_{\vc{y} \in (\Z^d)^{k-1}} \prod_{i = 1}^{k-1} f(n,y_i) \\
&= \|W\|_\infty \sum_{n = 1}^N \sum_{j = 1}^k \sum_{x \notin A_n} f(n,x) \bigg(\sum_{y \in \Z^d} f(n,y)\bigg)^{k-1} \\
&\leq k\|W\|_\infty \sum_{n = 1}^N \sum_{x \notin A_n} f(n,x)
< \frac{\eps}{8},
}
where the final inequality is a consequence of \eqref{small_before_N}. 
Considering the second case, we have
\eeq{ \label{bound_1_2}
\bigg| \sum_{n = N+1}^\infty \sum_{\vc{x} \in (\Z^d)^k} W(\vc{x}) \prod_{i = 1}^k f(n,x_i) \bigg|
&\leq \|W\|_\infty \sum_{n = N+1}^\infty \sum_{\vc{x} \in (\Z^d)^k} \prod_{i = 1}^k f(n,x_i) \\
&= \|W\|_\infty \sum_{n = N+1}^\infty \bigg(\sum_{x \in \Z^d} f(n,x)\bigg)^k \\
&\leq \|W\|_\infty \sum_{n = N+1}^\infty \sum_{x \in \Z^d} f(n,x) < \frac{\eps}{8k} < \frac{\eps}{8},
}
where we have used \eqref{small_after_N} in the penultimate inequality.
Together, \eqref{bound_1_1} and \eqref{bound_1_2} yield
\eeq{ \label{bound_1}
D_1 < \frac{\eps}{4}.
}
Next we analyze the second difference, $D_2$, from \eqref{triple_sum}.
Here we use the translation invariance of $W$:
Making use of \eqref{An_translate}, we determine that
\eeq{ \label{before_telescope}
D_2 &= \bigg|\sum_{n = 1}^N \sum_{\vc{x} \in A_n^k} W(\vc{x}) \prod_{i = 1}^k f(n,x_i) - \sum_{n = 1}^N \sum_{\vc{x} \in A_n^k} W(\vc{x}+z_n) \prod_{i = 1}^k g(\tau(n),x_i+z_n)\bigg| \\
&= \bigg|\sum_{n = 1}^N \sum_{\vc{x} \in A_n^k} W(\vc{x}) \prod_{i = 1}^k f(n,x_i)
- \sum_{n = 1}^N \sum_{\vc{x} \in A_n^k} W(\vc{x}) \prod_{i = 1}^k g(\phi(n,x_i))\bigg| \\
&\leq \|W\|_\infty \sum_{n = 1}^N \sum_{\vc{x} \in A_n^k} \bigg|\prod_{i = 1}^k f(n,x_i) - \prod_{i = 1}^k g(\phi(n,x_i))\bigg|.
}
For $1 \leq n \leq N$ and $\vc{x} \in A_n^k$, we can use a telescoping sum to write
\eq{
&\bigg|\prod_{i = 1}^k f(n,x_i) - \prod_{i = 1}^k g(\phi(n,x_i))\bigg| \\
&= \bigg|\sum_{i = 1}^k f(n,x_1)\cdots f(n,x_{i-1})\big(f(n,x_i) - g(\phi(n,x_i))\big) \cdot g(\phi(n,x_{i+1}))\cdots g(\phi(n,x_k))\bigg| \\
&\leq\sum_{i = 1}^k |f(n,x_i) - g(\phi(n,x_i))|.
}
Therefore, \eqref{before_telescope} becomes 
\eq{
D_2 \leq \|W\|_\infty \sum_{n = 1}^N \sum_{\vc{x} \in A_n^k} \sum_{i = 1}^k |f(n,x_i) - g(\phi(n,x_i))|.
}
Now, given any $u = (n,x) \in A$, the summand $|f(u) - g(\phi(u))|$ will appear $k|A_n|^{k-1}$ times in the above sum: There must be some $i$ for which $x_i = x$, and the remaining $k-1$ coordinates of $\vc{x}$ can be any elements of $A_n$.
Using this fact and \eqref{delta_condition_3}, we arrive at
\eeq{ \label{bound_2}
D_2 &\leq k|A|^{k-1}\|W\|_\infty\sum_{u \in A} |f(u) - g(\phi(u))| < \frac{\eps}{8}.
}
Finally, we need to bound the third difference, $D_3$, in \eqref{triple_sum}.
Recall the map $n \mapsto \tau(n)$ from \eqref{An_translate}.
For each $\ell \in \N$, consider the partition of $(\Z^d)^k = J_1(\ell) \cup J_2(\ell)$, where
\eq{
J_1(\ell) &\coloneqq  \bigcup_{1 \leq n \leq N\, :\, \tau(n) = \ell} \{\vc{x} +z_n : \vc{x} \in A_n^k\}, \qquad
J_2(\ell) \coloneqq  (\Z^d)^k \setminus J_1(\ell).
}
In this notation, the product $W(\vc{y}) \prod_{i = 1}^k g(\ell,y_i)$ appears in the sum
\eq{
\sum_{n = 1}^N \sum_{\vc{x} \in A_n^k} W(\vc{x}+z_n) \prod_{i = 1}^k g(\tau(n),x_i+z_n)
}
if and only if $\vc{y} \in J_1(\ell)$.
Furthermore, in this case it will appear exactly once, since
\eq{
\big\{(\tau(n),x+z_n):x \in A_n\big\} \cap \big\{(\tau(m),x+z_m):x \in A_m\big\}
&= \phi(\{n\} \times A_n) \cap \phi(\{m\} \times A_m) = \varnothing
} 
for $n \neq m$.
Therefore,
\eeq{ \label{bound_3_rewrite}
D_3 &= \bigg|\sum_{n = 1}^N \sum_{\vc{x} \in A_n^k} W(\vc{x}+z_n) \prod_{i = 1}^k g(\tau(n),x_i+z_n)- \sum_{n \in \N} \sum_{\vc{x} \in (\Z^d)^k} W(\vc{x}) \prod_{i = 1}^k g(n,x_i)\bigg| \\
&= \bigg|\sum_{\ell \in \N} \sum_{\vc{x} \in J_2(\ell)} W(\vc{x}) \prod_{i = 1}^k g(\ell,x_i)\bigg|.
}
To analyze this quantity, we consider the further partition $J_2(\ell) = J_3(\ell) \cup J_4(\ell)$, where
\eq{
J_3(\ell) \coloneqq  \Big\{\vc{x} \in J_2(\ell) : \max_{i \neq j} \|x_i - x_j\|_1 \geq K\Big\}, \qquad
J_4(\ell) \coloneqq  J_2(\ell) \setminus J_3(\ell).
}
The sum over the various $J_3(\ell)$ is easy to control because of \eqref{K_choice_1}:
\eeq{ \label{bound_3_1}
\bigg|\sum_{\ell \in \N} \sum_{\vc{x} \in J_3(\ell)} W(\vc{x}) \prod_{i = 1}^k g(\ell,x_i)\bigg|
< \frac{\eps}{8} \sum_{\ell \in \N} \sum_{\vc{x} \in J_3(\ell)} \prod_{i = 1}^k g(\ell,x_i) 
&\leq \frac{\eps}{8} \sum_{\ell \in \N} \sum_{\vc{x} \in (\Z^d)^k} \prod_{i = 1}^k g(\ell,x_i) \\
&= \frac{\eps}{8} \sum_{\ell \in \N} \bigg(\sum_{x \in \Z^d} g(\ell,x)\bigg)^k \\
&\leq \frac{\eps}{8} \sum_{\ell \in \N} \sum_{x \in \Z^d} g(\ell,x)
\leq \frac{\eps}{8}.
}
Next considering $J_4(\ell)$, we make the following observation.
If $\vc{x} \in J_4(\ell)$, then there must be some coordinate $x_i$ for which $(\ell,x_i) \notin \phi(A)$.
Indeed, if $(\ell,x_i) = \phi(n,x_i')$ and $(\ell,x_j) = \phi(m,x_j')$ both belong to $\phi(A)$ (that is, $x_i' \in A_n$ and $x_j' \in A_m$), then
\eq{
\vc{x} \in J_4(\ell) \quad &\stackrel{\phantom{\eqref{deg_big_enough}}}{\implies} \quad
\|(\ell,x_i) - (\ell,x_j)\|_1 < K \\
&\stackrel{\eqref{deg_big_enough}}{\implies}\quad \|(n,x_i') - (m,x_j')\|_1 < K < \infty
\quad \implies \quad n = m.
}
Therefore, if it were the case that every coordinate $x_i$ satisfied $(\ell,x_i) \in \phi(A)$, then $\vc{x}$ would belong to $A_n^k + z_n$ for some $n$ such that $\ell = \tau(n)$.
This contradicts $J_4(\ell) \cap J_1(\ell) = \varnothing$.

We now consider one final (non-disjoint) partition: $J_4(\ell) = J_5(\ell) \cup J_6(\ell)$, where
\eq{
J_5(\ell) &\coloneqq  \{\vc{x} \in J_4(\ell) : (\ell,x_i) \in \phi(C \setminus A) \text{ for some $i$}\},
\intertext{and}
J_6(\ell) &\coloneqq  \{\vc{x} \in J_4(\ell) : (\ell,x_i) \notin \phi(C) \text{ for some $i$}\}.
}
\textit{A priori}, these definitions only imply $J_5(\ell) \cup J_6(\ell) \subset J_4(\ell)$, but the observation of the previous paragraph ensures that $J_5(\ell) \cup J_6(\ell) = J_4(\ell)$.
For the sum over the $J_5(\ell)$, there is a straightforward upper bound:
\eeq{
\bigg|\sum_{\ell \in \N} \sum_{\vc{x} \in J_5(\ell)} W(\vc{x}) \prod_{i = 1}^k g(\ell,x_i)\bigg|
&\leq \|W\|_\infty \sum_{\ell \in \N} \sum_{j = 1}^k \sum_{\vc{x} \in (\Z^d)^k\, :\, (\ell,x_j) \in \phi(C\setminus A)} \prod_{i = 1}^k g(\ell,x_i) \\
&= \|W\|_\infty \sum_{\ell \in \N} \sum_{j = 1}^k \sum_{x_j \in \Z^d\, :\, (\ell,x_j) \in \phi(C \setminus A)} g(\ell,x_j) \sum_{\vc{y} \in (\Z^d)^{k-1}} \prod_{i = 1}^{k-1} g(\ell,y_i) \\
&= k\|W\|_\infty \sum_{\ell \in \N} \sum_{x \in \Z^d\, :\, (\ell,x) \in \phi(C \setminus A)} g(\ell,x)\bigg(\sum_{y \in \Z^d} g(\ell,y)\bigg)^{k-1} \\
&\leq k\|W\|_\infty \sum_{\ell \in \N} \sum_{x \in \Z^d\, :\, (\ell,x) \in \phi(C \setminus A)} g(\ell,x) \\
&= k\|W\|_\infty \sum_{u \in C \setminus A} g(\phi(u)) \\
&\leq k\|W\|_\infty \bigg(\sum_{u \in C \setminus A} f(u) + \sum_{u \in C \setminus A} |f(u) - g(\phi(u))|\bigg) \\
&\leq k\|W\|_\infty\Big(\frac{\eps}{4k\|W\|_\infty} + d_{\alpha,\phi}(f,g)\Big)
< \frac{3\eps}{8} \label{bound_3_2},
}
where we have used \eqref{small_total} and \eqref{delta_condition_3} to establish the final two inequalities.
Now turning our focus to the sum over the $J_6(\ell)$, we define for each $x_1 \in \Z^d$ the set
\eq{
\NN(x_1) \coloneqq  \Big\{(x_2,\dots,x_{k}) \in (\Z^d)^{k-1} : \max_{1 \leq i < j \leq k} \|x_i - x_j\|_1 < K \Big\}.
}
Note that
\eq{
|\NN(x_1)| \leq (2K)^{(k-1)d},
}
since there are no more than $(2K)^d$ elements of $\Z^d$ at distance less than $K$ from $x_1$, and each of $x_2,\dots,x_{k}$ must satisfy this property.
By definition, if $\vc{x} \in J_4(\ell)$, then for every $j$ we have 
$(x_1,\dots,x_{j-1},x_{j+1},\dots,x_k) \in \NN(x_j)$.
Therefore, we obtain the bound
\eeq{ \label{first_attack}
\bigg|\sum_{\ell \in \N} \sum_{\vc{x} \in J_6(\ell)} W(\vc{x}) \prod_{i = 1}^k g(\ell,x_i)\bigg|
&\leq \|W\|_\infty \sum_{\ell \in \N} \sum_{j = 1}^{k} \sum_{\vc{x} \in J_6(\ell)\, :\, (\ell,x_{j}) \notin \phi(C)}\prod_{i = 1}^{k} g(\ell,x_i) \\
&\leq \|W\|_\infty \sum_{\ell \in \N} \sum_{j = 1}^k \sum_{x_j \in \Z^d\, :\, (\ell,x_j) \notin \phi(C)} \sum_{\vc{y} \in \NN(x_j)} g(\ell,x_j) \prod_{i=1}^{k-1} g(\ell,y_i) \\
&= k\|W\|_\infty \sum_{\ell \in \N} \sum_{x \in \Z^d\, :\, (\ell,x) \notin \phi(C)} \sum_{\vc{y} \in \NN(x)} g(\ell,x) \prod_{i = 1}^{k-1} g(\ell,y_i).
}
Now notice that
\eq{
\vc{y} = (y_1,y_2,\dots,y_{k-1}) \in \NN(x) \quad \iff \quad
(x,y_2,\dots,y_{k-1}) \in \NN(y_1),
}
and that
\eq{
(\ell,x) \notin \phi(C) \quad \implies \quad g(\ell,x) \leq \bigg({\sum_{u \notin \phi(C)} g(u)^\alpha}\bigg)^{1/\alpha} \leq {d_{\alpha,\phi}(f,g)}^{1/\alpha} < {\delta}^{1/\alpha}.
}
Therefore, we can rewrite \eqref{first_attack} as
\eeq{
\bigg|\sum_{\ell \in \N} \sum_{\vc{x} \in J_6(\ell)} W(\vc{x}) \prod_{i = 1}^k g(\ell,x_i)\bigg|
&\leq k\|W\|_\infty \sum_{\ell \in \N} \sum_{y \in \Z^d} \sum_{\vc{x} \in \NN(y)\, :\, (\ell,x_1) \notin \phi(C)} g(\ell,y) \prod_{i = 1}^{k-1} g(\ell,x_i) \\
&\leq k\|W\|_\infty \sum_{\ell \in \N} \sum_{y \in \Z^d}g(\ell,y) \sum_{\vc{x} \in \NN(y)} {\delta}^{1/\alpha}\prod_{i = 2}^{k-1} g(\ell,x_i) \\
&\leq k\|W\|_\infty (2K)^{(k-1)d}{\delta}^{1/\alpha} \sum_{\ell \in \N} \sum_{y \in \Z^d} g(\ell,y) \\
&\leq k\|W\|_\infty (2K)^{(k-1)d}{\delta}^{1/\alpha} 
\stackrel{\mbox{\footnotesize{\eqref{delta_condition_4}}}}{<} \frac{\eps}{8}. \label{bound_3_3}
}
In light of \eqref{bound_3_rewrite}, \eqref{bound_3_1}--\eqref{bound_3_3} now yield
\eeq{  \label{bound_3}
D_3 &=\bigg|\sum_{\ell \in \N} \sum_{\vc{x} \in J_2(\ell)} W(\vc{x}) \prod_{i = 1}^k g(\ell,x_i)\bigg| \\
&\leq \bigg|\sum_{\ell \in \N} \sum_{\vc{x} \in J_3(\ell)} W(\vc{x}) \prod_{i = 1}^k g(\ell,x_i)\bigg|  + \bigg|\sum_{\ell \in \N} \sum_{\vc{x} \in J_5(\ell)} W(\vc{x}) \prod_{i = 1}^k g(\ell,x_i)\bigg| + \bigg|\sum_{\ell \in \N} \sum_{\vc{x} \in J_6(\ell)} W(\vc{x}) \prod_{i = 1}^k g(\ell,x_i)\bigg|  \\
&< \frac{\eps}{8} + \frac{3\eps}{8} + \frac{\eps}{8} = \frac{5\eps}{8}.
}
Then using \eqref{bound_1}, \eqref{bound_2}, and \eqref{bound_3} in \eqref{triple_sum}, we find
\eq{
|I(W,f)-I(W,g)| < \frac{\eps}{4} + \frac{\eps}{8} + \frac{5\eps}{8} = \eps,
}
as desired.
\end{proof}

\section{The update map} \label{transformation}
Throughout the remainder of the chapter, we fix $\beta \in (0,\beta_{\max}) $ according to \eqref{mgf_assumption}, and we also fix some $\alpha > 1$ such that $\alpha\beta < \beta_{\max}$.
We then restrict our attention to $\SS$ equipped with the metric $d_\alpha$, and $\PP(\SS)$ with $\WW_\alpha$.
Proposition \ref{metrics_equivalent} tells us that the topology on $\SS$ does not depend on $\alpha$, although the same is not true for the topology on $\PP(\SS)$ induced by $\WW_\alpha$.
Indeed, there can exist functions $\vphi : \SS \to \R$ which are Lipschitz-1 with respect to some $d_\alpha$ but not Lipschitz at all with respect to some other $d_{\alpha'}$.

Consider the distribution of $\sigma_n$ under $\mu_n^\beta$. 
In this section we identify said distribution with an element $f_n \in \SS$ and define the Markov kernel $T$ that maps $f_n$ to the law of $f_{n+1}$ when conditioned on the random environment only up to time $n$.
In the following, the notation $u \sim v$ is used when $u,v \in \N \times \Z^d$. 
The same notation will be used for adjacent elements $x,y \in \Z^d$.

\subsection{Definition and continuity of the update map} \label{endpoint_distributions}
To make the setup precise, we recall the polymer measure $\mu_n^\beta$ defined by \eqref{polymer_measure_def} and the endpoint probability mass function $f_n : \Z^d \to \R$ given by
\eeq{
f_n(x) &\coloneqq  \mu_n^\beta(\sigma_n = x)  \\
&= \frac{Z_n(\beta,x)}{Z_n(\beta)} \\
&= \frac{1}{Z_n(\beta)}\sum_{0=\sigma_0,\sigma_1,\cdots,\sigma_{n}=x} \exp\bigg(\beta\sum_{i = 1}^{n} \omega(i,\sigma_i)\bigg)\e^{\beta\omega(n,x)}\bigg(\prod_{i=1}^{n} P(\sigma_{i-1},\sigma_i)\bigg). \label{fn_def}
}
Then $f_n$ is a $[0,1]$-valued function on $\Z^d$ and random depending on the environment $\omega$.
Its value at $x$ gives the probability that a polymer sampled from $\mu_n^\beta$ has $x$ as its endpoint.

When the polymer is extended from length $n$ to length $n+1$, the endpoint distribution updates to
\eq{
f_{n+1}(x)
&= \frac{1}{Z_{n+1}(\beta)}\sum_{0=\sigma_0,\sigma_1,\cdots,\sigma_n,\sigma_{n+1}=x} \exp\bigg(\beta\sum_{i = 1}^{n} \omega(i,\sigma_i)\bigg)\e^{\beta\omega(n+1,x)}\bigg(\prod_{i=1}^{n} P(\sigma_{i-1},\sigma_i)\bigg)P(\sigma_n,x) \\
&= \frac{Z_n(\beta)}{Z_{n+1}(\beta)}\sum_{\sigma_n\in\Z^d} f_n(\sigma_n)\e^{\beta \omega(n+1,x)}P(\sigma_n,x),
}
where the fact that $\sum_x f_{n+1}(x) = 1$ implies
\eeq{
\frac{Z_{n+1}(\beta)}{Z_{n}(\beta)} = \frac{Z_{n+1}(\beta)}{Z_{n}(\beta)}\sum_{x \in \Z^d} f_{n+1}(x) = \sum_{x \in \Z^d}\sum_{\sigma_n \in\Z^d} f_n(\sigma_n) \e^{\beta \omega(n+1,x)}P(y,x). \label{Zfrac}
}
We can thus write
\eq{
f_{n+1}(x) = \frac{\sum_{y\in\Z^d} f_n(y)\e^{\beta \omega(n+1,x)}P(y,x)}{\sum_{z\in\Z^d}\sum_{y\in\Z^d} f_n(y)\e^{\beta \omega(n+1,z)}P(y,z)}.
}
Recall that $(\omega(n+1,x))_{x \in \Z^d}$ is independent of $\FF_n$, while $f_n$ is measurable with respect to $\FF_n$.
Therefore, the above identity shows how $f_0 \mapsto f_1 \mapsto \cdots $ forms a Markov chain. 
Using the quotient map from Lemma \ref{S_meas}, we can embed this chain into $\SS$.
That is, we identify $f_{n}$ with its equivalence class in $\SS$ so that a representative takes values on $\N \times \Z^d$ instead of $\Z^d$. 
For concreteness, one could take the representative
\eeq{ \label{fn_def2}
f_n(k,x) = \begin{cases}
\frac{Z_n(\beta,x)}{Z_n(\beta)}  &\text{if $k = 1$,} \\
0 &\text{otherwise,}
\end{cases}
\quad (k,x)\in\N\times\Z^d.
}
Then the law of $f_{n+1} \in \SS$ given $f_{n} = f$ is the law of the random variable $F \in \SS$ defined by
\eeq{ \label{F_def_0}
F(u) = \frac{\sum_{v \in \N \times \Z^d} f(v)\e^{\beta\omega_u}P(v,u)}{\sum_{w\in \N \times \Z^d}\sum_{v\in \N \times \Z^d} f(v)\e^{\beta\omega_w}P(v,w)}, \quad u\in\N\times\Z^d,
}
where $(\omega_u)_{u \in \N \times \Z^d}$ is an i.i.d.~collection of random variables having law $\mathfrak{L}_\omega$, and
\eq{
P(v,u) = \begin{cases}
P(y,x) &\text{if $v = (n,y)$, $u = (n,x)$,} \\
0 &\text{if $v = (m,y)$, $u = (n,x)$, $m\neq n$.}
\end{cases}
}
To simplify notation, we write $v \sim u$ in the first case (i.e.~$v$ and $u$ have the same first coordinate) and $v \nsim u$ otherwise.

In addition to the Markov chain of endpoint distributions, there is the corresponding process $(\log Z_n(\beta))_{n\geq0}$ of free energies, which satisfies the recursion
\eq{
\log Z_{n+1}(\beta) &= \log Z_{n}(\beta) + \E\givenk[\Big]{\log \frac{Z_{n+1}(\beta)}{Z_{n}(\beta)}}{\FF_{n}} + (\log Z_{n+1}(\beta) - \E\givenk{\log Z_{n+1}(\beta)}{\FF_{n}}) \\
&= \log Z_n(\beta) + R(f_n) + d_n,
}
where 
\eeq{ \label{R_predef}
R(f_n)\coloneqq\E\givenk[\Big]{\log \frac{Z_{n+1}(\beta)}{Z_{n}(\beta)}}{\FF_{n}}
\stackrel{\mbox{\footnotesize\eqref{Zfrac}}}{=}\E\givenk[\bigg]{\log\sum_{x,y\in\Z^d} f_n(y)\e^{\beta \omega(n+1,x)}P(y,x)}{\FF_n},
}
and
\eeq{ \label{d_n_def}
d_n \coloneqq \log \frac{Z_{n+1}(\beta)}{Z_n(\beta)} - R(f_n).
} 
Notice that $d_n$ is a martingale increment, meaning the Doob decomposition of $(\log Z_n(\beta))_{n\geq0}$ is
\eeq{ \label{doob_decomposition}
\log Z_n(\beta) = \sum_{i=0}^{n-1} d_i + \sum_{i=0}^{n-1} R(f_i),
}
a fact which has been frequently used in the literature (e.g.~\cite{carmona-hu02,comets-shiga-yoshida03,comets-shiga-yoshida04,carmona-hu06}).

Next we wish to generalize the Markov chain of endpoint distributions to the entire space $\SS$, and then reinterpret the above observations about the free energy to match this generalization.
For a given $f\in\SS$, choose a representative (also called $f$) from $\SS_0$.
Consider the law of the random variable $F \in \SS$ whose representative is defined by
\eeq{
F(u) = \frac{\sum_{v\sim u} f(v)\e^{\beta\omega_u}P(v,u)}{\sum_{w\in\N\times\Z^d}\sum_{v\sim w} f(v)\e^{\beta\omega_w}P(v,w)+(1-\|f\|)\e^{\lambda(\beta)}}, \quad u\in\N\times\Z^d. \label{F_def}
}
Notice that \eqref{F_def} is a generalization of \eqref{F_def_0} because in the latter, $\|f_n\| = 1$.
The additional summand in the denominator of \eqref{F_def} is needed so that $F$ is defined even when $f \equiv 0$; its precise value is chosen so that \eqref{Zfrac} is generalized in an averaged form:
\eq{
\E\bigg[\sum_{u \in \N \times \Z^d} \sum_{v \sim u} f(v) \e^{\beta \omega_u}P(v,u) + (1 - \|f\|){\e^{\lambda(\beta)}}\bigg] 
&= {\e^{\lambda(\beta)}}
= \E\Big(\frac{Z_{n+1}(\beta)}{Z_n(\beta)}\Big).
}
It is thus apparent that the correct extension of \eqref{R_predef} is to define $R : \SS\to\R$ by
\eeq{
R(f) &\coloneqq  \E\log\wt F,\qquad \wt F \coloneqq \sum_{u \in \N \times \Z^d} \sum_{v \sim u} f(v)\e^{\beta \omega_u}P(v,u) + (1-\|f\|){\e^{\lambda(\beta)}}, \label{R_def}
}
where the expectation is over the environment $(\omega_u)_{u\in\N\times\Z^d}$.

\begin{remark}
Although the indexing of $\omega_u$ by $u \in \N \times \Z^d$ might appear to reflect a notion of time, we are not using $\N$ to consider time.
Rather, in order to compactify the space of measures on $\Z^d$, we needed to pass to subprobability measures on $\N \times \Z^d$.
To avoid confusion, we will never write $\N$ to index time.
Following this rule, we will write $\omega_u$ whenever we wish to think of a random environment on $\N \times \Z^d$, always at a \textit{fixed} time.
When considering the original random environment defining the polymer measures, we will follow the original $\omega(i,x)$ notation.
In either case, we will continue to use boldface $\vc\omega$ when referring to the entire collection of environment random variables.
\end{remark}

We must now check that the quantities defined above are measurable and do not depend on the representative $f\in \SS_0$.
That is, let us define $Tf$ to be the law of $F$ given $f \in \SS$, and then we need to show the following:
\begin{itemize}
\item[(i)] Given any $f \in \SS_0$, the map $\R^{\N\times\Z^d} \to \SS$ given by $\vc\omega \mapsto F$ is Borel measurable, where $\R^{\N\times\Z^d}$ is equipped with the product topology and product measure $(\mathfrak{L}_\omega)^{\otimes \N\times\Z^d}$, and $\mathfrak{L}_\omega$ is the law of $\omega$.
\item[(ii)] The law of $F$ does not depend on the representative $f \in \SS_0$.
\item[(iii)] The value of $R$ does not depend on the representative $f\in\SS_0$.
\end{itemize}
Claim (i) is immediate, since $\vc \omega \mapsto F$ is clearly a measurable map from $\R^{\N\times\Z^d}$ to $\SS_0$.
After all, $F$ can be expressed using only sums and quotients involving the measurable functions $(\e^{\beta\omega_u})_{u\in\Z^d}$.
And then $F \to \iota(F)$ from $\SS_0$ to $\SS$ is measurable by Lemma \ref{S_meas}.
Claims (ii) and (iii) are given by the following lemma.

\begin{prop} \label{same_law}
Suppose $f,g \in \SS_0$ satisfy $d_\alpha(f,g) = 0$.
Define $F$ and $\wt F$ as in~\eqref{F_def} and \eqref{R_def}, 
and similarly define
\eeq{ \label{G_def}
G(u) \coloneqq  \frac{\sum_{v \sim u} g(v) \e^{\beta \eta_u}P(v,u)}{\wt G}, \quad \wt G \coloneqq{\sum_{w \in \N \times \Z^d} \sum_{v \sim w} g(v)\e^{\beta \eta_w}P(v,w) + (1 - \|g\|){\e^{\lambda(\beta)}}}, 
}
where the $\eta_w$ variables are i.i.d.~with law $\mathfrak{L}_\omega$.
Then the following statements hold:
\begin{itemize}
\item[(a)] The law of $F \in \SS$ is equal to the law of $G \in \SS$.
\item[(b)] The law of $\wt F \in \R$ is equal to the law of $\wt G \in \R$.
In particular, $R(f) = R(g)$.
\end{itemize}
\end{prop}

\begin{proof}
To show the two claims, it suffices to exhibit a coupling of the environments $(\omega_u)_{u \in \N \times \Z^d}$ and $(\eta_u)_{u \in \N \times \Z^d}$ such that $F = G$ in $\SS$ and $\wt F = \wt G$ in $\R$.
So we let $H_f$ and $H_g$ denote the $\N$-supports of $f$ and $g$, respectively (recall the definition \eqref{Nsupport_def}), and take
$\tau : H_f \to H_g$ and $(x_n)_{n \in H_f}$ as in Corollary \ref{better_def_cor}, so that \eqref{better_def} holds.
Define the translation(s) $\psi : H_f \times \Z^d \to H_g \times \Z^d$ by $\psi(n,x) \coloneqq (\tau(n),x-x_n)$, so that \eqref{better_def} now reads as
\eeq{
f(v) = g(\psi(v)) \quad \text{for all $v \in H_f \times \Z^d$}. \label{f_equals_g}
}

Next we couple the environments.
Let $\eta_u$ be equal to $\omega_{\psi^{-1}(u)}$ whenever $u \in H_g \times \Z^d$.
Otherwise, we may take $\eta_u$ to be an independent copy of $\omega_u$.
Now, for any $u = (n,x)$ and $v = (n,y)$ with $n \in H_f$,
\eeq{
\psi(u) - \psi(v) = u - v \quad \implies \quad P(v,u) &= P(\sigma_1 = u - v) \\ &= P(\sigma_1 = \psi(u)-\psi(v)) = P(\psi(v),\psi(u)). \label{transitions_equal}
}
Therefore, for any $u \in H_f \times \Z^d$, we have
\eeq{ \label{numerator_same}
\sum_{v \sim u} f(v)\e^{\beta \omega_u}P(v,u)
&\stackrel{\mbox{\scriptsize\eqref{f_equals_g},\eqref{transitions_equal}}}{=} \sum_{v \sim u} g(\psi(v))\e^{\beta \omega_u}P(\psi(v),\psi(u)) \\
&\stackrel{\phantom{\eqref{f_equals_g},\eqref{transitions_equal}}}{=} \sum_{v \sim \psi(u)} g(v)\e^{\beta \eta_{\psi(u)}}P(v,\psi(u)).
}
On the other hand, for any $u \in (\N \setminus H_f) \times \Z^d$ we have $f(v) = 0$ for every $v \sim u$. 
Similarly, $g(v) = 0$ whenever $u \in (\N \setminus H_g) \times \Z^d$ and $v \sim u$.
Consequently, \eqref{numerator_same} trivially implies
\eq{
\sum_{w \in \N \times \Z^d} \sum_{v \sim w} f(v) \e^{\beta \omega_w}P(v,w)
= \sum_{w \in \N \times \Z^d} \sum_{v \sim w} g(v) \e^{\beta \eta_w}P(v,w). 
}
This identity, together with the fact that $\|f\| = \|g\|$ (see Lemma \ref{norm_equivalence}), shows $\wt F = \wt G$ and thus proves claim (b).
When we further consider \eqref{numerator_same}, we see that $F(u) = G(\psi(u))$ for all $u \in H_f \times \Z^d$.
Hence
\eq{
\alpha\sum_{u \in H_f \times \Z^d} |F(u) - G(\psi(u))| + \sum_{u \notin H_f \times \Z^d} F(u)^\alpha + \sum_{u \notin H_g \times \Z^d} G(u)^\alpha = 0.
}
Moreover, $\eps > 0$, we can find a finite subset $A \subset H_f \times \Z^d$ such that
\eq{
\sum_{u \notin A} F(u)^\alpha + \sum_{u \notin \psi(A)} G(u)^\alpha < \eps.
}
With $\phi \coloneqq \psi\rvert_{A}$ (so $\deg(\phi) = \infty$), we thus have $d_\alpha(F,G) \leq d_{\alpha,\phi}(F,G) < \eps$.
Letting $\eps$ tend to 0 gives claim (a).
\end{proof}

Given Proposition \ref{same_law}, the map $f \mapsto Tf \in \PP(\SS)$ is well-defined.
The next step in our construction is to establish continuity. 
The seemingly unexciting result below will require a careful and lengthy proof. 

\begin{prop} \label{continuous1} \hspace*{1ex}
\begin{itemize}
\item[(a)] $f\mapsto Tf$ is a continuous map from $(\SS,d_\alpha)$ to $(\PP(\SS),\WW_\alpha)$.
\item[(b)] For any positive integer $q$, $f\mapsto \E(\log^q\wt F)$ is a continuous map from $(\SS,d_\alpha)$ to $\R$.
In particular, the case $q=1$ implies $R(\cdot)$ is continuous.
\end{itemize}
\end{prop}

The following consequence of part (b) is not new but will be convenient to have recorded for later use.

\begin{cor} \label{increment_cor}
Recall the martingale increment $d_n$ from \eqref{d_n_def}.
For any positive integer $q$, there exists a constant $C$ depending only on $\mathfrak{L}_\omega$, $\beta$, $d$, and $q$, such that $\E(d_n^q) \leq C$ for every $n$.
\end{cor}

\begin{proof}
If we fix $f = f_n$ and take $\wt F$ as in \eqref{R_def}, then \eqref{Zfrac} tells us
\eq{
\E\givenk[\Big]{\log^q\frac{Z_{n+1}(\beta)}{Z_n(\beta)}}{\FF_n} = \E \log^q \wt F.
 }
By the compactness of $\SS$ and the continuity property of Proposition \ref{continuous1}(b), the above right-hand side is bounded by a constant depending only on $\mathfrak{L}_\omega$, $\beta$, $d$, and $q$.
The same is then true for $\E(\log^q \frac{Z_{n+1}(\beta)}{Z_n(\beta)})$, and thus also for $\E(d_n^q)$.
\end{proof}

In review, $f_n$ is a function on $\Z^d$ that is random with respect to $\FF_n$.
It is natural (and measurable) to identify $f_n$ with a function on $\N \times \Z^d$ that is supported on a single copy of $\Z^d$; this is \eqref{fn_def2}.
Finally, that function---thus far an element of $\SS_0$---is identified with its equivalence class $\iota(f_n) \in \SS$, where $\iota$ is the quotient map from Lemma \ref{S_meas}.
In accordance with our previous notation, the symbol $f_n$ will henceforth denote the equivalence class in $\SS$ unless stated otherwise, while $f_n(u)$ will denote the representative defined by \eqref{fn_def2}, evaluated at $u \in \N \times \Z^d$.
By the discussion preceding \eqref{F_def_0}, the law of $f_{n+1} \in \SS$ given $\FF_n$ is equal to $Tf_n$.
For this reason we refer to $T$ as the ``update map".

In preparation for the proof of Proposition \ref{continuous1}, we record the following results.

\begin{lemma} \label{amgm}
Let $X_1,X_2,\dots$ be  i.i.d.~copies of a positive random variable $X$.
If $c_1,c_2,\dots$ are nonnegative constants 
 satisfying $C \leq \sum_{i = 1}^\infty c_i < \infty$ for a positive constant $C$, then
\eq{
\E\bigg[\bigg(\sum_{i=1}^\infty c_i X_i\bigg)^{-q}\bigg] \leq C^{-q}\, \E(X^{-q}) \quad \text{for any $q > 0$.}
}
\end{lemma}

\begin{proof}
By Jensen's inequality, we have
\eq{
\bigg(\sum_{i=1}^\infty c_i X_i\bigg)^{-q} 
&= \bigg(\sum_{i=1}^\infty c_i\bigg)^{-q}\bigg(\sum_{i=1}^\infty \frac{c_i}{\sum_{j=1}^\infty c_j} X_i\bigg)^{-q} 
\leq C^{-q}\sum_{i=1}^\infty \frac{c_i}{\sum_{j=1}^\infty c_j} X_i^{-q}.
}
Now take expectation, using positivity to pass through the sum on the right-hand side.
\end{proof}

\begin{lemma} \label{LLN}
Let $X_1,X_2,\dots$  be i.i.d.~copies of an integrable, centered random variable $X$.
For any $t > 0$, there exists $b > 0$ such that whenever $c_1,c_2,\dots$ are constants satisfying $|c_i| \leq b$ for all $i$ and $\sum_{i = 1}^\infty|c_i| \leq 1$, then
\eeq{
\E\bigg|\sum_{i = 1}^\infty c_i X_i\bigg| \leq t. \label{LLN_to_show}
}
\end{lemma}

\begin{proof}
Fix $t > 0$ and choose $L$ sufficiently large that if $\bar X \coloneqq X\one_{\{|X|\leq L\}}$, then
\eeq{ \label{truncation_1}
\E|X-\bar X| \leq \frac{t}{4} \wedge 1.
}
In particular,
\eeq{ \label{truncation_2}
|\E(\bar X)| = |\E(\bar X - X)| \leq \E|\bar X - X| \leq \frac{t}{4} \wedge 1.
}
Given $t$ and $L$, let $b>0$ be sufficiently small that
\eeq{
\e^{-t^2/(128(L+1)^2b)} \leq \frac{t}{8(L+1)}. \label{exponential_bound}
}
As in the hypothesis, assume $|c_i| \leq b$ for all $i$, and $\sum_{i=1}^\infty |c_i| \leq 1$.
Consider that
\eeq{ \label{first_triangle}
\E\bigg|\sum_{i = 1}^\infty c_i X_i\bigg|
= \E\bigg|\sum_{i = 1}^\infty c_i\bar X_i + c_i(X-\bar X_i)\bigg|
&\stackrel{\phantom{\eqref{truncation_1}}}{\leq} \E\bigg|\sum_{i = 1}^\infty c_i\bar X_i\bigg|
+ \sum_{i=1}^\infty |c_i|\E|X-\bar X| \\
&\stackrel{\mbox{\scriptsize\eqref{truncation_1}}}{\leq}  \E\bigg|\sum_{i = 1}^\infty c_i\bar X_i\bigg|+ \frac{t}{4}.
}
In order to apply martingale inequalities, we recenter the remaining sum:
\eeq{ \label{second_triangle}
 \E\bigg|\sum_{i = 1}^\infty c_i\bar X_i\bigg|
 \leq  \E\bigg|\sum_{i = 1}^\infty c_i(\bar X_i-\E(\bar X))\bigg| + \sum_{i=1}^\infty |c_i\E(\bar X)| \stackrel{\mbox{\scriptsize\eqref{truncation_2}}}{\leq} \E\bigg|\sum_{i = 1}^\infty c_i(\bar X_i-\E(\bar X))\bigg| + \frac{t}{4}.
}
For each $i$, the random variable $c_i(\bar X_i - \E(\bar X))$ has mean $0$ and takes values between $-|c_i|(L+1)$ and $|c_i|(L+1)$.
Therefore, by the Azuma--Hoeffding inequality \cite[Theorem 2.8]{boucheron-lugosi-massart13},
\eq{
\P\bigg(\bigg|\sum_{i = 1}^\infty c_i(\bar X_i-\E(\bar X))\bigg| > \frac{t}{4}\bigg) &\leq 2\exp\Big(\frac{-(t/4)^2}{2\sum_{i = 1}^\infty (2|c_i|(L+1))^2}\Big) \\
&= 2\exp\Big(\frac{-t^2}{128(L+1)^2\sum_{i=1}^\infty |c_i|^2}\Big) \\
&\leq 2\exp\Big(\frac{-t^2}{128(L+1)^2b}\Big) \stackrel{\mbox{\scriptsize\eqref{exponential_bound}}}{\leq} \frac{t}{4(L+1)}.
}
Since $|\sum_{i = 1}^\infty c_i(\bar X_i - \E(\bar X))| \leq L+1$, we conclude
\eeq{ \label{third_triangle}
\E\bigg|\sum_{i = 1}^\infty c_i(\bar X_i-\E(\bar X))\bigg| \leq \frac{t}{2}.
}
Together, the inequalities \eqref{first_triangle}--\eqref{third_triangle} give \eqref{LLN_to_show}.
\end{proof}

\begin{proof}[Proof of Proposition \ref{continuous1}]
Since $(\SS,d_\alpha)$ is a compact metric space, uniform continuity 
will be implied by continuity.
So it suffices to prove continuity at a fixed $f \in \SS$.

Let $\eps > 0$ be given.
To prove continuity of $g\mapsto Tg$ at $f \in \SS$, we need to exhibit $\delta > 0$ such that if $g \in \SS$ satisfies $d_\alpha(f,g) < \delta$, then there exist representatives $f,g \in \SS_0$ and a coupling of environments $(\vc\omega,\vc\eta)$ such that under this coupling the following inequality holds:
\eq{
\E[d_\alpha(F,G)] < \eps.
}
Here $F,G\in\SS$ are given by \eqref{F_def} and \eqref{G_def}.
We shall see that it suffices to choose $\delta$ satisfying conditions \eqref{delta_condition_k}--\eqref{delta_condition_min} below.

Let $q \coloneqq \max_{x \in \Z^d} P(0,x)$, which is strictly less than $1$ by \eqref{walk_assumption_1}.
Next choose $t > 0$ sufficiently small to satisfy the following conditions
\begin{subequations}
\label{t_conditions}
\begin{align}
\frac{6t}{\e^{\lambda(\beta)}} &< \frac{\eps}{8\alpha}, \label{t_condition_1} \\
\frac{235t}{69}\e^{\lambda(-\beta)} &< \frac{\eps}{16\alpha}. \label{t_condition_2}
\end{align}
\end{subequations}
Given $t$, we can take $b > 0$ as in Lemma \ref{LLN}, and then choose $\kappa > 0$ to satisfy
\begin{subequations}
\label{kappa_conditions}
\begin{align}
\kappa &\leq \frac{1}{7}, \label{kappa_condition_1} \\
7\kappa \e^{\lambda(\beta)+\lambda(-\beta)} &< \frac{\eps}{8\alpha}, \label{kappa_condition_4} \\
7\kappa + (2\kappa)^{1/\alpha} &\leq b, \label{kappa_condition_2} \\
14\kappa(1-q)^{-1}\e^{\lambda(-\beta)}\E|\e^{\beta\omega}-\e^{\lambda(\beta)}| &<\frac{\eps}{16\alpha}, \label{kappa_condition_3} \\
6\e^{\lambda(\alpha\beta)+\lambda(-\alpha\beta)}\kappa &< \frac{\eps}{4}. \label{kappa_condition_5}
\end{align}
\end{subequations}
Then fix a representative $f \in \SS_0$, and choose $A \subset \N \times \Z^d$ finite but sufficiently large that
\eeq{
\sum_{u \notin A} f(u) < \kappa. \label{A_condition}
}
By possibly omitting some elements, we may assume $f$ is strictly positive on $A$.
Next, let $k$ be a positive integer sufficiently large that
\begin{subequations}
\label{k_conditions}
\begin{align}
2^{-k} &< \frac{\eps}{4}, \label{k_condition_1} \\
P(\|\sigma_1\|_1 > k) &< \kappa. \label{k_condition_2}
\end{align}
\end{subequations}

Now choose $\delta > 0$ such that
\begin{subequations}
\label{delta_conditions}
\begin{align}
  \delta &< 2^{-3k}. \label{delta_condition_k}
\intertext{If $A$ is nonempty, we will also demand that}
(2d)^k|A|\delta^{1/\alpha} &< \kappa. \label{delta_condition_A}
\intertext{Otherwise, we relax this assumption to}
  \delta^{1/\alpha} &< \kappa. \label{delta_condition_noA}
\intertext{Finally, we assume}
\delta^{1/\alpha} &< \inf_{u \in A} f(u). \label{delta_condition_min}
\end{align}
\end{subequations}
We claim this $\delta > 0$ is sufficient for $\eps > 0$ in the sense described above.

Assume $g \in \SS$ satisfies $d_\alpha(f,g) < \delta$.
Then given any representative $g \in \SS_0$, there exists an isometry $\psi : C \to \N \times \Z^d$ such that
\eeq{
d_{\alpha,\psi}(f,g) = \alpha\sum_{u \in C} |f(u)-g(\psi(u))| + \sum_{u\notin C}f(u)^\alpha + \sum_{u\notin \psi(C)}g(u)^\alpha + 2^{-\deg(\psi)} < \delta. \label{psi_condition}
}
By \eqref{delta_condition_k}, it follows that $\deg(\psi) > 3k$. 
In particular, upon defining $\phi \coloneqq \psi|_A$, we have $\deg(\phi) \geq \deg(\psi) > 3k$.
Therefore, Lemma \ref{extension} guarantees that $\phi$ can be extended to an isometry 
$\Phi : A^{(k)} \to \N \times \Z^d$ with $\deg(\Phi) > k$, where
\eq{
A^{(k)} = \{v \in \N \times \Z^d: \|u - v\|_1 \leq k \text{ for some $u \in A$}\}.
}
Alternatively, we can express this set as
\eq{
A^{(k)} = \bigcup_{u \in A} \NN(u,k), \qquad \NN(u,k) \coloneqq \{v \in \N \times \Z^d: \|u - v\|_1 \leq k\}.
}
The inequality $\deg(\Phi) > k$ implies that for any $B \subset A$,
\eeq{ \label{degree_consequence}
\Phi(B^{(k)}) = \Phi(B)^{(k)} \coloneqq \bigcup_{u \in B} \NN(\Phi(u),k).
}
Indeed,
\eq{
u \in B^{(k)} \quad \iff \quad \exists\, v \in B \cap \NN(u,k) \quad
\iff \quad \exists\, v \in B, \Phi(v) \in \NN(\Phi(u),k) \quad
\iff \quad \Phi(u) \in \Phi(B)^{(k)}.
}
Furthermore, \eqref{psi_condition} and \eqref{delta_condition_min} together force $A \subset C$, since
\eq{
u \in A \quad \implies \quad f(u) \geq \inf_{u \in A} f(u) > \delta^{1/\alpha}
\quad \implies \quad f(u)^\alpha > \delta > d_\psi(f,g) \quad \implies \quad u \in C.
}

Given a random environment $\vc\omega$, we couple to it $\vc\eta$ in the following way.
If $u \in \Phi(A^{(k)})$, then set $\eta_u = \omega_{\Phi^{-1}(u)}$.
Otherwise, let $\eta_u$ be an independent copy of $\omega_u$.
Note that $F$ and $G$ are now distributed as $Tf$ and $Tg$, respectively, when mapped to $\SS$.
On the other hand, $\Phi$ is deterministic (i.e.~does not depend on the environments $\vc\omega$ and $\vc\eta$, and so no measurability issues arise in the bound
\eq{
\E(d_\alpha(F,G)) \leq \E(d_{\alpha,\Phi}(F,G)),
}
where by definition
\eeq{ \label{d_Phi_def}
d_{\alpha,\Phi}(F,G) = \delta \sum_{u \in A^{(k)}} |F(u)-G(\Phi(u))| + \sum_{u \notin A^{(k)}} F(u)^\alpha + \sum_{u \notin \Phi(A^{(k)})} G(u)^\alpha + 2^{-\deg(\Phi)}.
}
It thus suffices to show $\E(d_{\alpha,\Phi}(F,G)) < \eps$.
To simplify notation, we will write
\eq{
\wt{f}(u) &= \sum_{v \sim u} f(v)\e^{\beta \omega_u}P(v,u), & \wt{F} &= \sum_{w \in \N \times \Z^d} \wt{f}(w) + (1 - \|f\|)\e^{\lambda(\beta)}, \\
\wt{g}(u) &= \sum_{v \sim u} g(v)\e^{\beta \eta_u}P(v,u), & \wt{G} &= \sum_{w \in \N \times \Z^d} \wt{g}(w) + (1 - \|g\|)\e^{\lambda(\beta)},
}
so that $F(u) = \wt{f}(u)/\wt{F}$ and $G(u) = \wt{g}(u)/\wt{G}$.

Consider the first of the four terms on the right-hand side of \eqref{d_Phi_def}.
Observe that for $u \in A^{(k)}$,
\eeq{
|F(u) - G(\Phi(u))| = \bigg|\frac{\wt{f}(u)}{\wt{F}} - \frac{\wt{g}(\Phi(u))}{\wt{G}}\bigg| 
&\leq \bigg|\frac{\wt{f}(u)}{\wt{F}} - \frac{\wt{g}(\Phi(u))}{\wt{F}}\bigg| + \bigg|\frac{\wt{g}(\Phi(u))}{\wt{F}}-\frac{\wt{g}(\Phi(u))}{\wt{G}}\bigg|  \\
&= \frac{|\wt{f}(u) - \wt{g}(\Phi(u))|}{\wt{F}} + \frac{\wt{g}(\Phi(u))}{\wt{G}}\cdot\bigg|\frac{\wt{G}}{\wt{F}} - 1\bigg|. \label{two_parts}
}
Summing over $A^{(k)}$ and taking expectation, we obtain the following from the first term in the last line of \eqref{two_parts}:
\eeq{ \label{first_part}
\E\Bigg[\sum_{u \in A^{(k)}} \frac{|\wt{f}(u) - \wt{g}(\Phi(u))|}{\wt{F}}\Bigg]
&= \sum_{u \in A^{(k)}} \E\bigg[\frac{|\wt{f}(u) - \wt{g}(\Phi(u))|}{\wt{F}}\bigg].
}
Meanwhile, the second term in \eqref{two_parts} gives
\eeq{ \label{second_part}
\E\Bigg[\sum_{u \in A^{(k)}} \frac{\wt{g}(\Phi(u))}{\wt{G}}\cdot\bigg|\frac{\wt{G}}{\wt{F}}-1\bigg|\Bigg]
\leq \E\, \bigg|\frac{\wt{G}}{\wt{F}}-1\bigg|
= \E\bigg[ \frac{|\wt{G}-\wt{F}|}{\wt{F}}\bigg].
}
These preliminary steps suggest two quantities we should control from above, namely the right-hand sides of \eqref{first_part} and \eqref{second_part}.
Consideration of the second and third terms on the right-hand side of \eqref{d_Phi_def} suggests four more quantities,
since the Harris--FKG inequality yields
\eeq{ \label{to_bound_2}
\E\bigg[\sum_{u \notin A^{(k)}} F(u)^\alpha\bigg]
\leq \E(\wt F^{-\alpha})\E\bigg[\sum_{u \notin A^{(k)}} \wt f(u)^\alpha\bigg],
}
and similarly
\eeq{ \label{to_bound_3}
\E\bigg[\sum_{u \notin \Phi(A^{(k)})} G(u)^\alpha\bigg]
\leq \E(\wt G^{-\alpha})\E\bigg[\sum_{u \notin \Phi(A^{(k)})} \wt g(u)^\alpha\bigg].
}
Therefore, we should seek an upper bound for $\E(\wt F^{-\alpha})$ and $\E(\wt G^{-\alpha})$, as well as $\E\big[\sum_{u \notin A^{(k)}} \wt f(u)^\alpha\big]$ and $\E\big[\sum_{u \notin \Phi(A^{(k)})} \wt g(u)^\alpha\big]$.
For clarity of presentation, we divide our task into the next four subsections.

\subsubsection{Upper bound for $\E(\wt{F}^{-\alpha})$ and $\E(\wt{G}^{-\alpha})$}
Suppose $q\in[-\alpha,0] \cup [1,\alpha]$.
Since
\eq{
\sum_{u\in\N\times\Z^d} \sum_{v\sim u} f(v)P(v,u) + (1-\|f\|) = 1,
}
Jensen's inequality applied to the convex function $t\mapsto t^q$ gives
\eeq{ \label{Fbound}
\E(\wt F^{q}) &\leq \E\bigg[\sum_{u\in\N\times\Z^d} \sum_{v\sim u} f(v)\e^{q\beta \omega_u} P(v,u) + (1-\|f\|)\e^{\alpha\lambda(\beta)}\bigg] \\
&= (\|f\|\e^{\lambda(q\beta)} + (1-\|f\|)\e^{q\lambda(\beta)}) \leq \e^{\lambda(q\beta)}.
}
By the same argument with $g$ in place of $f$, we obtain the fourth desired inequality:
\eeq{
\E(\wt{G}^{q}) \leq \e^{\lambda(q\beta)}. \label{Gbound}
}


\subsubsection{Upper bound for $\sum_{u \in A^{(k)}} \E(|\wt{f}(u) - \wt{g}(\Phi(u))|/\wt F)$}

Observe that for $u \in A^{(k)}$,
\eq{
|\wt f(u) - \wt g(\Phi(u))|
&= \bigg| \sum_{v \sim u} f(v)\e^{\beta \omega_u}P(v,u) - \sum_{v \sim \Phi(u)} g(v)\e^{\beta \eta_{\Phi(u)}}P(v,\Phi(u))\bigg| \\
&= \e^{\beta\omega_u}\bigg|\sum_{v \sim u} f(v)P(v,u) - \sum_{v \sim \Phi(u)} g(v)P(v,\Phi(u))\bigg|,
}
meaning $|\wt f(u) - \wt g(\Phi(u))|$ is a non-decreasing function of $\omega_u$ and independent of $\omega_w$ for $w \neq u$.
Since $\wt F^{-1}$ is a non-increasing function of all $\omega_w$, the Harris--FKG inequality yields
\eeq{ \label{initial_fkg}
\sum_{u \in A^{(k)}} \E\bigg[\frac{|\wt{f}(u) - \wt{g}(\Phi(u))|}{\wt F}\bigg]
\leq \E(\wt F^{-1}) \sum_{u \in A^{(k)}} \E|\wt f(u) - \wt g(\Phi(u))|,
}
where
\eeq{  \label{fixed_u}
 \sum_{u \in A^{(k)}}\E|\wt f(u) - \wt g(\Phi(u))|
&\leq \e^{\lambda(\beta)} \sum_{u \in A^{(k)}}\bigg(\bigg|\sum_{v \in \NN(u,k)} f(v)P(v,u) - \sum_{v \in \NN(\Phi(u),k)} g(v)P(v,\Phi(u))\bigg| \\
&\phantom{=\e^{\beta\omega_u}\sum_{u\in A^{(k)}}\bigg(}+\bigg|\sum_{\substack{v \sim u \\ v \notin \NN(u,k)}} f(v)P(v,u) - \sum_{\substack{v \sim \Phi(u) \\ v \notin \NN(\Phi(u),k)}} g(v)P(v,\Phi(u))\bigg|\bigg).
}
Of the two absolute values above, the second is easier to control.
Indeed,
\eeq{ \label{easier_one}
&\sum_{u \in A^{(k)}} \bigg|\sum_{\substack{v \sim u \\ v \notin \NN(u,k)}} f(v)P(v,u) - \sum_{\substack{v \sim \Phi(u) \\ v \notin \NN(\Phi(u),k)}} g(v)P(v,\Phi(u))\bigg| \\
&\stackrel{\phantom{\eqref{k_condition_2}}}{\leq} \sum_{v \in \N \times \Z^d} \sum_{\substack{u \sim v \\ u \notin \NN(v,k)}} (f(v)+g(v))P(v,u)  \\
&\stackrel{\mbox{\scriptsize\eqref{k_condition_2}}}{<} \sum_{v \in \N \times \Z^d} (f(v)+g(v))\kappa
\leq 2\kappa.
}
Next consider the first absolute value in the final line of \eqref{fixed_u}. 
The difference between the two sums can be bounded as
\eeq{ \label{individual_u}
&\bigg|\sum_{v \in \NN(u,k)} f(v)P(v,u) - \sum_{v \in \NN(\Phi(u),k)} g(v)P(v,\Phi(u))\bigg| \\
&\leq \bigg|\sum_{v \in \NN(u,k) \cap A} f(v)P(v,u)-\sum_{v \in \NN(\Phi(u),k) \cap \psi(A)} g(v)P(v,\Phi(u))\bigg| + \sum_{v \in \NN(u,k) \setminus A} f(v)P(v,u) \\
&\phantom{\leq}+ \sum_{v \in \NN(\Phi(u),k) \cap \psi(C\setminus A)} g(v)P(v,\Phi(u)) + \sum_{v \in \NN(\Phi(u),k) \setminus \psi(C)} g(v)P(v,\Phi(u)).
}
Each of the four terms above can be controlled separately.
For the first term, notice that because $\deg(\Phi) > k$,
\eq{
v \in \NN(u,k) \cap A \quad \implies \quad u - v = \Phi(u) - \Phi(v) = \Phi(u) - \psi(v).
}
That is, if $v \in \NN(u,k) \cap A$, then $\psi(v) \in \NN(\Phi(u),k) \cap \psi(A)$ and satisfies
\eq{
P(\psi(v),\Phi(u)) = P(\sigma_1 = \Phi(u) - \psi(v)) = P(\sigma_1 = u - v) = P(v,u).
}
Conversely,
\eq{
\psi(v) \in \NN(\Phi(u),k) \cap \psi(A) \quad \implies \quad \Phi(u)-\psi(v) = \Phi(u)-\Phi(v) = u - v.
}
That is, if $\psi(v) \in \NN(\Phi(u),k) \cap \psi(A)$, then $v \in \NN(u,k) \cap A$ and satisfies
\eq{
P(v,u) = P(\sigma_1 = u - v) = P(\sigma_1 = \Phi(u)-\psi(v)) = P(\psi(v),\Phi(u)).
}
We thus have a bijection between $\NN(u,k) \cap A$ and $\NN(\Phi(u),k) \cap \psi(A)$, meaning
\eq{
\bigg|\sum_{v \in \NN(u,k) \cap A} f(v)P(v,u)-\sum_{v \in \NN(\Phi(u),k) \cap \psi(A)} g(v)P(v,\Phi(u))\bigg|
&= \bigg|\sum_{v \in \NN(u,k) \cap A} \big[f(v) - g(\psi(v))\big]P(v,u)\bigg| \\
&\leq \sum_{v \in \NN(u,k) \cap A} |f(v) - g(\psi(v))|P(v,u) \\
&\leq \sum_{v \in A}|f(v) - g(\psi(v))|P(v,u) \\
&\leq \sum_{v \in C}|f(v) - g(\psi(v))|P(v,u),
}
where the final inequality is trivial since $A \subset C$.
Summing over $u \in A^{(k)}$ gives the total bound
\eeq{ \label{all_u_1}
&\sum_{u \in A^{(k)}} \bigg|\sum_{v \in \NN(u,k) \cap A} f(v)P(v,u)-\sum_{v \in \NN(\Phi(u),k) \cap \psi(A)} g(v)P(v,\Phi(u))\bigg| \\
&\leq \sum_{u \in A^{(k)}} \sum_{v \in C} |f(v)-g(\psi(v))|P(v,u) \\
&= \sum_{v \in C} \sum_{u \in A^{(k)}} |f(v) - g(\psi(v))|P(v,u) 
\leq \sum_{v \in C} |f(v) - g(\psi(v))|
\leq d_{\alpha,\psi}(f,g) \stackrel{\mbox{\scriptsize\eqref{psi_condition}}}{<} \delta \stackrel{\mbox{\scriptsize\eqref{delta_condition_noA}}}{<} \kappa.
}
Considering the second term on the right-hand side of \eqref{individual_u}, we have
\eeq{ \label{all_u_2}
\sum_{u \in A^{(k)}} \sum_{v \in \NN(u,k) \setminus A} f(v)P(v,u)
\leq \sum_{v \notin A} \sum_{u \in \NN(v,k)} f(v)P(v,u)
\leq \sum_{v \notin A} f(v) \stackrel{\mbox{\scriptsize\eqref{A_condition}}}{<} \kappa.
}
Similarly, for the third term,
\eeq{ \label{all_u_3}
&\sum_{u \in A^{(k)}} \sum_{v \in \NN(\Phi(u),k) \cap \psi(C\setminus A)} g(v)P(v,\Phi(u)) \\
&\leq \sum_{u \in A^{(k)}} \sum_{v \in \NN(\Phi(u),k) \cap \psi(C\setminus A)} \big[|f(\psi^{-1}(v))-g(v)| + f(\psi^{-1}(v))\big]P(v,\Phi(u)) \\
&\leq \sum_{v \in C \setminus A} \sum_{u \in \NN(\psi(v),k)}\big[|f(v)-g(\psi(v))| + f(v)\big]P(\psi(v),u) \\
&\leq \sum_{v \in C} |f(v)-g(\psi(v))| + \sum_{v \notin A} f(v)
\stackrel{\mbox{\scriptsize\eqref{A_condition}}}{\leq} d_{\alpha,\psi}(f,g) + \kappa \stackrel{\mbox{\scriptsize\eqref{psi_condition}}}{<} \delta + \kappa
\stackrel{\mbox{\scriptsize\eqref{delta_condition_noA}}}{<} 2\kappa.
}
Finally, for the fourth term,
\eeq{ \label{all_u_4} \raisetag{3.5\baselineskip}
\sum_{u \in A^{(k)}} \sum_{v \in \NN(\Phi(u),k) \setminus \psi(C)} g(v)P(v,\Phi(u))
&= \sum_{u \in A^{(k)}} \sum_{v \in \NN(\Phi(u),k) \setminus \psi(C)} g(v)P(0,\Phi(u)-v) \\
&\leq\sum_{u \in A^{(k)}} \sup_{v \notin \psi(C)} g(v) \\
&\leq (2d)^k|A| \bigg(\sum_{v \notin \psi(C)} g(v)^\alpha\bigg)^{1/\alpha} \\
&\leq (2d)^k|A| d_{\alpha,\psi}(f,g)^{1/\alpha} \stackrel{\mbox{\scriptsize\eqref{psi_condition}}}{<} (2d)^k|A|\delta^{1/\alpha} \stackrel{\mbox{\scriptsize\eqref{delta_condition_A}}}{<} \kappa.
}
Combining \eqref{individual_u}--\eqref{all_u_4}, we arrive at
\eeq{
\sum_{u \in A^{(k)}} \bigg|\sum_{v \in \NN(u,k)} f(v)P(v,u) - \sum_{v \in \NN(\Phi(u),k)} g(v)P(v,\Phi(u))\bigg| < 5\kappa. \label{all_u}
}
Using \eqref{easier_one} and \eqref{all_u} with \eqref{fixed_u} reveals 
\eq{
\sum_{u \in A^{(k)}} \E|\wt f(u) - \wt g(\Phi(u))| \leq 7\kappa \e^{\lambda(\beta)},
}
and thus we have the desired bound:
\eeq{ \label{part_2_bound}
\sum_{u \in A^{(k)}} \E\bigg[\frac{|\wt{f}(u) - \wt{g}(\Phi(u))|}{\wt F}\bigg]
&\stackrel{\mbox{\scriptsize\eqref{initial_fkg}}}{\leq} \E(\wt F^{-1}) \sum_{u \in A^{(k)}} \E|\wt f(u) - \wt g(\Phi(u))| \\
&\stackrel{\mbox{\scriptsize\eqref{Fbound}}}{\leq} 7\kappa \e^{\lambda(-\beta)}\e^{\lambda(\beta)}
 \stackrel{\mbox{\scriptsize\eqref{kappa_condition_4}}}{<} \frac{\eps}{8\alpha}.
}

\subsubsection{Upper bounds for $\E(|\wt{G}-\wt{F}|/\wt{F})$ and $\E(|\wt{G}-\wt{F}|/\wt{G})$}

First write the trivial equality
\eq{
\wt F - \wt G = (\wt F - \e^{\lambda(\beta)}) - (\wt G - \e^{\lambda(\beta)}),
}
and then observe that
\eeq{ \label{F_minus}
\wt F - \e^{\lambda(\beta)} &= \sum_{u \in \N \times \Z^d} \wt{f}(u) - \|f\|\e^{\lambda(\beta)} \\
&= \sum_{u \in \N \times \Z^d} \sum_{v \sim u} f(v) \e^{\beta \omega_u}P(v,u) - \sum_{u \in \N \times \Z^d} \sum_{v \sim u} f(v)P(v,u) \e^{\lambda(\beta)} \\
&= \sum_{u \in \N \times \Z^d} \sum_{v \sim u} f(v)\big[\e^{\beta \omega_u} - \e^{\lambda(\beta)}\big]P(v,u) \\
&= \sum_{u \in A^{(k)}} \sum_{v \sim u} f(v)\big[\e^{\beta \omega_u} - \e^{\lambda(\beta)}\big]P(v,u)
+ \sum_{u \notin A^{(k)}} \sum_{v \sim u} f(v)\big[\e^{\beta \omega_u} - \e^{\lambda(\beta)}\big]P(v,u).
}
Similarly,
\eq{
\wt{G} - \e^{\lambda(\beta)}
&= \sum_{u \in \Phi(A^{(k)})} \sum_{v \sim u} g(v)\big[\e^{\beta \eta_u} - \e^{\lambda(\beta)}\big]P(v,u) + \sum_{u \notin \Phi(A^{(k)})} \sum_{v \sim u} g(v)\big[\e^{\beta \eta_u} - \e^{\lambda(\beta)}\big]P(v,u) \\
&= \sum_{u \in A^{(k)}} \sum_{v \sim \Phi(u)} g(v)\big[\e^{\beta \omega_u} - \e^{\lambda(\beta)}\big]P(v,\Phi(u)) + \sum_{u \notin \Phi(A^{(k)})} \sum_{v \sim u} g(v)\big[\e^{\beta \eta_u} - \e^{\lambda(\beta)}\big]P(v,u).
}
Hence
\eq{
\wt{F} - \wt{G}
&= \sum_{u \in A^{(k)}}\bigg( \sum_{v \sim u} f(v)P(v,u) - \sum_{v \sim \Phi(u)} g(v)P(v,\Phi(u))\bigg)\big[\e^{\beta \omega_u} - \e^{\lambda(\beta)}\big] \\
&\phantom{=|} + \sum_{u \notin A^{(k)}} \sum_{v \sim u} f(v)\big[\e^{\beta \omega_u} - \e^{\lambda(\beta)}\big]P(v,u)
- \sum_{u \notin \Phi(A^{(k)})} \sum_{v \sim u} g(v)\big[\e^{\beta \eta_u} - \e^{\lambda(\beta)}\big]P(v,u).
}
In the notation of Lemma \ref{LLN}, the following random variables are i.i.d.: 
\eq{
X_{u} &\coloneqq \e^{\beta \omega_u} - \e^{\lambda(\beta)}, \quad u \in A^{(k)}, \\
X'_{u} &\coloneqq \e^{\beta \omega_u} - \e^{\lambda(\beta)}, \quad u \notin A^{(k)}, \\
X''_{u} &\coloneqq \e^{\beta \eta_u} - \e^{\lambda(\beta)}, \quad u \notin \Phi(A^{(k)}).
} 
Their coefficients are
\eq{
c_{u} &\coloneqq \sum_{v \sim u} f(v)P(v,u) - \sum_{v \sim \Phi(u)} g(v)P(v,\Phi(u)), \quad u \in A^{(k)},\\
c'_{u} &\coloneqq \sum_{v \sim u} f(v)P(v,u), \quad u \notin A^{(k)}, \\
c''_{u} &\coloneqq -\sum_{v \sim u} g(v)P(v,u), \quad u \notin \Phi(A^{(k)}).
}
That is,
\begin{align}
\wt F - \wt G &= \sum_{u \in A^{(k)}} c_u X_u + \sum_{u \notin A^{(k)}} c_u'X_u' + \sum_{u \notin \Phi(A^{(k)})} c_u''X_u'' \nonumber \\
\implies \quad |\wt F - \wt G| &\leq \bigg|\sum_{u \in A^{(k)}} c_u X_u\bigg|
+ \bigg|\sum_{u \notin A^{(k)}} c'_u X_u'\bigg|
+ \bigg|\sum_{u \notin \Phi(A^{(k)})} c''_u X_u''\bigg|. \label{tilde_difference_1} 
\end{align}
Assume we can show the following inequalities:
\begin{align} 
\sum_{u \in A^{(k)}} |c_u|  \leq 7\kappa &\stackrel{\mbox{\scriptsize\eqref{kappa_condition_1}},\mbox{\scriptsize\eqref{kappa_condition_2}}}{\leq}  b \wedge 1, \label{c_sum} \\
c'_u < 2\kappa &\stackrel{\hspace{3.5ex}\eqref{kappa_condition_2}\hspace{3.5ex}}{\leq} b  \quad \text{for all $u \notin A^{(k)}$,} \label{c_prime} \\
c''_u < (2\kappa)^{1/\alpha} + \kappa &\stackrel{\hspace{3.5ex}\eqref{kappa_condition_2}\hspace{3.5ex}}{\leq} b \quad \text{for all $u \notin \Phi(A^{(k)})$.} \label{c_prime_prime}
\end{align}
We also have the trivial inequalities
\eq{ 
\sum_{u \notin A^{(k)}} c_u' = \sum_{u \notin A^{(k)}} \sum_{v \sim u} f(v)P(v,u)
\leq \sum_{v \in \N \times \Z^d} \sum_{u \sim v} f(v)P(v,u)
= \|f\| \leq 1,
}
and similarly
\eq{ 
\sum_{u \notin \Phi(A^{(k)})} |c_u''| \leq \|g\| \leq 1.
}
Therefore, by the choice of $b$ in relation to Lemma \ref{LLN}, we deduce from \eqref{tilde_difference_1} that
\eeq{ \label{numerator_bound}
\E|\wt F - \wt G| \leq 3t.
}
There are now two cases to consider.

\textbf{Case 1.} If $\|f\|\leq1/2$, then
\eeq{ \label{part_3_bound_1}
\E\Big(\frac{|\wt F - \wt G|}{\wt F}\Big) \leq \frac{\E|\wt F - \wt G|}{(1-\|f\|)\e^{\lambda(\beta)}}
\leq \frac{3t}{(1-\|f\|)\e^{\lambda(\beta)}} \stackrel{\mbox{\scriptsize\eqref{t_condition_1}}}{<} \frac{\eps}{8\alpha}.
}
Similarly, if $\|g\|\leq 1/2$, then
\eeq{ \label{part_3_bound_1g}
\E\Big(\frac{|\wt F - \wt G|}{\wt G}\Big) \leq \frac{\E|\wt F - \wt G|}{(1-\|g\|)\e^{\lambda(\beta)}}
\leq \frac{3t}{(1-\|g\|)\e^{\lambda(\beta)}} \stackrel{\mbox{\scriptsize\eqref{t_condition_1}}}{<} \frac{\eps}{8\alpha}.
}

\textbf{Case 2.} On the other hand, if $\|f\| > 1/2$, then we can consider the three sums in \eqref{tilde_difference_1} separately.
First observe that for any $u \in A^{(k)}$, the quantities
\eq{
\wt F_u &\coloneqq \wt F - \sum_{v \sim u} f(v)\e^{\beta\omega_u} = \sum_{w \neq u} \sum_{v \sim w} f(v)\e^{\beta \omega_w}P(v,w), \\
\wt G_u &\coloneqq \wt G - \sum_{v \sim \Phi(u)} g(v)\e^{\beta\omega_u} = \sum_{w \neq \Phi(u)} \sum_{v \sim w} g(v)\e^{\beta \eta_w}P(v,w),
}
are independent of $X_u$.
When $\|f\|>1/2$, we have
\eq{
\sum_{w \neq u} \sum_{v \sim w} f(v)P(v,w) = \sum_{v}f(v)\sum_{\substack{w \sim v \\ w \neq u}}P(v,w)
= \sum_{v}f(v)(1-P(v,u))
\geq \|f\|(1-q) \geq \frac{1-q}{2},
}
and so Lemma \ref{amgm} guarantees
\eq{
\E(\wt F_u^{-1}) \leq 2(1-q)^{-1}\e^{\lambda(-\beta)}.
}
Hence
\eeq{ \label{equals1_sum1}
\E\bigg(\frac{\big|\sum_{u \in A^{(k)}} c_u X_u\big|}{\wt F}\bigg)
\leq \sum_{u \in A^{(k)}} |c_u|\E\Big(\frac{|X_u|}{\wt F_u}\Big)
&\stackrel{\phantom{\eqref{Fbound},\eqref{c_sum}}}{=} \sum_{u \in A^{(k)}} |c_u|\E|X_u|\E(\wt F_u^{-1}) \\
&\stackrel{\mbox{\scriptsize\eqref{Fbound},\eqref{c_sum}}}{\leq} 14\kappa(1-q)^{-1}\e^{\lambda(-\beta)}\E|\e^{\beta\omega}-\e^{\lambda(\beta)}|.
}
When $\|g\|>1/2$, we can apply the same argument with $\wt G_u$ replacing $\wt F_u$, and obtain
\eeq{ \label{equals1_sum1g}
\E\bigg(\frac{\big|\sum_{u \in A^{(k)}} c_u X_u\big|}{\wt G}\bigg)
\leq 14\kappa(1-q)^{-1}\e^{\lambda(-\beta)}\E|\e^{\beta\omega}-\e^{\lambda(\beta)}|.
}

Second, observe that 
\eq{
\sum_{u \in A^{(k)}} \sum_{v \sim u} f(v) P(v,u)
\geq \sum_{v \in A} \sum_{u \in \NN(v,k)} f(v) P(v,u) 
\stackrel{\mbox{\scriptsize\eqref{k_condition_2}}}{>} (1-\kappa)\sum_{v \in A} f(v) \stackrel{\mbox{\scriptsize\eqref{A_condition}}}{>} (1-\kappa)\Big(\frac{1}{2}-\kappa\Big).
}
and so by applying Lemma \ref{amgm} once more, we see
\eq{ 
\E\bigg[\bigg(\sum_{u \in A^{(k)}} \sum_{v \sim u} f(v)\e^{\beta \omega_u}P(v,u)\bigg)^{-1}\bigg] \leq (1-\kappa)^{-1}\Big(\frac{1}{2}-\kappa\Big)^{-1}\e^{\lambda(-\beta)}.
}
Again appealing to independence and then Lemma \ref{LLN}, we obtain the bound
\eeq{ \label{equals1_sum2}
\E\bigg(\frac{\big|\sum_{u \notin A^{(k)}} c_u'X_u'\big|}{\wt F}\bigg)
&\leq \E\bigg(\frac{\big|\sum_{u \notin A^{(k)}} c_u'X_u'\big|}{\sum_{u \in A^{(k)}} \sum_{v \sim u} f(v)\e^{\beta \omega_u}P(v,u)}\bigg) \\
&=\E\bigg|\sum_{u \notin A^{(k)}} c_u'X_u'\bigg|\E\bigg[\bigg(\sum_{u \in A^{(k)}} \sum_{v \sim u} f(v)\e^{\beta \omega_u}P(v,u)\bigg)^{-1}\bigg] \\
&\leq t(1-\kappa)^{-1}\Big(\frac{1}{2}-\kappa\Big)^{-1}\e^{\lambda(-\beta)}.
}
Similarly, when $\|g\|>1/2$, we have
\eq{
\sum_{u \in \Phi(A^{(k)})} \sum_{v \sim u} g(v) P(v,u)
\stackrel{\mbox{\footnotesize\eqref{degree_consequence}}}{\geq} \sum_{v \in \Phi(A)} \sum_{u \in \NN(v,k)} g(v) P(v,u) 
&\stackrel{\mbox{\scriptsize\eqref{k_condition_2}}}{>} (1-\kappa)\sum_{v \in \Phi(A)} g(v) \\
&\stackrel{\phantom{\mbox{\scriptsize\eqref{k_condition_2}}}}{>} (1-\kappa)\bigg(\sum_{v\in A} f(v) - d_{\alpha,\phi}(f,g)\bigg) \\
&\stackrel{\hspace{0.5ex}\mbox{\scriptsize\eqref{A_condition}}\hspace{0.5ex}}{>} (1-\kappa)\Big(\frac{1}{2}-\kappa-\delta\Big),
}
and so analogous reasoning leads to 
\eeq{ \label{equals1_sum2g}
\E\bigg(\frac{\big|\sum_{u \notin \Phi(A^{(k)})} c_u''X_u''\big|}{\wt G}\bigg)
&\leq t(1-\kappa)^{-1}\Big(\frac{1}{2}-\kappa-\delta\Big)^{-1}\e^{\lambda(-\beta)}.
}

Third and finally, we use independence and Lemma \ref{LLN} once more to obtain
\eeq{ \label{equals1_sum3}
\E\bigg(\frac{\big|\sum_{u \notin \Phi(A^{(k)})} c_u''X_u''\big|}{\wt F}\bigg)
= \E\bigg|\sum_{u \notin \Phi(A^{(k)})} c_u''X_u''\bigg|\E(\wt F^{-1})
\leq t\E(\wt F^{-1})
\stackrel{\mbox{\scriptsize\eqref{Fbound}}}{\leq} t\e^{\lambda(-\beta)},
}
where here we do not actually need $\|f\|\geq1/2$.
Similarly,
\eeq{ \label{equals1_sum3g}
\E\bigg(\frac{\big|\sum_{u \notin A^{(k)}} c_u'X_u'\big|}{\wt G}\bigg)
= \E\bigg|\sum_{u \notin A^{(k)}} c_u'X_u'\bigg|\E(\wt G^{-1})
\leq t\E(\wt G^{-1})
\stackrel{\mbox{\scriptsize\eqref{Gbound}}}{\leq} t\e^{\lambda(-\beta)}.
}

Combining \eqref{equals1_sum1}, \eqref{equals1_sum2}, and \eqref{equals1_sum3}, we conclude that if $\|f\| > 1/2$, then
\eeq{ \label{part_3_bound_2} \raisetag{2\baselineskip}
\E\Big(\frac{|\wt F - \wt G|}{\wt F}\Big) 
&\leq14\kappa(1-q)^{-1}\e^{\lambda(-\beta)}\E|\e^{\beta\omega}-\e^{\lambda(\beta)}| + t(1-\kappa)^{-1}\Big(\frac{1}{2}-\kappa\Big)^{-1}\e^{\lambda(-\beta)} + t\e^{\lambda(-\beta)} \\
&\hspace{-5.5ex}\stackrel{\phantom{xxx}\mbox{\scriptsize{\eqref{kappa_condition_1}}}\phantom{xxx}}{<} 14\kappa(1-q)^{-1}\e^{\lambda(-\beta)}\E|\e^{\beta\omega}-\e^{\lambda(\beta)}| + \frac{64t}{15}\e^{\lambda(-\beta)} \\
&\hspace{-6ex}\stackrel{\mbox{\scriptsize\eqref{t_condition_2},\eqref{kappa_condition_3}}}{<} \frac{\eps}{8\alpha}.
}
When $\|g\|>1/2$, we use \eqref{equals1_sum1g}, \eqref{equals1_sum2g}, and \eqref{equals1_sum3g} to conclude
\eeq{ \label{part_3_bound_2g} \raisetag{2\baselineskip}
\E\Big(\frac{|\wt F - \wt G|}{\wt G}\Big) 
&\leq14\kappa(1-q)^{-1}\e^{\lambda(-\beta)}\E|\e^{\beta\omega}-\e^{\lambda(\beta)}| + t(1-\kappa)^{-1}\Big(\frac{1}{2}-\kappa-\delta\Big)^{-1}\e^{\lambda(-\beta)} + t\e^{\lambda(-\beta)} \\
&\hspace{-6ex}\stackrel{\mbox{\scriptsize{\eqref{kappa_condition_1},\eqref{delta_condition_k}}}}{<} 14\kappa(1-q)^{-1}\e^{\lambda(-\beta)}\E|\e^{\beta\omega}-\e^{\lambda(\beta)}| + \frac{235t}{39}\e^{\lambda(-\beta)} \\
&\hspace{-6ex}\stackrel{\mbox{\scriptsize\eqref{t_condition_2},\eqref{kappa_condition_3}}}{<} \frac{\eps}{8\alpha}.
}
Having established \eqref{part_3_bound_1} and \eqref{part_3_bound_2}, as well as \eqref{part_3_bound_1g}, and \eqref{part_3_bound_2g} we may unconditionally write
\eeq{ \label{part_3_bound}
\E\Big(\frac{|\wt F - \wt G|}{\wt F}\Big) < \frac{\eps}{8\alpha}, \qquad \E\Big(\frac{|\wt F - \wt G|}{\wt G}\Big) < \frac{\eps}{8\alpha},
}
irrespective of the values of $\|f\|$ and $\|g\|$.

Now we check the inequalities \eqref{c_sum}--\eqref{c_prime_prime}.
First, for the $c_u$,
\eq{
\sum_{u \in A^{(k)}} |c_u| 
&\leq \sum_{u \in A^{(k)}} \bigg|\sum_{v \in \NN(u,k)} f(v)P(v,u) - \sum_{v \in \NN(\Phi(u),k)} g(v)P(v,\Phi(u))\bigg| \\
&\phantom{\leq}+ \sum_{u \in A^{(k)}}\bigg|\sum_{\substack{v\sim u \\ v \notin \NN(u,k)}} f(v)P(v,u) - \sum_{\substack{v \sim \Phi(u) \\ v \notin \NN(\Phi(u),k)}} g(v)P(v,\Phi(u))\bigg|,
}
where the first sum is bounded by $5\kappa$ according to $\eqref{all_u}$,
and the second sum satisfies
\eq{
\sum_{u \in A^{(k)}}\bigg|\sum_{\substack{v\sim u \\ v \notin \NN(u,k)}} f(v)P(v,u) - \sum_{\substack{v \sim \Phi(u) \\ v \notin \NN(\Phi(u),k)}} g(v)P(v,\Phi(u))\bigg|
&\leq \sum_{u \in \N\times\Z^d}\sum_{\substack{v\sim u \\ v \notin \NN(u,k)}} (f(v)+g(v))P(v,u) \\
&\hspace{-2.5ex}\stackrel{\phantom{\eqref{k_condition_2}}}{=}\sum_{v \in \N\times\Z^d}\sum_{\substack{u\sim v \\ u \notin \NN(v,k)}} (f(v)+g(v))P(v,u) \\
&\hspace{-2.5ex}\stackrel{\mbox{\scriptsize\eqref{k_condition_2}}}{<}\sum_{v\in\N\times\Z^d} (f(v)+g(v))\kappa \leq 2\kappa.
}
Hence \eqref{c_sum} holds:
\eq{
\sum_{u \in A^{(k)}} |c_k| \leq 5\kappa + 2\kappa = 7\kappa.
}
Next consider the $c_u'$.
By definition of $A^{(k)}$,
\eeq{ \label{notinAk_implication}
u \notin A^{(k)} \quad \implies \quad \NN(u,k) \subset (\N \times \Z^d) \setminus A,
}
which implies
\eq{
c'_u &=  \sum_{v \in \NN(u,k)} f(v)P(v,u) + \sum_{\substack{v \sim u \\ v \notin \NN(u,k)}} f(v)P(v,u) \\
&\leq \sum_{v \notin A} f(v)P(v,u) + P(\|\sigma_1\|_1 > k)
\stackrel{\mbox{\scriptsize\eqref{k_condition_2}}}{<} \sum_{v \notin A} f(v) + \kappa
\stackrel{\mbox{\scriptsize\eqref{A_condition}}}{<} 2\kappa < 7\kappa.
}
Last, consider the $c_u''$.
We have the implication
\eeq{ \label{notinPhiAk_implication}
u \notin \Phi(A^{(k)}) \stackrel{\mbox{\scriptsize\eqref{degree_consequence}}}{=} \Phi(A)^{(k)} \quad \implies \quad \NN(u,k) \subset (\N\times\Z^d)\setminus \Phi(A),
}
and thus
\eeq{ \label{c_prime_prime_setup}
|c_u''| &= \sum_{v \in \NN(u,k)}g(v)P(v,u) + \sum_{\substack{v\sim u \\ v \notin \NN(u,k)}} g(v)P(v,u) \\
&\leq \sum_{v \notin \Phi(A)} g(v)P(v,u) + \sum_{\substack{v\sim u \\ v \notin \NN(u,k)}} g(v)P(0,u-v) 
\leq \sup_{v \notin \Phi(A)} g(v) + P(\|\sigma_1\|_1 > k).
}
A further bound is needed:
\eeq{ \label{notinPhiA_sum}
\sum_{v \notin \Phi(A)} g(v)^\alpha
&\stackrel{\phantom{\mbox{\scriptsize\eqref{A_condition}}}}{\leq} \sum_{v \notin \psi(C)} g(v)^\alpha + \sum_{v \in \psi(C \setminus A)} g(v) \\
&\stackrel{\phantom{\mbox{\scriptsize\eqref{A_condition}}}}{=} \sum_{v \notin \psi(C)} g(v)^\alpha + \sum_{v \in C \setminus A} g(\psi(v)) \\
&\stackrel{\phantom{\mbox{\scriptsize\eqref{A_condition}}}}{\leq} \sum_{v \notin \psi(C)} g(v)^\alpha + \sum_{v \in C \setminus A} |f(v)-g(\psi(v))| + \sum_{v \in C \setminus A} f(v) \\
&\stackrel{\phantom{\mbox{\scriptsize\eqref{A_condition}}}}{\leq} \sum_{v \notin \psi(C)} g(v)^\alpha + \sum_{v \in C} |f(v) - g(\psi(v))| + \sum_{v \notin A} f(v) \\
&\stackrel{\mbox{\scriptsize\eqref{A_condition}}}{<} d_{\alpha,\psi}(f,g) + \kappa 
\stackrel{\mbox{\scriptsize\eqref{psi_condition}}}{<} \delta + \kappa
\stackrel{\mbox{\scriptsize\eqref{delta_condition_noA}}}{<} 2\kappa.
}
In particular,
\eq{
\sup_{v \in \Phi(A)} g(v) \leq \bigg(\sum_{v \notin \Phi(A)} g(v)^\alpha\bigg)^{1/\alpha} < (2\kappa)^{1/\alpha},
}
which completes \eqref{c_prime_prime_setup} to yield \eqref{c_prime_prime}:
\eq{ 
|c_u''|
\leq \sup_{v \notin \Phi(A)} g(v) + P(\|\sigma_1\|_1 > k)
\stackrel{\mbox{\scriptsize\eqref{k_condition_2}}}{\leq} (2\kappa)^{1/\alpha} + \kappa.
}

\subsubsection{Upper bounds for $\E\big[\sum_{u \notin A^{(k)}} \wt f(u)^\alpha\big]$ and $\E\big[\sum_{u \notin \Phi(A^{(k)})} \wt g(u)^\alpha\big]$}
By definition,
\eq{
\E\bigg[\sum_{u \notin A^{(k)}} \wt f(u)^\alpha\bigg]
= \sum_{u \notin A^{(k)}} \E\bigg(\sum_{v \sim u} f(v)\e^{\beta\omega_u}P(v,u)\bigg)^\alpha
= \e^{\lambda(\alpha\beta)}\sum_{u \notin A^{(k)}} \bigg(\sum_{v\sim u} f(v)P(v,u)\bigg)^\alpha.
}
Since the last inner sum is at most 1, we have
\eq{
\E\bigg[\sum_{u \notin A^{(k)}} \wt f(u)^\alpha\bigg]
\leq \e^{\lambda(\alpha\beta)}\sum_{u \notin A^{(k)}} \sum_{v\sim u} f(v)P(v,u),
}
where
\eq{
\sum_{u \notin A^{(k)}} \sum_{v\sim u} f(v)P(v,u)
&\stackrel{\phantom{\eqref{notinAk_implication}}}{=} \sum_{u \notin A^{(k)}}\bigg[ \sum_{v\in \NN(u,k)} f(v)P(v,u) + \sum_{\substack{v \sim u \\ v \notin \NN(u,k)}} f(v)P(v,u)\bigg] \\
&\stackrel{\mbox{\scriptsize\eqref{notinAk_implication}}}{\leq} \sum_{v \notin A} \sum_{u \sim v} f(v)P(v,u) + \sum_{v \in \N \times \Z^d} \sum_{\substack{u \sim v \\ u \notin \NN(v,k)}} f(v)P(v,u) \\
&\stackrel{\phantom{\eqref{notinAk_implication}}}{=} \sum_{v \notin A} f(v) + \sum_{v \in \N \times \Z^d} f(v)P(\|\sigma_1\|_1 > k)
\stackrel{\mbox{\scriptsize\eqref{A_condition},\eqref{k_condition_2}}}{<} 2\kappa.
}
Combining the two previous two displays yields
\eeq{
\E\bigg[\sum_{u \notin A^{(k)}} \wt f(u)^\alpha\bigg] < 2\kappa \e^{\lambda(\alpha\beta)}. \label{part_4_bound1}
}
For $\wt g$, Jensen's inequality gives
\eq{
\E\bigg[\sum_{u \notin \Phi(A^{(k))}} \wt g(u)^\alpha\bigg]
= \e^{\lambda(\alpha\beta)} \sum_{u \notin \Phi(A^{(k)})} \bigg(\sum_{v \sim u} g(v)P(v,u)\bigg)^\alpha
\leq \e^{\lambda(\alpha\beta)} \sum_{u \notin \Phi(A^{(k)})} \bigg(\sum_{v \sim u} g(v)^\alpha P(v,u)\bigg),
}
where
\eq{
 \sum_{u \notin \Phi(A^{(k)})} \bigg(\sum_{v \sim u} g(v)^\alpha P(v,u)\bigg)
 &\stackrel{\phantom{\eqref{k_condition_2}}}{=}  \sum_{u \notin \Phi(A^{(k)})}\bigg[\sum_{v \in \NN(u,k)} g(v)^\alpha P(v,u) + \sum_{\substack{v \sim u \\ v \notin \NN(u,k)}} g(v)^\alpha P(v,u)\bigg] \\
 &\stackrel{\mbox{\scriptsize\eqref{notinPhiAk_implication}\phantom{b}}}{\leq} \sum_{v \notin \Phi(A)}\sum_{u \in \NN(u,k)} g(v)^\alpha P(v,u) + \sum_{v \in \N \times \Z^d}\sum_{\substack{u \sim v \\ u \notin \NN(v,k)}} g(v)^\alpha P(v,u) \\
 &\stackrel{\mbox{\scriptsize\eqref{k_condition_2}}}{<} \sum_{v \notin \Phi(A)}g(v)^\alpha + \kappa
 \stackrel{\mbox{\scriptsize\eqref{notinPhiA_sum}}}{<} 3\kappa.
}
In summary,
\eeq{ \label{part_4_bound2}
\E\bigg[\sum_{u \notin \Phi(A^{(k))}} \wt g(u)^\alpha\bigg] < 3\kappa \e^{\lambda(\alpha\beta)}.
}
\subsubsection{Completing the proof}
Once \eqref{part_2_bound} and \eqref{part_3_bound} are used in \eqref{first_part} and \eqref{second_part}, respectively, \eqref{two_parts} gives
\eeq{ \label{final_sum_1}
\E\bigg[\alpha\sum_{u \in A^{(k)}} |F(u)-G(\Phi(u))|\bigg] < \alpha\Big(\frac{\eps}{8\alpha}+\frac{\eps}{8\alpha}\Big) = \frac{\eps}{4}.
}
Next, \eqref{Fbound} and \eqref{part_4_bound1} make \eqref{to_bound_2} read as
\eeq{ \label{final_sum_2}
\E\bigg[\sum_{u \notin A^{(k)}} F(u)^\alpha\bigg] < 2 \e^{\lambda(\alpha\beta)+\lambda(-\alpha\beta)}\kappa \stackrel{\mbox{\scriptsize\eqref{kappa_condition_5}}}{<} \frac{\eps}{4}.
}
Similarly, \eqref{Gbound} and \eqref{part_4_bound2} translate \eqref{to_bound_3} to
\eeq{ \label{final_sum_3}
\E\bigg[\sum_{u \notin \Phi(A^{(k))}} G(u)^\alpha\bigg] < 6\e^{\lambda(\alpha\beta)+\lambda(-\alpha\beta)}\kappa \stackrel{\mbox{\scriptsize\eqref{kappa_condition_5}}}{<} \frac{\eps}{4}.
}
Finally,
\eeq{ \label{final_sum_4}
\deg(\Phi) > k \quad {\implies} \quad 2^{-\deg(\Phi)} < 2^{-k} \stackrel{\mbox{\scriptsize\eqref{k_condition_1}}}{<} \frac{\eps}{4}.
}
Together, \eqref{final_sum_1}--\eqref{final_sum_4} produce the desired conclusion for part (a):
\eq{
W_\alpha(Tf,Tg) \leq \E[d_\alpha(F,G)] \leq \E(d_{\alpha,\Phi}(F,G)) < \eps.
}

For part (b), we assume $q$ is a positive integer.
By several applications of Cauchy--Schwarz, we have
\eeq{ \label{part_b_many_CS}
\big|\E(\log^q \wt F) - \E(\log^q \wt G)\big|
&\leq \E\, |\log^q \wt F - \log^q \wt G| \\
&= \E\, \bigg|\Big(\log\frac{\wt F}{\wt G}\Big)\sum_{i=0}^{q-1}(\log^{q-1-i} \wt F)(\log^i \wt G)\bigg| \\
&\leq \sqrt{\E\log^2\frac{\wt F}{\wt G}}\sqrt{\E\bigg[\bigg(\sum_{i=0}^{q-1}(\log^{q-1-i} \wt F)(\log^i \wt G)\bigg)^2\bigg]} \\
&\leq \sqrt{\E\log^2\frac{\wt F}{\wt G}}\sqrt{q\sum_{i=0}^{q-1}\sqrt{\E(\log^{4(q-1-i)} \wt F)\E(\log^{4i} \wt G)}}. 
}
Now, by the inequality $\log^2 x \leq 2(|x-1| +|x^{-1}-1|)$ for $x>0$, we have
\eq{
\E\log^2\frac{\wt F}{\wt G} \leq 2\Big(\E\, \Big|\frac{\wt F}{\wt G}-1\Big| + \E\, \Big|\frac{\wt G}{\wt F}-1\Big|\Big) \stackrel{\mbox{\footnotesize{\eqref{part_3_bound}}}}{\leq} \frac{\eps}{2}.
}
On the other hand, for each integer $i\geq0$, there exists a large enough constant $C_i$ so that $\log^{4i} x \leq C_i(x+x^{-1})$ for all $x>0$.
Therefore, we can appeal to \eqref{Fbound} and \eqref{Gbound} with $q=\pm1$ in order to upper bound the second square root in the final expression of \eqref{part_b_many_CS}.
In summary,
\eq{
\big|\E(\log^q \wt F) - \E(\log^q \wt G)\big| \leq C\sqrt{\eps},
}
where $C$ is a constant that depends only on $\mathfrak{L}_\omega$, $\beta$, $d$, and $q$.
The continuity of the map $f \mapsto \E(\log^q \wt F)$ is now clear.
\end{proof}

\subsection{Lifting the update map} \label{extension_to_measures}
Thus far we have only considered the random variable $F$ in \eqref{F_def} for \textit{fixed} $f \in \SS$.
If $f$ is itself a random element of $\SS$ drawn from the probability measure $\rho \in \PP(\SS)$, the resulting total law of $F$ will be written $\TT\rho$.
That is,
\eq{
\TT\rho(\AA) = \int Tf(\AA)\ \rho(\dd f), \quad \text{Borel $\AA \subset \SS$}.
}
More generally, for $\vphi : \SS\to\R$,
\eeq{
\int \vphi(g)\ \TT\rho(\dd g) = \int_\SS \int_\SS \vphi(g)\ Tf(\dd g)\, \rho(\dd f), \label{T_fubini}
}
where the integral is well-defined if $\vphi$ is nonnegative, or if
\eq{
\int_\SS \int_\SS |\vphi(g)|\ Tf(\dd g)\,\rho(\dd f) < \infty. 
}
In this notation, we have a kernel on measures $\PP(\SS) \to \PP(\SS)$ given by $\rho \mapsto \TT\rho$.
We can recover the map $T$ as  by restricting to Dirac measures; that is, 
$Tf = \TT\delta_f$, where $\delta_f \in \PP(\SS)$ is the unit point mass at $f$.

For completeness, let us check that the measure $\TT\rho$ is well-defined.
Specifically, we must check that for every Borel set $\AA\subset \SS$, $f \mapsto Tf(\AA)$ is measurable as a map $\SS \to \R$.
As we have just shown that $f \mapsto Tf$ is continuous and hence measurable, it suffices to prove the following lemma.

\begin{lemma} \label{evaluation_meas}
Let $(\XX,\tau)$ be a Polish space, and $A \subset \XX$ a Borel set.
The map $\PP(\XX) \to \R$ given by $\rho \mapsto \rho(A)$ is measurable.
\end{lemma}

\begin{proof}
We will make use of Dynkin's $\pi$-$\lambda$ theorem.
First define
\eq{
\CC \coloneqq  \{C \subset \XX : \text{$C$ is closed, $\rho \mapsto \rho(C)$ is measurable}\}.
}
It is easy to see that $\CC$ actually contains all closed sets: If $C \subset \XX$ is closed, then Lemma \ref{portmanteau} shows $\rho \mapsto \rho(C)$ is upper semi-continuous and hence measurable.
Furthermore, since closed sets remain closed under intersection, $C$ is a $\pi$-system.

Next define the larger collection
\eq{
\DD \coloneqq  \{A \subset \XX: \text{$A$ is Borel, $\rho \mapsto \rho(A)$ is measurable}\}.
}
We check that $\DD$ is a $\lambda$-system:
\begin{itemize}
\item Clearly the empty set is in $\DD$, since $\rho \mapsto \rho(\varnothing) = 0$ is constant.
\item If $A \in \CC$, then $\rho \mapsto \rho(\XX \setminus A) = 1 - \rho(A)$ is the difference of measurable functions, and so itself measurable.
That is, $\XX \setminus A \in \DD$.
\item If $A_1,A_2,\dots$ is a sequence in $\DD$ consisting of disjoint sets, then
\eq{
\rho\bigg(\bigcup_{i = 1}^\infty A_i\bigg) = \sum_{i = 1}^\infty \rho(A_i) = \lim_{n \to \infty} \sum_{i = 1}^n \rho(A_i),
}
meaning $\rho \mapsto \rho(\cup_{i = 1}^\infty A_i)$ is the limit of measurable functions, and so again measurable.
Therefore $\cup_{i = 1}^\infty A_i \in \DD$.
\end{itemize}
Now the $\pi$-$\lambda$ theorem implies that the $\sigma$-algebra generated by $\CC$ is contained in $\DD$.
As $\CC$ generates the Borel $\sigma$-algebra, the claim holds.
\end{proof}

For clarity, we review the measure theoretic definitions we have made.
Lemma \ref{evaluation_meas} verifies that $K(f,\dd g) \coloneqq  Tf(\dd g) : \SS \times \BB(\SS) \to [0,1]$ is a Markov kernel, where $\BB(\SS)$ denotes the Borel $\sigma$-algebra on $\SS$.
Then for any probability $\rho \in \PP(\SS)$, we can define $\TT\rho$ to be the product measure associated to $\rho$ and $K$:
For Borel subsets $\AA_1,\AA_2 \subset \SS$,
\eq{
\TT\rho(\AA_1 \times \AA_2) 
\coloneqq  (\rho \otimes K)(\AA_1 \times \AA_2)
= \int_{\AA_1} K(f,\AA_2)\ \rho(\dd f)
= \int_{\AA_1} Tf(\AA_2)\ \rho(\dd f).
}
We then know how to integrate functions $\Phi : \SS \times \SS \to \R$ with respect to $\TT\rho$, since
Fubini's theorem for kernels allows us to write
\eq{
\int_{\SS \times \SS} \Phi(f,g)\ \TT\rho(\dd f,\dd g) 
= \int_{\SS \times \SS} \Phi(f,g)\ (\rho \otimes K)(\dd f,\dd g)
&= \int_\SS \int_\SS \Phi(f,g)\ K(f,\dd g)\, \rho(\dd f) \\
&=  \int_\SS \int_\SS \Phi(f,g)\ Tf(\dd g)\, \rho(\dd f).
}
The above calculation holds when either $\Phi$ is nonnegative, or
\eq{
\int_\SS \int_\SS |\Phi(f,g)|\ Tf(\dd g)\, \rho(\dd f) < \infty.
}
Throughout the remainder of the chapter, we will only be interested in functions depending on a single coordinate: $\Phi(f,g) = \vphi(g)$, in which case
\eq{
 \int_\SS \int_\SS \Phi(f,g)\ Tf(\dd g)\, \rho(\dd f) = \int_\SS \int_\SS \vphi(g)\ Tf(\dd g)\, \rho(\dd f).
}
Therefore, it will not be a problem to henceforth write $\TT\rho$ for restriction of the product measure to the second coordinate:
$\TT\rho(\AA_2) = \TT\rho(\SS \times \AA_2)$.
That is, by $\TT\rho$ we mean a measure on $\SS$ rather than $\SS \times \SS$.
With this convention, the integral
\eq{
\int \vphi(g)\ \TT\rho(\dd g) = \int_\SS \int_\SS \vphi(g)\ Tf(\dd g)\, \rho(\dd f) 
}
is well-defined when either $\vphi$ is nonnegative, or
\eq{
\int_\SS \int_\SS |\vphi(g)|\ Tf(\dd g)\,\rho(\dd f) < \infty. 
}
Finally, we prove continuity for the lifted map $\rho \mapsto \TT\rho$.
%
\begin{prop} \label{continuous2}
$\rho\mapsto\TT\rho$ is a continuous map $(\PP(\SS),\WW_\alpha)\to(\PP(\SS),\WW_\alpha)$.
\end{prop}

\begin{proof}
Given $\eps > 0$, we may choose by Proposition \ref{continuous1} some $\delta > 0$ such that
\eeq{
d_\alpha(f,g) < \delta \quad \implies \quad \WW_\alpha(Tf,Tg) < \frac{\eps}{2}. \label{from_before}
}
Now suppose $\nu,\rho \in \PP(\SS)$ satisfy $\WW_\alpha(\nu,\rho) < \delta\eps/4$.
Then there exists a coupling $(f,g) \in \SS \times \SS$ of the distributions $\nu$ and $\rho$ such that
$\E[d_\alpha(f,g)] < \delta \eps/4$.
Denote the joint distribution of $(f,g)$ by $\pi \in \Pi(\nu,\rho)$.
Markov's inequality gives
\eeq{
\P(d_\alpha(f,g) \geq \delta) \leq \delta^{-1} \E[d_\alpha(f,g)] < \frac{\eps}{4}. \label{measure_bound}
}
To bound $\WW_\alpha(\TT\nu,\TT\rho)$, we will use definition \eqref{kantorovich}. For any Lipschitz function $\vphi : \SS \to \R$, we have
$\sup_{f \in \SS} |\vphi(f)| < \infty$ by compactness of $\SS$.
Considering such $\vphi$ with minimal Lipschitz constant at most 1, we can invoke \eqref{T_fubini} to write
\eeq{
\int \vphi(h)\ \TT\nu(\dd h) - \int \vphi(h)\ \TT\rho(\dd h)
&= \iint \vphi(h)\ Tf(\dd h)\, \nu(\dd f) - \iint \vphi(h)\ Tg(\dd h)\, \rho(\dd g) \\
&= \int\bigg(\int \vphi(h)\ Tf(\dd h) - \int \vphi(h)\ Tg(\dd h)\bigg)\, \pi(\dd f,\dd g)  \\
&\leq \int \WW_\alpha(Tf,Tg)\ \pi(\dd f,\dd g) 
= \E[\WW_\alpha(Tf,Tg)]. \label{each_vphi}
}
We now consider the expectation of $\WW_\alpha(Tf,Tg)$ over the set $\{d_\alpha(f,g) < \delta\}$, where we can apply \eqref{from_before}, and separately over the complement $\{d_\alpha(f,g) \geq \delta\}$, which has $\pi$-measure less than $\eps/4$ by \eqref{measure_bound}.
As the bound \eqref{each_vphi} holds for every $\vphi$, we have
\eq{
\WW_\alpha(\TT\nu,\TT\rho) 
&\leq \E\big[\WW_\alpha(Tf,Tg)\big] \\
&\leq \E\givenk[\big]{\WW_\alpha(Tf,Tg)}{d_\alpha(f,g) < \delta} + \frac{\eps}{4}\, \E\givenk[\big]{\WW_\alpha(Tf,Tg)}{d_\alpha(f,g) \geq \delta}
< \frac{\eps}{2} + \frac{\eps}{2} = \eps,
}
where we have applied Lemma \ref{trivial_bound} to bound the latter conditional expectation.
\end{proof}

\section{The empirical measure of the endpoint distribution} \label{free_energy}

\subsection{Definition and properties} \label{empirical_measures}
As discussed in Section \ref{endpoint_distributions}, the law of the endpoint distribution $f_{n+1}$ given $\FF_n$ is equal to $Tf_n$.
Now define the empirical probability measure on $\SS$ generated by the $f_i$,
\eeq{
\rho_n \coloneqq  \frac{1}{n} \sum_{i = 0}^{n-1} \delta_{f_i}. \label{mu_n_def}
}
Then $\rho_n$ is a random element of $\PP(\SS)$, measurable with respect to $\FF_n$.\footnote{By the discussion following Corollary \ref{increment_cor}, $f_i \in \SS$ is $\FF_n$-measurable for $0 \leq i \leq n$. It is clear that $f \mapsto \delta_f$ is a continuous (in particular, measurable) map $\SS \to \PP(\SS)$. It follows that $\rho_n \in \PP(\SS)$ is also $\FF_n$-measurable.}
While we will be interested in the quantity $\WW_\alpha(\rho_n,\TT\rho_n)$, it is easier to replace $\rho_n$ by 
the shifted empirical measure,
\eq{
\rho_n' \coloneqq  \frac{1}{n}\sum_{i=1}^{n} \delta_{f_i},
}
since $Tf_{i}$ is the distribution of $f_{i+1}$ given $\FF_{i}$.

Making use of the dual formulation \eqref{kantorovich} of Wasserstein distance, one has
\eq{
\WW_\alpha(\rho_n',\TT\rho_n) &= \sup_{\vphi} \biggl(\int \vphi(f)\ \rho_n'(\dd f) - \int \vphi(f)\ \TT\rho_n(\mathrm{d}f) \biggr)\\
&= \sup_\vphi \frac{1}{n}\sum_{i = 0}^{n-1} \bigl(\vphi(f_{i+1}) - \E\givenk{\vphi(f_{i+1})}{\FF_{i}}\bigr),
}
where the supremum is taken over all functions $\vphi : \SS \to \R$ satisfying 
\eeq{
|\vphi(f) - \vphi(g)| \leq d_\alpha(f,g) \quad \text{for all $f,g \in \SS$.} \label{lip1}
}
Notice that adding a constant to $\vphi$ does not change the value of
$\int \vphi(f)\, \rho'_n(\dd f) - \int \vphi(f)\, \TT\rho_n(\dd f)$, and so the supremum can equivalently be taken over such $\vphi$ satisfying $\vphi(\vc{0}) = 0$, where $\mathbf{0}$ denotes (the equivalence class of) the constant zero function.
Let us denote the set of such functions by
\eq{
\LL _\alpha \coloneqq \{\vphi : \SS \to \R : |\vphi(f) - \vphi(g)| \leq d_\alpha(f,g) \text{ for all $f,g \in \SS$},\ \vphi(\vc{0}) = 0\}.
}
Recall that the space of real-valued continuous functions on a compact metric space is equipped with the uniform norm,
\eq{
\|\vphi\|_\infty \coloneqq  \sup_{f \in \SS} |\vphi(f)| < \infty.
}
For $\vphi \in \LL_\alpha$, the Lipschitz condition \eqref{lip1} and Lemma \ref{trivial_bound} imply
$\|\vphi\|_\infty \leq 2$.
In particular, $\LL_\alpha$ is a uniformly bounded family of continuous functions.
Furthermore, since $\LL_\alpha$ consists of Lipschitz functions whose minimal Lipschitz constants are uniformly bounded, it is both equicontinuous and closed under the topology induced by the uniform norm.
By the Arzel\`a--Ascoli Theorem (see Munkres \cite[Theorem 47.1]{munkres00}), $\LL_\alpha$ is compact in this topology.
Having made this observation, we are now ready to prove the following convergence result.

\begin{prop} \label{primeT_as}
As $n \to \infty$, $\WW_\alpha(\rho_n',\TT\rho_n) \to 0$ almost surely.
\end{prop}

We will use the following well-known fact.

\begin{lemma}[{Burkholder--Davis--Gundy Inequality, see \cite[Theorem 1.1]{burkholder-davis-gundy72}}]
\label{bdg}
Let $(M_n)_{n \geq 0}$ be a martingale, and write
\eq{
M_n = \sum_{i = 0}^n d_i, \qquad d_0 \coloneqq M_0, \quad d_i \coloneqq M_{i}-M_{i-1} \text{ for $i \geq 1$}.
}
Let $M_n^* \coloneqq  \sup_{0 \leq i \leq n} M_n$.
Then for any $q \geq 1$, there are positive constants $c_q$ and $C_q$ such that
\eq{
c_q\, \E\bigg[\bigg(\sum_{i = 0}^n d_i^2\bigg)^{q/2}\bigg]
\leq \E\big[(M_n^*)^q\big]
\leq C_q\, \E\bigg[\bigg(\sum_{i = 0}^n d_i^2\bigg)^{q/2}\bigg] \quad \text{for all $n \geq 0$.}
}
\end{lemma}

\begin{proof}[Proof of Proposition \ref{primeT_as}]
We have
\eq{
\WW_\alpha(\rho_n',\TT\rho_n) = \sup_{\vphi \in \LL_\alpha} \frac{1}{n} \sum_{i = 0}^{n-1} \bigl(\vphi(f_{i+1}) - \E\givenk{\vphi(f_{i+1})}{\FF_i}\bigr).
}
Notice that for any fixed $\vphi \in \LL_\alpha$,
\eq{
M_{n}(\vphi) \coloneqq  \sum_{i = 0}^{n-1} \bigl(\vphi(f_{i+1}) - \E\givenk{\vphi(f_{i+1})}{\FF_i}\bigr)
}
defines a martingale $(M_n(\vphi))_{n \geq 0}$ adapted to the filtration $(\FF_{n})_{n \geq 0}$.
By Lemma \ref{bdg}, there is a constant $C = C(\vphi)$ such that
\eq{
\E\big[M_n(\vphi)^4\big] 
&\leq \E\big[M_n^*(\vphi)^4\big]
\leq C\, \E\bigg[\bigg(\sum_{i = 0}^{n-1} \big(\vphi(f_{i+1}) - \E\givenk{\vphi(f_{i+1})}{\FF_i}\big)^2\bigg)^2\bigg]
\leq 16Cn^2,  
}
where the final inequality follows from \eqref{lip1} and Lemma \ref{trivial_bound}.
A Markov bound now gives
\eq{
\sum_{n = 0}^\infty \P\Big(\frac{|M_n(\vphi)|}{n} > n^{-1/5}\Big)
= \sum_{n = 0}^\infty \P\big(M_n(\vphi)^4 > n^{16/5})
\leq \sum_{n = 0}^\infty n^{-16/5}\E\big[M_n(\vphi)^4\big]
&\leq \sum_{n = 0}^\infty 16Cn^{-6/5}
< \infty.
}
By Borel--Cantelli, we may conclude
\eeq{
\lim_{n \to \infty} \frac{|M_n(\vphi)|}{n} = 0 \quad \mathrm{a.s.} \label{phi_MG}
}
As discussed above, $\LL_\alpha$ is compact in the uniform norm topology.
In particular, it is separable.
Let $\vphi_1,\vphi_2,\dots$ be a countable, dense subset of $\LL_\alpha$.
Because of \eqref{phi_MG}, we can say that with probability one,
\eeq{
\lim_{n \to \infty} \frac{M_n(\vphi_j)}{n} = 0 \quad \text{for all $j \geq 1$.} \label{dense_to_0}
}
Assume that this almost sure event occurs.
Again from \eqref{lip1} and Lemma \ref{trivial_bound}, we know
\eq{
\|\vphi - \vphi'\|_\infty < \eps \quad \implies \quad \bigg|\frac{M_n(\vphi)}{n} - \frac{M_n(\vphi')}{n}\bigg| < 2\eps,
}
meaning $(M_n(\cdot)/n)_{n \geq 0}$ is an equicontinuous sequence of functions on the compact metric space $(\LL_\alpha,d_\alpha)$.
The assumption \eqref{dense_to_0} says that this family converges pointwise to $0$ on a dense subset.
The Arzel\`a--Ascoli Theorem forces this convergence to be uniform.
That is, for any $\eps > 0$, there is $N$ large enough that
\eq{
n \geq N \quad \implies \quad \WW_\alpha(\rho_n',\TT\rho_n) = \sup_{\vphi \in \LL_\alpha} \frac{M_n(\vphi)}{n} < \eps.
}
We conclude that $\WW_\alpha(\rho_n',\TT\rho_n)$ tends to 0 as $n \to \infty$.
As this holds given the almost sure event \eqref{dense_to_0}, we are done.
\end{proof}

\subsection{Convergence to fixed points of the update map} \label{convergence_fixed}
Proposition \ref{primeT_as} suggests that for large $n$, the empirical measure will be close to the set of fixed points of $\TT$:
\eeq{
\KK \coloneqq  \{\rho \in \PP(\SS) : \TT\rho = \rho\}. \label{K_def}
}
For $\UU \subset \PP(\SS)$, we will denote distance to $\UU$ by
\eq{
\WW_\alpha(\rho,\UU) \coloneqq  \inf_{\nu \in \UU} \WW_\alpha(\rho,\nu), \quad \rho \in \PP(\SS).
}

\begin{cor} \label{close_probability}
As $n \to \infty$, $\WW_\alpha(\rho_n,\KK) \to 0$ almost surely.
\end{cor}

\begin{proof}
Recall that $\vc{0}$ is the element of $\SS$ whose unique representative in $\SS_0$ is the constant zero function.
Notice that $\TT\delta_\vc{0} = \delta_\vc{0}$ so that $\KK$ is nonempty.
Next observe that from Lemma \ref{trivial_bound}, we have the trivial bound $\WW_\alpha(\rho_n,\rho_{n}') \leq 2/n$, and so Proposition \ref{primeT_as} immediately implies $\WW_\alpha(\rho_n,\TT\rho_n) \to 0$ almost surely.
On this almost sure event, it follows that $\WW_\alpha(\rho_n,\KK) \to 0$, since otherwise there would exist a subsequence $\rho_{n_k}$ remaining a fixed positive distance away from $\KK$.
By compactness, we could assume $\rho_{n_k}$ converges to some $\rho$, but then continuity of $\TT$ would force $\WW_\alpha(\rho,\TT\rho) = 0$; that is, $\rho_{n_k}$ converges to an element of $\KK$, which is a contradiction.
%
\end{proof}

Now that the set $\KK$ is seen to contain all possible limits of the empirical measure, we should like to have some description of its elements.
One suggestive fact proved below is that any measure in $\KK$ places all its mass on those elements of $\SS$ with norm 0 or 1.
This observation will be crucial in proving our characterization of the low temperature phase in Section \ref{empirical_limits}.

\begin{prop} \label{no_middle}
If $\rho \in \KK$, then
\eq{
\rho(\{f \in \SS : 0 < \|f\| < 1\}) = 0.
}
\end{prop}

\begin{proof}
First take $f \in \SS$ to be non-random.
Then $Tf$ is the law of the random function
\eq{
F(u) = \frac{\sum_{v \sim u} f(v) \e^{\beta \omega_u}P(v,u)}{\sum_{w \in \N \times \Z^d} \sum_{v \sim w} f(v) \e^{\beta \omega_w}P(v,w) + (1 - \|f\|){\e^{\lambda(\beta)}}}.
}
If $\|f\| = 0$ or $\|f\| = 1$, then $\|F\| = \|f\|$.
If instead $0 < \|f\| < 1$, then $\|F\|$ is random and still satisfies $0 < \|F\| < 1$.
By summing over $u \in \N \times \Z^d$, we have
\eq{
\E(\|F\|) = \E\bigg[\frac{\sum_{u \in \N \times \Z^d} \sum_{v \sim u} f(v) \e^{\beta \omega_u}P(v,u)}{\sum_{u \in \N \times \Z^d} \sum_{v \sim u} f(v) \e^{\beta \omega_u}P(v,u) + (1-\|f\|){\e^{\lambda(\beta)}}}\bigg].
}
Upon observing that for any constant $C > 0$, the mapping $t \mapsto \frac{t}{t+C}$
is strictly concave, we deduce from Jensen's inequality that
\eq{
\E(\|F\|) &\leq \frac{\E \big[\sum_{u \in \N \times \Z^d} \sum_{v \sim u} f(v) \e^{\beta \omega_u}P(v,u)\big]}{\E\big[\sum_{u \in \N \times \Z^d} \sum_{v \sim u} f(v) \e^{\beta \omega_u}P(v,u)\big] + (1-\|f\|){\e^{\lambda(\beta)}}} 
= \frac{\|f\| \e^{\lambda(\beta)}}{\e^{\lambda(\beta)}} = \|f\|,
}
where equality holds if and only if $\sum_{u \in \N \times \Z^d} \sum_{v \sim u} f(v) \e^{\beta \omega_u}P(v,u)$ is an almost sure constant.
Since the disorder distribution $\mathfrak{L}_\omega$ is not a point mass, this is not the case.

Now let $\rho \in \PP(\SS)$, and take $f \in \SS$ to be random with law $\rho$.
If $\rho$ assigns positive measure to the set $\{f \in \SS : 0 < \|f\| < 1\}$, then the above argument gives
\eq{
\int \|g\|\ \TT\rho(\dd g) = \iint \|F\|\ Tf(\dd F)\, \rho(\dd f) < \int \|f\|\ \rho(\dd f).
}
But when $\rho \in \KK$ 
so that $\TT\rho = \rho$, the first and last expressions above are equal. 
We conclude that any $\rho \in \KK$ must assign mass 0 to the set $\{f \in \SS : 0 < \|f\| < 1\}$.
\end{proof}

\subsection{Variational formula for the free energy} \label{calculations}
In order to connect the results of Section \ref{convergence_fixed} to the temperature conditions of Theorems \ref{intro_result1} and \ref{intro_result2}, we will need to relate the free energy $F_n(\beta) = \frac{1}{n}\log Z_n(\beta)$ to the abstract objects we have introduced.
Recall the Doob decomposition from \eqref{R_predef}--\eqref{doob_decomposition}:
\eq{
\log Z_n(\beta) = M_n + A_n, \qquad \text{where} \qquad M_n = \sum_{i=0}^{n-1} d_i, \quad  A_n = \sum_{i=0}^{n-1} R(f_i).
}
If we lift the functional $R : \SS\to\R$ of \eqref{R_def} to $\RR:\PP(\SS)\to\R$ by defining
\eeq{
\RR(\rho) \coloneqq  \int R(f)\ \rho(\dd f), \quad \rho\in\PP(\SS), \label{R_def2}
}
then we can conveniently rewrite the above decomposition as
\eeq{ \label{Fn_ito_empirical}
F_n(\beta) = \frac{\log Z_n(\beta)}{n} = \frac{M_n}{n} + \RR(\rho_n),
}
where $\rho_n$ is the empirical measure from \eqref{mu_n_def}.
Our next result, which says the mean-zero martingale $(M_n)_{n\geq0}$ has a vanishing contribution in the above expression, is merely a reinterpretation of arguments appearing in earlier works (e.g.~\cite{carmona-hu02,comets-shiga-yoshida03,comets-shiga-yoshida04,carmona-hu06}).

\begin{prop} \label{FR_prop}
As $n \to \infty$, $|F_n(\beta) - \RR(\rho_{n})| \to 0$ almost surely.
\end{prop}

\begin{proof}
By Corollary \ref{increment_cor}, we have $\E(d_i^4) \leq C_1$ for some constant $C_1$ independent of $i$.
From Lemma \ref{bdg}, we deduce that
\eq{
\E(M_n^4) \leq C_2\, \E\bigg[\bigg(\sum_{i = 0}^{n-1} d_i^2 \bigg)^2\bigg]
\leq C_2\sum_{i = 0}^{n-1} \sum_{j = 0}^{n-1} \sqrt{\E(d_i^4)\E(d_j^4)}
\leq Cn^2,
}
where the constant $C = C_1C_2$ is independent of $n$.
It follows that $\E\big[(n^{-1}M_n)^4] \leq Cn^{-2}$.
As in the proof of Proposition \ref{primeT_as}, an argument using Markov's inequality and Borel--Cantelli now shows
\eq{
\lim_{n \to \infty} |F_n(\beta) - \RR(\rho_{n})| = \lim_{n \to \infty} \frac{|M_n|}{n} = 0 \quad \mathrm{a.s.}
}
\end{proof}

Given the convergence of $\rho_{n}$ to the set $\KK$, one should expect $\E F_n(\beta) = \E\,\RR(\rho_{n})$ to become close to $\RR(\rho)$ for some $\rho \in \KK$.
One difficulty is that $\rho_{n}$ does not converge to a particular $\rho \in \KK$, but rather becomes arbitrarily close to the set $\KK$. 
Nevertheless, we can instead consider the subset
\eeq{
\MM = \Big\{\rho_0 \in \KK : \RR(\rho_0) = \inf_{\rho \in \KK} \RR(\rho)\Big\}, \label{M_def}
}
and show convergence to $\MM$.
By Proposition \ref{continuous1}(b) and Lemma \ref{portmanteau}, $\RR$ is continuous. 
Therefore, since $\KK$ is compact (being a closed subset of a compact metric space), the set $\MM$ is nonempty.
Moreover, $\MM$ is a closed subset of the compact space $\KK$, and so $\MM$ is compact.
The first step in proving the desired convergence is the following consequence of Corollary \ref{close_probability}.

\begin{prop} \label{lower_bound}
Let $\KK$ be defined by \eqref{K_def}.
Let $\RR : \PP(\SS) \to \R$ be defined by \eqref{R_def2}.
Then
\eeq{
\liminf_{n \to \infty} F_n(\beta) \geq \inf_{\rho \in \KK} \RR(\rho) \quad \mathrm{a.s.} \label{lower_bound_as}
}
In particular,
\eeq{
\liminf_{n \to \infty} \E F_n(\beta) \geq \inf_{\rho \in \KK} \RR(\rho). \label{lower_bound_eq}
}
\end{prop}

\begin{proof}
$\RR$ is continuous on the compact metric space $(\PP(\SS),\WW_\alpha)$, and thus uniformly continuous.
Therefore, a simple argument using Corollary \ref{close_probability} gives
\eeq{
\liminf_{n \to \infty} \RR(\rho_{n}) \geq \inf_{\rho \in \KK} \RR(\rho) \quad \mathrm{a.s.} \label{lower_bound_asR}
}
Since $\RR$ is bounded, we may apply Fatou's Lemma to obtain
\eeq{
\liminf_{n \to \infty} \E\, \RR(\rho_n) \geq \E\Big[\liminf_{n \to \infty} \RR(\rho_{n})\Big] \geq \inf_{\rho \in \KK} \RR(\rho). \label{lower_bound_eqR}
}
Now \eqref{lower_bound_as} follows from \eqref{lower_bound_asR} and Proposition \ref{FR_prop}, and \eqref{lower_bound_eq} follows from \eqref{lower_bound_eqR} and the fact that $\E F_n(\beta) = \E\, \RR(\rho_n)$.
\end{proof}

Following Proposition \ref{lower_bound}, we naturally ask if there is a matching upper bound.
The next result answers this question in the affirmative.
To state the full theorem, we need to denote one element of $\SS$ in particular.
Notice that for $f \in \SS_0$ satisfying $f(u) = 1$ for some $u \in \N \times \Z^d$, Proposition \ref{superisometry} implies $d_\alpha(f,g) = 0$ for $g \in \SS_0$ if and only if $g(v) = 1$ for some $v \in \N \times \Z^d$.
We can thus define the element $\vc{1} \in \SS$ whose representatives in $\SS_0$ are the norm-1 functions supported on a single point.

\begin{thm} \label{upper_bound}
Let $\KK$ be defined by \eqref{K_def}.  
Let $\RR : \PP(\SS) \to \R$ be defined by \eqref{R_def2}.
Then
\eq{
\limsup_{n \to \infty} \E F_n(\beta) \leq \inf_{\rho \in \KK} \RR(\rho),
}
and so
\eeq{
\lim_{n \to \infty} F_n(\beta) = \inf_{\rho \in \KK} \RR(\rho) \quad \mathrm{a.s.} \label{limit}
}
The minimum value is equal to
\eeq{
\inf_{\rho \in \KK} \RR(\rho) = \lim_{n \to \infty} \frac{1}{n} \sum_{i = 0}^{n-1} \RR(\TT^i\delta_{\vc{1}}). \label{minimal_value}
}
\end{thm}
The key ingredient in the proof of Theorem \ref{upper_bound} is the following lemma.
When $\|f_0\|=1$, the result can be regarded as Jensen's inequality applied to a weighted sum of partition functions, each one corresponding to a different starting vertex chosen at random according to $f_0$.

\begin{lemma} \label{upper_bound_lemma}
For any $f_0 \in \SS$ and $n \geq 1$,
\eq{
\sum_{i = 0}^{n-1} \RR(\TT^i \delta_{f_0}) \geq \E\log Z_n(\beta),
}
where $\delta_{f_0} \in \PP(\SS)$ is the unit mass at $f_0$.
Equality holds if and only if $f_0 = \vc{1}$.
\end{lemma}

\begin{proof}
If $f_0 = \vc 0$, then $\RR(\TT^i\delta_{\vc 0}) = \RR(\delta_{\vc 0}) = \lambda(\beta)$, and so the desired inequality is immediate from \eqref{jensen_applied_finite}.
So we may assume $f_0 \neq \vc 0$.
Fix a representative $f_0 \in \SS_0$.
Let $(\omega^{(i)}_u)_{u \in \N \times \Z^d}$, $1 \leq i \leq n$, be independent collections of i.i.d.~random variables with law $\mathfrak{L}_\omega$.
For $1 \leq i \leq n-1$, inductively define $f_i \in \SS$ to have representative
\eq{
f_i(u) = \frac{\sum_{v \sim u} f_{i-1}(v) \e^{\beta \omega^{(i)}_u}P(v,u)}{ \sum_{w \in \N \times \Z^d}\sum_{v \sim w} f_{i-1}(v) \e^{\beta \omega^{(i)}_w}P(v,w) + (1-\|f_{i-1}\|){\e^{\lambda(\beta)}}},
}
so that the law of $f_i$ is equal to $\TT$ applied to the law of $f_{i-1}$.
By induction, then,
the law of $f_i$ is $\TT^i \delta_{f_0}$.
By definitions \eqref{R_def2} and \eqref{R_def},
\eq{
\RR(\TT^i \delta_{f_0})
&= \E \log \wt F_{i},
}
where 
\eq{
\wt F_i &\coloneqq  \sum_{u_{i+1} \in \N \times \Z^d}\sum_{u_{i} \sim u_{i+1}} f_{i}(u_{i})\e^{\beta \omega^{(i+1)}_{u_{i+1}}}P(u_i,u_{i+1}) + (1 - \|f_{i}\|){\e^{\lambda(\beta)}}, 
}
and the expectation is over both $f_i$ and the $\omega^{(i+1)}$ variables.
Repeatedly rewriting $f_i$ in terms of $f_{i-1}$ leads to the identity
\eq{
\wt F_0\wt F_1\cdots \wt F_{n-1}
&= \sum_{u_n\in\N\times\Z^d}\sum_{u_0\sim u_1\sim\cdots\sim u_n} f_0(u_0)\e^{\beta\sum_{i=1}^n\omega_{u_i}^{(i)}}\prod_{i=1}^nP(u_{i-1},u_i) + (1-\|f_0\|){\e^{n\lambda(\beta)}} \\
&= \sum_{u_0 \in \N \times \Z^d} \frac{f_0(u_0)}{\|f_0\|}\bigg[\|f_0\|\sum_{u_{n} \sim \cdots \sim u_0}\e^{\beta \sum_{i = 1}^n \omega^{(i)}_{u_i}}\prod_{i=1}^nP(u_{i-1},u_i) + (1 - \|f_0\|)\E Z_{n}(\beta)\bigg].
}
Using the concavity of the $\log$ function in three consecutive steps, we deduce
\eq{
&\log(\wt F_0\wt F_1\cdots \wt F_{n-1}) \\
&\geq 
\sum_{u_0} \frac{f_0(u_0)}{\|f_0\|} \log \bigg[\|f_0\|\sum_{u_{n} \sim \cdots \sim u_0}\e^{\beta \sum_{i = 1}^n \omega^{(i)}_{u_i}}\prod_{i=1}^nP(u_{i-1},u_i)+ (1 - \|f_0\|)\E Z_{n}(\beta)\bigg] \\
&\geq 
\sum_{u_0} \frac{f_0(u_0)}{\|f_0\|} \bigg[\|f_0\|\log\sum_{u_{n} \sim \cdots \sim u_0}\e^{\beta \sum_{i = 1}^n \omega^{(i)}_{u_i}}\prod_{i=1}^nP(u_{i-1},u_i) +(1 - \|f_0\|)\log \E Z_{n}(\beta)\bigg] \\
&\geq 
\sum_{u_0} \frac{f_0(u_0)}{\|f_0\|} \bigg[\|f_0\|\log\sum_{u_{n} \sim \cdots \sim u_0}\e^{\beta \sum_{i = 1}^n \omega^{(i)}_{u_i}}\prod_{i=1}^nP(u_{i-1},u_i) +(1 - \|f_0\|)\E\log Z_{n}(\beta)\bigg],
}
where equality holds throughout if and only if $f_0(u_0) = 1$ for some $u_0 \in \N\times\Z^d$.
Since the random variable $\sum_{u_{n} \sim \cdots \sim u_0}\exp\big(\beta \sum_{i = 1}^n \omega^{(i)}_{u_i}\big)\prod_{i=1}^nP(u_{i-1},u_i)$ is equal in law to $Z_n(\beta)$ for any fixed $u_0 \in \N \times \Z^d$, taking expectation yields
\eq{
\E \log(\wt F_0\wt F_1\cdots \wt F_{n-1})
&\geq \sum_{u_0} \frac{f_0(u_0)}{\|f_0\|} \Big(\|f_0\|\, \E \log Z_n(\beta) + (1 - \|f_0\|)\, \E\log Z_n(\beta)\Big) 
= \E\log Z_n(\beta).
}
It follows that
\eq{
\sum_{i = 0}^{n-1}  \RR(\TT^i \delta_{f_0})
&= \sum_{i = 0}^{n-1} \E\log \wt F_{i}
= \E \log(\wt F_0\wt F_1\cdots \wt F_{n-1})
\geq \E\log Z_n(\beta),
}
with equality if and only if $f_0 = \vc{1}$.
\end{proof}

\begin{proof}[Proof of Theorem \ref{upper_bound}]
Consider any $\rho \in \KK$.
Using the fact that $\TT\rho = \rho$ and Lemma \ref{upper_bound_lemma}, we have
\eq{
\RR(\rho) = \frac{1}{n} \sum_{i = 0}^{n-1} \RR(\TT^i \rho) 
&= \int \frac{1}{n} \sum_{i = 0}^{n-1} \RR(\TT^i\delta_f)\ \rho(\dd f) 
\geq \int \frac{1}{n}\, \E \log Z_n(\beta)\ \rho(\dd f)
= \E F_n(\beta).
}
We note that linearity of $\TT$ and Fubini's theorem \eqref{T_fubini} have been applied, which is permissible since Proposition \ref{continuous1}(b) shows that $R$ is continuous on $\SS$ and hence bounded.

As the above estimate holds for every $\rho \in \KK$ and every $n$, we have
\eeq{
\limsup_{n \to \infty} \E F_n(\beta) \leq \inf_{\rho \in \KK} \RR(\rho). \label{upper_bound_eq}
}
It now follows from \eqref{lower_bound_eq} and \eqref{upper_bound_eq} that
\eeq{
\lim_{n \to \infty} \E F_n(\beta) = \inf_{\rho \in \KK} \RR(\rho). \label{expectations_converge}
}
Then \eqref{Fn_lim} improves \eqref{expectations_converge} to \eqref{limit}.
Finally, equation \eqref{minimal_value} follows from the final statement in Lemma~\ref{upper_bound_lemma}.
\end{proof}

We now turn to strengthening Corollary \ref{close_probability} by proving convergence not only to $\KK$, but to the smaller set $\MM$, defined in \eqref{M_def}.

\begin{thm} \label{close_M}
As $n \to \infty$, $\WW_\alpha(\rho_n,\MM) \to 0$ almost surely.
\end{thm}

\begin{proof}
As $n\to\infty$, $\WW_\alpha(\rho_n,\KK) \to 0$ almost surely (by Corollary \ref{close_probability}), and $\RR(\rho_n)$ converges almost surely to $\inf_{\rho\in\KK}\RR(\rho)$ (by Proposition \ref{FR_prop} and Theorem \ref{upper_bound}).
Therefore, by continuity of $\RR$ and compactness of $\PP(\SS)$, we have $\WW_\alpha(\rho_n,\KK_\delta) \to 0$ almost surely, where $\delta$ is any positive number, and
\eq{
\KK_\delta \coloneqq \Big\{\rho_0 \in \KK : \RR(\rho_0) < \inf_{\rho\in\KK} \RR(\rho)+\delta\Big\}.
} 

Now, given any $\eps > 0$, we can choose $\delta$ such that $\sup_{\nu\in\KK_\delta}\WW_\alpha(\nu,\MM) < \eps$.
(Indeed, if this were not the case, then one could find a sequence $(\nu_k)_{k \geq 1}$ in $\KK$ such that
$\RR(\nu_k) \searrow \inf_{\rho \in \KK} \RR(\rho)$ as $k \to \infty$, but $\WW_\alpha(\nu_k,\MM) \geq \eps$ for all $k$.
Since $\KK$ is compact, we may pass to a subsequence and assume $\nu_k$ converges to some $\nu \in \KK$.
In particular, $\WW_\alpha(\nu,\MM) \geq \eps$.
But the continuity of $\RR$ implies $\RR(\nu) = \inf_{\rho \in \KK} \RR(\rho)$, meaning $\nu \in \MM$, a contradiction.)
As $\WW_\alpha(\rho_n,\KK_\delta) \to 0$ almost surely and $\eps > 0$ is arbitrary, we are done.
\end{proof}

\section{Limits of empirical measures} \label{empirical_limits}
In the first part of this section, we give a characterization of the low temperature regime in terms of the fixed points of the update map $\TT$.
This is stated as Theorem \ref{characterization}(b). 
We know from Theorem \ref{close_M} that those fixed points minimizing the energy functional $\RR$ constitute the possible limits of the empirical measure $\rho_n$, and so this characterization will allow us to study the Ces\`aro asymptotics of endpoint distributions.

In Section \ref{example_application}, we make this connection more concrete by outlining the general steps by which the theory developed in the previous three sections can be used to prove statements about directed polymers.
Indeed, the strategy described will be employed in Sections \ref{main_thm} and \ref{main_thm2}.
But to provide a more brief, instructional example in the present section, we give a new proof of Theorem \ref{characterization0},
which offered a first characterization of the low temperature phase.
Once appropriate definitions are made, only a short argument is needed to prove the statement.


\subsection{Characterizing high and low temperature phases} \label{norm_monotonicity}
Recall that $\vc{0}$ is the element of $\SS$ whose unique representative in $\SS_0$ is the constant zero function,
and $\vc{1}$ is the element of $\SS$ whose representatives in $\SS_0$ are the norm-1 functions supported on a single point.

\begin{lemma} \label{unique_max}
The function $R : \SS \to \R$ from \eqref{R_def} achieves a unique minimum $R(\vc{1}) = \E\log Z_1(\beta)$,
and a unique maximum $R(\vc{0}) = \lambda(\beta)$.
\end{lemma}

\begin{proof}
The first statement is immediate from Lemma \ref{upper_bound_lemma} by taking $n = 1$.
The second statement follows from the Jensen inequality $R(f) \leq \log \E(\wt F) = \lambda(\beta)$, where $\wt F$ is defined in \eqref{R_def}.
Moreover, the inequality is strict whenever $f \neq \vc{0}$, since then $\wt{F}$ is not an almost sure constant.
\end{proof}

Now we characterize the high and low temperature regimes by the elements of $\KK$ and $\MM$, defined by \eqref{K_def} and \eqref{M_def}, respectively. 
Notice that $\delta_{\vc 0}$ is always an element of $\KK$; the temperature regime is determined by whether it is also an element of $\MM$.

\begin{thm} \label{characterization}
Assume \eqref{walk_assumption_1} and \eqref{mgf_assumption}. 
Then the following statements hold:
\begin{itemize}
\item[(a)] If $0 \leq \beta \leq \beta_{\mathrm{c}}$, then $\KK = \MM = \{\delta_{\vc{0}}\}$.
\item[(b)] If $\beta > \beta_{\mathrm{c}}$, then $\rho(\{f \in \SS : \|f\| = 1\}) = 1$ for every $\rho \in \MM$, and so $\TT$ has more than one fixed point.
\end{itemize}
\end{thm}

\begin{proof}
Since the hypotheses of (a) and (b) are complementary, it suffices to prove their converses.
We always have $\TT\delta_{\vc{0}} = \delta_{\vc{0}}$.
If $\KK$ contains no other elements of $\PP(\SS)$, then $\MM = \{\delta_{\vc{0}}\}$ and Theorem \ref{upper_bound} shows
\eq{
p(\beta) =  \lim_{n \to \infty} \E F_n(\beta) = \RR(\delta_\vc{0}) = R(\vc{0}) = \lambda(\beta).
}
That is, $0 \leq \beta \leq \beta_{\mathrm{c}}$ when $\KK = \MM = \{\delta_{\vc{0}}\}$.

If instead there exists $\rho \in \KK$ distinct from $\delta_\vc{0}$, then Proposition \ref{no_middle} implies that $\rho$ assigns positive mass to the set
\eq{
\UU \coloneqq  \{f \in \SS : \|f\| = 1\}, 
}
which is measurable by Lemma \ref{norm_equivalence}.
Moreover, Proposition \ref{no_middle} guarantees $\rho(\{\vc{0}\}) = 1 - \rho(\UU)$,
and Lemma \ref{unique_max} shows $R(f) < R(\vc{0})$ for all $f \in \UU$.
It follows that
\eq{
\lim_{n \to \infty} \E F_n(\beta) \leq \RR(\rho) = \int R(f)\ \rho(\dd f) < R(\vc{0}).
}
In this case we have $p(\beta) < \lambda(\beta)$, meaning $\beta > \beta_{\mathrm{c}}$.
Furthermore, we can consider the probability measure
\eq{
\rho_\UU(\AA) \coloneqq  \frac{\rho(\AA \cap \UU)}{\rho(\UU)}, \quad \text{Borel } \AA \subset \SS.
}
Notice that $\UU$ and $\SS\setminus\UU$ are both closed under $\TT$:
\begin{align}
f \in \UU \quad \implies \quad Tf(\UU)= 1, \label{invariant1} \\
f \notin \UU \quad \implies \quad Tf(\UU)  = 0. \label{invariant2}
\end{align}
This observation was made in the proof of Proposition \ref{no_middle}.
Therefore, $\TT\rho = \rho$ implies $\rho_\UU$ is also an element of $\KK$.
Indeed, for every Borel set $\AA \subset \SS$,
\eq{
\TT\rho_\UU(\AA)
= \int_\SS Tf(\AA)\ \rho_\UU(\dd f)
= \int_\UU  Tf(\AA)\ \rho_\UU(\dd f)
= \int_\UU  Tf(\AA\cap \UU)\ \rho_\UU(\dd f),
}
where the second equality is justified by $\rho_\UU(\UU) = 1$, and the third equality by \eqref{invariant1}.
Then using the fact that $\rho_\UU = \rho/\rho(\UU)$ on $\UU$, followed by \eqref{invariant2}, we in turn get
\eq{
\TT\rho_\UU(\AA) = \int_\UU  Tf(\AA\cap \UU)\ \rho_\UU(\dd f)
= \frac{1}{\rho(\UU)} \int_\UU  Tf(\AA\cap \UU)\ \rho(\dd f)
&= \frac{1}{\rho(\UU)}  \int_\SS  Tf(\AA\cap \UU)\ \rho(\dd f) \\
&= \frac{\rho(\AA \cap \UU)}{\rho(\UU)} 
= \rho_\UU(\AA).
}
We have thus shown $\TT\rho_\UU(\AA) = \rho_\UU(\AA)$ for every Borel subset $\AA$, and so $\rho_\UU \in \KK$ as claimed.

If $\rho(\UU) < 1$, then $\rho_\UU$ satisfies
\eq{
\RR(\rho_\UU) = \frac{1}{\rho(\UU)} \int_\UU R(f)\ \rho(\dd f)
&= \int_\UU R(f)\ \rho(\dd f) + \frac{1 - \rho(\UU)}{\rho(\UU} \int_\UU R(f)\ \rho(\dd f) \\
&< \int_\UU R(f)\ \rho(\dd f) + \big(1 - \rho(\UU)\big)R(\vc{0}) = \RR(\rho).
}
It follows that $\rho \in \MM$ only if $\rho(\UU) = 1$.
\end{proof}

\subsection{An illustrative application} \label{example_application}
To demonstrate how the abstract setup can be employed to prove results on directed polymers, we will use it to prove the following generalization of Theorem \ref{characterization0}.
 
\begin{thm} \label{reprove_equivalence}
Assume \eqref{walk_assumption_1} and \eqref{mgf_assumption}.
Then the following statements hold:
\begin{itemize}
\item[(a)] If $0 \leq \beta \leq \beta_{\mathrm{c}}$, then
\eq{
\lim_{n \to \infty} \frac{1}{n} \sum_{i = 0}^{n-1} \max_{x \in \Z^d} f_i(x) = 0 
 \quad \mathrm{a.s.}
 }
\item[(b)] If $\beta > \beta_{\mathrm{c}}$, then there exists $c > 0$ such that
\eeq{
\liminf_{n \to \infty} \frac{1}{n} \sum_{i = 0}^{n-1}  \max_{x \in \Z^d} f_i(x) \geq c \quad \mathrm{a.s.} \label{liminf_max}
}
 \end{itemize}
\end{thm}
The first step is to define the functional(s) relevant to the problem.
At first, definitions are made on the space $\SS_0$.
For instance, from \eqref{liminf_max} we are motivated to define $\max : \SS_0 \to [0,1]$ by
\eq{
\max(f) \coloneqq  \max_{u \in \N \times \Z^d} f(u).
}
By appealing to Corollary \ref{defined_pspm}, we see that the functional is well-defined on the quotient space $\SS$.
This should be true in order for the abstract machinery to be applicable.
Notice that when $f_i$ is seen as an element of $\SS$,
\eeq{
\frac{1}{n} \sum_{i = 0}^{n-1} \max_{x \in \Z^d} f_i(x) = \int \max(f)\ \rho_{n}(\dd f). \label{max_abstract}
}
The second step is to prove a continuity condition so that Portmanteau (Lemma \ref{portmanteau}) may be applied.
Therefore, it is important to understand the topology induced by the metric $d_\alpha$.
Both here and in later sections, this step is the most technical part of the process, but is generally straightforward and follows a similar proof strategy.

\begin{lemma} \label{max_lemma}
The function $\max : \SS \to [0,1]$ given by
\eq{
\max(f) \coloneqq  \max_{u \in \N \times \Z^d} f(u)
}
is continuous and thus measurable.
\end{lemma}

\begin{proof}
Let $f \in \SS$ be given, and fix a representative in $\SS_0$.
Choose $u \in \N \times \Z^d$ such that $f(u) = \max_{v \in \N \times \Z^d} f(v)$.
If $f(u) = 0$, then $f$ is identically zero, and
\eq{
d_\alpha(f,g) < \delta \quad &\implies \quad \exists\ \phi : A \to \N \times \Z^d,\quad
\sum_{v \in A} g(\phi(v)) + \sum_{v \notin \phi(A)} g(v)^\alpha < \delta \\
&\implies \quad \max_{v \in \N \times \Z^d} g(v) < \max\{\delta,\delta^{1/\alpha}\}.
}
Otherwise, for any $\delta \in (0,f(u)^\alpha)$, we have
\eq{
d_\alpha(f,g) < \delta \quad &\implies \quad \exists\ \phi : A \to \N \times \Z^d,\quad
d_{\alpha,\phi}(f,g) < \delta < f(u)^\alpha \\
&\implies \quad u \in A \text{ and } |f(v) - g(\phi(v))| < \delta \text{ for all $v \in A$, and }
g(v)^\alpha < \delta \text{ for all $v \notin \phi(A)$}  \\
&\implies \quad \Big|\max_{v \in \N \times \Z^d} g(v) - \max_{v \in \N \times \Z^d} f(v)\Big| < \max\{\delta,{\delta}^{1/\alpha}\}.
}
From these two cases, we conclude that $f \mapsto \max f$ is continuous on $\SS$.
\end{proof}

After the appropriate functionals have been defined, their limiting behavior (in a Ces\`aro sense) can be determined by applying Theorem \ref{characterization}.
Consequently, the next step in our program is to study how the functionals of interest behave according to the elements of $\MM$.
This will depend on the value of $\beta$.
For instance, in the high temperature phase, $\MM$ consists of only the point mass $\delta_\vc{0}$ at the zero function $\vc{0}$, and so trivially we have
\eeq{
0 \leq \beta \leq \beta_{\mathrm{c}} \quad \implies \quad \int \max(f)\ \rho(\dd f) = \max(\vc{0}) = 0 \quad \text{for all $\rho \in \MM$.} \label{max_high}
}
On the other hand, in the low temperature phase, every $\rho \in \MM$ is supported on those $f \in \SS$ with $\|f\| = 1$.
So clearly
\eq{
\beta > \beta_{\mathrm{c}} \quad \implies \quad \int \max(f)\ \rho(\dd f) > 0 \quad \text{for all $\rho \in \MM$.}
}
Furthermore, because of Lemma \ref{portmanteau}, Lemma \ref{max_lemma} shows $\rho \mapsto \int \max(f)\, \rho(\dd f)$ is continuous on the compact set $\MM$, and so we actually have
\eeq{
\beta > \beta_{\mathrm{c}} \quad \implies \quad \int \max(f)\ \rho(\dd f) \geq c \quad \text{for all $\rho \in \MM$,} \label{max_low}
}
for some $c > 0$.
Now we are poised to prove the desired result.
The final step is to interpret the above observations in terms of the directed polymer model, via the almost sure convergence to $\MM$ that is stated in Theorem \ref{close_M}.
It is useful to remember that $\SS$ and $\PP(\SS)$ are compact spaces.

\begin{proof}[Proof of Theorem \ref{reprove_equivalence}]
By the (uniform) continuity of the map 
\eq{
\rho \mapsto \int\max(f)\, \rho(\dd f)
}
on the compact space $\PP(\SS)$, we can find for any $\eps > 0$ some $\delta > 0$ such that
\eq{
\WW_\alpha(\rho,\MM) < \delta \quad &\implies \quad 
\inf_{\rho \in \MM} \int \max(f)\ \rho(\dd f) - \eps 
\leq \int \max(f)\ \rho(\dd f)
\leq \sup_{\rho \in \MM} \int \max(f)\ \rho(\dd f) + \eps. 
}
Thus Theorem \ref{close_M} implies that almost surely,
\eq{
\inf_{\rho \in \MM} \int \max(f)\ \rho(\dd f)
&\leq \liminf_{n \to \infty} \int \max(f)\ \rho_{n}(\dd f)  \\
&\leq \limsup_{n \to \infty} \int \max(f)\ \rho_{n}(\dd f)
\leq \sup_{\rho \in \MM} \int \max(f)\ \rho(\dd f).
}
When $0 \leq \beta \leq \beta_{\mathrm{c}}$, \eqref{max_high} says that the infimum and supremum appearing above are both equal to 0, and so
\eeq{
\lim_{n \to \infty} \int \max(f)\ \rho_{n}(\dd f) = 0 \quad \mathrm{a.s.} \label{max_claim2}
}
When $\beta > \beta_{\mathrm{c}}$, \eqref{max_low} shows that the infimum is bounded below by $c > 0$, in which case we have
\eeq{
\liminf_{n \to \infty} \int \max(f)\ \rho_{n}(\dd f) \geq c \quad \mathrm{a.s.} \label{max_claim1}
}
Recalling \eqref{max_abstract}, we see that \eqref{max_claim2} and \eqref{max_claim1} are exactly what we wanted to show.
\end{proof}

\section{Asymptotic pure atomicity} \label{main_thm}
Following the approach outlined in Section \ref{example_application}, we will prove that directed polymers are asymptotically purely atomic if and only if the parameter $\beta$ falls in the low temperature regime.
We say that the sequence of endpoint probability mass functions $(f_i)_{i \geq 0}$ is \textit{asymptotically purely atomic}\footnote{In \cite{vargas07}, Vargas defines asymptotic pure atomicity by the same convergence condition, but in probability rather than almost surely.  \textit{A priori}, the definition considered here is a stronger one, but since \eqref{noVSD_claim} implies the negation of the ``in probability" definition, Theorem \ref{total_mass} shows the two notions are equivalent given \eqref{walk_assumption_1} and \eqref{mgf_assumption}.}  
if for every sequence $(\eps_i)_{i \geq 0}$ tending to $0$ as $i \to \infty$, we have
\eeq{ \label{apa_def}
\lim_{n \to \infty} \frac{1}{n} \sum_{i = 0}^{n-1} \mu_i^\beta(\sigma_i \in \AA_i^{\eps_i}) = 1 \quad \mathrm{a.s.},
}
where
\eq{
\AA_i^\eps \coloneqq  \{x \in \Z^d : f_i(x) > \eps\}, \quad i \geq 0,\ \eps > 0.
}
In \cite{vargas07}, Vargas defines asymptotic pure atomicity by the same convergence condition, but in probability rather than almost surely.  \textit{A priori}, the definition considered here is a stronger one, but since \eqref{noVSD_claim} implies the negation of the ``in probability" definition, Theorem \ref{total_mass} shows the two definitions are equivalent given \eqref{walk_assumption_1} and \eqref{mgf_assumption}.

We begin by making the necessary definitions in the abstract setting.

\subsection{Definitions of relevant functionals}
Observe that the quantity of interest, $\mu_i^\beta(\sigma_i \in \AA_i^{\eps})$, is a function of the mass function $f_i$.
Specifically,
\eq{
\mu_i^\beta(\sigma_i \in \AA_i^{\eps}) = \sum_{x \in \Z^d\, : f_i(x) > \eps} f_i(x).
}
We are thus motivated to define $\|\cdot\|_\eps : \SS_0 \to [0,1]$ by
\eq{
\|f\|_\eps \coloneqq  \sum_{u\in\N\times\Z^d\, :\, f(u) > \eps} f(u).
}
For any $\eps \in (0,1)$, the map $f \mapsto \|f\|_\eps$ satisfies (i)--(iii) of Corollary \ref{defined_pspm} and thus induces a well-defined function on $\SS$.
To establish asymptotic pure atomicity, it will be important that this function is lower semi-continuous, a fact we prove in the lemma below.
Another useful functional will be $\vc{I}_\eps : \SS_0 \to \{0,1\}$ given by
\eq{
\vc{I}_\eps(f) = \begin{cases}
1 &\text{if } \max_{u \in \N \times \Z^d} f(u) \geq \eps, \\
0 &\text{otherwise}.
\end{cases}
}
Once more, it is clear that $\vc{I}_\eps$ satisfies the hypotheses of Corollary \ref{defined_pspm}, and so it is well-defined on~$\SS$.

\begin{lemma} \label{eps_norm_equivalence}
For any $\eps \in (0,1)$, the following statements hold:
\begin{itemize}
\item[(a)] The map $\|\cdot\|_\eps : \SS \to [0,1]$ is lower semi-continuous and thus measurable.
\item[(b)] The map $\vc{I}_\eps : \SS \to \{0,1\}$ is upper semi-continuous and thus measurable.
\end{itemize}
\end{lemma}

\begin{proof}
The proof of (a) will be similar to that of Lemma \ref{norm_equivalence}.
It suffices to fix $f \in \SS$, let $\delta_2 > 0$ be arbitrary, and find $\delta_1 > 0$ sufficiently small that
\eq{
d_\alpha(f,g) < \delta_1 \quad \implies \quad \|g\|_\eps > \|f\|_\eps - \delta_2.
}
This task is trivial if $\|f\|_\eps = 0$.
Otherwise, we pick a representative $f \in \SS_0$, and consider the nonempty set
$A \coloneqq  \{u \in \N \times \Z^d : f(u) > \eps\}$.
Since $A$ is finite, $f$ achieves its minimum $t>\eps$ over $A$.
Now choose $\delta_1>0$ such that
\eq{
\delta_1 < \min\{\delta_2,\eps^\alpha,t-\eps\}.
}
If $d_\alpha(f,g) < \delta_1$, then there is a representative $g \in \SS_0$ and an isometry $\phi : C \to \N \times \Z^d$ such that $d_{\alpha,\phi}(f,g) < \delta_1$.
It follows that $A \subset C$, since otherwise $d_{\alpha,\phi}(f,g) \geq f(u)^\alpha > \eps^\alpha > \delta_1$ for some $u \in A \setminus C$.
Consequently,
\eq{
\sum_{u \in A} |f(u) - g(\phi(u))| \leq d_{\alpha,\phi}(f,g) < \delta_1 < t - \eps.
}
In particular, for every $u \in A$,
\eq{
g(\phi(u)) \geq f(u) - |f(u) - g(\phi(u))| > t - (t - \eps) = \eps.
}
Hence
\eq{
\|g\|_\eps \geq \sum_{u \in A} g(\phi(u)) &\geq \sum_{u \in A} f(u) - \sum_{u \in A} |f(u) - g(\phi(u))| \\
&\geq \|f\|_\eps - d_{\alpha,\phi}(f,g) > \|f\|_\eps - \delta_1 > \|f\|_\eps - \delta_2,
}
which completes the proof of (a).

For claim (b) we need only to consider the case when $(f_n)_{n \geq 1}$ is a sequence in $\SS$ satisfying $\vc{I}_\eps(f_n) = 1$ for infinitely many $n$, and $d_\alpha(f_n,f) \to 0$ as $n \to \infty$.
We must show $\vc{I}_\eps(f) = 1$.
By passing to a subsequence, we may assume $\vc{I}_\eps(f_n) = 1$ for all $n$.
That is, $\max(f_n) \geq \eps$ for all $n$, and so Lemma \ref{max_lemma} forces $\max(f) \geq \eps$, meaning $\vc{I}_\eps(f) = 1$.
\end{proof}

\subsection{Proof of asymptotic pure atomicity at low temperature} \label{characterize_proof}
First we simplify the problem by providing a sufficient condition for asymptotic pure atomicity.

\begin{lemma} \label{equiv_apa}
If for every $c < 1$, there is $\eps > 0$ such that
\eeq{ \label{equiv_apa_eq}
\liminf_{n \to \infty} \frac{1}{n} \sum_{i = 0}^{n-1} \mu_i^\beta(\sigma_i \in \AA_i^\eps) > c \quad \mathrm{a.s.}, 
}
then $(f_i)_{i \geq 0}$ is asymptotically purely atomic.
\end{lemma}

\begin{proof}
Let $(\eps_i)_{i \geq 0}$ be a sequence tending to 0 as $i \to \infty$.
We will show that under the given hypothesis, for any $k \in \N$ there is almost surely some $N$ such that
\eeq{
n \geq N \quad \implies \quad \frac{1}{n} \sum_{i = 0}^{n-1} \mu_{i}^\beta(\sigma_i \in \AA_i^{\eps_i}) \geq 1 - \frac{1}{k}. \label{apa_lemma_want}
}
This is sufficient to verify asymptotic pure atomicity, since then with probability one,
\eq{
\liminf_{n \to \infty} \frac{1}{n} \sum_{i = 0}^{n-1} \mu_i^\beta(\sigma_i \in \AA_i^{\eps_i}) \geq 1 - \frac{1}{k} \quad \text{for every $k$},
}
which of course implies \eqref{apa_def}.
So set $t \coloneqq  1 - 1/k$, and choose any $c \in (t,1)$.
Then let $\eps > 0$ be such that \eqref{equiv_apa_eq} is satisfied.
In particular, there almost surely exists $N_1$ such that
\eq{
n \geq N_1 \quad \implies \quad \frac{1}{n} \sum_{i = 0}^{n-1} \mu_i^\beta(\sigma_i \in \AA_i^\eps) > c.
}
Next take $N_2$ large enough that
\eq{
i \geq N_2 \quad \implies \quad \eps_i < \eps \quad \implies \quad \mu_i^\beta(\sigma_i \in \AA_i^{\eps_i}) \geq \mu_i^\beta(\sigma_i \in \AA_i^\eps),
}
and then finally choose $N_3 > N_2$ so large that
\eq{
c - \frac{N_2}{N_3} > t.
}
It follows that for all $n \geq \max\{N_1,N_3\}$,
\eq{
\frac{1}{n} \sum_{i = 0}^{n-1} \mu_{i}^\beta(\sigma_i \in \AA_i^{\eps_i})
&\geq \frac{1}{n} \sum_{i = N_2}^{n-1} \mu_{i}^\beta(\sigma_i \in \AA_i^{\eps_i}) \\
&\geq \frac{N_2}{n} + \frac{1}{n} \sum_{i = N_2}^{n-1} \mu_{i}^\beta(\sigma_i \in \AA_i^{\eps_i}) - \frac{N_2}{N_3} \\
&\geq \frac{1}{n}\sum_{i = 0}^{N_2-1} \mu_{i}^\beta(\sigma_{i} \in \AA_i^{\eps}) + \frac{1}{n} \sum_{i = N_2}^{n-1} \mu_{i}^\beta(\sigma_{i} \in \AA_i^{\eps_i}) - \frac{N_2}{N_3} \\
&\geq \frac{1}{n}\sum_{i = 0}^{n-1} \mu_{i}^\beta(\sigma_{i} \in \AA_i^{\eps}) - \frac{N_2}{N_3}
> c - \frac{N_2}{N_3} > t.
}
So $N = \max\{N_1,N_3\}$ suffices for \eqref{apa_lemma_want}.
\end{proof}

We are now ready to prove Theorem \ref{intro_result1}. For the convenience of the reader, we will restate the result here before giving the proof.

\begin{thm} \label{total_mass}
Assume \eqref{walk_assumption_1} and \eqref{mgf_assumption}.
Then the following statements hold:
\begin{itemize}
\item[(a)] If $\beta > \beta_{\mathrm{c}}$, then $(f_i)_{i \geq 0}$ is asymptotically purely atomic.
\item[(b)] If $0 \leq \beta \leq \beta_{\mathrm{c}}$, then there is a sequence $(\eps_i)_{i \geq 0}$ tending to 0 as $i \to \infty$, such that
\eeq{
\lim_{n \to \infty} \frac{1}{n} \sum_{i = 0}^{n-1} \mu_{i}^\beta(\sigma_i \in \AA_i^{\eps_i}) = 0 \quad \mathrm{a.s.}\label{noVSD_claim}
}
\end{itemize}
\end{thm}

While the functionals $\|\cdot\|_\eps$ and $\vc\II_\eps(\cdot)$ are not continuous, their semi-continuity will be sufficient to make a limiting argument, thanks to the following generalization of Dini's first theorem.

\begin{lemma} \label{dini}
Let $(\XX,m)$ be a compact metric space.
If $(F_n)_{n \geq 1}$ is a non-decreasing sequence of lower semi-continuous functions $\XX \to \R$ converging pointwise to an upper semi-continuous function $F$, then the convergence is uniform.
\end{lemma}

\begin{proof}
Let $G_n \coloneqq  F - F_n$ so that $G_n \searrow 0$ pointwise on $\XX$, and $G_n$ is upper semi-continuous.
For given $\eps > 0$, consider the set
\eq{
U_n \coloneqq  \{x \in \XX : G_n(x) < \eps\},
}
which is open because $G_n$ is upper semi-continuous.
Since $G_n(x) \to 0$ for every $x \in \XX$, we have $\bigcup_n U_n = \XX$.
By compactness of $\XX$, there is a finite list $n_1 < n_2 < \dots < n_k$ such that
$\bigcup_{j = 1}^k U_{n_j} = \XX$.
But the monotonicity assumption guarantees $U_n$ is an ascending sequence.
Hence $\XX = \bigcup_{j = 1}^k U_{n_j} = U_{n_k} = U_n$ for all $n \geq n_k$, meaning
\eq{
n \geq n_k \quad &\implies \quad
U_n = \XX \quad\implies \quad
|F(x) - F_n(x)| = F(x) - F_n(x) < \eps \quad \text{for all $x \in \XX$.}
}
That is, $F_n \nearrow F$ uniformly on $\XX$.
\end{proof}

\begin{proof}[Proof of Theorem \ref{total_mass}]
We first prove (a).
For $\eps > 0$, define $\|\cdot\|_\eps$ on $\PP(\SS)$ by
\eq{
\|\rho\|_\eps \coloneqq  \int \|f\|_\eps\ \rho(\dd f)
= \int \sum_{u\, : f(u) > \eps} f(u)\ \rho(\dd f).
}
In light of Lemma \ref{portmanteau}, Lemma \ref{eps_norm_equivalence}(a) implies $\|\cdot\|_\eps : \PP(\SS) \to \R$ is lower semi-continuous.
For any $f \in \SS$ such that $\|f\| = 1$, we have $\|f\|_\eps \nearrow 1$ as $\eps \to 0$.
In particular, under the assumption that $\beta > \beta_{\mathrm{c}}$,
Theorem \ref{characterization}(b) and monotone convergence imply that for any $\rho \in \MM$,
\eq{
\|\rho\|_\eps \nearrow 1 \quad \text{as $\eps \to 0$.} 
}
Since $\MM$ is compact, Lemma \ref{dini} strengthens this pointwise convergence 
to uniform convergence.
That is, for any $c < 1$, there is $\eps > 0$ such that $\|\rho\|_\eps > c$ for all $\rho\in\MM$.
Furthermore, by compactness of $\PP(\SS)$ and lower semi-continuity of $\|\cdot\|_\eps$, we can find $\delta > 0$ such that for any $\rho \in \PP(\SS)$,
\eq{
\WW_\alpha(\rho,\MM) < \delta \quad \implies \quad \|\rho\|_\eps > c. 
}
Now Theorem \ref{close_M} implies
\eq{
\liminf_{n\to\infty} \frac{1}{n} \sum_{i = 0}^{n-1} \mu_i^\beta(\sigma_i \in \AA_i^\eps) 
= \liminf_{n\to\infty} \|\rho_{n}\|_\eps \geq c \quad \mathrm{a.s.}
}
We have thus verified the hypothesis of Lemma \ref{equiv_apa}, and so $(f_i)_{i \geq 0}$ is asymptotically purely atomic.

For (b), we assume $0 \leq \beta \leq \beta_{\mathrm{c}}$.
For $\eps > 0$, define $\vc{\II}_\eps : \PP(\SS) \to \R$ by
\eq{
\vc\II_\eps(\rho) \coloneqq  \int \vc{I}_\eps(f)\ \rho(\dd f)
= \int \one_{\{\max_{u \in \N \times \Z^d} f(u) \geq \eps\}}\ \rho(\dd f).
}
Considering Lemma \ref{portmanteau}, we see from Lemma \ref{eps_norm_equivalence}(b) that $\vc\II_\eps$ is an upper semi-continuous map.
By Theorem \ref{characterization} and Theorem \ref{close_M}, $\WW_\alpha(\rho_n,\delta_\vc{0}) \to 0$ almost surely as $n \to \infty$, and so
\eq{
\limsup_{n \to \infty} \vc\II_\eps(\rho_n) \leq \vc\II_\eps(\delta_\vc{0}) = \vc{I}_\eps(\vc{0}) = 0 \quad \mathrm{a.s.} \quad \text{for any $\eps>0$}.
}
In particular, for any $j \in \N$,
\eq{
\lim_{N \to \infty} \P\bigg(\bigcap_{n = N}^\infty \{\vc\II_\eps(\rho_{n}) < 2^{-j}\}\bigg)=1. 
}
Now let $(\kappa_j)_{j \geq 1}$ be any decreasing sequence tending to $0$ as $j \to \infty$.
Set $M_0 \coloneqq  0$.
By taking complements in the above display, we can inductively choose $M_j > M_{j-1}$ such that
\eeq{ \label{bad_probability}
\P\bigg(\bigcup_{n = M_j}^\infty \{\vc\II_{\kappa_{j+1}}(\rho_{n}) \geq 2^{-(j+1)}\} \bigg) &< 2^{-(j+1)}, \quad j\geq1.
}
Define $\eps_i \coloneqq  \kappa_j$ when $M_{j-1} \leq i \leq M_{j} - 1$.
For $M_{\ell-1} \leq n \leq M_{\ell}$ and any $k<\ell$, the monotonicity of $\kappa_j$ gives
\eq{
\frac{1}{n} \sum_{i = 0}^{n-1} \vc{I}_{\eps_i}(f_i)
&= \frac{1}{n} \bigg[\sum_{i = 0}^{M_k-1} \vc{I}_{\eps_i}(f_i) + \sum_{j = k+1}^{\ell-1} \sum_{i = M_{j-1}}^{M_{j}-1} \vc{I}_{\kappa_j}(f_i) + \sum_{i = M_{\ell-1}}^{n-1} \vc{I}_{\kappa_{\ell}}(f_i)\bigg]  \\
&\leq \frac{1}{n}\bigg[\sum_{i = 0}^{M_k-1} \vc{I}_{\kappa_{k}}(f_i) + \sum_{j = k+1}^{\ell} \sum_{i = 0}^{n-1} \vc{I}_{\kappa_j}(f_i)\bigg]\\
&\leq \sum_{j = k}^{\ell} \frac{1}{n} \sum_{i = 0}^{n-1} \vc{I}_{\kappa_j}(f_i)
= \sum_{j = k}^\ell \vc\II_{\kappa_j}(\rho_{n}),
}
where we have identified $f_i$ with an element of $\SS$ (this identification is measurable by Lemma~\ref{S_meas}, cf.~the discussion following Corollary \ref{increment_cor}).
Writing a more general inequality, we can say that for all $n \geq M_k$,
\eq{
\frac{1}{n} \sum_{i = 0}^{n-1} \vc{I}_{\eps_i}(f_i)
\leq \sum_{j \geq k :\ n \geq M_{j-1}} \vc\II_{\kappa_j}(\rho_{n}).
}
It follows from a union bound and \eqref{bad_probability}  that
\eq{
\P\bigg(\bigcup_{n = M_k}^\infty \bigg\{\frac{1}{n} \sum_{i = 0}^{n-1} \vc{I}_{\eps_i}(f_i) \geq \sum_{j \geq k} 2^{-j}\bigg\}\bigg)
&\leq \P\bigg(\bigcup_{n = M_{k-1}}^\infty \bigcup_{j \geq k\, :\, n \geq M_{j-1}} \{\vc\II_{\kappa_j}(\rho_{n}) \geq 2^{-j}\}\bigg) \\
&= \P\bigg(\bigcup_{j \geq k} \bigcup_{n= M_{j-1}}^\infty \{\vc\II_{\kappa_j}(\rho_{n}) \geq 2^{-j}\}\bigg) \\
&\leq \sum_{j \geq k} \P\bigg(\bigcup_{n = M_{j-1}}^\infty \{\vc\II_{\kappa_j}(\rho_{n}) \geq 2^{-j}\}\bigg) 
< \sum_{j \geq k} 2^{-j} = 2^{-k+1}.
}
By Borel--Cantelli, the following is true with probability one: For only finitely many $k$ is
\eq{
\frac{1}{n} \sum_{i = 0}^{n-1} \vc{I}_{\eps_i}(f_i) \geq 2^{-k+1} \quad \text{for some $n \geq M_k$.}
}
This implies
\eeq{
\lim_{n \to \infty} \frac{1}{n} \sum_{i = 0}^{n-1} \vc{I}_{\eps_i}(f_i) = 0 \quad \mathrm{a.s.} \label{bigger_to_0}
}
Finally, note that $\mu_i^\beta(\sigma_i \in \AA_i^\eps)$ is nonzero only when $\AA_i^\eps$ is nonempty, in which case $\vc{I}_\eps(f_i) = 1$.
Hence
\eq{
\mu_i^\beta(\sigma_i \in \AA_i^\eps) \leq \vc{I}_\eps (f_i) \quad \text{for any $i \geq 0$, $\eps > 0$,}
}
and so \eqref{noVSD_claim} follows from \eqref{bigger_to_0}.
\end{proof}

\begin{cor} \label{apa_either_way}
Assume \eqref{walk_assumption_1} and \eqref{mgf_assumption}. 
Let $g_i(x) \coloneqq  \mu_{i-1}^\beta(\sigma_i =x)$. 
Then the sequence $(g_i)_{i \geq 1}$ is asymptotically purely atomic if and only if $(f_i)_{i \geq 0}$ is asymptotically purely atomic.
\end{cor}

\begin{proof}
Let $\kappa$ be an arbitrary positive number, and then choose $k$ sufficiently large that $P(\|\sigma_1\|_1 > k) < \kappa$.
For each $y\in\Z^d$, define the set
\eq{
\S(y,k) \coloneqq \{x\in\Z^d : \|x-y\|\leq k,\ P(y,x) > 0\}.
}
Set $\delta \coloneqq \min_{x\in\S(y,k)} P(y,x) > 0$ (by translation invariance, this number $\delta$ does not depend on $y$).
Let us write
\eq{
\BB_i^\eps \coloneqq  \{x \in \Z^d : \mu_{i-1}^\beta(\sigma_i = x) > \eps\}, \quad i \geq 1,\ \eps > 0.
}

Now suppose $(f_i)_{i \geq 0}$ is asymptotically purely atomic.
Let $(\eps_i)_{i \geq 0}$ be any sequence tending to 0 as $i \to \infty$.
Then $\eps_i/\delta$ also tends to 0 as $i \to \infty$, and so
\eeq{
\lim_{n \to \infty} \frac{1}{n} \sum_{i = 0}^{n-1} \mu_i^\beta(\sigma_i \in \AA_i^{\eps_i/\delta}) = 1 \quad \mathrm{a.s.} \label{b_apa}
}
The observation
\eeq{
\mu_{i-1}^\beta(\sigma_i = x) = \sum_{y\in\Z^d} \mu_{i-1}^\beta(\sigma_{i-1} = y)P(y,x) \label{next_step_repeat}
}
shows
\eq{
y \in \AA_{i-1}^{\eps/\delta} \quad \implies \quad x \in \BB_i^\eps \text{ for all $x\in\S(y,k)$.}
}
Consequently, for any $\eps > 0$ we have the bound
\eq{
\mu_{i-1}^\beta(\sigma_i \in \BB_i^\eps)
&\geq \sum_{y \in \AA_{i-1}^{ \eps/\delta}} \sum_{x \in\S(y,k)} \mu_{i-1}^\beta(\sigma_{i-1} = y)P(y,x)\\
&\geq (1-\kappa)\sum_{y \in \AA_{i-1}^{\eps/\delta}} \mu_{i-1}^\beta(\sigma_{i-1} = y) = (1-\kappa)\mu_{i-1}^\beta(\sigma_{i-1} \in \AA_{i-1}^{\eps/\delta}).
}
It now immediately follows from \eqref{b_apa} that
\eq{
\liminf_{n \to \infty} \frac{1}{n} \sum_{i = 1}^{n} \mu_{i-1}^\beta(\sigma_i \in \BB_i^{\eps_i}) \geq 1-\kappa \quad \mathrm{a.s.}
}
Taking $\kappa\to0$, we conclude that $(g_i)_{i \geq 1}$ is asymptotically purely atomic.

Conversely, assume $(f_i)_{i \geq 0}$ is not asymptotically purely atomic.
By Theorem \ref{total_mass}, we must have $0 \leq \beta \leq \beta_{\mathrm{c}}$, and so \eqref{bigger_to_0} holds for some sequence $(\eps_i)_{i \geq 0}$ tending to 0 as $i \to \infty$.
From \eqref{next_step_repeat}, we see that $\BB_i^\eps$ is nonempty only when $\AA_{i-1}^\eps$ is nonempty, in which case $\vc{I}_\eps(f_{i-1}) = 1$.
As a result,
\eq{
\mu_{i-1}^\beta(\sigma_i \in \BB_i^\eps) \leq \vc{I}_{\eps}(f_{i-1}) \quad \text{for any $i \geq 1$, $\eps > 0$,}
}
and so \eqref{bigger_to_0} forces
\eq{
\lim_{n \to \infty} \frac{1}{n} \sum_{i = 1}^{n} \mu_{i-1}^\beta(\sigma_i \in \BB_i^{\eps_{i-1}}) = 0 \quad \mathrm{a.s.}
}
Indeed, $(g_i)_{i \geq 1}$ is not asymptotically purely atomic.
\end{proof}

\section{Geometric localization} \label{main_thm2}
Recall the following definition from Section \ref{endpoint_results}.
We say that the sequence $(f_i)_{i \geq 0}$ exhibits \textit{geometric localization with positive density} if for every $\delta > 0$, there is $K < \infty$ and $\theta  > 0$ such that
\eq{
\liminf_{n \to \infty} \frac{1}{n} \sum_{i = 0}^{n-1} \one_{\{f_i \in \GG_{\delta,K}\}} \geq \theta \quad \mathrm{a.s.},
}
where
\eeq{ \label{g_set_def}
&\GG_{\delta,K} \coloneqq \Big\{f : \Z^d \to [0,1]\ \Big|\ \|f\| = 1, \exists\, D\subset \Z^d \text{ such that }\diam(D)\leq K\text{ and }\sum_{x\in D}f(x) > 1-\delta\Big\},
} 
and
\eq{
\diam(D) \coloneqq  \sup\{\|x - y\|_1 : x,y \in D\}.
}
If $K$ can always be chosen so that $\theta\geq1-\delta$, then the sequence is ``geometrically localized with \textit{full} density."
The main goal of this section is to establish that there is positive density geometric localization if and only if $\beta > \beta_{\mathrm{c}}$.

\subsection{Definitions of relevant functionals}
To employ the theory of partitioned subprobability measures, we make a few definitions.
For $f \in \SS_0$, let $H_f$ denote $\N$-support of $f$:
\eq{
H_f \coloneqq  \{n \in \N : f(n,x) > 0 \text{ for some $x \in \Z^d$}\}.
}
The first functional of interest is the size of the $\N$-support of $f$,
\eq{
N(f) \coloneqq  |H_f|,
}
which we call the \textit{support number} of $f$.
Second, for $f\in \SS$ and $\delta \in (0,1)$, it is natural in studying geometric localization to consider the quantity
\eq{
W_\delta(f) &\coloneqq  \inf\Big\{\diam(D) : D \subset \Z^d,\  \sum_{x \in D} f(n,x) > 1 - \delta \text{ for some $n \in \N$}\Big\}.
}
In words, $W_\delta(f)$ is the smallest $K$ such that a region of diameter $K$ in $\N \times \Z^d$ has $f$-mass strictly greater than $1-\delta$.
Here we follow the convention that the infimum of an empty set is infinity, meaning $W_\delta(f) = \infty$ whenever there is no one copy of $\Z^d$ on which $f$ has mass greater than $1-\delta$.
We will write
\eq{
\VV_{\delta,K} \coloneqq  \{f \in \SS : W_\delta(f) \leq K\},
}
so that we can naturally identify $\GG_{\delta,K}$ in $\SS$ by
\eeq{
\GG_{\delta,K} = \VV_{\delta,K} \cap \{f \in \SS : N(f) = 1,\, \|f\| = 1\}. \label{g_intersection}
}
Next define $q_n : \SS_0 \to [0,1]$ by
\eq{
q_n(f) \coloneqq  \sum_{x\in \Z^d} f(n,x),
}
and let
\eq{
m(f) \coloneqq  \max_{n\in \N} q_n(f).
}
Finally, it will also be useful to analyze the function
\eq{
Q(f) \coloneqq  \sum_{n \in \N} \frac{q_n(f)}{1 - q_n(f)},
}
where $1/0 = \infty$.
Observe that $Q(f) = \infty$ if and only if $N(f) = 1$ and $\|f\| = 1$.

It is easy to see from Corollary \ref{defined_pspm} that $N, W_\delta$, $m$, and $Q$ are well-defined on $\SS$.
The fact that $N : \SS \to \N \cup \{0,\infty\}$ is measurable is surprisingly non-trivial.
We must first prove it for $N$ viewed as a function on $\SS_0$.

\begin{lemma} \label{N_meas1}
$N(\cdot) : \SS_0 \to \N \cup \{0,\infty\}$ is measurable.
\end{lemma}

\begin{proof}
Fix $N \in \{0\} \cup \N$.
Let $U_N \coloneqq  \{f \in \SS_0 : N(f) = N\}$ so that
\eq{
U_N &= \bigcup_{A \subset \N\, :\, |A| = N} \bigg[\bigg(\bigcap_{n \in A} \bigcup_{x \in \Z^d} \{f \in \SS_0 : f(n,x) > 0\} \bigg) \cap \bigg(\bigcap_{n \in \N \setminus A} \bigcap_{x \in \Z^d} \{f \in \SS_0 : f(n,x) = 0\}\bigg)\bigg].
}
For any $(n,x) \in \N \times \Z^d$, the map $f \mapsto f(n,x)$ is continuous and thus measurable.
Furthermore, all intersections and unions in the above display are taken over countable sets, and so $U_N$ is measurable.
It follows that $U_\infty \coloneqq  \{f \in \SS_0 : N(f) = \infty\} = \SS_0 \setminus \big(\bigcup_{N \geq 0} U_N\big)$ is also measurable.
We conclude that $N(\cdot) : \SS_0 \to \N \cup \{0,\infty\}$ is a measurable function.
\end{proof}

By Corollary \ref{defined_pspm}, $N(\cdot)$ is well-defined on $\SS$.
That is, if we define for $f \in \SS_0$ the set
\eq{
A_f \coloneqq \{g \in \SS_0 : d_\alpha(f,g) = 0\},
}
then $N(g) = N(f)$ for all $g \in A_f$.
It remains to show, however, that $N(\cdot) : \SS \to \N \cup \{0,\infty\}$ is a measurable map.

\begin{lemma} \label{at_most_countable}
If $N(f) < \infty$, then the cardinality of $A_f$ is at most countably infinite.
\end{lemma}

\begin{proof}
Fix $f \in \SS_0$ such that $N(f) < \infty$.
By Corollary \ref{better_def_cor}, for any $g \in A_f$ there is $\tau : H_f \to H_g$ and a set of vectors $(x_n)_{n \in H_f}$ such that \eqref{better_def} holds.
Notice that because $|H_f| < \infty$, there are only countably many choices for the set $\tau(H_f) = H_g$, and only countably many possibilities for the $x_n$.
As $g$ is determined from $f$ by $H_g$ and $(x_n)_{n \in H_f}$, we conclude that $A_f$ is at most countably infinite.
\end{proof}

We will use the following fact about countable-to-one Borel maps.
These are (Borel) measurable functions $f : \XX \to \YY$ between two topological spaces such that for every $y\in\YY$, $f^{-1}(\{y\})$ is at most countably infinite.

\begin{lemma}[{see Srivastava \cite[Theorem 4.12.4]{srivastava98}}] \label{countable_to_one}
Suppose $\XX,\YY$ are Polish spaces and $F : \XX \to \YY$ is a countable-to-one Borel map. 
Then $F(B)$ is Borel for every Borel set $B$ in $\XX$. 
\end{lemma}

Recall the quotient map $\iota : \SS_0 \to \SS$ defined in Lemma \ref{S_meas}(b), sending $f \in \SS_0$ to its equivalence class in $\SS$.

\begin{prop} \label{N_meas2}
The map $N(\cdot) : \SS \to \N \cup \{0,\infty\}$ is measurable.
\end{prop}

\begin{proof}
It suffices to show that for every $N \in \N \cup \{0,\infty\}$, the set
\eq{
\UU_N \coloneqq  \{f \in \SS : N(f) = N\}
}
is measurable.
We first restrict to the finite case.
For any nonnegative integer $N$, observe that $\XX \coloneqq  \bigcup_{i = 0}^N U_i$ is a Polish space (under the $\ell^1$ metric) and, by Lemma \ref{N_meas1}, a measurable subset of $\SS_0$.
Lemma \ref{at_most_countable} says that $\iota |_{\XX} : \XX \to \SS$ is a countable-to-one map.
Furthermore, this map is measurable by Lemma \ref{S_meas}(b).
Now Lemma \ref{countable_to_one} shows $\iota(U_N) = \UU_N$ is a measurable subset of $\SS$.

Given that $\UU_N$ is measurable for every finite $N$, the set $\UU_\infty = \SS \setminus \big(\bigcup_{N \geq 0} \UU_N\big)$ is also measurable.
We have thus shown $N(\cdot) : \SS \to \N \cup \{0,\infty\}$ is measurable.
\end{proof}

A more central part of our methods is the semi-continuity of relevant functionals, which is 
the content of
the next lemma.

\begin{lemma} \label{more_equivalences}
The following statements hold:
\begin{itemize}
\item[(a)] For any $\delta \in (0,1)$, $W_\delta: \SS \to \N \cup \{0,\infty\}$ is upper semi-continuous and thus measurable.
\item[(b)] $m : \SS \to [0,1]$ is lower semi-continuous and thus measurable.
\item[(c)] $Q : \SS \to [0,\infty]$ is lower semi-continuous and thus measurable.
\end{itemize}
\end{lemma}

These continuity properties can be understood intuitively.
For instance, (a) is simply a consequence of the fact that if a measure on $\Z^d$ places mass greater than $1-\delta$ on a compact set $K$, then any sufficiently similar measure (in the weak or vague topologies) must do the same.
On the other hand, (b) and (c) come from the observation that a converging sequence of partitioned subprobability measures can divide mass on one copy of $\Z^d$ between several copies in the limit when large parts of the mass are drifting away from each other.
But the reverse is not true, because vertices in distinct copies of $\Z^d$ are considered infinitely far apart.
The proofs that make these ideas precise are 
similar in spirit the proof of Lemma \ref{norm_equivalence} but substantially more involved.

\begin{proof}[Proof of Lemma \ref{more_equivalences}]
For (a), we wish to show that for any $f \in \SS$, there is $\eps > 0$ such that
\eq{
d_\alpha(f,g) < \eps \quad \implies \quad W_\delta(g) \leq W_\delta(f).
}
It is only necessary to consider the case when $K \coloneqq  W_\delta(f)$ is finite (in particular, $f$ is nonzero).
Then $m(f) > 1 - \delta$, and we can select a representative $f \in \SS_0$ and $D \subset \Z^d$ such that
\eq{
\sum_{x \in D} f(1,x) > 1 - \delta, \qquad \diam(D) \leq K.
}
By possibly omitting some elements of $D$, we may assume $f$ is strictly positive on $\{1\} \times D$.
Choose $\eps > 0$ to satisfy the following three inequalities:
\begin{subequations}
\begin{align}
\eps &< \inf_{x \in D} f(1,x)^\alpha, \label{eps_choice_1} \\
\eps &< 2^{-K}, \label{eps_choice_2} \\
\sum_{x \in D} f(1,x) &> 1 - \delta + \eps. \label{eps_choice_3}
\end{align}
\end{subequations}
Suppose $g \in \SS$ has $d_\alpha(f,g) < \eps$.
Then there is some representative $g \in \SS_0$ and some isometry $\phi : A \to \N \times \Z^d$ so that $d_{\alpha,\phi}(f,g) < \eps$.
It follows from \eqref{eps_choice_1} that $A \supset \{1\} \times D$.
Furthermore, \eqref{eps_choice_2} guarantees $\deg(\phi) > K$, implying
\eeq{
x,y \in D \quad \implies \quad \phi(1,x) - \phi(1,y) = x - y. \label{perfect_on_D}
}
We then have
\eq{
\sum_{x \in D} g(\phi(1,x)) 
&\stackrel{\phantom{\mbox{\footnotesize\eqref{eps_choice_3}}}}{\geq} \sum_{x \in D} f(1,x) - \sum_{x \in D} |f(1,x) - g(\phi(1,x))| \\
&\stackrel{\mbox{\footnotesize\eqref{eps_choice_3}}}{>} 1 - \delta + \eps - \sum_{u \in A} |f(u) - g(\phi(u))| \\
&\stackrel{\phantom{\mbox{\footnotesize\eqref{eps_choice_3}}}}{\geq} 1 - \delta + \eps - d_{\alpha,\phi}(f,g) > 1 - \delta.
}
Because of \eqref{perfect_on_D}, the set $\{\phi(1,x) : x \in D\}$ has diameter equal to $\diam(D) \leq K$, and so the above inequality shows $W_\delta(g) \leq K = W_\delta(f)$.

Next we prove (b) and (c) together, for which
we consider two cases.
First, if $m(f) = 1$, then $Q(f) = \infty$, and we must show that for any $\eps > 0$ and any $L > 0$, there is $\delta > 0$ such that
\eq{
d_\alpha(f,g) < \delta \quad \implies \quad m(g) > 1 - \eps, \quad Q(g) > L.
}
In this case, it suffices to prove $\delta$ may be chosen so that $m(g) > 1 - \eps$, since any $\eps \leq 1/(L+1)$ gives
\eq{
m(g) > 1 - \eps \quad \implies \quad Q(g) > \frac{1-\eps}{1-(1-\eps)} \geq \frac{L/(L+1)}{1/(L+1)}= L.
}
The first step in doing so is to choose $A \subset \N \times \Z^d$ finite but sufficiently large that
\eeq{
\sum_{u \notin A} f(u) < \frac{\eps}{2}. \label{really_small_sum}
}
On the other hand, if $m(f) < 1$, so that $Q(f) < \infty$, then we shall find $\delta > 0$ satisfying
\eq{
d_\alpha(f,g) < \delta \quad \implies \quad m(g) > m(f) - \eps, \quad Q(g) > Q(f) - \eps.
}
This is trivial if $f = \vc{0}$.
Otherwise, fix a representative $f \in \SS_0$, and again choose a finite $A \subset \N \times \Z^d$, but now satisfying a slightly different condition:
\eeq{
\frac{\sum_{u \notin A} f(u)}{(1-m(f))^2} < \frac{\eps}{2}. \label{really_small_sum2}
}
In either case---\eqref{really_small_sum} or \eqref{really_small_sum2}---we may assume $f$ is strictly positive on $A$ by possibly omitting some elements.
Next consider the integer
\eq{
K \coloneqq  \sup\{\|x-y\|_1 : (n,x),(n,y) \in A \text{ for some $n \in \N$}\},
}
and take $\delta > 0$ satisfying the following conditions:
\begin{subequations}
\begin{align}
\delta &< \inf_{u \in A} f(u)^\alpha, \label{delta_condition1} \\
\delta &< 2^{-K}, \label{delta_condition2} \\
\delta &< \frac{\eps}{2}. \label{delta_condition3}
\intertext{If $m(f) < 1$, we will further assume}
0 < \hspace{0.55in}&\hspace{-0.55in}\frac{\delta}{(1-m(f))(1-m(f)-\delta)} < \frac{\eps}{2}.
\label{delta_condition4}
\end{align}
\end{subequations}
Now, if $d_\alpha(f,g) < \delta$, then there is a representative $g \in \SS_0$ and an isometry $\phi : C \to \N \times \Z^d$ such that $d_{\alpha,\phi}(f,g) < \delta$.
By \eqref{delta_condition1}, we must have $A \subset C$.
Furthermore, \eqref{delta_condition2} ensures $\deg(\phi) > K$, meaning the following implication holds:
\eq{
(n,x),(n,y) \in A \quad &\implies \quad \|(n,x)-(n,y)\|_1 \leq K \quad \implies \quad
\|\phi(n,x) - \phi(n,y)\|_1 = \|x-y\|_1 < \infty.
}
Consequently, for any $n$ such that $A \cap (\{n\} \times \Z^d)$ is nonempty, we can define $\tau(n)$ to be the unique integer such that 
\eq{
\phi(n,x) \in \{\tau(n)\} \times \Z^d \quad \text{whenever $(n,x) \in A$}.
}
Note that $n \mapsto \tau(n)$ may not be injective.
We consider the quantities
\eq{
r_n \coloneqq  \sum_{x\, :\, (n,x) \in A} f(n,x), \qquad 
u_n &\coloneqq  \sum_{x\, :\, (n,x) \in A} g(\phi(n,x)), \qquad
U_n \coloneqq  \sum_{m\, :\, \tau(m) = \tau(n)} u_m,
}
the first two of which satisfy
\eeq{
\sum_{n \in \N}|r_n-u_n|
&= \sum_{n \in \N}\bigg|\sum_{x\, :\, (n,x) \in A} f(n,x) - g(\phi(n,x))\bigg|
\leq \sum_{u \in A} |f(u) - g(\phi(u))| < \delta. \label{first_two}
}
Also notice that \eqref{really_small_sum} or \eqref{really_small_sum2} implies
\eq{
m(f) - \max_{n \in \N} r_n < \frac{\eps}{2},
}
and so
\eq{
m(g) \geq \max_{n \in \N} U_n \geq \max_{n \in \N} u_n 
\stackrel{\mbox{\footnotesize\eqref{first_two}}}{\geq}  \max_{n \in \N} r_n - \delta
> m(f) - \frac{\eps}{2} - \delta
\stackrel{\mbox{\footnotesize\eqref{delta_condition3}}}{>} m(f) - \eps.
}
This completes the proof in the case $m(f) = 1$.
Otherwise, we use the inequality $u_n < r_n + \delta \leq m(f) + \delta$ from \eqref{first_two} to conclude
\eq{ 
\sum_{n \in \N} \frac{r_n}{1-r_n} - \sum_{n \in \N} \frac{u_n}{1-u_n}
= \sum_{n \in \N} \frac{r_n - u_n}{(1-r_n)(1-u_n)}
&\stackrel{\phantom{\mbox{\footnotesize\eqref{first_two}}}}{<} \frac{1}{(1-m(f))(1-m(f)-\delta)}\sum_{n \in \N}|r_n-u_n| \\
&\stackrel{\mbox{\footnotesize\eqref{first_two}}}{<} \frac{\delta}{(1-m(f))(1-m(f)-\delta)}
\stackrel{\mbox{\footnotesize\eqref{delta_condition4}}}{<} \frac{\eps}{2}.
}
Also observe that
\eq{ 
\frac{U_n}{1-U_n} = \frac{\sum_{m\, :\, \tau(m) = \tau(n)} u_m}{1-\sum_{m\, :\, \tau(m) = \tau(n)} u_m}
\geq \sum_{m\, :\, \tau(m) = \tau(n)} \frac{u_m}{1-u_m}.
}
Together, the previous two displays show
\eeq{ \label{t_to_r}
\sum_{n \in \N} \frac{q_n(g)}{1-q_n(g)} \geq \sum_{n \in \N} \frac{U_n}{1-U_n} \geq \sum_{n \in \N} \frac{u_n}{1-u_n} > \sum_{n \in \N} \frac{r_n}{1-r_n} - \frac{\eps}{2}.
}
We also have
\eeq{ \label{s_to_r}
\sum_{n \in \N} \frac{q_n(f)}{1-q_n(f)} - \sum_{n \in \N} \frac{r_n}{1-r_n} 
= \sum_{n \in \N} \frac{q_n(f) - r_n}{(1 - q_n(f))(1-r_n)}
&\leq \frac{1}{(1-m(f))^2} \sum_{n \in \N} (q_n(f) - r_n )\\
&= \frac{\sum_{u \notin A} f(u)}{(1-m(f))^2}
\stackrel{\mbox{\footnotesize\eqref{really_small_sum2}}}{<} \frac{\eps}{2}.
}
Applying \eqref{s_to_r} in the rightmost expression of \eqref{t_to_r}, we obtain the desired result:
\eq{
Q(g) = \sum_{n \in \N} \frac{q_n(g)}{1-q_n(g)} > \sum_{n \in \N} \frac{q_n(f)}{1-q_n(f)} - \eps = Q(f) - \eps.
}
\end{proof}

An observation that will be useful is that at low temperature, the functional $Q$ must have infinite expectation according to any element of $\MM$.

\begin{lemma} \label{Q_lemma}
Assume $\beta > \beta_{\mathrm{c}}$.
Then for any $\rho \in \MM$,
\eq{
\int Q(f)\ \rho(\dd f) = \infty.
}
\end{lemma}

\begin{proof}
Consider any $\rho \in \MM$.
By Theorem \ref{characterization}, the assumption of $\beta > \beta_{\mathrm{c}}$ implies 
$\rho(\{f \in \SS: \|f\| = 1\}) = 1$.
Suppose toward a contradiction that $\rho \in \MM$ satisfies
\eeq{
\int Q(f)\ \rho(\dd f) < \infty. \label{finite_Q_integral}
}
It must then be the case that $\rho(\{f \in \SS: N(f) = 1\}) = 0$, since $Q(f) = \infty$ whenever $N(f) = 1$ and $\|f\| = 1$.
Let $f \in \SS$ be random with law $\rho$.
By the previous observation, $1-q_n(f) > 0$ for all $n \in \N$ with $\rho$-probability 1.
Consider an environment $(\omega_u)_{u \in \N \times \Z^d}$ of i.i.d.~random variables with law $\mathfrak{L}_\omega$, that is independent of $f$.
When viewed as an element of $\SS$, the law of the function
\eq{
F(u) \coloneqq  \frac{\sum_{v \sim u} f(v)\e^{\beta \omega_u}P(v,u)}{\sum_{w \in \N \times \Z^d} \sum_{v \sim w} f(v)\e^{\beta \omega_w}P(v,w)}, \quad u \in \N \times \Z^d,
}
is $\TT\rho = \rho$.
Observe that because $\|F\| = 1$,
\eq{
Q(F) = \sum_{n \in \N} \frac{q_n(F)}{1-q_n(F)}
&= \sum_{n \in \N} \frac{\sum_{x \in \Z^d} F(n,x)}{\sum_{k \neq n} \sum_{x \in \Z^d} F(k,x)} \\
&= \sum_{n \in \N} \frac{\sum_{x \in \Z^d} \sum_{y\in\Z^d} f(n,y)\e^{\beta \omega_{n, x}}P(y,x)}{\sum_{k \neq n} \sum_{x \in \Z^d} \sum_{y \in\Z^d} f(k,y)\e^{\beta \omega_{k, x}}P(y,x)}.
}
Conditioned on $f$, the numerator and denominator of
\eq{
\frac
{\sum_{x \in \Z^d} \sum_{y\in\Z^d} f(n,y)\e^{\beta \omega_{n, x}}P(y,x)}
{\sum_{k \neq n} \sum_{x \in \Z^d} \sum_{y \in\Z^d} f(k,y)\e^{\beta \omega_{k, x}}P(y,x)}
}
are independent.
Hence
\eq{
&\E\givenk{Q(F)}{f} = \sum_{n \in \N} \E\givenk[\bigg]{\frac{\sum_{x \in \Z^d} \sum_{y\in\Z^d} f(n,y)\e^{\beta \omega_{n, x}}P(y,x)}{\sum_{k \neq n} \sum_{x \in \Z^d} \sum_{y \in\Z^d} f(k,y)\e^{\beta \omega_{k, x}}P(y,x)}}{f} \\
&= \sum_{n \in \N} \E\givenk[\bigg]{{\sum_{x \in \Z^d} \sum_{y\in\Z^d} f(n,y)\e^{\beta \omega_{n, x}}P(y,x)}}{f}
\E\givenk[\bigg]{\frac{1}{{\sum_{k \neq n} \sum_{x \in \Z^d} \sum_{y \in\Z^d} f(k,y)\e^{\beta \omega_{k, x}}P(y,x)}}}{f} \\
&> \sum_{n \in \N} \frac{\E\givenk[\big]{{\sum_{x \in \Z^d} \sum_{y\in\Z^d} f(n,y)\e^{\beta \omega_{n, x}}P(y,x)}}{f}}{\E\givenk[\big]{{\sum_{k \neq n} \sum_{x \in \Z^d} \sum_{y \in\Z^d} f(k,y)\e^{\beta \omega_{k, x}}P(y,x)}}{f}}
\\
&= \sum_{n \in \N} \frac{ \e^{\lambda(\beta)} q_n(f)}{\e^{\lambda(\beta)}(1-q_n(f))}
= Q(f),
}
where the strict inequality is due to the strict convexity of $t \mapsto 1/t$ on $(0,\infty)$ and the non-degeneracy of $\mathfrak{L}_\omega$.
But now \eqref{finite_Q_integral} allows us to write
\eq{
\int Q(f)\ \rho(\dd f) &= \iint Q(F)\ Tf(\dd F)\, \rho(\dd f)
= \int \E\givenk{Q(F)}{f}\, \rho(\dd f)
> \int Q(f)\, \rho(\dd f),
}
yielding the desired contradiction.
\end{proof}

\subsection{Proof of geometric localization with positive density}
We can now prove the main theorem of Section \ref{main_thm2}. 
This theorem, which encompasses Theorem \ref{intro_result2}, shows that positive density geometric localization is equivalent to $\beta > \beta_{\mathrm{c}}$. 
It also proves that the single-copy condition (stated as \eqref{single_copy_assumption} below) implies full density geometric localization.

\begin{thm} \label{localized_subsequence}
Assume \eqref{walk_assumption_1} and \eqref{mgf_assumption}.
Let $f_i(\cdot) \coloneqq  \mu_i^\beta(\sigma_i = \cdot)$.
Then the following statements hold:
\begin{itemize}
\item[(a)] If $\beta > \beta_{\mathrm{c}}$, then $(f_i)_{i \geq 0}$ is geometrically localized with positive density. Moreover, the quantities $K$ and $\theta$ in the definition of positive density geometric localization are deterministic, and depend only on the choice of $\delta$, as well as $\mathfrak{L}_\omega$, $\beta$, and $d$. 
\item[(b)] If $\beta > \beta_{\mathrm{c}}$ and 
\eeq{
\rho(\UU_1) = \rho(\{f \in \SS: N(f) = 1\}) = 1 \quad \text{for all $\rho \in \MM$}, \label{single_copy_assumption}
}
then $(f_i)_{i \geq 0}$ is geometrically localized with full density.
\item[(c)] If $0 \leq \beta \leq \beta_{\mathrm{c}}$, then for any $\delta \in (0,1)$ and any $K$,
\eeq{
\lim_{n \to \infty} \frac{1}{n} \sum_{i = 0}^{n-1} \one_{\{f_i \in \GG_{\delta,K}\}} = 0 \quad \mathrm{a.s.} \label{no_localization}
}
\end{itemize}
\end{thm}

\begin{proof}
Fix $\delta\in (0,1)$ throughout. 
Recall that
\eq{
\VV_{\delta,K} = \{f \in \SS : W_\delta(f) \leq K\} = \{f \in \SS : W_\delta(f) < K+1\},
}
which is open by the upper semi-continuity of $W_\delta$.
For (a) and (b), we assume $\beta > \beta_{\mathrm{c}}$.
Given $\delta > 0$, define the set
\eq{
\UU_\delta \coloneqq  \{f\in \SS: m(f) > 1-\delta\} = \bigcup_{K = 0}^\infty \VV_{\delta,K}.
}
For (b) only, note the following consequence of Theorem \ref{characterization}(b):
\eq{
 \eqref{single_copy_assumption} \quad \implies \quad \rho(\UU_\delta) = 1 \quad \text{for all $\rho \in \MM$.} 
} 
Otherwise, we refer to Lemma \ref{Q_lemma} which tells us that for any $\rho\in  \MM$,
\eq{
\int Q(f)\ \rho(\dd f) =\infty.
}
It follows that $\rho(\UU_\delta) > 0$,
since otherwise the $\rho$-essential supremum of $m(f)$ would be strictly less than 1, forcing the $\rho$-essential supremum of $Q(f)$ to be finite.
By Lemma \ref{more_equivalences}(b), $\UU_\delta$ is an open set. Consequently, the map $\rho \mapsto \rho(\UU_\delta)$ is lower semi-continuous on $\PP(\SS)$. 
Since a lower semi-continuous function on a compact set attains its minimum value, we must have
\eq{
\Theta\coloneqq  \inf_{\rho\in\MM} \rho(\UU_\delta) > 0.
}
For any $f \in \UU_\delta$, the quantity $W_\delta(f)$ is finite.
Therefore, for any $\rho \in \MM$,
\eq{
\rho(\VV_{\delta,K}) \wedge \Theta \nearrow \rho(\UU_\delta) \wedge \Theta = \Theta \quad \text{as $K \nearrow \infty$.}
}
Since $\rho \mapsto \rho(\VV_{\delta,K}) \wedge \Theta$ is lower semi-continuous, Lemma \ref{dini} upgrades this convergence to be uniform on the compact set $\MM$.
That is, for any $\theta<\Theta$, we can find $K$ large enough that $\rho(\VV_{\delta,K}) > \theta$ for every $\rho \in \MM$.
Furthermore, by compactness of $\PP(\SS)$ there exists $\xi > 0$ such that
\eq{
\WW_\alpha(\rho,\MM) < \xi \quad \implies \quad \rho(\VV_{\delta,K}) > \theta.
}
We may now conclude from Theorem \ref{close_M} and \eqref{g_intersection} that
\eq{
\theta \leq \liminf_{n\to\infty} \rho_{n}(\VV_{\delta,K}) 
&= \liminf_{n\to\infty} \frac{1}{n}\sum_{i=0}^{n-1} \one_{\{f_i \in \VV_{\delta,K}\}} 
= \liminf_{n\to\infty} \frac{1}{n}\sum_{i=0}^{n-1} \one_{\{f_i \in \GG_{\delta,K}\}} \quad \mathrm{a.s.},
}
which completes the proof of (a) and (b).

For claim (c), suppose $0 \leq \beta \leq \beta_{\mathrm{c}}$ so that Theorem \ref{characterization} gives $\MM = \{\delta_\vc{0}\}$.
Fix $K > 0$.
Recall from the proof of Theorem \ref{reprove_equivalence} that in this high temperature case,
\eeq{
\lim_{n \to \infty} \int \max(f)\ \rho_{n}(\dd f) = \lim_{n \to \infty} \frac{1}{n} \sum_{i = 0}^{n-1} \max_{x \in \Z^d} f_i(x) = 0 \quad \mathrm{a.s.} \label{max0}
}
Notice that if $\eps > 0$ is sufficiently small that
\eeq{
D \subset \Z^d,\, \diam(D) \leq K \quad \implies \quad \eps|D| < 1 - \delta, \label{diam_constraint}
}
then
\eeq{
\frac{1}{n} \sum_{i = 0}^{n-1} \max_{x \in \Z^d} f_i(x) < \eps^2 \quad \implies \quad
\frac{1}{n} \sum_{i = 0}^{n-1} \one_{\{f_i \in \GG_{\delta,K}\}} < \eps. \label{max_diam}
}
Indeed, the hypothesis in \eqref{max_diam} implies there are fewer than $\eps n$ numbers $i$ between 0 and $n-1$ such that $\max_{x \in \Z^d} f_i(x) \geq \eps$, and \eqref{diam_constraint} implies all the remaining $i$ must satisfy $f_i \notin \GG_{\delta,K}$.
Therefore, \eqref{max_diam} is true, and so \eqref{max0} implies \eqref{no_localization}.
\end{proof}

\subsection{Localization in a favorite region}
In this final section we prove that single-copy condition~\eqref{single_copy_assumption} implies localization in a ``favorite region" of size $O(1)$.
Recall that the mode of a probability mass function is a location where the function attains its maximum. For any $n\ge 0$ and $K\ge 0$, let $\CC_{n}^K$ be the set of all $x\in \Z^d$ that are within $\ell^1$ distance $K$ from {\it every} mode of the random probability mass function $f_n$. Note that $\CC_n^K$ is a set with diameter at most $2K$. The following theorem establishes localization of the endpoint in $\CC_n^K$, if the single-copy condition holds.

\begin{prop}\label{localization_thm}
Assume \eqref{walk_assumption_1}, \eqref{mgf_assumption}, and \eqref{single_copy_assumption}.
Then
\eeq{
\lim_{K \to \infty} \liminf_{n \to \infty} \frac{1}{n} \sum_{i = 0}^{n-1} \mu_i^\beta(\sigma_i \in \CC_{i}^K) = 1\quad \mathrm{a.s.} \label{favorite_region}
}
\end{prop}

\begin{proof}
For the sake of completeness, we first verify that $\mu_i^\beta(\sigma_i \in \CC_i^K)$ is a measurable function (with respect to $\FF_i$).
For any Borel set $B \subset \R$, the event $\{\mu_i^\beta(\sigma_i \in \CC_i^K) \in B\}$ can be expressed as
\eq{
\bigcup_{\substack{A \subset \Z^d \\ |A| \leq (2d)^i}} \bigg[\bigg(\bigcap_{x \in A} \{f_i(x) = \max_{y\in\Z^d} f_i(y)\}\bigg) \cap
\bigg\{\sum_{\substack{y \in \Z^d\, :\, \|x - y\|_1 \leq K\, \forall\, x \in A }} f_i(y) \in B\bigg\}\bigg].
}
It is clear that the above display is a measurable event, and so $\mu_i^\beta(\sigma_i \in \CC_i^K)$ is measurable.

Assume \eqref{single_copy_assumption}, so that $(f_i)_{i \geq 0}$ is geometrically localized with full density.
(In particular, by Theorem \ref{localized_subsequence}(c), we must have $\beta > \beta_{\mathrm{c}}$.)
As in the proof of Theorem \ref{total_mass}(a), the assumption $\beta > \beta_{\mathrm{c}}$ ensures that the conclusion of Lemma \ref{equiv_apa} holds.
Consequently, given any $\theta<1$, we can choose $\eps \in (0,1-\theta)$ such that
\eq{
\liminf_{n \to \infty} \frac{1}{n} \sum_{i = 0}^{n-1} \one_{\{\AA_i^\eps\text{ is nonempty}\}} > \theta \quad \mathrm{a.s.}
}
And by geometric localization with full density, there is $K$ such that
\eq{
\liminf_{n \to \infty} \frac{1}{n} \sum_{i = 0}^{n-1} \one_{\{f_i \in \GG_{\eps,K}\}} > \theta \quad \mathrm{a.s.} 
}
Therefore, there almost surely exists some $N$ satisfying
\eq{
n \geq N \quad &\implies \quad \frac{1}{n} \sum_{i = 0}^{n-1} \one_{\{\AA_i^\eps \text{ is nonempty}\}} > \theta \quad \text{and} \quad
\frac{1}{n}\sum_{i = 0}^{n-1} \one_{\{f_i \in \GG_{\eps,K}\}} > \theta.
}
%
%
So for $n \geq N$, there are at least $(2\theta - 1)n$ numbers $i$ between 0 and $n-1$ for which both of the following statements are true:
First, $f_i(x) > \eps$ for some (and thus any) mode $x \in \Z^d$ of $f_i$. 
Second $\mu_i^\beta(\sigma_i \in D_i) > 1 - \eps$ for some $D_i \subset \Z^d$ with $\diam(D_i) \leq K$.
For such $i$, \textit{all} modes must belong to $D_i$, which implies $D_i \subset \CC_i^K$.
In particular, 
\eq{
\mu_i^\beta(\sigma_i \in \CC_i^K) \geq \mu_i^\beta(\sigma_i \in D_i) > 1 - \eps > \theta,
}
and so
\eq{
n \geq N \quad \implies \quad \frac{1}{n} \sum_{i = 0}^{n-1} \mu_i^\beta(\sigma_i \in \CC_i^K) > \theta(2\theta - 1).
}
As $\mu_i^\beta(\sigma_i \in \CC_i^L) \geq \mu_i^\beta(\sigma_i \in \CC_i^{K})$ for $L \geq K$, we in fact have
\eq{
\lim_{K \to \infty} \liminf_{n \to \infty} \frac{1}{n} \sum_{i = 0}^{n-1}\mu_i^\beta(\sigma_i \in \CC_i^K) \geq \theta(2\theta-1) \quad \mathrm{a.s.}
}
As $\theta < 1$ is arbitrary, we can take a countable sequence $\theta_k \to 1$ to conclude \eqref{favorite_region}.
\end{proof}

\section{Open problems}
\label{open_problems}
Below we mention some remaining questions concerning the methods of this chapter.
For a broader discussion of  open problems on directed polymers, we refer the reader to Section 12.9 of~\cite{denHollander09}.

\begin{enumerate}

\item For the (1+1)-dimensional log-gamma model, localization occurs around a single favorite region \cite{comets-nguyen16}. Does this hold more generally, specifically in the low temperature regime? 
We showed via Theorem \ref{localized_subsequence}(b) and Proposition \ref{localization_thm} that the answer is yes if the single-copy condition holds.
Nevertheless, it is challenging to rule out the possibility that the endpoint distribution maintains an unbounded number of favorite regions, so that some $\rho \in \MM$ may put mass on partitioned subprobability measures supported on infinitely many copies of $\Z^d$. Such behavior has been observed for other types of Gibbs measures \cite{comets-dembo01,barral-rhodes-vargas12}.

\item The set $\MM$ defined by \eqref{M_def} is closed and convex in $\PP(\SS)$.
By Theorem \ref{characterization}, $\MM$ is a singleton at high temperature.  
Is the same true in the low temperature phase?  If true, this would imply that the empirical measure converges to a deterministic limit, instead of the subsequential convergence that is established in this chapter.
If false, can one describe the extreme points of $\MM$?

\item Does the endpoint distribution converge in law? In other words, can we go beyond the Ces\`aro averages and prove limiting results for the actual endpoint distribution? Note that if this is true, then the limiting law must be an element of $\MM$.

\item Can the methods of this chapter be extended to understand the distinction between  \eqref{strong_disorder} and \eqref{VSD}? Such an extension may lead to the resolution of some longstanding questions in this area, such as the following. (a) For $P = $ SRW and $d \geq 3$, do the critical values $\beta_{\mathrm{c}}$ and $\bar{\beta}_\cc$ coincide?
This is known to be the case for $d = 1$ \cite{comets-vargas06} and for $d = 2$ \cite{lacoin10}.
If the answer is no, then there would exist \textit{critical strong disorder}, in which  \eqref{strong_disorder} holds but not \eqref{VSD}. (b) For $P = $ SRW and $d \geq 3$, is there strong disorder at inverse temperature $\bar{\beta}_{\mathrm{c}}$?
For analogous models on self-similar trees, the answer is yes \cite{kahane-peyriere76}.

\end{enumerate}

    
    \chapter{Localization in Gaussian disordered systems} \label{replica}

\section{Introduction}
A ubiquitous theme in statistical mechanics is to understand how a system
behaves differently at high and low temperatures. In a disordered system,
where the interactions between its elements are governed by random quantities, the strength of the disorder is determined by temperature. Namely,
high temperatures mean the disorder is weak, and the system is likely to
resemble a generic one based on entropy. On the other hand, low temperatures indicate strong disorder, which creates dramatically different behavior
in which the system is constrained to a small set of states that are energetically favorable. In the latter case, this concentration phenomenon is often
called ``localization".

A useful statistic in distinguishing different temperature regimes is the
so-called ``replica overlap". That is, given the disorder, one can study the
similarity of two independently observed states. If the disorder is strong,
then these two states should closely resemble one another with good probability, since we believe the system is bound to a relatively small number
of possible realizations. Some version of this statement has been rigorously
established in a number of contexts, most famously in spin glass theory but
also in the settings of disordered random walks and disordered Brownian motion. Unfortunately, it does not follow that the number of realizable states
is small, but only that there is small number of states that are observed with
positive probability.

In the present study, our entry point to this problem is to consider \textit{conditional} overlap. Whereas previous results in the literature show the overlap
distribution between two independent states has a nonzero component, we
ask whether the same is true even if one conditions on the first state. That is,
does a typical state \textit{always} have positive expected overlap with an independent one? We show that for a broad class of Gaussian disordered systems,
the answer is yes, the key implication being that the \textit{entire} realizable state
space is small. Specifically, there is an $O(1)$ number of states such that all but a negligible fraction of samples from the system will have positive overlap with one of these
states. 

The general setting, notation, motivation, and results are given in Sections~\ref{model}--\ref{results}, respectively. 
The consequences for spin glasses, directed polymers,
and other Gaussian fields are discussed in Sections~\ref{applications} and~\ref{other_fields}.

\subsection{Model and assumptions} \label{model}
Let $(\Omega,\FF,\P)$ be an abstract probability space, and $(\Sigma_n)_{n\geq1}$ a sequence of Polish spaces 
equipped respectively with probability measures $(P_n)_{n\geq1}$. 
For each $n$, we consider a centered Gaussian field $H_n$ indexed by $\Sigma_n$ and defined on $\Omega$.  
Viewing this field as a Hamiltonian, we have the associated Gibbs measure at inverse temperature $\beta$:
\eq{
\mu_{n}^\beta(\dd\sigma) \coloneqq \frac{\e^{\beta H_n(\sigma)}}{Z_n(\beta)}\ P_n(\dd\sigma), \quad \text{where} \quad
Z_n(\beta) \coloneqq \int \e^{\beta H_n(\sigma)}\ P_n(\dd\sigma).
}
%
Our results concern the relationship between the \textit{free energy},
\eq{
F_n(\beta) \coloneqq \frac{1}{n}\log Z_n(\beta),
}
and the covariance structure of $H_n$.
We make the following assumptions:
\begin{itemize}
\item There is a deterministic function $p : \R \to \R$ such that
\[\label{free_energy_assumption} \tag{A1}
\lim_{n\to\infty} F_n(\beta) = p(\beta) \quad \P\text{-}\mathrm{a.s.} \text{ and in $L^1(\P)$, for every $\beta\in\R$}. 
\]
\item For every $\sigma\in\Sigma_n$,
\[ \label{variance_assumption} \tag{A2}
\Var H_n(\sigma) = n.
\]
\item For every $\sigma^1,\sigma^2\in\Sigma_n$,
\[ \label{positive_overlap} \tag{A3}
\Cov(H_n(\sigma^1),H_n(\sigma^2))\geq-n\EEE_n,
\]
where $\EEE_n$ is a nonnegative constant tending to $0$ as $n\to\infty$.
\item For each $n$, there exist measurable real-valued functions $(\vphi_{i,n})_{i=1}^\infty$ on $\Sigma_n$ 
and i.i.d.~standard normal random variables $(g_{i,n})_{i=1}^\infty$ defined on $\Omega$ so that for each $\sigma\in\Sigma_n$, with $\P$-probability $1$, 
\[ \label{field_decomposition} \tag{A4}
H_n(\sigma)  = \sum_{i=1}^\infty g_{i,n}\vphi_{i,n}(\sigma), 
\]
where the series on the right converges in $L^2(\P)$. 
\end{itemize}

\begin{remark}
In all applications of interest (see Section~\ref{applications}), the hypothesis \eqref{positive_overlap} is trivially satisfied with $\EEE_n = 0$.
Nevertheless, we assume throughout only that $\EEE_n\to0$ (at any rate).
This modest relaxation is made so our results can apply to slightly more general models, for instance perturbations of the standard models we will soon describe.
\end{remark}

\begin{remark}
The condition \eqref{field_decomposition} is very mild: For example, it always holds when $\Sigma_n$ is finite. More generally, a sufficient condition for the existence of a representation \eqref{field_decomposition} is that $\Sigma_n$ is compact in the metric defined by $H_n$ (namely, the metric that defines the distance between $\sigma$ and $\sigma'$ as the $L^2$ distance between the random variables $H_n(\sigma)$ and $H_n(\sigma')$). For a proof of this standard result, see \cite[Theorem 3.1.1]{adler-taylor07}. 
Furthermore, in all applications of interest, $H_n$ will actually be explicitly defined using a sum of the form~\eqref{field_decomposition}.


\end{remark}

\subsection{Notation} \label{notation}
Unless stated otherwise, ``almost sure" and ``in $L^\alpha$" statements are with respect to $\P$.
We will use $E_n$ and $\E$ to denote expectation with respect to $P_n$ and $\P$, respectively.
Absent any decoration, $\langle \cdot \rangle$ will always denote expectation with respect to $\mu_{n}^{\beta}$, meaning
\eq{
\langle f(\sigma)\rangle = \frac{E_n(f(\sigma)\e^{\beta H_n(\sigma)})}{E_n(\e^{\beta H_n(\sigma)})}.
}
At various points in the chapter, we will decorate $\langle\cdot\rangle$ to denote expectation
with respect to some perturbation of $\mu_{n}^\beta$. 
The type of perturbation will
change between sections.
The symbols $\sigma^{j}$, $j = 1,2,\dots$, shall denote independent samples from $\mu_{n}^\beta$ if appearing within $\langle\cdot\rangle$, or from $P_n$ if appearing within~$E_n(\cdot)$.
We will refer to the vector $\vc g_n = (g_{i,n})_{i=1}^\infty$ as the \textit{disorder} or \textit{random environment}. 
Sometimes we will consider multiple environments at the same time, which will necessitate that we write $\mu_{n,\vc g_n}^\beta$ instead of $\mu_n^\beta$ to emphasize the dependence on the environment $\vc g_n$.

In the sequel, $\sum_i$ will always mean $\sum_{i=1}^\infty$,
and we will condense our notation to $\vphi_i = \vphi_{i,n}(\sigma)$ when we are dealing with some fixed $n$. Similarly, $g_{i,n}$ will be shortened to $g_i$ and $\vc g_n$ will be shortened to $\vc g$. 
Also, $C(\cdot)$ will indicate a positive constant that depends only on the argument(s).
In particular, no such constant depends on $\vc g$ or $n$.
We will not concern ourselves with the
precise value, which may change from line to line.

\subsection{Motivation} \label{motivation}
Our results will be stated in terms of the correlation or \textit{overlap function},
\eq{
\RR(\sigma^1,\sigma^2) &\coloneqq \frac{1}{n}\Cov(H_n(\sigma^1),H_n(\sigma^2)), 
\quad \sigma^1,\sigma^2\in\Sigma_n.
}
Note that \eqref{variance_assumption} and \eqref{positive_overlap} imply
\eq{
-\EEE_n\leq\RR(\sigma^1,\sigma^2) \leq 1.
}
We will often abbreviate $\RR(\sigma^j,\sigma^k)$ to $\RR_{j,k}$. 

The Gaussian process $(H_n(\sigma))_{\sigma\in\Sigma_n}$ naturally defines a
(pseudo)metric $\rho$ on $\Sigma_n$, given by
\eeq{\label{rhodef}
\rho(\sigma^1,\sigma^2)\coloneqq 1-\RR_{1,2}.
}
Given the metric topology, we can study the so-called ``energy landscape"
of $\beta H_n$ on $\Sigma_n$. 
The geometry of this landscape is intimately
related to the free energy. By Jensen's inequality,
\eeq{ \label{basic_jensen_gap}
\E F_n(\beta) \leq \frac{1}{n}\log \E Z_n(\beta) \stackrel{\mbox{\footnotesize(\text{Lemma }\ref{moments_lemma})}}{=}\frac{\beta^2}{2},
}
which in particular implies $p(\beta) \leq \beta^2/2$. In general, whether or not this
inequality is strict determines the nature of the energy landscape: In order
for $p(\beta) = \beta^2/2$, the fluctuations of $\log Z_n(\beta)$ must be relatively small so that
the Jensen gap in \eqref{basic_jensen_gap} is $o(1)$. This behavior arises when the Gaussian deviations
of $\beta H_n(\sigma)$ are washed out by the entropy of $P_n$, creating a more or less flat
landscape. On the other hand, if $p(\beta) < \beta^2/2$, then these deviations will
have overcome the entropy of $P_n$, producing large peaks and valleys where
$\beta H_n(\sigma)$ is exceptionally positive or negative. From a physical perspective,
this latter scenario is more interesting, as these peaks can account for an
exponentially vanishing fraction of the state space even as their union accounts for a non-vanishing fraction of the mass of $\mu_{n}^\beta$.
The primary goal of this chapter is to
give a sufficient condition for when (in a sense Theorem~\ref{easy_cor} makes precise)
$\mu_{n}^\beta$ places all of its mass on this union of peaks.

Suppose that $p(\cdot)$ is differentiable at $\beta\geq0$. 
Using Gaussian integration by parts, it is not difficult to show (as we do in Corollary~\ref{overlap_identity_cor}) that
\eeq{ \label{expected_overlap_formula}
\lim_{n\to\infty} \E\langle \RR_{1,2}\rangle = 1 - \frac{p'(\beta)}{\beta}.
}
This identity has been observed before (e.g.~see \cite{aizenman-lebowitz-ruelle87,comets-neveu95,talagrand06III,panchenko08}, \cite[Lemma 7.1]{carmona-hu02}, and \cite[Theorem 6.1]{comets17}).
For this reason, the condition in which we are interested is $p'(\beta)<\beta$.
To improve upon \eqref{expected_overlap_formula}, a first step is to show that if $\E\langle \RR_{1,2}\rangle$ is bounded away
from $0$, then the random variable $\langle \RR_{1,2}\rangle$ is itself stochastically bounded away
from $0$. This is the content of Theorem~\ref{averages_squared}. 
The more substantial contribution
of this chapter, however, is to bootstrap this result to a proof of Theorem~\ref{expected_overlap_thm},
which roughly says that $\langle \RR_{1,2}\rangle$ is stochastically bounded away from $0$ even
conditional on $\sigma^1$.

It follows from Corollary~\ref{overlap_identity_cor} that $p'(\beta) < \beta$ implies $p(\beta) < \beta^2/2$,
but it is natural to ask whether the two conditions are equivalent.
This equivalence is true for spin glasses \cite{talagrand06III,panchenko08} and is believed to be true for directed
polymers \cite[Conjecture 6.1]{comets17}. But at the level of generality considered in
this chapter, we are not aware of any conjecture. 
In any case, for the examples we consider in Section~\ref{applications}, both conditions will be true for sufficiently large $\beta$.

\subsection{Results} \label{results}
Our main result is Theorem~\ref{easy_cor}, stated below. It says that at low temperatures, one can find a finite number of (random) states such that
almost any sample from the Gibbs measure will have positive overlap with
at least one of them.
To state this precisely, let us define the sets
\eeq{ \label{ball_def}
\BB(\sigma,\delta) \coloneqq \{\sigma'\in\Sigma_n : \RR(\sigma, \sigma')\geq\delta\}, \quad \sigma\in\Sigma_n,\, \delta>0.
}
In terms of the metric $\rho$ defined in \eqref{rhodef}, this is just the ball of radius $1-\delta$ centered at $\sigma$. Typically, such balls have  vanishingly small size under $P_n$ as $n\to\infty$, which should be contrasted with the following behavior of the Gibbs measure.
\begin{thm} \label{easy_cor}
Assume \eqref{free_energy_assumption}, \eqref{variance_assumption}, \eqref{positive_overlap}, and \eqref{field_decomposition}.
If $\beta\geq0$ is a point of differentiability for $p(\cdot)$, and $p'(\beta) < \beta$,
then for every $\eps > 0$, there exist integers $k = k(\beta,\eps)$ and $n_0 = n_0(\beta,\eps)$ and a number $\delta = \delta(\beta,\eps)>0$ such that the following is true for all $n\geq n_0$.
With $\P$-probability at least $1-\eps$, there exist $\sigma^1,\dots,\sigma^k\in\Sigma_n$ such that 
\eq{
\mu_{n}^\beta\Big(\bigcup_{j=1}^k \BB(\sigma^j, \delta)\Big) \geq 1 - \eps.
}
\end{thm}
It is worth noting that in some cases, such as the directed polymer model defined in Section~\ref{directed_polymers}, it is possible (although unproven) that $k$ can be taken equal to $1$ if $\delta$ is chosen sufficiently small.
For other models, however, such as polymers on trees or the Random Energy Model discussed in Section~\ref{other_fields}, $k$ will necessarily diverge as $\eps\to0$.

We will derive Theorem~\ref{easy_cor} as a corollary of Theorem~\ref{expected_overlap_thm}, stated below. In fact,  Theorem~\ref{easy_cor} is actually equivalent to Theorem~\ref{expected_overlap_thm}, although the latter has a less transparent statement, which is why we have stated Theorem~\ref{easy_cor} as our main result.


Theorem~\ref{expected_overlap_thm} concerns the following function on $\Sigma_n$.
For given $\sigma^1\in\Sigma_n$, we will write the conditional expectation of $\RR_{1,2}$ as
\eeq{\label{rdef}
\RR(\sigma^1) \coloneqq \givena{\RR_{1,2}}{\sigma^1}
= \frac{1}{n}\sum_{i=1}^\infty \vphi_{i,n}(\sigma^1) \langle\vphi_{i,n}(\sigma^2)\rangle.
}
(Note that the expectation $\givena{\cdot}{\sigma^1}$ can be exchanged with the sum because of Fubini's theorem, in light of \eqref{variance_assumption}.)
Given $\delta>0$, we consider the set
\eeq{ \label{A_def}
\AA_{n,\delta} \coloneqq \{\sigma\in\Sigma_n : \RR(\sigma) \leq \delta\}. 
} 
With this notation, the quantity $\langle \one_{\AA_{n,\delta}}\rangle$ 
is the probability that a state sampled from $\mu_{n}^\beta$ has expected overlap at most $\delta$ with an independent sample from $\mu_{n}^\beta$.
Theorem~\ref{expected_overlap_thm} says that at low temperatures and for small $\delta$, this probability is typically small.

\begin{thm} \label{expected_overlap_thm}
Assume \eqref{free_energy_assumption}, \eqref{variance_assumption}, \eqref{positive_overlap}, and \eqref{field_decomposition}.
If $\beta\geq0$ is a point of differentiability for $p(\cdot)$, and $p'(\beta) < \beta$, then for every $\eps>0$, there exists $\delta = \delta(\beta,\eps) > 0$ sufficiently small that
\eeq{ \label{expected_overlap_thm_eq}
\limsup_{n\to\infty} \E\langle\one_{\AA_{n,\delta}}\rangle \leq \eps.
}
\end{thm}

To prove Theorem~\ref{expected_overlap_thm}, we first have to prove a weaker theorem stated below. This result considers the following event in the $\sigma$-algebra $\FF$,
\eq{
B_{n,\delta} \coloneqq \{\langle \RR_{1,2}\rangle \leq \delta\},
}
and shows that its probability is small at low temperature. 

\begin{thm} \label{averages_squared}
Assume \eqref{free_energy_assumption}, \eqref{variance_assumption}, \eqref{positive_overlap}, and \eqref{field_decomposition}.
If $\beta\geq0$ is a point of differentiability for $p(\cdot)$, and $p'(\beta) < \beta$, then for every $\eps > 0$, there exists $\delta = \delta(\beta,\eps) > 0$ sufficiently small such that
\eeq{ \label{averages_squared_eq}
\limsup_{n\to\infty} \P(B_{n,\delta}) \leq \eps.
}
\end{thm}


Theorem~\ref{averages_squared} is proved in Section~\ref{proof_1}, Theorem~\ref{expected_overlap_thm} in Section~\ref{proof_2}, and the equivalence of Theorems~\ref{easy_cor} and~\ref{expected_overlap_thm} in Section~\ref{cor_proof}.
In Section~\ref{prep_section}, we provide some general facts that are needed in the main arguments. 
A detailed sketch of the proof technique is given  in Section~\ref{sketchsec}. 
We will often simplify notation by writing $\AA_\delta$ and $B_\delta$, where the dependence on $n$ is understood and will not be a source of confusion.

\subsection{Applications} \label{applications}
For many applications, it would suffice to consider $\Sigma_n$ which is finite for every $n$.
Other applications, however, such as spherical spin glasses or directed polymers with a reference walk of unbounded support, require $\Sigma_n$ to be infinite.
It is for this reason that we have stated the setting and results in the generality seen above.
Now we discuss specific models of interest.

\subsubsection{Spin glasses} \label{spin_glasses}
Let $\Sigma_n = \{\pm 1\}^n$ (Ising case) or $\Sigma_n = \{\sigma\in\R^n :\|\sigma\|_2 = \sqrt{n}\}$ (spherical case), and take $P_n$ to be uniform measure on $\Sigma_n$. 
In the mean-field models, the Hamiltonian is of the form
\eeq{ \label{sg_hamiltonian}
H_n(\sigma) = \sum_{p\geq2} \frac{\beta_p}{n^{(p-1)/2}}\sum_{i_1,\dots,i_p = 1}^n g_{i_1,\dots,i_p}\sigma_{i_1}\cdots\sigma_{i_p}.
}
We will assume
\eeq{ \label{sg_assumption_1}
\sum_{p\geq2}\beta_p^2(1+\eps)^p < \infty \quad \text{for some $\eps>0$},
}
which is more restrictive than what we require but standard in the literature.
Standard applications of Gaussian concentration show that $|F_n(\beta) - \E F_n(\beta)|\to0$ almost surely and in $L^1$.
Assumption \eqref{free_energy_assumption} then follows from the convergence of $\E F_n(\beta)\to p(\beta)$, where $p(\beta)$ is given by a formula depending on the model.
In the Ising case, there is the celebrated Parisi formula \cite{parisi79,parisi80}, proved by
Talagrand \cite{talagrand06} for even-spin models, building on the seminal work of Guerra \cite{guerra03}.
It was later extended by Panchenko \cite{panchenko14} to general mixed $p$-spins.
For the spherical model, there is a simpler and elegant formula predicted
by Crisanti and Sommers \cite{crisanti-sommers92}, and proved by Talagrand \cite{talagrand06II} and Chen \cite{chen13}.

To accommodate assumptions \eqref{variance_assumption} and \eqref{positive_overlap}, one should assume the function $\xi(q) \coloneqq \sum_{p\geq2}\beta_p^2q^p$ satisfies
\eeq{ \label{sg_assumption_2}
\xi(1) = 1 \quad \text{and} \quad \xi(q) \geq 0 \quad \text{for all $q\in[-1,1]$.}
}
This is because 
\eq{
\RR_{j,k} = \xi(R_{j,k}), \quad \text{where} \quad R_{j,k} \coloneqq \frac{1}{n}\sum_{i=1}^n \sigma^j_i\sigma^k_i \in [-1,1].
}
Note that the second assumption in \eqref{sg_assumption_2} is automatic if $\beta_p = 0$ for all odd $p$.
When $\xi(q) = q^2$, \eqref{sg_hamiltonian} is the classical Sherrington--Kirkpatrick (SK) model~\cite{sherrington-kirkpatrick75} if $\Sigma_n = \{\pm1\}^n$, or the spherical SK model~\cite{kosterlitz-thouless-jones76} if $\Sigma_n = \{\sigma\in\R^n : \|\sigma\|_2 = \sqrt{n}\}$.

In the spin glass literature, $R_{1,2}$ is the usual \textit{replica overlap} that is studied as an order parameter for the system \cite{talagrand03}.
Roughly speaking, $R_{1,2}$ converges
to $0$ when $p(\beta) = \beta^2/2$, but converges in law to a non-trivial distribution
when $p(\beta) < \beta^2/2$. 
In the latter case, the model exhibits what is known as
replica symmetry breaking (RSB). 
If the limiting distribution of $R_{1,2}$, called
the Parisi measure, contains $k + 1$ distinct atoms (one of which must be $0$
\cite{auffinger-chen15I}), then $\xi$ is said to be $k$RSB. 
For instance, spherical pure $p$-spin models
are $1$RSB for large $\beta$ \cite{panchenko-talagrand07}, and it was recently shown that some spherical mixed spin models are $2$RSB at zero temperature \cite{auffinger-zeng19}.
In the Ising case, however, the Parisi measure is
expected to have an infinite support throughout the low-temperature phase
(with $0$ in the support but not as an atom; see \cite[Page 15]{bolthausen07}), a behavior
referred to as \textit{full}-RSB (FRSB). Proving such a statement is a problem of
great interest and has been solved at zero temperature \cite{auffinger-chen-zeng20}.
For spherical models, the situation is somewhat clearer; in \cite{chen-sen17}, sufficient conditions were given for both $1$RSB and FRSB, again at zero temperature.

The simplest type of symmetry breaking, $1$RSB, admits the following
heuristic picture. The state space $\Sigma_n$ is (from the perspective of $\mu_{n}^\beta$)
separated into many orthogonal parts called ``pure states", within which the intra-cluster overlap concentrates on some positive value $q > 0$. 
In the $2$RSB picture, the pure states are not necessarily orthogonal, but rather grouped
together into larger clusters which are themselves orthogonal. In this case,
the overlap could be $q$ (same pure state), $q'\in (0, q)$ (same cluster but different pure state), or $0$ (different clusters). 
The complexity increases in the same
fashion for general $k$RSB. 
In FRSB, the clusters become infinitely nested,
yielding a continuous spectrum of possible overlaps while maintaining ``ultrametric" structure \cite{panchenko13II}. 
In any case, though, there should be asymptotically
no part of the state space which is orthogonal to everything; that is, the pure
states exhaust $\mu_{n}^\beta$.

Absent the intricate hierarchical picture described above, the following rephrasing of Theorem~\ref{easy_cor} confirms this idea.

\begin{thm} \label{sg_thm}
Assume \eqref{sg_assumption_1} and \eqref{sg_assumption_2}, and that $\beta\geq0$ is a point of differentiability for $p(\cdot)$ such that $p'(\beta) < \beta$.
Then for every $\eps > 0$, there exist integers $k = k(\beta,\eps)$ and $n_0 = n_0(\beta,\eps)$ and a number $\delta = \delta(\beta,\eps)>0$ such that the following is true for all $n\geq n_0$.
With $\P$-probability at least $1-\eps$, there exist $\sigma^1,\dots,\sigma^k\in\Sigma_n$ such that
\eq{ 
\mu_{n}^\beta\Big(\bigcup_{j=1}^k \{\sigma^{k+1}\in\Sigma_n : |R_{j,k+1}|\geq\delta\}\Big) \geq 1 - \eps.
}
\end{thm}

%

The proof of the above Theorem follows simply from Theorem~\ref{easy_cor} and the observation that by \eqref{sg_assumption_1}, $\xi$ is continuous at $0$. 

Under strong assumptions on $\xi$ and the overlap distribution, namely the (extended) Ghirlanda--Guerra identities, much more precise results were proved by \mbox{Talagrand} \cite[Theorem 2.4]{talagrand10} and later Jagannath \cite[Corollary 2.8]{jagannath17}. 
For spherical pure spin models, similar results were
proved by Subag \cite[Theorem 1]{subag17}. 
An advantage of our approach, beyond its
generality, is that our assumptions on $\xi$ are elementary to check and fairly
loose (they include all even spin models), and the temperature condition
$p'(\beta) < \beta$ is explicit and sharp.

While the literature on replica overlaps in spin glasses is vast, the reader
will find much information in \cite{mezard-parisi-virasoro87,talagrand11I,talagrand11II,panchenko13}; see also \cite{jagannath-tobasco17II} and references therein.

\subsubsection{Directed polymers} \label{directed_polymers}
Given a positive integer $d$, let $\Sigma_n$ be the set of all maps from $\{0,1,\ldots,n\}$ into $\Z^d$, and let $P_n$ be the law, projected onto $\Sigma_n$, of a homogeneous random walk on $\Z^d$ starting at the origin.
That is, there is some probability mass function
$K$ on $\Z^d$ such that
\begin{subequations} \label{walk_assumption}
\begin{align}
P_n(\sigma(0) = 0) &= 1,\\
P_n\givenp{\sigma(i) = y}{\sigma(i-1)=x} &= K(y-x), \quad 1\leq i\leq n.
\end{align}
\end{subequations}
Let $(g(i,x) : i\geq1,x\in\Z^d)$ be i.i.d.~standard normal random variables.
The Hamiltonian for the model of directed polymers in Gaussian environment is then given by
\eq{
H_n(\sigma) = \sum_{i=1}^n g(i,\sigma(i)) = \sum_{i=1}^n\sum_{x\in\Z^d} g(i,x)\one_{\{\sigma(i)=x\}}.
}
In this case, the overlap between two paths is the fraction of time they intersect:
\eeq{ \label{polymer_overlap}
\RR_{1,2} = \frac{1}{n}\sum_{i=1}^n \one_{\{\sigma^1(i)=\sigma^2(i)\}}.
}
The assumption \eqref{free_energy_assumption} holds for any $K$ (see Proposition \ref{free_energy_converges}), although typically
$P_n$ is taken to be standard simple random walk; all the references below
refer to this case. 
Alternatively, one can consider point-to-point polymer measures, meaning the endpoint of the polymer is fixed.
This case is studied in \cite{rassoul-seppalainen14,georgiou-rassoul-seppalainen16} and accommodates the same structure as above, up to changing the reference measure $P_n$.

Notice that the identity \eqref{expected_overlap_formula} immediately implies $\lim_{n\to\infty}\E\langle \RR_{1,2}\rangle > 0$ when $p'(\beta) < \beta$. 
Theorem~\ref{averages_squared} goes a step further, showing that the random variable $\langle \RR_{1,2}\rangle$ is itself stochastically bounded away from 0.
For a certain class of
bounded random environments,  a quantitative version of Theorem~\ref{averages_squared} was proved by Chatterjee~\cite{chatterjee19}, but Theorem~\ref{expected_overlap_thm} is the first of its kind. 
Unlike some other conjectured polymer properties, the statement \eqref{expected_overlap_thm_eq} has not been verified for the so-called exactly solvable models in $d = 1$ \cite{seppalainen12,corwin-seppalainen-shen15,oconnell-ortmann15,barraquand-corwin17,thiery-doussal15}. For heavy-tailed
environments, a stronger notion of localization is considered in \cite{auffinger-louidor11,torri16} and
also discussed in \cite{dey-zygouras16,berger-torri19}. 
Historically, studying pathwise localization has found
somewhat greater success in the context of \textit{continuous} space-time polymer
models \cite{comets-yoshida05,comets-yoshida13,comets-cranston13,comets-cosco??}.

For polymers in Gaussian environment, it is known (see \cite[Proposition 2.1(iii)]{comets17}) that $p'$
is bounded from above by a constant, and so $\E\langle \RR_{1,2}\rangle\to1$ as $\beta\to\infty$ by \eqref{expected_overlap_formula}. 
(While convexity guarantees $p(\cdot)$ is differentiable almost everywhere, it is an open problem to show that $p(\cdot)$ is everywhere differentiable, let
alone analytic away from the critical value separating the high and low temperature phases.) In this sense, the polymer measure becomes completely
localized near the maximizer of $H_n(\cdot)$ as $\beta\to\infty$. A main motivation for the
present study was to formulate a version of ``complete localization" for fixed
$\beta$ in the low-temperature regime.

In Chapter \ref{endpoint}, following \cite{vargas07}, we phrased complete localization in terms of the endpoint
distribution: the law of $\sigma_n$ under $\mu^\beta_{n}$.
Loosely speaking, what was shown
is that if $p(\beta) < \beta^2/2$, then with probability at least $1-\eps$, one can find sufficiently many (independent of $n$) random vertices $x_1,\dots,x_k$ in $\Z^d$ so that
\eeq{ \label{atomic_localization}
\mu^\beta_{n}\big(\big\{\sigma : \sigma(n) \in \{x_1,\dots,x_k\}\big\}\big) \geq 1 - \eps.
}
This behavior is called ``asymptotic pure atomicity", referring to the fact that
even as $n$ grows large, the endpoint distribution remains concentrated on an
$O(1)$ number of sites (rather than diffuse polynomially as in simple random
walk). 
This is analogous to the results of this chapter, except that the endpoint statistic
has been used to reduce the state space to $\Z^d$. 
The pathwise localization in Theorem~\ref{easy_cor} describes a more global phenomenon occurring in the original state space $\Sigma_n$.
Rephrased below, it says that up to arbitrarily small probabilities, the Gibbs measure is concentrated on paths intersecting  one of a few distinguished paths a positive fraction of the time. 

\begin{thm} \label{polymer_thm}
Assume \eqref{walk_assumption} and that $\beta\geq0$ is a point of differentiability for $p(\cdot)$ such that $p'(\beta) < \beta$.
Then for every $\eps > 0$, there exist integers $k = k(\beta,\eps)$ and $n_0 = n_0(\beta,\eps)$ and a number $\delta = \delta(\beta,\eps)>0$ such that the following is true for all $n\geq n_0$.
With $\P$-probability at least $1-\eps$, there exist paths $\sigma^1,\dots,\sigma^k\in\Sigma_n$ such that
\eq{
\mu_{n}^\beta\bigg(\bigcup_{j=1}^k \Big\{\sigma^{k+1}\in\Sigma_n: \frac{1}{n}\sum_{i=1}^n\one_{\{\sigma^{k+1}(i)=\sigma^j(i)\}}\geq\delta\Big\}\bigg) \geq 1 - \eps.
}
\end{thm}

In Section~\ref{no_path_atom}, we demonstrate that path localization does not
occur in the atomic sense \eqref{atomic_localization}. 
That is, any bounded number of paths will have a total mass under
$\mu^\beta_{n}$ that decays to $0$ as $n\to\infty$.
For this reason, the definitions from \cite{vargas07} and Chapter \ref{endpoint} of complete localization for the endpoint are inadequate for path
localization, necessitating a statement in terms of overlap. This distinguishes
the lattice polymer model from its mean-field counterpart on
regular trees, which is simply the statistical mechanical version of branching
random walk \cite{derrida-spohn88,comets17}. 
For those models, the endpoint distribution on the
leaves of the tree is obviously equivalent to the Gibbs measure because each
leaf is the termination point of a unique path. Moreover, the results of Chapter \ref{endpoint}
can be interpreted equally well (and improved upon) in that setting (see
\cite{barral-rhodes-vargas12,jagannath16}), and so we will not elaborate on the fact that polymers on trees also fit into the framework of this chapter.

\subsection{Other Gaussian fields} \label{other_fields}
Here we mention several other models to
which our results apply but for which they are not new. Indeed, each model
below is known to exhibit Poisson--Dirichlet statistics for the masses assigned by
$\mu^\beta_{n}$ to the ``peaks" discussed in the motivating Section~\ref{motivation}. 
In particular, asymptotically no mass is given to states having vanishing expected overlap with an independent sample.
\begin{itemize}
\item Derrida's Random Energy Model (REM) \cite{derrida80,derrida81} is set on the hypercube
$\Sigma_n = \{\pm1\}^n$ with uniform measure, and has the simplest possible
covariance structure: $\RR_{j,k} = \delta_{j,k}$.
With $\beta_\mathrm{c} = \sqrt{2\log 2}$, the following formula holds \cite[Theorem 9.1.2]{bovier06}:
\eq{
p(\beta) = \begin{cases}
\beta^2/2 &\beta \leq \beta_\mathrm{c}, \\
\beta_\mathrm{c}^2/2 + (\beta - \beta_\mathrm{c})\beta_\mathrm{c} &\beta>\beta_\mathrm{c}.
\end{cases}
}
See also \cite[Chapter 1]{talagrand03book}, in particular Theorem 1.2.1.
\item The generalized random energy models have non-trivial covariance structure \cite{derrida85}, and can be tuned to have an arbitrary number of phase transitions.
The condition $p'(\beta) < \beta$ is satisfied as soon as the first phase transition occurs.
See also \cite[Chapter 10]{bovier06}.

\item Finally, in \cite{arguin-zindy14} Arguin and Zindy studied a discretization of a log-correlated Gaussian field from \cite{barral-mandelbrot02,bacry-muzy03} 
which has the same free energy as the REM. 
Their particular model had the
technical complication of correlations not following a
tree structure, unlike for instance the discrete Gaussian free field.
\end{itemize}

\section{Proof sketches}\label{sketchsec}
The proofs of Theorems~\ref{expected_overlap_thm} and~\ref{averages_squared} are long, but they contain ideas that may be useful for other problems. 
Therefore, we have included this proof-sketch section which, while still rather lengthy, distills the arguments to their central ideas.
It introduces some of the notations that will be used later in the chapter; however, these notations will be reintroduced in the later sections, so it is safe to skip directly to Section~\ref{prep_section} should the reader decide to do so.

\subsection{{Proof sketch of Theorem~\ASref}}\label{sketchsec1}
For simplicity, let us assume that the representation \eqref{field_decomposition} consists of only finitely many terms:
\eq{
H_n(\sigma) = \sum_{i=1}^N g_i \vphi_i(\sigma).
}
Following the argument described below, the general case is handled by some routine calculations (made in Section~\ref{gibbs_measure}) to check that sending $N\to\infty$ poses no issues.

Given \eqref{expected_overlap_formula}, it is clear that $p'(\beta) < \beta$ would imply \eqref{averages_squared_eq} if we knew that $\langle\RR_{1,2}\rangle$ concentrates around its mean as $n\to\infty$.
Unfortunately, this may not be true in general.
Therefore, as a way of artificially imposing concentration, we let the environment evolve as an Ornstein--Uhlenbeck (OU) flow, and then eventually take an average over a short time interval.
Formally, this means we consider
\eeq{ \label{OU_def_sketch}
\vc g_t \coloneqq \e^{-t}\vc g + \e^{-t} \vc W({\e^{2t}-1}), \quad t\geq0,
}
where $\vc W(\cdot) = (W_i(\cdot))_{i=1}^N$ are independent Brownian motions that are also independent of $\vc g = \vc g_0$.
Recall the OU generator $\LL \coloneqq \Delta - \vc x \cdot \nabla$, and the fact that $\E\LL f(\vc g) =0$ for any $f$ with suitable regularity.
By expanding $f$ in an  orthonormal basis of eigenfunctions of $\LL$, and expressing both $\LL f(\vc g_t)$ and $\E \|\nabla f(\vc g)\|^2$ using the coefficients from this expansion, one can show that
\eeq{ \label{general_variance_bound}
\Var\bigg(\frac{1}{t} \int_0^{t} \LL f(\vc g_s)\ \dd s\bigg) \le \frac{2}{t} \E\|\nabla f(\vc g)\|^2.  
}
This inequality, established in Lemma~\ref{OU_variance_lemma}, provides the proof's essential estimate when applied to $f(\vc g) = F_n(\beta)$.
For this $f$, it is easy to verify that $\E\|\nabla f(\vc g)\|^2=O(1/n)$, and
\eq{
\LL f(\vc g_t) = \beta^2 - \beta^2 \langle \RR_{1,2}\rangle_t - \beta \frac{\partial}{\partial \beta} F_{n,t}(\beta),
}
where $\langle \RR_{1,2}\rangle_t$ and $F_{n,t}(\beta)$ are the expected overlap and free energy, respectively, in the environment $\vc g_t$.
Moreover, from standard methods (worked out in Section~\ref{derivatives_section}), it follows that $\frac{\partial}{\partial\beta}F_{n,t}(\beta)\approx p'(\beta)$ with high probability. 
Combining these observations about $f$ with the general variance estimate \eqref{general_variance_bound}, we arrive at
\eeq{\label{sketcheq2}
\frac{1}{T/n}\int_0^{T/n} \langle \RR_{1,2}\rangle_t \ \dd t = 1- \frac{p'(\beta)}{\beta} + O(1/T). 
}
In other words, averaging $\langle \RR_{1,2}\rangle_t$ over a long enough interval, but whose size is still $O(1/n)$, results in a value close to the expectation suggested by \eqref{expected_overlap_formula}.
We choose $T=T(\eps)$ large enough depending on $\eps$, which determines the level of precision required in \eqref{sketcheq2}.

Next comes the most crucial step in the proof, where we show that if $\langle \RR_{1,2}\rangle=\langle \RR_{1,2}\rangle_0\le \delta$ for some small $\delta$, then for each $t\in[0,T(\eps)/n]$, the quantity $\langle \RR_{1,2}\rangle_t$ is also small with high probability. 
If $p'(\beta)<\beta$, this leads to a contradiction to \eqref{sketcheq2} if $\delta$ is small enough. 
To avoid this contradiction, the probability of $\langle \RR_{1,2}\rangle\le \delta$ happening in the first place must be small, which is what we want to show.

To demonstrate our crucial claim, we consider any $t = T/n$, where $T\leq T(\eps)$ and $n$ is large.
First, note that
\eeq{\label{rrab}
\langle \RR_{1,2}\rangle_t = \frac{\langle \RR_{1,2}\e^{\beta A_t+ \beta B_t}\rangle}{\langle \e^{\beta A_t+\beta B_t}\rangle},
}
where $B_t$ comes from the Brownian part of \eqref{OU_def_sketch}, and $A_t$ comes from the initial environment:
\eq{
A_t &\coloneqq (\e^{-t} - 1)(H_n(\sigma^1) + H_n(\sigma^2)), \qquad B_t \coloneqq  \e^{-t}\sum_{i} W_i(\e^{2t}-1)(\vphi_i(\sigma^1)+\vphi_i(\sigma^2)).
}
Since $t=T/n \ll 1$, we have
\[
A_t \approx -\frac{T}{n} (H_n(\sigma^1) + H_n(\sigma^2)).
\]
By standard arguments (again presented in Section~\ref{derivatives_section}), $H_n(\sigma^1)/n$ and $H_n(\sigma^2)/n$ are both close to $p'(\beta)$ with high probability under the Gibbs measure. 
Thus, for fixed $t$, the random variable $A_t$ behaves like a constant inside $\langle\cdot\rangle$. 
Consequently, we can reduce \eqref{rrab} to 
\eeq{ \label{gibbs_t_1}
\langle \RR_{1,2}\rangle_t \approx \frac{\langle \RR_{1,2}\e^{\beta B_t}\rangle}{\langle \e^{\beta B_t}\rangle}.
}
Now let $h_i := W_i(\e^{2t}-1)/\sqrt{\e^{2t}-1}$, so that $h_i\sim \NN(0,1)$.  
Again since $t=T/n \ll 1$, we have
\eq{
B_t &= \sqrt{1-\e^{-2t}}\sum_i h_i (\vphi_i(\sigma^1)+\vphi_i(\sigma^2)) \approx \sqrt{\frac{2T}{n}}\sum_i h_i (\vphi_i(\sigma^1)+\vphi_i(\sigma^2)).
}
Thus, if $\E_{\vc h}$ denotes expectation in $\vc h = (h_1,\ldots, h_N)$ only, then 
\eq{
\E_{\vc h} \langle \e^{\beta B_t}\rangle 
\approx
 \Big\langle\exp\Big(\frac{\beta^2T}{n}\sum_i (\vphi_i(\sigma^1)+\vphi_i(\sigma^2))^2\Big)\Big\rangle \stackrel{\mbox{\footnotesize\eqref{variance_assumption}}}{=} \exp\big(2\beta^2 T(1+\RR_{1,2})\big). 
}
In the event that $\langle \RR_{1,2}\rangle$ is small, the 
assumption \eqref{positive_overlap} implies that $\RR_{1,2}\approx 0$ with high probability under the Gibbs measure. 
Therefore, conditional on this event (which depends only on $\vc g$, not $\vc h$), we have
\[
\E_{\vc h} \langle \e^{\beta B_t}\rangle\approx \e^{2\beta^2T}.
\]
By a similar argument, we also have
\eq{
\E_{\vc h} \langle \e^{\beta B_t}\rangle^2 &\approx \E_{\vc h} \Big\langle \exp\Big(\beta\sqrt{\frac{2T}{n}} \sum_i h_i (\vphi_i(\sigma^1)+\vphi_i(\sigma^2)+ \vphi_i(\sigma^3)+\vphi_i(\sigma^4))\Big)\Big\rangle\\
&= \Big\langle \exp\Big(\frac{\beta^2T}{n} \sum_i (\vphi_i(\sigma^1)+\vphi_i(\sigma^2)+ \vphi_i(\sigma^3)+\vphi_i(\sigma^4))^2\Big)\Big\rangle
\approx \e^{4\beta^2T}.
}
In summary, if $\langle \RR_{1,2}\rangle \approx 0$, then
\[
\Var_{\vc h} \langle \e^{\beta B_t}\rangle = \E_{\vc h} \langle \e^{\beta B_t}\rangle^2 - (\E_{\vc h} \langle \e^{\beta B_t}\rangle)^2 \approx 0,
\]
and thus, with high probability, 
\eeq{ \label{gibbs_t_2}
 \langle \e^{\beta B_t}\rangle\approx \E_{\vc h} \langle \e^{\beta B_t}\rangle \approx \e^{2\beta^2T}.
}
By following exactly the same steps with $\langle \RR_{1,2} \e^{\beta B_t}\rangle$ instead of $\langle \e^{\beta B_t}\rangle$, we show that 
\eeq{ \label{gibbs_t_3}
\langle \RR_{1,2} \e^{\beta B_t}\rangle \approx \langle \RR_{1,2} \rangle \e^{2\beta^2T}.
}
Combining \eqref{gibbs_t_1}--\eqref{gibbs_t_3}, we conclude that if $\langle \RR_{1,2}\rangle \approx 0$, then $\langle \RR_{1,2}\rangle_t \approx \langle \RR_{1,2}\rangle \approx 0$.

\subsection{{Proof sketch of Theorem~\EOref}}\label{sketchsec2}
We begin this proof sketch where the previous section left off, namely the observation that if the average overlap $\langle\RR_{1,2}\rangle$ in environment $\vc g$ is small, then Gibbs averages of the type in \eqref{gibbs_t_2} and \eqref{gibbs_t_3} are well concentrated.
By the same type of argument --- see Lemma~\ref{h_variance_lemma}(b) and \eqref{upper_X3} --- we can say something more general: no matter the size of $\langle\RR_{1,2}\rangle$, these averages remain concentrated so as long as they are restricted to the set $\AA_{n,\delta}$ defined in \eqref{A_def}, where \textit{conditional} average overlap $\givena{\RR_{1,2} }{\sigma^1}$ is small.
That is, if $\wt H_n$ is an independent Hamiltonian (i.e.~defined with $\vc h$, an independent copy of $\vc g$), then with high probability,
\eeq{ \label{small_fluctuations_restricted}
\langle\one_{\AA_{n,\delta}}\e^{\frac{\beta}{\sqrt{n}}\wt H_n(\sigma)}\rangle \approx \E_{\vc h}  \langle\one_{\AA_{n,\delta}}\e^{\frac{\beta}{\sqrt{n}}\wt H_n(\sigma)}\rangle\stackref{variance_assumption}{=} \e^{\frac{\beta^2}{2}}\langle \one_{\AA_{n,\delta}}\rangle.
}
In fact, the opposite is true off of the set $\AA_{n,\delta}$. 
If $\langle \RR_{1,2}\rangle$ is not too small relative to $\delta$, then the fluctuations of $\langle\one_{\AA_{n,\delta}^\cc}\e^{\frac{\beta}{\sqrt{n}}\wt H_n(\sigma)}\rangle$ due to $\vc h$ are $\Omega(1)$ as $n\to\infty$.
This is again an elementary calculation; see \eqref{lower_X4_prep}--\eqref{lower_X4_almost}.

On the other hand, a convenient consequence of Gaussianity is that $H_n + \frac{1}{\sqrt{n}}\wt H_n \stackrel{\dd}{=} \sqrt{1+\frac{1}{n}}H_n$.
That is, an environment perturbation is equivalent in distribution to a temperature perturbation.
(In fact, this simple observation underlies the Aizenman--Contucci identities \cite{aizenman-contucci98}, the predecessor of the Ghirlanda--Guerra identities.)
Therefore, if we keep track of the dependence on $\beta$ by writing $\langle \cdot\rangle_\beta$, and abbreviate $\AA_{n,\delta}$ to $\AA_{\delta}$, we have
\eeq{ \label{sketch_2_1}
\langle\one_{\AA_{\delta}}\rangle_{\beta\sqrt{1+\frac{1}{n}}} 
\stackrel{\dd}{=}\frac{\langle\one_{\AA_{\delta}}\e^{\frac{\beta}{\sqrt{n}}\wt H_n(\sigma)}\rangle_\beta}{\langle \e^{\frac{\beta}{\sqrt{n}}\wt H_n(\sigma)}\rangle_\beta}.
}
By rewriting the denominator in a trivial way and using our observation \eqref{small_fluctuations_restricted}, we see that with high probability, 
\eeq{ \label{sketch_2_2}
\frac{\langle\one_{\AA_{\delta}}\e^{\frac{\beta}{\sqrt{n}}\wt H_n(\sigma)}\rangle_\beta}{\langle \e^{\frac{\beta}{\sqrt{n}}\wt H_n(\sigma)}\rangle_\beta} 
&= \frac{\langle\one_{\AA_{\delta}}\e^{\frac{\beta}{\sqrt{n}}\wt H_n(\sigma)}\rangle_\beta}{\langle\one_{\AA_{\delta}}\e^{\frac{\beta}{\sqrt{n}}\wt H_n(\sigma)}\rangle_\beta+\langle\one_{\AA^\cc_{\delta}}\e^{\frac{\beta}{\sqrt{n}}\wt H_n(\sigma)}\rangle_\beta}  \\
&\approx \frac{\e^{\frac{\beta^2}{2}}\langle\one_{\AA_{\delta}}\rangle_\beta}{\e^{\frac{\beta^2}{2}}\langle\one_{\AA_{\delta}}\rangle_\beta+\langle\one_{\AA^\cc_{\delta}}\e^{\frac{\beta}{\sqrt{n}}\wt H_n(\sigma)}\rangle_\beta}.
}
In the last expression above, the only term depending on $\vc h$ is the second summand in the denominator.
Therefore, Jensen's inequality gives
\eeq{ \label{sketch_2_3}
\E_\vc h\bigg[\frac{\e^{\frac{\beta^2}{2}}\langle\one_{\AA_{\delta}}\rangle_\beta}{\e^{\frac{\beta^2}{2}}\langle\one_{\AA_{\delta}}\rangle_\beta+\langle\one_{\AA^\cc_{\delta}}\e^{\frac{\beta}{\sqrt{n}}\wt H_n(\sigma)}\rangle_\beta}\bigg]
&> \frac{\e^{\frac{\beta^2}{2}}\langle\one_{\AA_{\delta}}\rangle_\beta}{\e^{\frac{\beta^2}{2}}\langle\one_{\AA_{\delta}}\rangle_\beta+\E_{\vc h}\langle\one_{\AA^\cc_{\delta}}\e^{\frac{\beta}{\sqrt{n}}\wt H_n(\sigma)}\rangle_\beta} \\
&=\frac{\e^{\frac{\beta^2}{2}}\langle\one_{\AA_{\delta}}\rangle_\beta}{\e^{\frac{\beta^2}{2}}\langle\one_{\AA_{\delta}}\rangle_\beta+\e^{\frac{\beta^2}{2}}\langle\one_{\AA^\cc_{\delta}}\rangle_\beta} 
= \langle \one_{\AA_{\delta}}\rangle_\beta.
}
A more careful analysis shows that the Jensen gap is large enough that we can replace the lower bound by $(1+\gamma)\langle\one_{\AA_{\delta}}\rangle_\beta - C\sqrt{\delta}$, where $\gamma$ and $C$ are positive constants.
One important caveat is that this stronger lower bound is valid only when $\langle\RR_{1,2}\rangle$ is not too small (so that the fluctuations of $\langle\one_{\AA^\cc_{\delta}}\e^{\frac{\beta}{\sqrt{n}}\wt H_n(\sigma)}\rangle_\beta$ are order $1$), which is why Theorem~\ref{averages_squared} is needed beforehand.
Reading \eqref{sketch_2_1}--\eqref{sketch_2_3} from start to end, we obtain
\eeq{ \label{sketch_2_4}
\E\langle\one_{\AA_\delta}\rangle_{\beta\sqrt{1+\frac{1}{n}}} \geq (1+\gamma)\E\langle\one_{\AA_\delta}\rangle_\beta - C\sqrt{\delta}.
}
While the above inequality is the most important step of the proof, a key shortcoming is that the set $\AA_\delta$ is defined using $\langle\cdot\rangle_\beta$ rather than $\langle\cdot\rangle_{\beta\sqrt{1+\frac{1}{n}}}$.
Since we will want to apply the inequality iteratively, we need to replace $\AA_\delta$ on the left-hand side by $\AA_{\delta,1}$, where
\eq{
\AA_{\delta,k} \coloneqq \Big\{\sigma\in\Sigma_n : \frac{1}{n}\sum_i\vphi_i(\sigma)\langle\vphi_i\rangle_{\beta\sqrt{1+\frac{k}{n}}}\leq\delta\Big\}, \quad k = 0,1,2,\dots
}
To make this replacement, we produce a complementary inequality, again using the equivalence of environment/temperature perturbations.
For simplicity, let us assume $\RR_{1,2}\geq0$, which is essentially realized by \eqref{positive_overlap} for large $n$.
Observe that
\eq{
\givena{\RR_{1,2}}{\sigma^1}_{\beta\sqrt{1+\frac{1}{n}}} 
&\stackrel{\dd}{=} \frac{\givena{\RR_{1,2}\e^{\frac{\beta}{\sqrt{n}}\wt H_n(\sigma^2)}}{\sigma^1}_\beta}{\langle\e^{\frac{\beta}{\sqrt{n}}\wt H_n(\sigma)}\rangle_\beta}  
\leq \sqrt{\givena{\RR_{1,2}}{\sigma^1}_\beta}\underbrace{\sqrt{\langle \e^{\frac{2\beta}{\sqrt{n}}\wt H_n(\sigma)}\rangle_\beta}\langle\e^{\frac{-\beta}{\sqrt{n}}\wt H_n(\sigma)}\rangle_\beta}_{X},
}
where we have applied Cauchy--Schwarz (and then $\RR_{1,2}^2 \leq \RR_{1,2}\leq1$) and Jensen's inequality (using the convexity of $x\mapsto x^{-1}$).
When $\sigma^1 \in \AA_{\delta}=\AA_{\delta,0}$, the final expression is at most $X\sqrt{\delta}$, and so the inequality implies $\AA_{\delta,0} \subset \AA_{X\sqrt{\delta},1}$.
Now, the random variable $X$ has moments of all orders (admitting simple upper bounds), and so it can be essentially regarded as a large constant.
In particular, when $\delta$ is small, we will have $X \leq \delta^{-1/4}$ with high probability, in which case
$\AA_{\delta,0} \subset \AA_{\delta^{1/4},1}$.
Combining these ideas with \eqref{sketch_2_4}, we show
\eq{
\E\langle \one_{\AA_{\delta^{1/4},1}}\rangle_{\beta\sqrt{1+\frac{1}{n}}} \geq (1+\gamma)\E\langle\one_{\AA_\delta}\rangle_\beta - C\sqrt{\delta}.
}
More generally, for any integer $k\geq1$,
\eeq{ \label{ineq_to_be_iterated}
\E\langle \one_{\AA_{\delta^{1/4},k}}\rangle_{\beta\sqrt{1+\frac{k}{n}}} \geq (1+\gamma)\E\langle\one_{\AA_{\delta,k-1}}\rangle_{\beta\sqrt{1+\frac{k-1}{n}}} - C\sqrt{\delta}.
}
This inequality can now be iterated, with $\delta$ being replaced by $\delta^{1/4}$, then $\delta^{1/16}$, and so on, as the expectation on the left is inserted on the right in the next iteration.

Since the left-hand side of \eqref{ineq_to_be_iterated}  is always at most $1$, we clearly obtain a contradiction if $\E\langle\one_{\AA_{\delta,0}}\rangle_\beta$ is larger than $x$, where $x$ is the solution to $x = (1+\gamma)x - C\sqrt{\delta}$. 
This would complete the proof of Theorem~\ref{expected_overlap_thm} if not for the subtlety that $\gamma$ actually depends on $k$ in a non-trivial way.
Nevertheless, \eqref{ineq_to_be_iterated} can still be used to derive a contradiction of the same spirit unless 
$\E\langle\one_{\AA_{\delta^{1/4^k},k}}\rangle$ is small for some $k\leq K$, where $K$ is large and tends to infinity as $\eps\to0$, but crucially does not depend on $n$.
This approach is reminiscent of tower-type arguments in extremal combinatorics.

Replacing $\delta$ by $\delta^{4^k}$, we can then say $\E\langle\one_{\AA_{\delta,k}}\rangle$ is small.
Finally, to deduce the smallness of $\E\langle \one_{\AA_{\delta,0}}\rangle$ from the smallness of $\E\langle\one_{\AA_{\delta,k}}\rangle$, we make use of standard arguments showing that if an event is rare at inverse temperature $\beta$, then it remains rare at inverse temperature $\beta + O(1/n)$.

\subsection{{Proof sketch of Theorem \ECref}}
To deduce Theorem~\ref{easy_cor} from Theorem~\ref{expected_overlap_thm}, simply let $\sigma^1,\ldots,\sigma^k, \sigma^{k+1}$ be i.i.d.~draws from the Gibbs measure. Then by the law of large numbers, when $k$ is large,
\[
\frac{1}{k}\sum_{j=1}^k \RR_{j, k+1} \approx \RR(\sigma^{k+1})
\]
with high probability. 
But by Theorem~\ref{expected_overlap_thm}, we know that with high probability, $\RR(\sigma^{k+1})$ is not close to zero. Therefore, with high probability, there must exist $1\le j\le k$ such that $\RR_{j,k+1}$ is not close to zero.

\section{General preliminaries} \label{prep_section}
In this preliminary section, we record several facts needed in the proofs of Theorems~\ref{expected_overlap_thm} and~\ref{averages_squared}.
These preparatory results are mostly elementary.


\subsection{The Gibbs measure and partition function} \label{gibbs_measure}
 In order for our results to apply to a broad collection of models, we have allowed the state
space $\Sigma_n$ to be completely general, and the Hamiltonian $H_n$ to consist of countably infinite summands. 
We begin by checking that these assumptions pose no issues to computation. 
So for the remainder of Section~\ref{gibbs_measure}, we fix the value of $n$. 
 
Let $\langle\cdot\rangle_N$ denote expectation with respect to the Gibbs measure when the Hamiltonian is replaced by the finite sum $H_{n,N}\coloneqq\sum_{i=1}^Ng_i\vphi_i$.
That is,
\eeq{ \label{finite_gibbs_expectation}
\langle f(\sigma)\rangle_N = \frac{E_n(f(\sigma)\e^{\beta H_{n,N}(\sigma)})}{E_n(\e^{\beta H_{n,N}(\sigma)})}.
}
So that we can pass from $\langle \cdot\rangle_N$ to $\langle\cdot\rangle$,  we begin with the following lemma. 

\begin{lemma} \label{exp_moments_lemma}
For all $\beta\in\R$ and any $f\in L^2(\Sigma_n)$, the following limits hold almost surely and in $L^\alpha$ for any $\alpha\in[1,\infty)$:
\begin{subequations}
\begin{align} 
\lim_{N\to\infty} \langle f(\sigma)\rangle_N &= \langle f(\sigma)\rangle < \infty, \label{exp_moments_lemma_a}\\
\lim_{N\to\infty} \langle H_{n,N}(\sigma)\rangle_N &= \langle H_n(\sigma)\rangle < \infty. \label{exp_moments_lemma_b}
\end{align}
\end{subequations}
\end{lemma}


\begin{proof}
We organize the proof into a sequence of claims.
\begin{claim} \label{claim_1}
With $\P$-probability equal to $1$,
\eq{
\lim_{N\to\infty} H_{n,N}(\sigma) = H_n(\sigma) \quad \text{for $P_n$-$\mathrm{a.e.}$~$\sigma\in\Sigma_n$}.
}
\end{claim}

\begin{proof} Observe that for fixed $\sigma\in\Sigma_n$, the
sequence $(H_{n,N} (\sigma))_{N\geq0}$ is a martingale with respect to $\P$. 
Since
\eq{
\sup_{N\geq0}\E[H_{n,N}(\sigma)^2] = \sup_{N\geq0}\sum_{i=1}^N \vphi_i(\sigma)^2\stackrel{\mbox{\footnotesize\eqref{variance_assumption},\eqref{field_decomposition}}}{=} n,
}
the martingale convergence theorem guarantees that $H_{n,N}(\sigma)$ converges $\P$-almost surely as $N\to\infty$ to a limit we call $H_n(\sigma)$.
Now Fubini's theorem proves the claim:
\eq{
\E E_n(\one_{\{H_{n,N}(\sigma)\to H_n(\sigma)\}}) = E_n(\E[\one_{\{H_{n,N}(\sigma)\to H_n(\sigma)\}}]) = E_n(1) = 1.
}
\end{proof}

\begin{claim} \label{claim_2}
There exist nonnegative random variables $(M^+(\sigma))_{\sigma\in\Sigma_n}$ and $(M^-(\sigma))_{\sigma\in\Sigma_n}$ such that
\eeq{ \label{every_N}
\pm H_{n,N}(\sigma) \leq M^\pm (\sigma) \quad \text{for all $N\geq0,\, \sigma\in\Sigma_n$},
}
and
\eeq{ \label{good_denominator}
\E E_n(\e^{\beta M^\pm(\sigma)}) < \infty \quad \text{for all $\beta\geq0$}.
}
\end{claim}

\begin{proof}
We simply take
\eq{
M^\pm(\sigma) \coloneqq \sup_{N\geq0} \pm H_{n,N}(\sigma)\geq \pm H_{n,0}(\sigma) = 0,
}
so that \eqref{every_N} is satisfied by definition.
Since $M^+ \stackrel{\text{d}}{=} M^-$, we need only check \eqref{good_denominator} for $M^+$. 
Observe that for any $\beta \geq 0$, $(\e^{\beta H_{n,N}(\sigma)})_{N\geq0}$ is a submartingale.
By Doob's inequality, for any $\lambda > 0$ and any integer $m\geq0$,
\eq{
\P\Big(\max_{0\leq N\leq m} \e^{\beta H_{n,N}(\sigma)}\geq\lambda\Big)
&= \P\Big(\max_{0\leq N\leq m} \e^{2\beta H_{n,N}(\sigma)}\geq\lambda^2\Big)  \\
&\leq \lambda^{-2}\E(\e^{2\beta H_{n,m}(\sigma)})
= \lambda^{-2}\e^{2\beta^2\sum_{i=1}^m\vphi_i^2(\sigma)}\stackrel{\mbox{\footnotesize\eqref{variance_assumption}}}{\leq} \lambda^{-2}\e^{2\beta^2n}.
}
Therefore, for any $0 < \eps<\lambda $, 
\eq{
\P(\e^{\beta M^+(\sigma)}\geq\lambda) &\leq \P\Big(\e^{\beta M^+(\sigma)}\geq\lambda-\frac{\eps}{2}\Big) 
\leq \lim_{m\to\infty} \P\Big(\max_{0\leq N\leq m} \e^{\beta H_{n,N}(\sigma)}\geq\lambda-\eps\Big) \leq (\lambda-\eps)^{-2}\e^{2\beta^2n},
}
which implies
\eq{ 
\E(\e^{\beta M^+(\sigma)}) &= \int_0^\infty \P(\e^{\beta M^+(\sigma)}\geq\lambda)\ \dd\lambda 
\leq 1+\eps + \e^{2\beta^2n}\int_{1+\eps}^\infty (\lambda-\eps)^{-2}\ \dd\lambda < \infty.
}
Since Tonelli's theorem gives $\E E_n(\e^{\beta M^+(\sigma)}) = E_n(\E \e^{\beta M^+(\sigma)})$, \eqref{good_denominator} follows from the above display.
\end{proof}

\begin{claim} \label{claim_3}
For any $f\in L^2(\Sigma_n)$ and any continuous function $\phi : \R \to \R$ such that $|\phi(x)| \leq a\e^{b|x|}$ for all $x\in\R$, for some $a,b\geq0$, we have
\eeq{ \label{claim_3_eq}
\lim_{N\to\infty} E_n[f(\sigma)\phi(H_{n,N}(\sigma))] = E_n[f(\sigma)\phi(H_n(\sigma))] \quad \mathrm{a.s.}
}
\end{claim}

\begin{proof}
By Claim~\ref{claim_1} and the continuity of $\phi$, we almost surely have that $\phi(H_{n,N}(\sigma)) \to \phi(H_{n}(\sigma))$ for $P_n$-a.e.~$\sigma\in\Sigma_n$, as $N\to\infty$.
And by hypothesis,
\eeq{ \label{exponential_domination}
|\phi(H_{n,N}(\sigma))| \leq a(\e^{bM^+(\sigma)} + \e^{bM^-(\sigma)}).
}
Since
\eq{
E_n\big[|f(\sigma)|(\e^{b M^+(\sigma)} + \e^{bM^-(\sigma)})\big] 
&\leq \sqrt{ E_n[f(\sigma)^2] E_n[(\e^{b M^+(\sigma)} + \e^{bM^-(\sigma)})^2]} \\
&\leq \sqrt{ E_n[f(\sigma)^2] E_n[2(\e^{2b M^+(\sigma)} + \e^{2bM^-(\sigma)})]},
}
and Claim~\ref{claim_2} implies that almost surely $E_n(\e^{2b M^\pm(\sigma)}) < \infty$,
\eqref{claim_3_eq} now follows from dominated convergence (with respect to $P_n$).
\end{proof}

\begin{claim} \label{claim_4}
For any $f\in L^2(\Sigma_n)$ and any continuous function $\phi : \R \to \R$ such that $|\phi(x)| \leq a\e^{b|x|}$ for all $x\in\R$, for some $a,b\geq0$, we have
\eeq{ \label{claim_4_eq}
\lim_{N\to\infty} \langle f(\sigma)\phi(H_{n,N}(\sigma))\rangle_N = \langle f(\sigma)\phi(H_n(\sigma))\rangle \quad \mathrm{a.s.}\text{ and in } L^\alpha, \alpha\in[1,\infty).
}
\end{claim}

\begin{proof}
Recall that
\eq{
\langle f(\sigma)\phi(H_{n,N}(\sigma))\rangle_N &= \frac{E_n[f(\sigma)\phi(H_{n,N}(\sigma))\e^{\beta H_{n,N}(\sigma)}]}{E_n(\e^{\beta H_{n,N}(\sigma)})}, \\
\langle f(\sigma)\phi(H_n(\sigma))\rangle &= \frac{E_n[f(\sigma)\phi(H_{n}(\sigma))\e^{\beta H_{n}(\sigma)}]}{E_n(\e^{\beta H_{n}(\sigma)})}.
}
Since $|\phi(x)|\e^{\beta x} \leq  a\e^{(b+\beta)|x|}$, the almost sure part of \eqref{claim_4_eq} is immediate from Claim~\ref{claim_3}.
The convergence in $L^\alpha$ is then a consequence of dominated convergence (with respect to $\P$).
Indeed, by Cauchy--Schwarz and Jensen's inequality, we have the majorization 
\eq{
|\langle f(\sigma)\phi(H_{n,N}(\sigma))\rangle_N|
&\stackrefp{exponential_domination}{=} \frac{|E_n(f(\sigma)\phi(H_{n,N}(\sigma))\e^{\beta H_{n,N}(\sigma)})|}{E_n(\e^{\beta H_{n,N}(\sigma)})} \\
&\stackrel{\phantom{\mbox{\footnotesize\eqref{exponential_domination}}}}{\leq} \frac{\sqrt{E_n(f(\sigma)^2)E_n(\phi(H_{n,N}(\sigma))^2\e^{2\beta H_{n,N}(\sigma)})}}{E_n(\e^{-\beta M^-(\sigma)})} \\
&\stackrel{\mbox{\footnotesize\eqref{exponential_domination}}}{\leq} \sqrt{ E_n(f(\sigma)^2)  E_n[2a^2(\e^{2(b+\beta)M^+(\sigma)} + \e^{2(b+\beta)M^-(\sigma)})]}E_n(\e^{\beta M^-(\sigma)}),
}
where the final expression has moments of all orders by \eqref{good_denominator}.
\end{proof}

We now complete the proof of Lemma~\ref{exp_moments_lemma} by taking $\phi \equiv 1$ for \eqref{exp_moments_lemma_a}, and $f\equiv1$, $\phi(x) = x$ for \eqref{exp_moments_lemma_b}.

\end{proof}

\begin{remark} The essential feature of the above proof was checking in Claim~\ref{claim_2} that \eqref{variance_assumption} is enough
to guarantee the first equality below:
\eeq{ \label{freq_identity}
\E(\e^{\beta\sum_{i=1}^\infty g_i\vphi_i}) = \lim_{N\to\infty} \E(\e^{\beta\sum_{i=1}^Ng_i\vphi_i})
= \lim_{N\to\infty} \e^{\frac{\beta^2}{2}\sum_{i=1}^N\vphi_i^2}
\stackrel{\mbox{\footnotesize\eqref{variance_assumption}}}{=} \e^{\frac{\beta^2}{2}n}.
}
\end{remark}
We will frequently use the above identity, an easy consequence of which is
the following.

\begin{lemma} \label{moments_lemma}
For any $\beta\in\R$, we have
\eeq{ \label{first_moment}
\E Z_n(\beta) = \e^{\frac{\beta^2}{2}n},
}
as well as
\eeq{ \label{negative_first_moment}
\E[Z_n(\beta)^{-1}] \leq \e^{\frac{\beta^2}{2}n}.
}
\end{lemma}

\begin{proof}
By exchanging the order of expectation in the identity $\E Z_n(\beta) = \E[E_n(\e^{\beta H_n(\sigma)})]$
(which we are permitted to do by Tonelli's theorem) and applying \eqref{freq_identity}, we
obtain \eqref{first_moment}. 
For \eqref{negative_first_moment}, we apply Jensen's inequality to obtain
\eq{
Z_n(\beta)^{-1} = [E_n(\e^{\beta H_n(\sigma)})]^{-1} \leq E_n(\e^{-\beta H_n(\sigma)}),
}
then take expectation $\E(\cdot)$ of both sides, and again exchange the order of
expectation. 
\end{proof}

Let us also record two consequences of Lemma~\ref{exp_moments_lemma} that will be needed later in the chapter.

\begin{cor}
For any $\beta\in\R$, the following limits hold almost surely and in $L^\alpha$ for any $\alpha\in[1,\infty)$:
\eeq{ \label{limit_for_later}
\lim_{N\to\infty} \sum_{i=1}^N \langle \vphi_i^2\rangle_N &= n \qquad \text{and} \qquad
\lim_{N\to\infty} \sum_{i=1}^N \langle\vphi_i\rangle^2_N = \sum_{i=1}^\infty \langle\vphi_i\rangle^2.
}
\end{cor}

\begin{proof}
First we argue the almost sure statements.
The $L^\alpha$ statements will then follow from bounded convergence, since \eqref{variance_assumption} gives the uniform bound
\eq{
0\leq\sum_{i=1}^N \langle \vphi_i\rangle_N^2 \leq \sum_{i=1}^N \langle \vphi_i^2\rangle_N \leq n \quad \text{for every $N$.}
}
So we fix the disorder $\vc g$.
By Lemma~\ref{exp_moments_lemma}, it is almost surely the case that for every $i\geq1$, $\langle \vphi_i\rangle_N \to \langle \vphi_i\rangle$ and $\langle \vphi_i^2\rangle_N \to \langle \vphi_i^2\rangle$ as $N\to\infty$.
We also know $\sum_{i=1}^\infty \vphi_i^2 = n$.
In particular, given $\eps>0$, we can choose $M$ so large that
\eq{
n-\eps \leq \sum_{i=1}^M \langle \vphi_i^2\rangle \leq n \quad \implies \quad
\sum_{i=M+1}^\infty \langle \vphi_i^2\rangle \leq \eps.
}
Given $M$, there is $N_0$ such that for all $N\geq N_0$,
\eq{
\bigg|\sum_{i=1}^M (\langle \vphi_i^2\rangle_N -  \langle\vphi_i^2\rangle)\bigg| \leq \eps \qquad \text{and} \qquad
\bigg|\sum_{i=1}^M (\langle \vphi_i\rangle_N^2 -  \langle\vphi_i\rangle^2)\bigg| \leq \eps.
}
In particular, for all $N\geq N_0 \vee M$,
\eq{
n-2\eps\leq\sum_{i=1}^M \langle \vphi_i^2\rangle_N \leq n 
\quad &\implies \quad
n-2\eps\leq\sum_{i=1}^N \langle \vphi_i^2\rangle_N \leq n,
}
and also
\eq{
\bigg|\sum_{i=1}^N \langle \vphi_i\rangle_N^2 - \sum_{i=1}^\infty \langle\vphi_i\rangle^2\bigg|
&\leq \bigg|\sum_{i=1}^M (\langle \vphi_i\rangle_N^2 - \langle\vphi_i\rangle^2)\bigg| + \sum_{i=M+1}^\infty (\langle\vphi_i\rangle_N^2 + \langle\vphi_i\rangle^2) \\
&\leq \bigg|\sum_{i=1}^M (\langle \vphi_i\rangle_N^2 - \langle\vphi_i\rangle^2)\bigg| + \sum_{i=M+1}^\infty (\langle\vphi_i^2\rangle_N + \langle\vphi_i^2\rangle) 
\leq 4\eps.
}
\end{proof}

\subsection{Derivative of free energy} \label{derivatives_section}
This section records some important facts regarding convergence of the free energy's derivative.
%
By Lemma~\ref{exp_moments_lemma},
 it is almost surely the case that the random variable
$H_n(\sigma)$ has exponential moments of all orders with respect to $P_n$.
Standard
calculations then show that the free energy $F_n(\beta) = \frac{1}{n}\log Z_n(\beta)$ satisfies
\eeq{ \label{nrg_2deriv}
F_n'(\beta) = \frac{\langle H_n(\sigma)\rangle}{n} \quad \text{and} \quad
F_n''(\beta) = \frac{\langle H_n(\sigma)^2\rangle - \langle H_n(\sigma)\rangle^2}{n} \quad \mathrm{a.s.}
}
Recall from \eqref{free_energy_assumption} that $F_n(\beta) \to p(\beta)$. 
Since $F_n(\cdot)$ is convex for every $n$, $p(\cdot)$ is necessarily convex.
This assumption implies the following lemma, which is a general fact about the convergence of convex functions.

\begin{lemma} \label{betas_converging}
If $p(\cdot)$ is differentiable at $\beta$, and $\beta_n = \beta + \delta(n)$ with $\delta(n) \to 0$ as $n\to\infty$, then
\eq{
\lim_{n\to\infty} F_n'(\beta_n) = p'(\beta) \quad \mathrm{a.s.}\text{ and in }L^1.
}
\end{lemma}

\begin{proof}
Let $\eps > 0$.
By differentiability, we can choose $h > 0$ sufficiently small that
\eeq{ \label{eps_1}
p'(\beta) - \eps \leq \frac{p(\beta)-p(\beta-h)}{h} \leq \frac{p(\beta+h)-p(\beta)}{h} \leq p'(\beta) + \eps,
}
where the middle inequality is due to convexity.
Given $h$, we next choose $\delta > 0$ such that
\begin{subequations}
\begin{align}
0 \leq \frac{p(\beta+\delta+h) - p(\beta+\delta)}{h} - \frac{p(\beta+h) - p(\beta)}{h} \leq \eps \label{eps_2a},
\intertext{as well as} 
0 \leq \frac{p(\beta) - p(\beta-h)}{h} - \frac{p(\beta-\delta)-p(\beta-\delta-h)}{h} \leq \eps, \label{eps_2b}
\end{align}
\end{subequations}
which is possible by the continuity of $p(\cdot)$.
Now, convexity of $F_n$ implies the following for all $n$ such that $\delta(n) \leq \delta$:
\begin{subequations}
\begin{align}
F_n'(\beta_n) \leq \frac{F_n(\beta+\delta(n)+h)-F_n(\beta+\delta(n))}{h} &\leq \frac{F_n(\beta+\delta+h)-F_n(\beta+\delta)}{h}. \label{eps_3a}
\intertext{Similarly, for all $n$ such that $\delta(n) \geq -\delta$,}
F_n'(\beta_n) &\geq\frac{F_n(\beta-\delta)-F_n(\beta-\delta-h)}{h}.
\label{eps_3b} 
\end{align} 
\end{subequations}
Upon defining
\eeq{ \label{Delta_def}
\Delta_n^-(\beta,h) &\coloneqq \frac{F_n(\beta)-F_n(\beta-h)}{h} - \frac{p(\beta)-p(\beta-h)}{h}, \\
\Delta_n^+(\beta,h) &\coloneqq \frac{F_n(\beta+h)-F_n(\beta)}{h} - \frac{p(\beta+h)-p(\beta)}{h},
}
it follows that for all sufficiently large $n$, 
\eq{
F_n'(\beta_n) - p'(\beta)
&\stackrel{\mbox{\hspace{3ex}\footnotesize\eqref{eps_3a}\hspace{3ex}}}{\leq} \frac{F_n(\beta+\delta+h)-F_n(\beta+\delta)}{h} - p'(\beta) \\
&\stackrel{\mbox{\footnotesize\eqref{eps_1},\eqref{eps_2a}}}{\leq} \Delta_n^+(\beta+\delta,h) + 2\eps.
}
Analogously, \eqref{eps_1}, \eqref{eps_2b}, and \eqref{eps_3b} together yield the lower bound
\eq{
F_n'(\beta_n) - p'(\beta) \geq \Delta_n^-(\beta-\delta,h) - 2\eps.
}
By \eqref{free_energy_assumption},  both $\Delta_n^-(\beta-\delta,h)$ and $\Delta_n^+(\beta+\delta,h)$ tend to $0$ almost surely and in $L^1$ as $n\to\infty$.
As $\eps$ is arbitrary, the desired result follows.
\end{proof}

\begin{cor} \label{overlap_identity_cor}
For every $\beta\geq0$ at which $p(\cdot)$ is differentiable, 
\eeq{ \label{overlap_ito_derivative}
p'(\beta) = \beta\big(1 - \lim_{n\to\infty} \E \langle \RR_{1,2}\rangle\big).
}
In particular, $0 \leq p'(\beta) \leq \beta$, and there is thus some $\beta_\cc \in [0,\infty]$ such that
\eq{
0 \leq \beta \leq \beta_\cc \quad &\implies \quad p(\beta) = \frac{\beta^2}{2}, \\
\beta>\beta_\cc \quad &\implies \quad p(\beta) < \frac{\beta^2}{2}.
}
\end{cor}

\begin{proof}
Using the notation of Lemma~\ref{exp_moments_lemma}, we have
\eq{
\E F_n'(\beta) 
\stackrel{\mbox{\footnotesize\eqref{nrg_2deriv}}}{=} \frac{\E\langle H_n(\sigma)\rangle}{n} 
\stackrel{\mbox{\footnotesize\eqref{exp_moments_lemma_b}}}{=}\lim_{N\to\infty} \frac{\E\langle H_{n,N}(\sigma)\rangle_N}{n} 
&=\lim_{N\to\infty} \E\Big\langle\frac{1}{n} \sum_{i=1}^N g_i\vphi_i\Big\rangle_N \\
&=\lim_{N\to\infty}  \frac{1}{n}\sum_{i=1}^N \E[g_i\langle\vphi_i\rangle_N].
}
By Gaussian integration by parts,
\eq{
\E[g_i\langle\vphi_i\rangle_N] = \E\Big[\frac{\partial}{\partial g_i}\langle\vphi_i\rangle_N\Big] = \beta\E[\langle\vphi_i^2\rangle_N - \langle\vphi_i\rangle_N^2],
}
and then Lemma~\ref{betas_converging} allows us to write
\eq{
p'(\beta) = \lim_{n\to\infty} \E F_n'(\beta) 
&\stackrel{\phantom{\mbox{\footnotesize\eqref{limit_for_later}}}}{=} \lim_{n\to\infty} \lim_{N\to\infty} \beta\E\bigg[\frac{1}{n}\sum_{i=1}^N (\langle\vphi_i^2\rangle_N - \langle\vphi_i\rangle_N^2)\bigg] \\
&\stackrel{\mbox{\footnotesize\eqref{limit_for_later}}}{=} \lim_{n\to\infty} \beta\E\bigg[1-\frac{1}{n}\sum_{i=1}^\infty \langle\vphi_i\rangle^2\bigg]
=\lim_{n\to\infty} \beta(1 - \E\langle \RR_{1,2}\rangle),
}
which completes the proof of \eqref{overlap_ito_derivative}.
The inequalities $0\leq p'(\beta)\leq\beta$ now follow from
\eq{
1 \stackref{variance_assumption}{\geq} \lim_{n\to\infty}\E\langle\RR_{1,2}\rangle \stackref{positive_overlap}{\geq} -\lim_{n\to\infty}\EEE_n = 0.
}

For the second part of the claim, we recall that $p(\cdot)$ is convex and thus absolutely continuous.
Since $p(0) = 0$, we then have
\eq{
\frac{\beta^2}{2} - p(\beta) = \int_0^{\beta} [t - p'(t)]\ \dd t.
}
Since the integrand is nonnegative, it follows that $\beta\mapsto \beta^2/2 - p(\beta)$ is non-decreasing for $\beta \geq 0$.
\end{proof}
%

So that we can be explicit in the inverse
temperature parameter $\beta$, for the remainder of the section we will write $\langle\cdot\rangle_\beta$
for expectation with respect to $\mu_{n}^\beta$.
In light of \eqref{nrg_2deriv}, Lemma~\ref{betas_converging} implies
\eq{
\lim_{n\to\infty} \Big|\frac{\langle H_n(\sigma)\rangle_\beta}{n} - p'(\beta)\Big| = 0 \quad \mathrm{a.s.}\quad \text{whenever $p'(\beta)$ exists}.
}
We will require the following stronger form of this result, which also appears in 
\cite[Theorem 3]{auffinger-chen18}.
Our proof is adapted from the elegant approach of \cite{panchenko10}, and included for completeness.

\begin{lemma} \label{lemma:step_1_2}
If $\beta$ is a point of differentiability for $p(\cdot)$, then
\eq{
\lim_{n\to\infty} \Big\langle\Big|\frac{H_n(\sigma)}{n} - p'(\beta)\Big|\Big\rangle_\beta = 0 \quad \mathrm{a.s.} \text{ and in }L^1.
}
\end{lemma}

\begin{proof}
By Lemma~\ref{betas_converging}, it suffices to show that if $\beta_0$ is a point of differentiability for $p(\cdot)$, then
\eq{
\lim_{n\to\infty} \Big\langle\Big|\frac{H_n(\sigma)}{n} - F_n'(\beta_0)\Big|\Big\rangle_{\beta_0} = 0 \quad \mathrm{a.s.} \text{ and in }L^1.
}
Fix $\eps > 0$ and choose $h>0$ small enough that
\eeq{ \label{h_choice}
p'(\beta_0)-\eps \leq \frac{p(\beta_0)-p(\beta_0-h)}{h} \leq \frac{p(\beta_0+h)-p(\beta_0)}{h} \leq p'(\beta_0)+\eps.
}
Given $h$, differentiability allows us to take $\beta_1>\beta_0$ sufficiently close to $\beta_0$ to satisfy
\eeq{ \label{beta_prime_choice}
\frac{p(\beta_1+h)-p(\beta_1)}{h} \leq \frac{p(\beta_0+h)-p(\beta_0)}{h} + \eps \leq p'(\beta_0)+2\eps.
}
By adding and subtracting $\langle |H_n(\sigma^1)-H_n(\sigma^2)|\rangle_{\beta_0}$, we have
\eeq{ \label{trivial_integrals}
&\int_{\beta_0}^{\beta_1} \langle |H_n(\sigma^1)-H_n(\sigma^2)|\rangle_\beta\ \dd \beta \\
&= (\beta_1-{\beta_0})\langle |H_n(\sigma^1)-H_n(\sigma^2)|\rangle_{\beta_0}
+ \int_{\beta_0}^{\beta_1} \big[\langle|H_n(\sigma^1)-H_n(\sigma^2)|\rangle_\beta - \langle|H_n(\sigma^1)-H_n(\sigma^2)|\rangle_{{\beta_0}}\big]\ \dd \beta \\
&= (\beta_1-{\beta_0})\langle |H_n(\sigma^1)-H_n(\sigma^2)|\rangle_{\beta_0}
+\int_{\beta_0}^{\beta_1} \int_{\beta_0}^\beta \frac{\partial}{\partial x} \langle |H_n(\sigma^1)-H_n(\sigma^2)|\rangle_x\ \dd x\, \dd \beta.
}
A simple calculation, followed by Cauchy--Schwarz, shows
\eq{ 
\Big|\frac{\partial}{\partial x} \langle |H_n(\sigma^1)-H_n(\sigma^2)|\rangle_x\Big|
&= \big|\big\langle |H_n(\sigma^1)-H_n(\sigma^2)| \cdot \big(H_n(\sigma^1)+H_n(\sigma^2)-2H_n(\sigma^3)\big)\big\rangle_x\big| \\
&\leq \sqrt{\big\langle \big(H_n(\sigma^1)-H_n(\sigma^2)\big)^2\rangle_x\big\langle \big(H_n(\sigma^1)+H_n(\sigma^2)-2H_n(\sigma^3)\big)^2\big\rangle_x}.
}
By another application of Cauchy--Schwarz, we have
\eq{
\big\langle \big(H_n(\sigma^1)+H_n(\sigma^2)-2H_n(\sigma^3)\big)^2\big\rangle_x
&= \big\langle \big(H_n(\sigma^1)-H_n(\sigma^3)+H_n(\sigma^2)-H_n(\sigma^3)\big)^2\big\rangle_x \\
&\leq 2\big\langle\big(H_n(\sigma^1)-H_n(\sigma^3)\big)^2\big\rangle_x + 2\big\langle\big(H_n(\sigma^1)-H_n(\sigma^2)\big)^2\big\rangle_x \\
&= 4\big\langle\big(H_n(\sigma^1)-H_n(\sigma^2)\big)^2\big\rangle_x.
}
From the previous two displays, we find
\eq{
\Big|\frac{\partial}{\partial x} \langle |H_n(\sigma^1)-H_n(\sigma^2)|\rangle_x\Big| 
\leq 2\big\langle\big(H_n(\sigma^1)-H_n(\sigma^2)\big)^2\big\rangle_x
&= 4\langle H_n(\sigma)^2\rangle_x-4\langle H_n(\sigma)\rangle_x^2.
}
In light of this inequality, \eqref{trivial_integrals} now shows
\eq{
&\langle |H_n(\sigma^1)-H_n(\sigma^2)|\rangle_{\beta_0}  \\
&\leq \frac{1}{\beta_1-{\beta_0}}\int_{\beta_0}^{\beta_1}\langle|H_n(\sigma^1)-H_n(\sigma^2)|\rangle_\beta\ \dd \beta 
+\frac{4}{\beta_1-{\beta_0}}\int_{\beta_0}^{\beta_1}\int_{\beta_0}^\beta (\langle H_n(\sigma)^2\rangle_x-\langle H_n(\sigma)\rangle_x^2)\ \dd x\, \dd \beta \\
&\leq \frac{2}{\beta_1-{\beta_0}}\int_{\beta_0}^{\beta_1}\big\langle|H_n(\sigma)-\langle H_n(\sigma)\rangle_\beta|\big\rangle_\beta\ \dd \beta 
+4\int_{\beta_0}^{\beta_1} (\langle H_n(\sigma)^2\rangle_x-\langle H_n(\sigma)\rangle_x^2)\ \dd x,
}
where
\eq{
\frac{2}{\beta_1-{\beta_0}}\int_{\beta_0}^{\beta_1}\big\langle|H_n(\sigma)-\langle H_n(\sigma)\rangle_\beta|\big\rangle_\beta\ \dd \beta
&\leq 2\bigg(\frac{1}{\beta_1-{\beta_0}}\int_{\beta_0}^{\beta_1}\big\langle|H_n(\sigma)-\langle H_n(\sigma)\rangle_\beta|\big\rangle_\beta^2\ \dd \beta\bigg)^{1/2} \\
&\leq2\bigg(\frac{1}{\beta_1-{\beta_0}}\int_{\beta_0}^{\beta_1}(\langle H_n(\sigma)^2\rangle_\beta-\langle H_n(\sigma)\rangle_\beta^2)\ \dd \beta\bigg)^{1/2}. \\
}
In summary,
\eeq{ \label{beta_integral_summary}
\Big\langle\Big|\frac{H_n(\sigma)}{n}-F_n'(\sigma)\Big|\Big\rangle_{\beta_0}
&= \Big\langle\Big|\frac{H_n(\sigma)}{n}-\frac{\langle H_n(\sigma)\rangle_{\beta_0}}{n}\Big|\Big\rangle_{\beta_0} \\
&\leq\frac{\langle|H_n(\sigma^1)-H_n(\sigma^2)|\rangle_{\beta_0}}{n}  
\leq 2\sqrt{\frac{I_n(\beta_1)}{n(\beta_1-{\beta_0})}}+4I_n(\beta_1),
}
where
\eq{
I_n(\beta_1) &\coloneqq \frac{1}{n}\int_{\beta_0}^{\beta_1}(\langle H_n(\sigma)^2\rangle_\beta-\langle H_n(\sigma)\rangle_\beta^2)\ \dd \beta 
\stackrel{\mbox{\footnotesize\eqref{nrg_2deriv}}}{=} F_n'(\beta_1) - F_n'(\beta_0).
}
Therefore, convexity of $F_n(\cdot)$ implies
\eq{
I_n(\beta_1) 
&\stackrel{\phantom{\mbox{\footnotesize\eqref{h_choice},\eqref{beta_prime_choice}}}}{\leq} \frac{F_n(\beta_1+h)-F_n(\beta_1)}{h} - \frac{F_n({\beta_0})-F_n({\beta_0}-h)}{h} \\
&\stackrel{{\mbox{\footnotesize\hspace{3.5ex}\eqref{Delta_def}\hspace{3.5ex}}}}{=} \frac{p(\beta_1+h)-p(\beta_1)}{h} - \frac{p({\beta_0})-p({\beta_0}-h)}{h} + \Delta_n^+(\beta_1,h) + \Delta_n^-({\beta_0},h) \\
&\stackrel{\mbox{\footnotesize\eqref{h_choice},\eqref{beta_prime_choice}}}{\leq} 3\eps + \Delta_n^+(\beta_1,h) + \Delta_n^-({\beta_0},h).
}
As $n\to\infty$, \eqref{free_energy_assumption} shows that $\Delta_n^+(\beta_1,h)$ and $\Delta_n^-({\beta_0},h)$ each converge to $0$ almost surely and in $L^1$.
Thus \eqref{beta_integral_summary} and the above display together yield the desired result, as $\eps$ is arbitrary.
\end{proof}

\subsection{Temperature perturbations} \label{temp_perturbations}
Here we derive upper bounds for the effects of temperature perturbations on certain expectations with respect to~$\mu_{n}^\beta$.

\begin{lemma} \label{connecting_betas}
The following statements hold for any $\beta_1\geq\beta_0\geq0$.
\begin{itemize}
\item[(a)] For any measurable $f : \Sigma_n\to[-1,1]$,
\eq{
|\langle f(\sigma)\rangle_{\beta_1} - \langle f(\sigma)\rangle_{\beta_0}| \leq \sqrt{n(\beta_1-\beta_0)(F_n'(\beta_1)-F_n'(\beta_0))}.
}
\item[(b)] For any $\sigma \in \Sigma_n$,
\eeq{\label{connecting_betas_2}
\frac{1}{n}\Big|\sum_i\vphi_i\langle\vphi_i\rangle_{\beta_1} - \sum_i \vphi_i\langle\vphi_i\rangle_{\beta_0}\Big| \leq \sqrt{n(\beta_1-\beta_0)(F_n'(\beta_1)-F_n'(\beta_0))}.
}
\item[(c)] Finally,
\eeq{ \label{connecting_betas_3}
\frac{1}{n}\Big|\sum_i\langle\vphi_i\rangle^2_{\beta_1} -\sum_i\langle\vphi_i\rangle^2_{\beta_0}\Big|
\leq 2\sqrt{n(\beta_1-\beta_0)(F_n'(\beta_1)-F_n'(\beta_0))}.
}
\end{itemize}
\end{lemma}

\begin{proof}
All three claims follow from two crucial observations.
First, for any $f \in L^2(\Sigma_n)$,
\eeq{ \label{crucial_1}
\Big|\frac{\partial}{\partial \beta}\langle f(\sigma)\rangle_\beta\Big| 
&\stackrefp{nrg_2deriv}{=}  |\langle f(\sigma)H_n(\sigma)\rangle_\beta - \langle f(\sigma) \rangle_\beta\langle H_n(\sigma)\rangle_\beta| \\
&\stackrel{\phantom{\mbox{\footnotesize\eqref{nrg_2deriv}}}}{\leq} \sqrt{\langle H_n(\sigma)^2\rangle_\beta - \langle H_n(\sigma)\rangle_\beta^2}\sqrt{\langle f(\sigma)^2\rangle_\beta - \langle f(\sigma)\rangle_\beta^2} \\
&\stackrel{\mbox{\footnotesize\eqref{nrg_2deriv}}}{=} \sqrt{nF_n''(\beta)}\sqrt{\langle f(\sigma)^2\rangle_\beta - \langle f(\sigma)\rangle_\beta^2}
\leq \sqrt{nF_n''(\beta)}\sqrt{\langle f(\sigma)^2\rangle_\beta}. 
}
And second,
\eeq{ \label{crucial_2}
\int_{\beta_0}^{\beta_1} \sqrt{nF_n''(\beta)}\ \dd \beta
&\leq \sqrt{n(\beta_1-\beta_0)\int_{\beta_0}^{\beta_1} F_n''(\beta)\ \dd \beta} 
= \sqrt{n(\beta_1-\beta_0)(F_n'(\beta_1)-F_n'(\beta_0))}.
}
Then part (a) immediately follows, since
\eq{
|f| \leq 1 \quad &\stackrel{\mbox{\footnotesize\eqref{crucial_1}}}{\implies} \quad 
\Big|\frac{\partial}{\partial \beta}\langle f(\sigma)\rangle_\beta\Big| \leq \sqrt{nF_n''(\beta)} \\
&\stackrel{\mbox{\footnotesize\eqref{crucial_2}}}{\implies} \quad
|\langle f(\sigma)\rangle_{\beta_1} - \langle f(\sigma)\rangle_{\beta_0}| \leq \sqrt{n(\beta_1-\beta_0)(F_n'(\beta_1)-F_n'(\beta_0))}.
}
For part (b), we first observe that if $0\leq\beta\leq\beta_1$, then
\eq{
\Big|\frac{\partial}{\partial\beta}\langle\vphi_i\rangle_\beta\Big| &\stackrel{\mbox{\footnotesize\eqref{crucial_1}}}{\leq}\sqrt{nF_n''(\beta)}\sqrt{\langle\vphi_i^2\rangle_\beta} \\
&\stackrel{\phantom{\mbox{\footnotesize\eqref{crucial_1}}}}{=} \sqrt{nF_n''(\beta)}\sqrt{\frac{E_n(\vphi_i^2\e^{\beta H_n(\sigma)})}{Z_n(\beta)}} \\
&\stackrel{\phantom{\mbox{\footnotesize\eqref{crucial_1}}}}{\leq} \sqrt{nF_n''(\beta)}\sqrt{\frac{E_n(\vphi_i^2)}{Z_n(\beta)} + \frac{E_n(\vphi_i^2\e^{\beta_1 H_n(\sigma)})}{Z_n(\beta)}}  \\
&\stackrel{\phantom{\mbox{\footnotesize\eqref{crucial_1}}}}{\leq} \sqrt{n\max_{\beta_0\in[0,\beta_1]}F_n''(\beta_0)}\sqrt{\frac{\max(Z_n(0),Z_n(\beta_1))}{\min_{\beta_0\in[0,\beta_1]} Z_n(\beta_0)}}\sqrt{\langle \vphi_i^2\rangle_0 + \langle\vphi_i^2\rangle_{\beta_1}},
}
where now the right-hand side is independent of $\beta$ and (almost surely) finite.
Moreover, we have the following finiteness condition when summing over $i$:
\eq{
\sum_{i} |\vphi_i|\sqrt{\langle \vphi_i^2\rangle_0 + \langle\vphi_i^2\rangle_{\beta_1}}
&\leq\sqrt{ \sum_i \vphi_i^2\sum_i (\langle\vphi_i^2\rangle_0 + \langle\vphi_i^2\rangle_{\beta_1})} 
\stackrel{\mbox{\footnotesize\eqref{variance_assumption}}}{=} \sqrt{2}n < \infty.
}
It thus follows that
\eq{
\frac{\partial}{\partial\beta} \sum_i \vphi_i\langle\vphi_i\rangle_\beta = \sum_i \vphi_i\frac{\partial}{\partial\beta}\langle\vphi_i\rangle_\beta.
}
In particular,
\eq{
\Big|\frac{\partial}{\partial\beta} \frac{1}{n}\sum_i \vphi_i\langle\vphi_i\rangle_\beta\Big|
&\stackrel{\phantom{\mbox{\footnotesize\eqref{crucial_1}}}}{\leq} \frac{1}{n}\sum_i \Big|\vphi_i\frac{\partial}{\partial\beta}\langle\vphi_i\rangle_\beta\Big| \\
&\stackrel{\mbox{\footnotesize\eqref{crucial_1}}}{\leq} \sqrt{\frac{F_n''(\beta)}{n}} \sum_i |\vphi_i| \sqrt{\langle\vphi_i^2\rangle_\beta} \\
&\stackrel{\phantom{\mbox{\footnotesize\eqref{crucial_1}}}}{\leq} \sqrt{\frac{F_n''(\beta)}{n}}\sqrt{\sum_i \vphi_i^2\sum_i \langle\vphi_i^2\rangle_\beta}
\stackrel{\mbox{\footnotesize\eqref{variance_assumption}}}{=} \sqrt{nF_n''(\beta)}.
}
As in part (a), \eqref{crucial_2} now proves \eqref{connecting_betas_2}.
For part (c), we can argue similarly in order to obtain
\eq{
\Big|\frac{\partial}{\partial\beta} \frac{1}{n}\sum_i \langle\vphi_i\rangle_\beta^2\Big| 
&\stackrel{\phantom{\mbox{\footnotesize\eqref{crucial_1}}}}{=} \Big|\frac{2}{n}\sum_i \langle\vphi_i\rangle_\beta\frac{\partial}{\partial\beta}\langle\vphi_i\rangle_\beta\Big| \\
&\stackrel{{\mbox{\footnotesize\eqref{crucial_1}}}}{\leq} 2\sqrt{\frac{F_n''(\beta)}{n}}\sum_i|\langle\vphi_i\rangle_\beta|\sqrt{\langle \vphi_i^2\rangle_\beta} \\
&\stackrel{\phantom{\mbox{\footnotesize\eqref{crucial_1}}}}{\leq} 2\sqrt{\frac{F_n''(\beta)}{n}}\sum_i \langle \vphi_i^2\rangle_\beta 
\stackrel{\mbox{\footnotesize\eqref{variance_assumption}}}{=} 2\sqrt{nF_n''(\beta)},
}
from which \eqref{crucial_2} proves \eqref{connecting_betas_3}.
\end{proof}

\section{{Proof of Theorem~\ASref}} \label{proof_1}
Recall the event under consideration:
\eq{
B_\delta = \Big\{\frac{1}{n}\sum_i\langle\vphi_i\rangle^2\leq\delta\Big\}.
}
The proof of Theorem~\ref{averages_squared} is a perturbative argument using an Ornstein--Uhlenbeck (OU) flow on the environment,
\eeq{ \label{OU_def}
\vc g_t \coloneqq \e^{-t}\vc g + \e^{-t} \vc W({\e^{2t}-1}), \quad t\geq0,
}
where $\vc W(\cdot) = (W_i(\cdot))_{i=1}^\infty$ is a collection of independent Brownian motions that are also independent of $\vc g = \vc g_0$, and the above definition is understood coordinate-wise.
Within Section~\ref{proof_1}, we denote expectation with respect to $\mu_{n,\vc g_t}^\beta$ by $\langle \cdot \rangle_t$, not to be confused with $\langle\cdot\rangle_\beta$ used in Section~\ref{prep_section}.
We now prove Theorem~\ref{averages_squared} by juxtaposing the following two propositions.
Notice that if $\P(B_\delta) = 0$, then there is nothing to be done; therefore, we may henceforth assume $\P(B_\delta) > 0$ so that conditioning on $B_\delta$ is well-defined.

\begin{prop} \label{prep_prop_good}
If $\beta$ is a point of differentiability for $p(\cdot)$, and $p'(\beta) < \beta$, then there exists $\kappa = \kappa(\beta) > 0$ such that the following holds:
For any $\eps > 0$, there is $T = T(\beta,\eps)$ sufficiently large that
\eeq{ \label{prep_eq_good}
\liminf_{n\to\infty} \P\bigg(\Big|\kappa - \frac{1}{T/n}\int_0^{T/n} \frac{1}{n}\sum_i \langle \vphi_i\rangle_t^2\ \dd t\Big| \leq \eps \bigg) \geq 1 - \eps.
}
More specifically,
\eq{
\kappa(\beta) = \frac{\beta-p'(\beta)}{\beta}.
}
\end{prop}

For the statement of the second result, 
let $\FFF_t$ denote the $\sigma$-algebra generated by $(\vc W(s))_{0\leq s\leq \e^{2t}-1}$ and the initial condition $\vc g_0$.

\begin{prop} \label{prep_prop_bad}
Assume $\beta$ is a point of differentiability for $p(\cdot)$.
Then there is a process $(I_t)_{t>0}$ adapted to the filtration $(\FFF_t)_{t>0}$, such that
the following statements hold:
\begin{itemize}
\item[(a)] For any $T,\eps>0$,
\eeq{ \label{prep_prop_bad_eq1}
\lim_{n\to\infty} \P\bigg(\Big|I_{T/n} - \frac{1}{T/n}\int_0^{T/n} \frac{1}{n}\sum_{i}\langle\vphi_i\rangle_t^2\ \dd t\Big| > \eps\bigg) = 0.
} 
\item[(b)] For any $T,\eps_1,\eps_2>0$, there exist $\delta_1 = \delta_1(\beta,T,\eps_1,\eps_2)>0$ sufficiently small and $n_0=n_0(\beta,T,\eps_1,\eps_2)$ sufficiently large, that
\eeq{ \label{prep_prop_bad_eq2}
\P\givenp[\bigg]{\Big|I_{T/n} - \frac{1}{n}\sum_i \langle\vphi_i\rangle^2\Big| \geq \eps_1}{B_{\delta}} \leq \eps_2 \quad \text{for all $0<\delta\leq\delta_1$, $n\geq n_0$}.
}
\end{itemize}
\end{prop}


\begin{proof}[Proof of Theorem~\ref{averages_squared}]
Let $\eps > 0$ be given, and assume the hypotheses of Proposition~\ref{prep_prop_good}.
By that result, there is $\kappa>0$ and $T$ large enough that
\eeq{ \label{first_liminf}
\liminf_{n\to\infty} \P\bigg(\frac{1}{T/n}\int_0^{T/n} \frac{1}{n}\sum_i \langle \vphi_i\rangle_t^2\ \dd t \geq \frac{4\kappa}{5}\bigg) \geq 1-\frac{\eps}{2}.
}
Let $(I_t)_{t\geq0}$ be the process guaranteed by Proposition~\ref{prep_prop_bad}, and define the events
\eq{
G &\coloneqq \bigg\{\frac{1}{T/n}\int_0^{T/n} \frac{1}{n}\sum_i \langle \vphi_i\rangle_t^2\ \dd t \geq \frac{4\kappa}{5}\bigg\}, \\
H &\coloneqq \bigg\{\frac{1}{T/n}\int_0^{T/n} \frac{1}{n}\sum_i \langle \vphi_i\rangle_t^2\ \dd t \leq \frac{3\kappa}{5} \bigg\}, \\
H_1 &\coloneqq \bigg\{\Big|I_{T/n} - \frac{1}{T/n}\int_0^{T/n} \frac{1}{n}\sum_i\langle\vphi_i\rangle_t^2\ \dd t\Big| \leq \frac{\kappa}{5}\Big\}, \\
H_2 &\coloneqq \bigg\{\Big|I_{T/n} - \frac{1}{n}\sum_i \langle\vphi_i\rangle^2\Big| \leq \frac{\kappa}{5}\bigg\}.
}
By Proposition~\ref{prep_prop_bad}(a),
\eeq{ \label{second_liminf}
\lim_{n\to\infty} \P(H_1) = 1.
}
And by Proposition~\ref{prep_prop_bad}(b), we can choose $0 < \delta \leq \kappa/5$ sufficiently small and $n_0$ sufficiently large that
\eeq{ \label{conditional_assumption}
\P\givenp{H_2}{B_{\delta}} \geq \frac{1}{2} \quad \text{for all $n\geq n_0$}.
}
Observe that $B_{\delta} \cap H_1\cap H_{2} \subset H$, and clearly the events $G$ and $H$ are disjoint.
We thus have
\eq{
\P(B_{\delta} \cap H_1 \cap H_2) \leq \P(H) \leq 1 - \P(G).
}
On the other hand,
\eq{
\P(B_{\delta} \cap H_1 \cap H_2)
\geq \P(H_1) + \P(H_2 \cap B_{\delta}) - 1
&\stackrel{\phantom{\mbox{\footnotesize{\eqref{conditional_assumption}}}}}{=} \P(H_1) - 1 + \P\givenp{H_2}{B_{\delta}}\P(B_{\delta}) \\
&\stackrel{\mbox{\footnotesize{\eqref{conditional_assumption}}}}{\geq} \P(H_1) - 1 + \frac{\P(B_{\delta})}{2}.
}
Putting the two previous displays together, we find
\eq{
\P(B_{\delta}) \leq 2\big(2 - \P(G) - \P(H_1)\big),
}
and so
\eq{
\limsup_{n\to\infty} \P(B_{\delta}) \leq 2\big(2 - \liminf_{n\to\infty} \P(G) - \lim_{n\to\infty} \P(H_1)\big) 
\stackrel{\mbox{\footnotesize\eqref{first_liminf},\eqref{second_liminf}}}{\leq}\eps.
}

\end{proof}

\subsection{Proof of Proposition~\ref{prep_prop_good}}
We will need to recall some facts about Ornstein--Uhlenbeck processes.
To avoid technical complications, we restrict ourselves to finite-dimensional OU processes, and then take an appropriate limit at a later stage.

\subsubsection{General OU theory} \label{OU_theory}
Fix a positive integer $N$, and consider a vector $\vc g = (g_1,\dots,g_N)$ of i.i.d.~standard normal random variables.
Let $\vc W = (\vc W(t))_{t\geq0}$ be an independent $N$-dimensional Brownian motion.
The OU flow starting at $\vc g$ is given by 
\eq{
\vc g_t \coloneqq \e^{-t}\vc g + \e^{-t}\vc W(\e^{2t}-1), \quad t\geq0.
}
This is a continuous-time, stationary Markov chain.
Let $(\PP_t)_{t\geq0}$ denote the OU semigroup; that is, for $f: \R^N \to \R$,
\eq{
\PP_tf(\vc x) \coloneqq \E f\big(\e^{-t}\vc x + \e^{-t}\vc W({\e^{2t}-1})\big), \quad \vc x\in\R^N.
}
Denote the OU generator by $\LL \coloneqq \Delta - \vc x \cdot \nabla$.
It is especially useful to consider the spectral decomposition of $\LL$, whose eigenfunctions are the multivariate Hermite polynomials.
For our purposes, it suffices to recall the following well-known facts (see, for instance, \cite[Chapter 6]{chatterjee14}):
\begin{itemize}
\item Let $\gamma_N$ denote the $N$-dimensional standard Gaussian measure.
There is an orthonormal basis $\{\phi_j\}_{j=0}^\infty$ of $L^2(\gamma_N)$ consisting of eigenfunctions of $\LL$, where $\phi_0 \equiv 1$, $\LL \phi_0 = \lambda_0 \phi_0= 0$, and $\LL \phi_j = -\lambda_j\phi_j$ with $\lambda_j > 0$ for $j\geq1$.
Therefore, if $f = \sum_{j=0}^\infty a_j\phi_j \in L^2(\gamma_N)$, then
\begin{align} 
\E f(\vc g) &= a_0, \label{mean_coef} \\
\LL f &= -\sum_{j=1}^\infty \lambda_ja_j\phi_j, \label{apply_L} \\
\implies \E \LL f(\vc g) &= 0 \label{L_zero_mean}.
\end{align}
Furthermore, if $f_1 = \sum_{j=0}^\infty a_j\phi_j, f_2 = \sum_{j=0}^\infty b_j\phi_j \in L^2(\gamma_N)$, then
\eeq{ \label{cov_formula}
\Cov\big(f_1(\vc g),f_2(\vc g)\big) = \sum_{j=1}^\infty a_jb_j.
}
\item The OU semigroup acts on $L^2(\gamma_N)$ by
\eq{
\PP_t\phi_j = \e^{-\lambda_j t}\phi_j, \quad j\geq0.
}
Therefore, if $f = \sum_{j=0}^\infty a_j\phi_j \in L^2(\gamma_N)$, then
\eeq{ \label{apply_PL}
\PP_t\LL f = -\sum_{j=1}^\infty \lambda_ja_j\e^{-\lambda_j t}\phi_j.
}
\item The associated Dirichlet form is given by
\eq{
- \E[f_1(\vc g)\LL f_2(\vc g)] = \E[\nabla f_1(\vc g)\cdot\nabla f_2(\vc g)],
}
whenever $f_1$ and $f_2$ are twice-differentiable functions in $L^2(\gamma_N)$ such that both expectations above are finite.
In particular, if $f_1 = f_2 = \sum_{j=0}^\infty a_j\phi_j\in L^2(\gamma_N)$ is twice-differentiable, then
\eeq{ \label{exp_grad_formula}
\E(\|\nabla f(\vc g)\|^2) = \sum_{j=1}^\infty \lambda_{j}a_{j}^2.
}
\end{itemize}


\begin{lemma} \label{OU_variance_lemma}
For any twice differentiable $f\in L^2(\gamma_N)$ with $\LL f \in L^2(\gamma_N)$, we have
\eq{
\Var\bigg(\frac{1}{t}\int_0^t \LL f(\vc g_s)\ \dd s\bigg) \leq \frac{2}{t}\E(\|\nabla f(\vc g)\|^2).
}
\end{lemma}

\begin{proof}
Take any $0\leq s\leq t$.
By the law of total variance, we have
\eq{
\Cov\big(f(\vc g_s),f(\vc g_t)\big) &= 
\E\big[\Cov\givenp{f(\vc g_s),f(\vc g_t)}{\vc g_s}\big] + \Cov\big(f(\vc g_s),\E\givenk{f(\vc g_t)}{\vc g_s}\big) \\
&= 0 + \Cov\big(f(\vc g_s),\E\givenk{f(\vc g_t)}{\vc g_s}\big) \\
&= \Cov\big(f(\vc g_s),\PP_{t-s}f(\vc g_s)\big) \\
&= \Cov\big(f(\vc g_0),\PP_{t-s}f(\vc g_0)\big).
}
In particular, if we write $f$ in the form $f = \sum_{j=0}^\infty a_{j}\phi_{j}$, then
\eq{
\Cov\big(\LL f(\vc g_s),\LL f(\vc g_t)\big) 
&\stackrel{\phantom{\mbox{\footnotesize\eqref{apply_L},\eqref{apply_PL},\eqref{cov_formula}}}}{=}  \Cov\big(\LL f(\vc g_0),\PP_{t-s}\LL f(\vc g_0)\big)  \\
&\stackrel{\mbox{\footnotesize\eqref{apply_L},\eqref{apply_PL},\eqref{cov_formula}}}{=} \sum_{j=1}^\infty \lambda_{j}^2a_j^2\e^{-\lambda_j(t-s)}.
}
Therefore,
\eq{
\int_0^t \Cov\big(\LL f(\vc g_s),\LL f(\vc g_t)\big)\ \dd s
&=\int_0^t \sum_{j=1}^\infty  \lambda_j^2a_j^2 \e^{-\lambda_j(t-s)}\ \dd s \\
&= \sum_{j=1}^\infty \lambda_ja_j^2(1-\e^{-\lambda_jt}) 
\leq \sum_{j=1}^\infty	 \lambda_ja_j^2
\stackrel{\mbox{\footnotesize\eqref{exp_grad_formula}}}{=} \E(\|\nabla f(\vc g)\|^2).
}
Hence
\eq{
\Var\bigg(\int_0^t \LL f(\vc g_s)\ \dd s\bigg)
&= \int_0^t\int_0^t \Cov\big(\LL f(\vc g_s),\LL f(\vc g_u)\big)\ \dd s\, \dd u \\
&= 2\int_0^t\int_0^u \Cov\big(\LL f(\vc g_s),\LL f(\vc g_u)\big)\ \dd s\, \dd u 
\leq 2t\E(\|\nabla f(\vc g)\|^2).
}
\end{proof}

\begin{proof}[Proof of Proposition~\ref{prep_prop_good}]
Let $(\vc g_t)_{t\geq0}$ be the OU flow from \eqref{OU_def}, and write
\eq{
g_i(t) \coloneqq \e^{-t}g_i + \e^{-t}W_i(\e^{2t}-1), \quad i\geq1.
}
Recall that $\langle \cdot \rangle_t$ denotes expectation with respect to $\mu_{n,\vc g_t}^\beta$. 
Let $Z_{n,t}(\beta)$ and $F_{n,t}(\beta)$ be the associated partition function and free energy, respectively.
That is, with $H_{n,t} \coloneqq \sum_i g_i(t)\vphi_i$, we have
\eq{
Z_{n,t}(\beta) \coloneqq E_n(\e^{\beta H_{n,t}}), \qquad
F_{n,t}(\beta) \coloneqq \frac{1}{n}\log Z_{n,t}(\beta).
}
So that we can use the finite-dimensional facts discussed before, define $H_{n,t,N} \coloneqq \sum_{i=1}^N g_i(t)\vphi_i$,
as well as
\eq{
Z_{n,t,N}(\beta) \coloneqq E_n(\e^{\beta H_{n,t,N}}), \qquad
F_{n,t,N}(\beta) \coloneqq \frac{1}{n}\log Z_{n,t,N}(\beta), \quad N\geq0.
}
Define $f : \R^{N} \to \R$ by
\eq{
f(\vc x) \coloneqq \frac{1}{n}\log E_n(\e^{\beta \sum_{i=1}^N x_i\vphi_i}),
}
so that $f(\vc g_t) = F_{n,t,N}(\beta)$, where $\vc g_t$ is understood to mean $(g_1(t),\dots,g_N(t))$.
Note that $f\in L^2(\gamma_{N})$, since
$\log^2 x \leq x + x^{-1}$ for $x>0$,
and so using the
same arguments as in Lemma~\ref{moments_lemma} yields
\eq{
\E \log^2 Z_{n,t,N}(\beta) &\leq \E Z_{n,t,N}(\beta)+\E[Z_{n,t,N}(\beta)^{-1}] \\
&\leq E_n(\E \e^{\beta H_{n,t,N}(\sigma)}) + E_n(\E \e^{-\beta H_{n,t,N}(\sigma)})
= 2E_n(\e^{\frac{\beta^2}{2}\sum_{i=1}^N\vphi_i^2}) \stackrel{\mbox{\footnotesize\eqref{variance_assumption}}}{\leq} 2\e^{\frac{\beta^2}{2}n}.
}
Similar to \eqref{finite_gibbs_expectation}, for general $\ff \in L^2(\Sigma_n)$, we define
\eeq{ \label{finite_gibbs_expectation_time}
\langle \ff(\sigma)\rangle_{t,N} = \frac{E_n(\ff(\sigma)\e^{\beta H_{n,t,N}(\sigma)})}{E_n(\e^{\beta H_{n,t,N}(\sigma)})}.
}
Observe that
\eq{
\frac{\partial f}{\partial x_i}(\vc g_t) = \frac{\beta\langle\vphi_i\rangle_{t,N}}{n}, \quad 1\leq i\leq N,
}
which implies
\eeq{ \label{grad_ineq}
\|\nabla f(\vc g_t)\|^2 = \frac{\beta^2}{n^2}\sum_{i=1}^N\langle\vphi_i\rangle_{t,N}^2 \leq 
\frac{\beta^2}{n^2}\sum_{i=1}^N\langle\vphi_i^2\rangle_{t,N} \stackrel{\mbox{\footnotesize\eqref{variance_assumption}}}{\leq} \frac{\beta^2}{n},
}
as well as
\eq{
\vc g_t \cdot \nabla f(\vc g_t) = \frac{\beta}{n}\sum_{i=1}^Ng_i(t)\langle \vphi_i\rangle_{t,N} 
= \frac{\beta}{n}\langle H_{n,t,N}(\sigma)\rangle_{t,N}
\stackrel{\mbox{\footnotesize\eqref{nrg_2deriv}}}{=} \beta F_{n,t,N}'(\beta),
}
where the derivative is with respect to $\beta$.
Note that
\eeq{ \label{Fprime_L2}
\E[F_{n,t,N}'(\beta)^2]
&\stackrel{\phantom{\mbox{\footnotesize\eqref{variance_assumption}}}}{=}  \frac{1}{n^2}\E\bigg[\Big(\sum_{i=1}^N g_i(t)\langle\vphi_i\rangle_{t,N}\Big)^2\bigg] \\
&\stackrel{\phantom{\mbox{\footnotesize\eqref{variance_assumption}}}}{\leq} \frac{1}{n^2}\E\bigg[\Big(\sum_{i=1}^N g_i(t)^2\Big)\Big(\sum_{i=1}^N\langle\vphi_i\rangle_{t,N}^2\Big)\bigg] \\
&\stackrel{\phantom{\mbox{\footnotesize\eqref{variance_assumption}}}}{\leq} \frac{1}{n^2}\E\bigg[\Big(\sum_{i=1}^N g_i(t)^2\Big)\Big(\sum_{i=1}^N\langle\vphi_i^2\rangle_{t,N}\Big)\bigg] \\
&\stackrel{\mbox{\footnotesize\eqref{variance_assumption}}}{\leq} \frac{1}{n}\E\Big(\sum_{i=1}^N g_i(t)^2\Big) 
= \frac{N}{n} < \infty.
}
Furthermore,
\eq{
\frac{\partial^2 f}{\partial x_i^2}(\vc g_t) = \frac{\beta^2}{n}(\langle \vphi_i^2\rangle_{t,N}-\langle\vphi_i\rangle_{t,N}^2), \quad 1\leq i\leq N.
}
We thus have
\eq{
\LL f(\vc g_t) &= \frac{\beta^2}{n} \sum_{i=1}^N (\langle\vphi_i^2\rangle_{t,N}-\langle\vphi_i\rangle_{t,N}^2) - \beta F_{n,t,N}'(\beta).
}
From \eqref{Fprime_L2}, it is clear that $\LL f \in L^2(\gamma_{N})$.
Therefore, by Lemma~\ref{OU_variance_lemma} and \eqref{grad_ineq},
\eq{
\Var\bigg(\frac{1}{t} \int_0^{t} \Big[\frac{\beta^2}{n} \sum_{i=1}^N (\langle\vphi_i^2\rangle_{s,N}-\langle\vphi_i\rangle_{s,N}^2) - \beta F_{n,s,N}'(\beta)\Big]\ \dd s\bigg) \leq \frac{2\beta^2}{tn}.
}
Moreover, from \eqref{L_zero_mean} we know
\eq{
\E\bigg(\frac{1}{t} \int_0^{t} \Big[\frac{\beta^2}{n} \sum_{i=1}^N (\langle\vphi_i^2\rangle_{s,N}-\langle\vphi_i\rangle_{s,N}^2) - \beta F_{n,s,N}'(\beta)\Big]\ \dd s\bigg) = 0.
}
We can now apply \eqref{exp_moments_lemma_a} (together with \eqref{nrg_2deriv}) and \eqref{limit_for_later} to take the limit $N\to\infty$ in the two previous displays and obtain
\eq{
\Var\bigg(\frac{1}{t} \int_0^{t} \Big[\beta^2 - \frac{\beta^2}{n} \sum_{i} \langle\vphi_i\rangle_{s}^2 - \beta F_{n,s}'(\beta)\Big]\ \dd s\bigg) &\leq \frac{2\beta^2}{tn}, \\
\E\bigg(\frac{1}{t} \int_0^{t} \Big[\beta^2 - \frac{\beta^2}{n} \sum_{i} \langle\vphi_i\rangle_{s}^2 - \beta F_{n,s}'(\beta)\Big]\ \dd s\bigg) &= 0.
}
Consequently, for any $\eps > 0$, Chebyshev's inequality shows
\eeq{ \label{limsup_prob_1}
\P\bigg(\Big|\frac{1}{t} \int_0^{t} \Big[\beta - \frac{\beta}{n}\sum_i\langle\vphi_i\rangle_s^2 - F_{n,s}'(\beta)\Big]\ \dd s\Big| \geq \frac{\eps}{2}\bigg) \leq \frac{8}{tn\eps^2}.
}
Now consider that
\eq{
\E\Big|p'(\beta)-\frac{1}{t}\int_0^t F_{n,s}'(\beta)\ \dd s\Big|
&\leq \frac{1}{t}\int_0^t \E|p'(\beta) - F_{n,s}'(\beta)|\ \dd s \\
&= \E|p'(\beta) - F_n'(\beta)|.
}
Therefore, if $\beta$ is a point of differentiability for $p(\cdot)$, then for any sequence $(t(n))_{n\geq1}$, Lemma~\ref{betas_converging} guarantees
\eeq{ \label{limsup_prob_2}
\limsup_{n\to\infty} \P\bigg(\Big|p'(\beta)-\frac{1}{t(n)}\int_0^{t(n)} F_{n,s}'(\beta)\ \dd s\Big| \geq \frac{\eps}{2}\bigg) = 0.
}
When $t = t(n) = T/n$ for fixed $T$, \eqref{limsup_prob_1} and \eqref{limsup_prob_2} together show
\eq{
\limsup_{n\to\infty} \P\bigg(\Big|\frac{1}{T/n} \int_0^{T/n} \Big[\beta - \frac{\beta}{n}\sum_i\langle\vphi_i\rangle_s^2 - p'(\beta)\Big]\ \dd s\Big| \geq \eps\bigg) \leq \frac{8}{T\eps^2}.
}
Assuming $p'(\beta)<\beta$, we let $\kappa = \kappa(\beta) \coloneqq \frac{\beta - p'(\beta)}{\beta} > 0$.
Then the previous display implies
\eq{
\limsup_{n\to\infty} \P\bigg(\Big|\kappa - \frac{1}{T/n} \int_0^{T/n} \frac{1}{n}\sum_i\langle\vphi_i\rangle_s^2\ \dd s\Big| \geq \eps\bigg) \leq \frac{8}{T\beta^2\eps^2}.
}
The proof is completed by taking $T = T(\beta,\eps)$ sufficiently large that
\eq{
\frac{8}{T\beta^2\eps^2} \leq \eps.
}
\end{proof}

\subsection{Proof of Proposition~\ref{prep_prop_bad}}
Let us rewrite \eqref{OU_def} as
\eq{
\vc g_t = \vc g + \e^{-t} \vc W(\e^{2t}-1) + (\e^{-t}-1)\vc g, \quad t\geq 0.
}
Recall that $\langle \cdot\rangle_0 = \langle\cdot\rangle$.
For any $f \in L^2(\Sigma_n)$, we have
\eq{
\langle f(\sigma)\rangle_t = \frac{\langle f(\sigma)\e^{\beta \e^{-t}\sum_{i} W_i(\e^{2t}-1)\vphi_i} \e^{\beta(\e^{-t} - 1)H_n(\sigma)}\rangle}{\langle\e^{\beta\e^{-t}\sum_{i} W_i(\e^{2t}-1)\vphi_i}\e^{\beta(\e^{-t} - 1)H_n(\sigma)}\rangle}.
}
In light of Lemma~\ref{lemma:step_1_2}, we anticipate that for $t = O(n^{-1})$,
\eeq{ \label{Qt_def}
\langle f(\sigma)\rangle_t  
&\approx \frac{\langle f(\sigma)\e^{\beta \e^{-t}\sum_{i} W_i(\e^{2t}-1)\vphi_i}\e^{-\beta tnp'(\beta)}\rangle }{\langle\e^{\beta\e^{-t}\sum_{i} W_i(\e^{2t}-1)\vphi_i}\e^{-\beta tnp'(\beta)} \rangle} \\
&= \frac{\langle f(\sigma)\e^{\beta \e^{-t}\sum_{i} W_i(\e^{2t}-1)\vphi_i}\rangle }{\langle\e^{\beta\e^{-t}\sum_{i} W_i(\e^{2t}-1)\vphi_i}\rangle}
 \eqqcolon Q_t(f).
}
Indeed,  the process that will satisfy the conclusions of Proposition~\ref{prep_prop_bad} is
\eeq{ \label{I_def}
I_t \coloneqq \frac{1}{t}\int_0^t \frac{1}{n}\sum_i Q_s(\vphi_i)^2\ \dd s, \quad t > 0.
}
To prove so, the following lemma will suffice.
Recall that
\eq{
B_\delta = \Big\{\frac{1}{n}\sum_i\langle\vphi_i\rangle^2\leq\delta\Big\}.
}

\begin{lemma} \label{prep_prop_bad_lemma}
For any $T,\eps>0$, the following statements hold:
\begin{itemize}
\item[(a)] 
If $\beta$ is a point of differentiability for $p(\cdot)$, then there is a sequence of nonnegative random variables $(M_n)$ depending only on $\beta$, $T$, and $\eps$, such that
\eeq{ \label{M_condition}
\limsup_{n\to\infty} \E(M_n) \leq \eps,
}
and for every $f \in L^2(\Sigma_n)$, $t\in[0,\frac{T}{n}]$,
\eeq{ \label{prep_prop_bad_lemma_eq1}
\E|Q_t(f)^2-\langle f(\sigma)\rangle_t^2| \leq \E(\langle f(\sigma)^2\rangle M_n).
}
\item[(b)] 
There exist $\delta_1 = \delta_1(\beta,T,\eps)>0$ sufficiently small and $n_0 = n_0(\beta,T,\eps)$ sufficiently large, that for every $n\geq n_0$, $f\in L^2(\Sigma_n)$, $ t\in[0, \frac{T}{n}]$, and $\delta\in(0,\delta_1]$, we have
\eeq{ \label{prep_prop_bad_lemma_eq2}
\E\givenp[\big]{|Q_t(f)^2- \langle f(\sigma)\rangle^2|}{B_{\delta}} \leq \eps \E\langle f(\sigma)^2\rangle.
}
\end{itemize}
\end{lemma}

Before checking these facts, let us use them to prove Proposition~\ref{prep_prop_bad}.
The idea is to use the above sequence $M_n$ to control the differences $Q_t(\vphi_i)^2-\langle\vphi_i\rangle^2$ simultaneously across all $i$ and $t\in[0,\frac{T}{n}]$; this will allow us to prove \eqref{prep_prop_bad_eq1}.
On the other hand, \eqref{prep_prop_bad_lemma_eq2} shows that when $\langle\RR_{1,2}\rangle$ is small, $Q_t(\vphi_i)^2$ remains close to $Q_0(\vphi_i)^2=\langle\vphi_i\rangle^2$.
That this approximation holds uniformly over $t\in[0,\frac{T}{n}]$ will lead to \eqref{prep_prop_bad_eq2}.

\begin{proof}[Proof of Proposition~\ref{prep_prop_bad}]
First we prove part (a). 
Let $T,\eps>0$ be fixed.
From Lemma~\ref{prep_prop_bad_lemma}(a), we identify a sequence of random variables $(M_n)$ such that \eqref{prep_prop_bad_lemma_eq1} holds, and
\eeq{ \label{M_condition_applied}
\limsup_{n\to\infty} \E(M_n) \leq \eps^2.
}
Under our definition \eqref{I_def}, we have
\eq{
\E\Big|I_{T/n} - \frac{1}{T/n}\int_0^{T/n} \frac{1}{n}\sum_i \langle \vphi_i\rangle_t^2\ \dd t\Big|
&\stackrel{\phantom{\mbox{\footnotesize{\eqref{prep_prop_bad_lemma_eq1}}}}}{=} \E\bigg|\frac{1}{T/n}\int_0^{T/n} \frac{1}{n}\sum_i [Q_t(\vphi_i)^2 - \langle \vphi_i\rangle_t^2]\ \dd t\bigg| \\
&\stackrel{\phantom{\mbox{\footnotesize{\eqref{prep_prop_bad_lemma_eq1}}}}}{\leq} \frac{1}{T/n}\int_0^{T/n} \frac{1}{n}\sum_i \E|Q_t(\vphi_i)^2 - \langle \vphi_i\rangle_t^2|\ \dd t \\
&\stackrel{\mbox{\footnotesize{\eqref{prep_prop_bad_lemma_eq1}}}}{\leq} \frac{1}{T/n}\int_0^{T/n} \frac{1}{n}\sum_i \E(\langle \vphi_i^2\rangle M_n)\ \dd t
\stackrel{\hspace{0.5ex}{\mbox{\footnotesize{\eqref{variance_assumption}}}}\hspace{0.5ex}}{=}  \E(M_n).
}
Now Markov's inequality and \eqref{M_condition_applied} together imply 
\eq{
\limsup_{n\to\infty} \P\bigg(\Big|I_{T/n} - \frac{1}{T/n}\int_0^{T/n} \frac{1}{n}\sum_i \langle \vphi_i\rangle_t^2\ \dd t\Big|\geq\eps\bigg) \leq \frac{\eps^2}{\eps} = \eps,
}
which completes the proof of (a).

Next we prove part (b).
Let $\eps_1,\eps_2 > 0$ be given.
Similar to above, for any $\delta>0$ we have
\eq{
\E\givenp[\Big]{\big|I_{T/n} - \frac{1}{n}\sum_i \langle \vphi_i\rangle^2\big|}{B_\delta}
&= \E\givenp[\Big]{\big|I_{T/n} - \frac{1}{T/n}\int_0^{T/n} \frac{1}{n}\sum_i \langle \vphi_i\rangle^2\ \dd t\big|}{B_\delta} \\
&\leq \frac{1}{T/n}\int_0^{T/n} \frac{1}{n}\sum_i \E\givenp[\big]{|Q_t(\vphi_i)^2 - \langle \vphi_i\rangle^2|}{B_\delta}\, \dd t.
} 
From Lemma~\ref{prep_prop_bad_lemma}(b), we choose $\delta_1$ 
sufficiently small that \eqref{prep_prop_bad_lemma_eq2} holds for all $\delta\in(0,\delta_1]$, with $\eps = \eps_1\eps_2$.
We then have, for all $n$ sufficiently large,
\eq{
\E\givenp[\Big]{\big|I_{T/n} - \frac{1}{n}\sum_i \langle \vphi_i\rangle^2\big|}{B_\delta}
\leq \frac{1}{T/n}\int_0^{T/n} \frac{1}{n}\sum_i \eps_1\eps_2\E\langle\vphi_i^2\rangle\ \dd s
\stackrel{\mbox{\footnotesize\eqref{variance_assumption}}}{=} \eps_1\eps_2.
}
Then applying Markov's inequality yields \eqref{prep_prop_bad_eq2}.
\end{proof}

It now remains to prove Lemma~\ref{prep_prop_bad_lemma}.
To do so, we will make use of the following preparatory result, which in fact is the common thread between the proofs of Theorems~\ref{expected_overlap_thm} and~\ref{averages_squared}.
Let $\vc h = (h_i)_{i=1}^\infty$ be an independent copy of the disorder $\vc g$.
We will use $\E_{\vc h}$ and $\Var_{\vc h}$ to denote expectation and variance with respect to $\vc h$, conditional on $\vc g$.
All statements involving these conditional quantities will be almost sure with respect to $\P$, although we will not repeatedly write this.

%

\begin{lemma} \label{h_variance_lemma}
Recall the constant $\EEE_n$ from \eqref{positive_overlap}.
For any $t\geq0$, the following statements hold:
\begin{itemize}
\item[(a)] For any $f\in L^2(\Sigma_n)$,
\eq{
\Var_{\vc h}\langle f(\sigma)\e^{\frac{t}{\sqrt{n}}\sum_i h_i\vphi_i}\rangle \leq \e^{2t^2}\langle f(\sigma)^2\rangle\sqrt{\frac{1}{n}\sum_i \langle\vphi_i\rangle^2+2\EEE_n}.
}
\item[(b)] For any measurable $f : \Sigma_n \to [0,1]$,
\eq{
\Var_{\vc h}\langle f(\sigma)\e^{\frac{t}{\sqrt{n}}\sum_i h_i\vphi_i}\rangle \leq \e^{2t^2}\Big(\Big\langle f(\sigma)\frac{1}{n}\sum_i \vphi_i\langle \vphi_i\rangle\Big\rangle+2\EEE_n\Big).
}
\end{itemize}
\end{lemma}

\begin{proof}
For any $f\in L^2(\Sigma_n)$,
\eeq{ \label{general_f}
\Var_{\vc h}\langle f(\sigma)\e^{\frac{t}{\sqrt{n}}\sum_i h_i\vphi_i}\rangle
&\stackrel{\phantom{\mbox{\footnotesize\eqref{freq_identity}}}}{=} \E_{\vc h}\langle f(\sigma^1)f(\sigma^2)\e^{\frac{t}{\sqrt{n}}\sum_{i} h_i(\vphi_i(\sigma^1)+\vphi_i(\sigma^2))}\rangle - \big(\E_{\vc h}\langle f(\sigma)\e^{\frac{t}{\sqrt{n}}\sum_i h_i\vphi_i}\rangle\big)^2\\
&\stackrel{\mbox{\footnotesize\eqref{freq_identity}}}{=} \e^{t^2}\big(\langle f(\sigma^1)f(\sigma^2)\e^{\frac{t^2}{n}\sum_i\vphi_i(\sigma^1)\vphi_i(\sigma^2)}\rangle - \langle f(\sigma)\rangle^2\big) \\
&\stackrel{\phantom{\mbox{\footnotesize\eqref{freq_identity}}}}{=} \e^{t^2}\langle f(\sigma^1)f(\sigma^2)(\e^{\frac{t^2}{n}\sum_i\vphi_i(\sigma^1)\vphi_i(\sigma^2)}-1)\rangle \\
&\stackrel{\phantom{\mbox{\footnotesize\eqref{freq_identity}}}}{\leq} \e^{t^2}\langle f(\sigma)^2\rangle\sqrt{\langle(\e^{\frac{t^2}{n}\sum_i\vphi_i(\sigma^1)\vphi_i(\sigma^2)}-1)^2\rangle}. \raisetag{2.5\baselineskip}
}
Now,  for all $x\in[-1,1]$, we have  $|\e^{t^2x}-1|\leq \e^{t^2}|x|$.
In particular, since
\eeq{ \label{unit_interval}
\Big|\frac{1}{n}\sum_i \vphi_i(\sigma^1)\vphi_i(\sigma^2)\Big|
\leq \frac{1}{n}\sqrt{\sum_i\vphi_i(\sigma^1)^2\sum_i\vphi_i(\sigma^2)^2} \stackrel{\mbox{\footnotesize\eqref{variance_assumption}}}{=} 1,
}
we see from \eqref{general_f} that
\eq{
\Var_{\vc h}\langle f(\sigma)\e^{\frac{t}{\sqrt{n}}\sum_i h_i\vphi_i}\rangle
&\stackrefp{positive_overlap}{\leq} \e^{2t^2}\langle f(\sigma)^2\rangle\sqrt{\Big\langle\Big(\frac{1}{n}\sum_i\vphi_i(\sigma^1)\vphi_i(\sigma^2)\Big)^2\Big\rangle} \\
&\stackref{positive_overlap}{\leq} \e^{2t^2}\langle f(\sigma)^2\rangle\sqrt{\Big\langle\frac{1}{n}\sum_i\vphi_i(\sigma^1)\vphi_i(\sigma^2)\Big\rangle+2\EEE_n} \\
&\stackrefp{positive_overlap}{=} \e^{2t^2}\langle f(\sigma)^2\rangle\sqrt{\frac{1}{n}\sum_i\langle\vphi_i\rangle^2+2\EEE_n}.
}
Alternatively, if $f:\Sigma_n\to[0,1]$, then we can use the equalities in \eqref{general_f} to write
\eq{
\Var_{\vc h}\langle f(\sigma)\e^{\frac{t}{\sqrt{n}}\sum_i h_i\vphi_i}\rangle 
&\stackrefp{positive_overlap}{=} \e^{t^2}\langle f(\sigma^1)f(\sigma^2)(\e^{\frac{t^2}{n}\sum_i\vphi_i(\sigma^1)\vphi_i(\sigma^2)}-1)\rangle \\
&\stackrefp{positive_overlap}{\leq} \e^{2t^2}\Big\langle f(\sigma^1)\Big|\frac{1}{n}\sum_i \vphi_i(\sigma^1)\vphi_i(\sigma^2)\Big|\Big\rangle \\
&\stackref{positive_overlap}{\leq} \e^{2t^2}\Big\langle f(\sigma^1)\Big(\frac{1}{n}\sum_i \vphi_i(\sigma^1)\vphi_i(\sigma^2)+2\EEE_n\Big)\Big\rangle \\
&\stackrefp{positive_overlap}{\leq} \e^{2t^2}\Big(\Big\langle f(\sigma^1)\frac{1}{n}\sum_i \vphi_i(\sigma^1)\langle\vphi_i(\sigma^2)\rangle\Big\rangle+2\EEE_n\Big).
}
\end{proof}

We are now ready to prove Lemma~\ref{prep_prop_bad_lemma}.

\begin{proof}[Proof of Lemma~\ref{prep_prop_bad_lemma}]
Let $f\in L^2(\Sigma_n)$ be arbitrary.
Recall the random variable $Q_t(f)$ defined in \eqref{Qt_def}.
Observe that for fixed $t\geq0$, $\e^{-t}\vc W(\e^{2t}-1)$ is equal in law to $\sqrt{1-\e^{-2t}}\vc h$, where $\vc h$ is an independent copy of $\vc g$.
Therefore, if we define
\eq{
X &\coloneqq  \langle f(\sigma)\e^{\beta\sqrt{1-\e^{-2t}}\sum_{i} h_i\vphi_i}\e^{\beta(\e^{-t}-1) H_n(\sigma)}\rangle, \\
Y &\coloneqq  \langle \e^{\beta\sqrt{1-\e^{-2t}}\sum_{i} h_i\vphi_i}\e^{\beta(\e^{-t}-1) H_n(\sigma)}\rangle, \\
X' &\coloneqq  \langle f(\sigma)\e^{\beta\sqrt{1-\e^{-2t}}\sum_{i} h_i\vphi_i}\rangle\e^{\beta (\e^{-t}-1)np'(\beta)}, \\
Y' &\coloneqq  \langle \e^{\beta\sqrt{1-\e^{-2t}}\sum_{i} h_i\vphi_i}\rangle\e^{\beta (\e^{-t}-1)np'(\beta)},
}
then
\eq{
(\langle f(\sigma)\rangle_t,Q_t(f)) \stackrel{\text{d}}{=} \Big(\frac{X}{Y},\frac{X'}{Y'}\Big).
}
Since the conclusions of Lemma~\ref{prep_prop_bad_lemma} depend only on marginal distributions at fixed $t\leq T/n$, it suffices to prove bounds of the form
\eeq{ \label{prep_prop_bad_lemma_eq1_new}
\E\Big|\Big(\frac{X}{Y}\Big)^2-\Big(\frac{X'}{Y'}\Big)^2\Big| \leq \E(\langle f(\sigma)^2\rangle M_n),
}
where $M_n$ satisfies \eqref{M_condition}, and
\eeq{ \label{prep_prop_bad_lemma_eq2_new}
\E\givenp[\bigg]{\Big|\Big(\frac{X'}{Y'}\Big)^2 - \langle f(\sigma)\rangle^2\Big|}{B_{\delta}} \leq \eps\E\langle f(\sigma)^2\rangle \quad \text{for all large enough $n$.}
}
So henceforth we fix $T,\eps>0$, and $t \in [0,\frac{T}{n}]$.
We will need the following four claims.
In checking these claims, we will frequently use the following inequality, which holds for any $c\geq0$:
\eeq{
n(1 - \e^{-ct}) \leq nct \leq cT. \label{frequent_ineq}
}

\begin{claim} \label{claim_denom_prime}
For any $q\in(-\infty,0]\cup[1,\infty)$,
\eeq{ \label{bad_prep_denom_bound_prime}
\E_{\vc h}[(Y')^{q}] \leq C(\beta,T,q).
}
\end{claim}

\begin{claim} \label{claim_num_prime}
For any $q\geq2$,
\eeq{ \label{bad_prep_num_bound_prime}
\E_{\vc h}[(X')^q] \leq C(\beta,T,q) \langle f(\sigma)^2\rangle^{q/2}.
}
\end{claim}

\begin{claim} \label{claim_denom}
Given any $q>0$, set $k = \floor{\log_2\frac{n}{qT}}$.
For all $n$ large enough that $k\geq1$,
\eeq{ \label{bad_prep_denom_bound}
\E_{\vc h}(Y^{-q}) \leq C(\beta,T,q)Z_n(\beta)^{-\frac{1}{2^k}}(Z_n(2\beta)^{\frac{1}{2^k}} +1).
}
\end{claim}

\begin{claim} \label{claim_var_bound}
For any even $q\geq 2$ and $\eps>0$, the following inequalities hold for all $n\geq(2q+1)T$: 
\eeq{ \label{bad_prep_var_bound}
\E_{\vc h}[(X-X')^q]
&\leq C(\beta,T,q)\langle f(\sigma)^2\rangle^{q/2}\Big[C(\eps)\Big\langle\Big|p'(\beta)-\frac{H_n(\sigma)}{n} \Big|\Big\rangle + \eps Z_n(\beta)^{-\frac{2(q+1)T}{n}}\Big],
}
and thus
\eeq{ \label{bad_prep_var_bound_Y}
\E_{\vc h}[(Y-Y')^q]
&\leq C(\beta,T,q)\Big[C(\eps)\Big\langle\Big|p'(\beta)-\frac{H_n(\sigma)}{n} \Big|\Big\rangle + \eps Z_n(\beta)^{-\frac{(2q+1)T}{n}}\Big].
}
\end{claim}

Before proving the claims, we use them to obtain the desired statements.

\subsubsection{Proof of Lemma~\ref{prep_prop_bad_lemma}(a)}
First note that for any random variables $W$ and $Z$,
\eeq{ \label{square_inside_outside}
\E|W^2-Z^2| &= \E|(W-Z)^2 + 2Z(W-Z)| \\
&\leq \E[(W-Z)^2] + 2\sqrt{\E(Z^2)\E[(W-Z)^2]}.
}
Therefore,
\eeq{ \label{prep_for_full_bound}
&\E_{\vc h}\Big|\Big(\frac{X}{Y}\Big)^2-\Big(\frac{X'}{Y'}\Big)^2\Big|
\leq\E_{\vc h}\Big[\Big(\frac{X}{Y} - \frac{X'}{Y'}\Big)^2\Big] + 2\sqrt{\E_{\vc h}\Big[\Big(\frac{X'}{Y'}\Big)^2\Big]\E_{\vc h}\Big[\Big(\frac{X}{Y} - \frac{X'}{Y'}\Big)^2\Big]} \\
&\stackrel{\phantom{\mbox{\footnotesize\eqref{bad_prep_denom_bound_prime},\eqref{bad_prep_num_bound_prime}}}}{\leq} 
\E_{\vc h}\Big[\Big(\frac{X}{Y} - \frac{X'}{Y'}\Big)^2\Big] + 2\big(\E_{\vc h}[(Y')^{-4}]\E_{\vc h}[(X')^4]\big)^\frac{1}{4}\sqrt{\E_{\vc h}\Big[\Big(\frac{X}{Y} - \frac{X'}{Y'}\Big)^2\Big]} \\
&\stackrel{\mbox{\footnotesize\eqref{bad_prep_denom_bound_prime},\eqref{bad_prep_num_bound_prime}}}{\leq}\E_{\vc h}\Big[\Big(\frac{X}{Y} - \frac{X'}{Y'}\Big)^2\Big]+C(\beta,T)\sqrt{\langle f(\sigma)^2\rangle}\sqrt{\E_{\vc h}\Big[\Big(\frac{X}{Y} - \frac{X'}{Y'}\Big)^2\Big]}.\hspace{0.4in} 
}
Let $\delta$ be a positive number to be chosen later.
Anticipating the application of Claims~\ref{claim_denom} and~\ref{claim_var_bound}, we condense notation by defining
\eq{
V_n^{(q)} &= \big(Z_n(\beta)^{-\frac{1}{2^k}}(Z_n(2\beta)^{\frac{1}{2^k}} + 1) \big)^{2/q}, \quad \text{where} \quad k = \Big\lfloor\log_2 \frac{n}{qT}\Big\rfloor, \\
W_n^{(q)} &= \Big(C(\delta)\Big\langle\Big|p'(\beta)-\frac{H_n(\sigma)}{n}\Big|\Big\rangle + \delta Z_n(\beta)^{-\frac{2(q+1)T}{n}}\Big)^{2/q}.
}
Because of \eqref{prep_for_full_bound}, we seek a bound of the form
\eq{ 
&\E_{\vc h}\Big[\Big(\frac{X}{Y} - \frac{X'}{Y'}\Big)^2\Big]
=\E_{\vc h}\Big[\Big(\frac{X-X'}{Y} - \frac{X'}{Y'}\frac{Y-Y'}{Y}\Big)^2\Big] \\
&\stackrel{\phantom{\mbox{\footnotesize\eqref{bad_prep_denom_bound_prime}--\eqref{bad_prep_var_bound_Y}}}}{\leq} 2\E_{\vc h}\Big[\frac{(X-X')^2}{Y^2}+\frac{(X')^2}{(Y')^2}\frac{(Y-Y')^2}{Y^2}\Big] \\
&\stackrel{\phantom{\mbox{\footnotesize\eqref{bad_prep_denom_bound_prime}--\eqref{bad_prep_var_bound_Y}}}}{\leq} 2\big(\E_{\vc h}[Y^{-4}]\E_{\vc h}[(X-X')^{4}]\big)^{1/2}
+ 2\big(\E_{\vc h}[(Y')^{-8}]\E_{\vc h}[(X')^{8}]\E_{\vc h}(Y^{-8})\E_{\vc h}[(Y-Y')^{8}]\big)^{1/4} \\
&\stackrel{\mbox{\footnotesize\eqref{bad_prep_denom_bound_prime}--\eqref{bad_prep_var_bound_Y}}}{\leq} C(\beta,T)\langle f(\sigma)^2\rangle(V_n^{(4)}W_n^{(4)}+V_n^{(8)}W_n^{(8)}).
}
Therefore, once we set
\eq{ 
M_n \coloneqq C(\beta,T)[(V_n^{(4)}W_n^{(4)}+V_n^{(8)}W_n^{(8)}) + (V_n^{(4)}W_n^{(4)}+V_n^{(8)}W_n^{(8)})^{1/2}]
}
and take expectation, \eqref{prep_for_full_bound} becomes
\eq{
\E\Big|\Big(\frac{X}{Y}\Big)^2-\Big(\frac{X'}{Y'}\Big)^2\Big| \leq \E(\langle f(\sigma)^2\rangle M_n),
}
which is exactly \eqref{prep_prop_bad_lemma_eq1_new}.
To complete the proof of Lemma~\ref{prep_prop_bad_lemma}(a), we need to show that given any $\eps >0$, we can choose $\delta$ sufficiently small that \eqref{M_condition} holds ($M_n$ depends on $\delta$ through $W_n^{(4)}$ and $W_n^{(8)}$).

Indeed, by Cauchy--Schwarz we have
\eeq{ \label{M_ineq}
\E(M_n) &\leq C(\beta,T)\bigg(\sqrt{\E[(V_n^{(4)})^2]\E[(W_n^{(4)})^2]} + \sqrt{\E[(V_n^{(8)})^2]\E[(W_n^{(8)})^2]} \\
&\phantom{\leq} + \sqrt{\sqrt{\E[(V_n^{(4)})^2]\E[(W_n^{(4)})^2]} + \sqrt{\E[(V_n^{(8)})^2]\E[(W_n^{(8)})^2]}}\,\bigg).
}
Next we observe that for $q \geq 4$ and $n$ sufficiently large such that $k = \floor{\log_2 \frac{n}{qT}}\geq1$,
\eeq{ \label{V_ineq}
\E[(V_n^{(q)})^2] 
&\stackrel{\phantom{\mbox{\footnotesize\eqref{first_moment},\eqref{negative_first_moment}}}}{\leq}  \Big(\E\big[Z_n(\beta)^{-\frac{1}{2^k}}(Z_n(2\beta)^{\frac{1}{2^k}} + 1) \big]\Big)^{4/q} \\
&\stackrel{\phantom{\mbox{\footnotesize\eqref{first_moment},\eqref{negative_first_moment}}}}{\leq}  \Big(\sqrt{\E[Z_n(\beta)^{-\frac{2}{2^{k}}}]\E[Z_n(2\beta)^{\frac{2}{2^k}}]}+\E[Z_n(\beta)^{-\frac{1}{2^{k}}}]\Big)^{4/q} \\
&\stackrel{\phantom{\mbox{\footnotesize\eqref{first_moment},\eqref{negative_first_moment}}}}{\leq}  \Big(\sqrt{\E[Z_n(\beta)^{-1}]^\frac{2}{2^{k}}\E[Z_n(2\beta)]^{\frac{2}{2^k}}}+\E[Z_n(\beta)^{-1}]^\frac{1}{2^{k}}\Big)^{4/q} \\
&\stackrel{\mbox{\footnotesize\eqref{first_moment},\eqref{negative_first_moment}}}{\leq}\Big(\sqrt{\e^\frac{\beta^2n}{2^{k}}\e^{\frac{4\beta^2 n}{2^k}}}+\e^\frac{\beta^2 n}{2^{k+1}}\Big)^{4/q} \\
&\stackrel{\phantom{\mbox{\footnotesize\eqref{first_moment},\eqref{negative_first_moment}}}}{\leq} \Big(\sqrt{\e^{\beta^2qT}\e^{4\beta^2 qT}}+\e^\frac{\beta^2 qT}{2}\Big)^{4/q} = C(\beta,T,q).
}
Meanwhile, if $q\geq4$ and $n\geq2(q+1)T$, then
\eq{
\E[(W_n^{(q)})^2] 
&\stackrel{\phantom{\mbox{\footnotesize\eqref{negative_first_moment}}}}{\leq} \Big(C(\delta)\E\Big\langle\Big|p'(\beta)-\frac{H_n(\sigma)}{n}\Big|\Big\rangle + \delta \E[Z_n(\beta)^{-\frac{2(q+1)T}{n}}]\Big)^{4/q} \\
&\stackrel{\phantom{\mbox{\footnotesize\eqref{negative_first_moment}}}}{\leq} \Big(C(\delta)\E\Big\langle\Big|p'(\beta)-\frac{H_n(\sigma)}{n}\Big|\Big\rangle + \delta \E[Z_n(\beta)^{-1}]^{\frac{2(q+1)T}{n}}\Big)^{4/q} \\
&\stackrel{\mbox{\footnotesize\eqref{negative_first_moment}}}{\leq}  \Big(C(\delta)\E\Big\langle\Big|p'(\beta)-\frac{H_n(\sigma)}{n}\Big|\Big\rangle + \delta \e^{\beta^2(q+1)T}\Big)^{4/q}.
}
By Lemma~\ref{lemma:step_1_2}, the previous display shows
\eq{
\limsup_{n\to\infty} \E[(W_n^{(q)})^2] &\leq \delta^{4/q}\e^\frac{4\beta^2(q+1)T}{q} = C(\beta,T,q)\delta^{4/q}.
}
In light of \eqref{M_ineq} and \eqref{V_ineq}, it is clear from this inequality that $\delta$ can be chosen sufficiently small that \eqref{M_condition} holds.

\subsubsection{Proof of Lemma~\ref{prep_prop_bad_lemma}(b)}
To establish \eqref{prep_prop_bad_lemma_eq2_new}, it will be easier to replace $X'/Y'$ by $X''/Y''$, where
\eq{
X'' &\coloneqq \frac{X'}{\e^{\frac{\beta^2}{2}(1-\e^{-2t})n}\e^{\beta(\e^{-t}-1)np'(\beta)}} = \frac{\langle f(\sigma)\e^{\beta\sqrt{1-\e^{-2t}}\sum_ih_i\vphi_i}\rangle}{\e^{\frac{\beta^2}{2}(1-\e^{-2t})n}}, \\
\qquad Y'' &\coloneqq \frac{Y'}{\e^{\frac{\beta^2}{2}(1-\e^{-2t})n}\e^{\beta(\e^{-t}-1)np'(\beta)}} = \frac{\langle \e^{\beta\sqrt{1-\e^{-2t}}\sum_ih_i\vphi_i}\rangle}{\e^{\frac{\beta^2}{2}(1-\e^{-2t})n}}.
}
By Lemma~\ref{h_variance_lemma}(a),
\eq{
\Var_{\vc h}\langle f(\sigma)\e^{\beta\sqrt{1-\e^{-2t}}\sum_i h_i\vphi_i}\rangle
&\leq \e^{2\beta^2(1-\e^{-2t})n}\langle f(\sigma)^2\rangle\sqrt{\frac{1}{n}\sum_i \langle\vphi_i\rangle^2+2\EEE_n},
}
and so
\eeq{ \label{observation_1}
\Var_{\vc h}(X'') &\stackrel{\phantom{\mbox{\footnotesize\eqref{frequent_ineq}}}}{\leq} \e^{\beta^2(1-\e^{-2t})n}\langle f(\sigma)^2\rangle\sqrt{\frac{1}{n}\sum_i\langle\vphi_i\rangle^2+2\EEE_n} \\
&\stackrel{{\mbox{\footnotesize\eqref{frequent_ineq}}}}{\leq}
 C(\beta,T)\langle f(\sigma)^2\rangle\sqrt{\frac{1}{n}\sum_i\langle\vphi_i\rangle^2+2\EEE_n},
 }
as well as
\eq{
\Var_{\vc h}(Y'') &\leq C(\beta,T)\sqrt{\frac{1}{n}\sum_i\langle\vphi_i\rangle^2+2\EEE_n}.
}
Because
\eq{
\E_{\vc h} \langle f(\sigma) \e^{\beta\sqrt{1-\e^{-2t}}\sum_i h_i\vphi_i}\rangle
\stackrel{{\mbox{\footnotesize\eqref{freq_identity}}}}{=} \e^{\frac{\beta^2 }{2}(1-\e^{-2t})n}\langle f(\sigma)\rangle,
}
we have $\E_{\vc h}(Y'') = 1$ and can thus apply Chebyshev's inequality to obtain
\eeq{ \label{observation_2}
\P_{\vc h}(|Y'' - 1|\geq\theta) \leq \frac{C(\beta,T)}{\theta^2} \sqrt{\frac{1}{n}\sum_i\langle\vphi_i\rangle^2+2\EEE_n} \quad \text{for any $\theta>0$.}
}
We will use these inequalities in the following bound:
\eeq{ \label{observation_0}
&\E_{\vc h}\Big[\Big(\frac{X'}{Y'} - \langle f(\sigma)\rangle\Big)^2\Big]
= \E_{\vc h}\Big[\Big(\frac{X''}{Y''} - \langle f(\sigma)\rangle\Big)^2\Big] \\
&= \E_{\vc h}\Big[\Big(\frac{X''}{Y''}(1 - Y'') + X'' - \langle f(\sigma)\rangle\Big)^2\Big] \\
&\leq 2\E_{\vc h}\Big[\Big(\frac{X''}{Y''}\Big)^2(Y''-1)^2 + \big(X'' - \langle f(\sigma)\rangle\big)^2\Big] \\
&\leq  2\E_{\vc h}\Big[\Big(\frac{X''}{Y''}\Big)^2(\theta^2+\one_{\{|Y''-1|\geq\theta\}}(Y''-1)^2) + \big(X'' - \langle f(\sigma)\rangle\big)^2\Big] \\
&\leq 2\big(\E_{\vc h}[(Y'')^{-8}]\E_{\vc h}[(X'')^8])^{1/4}\sqrt{\E_{\vc h}\big[\big(\theta^2+\one_{\{|Y''-1|\geq\theta\}}(Y''-1)^2\big)^2\big]} 
+ 2\Var_{\vc h}(X'') \\
&\leq 2\sqrt{2}\big(\E_{\vc h}[(Y'')^{-8}]\E_{\vc h}[(X'')^8])^{1/4}\sqrt{\theta^4 + \sqrt{\P_{\vc h}(|Y''-1|\geq\theta)\E_{\vc h}[(Y''-1)^8]}}
+ 2\Var_{\vc h}(X'').
\raisetag{7\baselineskip}
}
Now,
\eeq{ \label{observation_3}
\E_{\vc h}[(Y'')^{-8}] = \frac{\E_{\vc h}[(Y')^{-8}]}{\e^{-4\beta^2(1-\e^{-2t})n}\e^{-8\beta(\e^{-t}-1)np'(\beta)}} 
\stackrel{\mbox{\footnotesize\eqref{frequent_ineq},\eqref{bad_prep_denom_bound_prime}}}{\leq} C(\beta,T),
}
and
\eeq{ \label{observation_4}
\E_{\vc h}[(X'')^8] = \frac{\E_{\vc h}[(X')^8]}{\e^{4\beta^2(1-\e^{-2t})n}\e^{8\beta(\e^{-t}-1)np'(\beta)}}
\stackrel{\mbox{\footnotesize\eqref{frequent_ineq},\eqref{bad_prep_num_bound_prime}}}{\leq}  C(\beta,T)\langle f(\sigma)^2\rangle^4.
}
In addition,
\eeq{ \label{observation_5}
\E_{\vc h}[(Y''-1)^8] 
&\stackrel{\phantom{\mbox{\footnotesize\eqref{frequent_ineq},\eqref{bad_prep_denom_bound_prime}}}}{\leq} 2^4(\E_{\vc h}[(Y'')^8]+1)  \\
&\stackrel{\phantom{\mbox{\footnotesize\eqref{frequent_ineq},\eqref{bad_prep_denom_bound_prime}}}}{=} 2^4\Big( \frac{\E_{\vc h}[(Y')^{8}]}{\e^{4\beta^2(1-\e^{-2t})n}\e^{8\beta(\e^{-t}-1)np'(\beta)}}+1\Big)  
\stackrel{\mbox{\footnotesize\eqref{frequent_ineq},\eqref{bad_prep_denom_bound_prime}}}{\leq} C(\beta,T).
}
Using \eqref{observation_1}, \eqref{observation_2}, and \eqref{observation_3}--\eqref{observation_5} in \eqref{observation_0}, we find
\eq{
\E_{\vc h}\Big[\Big(\frac{X'}{Y'} - \langle f(\sigma)\rangle\Big)^2\Big]
&\leq C(\beta,T)\langle f(\sigma)^2\rangle\sqrt{\theta^4+\frac{C(\beta,T)}{\theta}\Big(\frac{1}{n}\sum_i\langle\vphi_i^2\rangle+2\EEE_n\Big)^{1/4}} \\
&\phantom{\leq}+C(\beta,T)\langle f(\sigma)^2\rangle \sqrt{\frac{1}{n}\sum_i \langle\vphi_i\rangle^2+2\EEE_n}.
}
In particular, for any $\delta>0$ and $n$ large enough that $\EEE_n\leq\delta/2$,
\eq{
\one_{B_\delta}\E_{\vc h}\Big[\Big(\frac{X'}{Y'} - \langle f(\sigma)\rangle\Big)^2\Big] \leq \one_{B_\delta}C(\beta,T)\langle f(\sigma)^2\rangle\Big(\sqrt{\theta^4+\theta^{-1}(2\delta)^{1/4}}+\sqrt{2\delta}\Big),
}
and so \eqref{square_inside_outside} implies
\eq{
\one_{B_\delta}\E_{\vc h}\Big|\Big(\frac{X'}{Y'}\Big)^2 - \langle f(\sigma)\rangle^2\Big|
&\leq \one_{B_\delta}\E_{\vc h}\Big[\Big(\frac{X'}{Y'} - \langle f(\sigma)\rangle\Big)^2\Big]
+ 2\one_{B_\delta}\sqrt{\langle f(\sigma)\rangle^2\E_{\vc h}\Big[\Big(\frac{X'}{Y'} - \langle f(\sigma)\rangle\Big)^2\Big]} \\
&\leq \one_{B_\delta}C(\beta,T)\langle f(\sigma)^2\rangle\Big(\sqrt{\theta^4+\theta^{-1}\delta^{1/4}}+\sqrt{\delta} 
+ \sqrt{\sqrt{\theta^4+\theta^{-1}\delta^{1/4}}+\sqrt{\delta}}\, \Big).
}
Given $\eps>0$, we choose $\theta$ and $\delta$ small enough (in that order, and depending only on $\beta$, $T$, and $\eps$) so that the rightmost expression above is at most $\one_{B_\delta}\eps\langle f(\sigma)^2\rangle$. 
Moreover, it is clear that once $\theta$ and $\delta$ are chosen, $\one_{B_\delta}$ could be replaced by $\one_{B_{\delta'}}$ for any $\delta'\in(0,\delta)$, and the rightmost expression will be bounded from above by $\one_{B_{\delta'}}\eps\langle f(\sigma)^2\rangle$.
Taking expectations on both sides yields \eqref{prep_prop_bad_lemma_eq2_new}.

\subsubsection{Proof of Claim~\ref{claim_denom_prime}}
Assume $q\leq0$ or $q\geq1$.
Using Jensen's inequality, we have
\eq{ 
\E_{\vc h}[(Y')^{q}] &\stackrel{\phantom{\mbox{\footnotesize\eqref{freq_identity}}}}{=} \e^{q\beta (\e^{-t}-1)np'(\beta)}\E_{\vc h}\big[\langle\e^{\beta\sqrt{1-\e^{-2t}}\sum_{i} h_i\vphi_i}\rangle^{q}\big] \\
&\stackrel{\phantom{\mbox{\footnotesize\eqref{freq_identity}}}}{\leq} \e^{q\beta (\e^{-t}-1)np'(\beta)}\E_{\vc h}\langle\e^{q\beta\sqrt{1-\e^{-2t}}\sum_{i} h_i\vphi_i}\rangle \\
&\stackrel{{\mbox{\footnotesize\eqref{freq_identity}}}\hspace{0.5ex}}{=} \e^{q\beta (\e^{-t}-1)np'(\beta)}\e^{\frac{q^2\beta^2}{2}(1-\e^{-2t})n} 
\stackrel{{\mbox{\footnotesize\eqref{frequent_ineq}}}}{\leq} C(\beta,T,q).
}

\subsubsection{Proof of Claim~\ref{claim_num_prime}}
Assume $q\geq2$.
By Cauchy--Schwarz and Jensen's inequality, we have
\eq{
\E_{\vc h}[(X')^q]
&\stackrel{\phantom{\mbox{\footnotesize\eqref{freq_identity}}}}{=} \e^{q\beta(\e^{-t}-1)np'(\beta)}\E_{\vc h}(\langle f(\sigma) \e^{\beta\sqrt{1-\e^{-2t}}\sum_ih_i\vphi_i}\rangle^q) \\
&\stackrel{\phantom{\mbox{\footnotesize\eqref{freq_identity}}}}{\leq} \e^{q\beta(\e^{-t}-1)np'(\beta)}\E_{\vc h}(\langle f(\sigma)^2\rangle^{q/2} \langle\e^{2\beta\sqrt{1-\e^{-2t}}\sum_ih_i\vphi_i}\rangle^{q/2}) \\
&\stackrel{\phantom{\mbox{\footnotesize\eqref{freq_identity}}}}{\leq}\e^{q\beta(\e^{-t}-1)np'(\beta)} \langle f(\sigma)^2\rangle^{q/2}\E_{\vc h} \langle\e^{q\beta\sqrt{1-\e^{-2t}}\sum_ih_i\vphi_i}\rangle \\
&\stackrel{{\mbox{\footnotesize\eqref{freq_identity}}}\hspace{0.5ex}}{=}\e^{q\beta (\e^{-t}-1)np'(\beta)} \langle f(\sigma)^2\rangle^{q/2}\e^{\frac{q^2\beta^2}{2} (1-\e^{-2t})n} 
\stackrel{\mbox{\footnotesize\eqref{frequent_ineq}}}{\leq} C(\beta,T,q)\langle f(\sigma)^2\rangle^{q/2}.
}

\subsubsection{Proof of Claim~\ref{claim_denom}}
Assume $q>0$.
By Jensen's inequality,
\eeq{ \label{pre_squish_expectation}
\E_{\vc h}(Y^{-q}) &\stackrel{\phantom{\mbox{\footnotesize\eqref{freq_identity}}}}{=}  
\E_{\vc h}[\langle \e^{\beta\sqrt{1-\e^{-2t}}\sum_ih_i\vphi_i}\e^{\beta(\e^{-t}-1)H_n(\sigma)}\rangle^{-q}]\\
&\stackrel{\phantom{\mbox{\footnotesize\eqref{freq_identity}}}}{\leq} \E_{\vc h}\langle \e^{-q\beta \sqrt{1-\e^{-2t}}\sum_ih_i\vphi_i}\e^{q\beta (1-\e^{-t})H_n(\sigma)}\rangle \\
&\stackrel{{\mbox{\footnotesize\eqref{freq_identity}}}\hspace{0.5ex}}{=} \e^{\frac{q^2\beta^2}{2}(1-\e^{-2t})n}\langle \e^{\beta q(1-\e^{-t})H_n(\sigma)}\rangle 
\stackrel{\mbox{\footnotesize\eqref{frequent_ineq}}}{\leq} C(\beta,T,q)\langle \e^{\beta q(1-\e^{-t})H_n(\sigma)}\rangle.
}
Recall that $k = \floor{\log_2\frac{n}{qT}}$, and we assume $k\geq1$.
By \eqref{frequent_ineq},
\eq{
q(1-\e^{-t}) \leq \frac{qT}{n} = \frac{1}{2^{\log_2\frac{n}{qT}}}  \leq \frac{1}{2^k},
}
which implies
\eeq{ \label{squish_expectation}
\langle \e^{\beta q(1-\e^{-t})H_n(\sigma)}\rangle \leq \langle \e^{-\beta H_n(\sigma)/2^{k}}\rangle + \langle \e^{\beta H_n(\sigma)/2^{k}}\rangle.
}
Repeated applications of Cauchy--Schwarz yield
\eeq{ \label{repeated_CS}
\langle \e^{\beta H_n(\sigma)/2^k}\rangle 
&= \frac{E_n(\e^{\beta(1+\frac{1}{2^k})H_n(\sigma)})}{E_n(\e^{\beta H_n(\sigma)})} \\
&= \frac{E_n(\e^{\frac{\beta}{2}H_n(\sigma)}\e^{\beta(\frac{1}{2}+\frac{1}{2^k})H_n(\sigma)})}{E_n(\e^{\beta H_n(\sigma)})}  \\
&\leq \frac{\sqrt{E_n(\e^{\beta H_n(\sigma)})E_n(\e^{\beta(1+\frac{1}{2^{k-1}})H_n(\sigma)})}}{E_n(\e^{\beta H_n(\sigma)})} \\
&\leq \frac{\sqrt{E_n(\e^{\beta H_n(\sigma)})\sqrt{E_n(\e^{\beta H_n(\sigma)})E_n(\e^{\beta(1+\frac{1}{2^{k-2}})H_n(\sigma)})}}}{E_n(\e^{\beta H_n(\sigma)})} \\
&\hspace{1.3ex}\vdots \\
&\leq E_n(\e^{\beta H_n(\sigma)})^{-1+\sum_{i=1}^{k} \frac{1}{2^i}}E_n(\e^{2\beta H_n(\sigma)})^{\frac{1}{2^k}} \\
&= Z_n(\beta)^{-\frac{1}{2^k}}Z_n(2\beta)^{\frac{1}{2^k}}.
}
By similar manipulations,
\eeq{ \label{repeated_CS_2}
\langle \e^{-\beta H_n(\sigma)/2^k}\rangle \leq Z_n(\beta)^{-\frac{1}{2^k}}Z_n(0)^{\frac{1}{2^k}} = Z_n(\beta)^{-\frac{1}{2^k}}.
}
Together, \eqref{pre_squish_expectation}--\eqref{repeated_CS_2} yield \eqref{bad_prep_denom_bound}.

\subsubsection{Proof of Claim~\ref{claim_var_bound}}
Assume $q\geq2$ is even. 
By Cauchy--Schwarz and Jensen's inequality, we have
\eeq{ \label{next_1}
&\E_{\vc h}[(X-X')^q] \\
&\stackrel{\phantom{\mbox{\footnotesize\eqref{frequent_ineq}}}}{=}
\E_{\vc h}[\langle f(\sigma)\e^{\beta\sqrt{1-\e^{-2t}}\sum_{i} h_i\vphi_i}(\e^{\beta(\e^{-t}-1) H_n(\sigma)}-\e^{\beta (\e^{-t}-1)np'(\beta)})\rangle^q] \\
&\stackrel{\phantom{\mbox{\footnotesize\eqref{frequent_ineq}}}}{\leq}
\E_{\vc h}\big[\langle f(\sigma)^2\rangle^{q/2}\langle\e^{2\beta\sqrt{1-\e^{-2t}}\sum_{i} h_i\vphi_i}(\e^{\beta(\e^{-t}-1) H_n(\sigma)}-\e^{\beta (\e^{-t}-1)np'(\beta)})^2\rangle^{q/2}\big] \\
&\stackrel{\phantom{\mbox{\footnotesize\eqref{frequent_ineq}}}}{\leq} 
\langle f(\sigma)^2\rangle^{q/2}\e^{q\beta (\e^{-t}-1)np'(\beta)}
\E_{\vc h}\langle\e^{q\beta\sqrt{1-\e^{-2t}}\sum_ih_i\vphi_i}(\e^{\beta(1-\e^{-t})(np'(\beta)- H_n(\sigma))}-1)^q\rangle \\
&\stackrel{\mbox{\footnotesize\eqref{freq_identity}}}{=} \langle f(\sigma)^2\rangle^{q/2}\e^{q\beta (\e^{-t}-1)np'(\beta)}\e^{\frac{q^2\beta^2}{2}(1-\e^{-2t})n}\langle(\e^{\beta(1-\e^{-t})(np'(\beta)- H_n(\sigma))}-1)^q\rangle \\
&\stackrel{\mbox{\footnotesize\eqref{frequent_ineq}}}{\leq} C(\beta,T,q)\langle f(\sigma)^2\rangle^{q/2} \langle(\e^{\beta(1-\e^{-t})(np'(\beta)- H_n(\sigma))}-1)^q\rangle. \raisetag{4.5\baselineskip}
}
For any $L>0$, we have the inequality $(\e^x - 1)^q \leq C(L,q)|x|$ 
for all $x\leq L$.
Hence
\eeq{ \label{next_2}
&\langle(\e^{\beta(1-\e^{-t})(np'(\beta)- H_n(\sigma))}-1)^q\rangle \\
&\stackrel{\phantom{\mbox{\footnotesize\eqref{frequent_ineq}}}}{\leq} C(L,q)\beta (1-\e^{-t})n\Big\langle\Big|p'(\beta)-\frac{H_n(\sigma)}{n}\Big|\Big\rangle \\
&\phantom{\stackrel{\mbox{\footnotesize\eqref{frequent_ineq}}}{\leq}} + \langle(\e^{\beta(1-\e^{-t})(np'(\beta)- H_n(\sigma))}-1)^q\one_{\{\beta(1-\e^{-t})(np'(\beta)- H_n(\sigma)) > L\}}\rangle \\
&\stackrel{{\mbox{\footnotesize\eqref{frequent_ineq}}}}{\leq} C(\beta,T,L,q) \Big\langle\Big|p'(\beta)-\frac{H_n(\sigma)}{n}\Big|\Big\rangle \\
&\phantom{\stackrel{\mbox{\footnotesize\eqref{frequent_ineq}}}{\leq}} + \langle(\e^{\beta(1-\e^{-t})(np'(\beta)- H_n(\sigma))}-1)^q\one_{\{\beta(1-\e^{-t})(np'(\beta)- H_n(\sigma)) > L\}}\rangle.
}
Assume $L\geq 2\beta T p'(\beta)$ so that whenever
\eq{
\beta(1-\e^{-t})\big(np'(\beta)- H_n(\sigma)\big) &> L \geq 2\beta Tp'(\beta) \stackrel{{\mbox{\footnotesize\eqref{frequent_ineq}}}}{\geq}2\beta (1-\e^{-t})n p'(\beta),
}
it follows that
\eq{
-\beta(1-\e^{-t}) H_n(\sigma) &> \beta (1-\e^{-t})np'(\beta) \\
\stackrefp{frequent_ineq}{\implies} \quad -2\beta(1-\e^{-t}) H_n(\sigma)&> \beta(1-\e^{-t})\big(np'(\beta)- H_n(\sigma)\big) > L\geq 0 \\
\stackrel{\mbox{\footnotesize\eqref{frequent_ineq}}}{\implies} \quad \phantom{{-2\beta(1-\e^{-t}) H_n(\sigma)}}\llap{$\displaystyle-\frac{2\beta T}{n} H_n(\sigma)$} &> \beta(1-\e^{-t})\big(np'(\beta)- H_n(\sigma)
\big) > L \geq 0.
}
We thus have
\eeq{ \label{next_3}
&\langle(\e^{\beta(1-\e^{-t})(np'(\beta)-H_n(\sigma))}-1)^q\one_{\{\beta(1-\e^{-t})(np'(\beta)- H_n(\sigma)) > L\}}\rangle \\
&\leq \langle \e^{-\frac{2q\beta T}{n} H_n(\sigma)}\one_{\{-\frac{2\beta T}{n} H_n(\sigma) > L\}}\rangle \\
&\leq \e^{-L}\langle \e^{-\frac{2(q+1)\beta T}{n} H_n(\sigma)}\rangle \\
&= \e^{-L}\frac{E_n[\e^{\beta(1-\frac{2(q+1)T}{n}) H_n(\sigma)}]}{E_n[\e^{\beta H_n(\sigma)}]} \\
&\leq \e^{-L} \frac{(E_n[\e^{\beta H_n(\sigma)}])^{1-\frac{2(q+1)T}{n}}}{E_n[\e^{\beta H_n(\sigma)}]} \\
&= \e^{-L} Z_n(\beta)^{-\frac{2(q+1)T}{n}}.
}
Combining \eqref{next_1}--\eqref{next_3}, we have now shown that 
\eq{ 
\E_{\vc h}[(X-X')^q]&\leq \langle f(\sigma)^2\rangle^{q/2}\Big[C(\beta,T,L,q)\Big\langle\Big|p'(\beta)-\frac{H_n(\sigma)}{n} \Big|\Big\rangle
+ C(\beta,T,q)\e^{-L}Z_n(\beta)^{-\frac{2(q+1)T}{n}}\Big].
}
Finally, given $\eps>0$, we choose $L$ large enough that $\e^{-L} \leq \eps$, thereby producing \eqref{bad_prep_var_bound}.
Then \eqref{bad_prep_var_bound_Y} is the special case when $f\equiv1$.
\end{proof}



\section{{Proof of Theorem~\EOref}} \label{proof_2}

In this section, we consider perturbations to the environment of the form
\eq{
\vc g^{(k)} \coloneqq \vc g + \frac{1}{\sqrt{n}}\sum_{j=1}^k \vc h^{(j)}, \quad k\geq0,
}
where the $\vc h^{(j)}$'s are independent copies of $\vc g$.
An important observation is that
\eeq{ \label{temperature_equivalence}
 \vc g^{(k)}\stackrel{\text{d}}{=} \sqrt{1+\frac{k}{n}}\,\vc g \quad \implies \quad
 \mu_{n, \vc g^{(k)}}^\beta \stackrel{\text{d}}{=} \mu_{n,\vc g}^{\beta\sqrt{1+\frac{k}{n}}}.
}
We will continue to use $\E$ to denote expectation with respect to $\vc g$ and the $\vc h^{(k)}$'s jointly, whereas $\E_{\vc h^{(k)}}$ will denote expectation with respect to $\vc h^{(k)}$ conditional on $\vc g$ and $\vc h^{(j)}$, $1 \leq j \leq k-1$.
As before, all statements involving $\E_{\vc h^{(k)}}$ and $\Var_{\vc h^{(k)}}$ are to be interpreted as almost sure statements.

As in Section~\ref{prep_section}, $\langle\cdot\rangle_\beta$ will denote expectation with respect to $\mu_{n,\vc g}^\beta$.
On the other hand, we will write $\llangle\cdot\rrangle_k$ to denote expectation under the measure $\mu_{n,\vc g^{(k)}}^\beta$, where the dependence on $\beta$ is understood.
That is,
\eeq{ \label{k_induction}
{\llangle f(\sigma)\rrangle}_{k} 
&\coloneqq \frac{E_n(f(\sigma)\e^{\beta(H_n(\sigma)+\frac{1}{\sqrt{n}}\sum_{j=1}^k\sum_ih_i^{(j)}\vphi_i)})}{E_n(\e^{\beta(H_n(\sigma)+\frac{1}{\sqrt{n}}\sum_{j=1}^k\sum_ih_i^{(j)}\vphi_i)})}  = \frac{{\llangle f(\sigma)\e^{\frac{\beta}{\sqrt{n}}\sum_ih_i^{(k)}\vphi_i}\rrangle}_{k-1}}{{\llangle \e^{\frac{\beta}{\sqrt{n}}\sum_ih_i^{(k)}\vphi_i}\rrangle}_{k-1}}.
}
For $\delta >0$, 
define the set
\eq{
\AA_{\delta,k} \coloneqq \Big\{\sigma^1\in\Sigma_n : \frac{1}{n}\sum_{i}\vphi_i(\sigma^1)\llangle\vphi_i(\sigma^2)\rrangle_k\leq\delta\Big\},
}
where $\AA_{\delta,0} = \AA_\delta$ is the set under consideration in Theorem~\ref{expected_overlap_thm}, whose proof will rely on Propositions~\ref{pre_iteration_1} and~\ref{pre_iteration_2} below.

\begin{prop} \label{pre_iteration_1}
For any $\delta_0>0$, there exists $n_0 = n_0(\delta_0)$ such that for all $n\geq n_0$, $k\geq 1$, and $\delta\geq\delta_0$,
\eq{
\E\llangle \one_{\AA_{\delta,k-1}}\rrangle_k
\leq \E\llangle\one_{\AA_{\delta^{1/4},k}}\rrangle_k + C(\beta)\delta.
}
\end{prop}

\begin{proof}
For any measurable $f:\Sigma_n\to[0,1]$, an application of \eqref{k_induction}, followed by Cauchy--Schwarz and Jensen's inequality, gives
\eq{
\llangle f(\sigma) \rrangle_k 
&\leq \frac{\sqrt{\llangle f(\sigma)^2\rrangle_{k-1}}\sqrt{\llangle\e^{\frac{2\beta}{\sqrt{n}}\sum_ih_i^{(k)}\vphi_i}\rrangle_{k-1}}}{\llangle\e^{\frac{\beta}{\sqrt{n}}\sum_ih_i^{(k)}\vphi_i}\rrangle_{k-1}} \\
&\leq \sqrt{\llangle f(\sigma)\rrangle_{k-1}}\sqrt{\llangle\e^{\frac{2\beta}{\sqrt{n}}\sum_ih_i^{(k)}\vphi_i}\rrangle_{k-1}}\llangle\e^{\frac{-\beta}{\sqrt{n}}\sum_ih_i^{(k)}\vphi_i}\rrangle_{k-1}.
}
So we define the random variable
\eq{
X \coloneqq \sqrt{2\llangle\e^{\frac{2\beta}{\sqrt{n}}\sum_ih_i^{(k)}\vphi_i}\rrangle_{k-1}}\llangle\e^{\frac{-\beta}{\sqrt{n}}\sum_ih_i^{(k)}\vphi_i}\rrangle_{k-1},
}
and consider, for fixed $\sigma^1$, the function $f_{\sigma^1}(\sigma^2) = 0\vee\frac{1}{n}\sum_i \vphi_i(\sigma^1)\vphi_i(\sigma^2)$.
By \eqref{unit_interval}, $f_{\sigma^1}$ is $[0,1]$-valued, and \eqref{positive_overlap} implies
\eq{
f_{\sigma^1}(\sigma^2) \leq \EEE_n + \frac{1}{n}\sum_i \vphi_i(\sigma^1)\vphi_i(\sigma^2).
}
So the above estimate shows
\eq{
\frac{1}{n}\sum_i \vphi_i(\sigma^1)\llangle\vphi_i(\sigma^2)\rrangle_k
\leq \llangle f_{\sigma^1}(\sigma^2)\rrangle_k
&\leq \frac{X}{\sqrt{2}}\sqrt{\EEE_n+\frac{1}{n}\sum_i \vphi_i(\sigma^1)\llangle\vphi_i(\sigma^2)\rrangle_{k-1}}.
}
In particular, when $n$ is sufficiently large that $\EEE_n\leq\delta$, 
\eq{
\one_{\AA_{\delta,k-1}}(\sigma^1)\frac{1}{n}\sum_i \vphi_i(\sigma^1)\llangle\vphi_i(\sigma^2)\rrangle_k\leq X\sqrt{\delta}.
}
We have thus shown
$\AA_{\delta,k-1} \subset \AA_{X\sqrt{\delta},k}$, which implies
\eq{
\E\llangle \one_{\AA_{\delta,k-1}}\rrangle_k \leq \E\llangle \one_{\AA_{X\sqrt{\delta},k}}\rrangle_k
\leq \E\llangle\one_{\AA_{t\sqrt{\delta},k}}\rrangle_k + \P(X > t) \quad \text{for any $t>0$},
}
where in the second inequality we have used the fact that if $\delta_1\leq\delta_2$, then $\AA_{\delta_1,k}\subset \AA_{\delta_2,k}$.
To handle the last term in the above display, we note that for any $p\geq1$,
\eq{
\P(X > t) 
= \P(X^p > t^p)
\leq t^{-p}\E(X^p)
&= t^{-p}2^{p/2}\E\big[\llangle\e^{\frac{2\beta}{\sqrt{n}}\sum_ih_i^{(k)}\vphi_i}\rrangle_{k-1}^{p/2}\llangle\e^{\frac{-\beta}{\sqrt{n}}\sum_ih_i^{(k)}\vphi_i}\rrangle_{k-1}^p] \\
&\leq t^{-p}2^{p/2}\sqrt{\E[\llangle\e^{\frac{2\beta}{\sqrt{n}}\sum_ih_i^{(k)}\vphi_i}\rrangle_{k-1}^p]\cdot\E[\llangle\e^{\frac{-\beta}{\sqrt{n}}\sum_ih_i^{(k)}\vphi_i}\rrangle_{k-1}^{2p}]} \\
&\leq t^{-p}2^{p/2}\sqrt{\E\llangle\e^{\frac{2\beta p}{\sqrt{n}}\sum_ih_i^{(k)}\vphi_i}\rrangle_{k-1}\cdot\E\llangle\e^{\frac{-2\beta p}{\sqrt{n}}\sum_ih_i^{(k)}\vphi_i}\rrangle_{k-1}}.
}
Now, for any $\theta\in\R$ and any $k\geq1$,
\eq{
\E\llangle \e^{\frac{\theta}{\sqrt{n}}\sum_ih_i^{(k)}\vphi_i}\rrangle _{k-1}
= \E\big[ \E_{\vc h^{(k)}}\llangle\e^{\frac{\theta}{\sqrt{n}}\sum_ih_i^{(k)}\vphi_i}\rrangle_{k-1}\big]
\stackrel{\mbox{\footnotesize\eqref{freq_identity}}}{=} \e^{\frac{\theta^2}{2}}.
}
Hence
\eq{
\P(X>t) \leq t^{-p}2^{p/2}\e^{2\beta^2p^2}.
}
Choosing $t=\delta^{-1/4}$ and $p = 4$, we arrive at
\eq{
\E\llangle \one_{\AA_{\delta,k-1}}\rrangle_k
\leq \E\llangle\one_{\AA_{\delta^{1/4},k}}\rrangle_k + C(\beta)\delta,
}
which holds for all $n$ such that $\EEE_n\leq\delta$.
\end{proof}

Next we consider the event
\eq{
B_{\delta,k} \coloneqq \Big\{\frac{1}{n}\sum_{i} \llangle\vphi_i\rrangle_{k}^2 \leq \delta\Big\},
}
where $B_{\delta,0} = B_\delta$ is the event under consideration in Theorem~\ref{averages_squared}.

\begin{lemma} \label{alpha_lemma}
Assume $\beta$ is a point of differentiability for $p(\cdot)$, and $p'(\beta) < \beta$.
For any $\eps>0$, there is $\delta = \delta(\beta,\eps) >0$ sufficiently small that for any positive constant $K$, the following is true.
If $k(n) \in \{0,1,\dots,K\}$ for all $n$, then
\eeq{ \label{alpha_choice}
\limsup_{n\to\infty} \P(B_{\delta,k(n)}) \leq \eps.
}
\end{lemma}

\begin{proof}
By Theorem~\ref{averages_squared}, there is $\delta>0$ sufficiently small that
\eeq{ \label{alpha_choice_2}
\limsup_{n\to\infty} \P(B_{2\delta,0}) \leq \eps.
}
Let us write $\beta_n \coloneqq \beta\sqrt{1+\frac{k(n)}{n}}$, and
then observe that 
\eeq{ \label{alpha_choice_3}
\P(B_{\delta,k(n)}) 
&\stackrel{\phantom{\mbox{\footnotesize\eqref{temperature_equivalence}}}}{=} \P\Big(\frac{1}{n}\sum_i \llangle\vphi_i\rrangle_{k(n)}^2 \leq \delta\Big) \\
&\stackrel{\mbox{\footnotesize\eqref{temperature_equivalence}}}{=} \P\Big(\frac{1}{n}\sum_i \langle\vphi_i\rangle_{\beta_n}^2 \leq \delta\Big)
\leq \P(B_{2\delta,0})+\P\bigg(\Big|\frac{1}{n}\sum_i \langle\vphi_i\rangle_{\beta_n}^2 -\frac{1}{n}\sum_i\langle\vphi_i\rangle_{\beta}^2\Big|\geq \delta\bigg). \raisetag{3.5\baselineskip}
}
Since $\sqrt{1+\frac{k(n)}{n}} \leq 1 + \frac{k(n)}{n} \leq 1 + \frac{K}{n}$, we have $0\leq \beta_n - \beta \leq \frac{\beta K}{n}$, and thus Lemma~\ref{connecting_betas}(c) gives
\eq{
\Big|\frac{1}{n}\sum_i \langle\vphi_i\rangle_{\beta_n}^2 -\frac{1}{n}\sum_i\langle\vphi_i\rangle_{\beta}^2\Big| \leq 2\sqrt{\beta K}\sqrt{F_n'(\beta_n) - F_n'(\beta)}.
}
By Lemma~\ref{betas_converging}, the right-hand side above converges to $0$ almost surely as $n\to\infty$.
In particular,
\eq{
\lim_{n\to\infty} \P\bigg(\Big|\frac{1}{n}\sum_i \langle\vphi_i\rangle_{\beta_n}^2 -\frac{1}{n}\sum_i\langle\vphi_i\rangle_{\beta}^2\Big|\geq \delta\bigg) = 0,
}
and so \eqref{alpha_choice} follows from \eqref{alpha_choice_2} and \eqref{alpha_choice_3}.
\end{proof}

\begin{prop} \label{pre_iteration_2}
Given any $\alpha>0$, there are positive constants $C_1(\alpha,\beta)$ and $C_2(\beta)$ 
such that the following holds for any $\delta_0\in(0,1)$.
There exists $n_0 = n_0(\delta_0)$ so that for every $n\geq n_0$, $k\geq1$, and $\delta\in[\delta_0,1)$,
\eq{
\E_{\vc h^{(k)}}\llangle \one_{\AA_{\delta,k-1}}\rrangle_k \geq \llangle \one_{\AA_{\delta,k-1}}\rrangle_{k-1} + C_1(\alpha,\beta)\llangle\one_{\AA_{\delta,k-1}}\rrangle_{k-1}\one_{B_{\alpha,k-1}^\mathrm{c}} - C_2(\beta)\sqrt{\delta}.
}
\end{prop}
\begin{proof}
Let $\delta_0\in(0,1)$ be given, and take $n_0$ such that $\EEE_n\leq\delta_0/2$ for all $n\geq n_0$.
Consider any $\delta\in[\delta_0,1)$,
and define the random variables
\eq{
X &\coloneqq \llangle \e^{\frac{\beta}{\sqrt{n}}\sum_ih_i^{(k)}\vphi_i}\rrangle_{k-1}, \\
X_1 &\coloneqq \llangle \one_{\AA_{\delta,k-1}}\e^{\frac{\beta}{\sqrt{n}}\sum_i h_i^{(k)}\vphi_i}\rrangle_{k-1}, \\
X_2 &\coloneqq \llangle \one_{\AA_{\delta,k-1}^\mathrm{c}}\e^{\frac{\beta}{\sqrt{n}}\sum_i h_i^{(k)}\vphi_i}\rrangle_{k-1}, \\
Y_1 &\coloneqq \E_{\vc h^{(k)}} X_1 \stackrel{\mbox{\footnotesize\eqref{freq_identity}}}{=} \e^{\frac{\beta^2}{2}}\llangle \one_{\AA_{\delta,k-1}}\rrangle_{k-1}, \\
Y_2 &\coloneqq \E_{\vc h^{(k)}} X_2 \stackrel{\mbox{\footnotesize\eqref{freq_identity}}}{=} \e^{\frac{\beta^2}{2}}\llangle \one_{\AA_{\delta,k-1}^\mathrm{c}}\rrangle_{k-1}.
}

\noindent {\bf Step 1.} \textit{Show that $X_1$ is concentrated at $Y_1$, but $X_2$ is not concentrated at $Y_2$ when $B_{\alpha,k-1}^\cc$ occurs.}

First observe that for any $\theta \in (-\infty,0] \cup [1,\infty)$, Jensen's inequality implies
\eeq{ \label{gaussian_bound}
\E_{\vc h^{(k)}} X^\theta 
\leq \E_{\vc h^{(k)}}\llangle \e^{\frac{\theta\beta}{\sqrt{n}}\sum_i h_i^{(k)}\vphi_i}\rrangle_{k-1} 
\stackrel{\mbox{\footnotesize\eqref{freq_identity}}}{=} \e^{\frac{(\theta\beta)^2}{2}}.
}
In particular, for any $t > \e^{\frac{\beta^2}{2}}\geq Y_2$,
\eeq{ \label{upper_X4_truncated}
\E_{\vc h^{(k)}}[(X_2 - Y_2)^2\one_{\{X_2>t\}}]
&\leq \frac{\E_{\vc h^{(k)}}[(X_2 - Y_2)^4\one_{\{X_2>t\}}]}{(t-\e^{\frac{\beta^2}{2}})^2} \\
&\leq \frac{\E_{\vc h^{(k)}}(X_2^4)}{(t-\e^{\frac{\beta^2}{2}})^2}
\leq \frac{\E_{\vc h^{(k)}}(X^4)}{(t-\e^{\frac{\beta^2}{2}})^2}
\stackrel{\mbox{\footnotesize\eqref{gaussian_bound}}}{\leq} \frac{\e^{8\beta^2}}{(t-\e^{\frac{\beta^2}{2}})^2}.
}
On the other hand,
\eeq{ \label{lower_X4_prep}
\Var_{\vc h^{(k)}}(X_2)
= \Var_{\vc h^{(k)}}(X-X_1)
&= \Var_{\vc h^{(k)}}(X) - 2\Cov_{\vc h^{(k)}}(X,X_1) + \Var_{\vc h^{(k)}}(X_1) \\
&\geq \Var_{\vc h^{(k)}}(X) - 2\sqrt{\Var_{\vc h^{(k)}}(X)\Var_{\vc h^{(k)}}(X_1)}.
}
We have the upper bound
\eeq{
\Var_{\vc h^{(k)}}(X) \leq \E_{\vc h^{(k)}}(X^2) \stackrel{\mbox{\footnotesize\eqref{gaussian_bound}}}{\leq} \e^{2\beta^2}, \label{upper_X2}
}
as well as the lower bound
\eeq{ \label{lower_X2}
\Var_{\vc h^{(k)}}(X)
&\stackrel{\phantom{\mbox{\footnotesize\eqref{freq_identity}}}}{=}
\E_{\vc h^{(k)}}\llangle \e^{\frac{\beta}{\sqrt{n}}\sum_i h_i^{(k)}(\vphi_i(\sigma^1) + \vphi_i(\sigma^2))}\rrangle_{k-1} - \big(\E_{\vc h^{(k)}} \llangle \e^{\frac{\beta}{\sqrt{n}}\sum_i h_i^{(k)}\vphi_i}\rrangle_{k-1}\big)^2 \\
&\stackrel{\mbox{\footnotesize\eqref{freq_identity}}}{=}\e^{\beta^2}\big(\llangle\e^{\frac{\beta^2}{n}\sum_i \vphi_i(\sigma^1)\vphi_i(\sigma^2)}\rrangle_{k-1} - 1\big) \\
&\stackrel{\phantom{\mbox{\footnotesize\eqref{freq_identity}}}}{\geq} \e^{\beta^2}\big(\e^{\frac{\beta^2}{n}\sum_i \llangle \vphi_i\rrangle_{k-1}^2}-1\big) \\
&\stackrel{\phantom{\mbox{\footnotesize\eqref{freq_identity}}}}{\geq} \e^{\beta^2}\frac{\beta^2}{n}\sum_i\llangle \vphi_i\rrangle_{k-1}^2. 
\raisetag{4\baselineskip}
}
Meanwhile, we have $\EEE_n\leq\delta_0/2\leq\delta/2$ for all $n\geq n_0$.
Hence Lemma~\ref{h_variance_lemma}(b) implies
\eeq{ \label{upper_X3}
\Var_{\vc h^{(k)}}(X_1) 
&\leq \e^{2\beta^2}\Big(\Big\llangle \one_{\AA_{\delta,k-1}}(\sigma)\frac{1}{n}\sum_i\vphi_i\llangle\vphi_i\rrangle_{k-1}\Big\rrangle_{k-1}+2\EEE_n\Big) \\ 
&\leq 2\e^{2\beta^2}\delta \quad \text{for all $n\geq n_0$.}
}
Using \eqref{upper_X2}--\eqref{upper_X3} in \eqref{lower_X4_prep} yields
\eeq{ \label{lower_X4_almost}
\Var_{\vc h^{(k)}}(X_2) \geq \beta^2\e^{\beta^2}\frac{1}{n}\sum_i\llangle \vphi_i\rrangle_{k-1}^2- 2\e^{2\beta^2}\sqrt{2\delta} \quad \text{for all $n\geq n_0$.}
}
So on the event $B_{\alpha,k-1}^\mathrm{c} = \{\frac{1}{n}\sum_{i} \llangle\vphi_i\rrangle_{k-1}^2 > \alpha\}$,
 \eqref{lower_X4_almost} shows
\eeq{ \label{lower_X4}
\Var_{\vc h^{(k)}}(X_2)\one_{B_{\alpha,k-1}^\mathrm{c}} \geq (\beta^2\e^{\beta^2}\alpha - 2\e^{2\beta^2}\sqrt{2\delta})\one_{B_{\alpha,k-1}^\mathrm{c}}
}
for all $n\geq n_0$.
Given $\alpha$ and $\beta$, we fix $t = t(\alpha,\beta)$ large enough such that 
\begin{subequations} \label{t_choices}
\begin{align}
\label{t_lower}
t > \e^{\frac{\beta^2}{2}} &\geq \max(Y_1,Y_2) \hspace{0.4in}
\intertext{and}
\label{t_upper}
\frac{\e^{8\beta^2}}{(t-\e^{\frac{\beta^2}{2}})^2} &\leq \frac{1}{2}\beta^2\e^{\beta^2}\alpha.
\end{align}
\end{subequations}
Because of \eqref{t_upper}, the inequalities \eqref{upper_X4_truncated} and \eqref{lower_X4} together yield
\eeq{ \label{lower_X4_truncated}
\E_{\vc h^{(k)}}[(X_2-Y_2)^2\one_{\{X_2\leq t\}}]\one_{B_{\alpha,k-1}^\mathrm{c}}
&= \big(\Var_{\vc h^{(k)}}(X_2) - \E[(X_2-Y_2)^2\one_{\{X_2>t\}}]\big)\one_{B_{\alpha,k-1}^\mathrm{c}} \\
&\geq \Big(\frac{1}{2}\beta^2\e^{2\beta^2}\alpha - 2\e^{2\beta^2}\sqrt{2\delta}\Big)\one_{B_{\alpha,k-1}^\mathrm{c}}\\
&=(C_1(\alpha,\beta) - C_2(\beta)\sqrt{\delta})\one_{B_{\alpha,k-1}^\mathrm{c}} \quad \text{for all $n\geq n_0$.}
}

\noindent {\bf Step 2.} \textit{Since $X_1 \approx Y_1$, obtain an upper bound on the error in the following approximation:}
\eq{
\E_{\vc h^{(k)}}\Big(\frac{X_1}{X_1+X_2}\Big)
\approx \E_{\vc h^{(k)}}\Big(\frac{Y_1}{Y_1+X_2}\Big).
}

Simple algebra gives
\eq{
\frac{X_1}{X_1+X_2} - \frac{Y_1}{Y_1+X_2} 
&= \frac{X_2(X_1-Y_1)}{(X_1+X_2)(Y_1+X_2)} 
= \frac{X_2(X_1-Y_1)}{X(Y_1+X_2)},
}
and
\eeq{ \label{approximation_error_bound}
\Big|\E_{\vc h^{(k)}}\Big(\frac{X_2(X_1-Y_1)}{X(Y_1+X_2)}\Big)\Big|
&\stackrel{\phantom{\mbox{\footnotesize\eqref{upper_X3},\eqref{gaussian_bound}}}}{\leq} \E_{\vc h^{(k)}} \Big(\frac{|X_1-Y_1|}{X}\Big) \\
&\stackrel{\phantom{\mbox{\footnotesize\eqref{upper_X3},\eqref{gaussian_bound}}}}{\leq} \E_{\vc h^{(k)}}(X^{-2})\sqrt{\Var_{\vc h^{(k)}}(X_1)} \\
&\stackrel{\mbox{\footnotesize\eqref{gaussian_bound},\eqref{upper_X3}}}{\leq} C(\beta)\sqrt{\delta} \quad \text{for all $n\geq n_0$.}\\
}

\noindent {\bf Step 3.} \textit{Since $X_2$ is not concentrated at $Y_2$ when $B_{\alpha,k-1}^\cc$ occurs,}
\textit{obtain a lower bound on the gap in the following application of Jensen's inequality:}
\eq{
\E_{\vc h^{(k)}}\Big(\frac{Y_1}{Y_1+X_2}\Big) = \frac{Y_1}{Y_1+Y_2} + (\text{Jensen gap}).
}

We consider the function $f : (-Y_1,\infty) \to [0,1]$ given by
\eq{
f(x) \coloneqq \frac{Y_1}{Y_1+x}, \quad \text{for which} \quad f''(x) = \frac{2Y_1}{(Y_1+x)^3} \geq 0.
}
In particular, we consider its Taylor series approximation about $Y_2$,
\eq{
f(x) = f(Y_2) + (x-Y_2)f'(Y_2) + \frac{(x-Y_2)^2}{2}f''(\xi_x),
}
where $\xi_x$ belongs to the interval between $x$ and $Y_2$.
We note that such an expansion exists because the identity $Y_1 + Y_2 = \e^{\frac{\beta^2}{2}}$ shows $Y_2 > -Y_1$.
Jensen's inequality implies
\eq{
\E_{\vc h^{(k)}}f(X_2) \geq f(\E_{\vc h^{(k)}}X_2) = f(Y_2)
= \frac{Y_1}{Y_1+Y_2} = \llangle \one_{\AA_{\delta,k-1}}\rrangle_{k-1}.
}
We will now produce a lower bound on the Jensen gap.

First observe that $f''$ is decreasing on $(-Y_1,\infty)$.
Consequently, if $x \in [Y_2,t]$, then $f''(\xi_{x}) \geq f''(x) \geq f''(t)$.
Similarly, if $x\leq Y_2$, then $f''(\xi_{x}) \geq f''(Y_2) \geq f''(t)$.
Therefore, for all $n\geq n_0$, we have
\eeq{ \label{jensen_gap}
\E_{\vc h^{(k)}}f(X_2) - \llangle \one_{\AA_{\delta,k-1}}\rrangle_{k-1}
&\stackrel{\phantom{\mbox{\footnotesize\eqref{t_lower},\eqref{t_choices}}}}{=} \E_{\vc h^{(k)}}f(X_2) - f(Y_2) \\
&\stackrel{\phantom{\mbox{\footnotesize\eqref{t_lower},\eqref{t_choices}}}}{=} \frac{\E_{\vc h^{(k)}}[(X_2-Y_2)^2f''(\xi_{X_2})]}{2} \\
&\stackrel{\phantom{\mbox{\footnotesize\eqref{t_lower},\eqref{t_choices}}}}{\geq} \frac{f''(t)}{2}\E_{\vc h^{(k)}}[(X_2-Y_2)^2\one_{\{X_2\leq t\}}] \\
&\stackrel{\phantom{\mbox{\footnotesize\eqref{t_lower},\eqref{t_choices}}}}{\geq} \frac{Y_1}{(Y_1+t)^3} \E_{\vc h^{(k)}}[(X_2-Y_2)^2\one_{\{X_2\leq t\}}]\one_{B_{\alpha,k-1}^\mathrm{c}} \\
&\stackrel{\mbox{\footnotesize\eqref{t_lower},\eqref{lower_X4_truncated}}}{\geq} \frac{Y_1}{8t^3}(C_1(\alpha,\beta) - C_2(\beta)\sqrt{\delta})\one_{B_{\alpha,k-1}^\mathrm{c}}\\
&\stackrel{\phantom{\mbox{\footnotesize\eqref{t_lower},\eqref{t_choices}}}}{\geq}C_1(\alpha,\beta)\llangle\one_{\AA_{\delta,k-1}}\rrangle_{k-1}\one_{B_{\alpha,k-1}^\mathrm{c}} - C_2(\beta)\sqrt{\delta}, \raisetag{5.5\baselineskip}
}
where the second term in the final expression need not depend on $\alpha$ since $Y_1/(8t^3) \leq 1$. \\

\noindent {\bf Step 4.} \textit{Reckon the final bound.}

In summary, for all $n\geq n_0$,
\eq{
\E_{\vc h^{(k)}}\llangle \one_{\AA_{\delta,k-1}}\rrangle_k 
&\stackrel{\mbox{\footnotesize\eqref{k_induction}}}{=} \E_{\vc h^{(k)}}\Big(\frac{X_1}{X_1+X_2}\Big) \\
&\stackrel{\mbox{\footnotesize\eqref{approximation_error_bound}}}{\geq} \E_{\vc h^{(k)}}\Big(\frac{Y_1}{Y_1+X_2}\Big) - C(\beta)\sqrt{\delta} \\
&\stackrel{\phantom{\mbox{\footnotesize\eqref{approximation_error_bound}}}}{=} \E_{\vc h^{(k)}} f(X_2) - C(\beta)\sqrt{\delta} \\
&\stackrel{\mbox{\footnotesize\eqref{jensen_gap}}}{\geq} \llangle \one_{\AA_{\delta,k-1}}\rrangle_{k-1} + C_1(\alpha,\beta)\llangle\one_{\AA_{\delta,k-1}}\rrangle_{k-1}\one_{B_{\alpha,k-1}^\mathrm{c}} - C_2(\beta)\sqrt{\delta}.
}
\end{proof}

\begin{proof}[Proof of Theorem~\ref{expected_overlap_thm}]
Let $\eps> 0$ be given.
From Lemma~\ref{alpha_lemma}, we fix $\alpha = \alpha(\beta,\eps)>0$ so that for any bounded sequence $(k(n))_{n\geq1}$ of nonnegative integers, we have 
\eeq{ \label{alpha_choice_real}
\limsup_{n\to\infty} \P(B_{\alpha,k(n)}) \leq \frac{\eps}{2}.
}
We wish to find $\delta_* > 0$, depending only on $\beta$ and $\eps$, such that $\E\llangle \one_{\AA_{\delta_*}}\rrangle \leq \eps$.

Let $\delta_0 \in (0,1)$, its exact value to be decided later.
From Proposition~\ref{pre_iteration_2}, we know that for all $n \geq n_0 = n_0(\delta_0)$ and $\delta\in[\delta_0,1)$,
\eq{
\E\llangle \one_{\AA_{\delta,k-1}}\rrangle_k
&\geq \E\llangle \one_{\AA_{\delta,k-1}}\rrangle_{k-1} + C_1(\beta,\eps)\E(\llangle\one_{\AA_{\delta,k-1}}\rrangle_{k-1}\one_{B_{\alpha,k-1}^\mathrm{c}}) - C_2(\beta)\sqrt{\delta}.
}
And from Proposition~\ref{pre_iteration_1}, we can assume
\eq{
\E\llangle \one_{\AA_{\delta,k-1}}\rrangle_k
\leq \E\llangle\one_{\AA_{\delta^{1/4},k}}\rrangle_k + C(\beta)\delta \quad \text{for all $n\geq n_0$, $\delta\in[\delta_0,1)$.}
}
Linking the two inequalities, we find that
\eq{
\E\llangle\one_{\AA_{\delta^{1/4},k}}\rrangle_k 
&\geq \E\llangle \one_{\AA_{\delta,k-1}}\rrangle_{k-1} + \mathbf C_1(\beta,\eps)\E(\llangle\one_{\AA_{\delta,k-1}}\rrangle_{k-1}\one_{B_{\alpha,k-1}^\mathrm{c}}) - \mathbf C_2(\beta)\sqrt{\delta},
} 
where now we fix the constants $\mathbf C_1(\beta,\eps)$ and $\mathbf C_2(\beta)$.
Note that $\delta_0\leq\delta\leq\delta^{1/4}<1$, and so this reasoning can be iterated.
Iterating $K$ times produces the estimate
\eq{ 
1\geq\E\llangle \one_{\AA_{\delta^{1/4^K},K}}\rrangle_{K}
&\geq \sum_{k=0}^{K-1}\Big[\mathbf C_1(\beta,\eps)\E(\llangle \one_{\AA_{\delta^{1/4^k},k}}\rrangle_k\one_{B_{\alpha,k}^\mathrm{c}}) - \mathbf C_2(\beta)\sqrt{\delta^{1/4^k}}\ \Big]
+ \E\llangle \one_{\AA_{\delta,0}}\rrangle_0,
}
which implies the existence of some $k=k(n)\in\{0,1,\dots,K-1\}$ such that 
\eeq{ \label{k_existence}
\mathbf C_1(\beta,\eps)\E(\llangle \one_{\AA_{\delta^{1/4^k},k}}\rrangle_k\one_{B_{\alpha,k}^\mathrm{c}}) - \mathbf C_2(\beta)\sqrt{\delta^{1/4^k}} &\leq \frac{1}{K}.
}
So we take $K = K(\beta,\eps)$ large enough that 
\eeq{ \label{K_choice}
\frac{1}{\mathbf C_1(\beta,\eps)K} \leq \frac{\eps}{6},
}
and then choose $\delta_0 = \delta_0(\beta,K)$ small enough that 
\eeq{ \label{eps_choice}
\mathbf C_2(\beta)\sqrt{\delta_0^{1/4^K}} \leq \frac{1}{K}.
}
We now have, for all $n\geq n_0$,
\eq{
\E(\llangle \one_{\AA_{\delta_0^{1/4^k},k}}\rrangle_k\one_{B_{\alpha,k}^\mathrm{c}}) &\stackrel{\mbox{\footnotesize\eqref{k_existence}}}{\leq} \frac{1}{\mathbf C_1(\beta,\eps)}\Big(\frac{1}{K} + \mathbf C_2(\beta)\sqrt{\delta_0^{1/4^K}}\Big)\\
&\stackrel{\mbox{\footnotesize\eqref{eps_choice}}}{\leq} \frac{2}{\mathbf C_1(\beta,\eps)K} 
\stackrel{\mbox{\footnotesize\eqref{K_choice}}}{\leq}\frac{\eps}{3}.
}
Combining this bound with \eqref{alpha_choice_real}, we see that
\eeq{ \label{eps_0_bound}
\E \llangle \one_{\AA_{\delta_0^{1/4^k},k}}\rrangle_k 
\leq \E(\llangle \one_{\AA_{\delta_0^{1/4^k},k}}\rrangle_k\one_{B_{\alpha,k}^\mathrm{c}}) + \P(B_{\alpha,k}) \leq \eps \quad \forall\text{ large $n$.}
}
To now complete the proof, we must obtain from this result an analogous one with $k=0$. 

As in the proof of Lemma~\ref{alpha_lemma}, we will write $\beta_n \coloneqq \beta\sqrt{1+\frac{k}{n}}$.
For $\eta>0$, define the set
\eq{
\wt{\AA}_{\eta,k} \coloneqq \Big\{\sigma^1\in\Sigma_n : \sum_i \vphi_i(\sigma^1)\langle \vphi_i(\sigma^2)\rangle_{\beta_n}\leq \eta \Big\}.
}
It follows from \eqref{temperature_equivalence} that
\eeq{ \label{expectations_same}
\llangle\one_{\AA_{\eta,k}}\rrangle_k \stackrel{\text{d}}{=} \langle\one_{\wt{\AA}_{\eta,k}}\rangle_{\beta_n} \quad \text{for any $\eta>0$,}
}
Since $0\leq \beta_n-\beta \leq \frac{\beta K}{n}$, 
Lemma~\ref{connecting_betas}(b) implies
\eq{
\Big|\frac{1}{n}\sum_{i} \vphi_i\langle\vphi_i\rangle_{\beta_n} - \frac{1}{n}\sum_i\vphi_i\langle\vphi_i\rangle_\beta\Big|
&\leq \sqrt{\beta K}\sqrt{F_n'(\beta_n) - F_n'(\beta)}.
}
Denote the right-hand side above by $\Delta_n$.
Take $\delta_* \coloneqq \frac{1}{2}{\delta_0} \leq \frac{1}{2}\delta_0^{1/4^k}$.
From the above display,
$\AA_{\delta_*,0} \subset \wt{\AA}_{\delta_*+\Delta_n,k}.$
Hence
\eq{
\E\langle\one_{\AA_{\delta_*,0}}\rangle_\beta
&\stackrel{\phantom{\mbox{\footnotesize\eqref{expectations_same}}}}{\leq}  \E\langle\one_{\wt{\AA}_{\delta_*+\Delta_n,k}}\rangle_\beta \\
& \stackrel{\phantom{\mbox{\footnotesize\eqref{expectations_same}}}}{\leq}   \P(\Delta_n > \delta_*) + \E \langle\one_{\wt{\AA}_{2\delta_*,k}} \rangle_\beta \\
&\stackrel{\mbox{\footnotesize\eqref{expectations_same}}}{=}  \P(\Delta_n > \delta_*) + \E \langle\one_{\wt{\AA}_{2\delta_*,k}} \rangle_\beta -
\E \langle\one_{\wt{\AA}_{2\delta_*,k}} \rangle_{\beta_n} +\E\llangle\one_{\AA_{2\delta_*,k}}\rrangle_k \\
& \stackrel{\phantom{\mbox{\footnotesize\eqref{expectations_same}}}}{\leq}  \P(\Delta_n > \delta_*) + \E \langle\one_{\wt{\AA}_{2\delta_*,k}} \rangle_\beta -
\E \langle\one_{\wt{\AA}_{2\delta_*,k}} \rangle_{\beta_n} +\E\llangle\one_{\AA_{\delta_0^{1/4^k},k}}\rrangle_k.
}
And by Lemma~\ref{connecting_betas}(a),
\eq{
 | \langle\one_{\wt{\AA}_{2\delta_*,k}} \rangle_{\beta_n}
 -  \langle\one_{\wt{\AA}_{2\delta_*,k}} \rangle_\beta |
  \leq \Delta_n.
}
From the previous two displays and \eqref{eps_0_bound}, we have
\eq{
\E\langle\one_{\AA_{\delta_*,0}}\rangle_\beta \leq \P(\Delta_n > \delta_*) + \E(\Delta_n) + \eps \quad \text{for all large $n$}.
}
Finally, Lemma~\ref{betas_converging} shows that $\Delta_n \to 0$ almost surely and in $L^1$ as $n\to\infty$.
Consequently,
$\limsup_{n\to\infty} \E\langle\one_{\AA_{\delta_*,0}}\rangle_\beta \leq \eps$.
\end{proof}

\section{{Proof of equivalence of Theorems \ECref~and \EOref}} \label{cor_proof}
Theorem~\ref{easy_cor} is implied by Theorem~\ref{expected_overlap_thm} once we establish the following result.
Recall the definitions \eqref{ball_def} and \eqref{A_def}.

\begin{prop}
Suppose $H_n$ is defined by \eqref{field_decomposition},  where $(g_i)_{i=1}^\infty$ are i.i.d. random variables with zero mean and unit variance (not necessarily Gaussian).
Assume \eqref{free_energy_assumption}, \eqref{variance_assumption}, and \eqref{positive_overlap}. 
Then the following two statements are equivalent:
\begin{itemize}
\item[$\mathrm{(S1)}$] For every $\eps > 0$, there exist integers $k = k(\beta,\eps)$ and $n_0 = n_0(\beta,\eps)$ and a number $\delta = \delta(\beta,\eps)>0$ such that the following is true for all $n\geq n_0$.
With $\P$-probability at least $1-\eps$, there exist $\sigma^1,\dots,\sigma^k\in\Sigma_n$ such that 
\eq{
\mu_{n}^\beta\Big(\bigcup_{j=1}^k \BB(\sigma^j, \delta)\Big) \geq 1 - \eps.
}
\item[$\mathrm{(S2)}$] For every $\eps>0$, there exists $\delta = \delta(\beta,\eps) > 0$ sufficiently small that
\eq{ 
\limsup_{n\to\infty} \E\langle\one_{\AA_{n,\delta}}\rangle \leq \eps.
}
\end{itemize}
\end{prop}

\subsection{Proof of {$\mathrm{(S2)} \implies \mathrm{(S1)}$}}
Let  $\eps > 0$ be given.
By $\mathrm{(S2)}$, we can choose $\delta > 0$ small enough and $n_0$ large enough  so that
\eq{
\E\langle\one_{\AA_{n,2\delta}}\rangle \leq \frac{\eps^2}{2} \quad \text{for all $n\geq n_0$}.
}
It follows from Markov's inequality that
\eeq{ \label{Markov_for_cor}
\P\Big(\langle\one_{\AA_{n,2\delta}}\rangle > \frac{\eps}{2}\Big) \leq \eps.
}
Now, by the Paley--Zygmund inequality, for any $j\neq k+1$,
\eq{
\givena{\one_{\{\RR_{j,k+1}\geq\delta\}}}{\sigma^{k+1}}\one_{\{\RR(\sigma^{k+1}) > 2\delta\}} 
\geq\frac{1}{4}\frac{\RR(\sigma^{k+1})^2}{\givena{\RR_{j,k+1}^2}{\sigma^{k+1}}}\one_{\{\RR(\sigma^{k+1}) > 2\delta\}}
&\geq\delta^2\one_{\{\RR(\sigma^{k+1})>2\delta\}}.
}
Therefore,
\eq{
\givena{\one_{\bigcap_{j=1}^k\{\RR_{j,k+1}<\delta\}}}{\sigma^{k+1}}\one_{\{\RR(\sigma^{k+1}) > 2\delta\}}  \leq (1 - {\delta^2})^k \leq \e^{-\delta^2 k}.
}
Choosing $k = \ceil{-\delta^{-2}\log(\eps/2)} \vee 0$, we have
\eq{
\langle\one_{\bigcap_{j=1}^k\{\RR_{j,k+1}<\delta\}}\rangle 
\leq \frac{\eps}{2} + \langle\one_{\{\RR(\sigma^{k+1})\leq2\delta\}}\rangle
= \frac{\eps}{2} + \langle \one_{\AA_{n,2\delta}}\rangle.
}
Therefore, 
\eq{
\P\Big(\langle \one_{\bigcup_{j=1}^k\{\RR_{j,k+1}\geq\delta\}}\rangle \geq 1-\eps\Big)
&= \P\Big(\langle \one_{\bigcap_{j=1}^k\{\RR_{j,k+1}<\delta\}}\rangle \leq \eps\Big) 
\geq \P\Big(\langle \one_{\AA_{n,2\delta}}\rangle \leq \frac{\eps}{2}\Big) 
\stackrel{\mbox{\footnotesize\eqref{Markov_for_cor}}}{\geq} 1 - \eps.
}
This completes the proof, since
\eq{
\mu_{n}^\beta\Big(\bigcup_{j=1}^k \BB(\sigma^j,\delta)\Big) = \langle \one_{\bigcup_{j=1}^k\{\RR_{j,k+1}\geq\delta\}}\rangle.
}

\subsection{Proof of {$\mathrm{(S1)} \implies \mathrm{(S2)}$}}

We begin with a lemma that roughly states the following. 
If many 
random variables each have non-negligible positive correlation with a distinguished variable, then at least one pair of these variables has non-negligible positive correlation.

\begin{lemma} \label{one_to_pair}
For any $\delta\in(0,1]$, there exists $N_0 = N_0(\delta)$ such that the following holds for any integer $N\geq N_0$ and any $\sigma^0\in\Sigma_n$.
If $\sigma^1,\dots,\sigma^N\in\BB(\sigma^0,\delta)\subset\Sigma_n$, then
\eeq{ \label{one_to_pair_ineq}
\RR_{j,k} \geq \frac{\delta^2}{2} \quad \text{for some $1\leq j<k\leq N$}.
}
\end{lemma}

\begin{proof}
Consider the $(N+1)\times(N+1)$ matrix $\RR = (\RR_{j,k})_{0\leq i,j\leq N}$, where
\eq{
\RR_{j,k} = \RR(\sigma^j,\sigma^k) = \frac{1}{n}\sum_i \vphi_i(\sigma^j)\vphi_i(\sigma^k).
}
Observe that $\RR$ is positive semi-definite: for any $\vc x \in \R^{N+1}$,
\eq{
\langle \vc x,\RR \vc x\rangle = \sum_{0\leq j,k\leq N} \RR_{j,k}x_jx_k
&= \frac{1}{n}\sum_i \sum_{0\leq j,k\leq N} x_j\vphi_{i}(\sigma^j)x_k\vphi_{i}(\sigma^k) 
= \frac{1}{n}\sum_i \bigg(\sum_{j=0}^N x_j\vphi_i(\sigma^j)\bigg)^2 \geq 0.
}
Now let $\eta \coloneqq 0\vee\max_{1\leq j < k\leq N} \RR_{j,k}$. 
For $\vc x = (1,-x,\dots,-x) \in\R^{1+N}$ with $x\geq0$, our assumptions give 
\eq{
0 \leq \iprod{\vc x}{\RR \vc x} \leq 1 + Nx^2 - 2\delta N x + \eta N^2 x^2.
}
We now take $x = \delta/(1+\eta N)$ to obtain
\eq{
0 &\leq 1 + N\Big(\frac{\delta}{1+\eta N}\Big)^2 - 2\delta N\frac{\delta}{1+\eta N} + \eta N^2\Big(\frac{\delta}{1+\eta N}\Big)^2 \\
&= 1 + \frac{\delta^2}{1+\eta N}\Big[\frac{N}{1+\eta N} - 2N + \frac{\eta N^2}{1+\eta N}\Big] = 1-\frac{\delta^2N}{1+\eta N}.
} 
Supposing that $\eta < \delta^2/2$, we further see
\eq{
0 \leq 1 - \frac{\delta^2N}{1+\eta N} \leq 1 - \frac{\delta^2N}{1+\delta^2 N/2},
}
which yields a contradiction as soon as $\frac{\delta^2N}{1+\delta^2 N/2} > 1$.
\end{proof}

We will contrast Lemma~\ref{one_to_pair} with the one below, which says that if $\delta$ is small enough, then any non-negligible subset of $\AA_{n,\delta}$ has many nearly orthogonal elements.

\begin{lemma} \label{many_orthogonal}
For any $\eps_1,\eps_2>0$ and positive integer $N$, there is $\delta = \delta(\eps_1,\eps_2,N) > 0$ such that the following holds. 
If $\AA\subset\AA_{n,\delta}$ with $\langle \one_{\AA}\rangle \geq \eps_1$, then there are $\sigma^1,\dots,\sigma^N\in \AA$ such that
\eq{
\RR_{j,k} < \eps_2 \quad \text{for all $1\leq j<k\leq N$.}
}
\end{lemma}

\begin{proof}
Set $\delta \coloneqq \eps_1\eps_2/N$. 
Observe that for any $\sigma\in\AA$, we have the following implication:
\eeq{\label{high_overlap_set}
\delta \geq \RR(\sigma) \geq \eps_2\langle \one_{\BB(\sigma,\eps_2)}\rangle \quad \implies \quad \langle\one_{\BB(\sigma,\eps_2)}\rangle \leq \frac{\delta}{\eps_2} = \frac{\eps_1}{N}.
}
Therefore, one can inductively choose
\eq{
\sigma^1 \in \AA, \quad \sigma^2 \in \AA\setminus\BB(\sigma^1,\eps_2), \quad \sigma^3\in\AA\setminus(\BB(\sigma^1,\eps_2)\cup\BB(\sigma^2,\eps_2)), \dots
}
where \eqref{high_overlap_set} guarantees that
\eq{
\mu_n^\beta\big(\AA\setminus(\BB(\sigma^1,\eps_2)\cup\cdots\BB(\sigma^{k-1},\eps_2))\big) \geq \eps_1 - (k-1)\frac{\eps_1}{N}.
}
Hence $\sigma^{k}\in\AA\setminus(\BB(\sigma^1,\eps_2)\cup\cdots\BB(\sigma^{k-1},\eps_2))$ can be found so long as $k\leq N$.

\end{proof}

We can now complete the proof.
Assume that $\mathrm{(S1)}$ holds.
Suppose, contrary to $\mathrm{(S2)}$, that there is some $\eps \in (0,1)$ such that for every $\delta > 0$,
\eeq{ \label{contradiction_setup}
\limsup_{n\to\infty} \E\langle \one_{\AA_{n,\delta}}\rangle > 4\eps.
}
Note that for any $n$ such that $\E\langle \one_{\AA_{n,\delta}}\rangle\geq4\eps$, we have
\eq{
4\eps \leq \E\langle\one_{\AA_{n,\delta}}\rangle &\leq \P(\langle\one_{\AA_{n,\delta}}\rangle \geq 2\eps) +2\eps\P(\langle\one_{\AA_{n,\delta}}\rangle<2\eps) 
= (1-2\eps)\P(\langle\one_{\AA_{n,\delta}}\rangle\geq2\eps)+2\eps,
}
and thus $\P(\langle\one_{\AA_{n,\delta}}\rangle \geq2\eps) \geq 2\eps$. 

From $\mathrm{(S1)}$, we choose $k$ and $\delta$ so that for all $n$ large enough (depending on $\eps$ on $\beta$), the following is true with $\P$-probability at least $1-\eps$:
There exist $\sigma^1,\dots,\sigma^k \in \Sigma_n$ such that
\eeq{ \label{balls_covering}
\mu_{n}^\beta\Big(\bigcup_{j=1}^k \BB(\sigma^j, \delta)\Big) \geq 1-\eps.
}
Once $\delta$ has been determined, choose $N$ so that the conclusion of Lemma~\ref{one_to_pair} holds.
Then, given the values of $k$ and $N$, choose $\delta'$ so that the conclusion of Lemma~\ref{many_orthogonal} holds with $\eps_1 = \eps/k$ and $\eps_2 = \delta^2/2$.


In summary, if $n$ is large enough, and $\E\langle \one_{\AA_{n,\delta'}}\rangle\geq4\eps$ 
(by \eqref{contradiction_setup}, there are infinitely many $n$ for which this is the case), the following is true. 
With $\P$-probability at least $2\eps-\eps=\eps$, we have both $\langle \one_{\AA_{n,\delta'}}\rangle \geq 2\eps$ and \eqref{balls_covering} for some $\sigma^1,\dots,\sigma^k\in\Sigma_n$.
In this case, we have
\eq{
\mu_{n}^\beta\bigg(\AA_{n,\delta'} \cap \Big(\bigcup_{j=1}^k \BB(\sigma^j, \delta)\Big)\bigg) \geq 2\eps - \eps = \eps.
}
Therefore, there is some $j$ such that
\eq{
\mu_{n}^\beta\big(\AA_{n,\delta'} \cap \BB(\sigma^j, \delta)\big) \geq \frac{\eps}{k}.
}
By our choice of $\delta'$, we can find $\sigma^1,\dots,\sigma^{N}\in\AA_{n,\delta'} \cap \BB(\sigma^j, \delta)$ satisfying
\eq{
\RR_{j,k} < \frac{\delta^2}{2} \quad \text{for all $1\leq j<k\leq N$.}
}
But $\sigma^1,\dots,\sigma^{N}\in\BB(\sigma^j, \delta)$, and so the above display contradicts \eqref{one_to_pair_ineq}.

\section{Polymer measures are asymptotically non-atomic} \label{no_path_atom}
In this section we prove that directed polymers on the lattice are asymptotically non-atomic.
It is a striking phenomenon that at sufficiently small temperatures, the polymer endpoint distribution places a non-vanishing mass on a single element of $\Z^d$ (which is random and varies with $n$) \cite{comets-shiga-yoshida03}.
The fact that the polymer measures themselves do not share this property, stated below as Theorem~\ref{no_path_atom_thm}, justifies the investigation of replica overlap as an order parameter for path localization.
To emphasize the fact that the Gaussian environment can be replaced by a general one, we reintroduce notation for directed polymers.

Let $(\omega(i,x) : i\geq1, x \in\Z^d)$ be a collection of i.i.d.~random variables.
We will assume that 
\eeq{
\E(\e^{t\omega(i,x)}) < \infty \quad \text{for some $t > 0$}, \label{polymer_exp_moment}
} 
and also that 
\eeq{
\Var(\omega(i,x))> 0 \label{polymer_var}
} 
in order to avoid trivialities. 
Let $\PP_n$ denote the set of nearest-neighbor paths of length $n$ in $\Z^d$ starting at the origin.
Note that $|\PP_n| = (2d)^n$.
To each $\vc x = (0,x_1,\dots,x_n)$ in $\PP_n$ we associate the Hamiltonian energy
\eq{
H_n(\vc x) \coloneqq \sum_{i=1}^n \omega(i,x_i).
}
The polymer measure is then defined by
\eq{
\mu_{n}^\beta(\vc x) \coloneqq \frac{\e^{\beta H_n(\vc x)}}{\sum_{\vc y}\e^{\beta H_n(\vc y)}}, \quad \vc x \in \PP_n.
}

\begin{thm} \label{no_path_atom_thm}
Assume \eqref{polymer_exp_moment}.
Then for any $d\geq1$ and any $\beta\in[0,\infty)$,
\eeq{ \label{no_atom}
\max_{\vc x\in\PP_n} \mu_{n}^\beta(\vc x) = O(n^{-1}) \quad \mathrm{a.s.}\text{ as $n\to\infty$}.
}
\end{thm}

The remainder of Section~\ref{no_path_atom} is to prove Theorem~\ref{no_path_atom_thm}.
We begin by defining the \textit{passage time},
\eq{
L_n \coloneqq \max_{\vc x\in \PP_n} H_n(\vc x).
}
We will denote the set of maximizing paths by
\eeq{
\MM_n \coloneqq \{\vc x \in \PP_n : H_n(\vc x) = L_n\}. \label{maximizing_paths}
}
It is well-known (for instance, see \cite{georgiou-rassoul-seppalainen16}) that there is a finite constant $\lambda$ such that
\eeq{ \label{time_constant}
\lim_{n\to\infty} \frac{L_n}{n} = \sup_{n\geq1} \frac{\E(L_n)}{n} = \lambda \quad \mathrm{a.s.}
} 
The first equality above is a consequence of the superadditivity of $L_n$, and the second equality leads to a short proof of the following standard fact.

\begin{lemma} \label{trivial_lemma}
$\lambda > \E(\omega(i,x))$.
\end{lemma}

\begin{proof}
Let $\vc a = (1,0,\dots,0) \in \Z^d$ and $\vc 0 = (0,\dots,0) \in \Z^d$.
Observe that $L_2 \geq \max\{\omega(1,\vc a) + \omega(2,\vc 0),\, \omega(1,-\vc a) + \omega(2,\vc 0)\}$, and so
\eq{
2\lambda \geq \E(L_2) \geq \E\max\{\omega(1,\vc a) + \omega(2,\vc 0),\omega(1,-\vc a) + \omega(2,\vc 0)\} > 2\E(\omega(i,x)),
}
where the final equality is strict because $\Var(\omega(i,x)^2) > 0$. 
\end{proof}

\begin{defn}
For a nearest-neighbor path $\vc x = (x_0,x_1,\dots,x_n)$ of length $n$ in $\Z^d$, define the \textit{turns} of $\vc x$ to be the following set of indices:
\eeq{ \label{set_of_turns}
T(\vc x) \coloneqq \{1 \leq i \leq n-1 : x_{i+1} - x_i \neq x_i - x_{i-1}\}.
}
The number of turns of $\vc x$ will be denoted $t(\vc x) \coloneqq |T(\vc x)|$.

\end{defn}

\begin{lemma} \label{many_turns}
For any $\eps > 0$, there is $\delta = \delta(\eps,d) > 0$ small enough that
\eq{
|\{\vc x \in \PP_n : t(\vc x) < \delta n\}| \leq C(\eps,d)(1+\eps)^n \quad \text{for all $n\geq1$}.
}
\end{lemma}

\begin{proof}
Given an integer $j$, $0 \leq j \leq n-1$, we count the elements of $\{\vc x\in\PP_n : t(\vc x) = j\}$ as follows.
First, the number of choices for $x_1$ is $2d$.
Next, a turn should occur at exactly $j$ of the coordinates $x_1,\dots,x_{n-1}$.
Moreover, if a turn occurs at $x_i$, then there are $2d-1$ choices for $x_{i+1}-x_i$ (so as to avoid $x_i - x_{i-1}$).
Finally, if a turn does not occur at $x_i$, then there is only one choice for $x_{i+1}-x_i$, namely $x_i - x_{i-1}$.
Therefore, for any positive integer $k \leq \frac{n-1}{2}$, 
\eq{
|\{\vc x \in \PP_n : t(\vc x) < k\}| &= \sum_{j=0}^{k-1} 2d{n-1 \choose j}(2d-1)^j 
\leq 2dk{n-1\choose k}(2d-1)^{k-1}.
}
If $k = \ceil{\delta n}$ for $\delta\in(0,\frac{1}{2})$, then Stirling's approximation gives
\eq{
\lim_{n\to\infty} \frac{1}{n}\log {n - 1 \choose k} = -\delta \log \delta - (1-\delta)\log (1-\delta).
}
Therefore,
\eq{
\limsup_{n\to\infty}\frac{\log |\{\vc x \in \PP_n : t(\vc x) < \delta n\}|}{n}
&\leq -\delta \log \delta - (1-\delta)\log (1-\delta) + \delta \log(2d-1).
}
Now choose $\delta$ sufficiently small that the right-hand side above is strictly less than $\log(1+\eps)$.
Inverting the logarithm and choosing $C$ large enough now yields the desired result.
\end{proof}

\begin{lemma} \label{not_many_gaps}
Let $\{(\omega_i,\omega_i')\}_{i=1}^\infty$ denote a sequence of i.i.d.~pairs of independent random variables.
For any $\eps > 0$ and $\nu > 0$, there exists $D>0$ large enough that
\eq{
\P(|\{1\leq i\leq n-1 : \omega_i > \omega_i' + D\}| > \nu n) \leq \eps^n \quad \text{for all $n\geq1$}.
}
\end{lemma}

\begin{proof}
Choose $D>0$ large enough that $p \coloneqq \P(\{|\omega_i|\geq D/2\}\cup\{|\omega_i'|\geq D/2\})$ satisfies $p^{\nu} \leq \eps/2$.
We then have
\eq{
\P(|\{1\leq i\leq n : \omega_i > \omega_i'+D\}| > \nu n)
&\leq \P(|\{1\leq i\leq n-1: |\omega_i| \geq D/2 \text{ or } |\omega_i'| \geq D/2\}| > \nu n) \\
&\leq \sum_{j = \ceil{\nu n}}^{n-1} {n \choose j}p^j(1-p)^j 
\leq p^{\nu n}2^{n-1} \leq \eps^n.
}
\end{proof}

\begin{proof}[Proof of Theorem~\ref{no_path_atom_thm}]
Let $\omega$ denote a generic copy of $\omega(i,x)$, and $\bar\omega \coloneqq \E(\omega)$.
Set $\kappa \coloneqq (\lambda - \bar\omega)/2$, which is positive by Lemma~\ref{trivial_lemma}.
By assumption, there is $t > 0$ such that $\E(\e^{t \omega}) < \infty$.
Take any $s\in(0,t)$ and observe that for any given $\vc x\in\PP_n$,
\eq{
\P(H_n(\vc x) \geq (\bar\omega +\kappa) n)
&\leq \P(\e^{s(H_n(\vc x)-\bar\omega n)}\geq \e^{s\kappa n})
\leq \e^{-s \kappa n}\E(\e^{s(\omega-\bar\omega)})^n.
}
Using dominated convergence, it is easy to show that
\eq{
\lim_{s\searrow0} \frac{\E(\e^{s(\omega-\bar\omega)})-1}{\e^{s\kappa}-1}
= \lim_{s\searrow0} \frac{\E((\omega-\bar\omega)\e^{s(\omega-\bar\omega)})}{\kappa\e^{s\kappa}}
= 0,
}
and so we may choose $s$ sufficiently small that $\e^{-s\kappa}\E(\e^{s(\omega-\bar\omega)}) < 1$.
Set $\eta \coloneqq 1- \e^{-s\kappa}\E(\e^{s(\omega-\bar\omega)})$, and then choose $\eps > 0$ sufficiently small that $(1+\eps)(1-\eta) < 1$.
With $\delta$ as in Lemma~\ref{many_turns}, we have the union bound
\eq{
\P(\exists\, \vc x\in\PP_n : t(\vc x) < \delta n, H_n(\vc x) \geq (\bar\omega+\kappa)n) \leq C(1+\eps)^n(1-\eta)^n.
}
By our choice of $\eps$, Borel--Cantelli implies that the following statement holds almost surely:
\eq{
\exists\, n_0 : \forall\,  n\geq n_0,\,  \forall\, \vc x\in\PP_n, \quad
t(\vc x) < \delta n \implies H_n(\vc x) < (\bar\omega + \kappa)n.
}
On the other hand, it is apparent from \eqref{time_constant} and our choice of $\kappa$ that almost surely, we have $L_n > (\bar\omega+\kappa)n$ for all large $n$. 
For any such $n$, we then have $H_n(\vc x) > (\bar\omega+\kappa)n$ for every $\vc x\in\MM_n$, the set of maximizing paths defined in \eqref{maximizing_paths}. 
That is, almost surely:
\eq{
\exists\, n_1 : \forall\,  n\geq n_1,\,  \forall\, \vc x\in\MM_n, \quad
H_n(\vc x) \geq (\bar\omega + \kappa)n.
}
Together, the two previous displays show that almost surely,
\eeq{ \label{statement_1}
\exists\, n_2 : \forall\,  n\geq n_2,\,  \forall\, \vc x\in\MM_n, \quad
t(\vc x) \geq \delta n.
}
Recall from \eqref{set_of_turns} that $T(\vc x)$ denotes the set of turns in the path $\vc x\in\PP_n$.
For a given $\vc x\in\PP_n$ and $i \in T(\vc x)$, 
let $\vc x^{(i)}$ denote the unique element of $\PP_n$ such that $x^{(i)}_i \neq x_i$ but $x^{(i)}_j = x_j$ for all $j\neq i$. 
That is, $x^{(i)}_i - x^{(i)}_{i-1} = x_{i+1}-x_i$ while $x^{(i)}_{i+1} - x^{(i)}_{i} = x_{i}-x_{i-1}$.
Upon taking $\eps = 1/(4d)$ and $\nu = \delta/3$ in Lemma~\ref{not_many_gaps}, a union bound gives
\eq{
\P\Big(\exists\, \vc x\in\PP_n : |\{i \in T(\vc x) : H_n(\vc x) > H_n(\vc x^{(i)}) + D\}| > \frac{\delta}{3}n\Big) \leq 2^{-n}.
}
Therefore, we can again apply Borel--Cantelli to see that almost surely,
\eq{
\exists\, n_3 : \forall\,  n\geq n_3,\,  \forall\, \vc x\in\PP_n, \quad
|\{i \in T(\vc x) : H_n(\vc x) > H_n(\vc x^{(i)}) + D\}| \leq \frac{\delta}{3}n.
}
Now combining this statement with \eqref{statement_1}, we arrive at the following almost sure event:
\eq{   
\exists\, n_4 : \forall\,  n\geq n_4,\,  \forall\, \vc x\in\MM_n, \quad
|\{i \in T(\vc x) : H_n(\vc x) \leq H_n(\vc x^{(i)}) + D\}| \geq \frac{2\delta}{3}n.
}
In particular, since $\MM_n$ has at least one element (call it $\vc y$), we have the following for all $n\geq n_4$:
\eq{
\max_{\vc x\in\PP_n} \mu_n^\beta(\vc x) = \frac{\e^{\beta H_n(\vc y)}}{\sum_{x\in\PP_n}\e^{\beta H_n(\vc x)}}
&\leq \frac{\e^{\beta H_n(\vc y)}}{\sum_{i\in T(\vc y)}\e^{\beta H_n(\vc y^{(i)})}} \\
&\leq \frac{\e^{\beta H_n(\vc y)}}{\frac{2\delta}{3}n \e^{\beta H_n(\vc y)}\e^{-\beta D}} = \frac{3\e^{\beta D}}{2\delta n}.
}
Since $D$ and $\delta$ do not depend on $n$, \eqref{no_atom} follows.
\end{proof}

\section{Open problems}
There are a number of open questions which, if solved, would enhance the theory presented in this chapter. 
A partial list is the following.
\begin{enumerate}
\item Understand conditions under which the number of localizing regions is exactly one. As mentioned before, this requires more conditions than \eqref{free_energy_assumption}--\eqref{field_decomposition}, because it does not hold for some models (such as REM), whereas it is supposed to hold for many others. 
\item A close cousin of the above problem is to understand conditions under which $\RR_{1,2}$ is itself guaranteed to be away from zero with high probability. This would have important implications about the FRSB picture in mean-field spin glasses and path localization in directed polymers. 
\item Obtain a good quantitative bound on $\delta$ in terms of $\eps$ in Theorem \ref{expected_overlap_thm}. 
Our proof gives a very poor bound, since it is based on an iterative argument similar to those used in extremal combinatorics (see the proof sketch in Section \ref{sketchsec2}).
\item For directed polymers, prove a stronger theorem about path localization that says a typical path localizes within a narrow neighborhood of one or more fixed paths, rather than saying that a typical path has nonzero intersection with one or more fixed paths. 
\item Prove more general versions of Theorems \ref{easy_cor},  \ref{expected_overlap_thm} and \ref{averages_squared} that do not require the positive overlap condition \eqref{positive_overlap}. This would allow the theory to include other models of interest, such as the Edwards--Anderson model \cite{edwards-anderson75} of lattice spin glasses. 
\item For any finite $\beta$, prove estimates that stochastically bound $\langle \RR_{1,2}\rangle$ away from $1$.
More ambitiously, determine conditions which guarantee that $\langle \RR_{1,2}\rangle$ concentrates around its expectation as $n\to\infty$.
\end{enumerate}
    
    
    \chapter{Replica symmetry breaking in spin glasses} \label{rsb}
    
    \section{Introduction}
Spin glass theory, originally developed to study disordered magnets \cite{edwards-anderson75}, now includes applications in biology \cite{parisi90,barra-agliari10,agliari-barra-guerra-moauro11}, computer science \cite{nishimori01,mezard-montanari09}, neuroscience \cite{hopfield82,amit-gutfreund-sompolinsky85,bovier-picco98,barra-guerra08,barra-genovese-guerra10,agliari-barra-galluzzi-guerra-moauro12,barra-genovese-guerra-tantari12} and econometrics/quantitative sociology \cite{krapivsky-redner03,contucci-ghirlanda07,contucci-gallo-menconi08,barra-contucci10,barra-agliari12}, primarily due to interest in large-scale networks.
However, its prototypical mathematical model, namely that of Sherrington and Kirkpatrick (SK) \cite{sherrington-kirkpatrick75}, is fully mean-field and thus fails to capture the effect of global inhomogeneities and communities.
In order to examine models more faithful to real-world networks, physicists and mathematicians have in recent years advanced the study of bipartite or more general ``multi-species" spin systems, e.g.~\cite{gallo-contucci08,barra-genovese-guerra11,fedele-contucci11,fedele-unguendoli12,barra-galluzzi-guerra-pizzoferrato-tantari14,barra-contucci-mingione-tantari15}.
An ongoing task is to adapt results from classical spin glasses, such as SK, to their multi-species extensions, especially those regarding the so-called glassy phase observed at low temperatures.


This chapter focuses on the multi-species SK (MSK) model, which allows arbitrary interactions between sets of binary spins we call ``species" but remains mean-field in the sense that the number of species is fixed even as the population of each species grows to infinity.
This spin system was introduced by Barra, Contucci, Mingione, and Tantari \cite{barra-contucci-mingione-tantari15}, who also proposed a Parisi formula for the limiting free energy when the interaction parameters satisfy a convexity condition.
This formula was proved by Panchenko \cite{panchenko15I}, which allows us to proceed rigorously in the present work.
By entropic considerations, it is known that there does indeed exist a low temperature phase, i.e.~where the disorder is said to be ``symmetry breaking" \cite[Proposition 4.2]{barra-contucci-mingione-tantari15}.
Our main purpose is to prove a quantitative version of this fact.
We are able to do so in the two-species model, for which we find the analog of the de Almeida--Thouless (AT) line \cite{almeida-thouless78} from the classical SK model.
This is the content of Theorem \ref{2speciessymmetrybreaking}.
For three or more species, our calculations still predict an AT condition given in Corollary \ref{general_RSB_condition}; however, we have been unable to translate this condition into an explicit temperature threshold outside the two-species case.
Section \ref{3plus_difficulties} outlines the relevant difficulties.

The models under consideration are defined below in Section \ref{definitions}.
Our main results for the two-species SK model, namely Theorems \ref{uniquenessof2speciesrRSsol} and \ref{2speciessymmetrybreaking}, are stated in Section \ref{main_results}.
Their proofs are given in Sections \ref{sec:uniqueness} and \ref{sec:hessian}, respectively, and related results from the literature on single-species models are discussed in Section \ref{subsec:related_results}.

\subsection{The SK and MSK models} \label{definitions}
We consider a collection of Ising spins $\sigma = (\sigma_1,\dots,\sigma_N) \in \{\pm 1\}^N$, subject to the Hamiltonian
\eeq{ \label{sk_hamiltonian}
H_N(\sigma) = \frac{\beta}{\sqrt{N}}\sum_{i,j=1}^N g_{ij}\sigma_i\sigma_j + h\sum_{i=1}^N \sigma_i,
}
where $\beta > 0$ is the inverse temperature, $h \geq 0$ is the external field, and the disorder parameters $g_{ij}$ are independent, centered Gaussian random variables.
In the SK model, these parameters all have unit variance.
In the MSK model, their variances depend on $i$ and $j$ in the following way.
The spins are partitioned into $M \geq 2$ sets as $\{1,\dots,N\} = \bigcup_{s=1}^M I_s$, and then
\eq{
\E(g_{ij}^2) = \Delta_{st}^2 \quad \text{whenever $i \in I_s$ and $j \in I_t$}.
}
Thus $\Delta^2 = (\Delta_{st}^2)_{1\leq s,t\leq M}$ is a symmetric $M\times M$ matrix.
Considering the infinite volume limit, we assume that
\eq{
\lim_{N\to\infty} \frac{|I_s|}{N} = \lambda_s \in (0,1) \quad \text{for each $s = 1,\dots,M$.}
}
Of central interest is the free energy of the system,
\eq{
F_N(\beta) = \frac{\E \log Z_N(\beta)}{N}, \quad \text{where} \quad Z_N(\beta) = \sum_{\sigma \in \{\pm1\}^N} \e^{\beta H_N(\sigma)}.
}
Under the assumption that $\Delta^2$ is nonnegative definite, Barra \textit{et al.}~\cite[Theorem 1.2]{barra-contucci-mingione-tantari15} prove that $F_N(\beta)$ converges as $N\to\infty$, and Panchenko \cite[Theorem 1]{panchenko15I} verifies their prediction that the limit is given by the variational formula
\eeq{ \label{variational_pre}
\lim_{N\to\infty} F_N(\beta) = \inf \mathscr{P},
}
where $\mathscr{P}$ generalizes the famous Parisi formula \cite{parisi79,parisi80} proved by Talagrand \cite{talagrand06} for the SK model.
In fact, Panchenko shows $\lim_{N\to\infty} F_N(\beta) \geq \inf \mathscr{P}$ for general $\Delta^2$.
It is the upper bound $F_N(\beta) \leq \inf \mathscr{P}$, proved in \cite[Theorem 1.3]{barra-contucci-mingione-tantari15} using Guerra's interpolation method \cite{guerra03}, that requires the nonnegative definiteness assumption.

Let us now define $\mathscr{P}$ precisely.
Given an integer $k \geq 0$, consider a sequence
\begin{subequations} \label{parameters}
\begin{align}
0 = \zeta_0 < \zeta_1 < \cdots < \zeta_{k} < \zeta_{k+1} = 1, \label{parameters_zeta}
\intertext{and for each species $s = 1,\dots,M$, a corresponding sequence}
0 = q_0^s \leq q_1^s \leq \cdots \leq q_{k+1}^s \leq q_{k+2}^s = 1. \label{parameters_q}
\end{align}
\end{subequations}
With these parameters, for each $0 \leq \ell \leq k+2$ we define
\eeq{ \label{Qdef}
Q_{\ell} \coloneqq \sum_{s,t = 1}^M \Delta_{st}^2 \lambda_s \lambda_t q_{\ell}^s q_{\ell}^t, \qquad Q_{\ell}^s \coloneqq 2 \sum_{t =1}^M \Delta_{st}^2 \lambda_t q_{\ell}^t, \quad 1 \leq s \leq M,
}
and then
\eq{
X_{k+2}^s \coloneqq \log \cosh\bigg(h+\beta\sum_{\ell=0}^{k+1} \eta_{\ell+1}\sqrt{Q_{\ell+1}^s-Q_{\ell}^s}\bigg).
}
where $\eta_1,\dots,\eta_{k+2}$ are i.i.d.~standard normal random variables.
Next we inductively define
\eeq{ \label{X_def}
X_\ell^s \coloneqq \frac{1}{\zeta_\ell} \log \E_{\ell+1} \exp(\zeta_\ell X_{\ell+1}^s), \qquad 0 \leq \ell \leq k+1,
}
where $\E_{\ell+1}$ denotes expectation with respect to $\eta_{\ell+1}$. 
When $\ell=0$, \eqref{X_def} is understood to mean
\eq{
X_0^s = \lim_{\zeta\searrow0}\frac{1}{\zeta}\log \E_{1}\exp(\zeta X_1^s) = \E_1(X_1^s).
}
Finally, we can write
\eeq{ \label{parisi_expression}
\mathscr{P}(\zeta,q) \coloneqq \log{2} + \sum_{s = 1}^M \lambda_s X_0^s - \frac{\beta^2}{2} \sum_{\ell=1}^{k+1} \zeta_{\ell}(Q_{\ell+1} - Q_{\ell}),
}
so that \eqref{variational_pre} reads as
\eeq{ \label{variational}
\lim_{N\to\infty} F_N(\beta) = \inf_{\zeta,q} \mathscr{P}(\zeta,q).
}

The physical interpretation is as follows. 
Each species has an order parameter $\mu_s$, which is the limiting distribution of the overlap
\eeq{ \label{overlap_def}
R_s(\sigma^1,\sigma^2) = \frac{1}{|I_s|}\Big|\sum_{i\in I_s} \sigma_i^1\sigma_i^2\Big|,
}
where $\sigma^1,\sigma^2$ are independent samples from the Gibbs measure associated to \eqref{sk_hamiltonian}.
If the infimum in \eqref{variational} is achieved at $(\zeta,q)$, then $\mu_s = \sum_{\ell=1}^{k+1} (\zeta_{\ell}-\zeta_{\ell-1})\delta_{q_\ell^s}$. 
When $k=0$ and $\mu_s$ is a single atom for every $s$, we will say the system is in the ``replica symmetric" (RS) phase.
Otherwise, we will say the system has ``replica symmetry breaking" (RSB).
In the SK model (i.e.~$M=1$), there is a single order parameter $\mu$, and so it makes sense to discuss the \textit{level} of symmetry breaking.
That is, if $\mu$ consists of $k+1$ distinct atoms, then the model is said to exhibit ``$k$-step replica symmetry breaking" ($k$RSB); alternatively, if $\mu$ has infinite support---so the infimum in \eqref{variational} is not achieved---then there is ``full replica symmetry breaking" (FRSB).
In the MSK model, it may be the case that if $s\neq t$, then $\mu_s$ and $\mu_t$ can have a different number of atoms in their support.
That is, it is possible that for some $\ell$, one has $q^s_{\ell-1} < q^s_{\ell}$ but $q^t_{\ell-1} = q^t_{\ell}$.
Part of what is shown in \cite{panchenko15I}, however, is that $\zeta_\ell-\zeta_{\ell-1}$ is fixed across species.
It is thus reasonable to say that the MSK model exhibits $k$RSB if \eqref{variational} has a minimizer of the form \eqref{parameters}.

For more on the relationship between replica overlaps and the Parisi minimizer, we refer the reader to \cite{auffinger-chen15I,auffinger-chen15II,auffinger-chen-zeng20}, or to \cite{mezard-parisi-virasoro87,talagrand11I,talagrand11II,panchenko13} for extended treatment of the subject.

\subsection{Statements of main results} \label{main_results}
Consider the replica symmetric expression for the free energy, which involves a single parameter $q^s \in [0,1]$ for each species $s$.
By this we mean that in \eqref{parameters}, we set $k=0$ and $q_1^s = q^s$.
Using the formula for the moment generating function of the Gaussian distribution, one finds that
\eeq{ \label{RS_formula}
\mathscr{P}_\mathrm{RS}(q) = \log 2 + \sum_{s=1}^M \lambda_s\Big[\E_1\log\cosh(\beta\eta_1\sqrt{Q_1^s}+h) + \frac{\beta^2}{2}(Q_2^s-Q_1^s)\Big] - \frac{\beta^2}{2}(Q_2-Q_1).
}
By differentiating this expression with respect to each $q^t$, and then applying Gaussian integration to write
\eq{
\E_1[\eta_1\tanh(\beta\eta_1\sqrt{Q_1^s}+h)] 
&= \beta\sqrt{Q_1^s}\E_1\sech^2(\beta\eta_1\sqrt{Q_1^s}+h) \\
&= \beta\sqrt{Q_1^s}\big(1-\E_1\tanh^2(\beta\eta_1\sqrt{Q_1^s}+h)\big),
}
it follows that any critical point $q$ must satisfy
\eeq{ \label{P_deriv}
\frac{\partial\mathscr{P}_\mathrm{RS}}{\partial q^t} = \beta^2\lambda_t\sum_{s=1}^M\Delta_{st}^2\lambda_s\big[q^s-\E_1\tanh^2(\beta\eta_1\sqrt{Q_1^s}+h)\big] = 0,
\quad t = 1,\dots,M.
}
If $\Delta^2$ is invertible, then this system implies
\eeq{ \label{RS_condition}
q^s = \E_1\tanh^2(\beta\eta_1\sqrt{Q_1^s}+h), \quad s = 1,\dots,M.
}
Therefore, we define the set
\eq{
\CC(\beta,h) \coloneqq \{q \in [0,1]^M : \E \tanh^2(\beta\eta\sqrt{Q_1^s}+h) = q_s \text{ for each $s = 1,\dots,M$}\}.
}
As for the SK model, it is not difficult to show that for small $\beta$, $\CC(\beta,0)$ is a singleton; obtaining a sharp estimate requires more care. 
We attempt to do so in the two-species case as part of Theorem \ref{uniquenessof2speciesrRSsol} below.
We also expect $\CC(\beta,h)$ is a singleton whenever $h > 0$, 
and can prove such a statement when $M=2$.

To simplify notation and standardize temperature scale, we henceforth assume $M=2$ and
\eeq{ \label{delta_assumptions}
\Delta^2 = \begin{pmatrix}
\Delta_{11}^2 & 1 \\
1 & \Delta_{22}^2,
\end{pmatrix}
\quad
\text{where}
\quad
\Delta_{11}^2\Delta_{22}^2 > 1
\quad
\text{and} 
\quad
\lambda_1\Delta_{11}^2 \geq \lambda_2\Delta_{22}^2,
}
in particular ensuring \eqref{variational}.
The second inequality above is made without loss of generality, simply by relabeling the species if necessary.
With these assumptions, we can now state our first result.

\begin{thm} \label{uniquenessof2speciesrRSsol}
Assume \eqref{delta_assumptions}.
If either $h > 0$ or
\eeq{ \label{uniqueness_condition}
\beta^2 < \frac{1}{\lambda_1\Delta_{11}^2+\lambda_2\Delta_{22}^2 + \sqrt{(\lambda_1\Delta_{11}^2-\lambda_2\Delta_{22}^2)^2+4\lambda_1\lambda_2}},
}
then $\CC(\beta,h) = \{q_*\}$ is a singleton.
In this case,
\eeq{ \label{RS_minimizing}
\mathrm{RS}(\beta,h) \coloneqq \min_{q\in[0,1]^2}\mathscr{P}_\mathrm{RS}(q) = \mathscr{P}_\mathrm{RS}(q_*).
}
\end{thm}

The proof of Theorem \ref{uniquenessof2speciesrRSsol} is provided in Section \ref{sec:uniqueness}.
Based on analogy with the SK model (see Remark \ref{rmk:analogy} below), one might suspect that \eqref{uniqueness_condition} defines the RS phase of the MSK model.
This suspicion is supported by the striking similarity with \eqref{1RSB_condition}, which at least in the SK model is believed to define the RSB phase (see Section \ref{subsec:related_results}).
Notice that when $h=0$ and $q_s = 0$ for each $s$, the quantity $\gamma_s$ defined in Theorem \ref{2speciessymmetrybreaking} reduces to $\lambda_s$, and \eqref{1RSB_condition} becomes \eqref{uniqueness_condition}.

\begin{thm}\label{2speciessymmetrybreaking}
Assume \eqref{delta_assumptions} and $h>0$.
Let $q = q_*$ be the critical point from Theorem \ref{uniquenessof2speciesrRSsol}.
Define
\eq{
\gamma_s\coloneqq\lambda_s \E\, \mathrm{sech}^4\,(\beta\eta\sqrt{Q_1^s}+h), \quad s=1,2.
}
If
\eeq{
\beta^2>\frac{1}{\gamma_1 \Delta_{11}^2 +\gamma_2 \Delta_{22}^2 +\sqrt{(\gamma_1 \Delta_{11}^2 -\gamma_2 \Delta_{22}^2)^2 +4 \gamma_1 \gamma_2}}, \label{1RSB_condition}
}
then
 \eeq{\label{1RSB}
\lim_{N\to\infty} F_N(\beta) < \mathrm{RS}(\beta,h).
}
\end{thm}

\begin{remark} \label{rmk:analogy}
If $\Delta_{11}^2=\Delta_{22}^2=1$, then \eqref{uniqueness_condition} recovers the analogous result for the SK model, proven independently in \cite{guerra01} and \cite{latala02}.
Similarly, \eqref{1RSB_condition} recovers the AT condition proven in \cite{toninelli02}.
Notational choices in the SK model, however, sometimes replace $\beta^2$ with $\beta^2/2$.
\end{remark}

The proof of Theorem \ref{2speciessymmetrybreaking} is given in Section \ref{sec:hessian}.
In terms of the overlap order parameter \eqref{overlap_def}, Theorem \ref{2speciessymmetrybreaking} says that if \eqref{1RSB_condition} holds, then $R_s(\sigma^1,\sigma^2)$ has a limiting distribution that is nontrivial (i.e.~at least two points in the support) for some species $s$.
Moreover, our analysis in the two-species case suggests that \eqref{1RSB} is realized at smaller $\beta$ when both species break symmetry as opposed to just one; see \eqref{2speciescondition_new}, which is equivalent to \eqref{1RSB_condition}.
Therefore, it is plausible that under convexity \eqref{delta_assumptions}, symmetry breaking of one species implies symmetry breaking of all species.
Verifying this statement rigorously, however, requires a global analysis that goes beyond the perturbative approach of this chapter.

In the context of the classical SK model, the inequality \eqref{1RSB_condition} is obtained in \cite{toninelli02} as the necessary and sufficient condition for a certain second derivative to be positive.
The analogous object in the MSK model is an $M\times M$ Hessian matrix, and the relevant condition is, at least intuitively, the positivity of its top eigenvalue.
A difficulty posed by multi-dimensionality, however, is that the associated eigenvector need not have all positive entries, which is ultimately needed to conclude symmetry breaking because of the ordering in \eqref{parameters_q}.
When $M=2$, we are able to overcome this difficulty and prove the relevant eigenvector \textit{does} have positive coordinates, by direct analysis of matrix entries.

\subsection{Related results for Ising spin glasses} \label{subsec:related_results}

In the single-species case, \eqref{sk_hamiltonian} is often replaced by a more general Hamiltonian that considers interactions not just between pairs of spins, but also between $p$-tuples of spins for any $p \geq 2$.
More precisely, the mixed $p$-spin model with mixture $\xi(t) = \sum_{p\geq2} \beta_p^2t^p$, inverse temperature $\beta > 0$, and external field $h \geq 0$, has the Hamiltonian
\eq{
H_N(\sigma) = \beta\sum_{p=2}^\infty \frac{\beta_p}{N^{(p-1)/2}}\sum_{i_1,\dots,i_p=1}^N g_{i_1\cdots i_p}\sigma_{i_1}\cdots\sigma_{i_p} + h\sum_{i=1}^N\sigma_i, \quad \sigma \in \{\pm 1\}^N,
}
where the disorder variables $g_{i_1\cdots i_p}$ all i.i.d.~standard normals.
For such models (with suitable decay conditions on the $\beta_p$), the Parisi formula has been proved by Talagrand \cite{talagrand06} when $\beta_p = 0$ for all odd $p$, and by Panchenko \cite{panchenko14} in the general case.

In this more general setting, the set of RS critical points is
\eq{
\CC(\beta,h) \coloneqq \{q \in [0,1] : \E \tanh^2(\beta \eta \sqrt{\xi'(q)} + h) = q\},
}
As discussed in \cite{jagannath-tobasco17I}, the size of $\CC(\beta,h)$ is very difficult to determine in general.
Nevertheless, we are more generally concerned with the quantity
\eq{
\alpha(\beta,h) \coloneqq \min_{q \in \CC(\beta,h)} \beta^2 \xi''(q)\E \sech^4(\beta \eta \sqrt{\xi'(q)} + h).
}
The result of Toninelli \cite{toninelli02} for the SK model, which Theorem \ref{2speciessymmetrybreaking} generalizes,  can then be written as
\eq{
\alpha(\beta,h) > 1 \quad \implies \quad \lim_{N\to\infty} F_N(\beta)< \inf_{q \in \CC(\beta,h)} \mathscr{P}(q),
}
where $\mathscr{P}$ is the Parisi functional restricted to Dirac delta measures;
see \cite{jagannath-tobasco17I} by Jagannath and Tobasco, who extend this result to mixed $p$-spin models.
It is conjectured that the converse is also true, at least when $\beta_2 > 0$, and some partial results are given in \cite{aizenman-lebowitz-ruelle87,guerra-toninelli02II,talagrand02,jagannath-tobasco17I}.
However, for the Ghatak--Sherrington model (in which spin $0$ is allowed), the converse is known to be false by work of Panchenko \cite{panchenko05II}.

\subsection{Challenges with three or more species} \label{3plus_difficulties}
As discussed in Section \ref{main_results}, the relationship between \eqref{uniqueness_condition} and \eqref{1RSB_condition} generalizes the one between the analogous thresholds in the classical SK model.
It would be interesting to have a similar result that applies to the $M$-species model for any $M$.
Unfortunately, it is not clear how the techniques used in Sections \ref{sec:uniqueness} and \ref{sec:hessian} could be adapted to handle the case $M\geq3$.

The challenges are mainly linear algebraic.
For instance, if $A$ is an $M\times M$ positive definite matrix, then the signs of the off-diagonal entries of $A^{-1}$ cannot be determined from $\det(A)$, unless $M\leq2$.
This information is crucial in the proof of Theorem \ref{uniquenessof2speciesrRSsol} (see \eqref{inverse_A} and \eqref{B_inverse}), which establishes the uniqueness of the RS critical point.
Without this uniqueness, Theorem \ref{2speciessymmetrybreaking} can only be stated as Corollary \ref{general_RSB_condition}, which is useful in determining symmetry breaking only if one can identify a critical point with minimal energy.  This task is non-trivial even when $M=1$; for $M\geq 2$, it is not actually clear \textit{a priori} why there should even be only finitely many critical points.

Another difficulty is analyzing the matrix $K = K(\beta)$ appearing in the proof of Theorem \ref{2speciessymmetrybreaking}.
Determining the exact set of $\beta$ for which $K$ has a positive eigenvalue, as well as an eigenvector with all nonnegative entries, requires one to consider all the relationships among the $M(M+1)/2$ distinct entries of $K$ that yield this property.
Any argument that works for general $M$ will not be able to proceed directly as we do, and may have to first address the following question: If $K$ has a positive eigenvalue, is there necessarily an associated eigenvector with all nonnegative entries?

\subsection{Non-convex cases}
The validity of the Parisi variational formula \eqref{variational} is known only when $\Delta^2$ is nonnegative definite.
It seems likely that this formula fails in general, for instance in the bipartite model $\Delta^2 = \begin{pmatrix} 0 & 1 \\ 1 & 0 \end{pmatrix}$.
Theorem \ref{2speciessymmetrybreaking} would prove this to be the case if the second assumption in Corollary \ref{general_RSB_condition} were true.
This is because \cite[Theorem 3]{barra-genovese-guerra11} proves an RS regime at high temperatures, whereas the existence of the $x\in\R^M$ required in Corollary \ref{general_RSB_condition} holds trivially at all temperatures if $\Delta^2_{ss} = 0$ for some $s$.
Unfortunately, our arguments for \eqref{RS_minimizing} rely on convexity (again, see \eqref{B_inverse}).

\section{Uniqueness of critical point} \label{sec:uniqueness}

In this section we prove Theorem \ref{uniquenessof2speciesrRSsol}, which asserts the uniqueness of the replica symmetric critical point in the parameter regime $\{h>0\} \cup \{\beta < \beta_0\}$, where $\beta_0$ is identified from \eqref{uniqueness_condition}.
For the SK model, the argument to identify $\beta_0$ is straightforward and can be found in \cite[Section 1.3]{talagrand11I}. 
The two-dimensional nature of the problem here is handled by introducing a pair of inequalities and then optimizing over $\beta$.

Addressing the case $h>0$ is also straightforward in the SK model, at least once the clever Lemma \ref{LatalaGuerralemma} is realized.
In fact, we are able to make use of this lemma once more in the two-species case, since the signs of the entries in a $2\times2$ inverse matrix are easy to determine.

\begin{lemma}[{\cite{guerra01} and \cite{latala02}, see also \cite[Appendix A.14]{talagrand11II}}]\label{LatalaGuerralemma}
Let $\phi$ be an odd, twice-differentiable, increasing, and bounded function, that is strictly concave when $y>0$. 
Then the function $\Phi(x)\coloneqq\E(\phi(z\sqrt{x}+h)^{2})/x$ is strictly decreasing on $\mathbb{R}^{+}$ and vanishes as $x \implies \infty$.
\end{lemma}

\begin{proof}[Proof of Theorem \ref{uniquenessof2speciesrRSsol}]
Since $\mathscr{P}_\mathrm{RS}(q)$ is a continuous function of $q\in[0,1]^2$, it must attain a minimum.
Since $\mathscr{P}_\mathrm{RS}$ is differentiable on $(0,1)^2$, if this minimum is achieved at some $q \in (0,1)^2$, then $q$ must belong to $\CC(\beta,h)$.
Therefore, we will first show that if $h>0$, then \eqref{RS_condition} has at most one solution $q_* \in (0,1)^2$, and that $\mathscr{P}_\mathrm{RS}$ achieves its minimum in $(0,1)^2$.
In order to handle the case $h=0$, we will separately show that if \eqref{uniqueness_condition} holds, then $\CC(\beta,h)$ contains at most one point.
In particular, when $h=0$, it is clear that the single element must be $(0,0)$.

First assume $h>0$.
For ease of notation, we will write $Q^s = Q^s_1$ for $s = 1,2$, where $Q^s_1$ is defined in \eqref{Qdef}.
Considering $Q = \begin{pmatrix} Q^1 \\ Q^2 \end{pmatrix}$ and $q = \begin{pmatrix} q^1 \\ q^2\end{pmatrix}$ as vectors, we have
\eeq{\label{matrixrepofQandq}
Q=Aq,\quad\text{where}\quad A=\begin{pmatrix} 
2\lambda_1\Delta_{11}^2 & 2\lambda_2 \\
2\lambda_1 & 2\lambda_2\Delta_{22}^2
\end{pmatrix}.
}
Since $\Delta^2$ is positive definite, we have $\det A > 0$ and hence
\eeq{\label{inverse_A}
A^{-1}=\begin{pmatrix}
\phantom{-}a&-b\\-c&\phantom{-}d
\end{pmatrix},\quad\text{where}\quad a,b,c,d >0.
}
Assuming $q$ satisfies \eqref{RS_condition}, inversion of $A$ in \eqref{matrixrepofQandq} gives
\begin{subequations} \label{2speciesRSeq}
\begin{align}
aQ^1-bQ^2&=\E\tanh^{2}(\beta z \sqrt{Q^{1}}+h), \label{2speciesRSeq_1} \\
-cQ^1+dQ^2&=\E\tanh^{2}(\beta z \sqrt{Q^{2}}+h). \label{2speciesRSeq_2}
\end{align}
\end{subequations}
To show that $\CC(\beta,h)$ is a singleton, it suffices (by invertibility of $A$) to show that the system \eqref{2speciesRSeq} admits at most one solution $Q$.

Since $h > 0$, it is clear from \eqref{RS_condition} that $q^1,q^2 > 0$ and thus $Q^1,Q^2 > 0$.
Therefore, we can rewrite \eqref{2speciesRSeq} as
\begin{subequations}
\begin{align}
a-b\frac{Q^2}{Q^1}&=\frac{\E\tanh^{2}(\beta z \sqrt{Q^{1}}+h)}{Q^1}, \label{rewrite_1}\\
-c\frac{Q^1}{Q^2}+d&=\frac{\E\tanh^{2}(\beta z \sqrt{Q^{2}}+h)}{Q^2}. \label{rewrite_2}
\end{align}
\end{subequations}
Since $b>0$, the left-hand side of \eqref{rewrite_1} is strictly increasing in $Q^1$ (and approaches $a>0$ as $Q^1\to\infty$), whereas the right-hand side is strictly decreasing by Lemma \ref{LatalaGuerralemma} (and approaches $0$ as $Q^1\to\infty$).
Therefore, for each fixed value of $Q_2$, there is exactly one value of $Q_1$ satisfying \eqref{rewrite_1}.
Since $c>0$, 
we can also define $Q^1$ from $Q^2$ using \eqref{rewrite_2} instead of \eqref{rewrite_1}.
That is, given $x>0$, define $Q^1(x)$ by
\eq{
-c\frac{Q^1(x)}{x}+d&=\frac{\E\tanh^{2}(\beta z \sqrt{x}+h)}{x}.
}
Since the right-hand side above is strictly decreasing in $x$, it follows that $Q^1(x)/x$ is strictly increasing in $x$.
In particular, $Q^1(x)$ is strictly increasing in $x$.
Therefore, if we replace \eqref{rewrite_1} by
\eeq{
a-b\frac{x}{Q^1(x)}&=\frac{\E\tanh^{2}(\beta z \sqrt{Q^{1}(x)}+h)}{Q^1(x)} \label{rerewrite_1},
}
then the right-hand side is strictly decreasing in $x$ by Lemma \ref{LatalaGuerralemma}, while the left-hand side is strictly increasing in $x$.
Consequently, there is at most one value of $x$ such that \eqref{rerewrite_1} holds.

To complete the proof in the case $h > 0$, we must check that the minimum of the RS expression $\mathscr{P}_\mathrm{RS}$ is not obtained on the boundary.
Recall from \eqref{P_deriv} that
\eq{
\begin{pmatrix}
y^1 \\ y^2
\end{pmatrix}
\coloneqq
\begin{pmatrix} 
\frac{\partial \mathscr{P}_\mathrm{RS}(q)}{\partial q^1}\vspace{3pt} \\ \frac{\partial \mathscr{P}_\mathrm{RS}(q)}{\partial q^2}
\end{pmatrix}
= B\begin{pmatrix}
x^1 \\ x^2
\end{pmatrix},
}
where
\eq{
B = \beta^2\begin{pmatrix}
\lambda_1^2\Delta_{11}^2 & \lambda_1\lambda_2 \\
\lambda_1\lambda_2 & \lambda_2^2\Delta_{22}^2
\end{pmatrix},
\qquad
x^s = q^s - \E\tanh^2(\beta\eta\sqrt{Q^s}+h).
}
Because of \eqref{delta_assumptions}, it follows that
\eeq{ \label{B_inverse}
\begin{pmatrix}
x^1 \\ x^2
\end{pmatrix}
=
\begin{pmatrix}
\phantom{-}e & -f \\
-g & \phantom{-}h
\end{pmatrix}
\begin{pmatrix}
y^1 \\ y^2
\end{pmatrix},
\quad \text{where} \quad
e,f,g,h>0.
}
Suppose toward a contradiction that $\mathscr{P}_\mathrm{RS}$ is minimized as $\mathscr{P}(q^1,0)$ for some $q^1 \in (0,1]$.
We then have $y^1 \leq 0$ and $x_2 < 0$.
It follows that $y^2$ is negative, since $y^2 \geq 0$ would imply
\eq{
0 > x^2 = -gy^1 + hy^2 \geq hy^2 \geq 0.
}
So now $y^2 < 0$, meaning $\mathscr{P}(q^1,q^2) < \mathscr{P}(q^1,0)$ for small enough $q^2 > 0$, giving the desired contradiction.
By similar reasoning, $\mathscr{P}_\mathrm{RS}$ cannot be minimized along $\{0\} \times (0,1]$, $[0,1) \times \{1\}$, or $\{1\} \times [0,1)$.
Finally, it is clear from the bounds on $\tanh^2(\cdot)$ that $y^1$ and $y^2$ are both negative if $q^1=q^2=0$, and both positive if $q^1=q^2=1$, eliminating the possibility that $\mathscr{P}_\mathrm{RS}$ is minimized at either $(0,0)$ or $(1,1)$.
This completes the proof in the case $h > 0$.

For the next part, we assume \eqref{uniqueness_condition} holds, for $h$ possibly equal to $0$.
Let $F_1(q^1,q^2) \coloneqq \psi(Q^1)$ and $F_2(q^1,q^2) \coloneqq \psi(Q^2)$, where
\eq{
\psi(x) = \E f(\beta\eta\sqrt{x}+h), \quad \eta \sim \NN(0,1), \quad \text{and} \quad f(y) = \tanh^2(y).
}
Using Gaussian integration by parts, we find
\eq{
\psi'(x) = \frac{\beta}{2\sqrt{x}} \E[\eta f'(\beta\eta\sqrt{x}+h)]
= \frac{\beta^2}{2}\E f''(\beta\eta\sqrt{x}+h).
}
Also observe that
\eq{
f'(y) = 2\frac{\tanh y}{\cosh^2 y}, \qquad
f''(y) = 2\frac{1-2\sinh^2 y}{\cosh^4 y}.
}
One can check that $f''(y) \in \big[-\frac{2}{3},2\big]$ for all $y\in\R$.
Consequently,
\begin{align} \label{derivative_bounds1}
\Big|\frac{\partial F_1}{\partial q^1}\Big| &= 2\lambda_1\Delta_{11}^2|\psi'(Q^1)| \leq 2\beta^2\lambda_1\Delta_{11}^2,  &
\Big|\frac{\partial F_1}{\partial q^2}\Big| &= 2\lambda_2|\psi'(Q^1)| \leq 2\beta^2\lambda_2, \\
\Big|\frac{\partial F_2}{\partial q^1}\Big| &= 2\lambda_1|\psi'(Q^2)| \leq 2\beta^2\lambda_1,  &
\Big|\frac{\partial F_2}{\partial q^2}\Big| &= 2\lambda_2\Delta_{22}^2|\psi'(Q^2)| \leq 2\beta^2\lambda_2\Delta_{22}^2. \label{derivative_bounds2}
\end{align}
Suppose, toward a contradiction, that $(q^1,q^2)$ and $(p^1,p^2)$ are distinct elements of $\CC(\beta,h)$.
That is, each is a fixed point of $(F_1,F_2) : [0,1]^1\to[0,1]^2$.
Let $\gamma(t) = ((1-t)q^1+tp^1,(1-t)q^2+tp^2)$, $0\leq t\leq1$, be the line segment connecting these two points.
We must have
\eeq{ \label{path_integral}
\int_{0}^1 \nabla F_s(\gamma(t)) \cdot (p^1-q^1,p^2-q^2)\ \dd t = p^s-q^s,\quad s=1,2.
}
On the other hand, the bounds in \eqref{derivative_bounds1} reveal that
\eq{
&\int_0^1 |\nabla F_1(\gamma(t)) \cdot (p^1-q^1,p^2-q^2)|\ \dd t \leq 2\beta^2(\lambda_1\Delta_{11}^2|p^1-q^1| + \lambda_2|p^2-q^2|),
}
while \eqref{derivative_bounds2} gives
\eq{
&\int_0^1 |\nabla F_2(\gamma(t)) \cdot (p^1-q^1,p^2-q^2)|\ \dd t \leq 2\beta^2(\lambda_1|p^1-q^1| + \lambda_2\Delta_{22}^2|p^2-q^2|).
}
To derive a contradiction to \eqref{path_integral}, it suffices to show that either
\eq{
2\beta^2(\lambda_1\Delta_{11}^2|p^1-q^1| + \lambda_2|p^2-q^2|) < |p^1-q^1|
}
or
\eq{
2\beta^2(\lambda_1|p^1-q^1| + \lambda_2\Delta_{22}^2|p^2-q^2|) < |p^2-q^2|.
}
By scaling, this is equivalent to showing that for any $t\in[0,\infty]$, either
\eeq{ \label{first_curve}
L_1(t) \coloneqq 2\beta^2\Big(\lambda_1\Delta_{11}^2 + \lambda_2t\Big) < 1
}
or
\eq{
L_2(t) \coloneqq 2\beta^2\Big(\lambda_1\frac{1}{t} + \lambda_2\Delta_{22}^2\Big) < 1.
}
Since $L_1'(t) > 0 > L_2'(t)$ with $L_1(t) \to \infty$ as $t\to\infty$ and $L_2(t) \to \infty$ as $t\to0$,
the maximum value of $\min(L_1(t),L_2(t))$ will be achieved at the unique $t > 0$ such that $L_1(t)=L_2(t)$.
For this value of $t$ we have
\eq{
\lambda_1\Delta_{11}^2 + \lambda_2 t &= \lambda_1\frac{1}{t}+\lambda_2\Delta_{22}^2
\quad \implies \quad t = \frac{-\lambda_1\Delta_{11}^2+\lambda_2\Delta_{22}^2 + \sqrt{(\lambda_1\Delta_{11}^2-\lambda_2\Delta_{22}^2)^2+4\lambda_1\lambda_2}}{2\lambda_2}.
}
Using this value of $t$ in \eqref{first_curve}, we conclude that a contradiction is realized as soon as
\eq{
2\beta^2\frac{\lambda_1\Delta_{11}^2+\lambda_2\Delta_{22}^2 + \sqrt{(\lambda_1\Delta_{11}^2-\lambda_2\Delta_{22}^2)^2+4\lambda_1\lambda_2}}{2} < 1,
}
which is exactly \eqref{uniqueness_condition}.
\end{proof}

\section{Hessian condition for symmetry breaking} \label{sec:hessian}
In this section we prove Theorem \ref{2speciessymmetrybreaking}, which generalizes the de Almeida--Thouless condition for symmetry breaking in the SK model \cite{almeida-thouless78}.
The proof is a perturbative argument following the strategy of \cite{toninelli02}, in which a symmetry breaking parameter is introduced alongside the RS critical point---whose uniqueness was established in Section \ref{sec:uniqueness}---by bringing the latter's atomic weight $\zeta$ just slightly away from $1$.
A nice exposition of the original argument in the single-species case can by found in \cite[Section 13.3]{talagrand11II}.

Assume $h > 0$ and fix an RS critical point $q_*\in\CC(\beta,h)$.
For any $p\in[0,1]^M$ such that $p^s\geq q_*^s$ for each $s$, we can define
\eq{
V(p) = \frac{\partial \mathscr{P}_\mathrm{1RSB}(q_*,p,\zeta)}{\partial \zeta}\Big|_{\zeta = 1}.
}
Here $\mathscr{P}_\mathrm{1RSB}$ is the Parisi functional restricted to level-$1$ symmetry breaking.
An explicit expression \eqref{parisi_1level} is computed in Section \ref{sec:appendix}.
The quantity $V(p)$ is useful because of the following observations.

\begin{lemma} \label{V_calculations}
Fix any $q = q_*\in\CC(\beta,h)$, and recall $Q_1^s$ defined by \eqref{Qdef}.
With the notation
\eq{
\gamma_s = \lambda_s\E\sech^4(\beta\eta\sqrt{Q_1^s}+h), \quad s = 1,\dots,M,
}
the following equalities hold:
\begin{itemize}
\item[(a)] $V(q_*) = 0$
\item[(b)] $\nabla V(q_*) = 0$
\item[(c)] $HV(q_*) = \beta^{2}\Lambda(2\beta^{2}\Delta^{2}\Gamma\Delta^{2}-\Delta^{2})\Lambda$, where $\Lambda$ and $\Gamma$ are the $M\times M$ diagonal matrices with diagonal entries $(\lambda_s)_{s=1}^M$ and $(\gamma_s)_{s=1}^M$, respectively.
\end{itemize}
\end{lemma}
The calculations in Lemma \ref{V_calculations} are straightforward generalizations of the (elegant but somewhat tricky) procedure found in \cite[Section 13.3]{talagrand11II}, and thus postponed to Section \ref{sec:appendix}.
Most important, part (c) identifies a condition for symmetry breaking once we note the following result.

\begin{cor} \label{general_RSB_condition}
Assume \eqref{variational} and that $\mathrm{RS}(\beta,h) = \mathscr{P}_\mathrm{RS}(q_*)$ for some $q_*\in\CC(\beta,h)$.
If there exists a $x\in\R^M$ with all nonnegative entries such that $x^\intercal HV(q_*)x>0$, then
\eeq{
\lim_{N\to\infty} F_N(\beta) < \mathrm{RS}(\beta,h). \label{RSB_conclusion}
}
\end{cor}

\begin{proof}
First note that by \eqref{RS_condition}, we must have $q_*\in[0,1)^M$.
Since $\nabla V(q_*) = 0$ and $x^\intercal HV(q_*)x > 0$, there exists $\eps > 0$ small enough that $V(q_*+\eps x) > V(q_*) = 0$.
That is,
\eq{
\frac{\partial \mathscr{P}_\mathrm{1RSB}(q_*,q_*+\eps x,\zeta)}{\partial \zeta}\Big|_{\zeta=1} > 0,
}
implying there is $\zeta < 1$ such that
\eq{
\mathscr{P}_\mathrm{1RSB}(q_*,q_*+\eps x,\zeta) < \mathscr{P}_\mathrm{1RSB}(q_*,q_*+\eps x,1) \stackrel{{\mbox{\footnotesize\eqref{RS_equivalence}}}}{=} \mathscr{P}_\mathrm{RS}(q_*)
= \mathrm{RS}(\beta,h).
}
Because of \eqref{variational}, \eqref{RSB_conclusion} follows.
\end{proof}

It is apparent from Corollary \ref{general_RSB_condition} that in order to obtain the correct AT condition for multi-species models, one must have good understanding of the set $\CC(\beta,h)$.
Thanks to Theorem \ref{uniquenessof2speciesrRSsol}, this has been accomplished in the two-species case for $h>0$, allowing us to proceed with the following argument.

\begin{proof}[Proof of Theorem \ref{2speciessymmetrybreaking}]
We return to the case $M=2$.
We know from Theorem \ref{uniquenessof2speciesrRSsol} that $q_* \in \CC(\beta,h)$ is unique, and moreover that the hypothesis of Corollary \ref{general_RSB_condition} holds.
Therefore, it suffices to show that \eqref{1RSB_condition} implies the existence of $x\in\R^2$ with nonnegative entries such that $x^\intercal H x > 0$, where $H = HV(q_*)$.
To simplify the task, we let $K \coloneqq 2\beta^{2}\Delta^{2}\Gamma\Delta^{2}-\Delta^{2}$ and note that
\eq{
x^\intercal K x > 0 \quad \implies \quad (\Lambda^{-1}x)^\intercal H (\Lambda^{-1}x)
= \beta^2 x^\intercal K x > 0.
}
Therefore, we can replace $H$ by $K$, since multiplication by $\Lambda^{-1}$ preserves nonnegativity of coordinates.

More specifically, we have $K=\begin{pmatrix}u &v \\ v &t\end{pmatrix}$, where
\eq{
u&=2\beta^{2}(\gamma_{1}(\Delta_{11}^{2})^2+\gamma_{2})-\Delta_{11}^2,\\
t&=2\beta^{2}(\gamma_{1}+\gamma_{2}(\Delta_{22}^{2})^{2})-\Delta_{22}^{2}, \\
v&=2\beta^{2}(\gamma_{1}\Delta_{11}^2+\gamma_{2}\Delta_{22}^2)-1.
}
Now, $x^\intercal K x > 0$ for some $x$ with nonnegative entries if and only if at least one of the following three inequalities is true:
\begin{subequations}
\label{2speciescondition}
\begin{align}
u&>0\quad\text{or} \label{2speciescondition_1}\\
t&>0\quad\text{or} \label{2speciescondition_2}\\
u,t &\leq 0 \quad\text{and}\quad \sqrt{ut} <v \label{2speciescondition_3}. 
\end{align}
\end{subequations}
Direct computation shows that \eqref{2speciescondition} is equivalent to
\begin{subequations} \label{2speciescondition_new}
\begin{align}
2\beta^2 &> 2\beta^{2}_{u}\coloneqq\frac{\Delta_{11}^2}{\gamma_{1}(\Delta_{11}^2)^{2}+\gamma_{2}} \quad \text{or} \\
2\beta^2 &>2\beta^{2}_{t}\coloneqq\frac{\Delta_{22}^2}{\gamma_{1}+\gamma_{2}(\Delta_{22}^2)^{2}} \quad \text{or} \\
\min(2\beta^2_u,2\beta^2_t)\geq2\beta^2&>2\beta^{2}_{v}\coloneqq\frac{1}{\gamma_{1}\Delta_{11}^2+\gamma_{2}\Delta_{22}^2} 
\quad \text{and} \quad 2\beta^{2}_{m}<2\beta^{2} <2\beta^{2}_{M},
\end{align}
\end{subequations}
where
\eq{
2\beta^{2}_{m}&=\frac{2}{\gamma_1 \Delta_{11}^2 +\gamma_2 \Delta_{22}^2 +\sqrt{(\gamma_1 \Delta_{11}^2 -\gamma_2 \Delta_{22}^2)^2 +4 \gamma_1 \gamma_2}}, \\
2\beta^{2}_{M}&=\frac{2}{\gamma_1 \Delta_{11}^2 +\gamma_2 \Delta_{22}^2 -\sqrt{(\gamma_1 \Delta_{11}^2 -\gamma_2 \Delta_{22}^2)^2 +4 \gamma_1 \gamma_2}}.
}
Note that
\eq{
(\gamma_1\Delta_{11}^2-\gamma_2\Delta_{22}^2)^2+4\gamma_1\gamma_2
&< (\gamma_1\Delta_{11}^2-\gamma_2\Delta_{22}^2)^2+4\gamma_1\gamma_2\Delta_{11}^2\Delta_{22}^2 
= (\gamma_1\Delta_{11}^2+\gamma_2\Delta_{22}^2)^2,
}
which ensures $0<\beta^2_v < \beta^2_m < \beta^2_M$.
We also claim that
\eeq{ \label{order_claim}
\beta^2_m < \min(\beta_u^2,\beta_t^2) \leq \max(\beta_u^2,\beta_t^2) < \beta^2_M.
}
For instance, suppose $\gamma_1\Delta_{11}^2 \leq \gamma_2\Delta_{22}^2$.
Then
\eq{
\gamma_1\Delta_{11}^2(\Delta_{11}^2\Delta_{22}^2-1) &\leq \gamma_2\Delta_{22}^2(\Delta_{11}^2\Delta_{22}^2-1) \\
\implies \quad (\gamma_1(\Delta_{11}^2)^2+\gamma_2)\Delta_{22}^2 &\leq (\gamma_1+\gamma_2(\Delta_{22}^2)^2)\Delta_{11}^2 \\
\implies \quad
2\beta_t^2 = \frac{\Delta_{22}^2}{\gamma_1 + \gamma_2(\Delta_{22}^2)^2} &\leq \frac{\Delta_{11}^2}{\gamma_1(\Delta_{11}^2)^2+\gamma_2} = 2\beta_u^2.
}
To establish the first bound in \eqref{order_claim}, we observe that $\beta_t^2 > \beta_m^2$ if and only if
\eq{
\Delta_{22}^2\Big(\gamma_1\Delta_{11}^2+\gamma_2\Delta_{22}^2+\sqrt{(\gamma_1\Delta_{11}^2-\gamma_2\Delta_{22}^2)^2+4\gamma_1\gamma_2}\, \Big)&> 2\big(\gamma_1+\gamma_2(\Delta_{22}^2)^2\big) \\
\iff \quad \Delta_{22}^2\sqrt{(\gamma_1\Delta_{11}^2-\gamma_2\Delta_{22}^2)^2+4\gamma_1\gamma_2} &> \gamma_1(2-\Delta_{11}^2\Delta_{22}^2)+\gamma_2(\Delta_{22}^2)^2 \\
\iff \quad 0 &> 4\gamma_1^2(1-\Delta_{11}^2\Delta_{22}^2),
}
which is true by \eqref{delta_assumptions}.
To establish the last bound in \eqref{order_claim}, we drop the term $4\gamma_1\gamma_2$ from the denominator of $\beta_M^2$:
\eq{
2\beta_M^2 > \frac{2}{\gamma_1\Delta_{11}^2+\gamma_2\Delta_{22}^2-\sqrt{(\gamma_1\Delta_{11}^2-\gamma_2\Delta_{22}^2)^2}}
= \frac{1}{\gamma_1\Delta_{11}^2} > 2\beta_u^2.
}
We have thus proved \eqref{order_claim} under the assumption $\gamma_1\Delta_{11}^2\leq\gamma_2\Delta_{22}^2$, but the proof is analogous in the reverse case.
Finally, because of \eqref{order_claim}, we see that \eqref{2speciescondition_new} is equivalent to the single condition $\beta^{2}>\beta^{2}_{m}$.
\end{proof}


%
%

\section{Proof of Lemma \RefVLemma} \label{sec:appendix}
Here we consider \eqref{parisi_expression} when $k = 1$, and
\eq{
\zeta_1 = \zeta \in (0,1), \qquad
q_1^s = q^s, \qquad 
q_2^s = p^s.
}
Recall that with these choices, we have $Q_0 = Q_0^s = 0$ and
\eq{
Q_1^s &= 2\sum_t \Delta_{st}^2\lambda_tq^t, &
Q_1 &= \sum_{s,t}\Delta_{st}^2\lambda_s\lambda_tq^sq^t, \\
Q_2^s &= 2\sum_{t}\Delta_{st}^2\lambda_tp^t, &
Q_2 &= \sum_{s,t}\Delta_{st}^2\lambda_s\lambda_tp^sp^t, \\
Q_3^s &= 2\sum_{t}\Delta_{st}^2\lambda_t, &
Q_3 &= \sum_{s,t}\Delta_{st}^2\lambda_s\lambda_t.
}
We have
\eq{
\mathscr{P}_\mathrm{1RSB}(q,p,\zeta) &\coloneqq \log 2 + \sum_s \lambda_s X_0^s - \frac{\beta^2}{2}\sum_{\ell=1}^2 \zeta_\ell(Q_{\ell+1}-Q_{\ell}) \\
&= \log 2 + \sum_s \lambda_s X_0^s - \frac{\beta^2}{2}(Q_3 - Q_2+\zeta(Q_2-Q_1)),
}
where
\eq{
X_3^s &= 
\log \cosh(\beta\eta_3\sqrt{Q_3^s-Q_2^s} + \beta\eta_2\sqrt{Q_2^s-Q_1^s}+\beta\eta_1\sqrt{Q_1^s}+h) \\
\implies X_2^s &= \log \E_3 \cosh(\beta\eta_3\sqrt{Q_3^s-Q_2^s} + \beta\eta_2\sqrt{Q_2^s-Q_1^s}+\beta\eta_1\sqrt{Q_1^s}+h) \\
&= \frac{\beta^2}{2}(Q_3^s-Q_2^s) + \log\cosh(\beta\eta_2\sqrt{Q_2^s-Q_1^s}+\beta\eta_1\sqrt{Q_1^s}+h) \\
\implies X_1^s &= \frac{1}{\zeta}\log\E_2 \exp(\zeta_2 X_2^s) \\
&= \frac{\beta^2}{2}(Q_3^s-Q_2^s) + \frac{1}{\zeta}\log \E_2 \cosh^\zeta(\beta\eta_2\sqrt{Q_2^s-Q_1^s} + \beta\eta_1\sqrt{Q_1^s}+h) \\
\implies X_0^s &= \E X_1^s = \frac{\beta^2}{2}(Q_3^s-Q_2^s) + \frac{1}{\zeta} \E_1\log\E_2 \cosh^\zeta(\beta\eta_2\sqrt{Q_2^s-Q_1^s} + \beta\eta_1\sqrt{Q_1^s}+h).
}
In simplifying $X_2^s$, we have used the fact that for $\eta \sim \NN(0,1)$ and $\sigma>0$,
\eeq{ \label{gaussian_calculation}
\E \cosh(\sigma \eta +h) &= \frac{1}{2}\E(\e^{\sigma \eta+h} + \e^{-\sigma \eta-h}) 
= \frac{1}{2}(\e^{\sigma^2/2 + h} + \e^{\sigma^2/2 - h}) = \e^{\sigma^2/2}\cosh(h).
}
In summary,
\eeq{ \label{parisi_1level}
\mathscr{P}_\mathrm{1RSB}(q,p,\zeta) &= \log 2+ \sum_s \lambda_s \frac{1}{\zeta} \E_1\log\E_2 \cosh^\zeta(\beta\eta_2\sqrt{Q_2^s-Q_1^s} + \beta\eta_1\sqrt{Q_1^s}+h) \\ 
&\phantom{=} +\frac{\beta^2}{2}\sum_s \lambda_s (Q_3^s - Q_2^s) - \frac{\beta^2}{2}(Q_3-Q_2+\zeta(Q_2-Q_1)).
}
Notice that when $\zeta = 1$, we recover the replica symmetric expression \eqref{RS_formula}:
\eeq{ \label{RS_equivalence}
\mathscr{P}_\mathrm{1RSB}(q,p,1) &= \log 2 + \sum_s\lambda_s \E_1\log\E_2\cosh(\beta\eta_2\sqrt{Q_2^s-Q_1^s} + \beta\eta_1\sqrt{Q_1^s}+h) \\
&\phantom{=}+ \frac{\beta^2}{2} \sum_s \lambda_s(Q_3^s - Q_2^s)- \frac{\beta^2}{2}(Q_3-Q_2+Q_2-Q_1) \\
&= \log 2 + \sum_s \lambda_s \E_1\bigg[\frac{\beta^2}{2}(Q_2^s-Q_1^s) + \log \cosh(\beta\eta_1\sqrt{Q_1^s}+h)\bigg] \\
&\phantom{=}+ \frac{\beta^2}{2} \sum_s \lambda_s(Q_3^s - Q_2^s) - \frac{\beta^2}{2}(Q_3-Q_1) \\
&= \log 2 + \sum_s \lambda_s \E_1\log \cosh(\beta\eta_1\sqrt{Q_1^s}+h) \\
&\phantom{=}+ \frac{\beta^2}{2}\sum_s \lambda_s(Q_3^s-Q_1^s) - \frac{\beta^2}{2}(Q_3-Q_1) = \mathscr{P}_\mathrm{RS}(q).
}
Henceforth fix an RS critical point $q = q_*\in\CC(\beta,h)$.
For ease of notation, let us write 
\eq{
Y_1^s &\coloneqq \beta\eta_1\sqrt{Q_1^s}+h, \qquad Y_2^s \coloneqq \beta\eta_2\sqrt{Q_2^s-Q_1^s} +Y_1^s,
}
so that
\eq{
\mathscr{P}_\mathrm{1RSB}(q,p,\zeta) &= \log 2+ \sum_s \lambda_s \frac{1}{\zeta} \E_1\log\E_2 \cosh^\zeta Y_2^s \\
&\phantom{=}+\frac{\beta^2}{2}\sum_s \lambda_s (Q_3^s - Q_2^s) -\frac{\beta^2}{2}(Q_3-Q_2+\zeta(Q_2-Q_1)).
}
We can then calculate
\eq{
\frac{\partial \mathscr{P}_\mathrm{1RSB}(q_*,p,\zeta)}{\partial \zeta}
&= \sum_s \lambda_s\bigg[\frac{-\E_1\log\E_2 \cosh^\zeta Y_2^s}{\zeta^2} +\E_1\bigg(\frac{\E_2\log(\cosh Y_2^s)\cosh^\zeta Y_2^s}{\zeta\E_2 \cosh^\zeta Y_2^s}\bigg)\bigg] - \frac{\beta^2}{2}(Q_2-Q_1),
}
which gives
\eeq{ \label{pre_deriv}
V(p) &= \sum_s\lambda_s\bigg[-\E_1\log\E_2 \cosh Y_2^s + \E_1\bigg(\frac{\E_2\log(\cosh Y_2^s)\cosh Y_2^s}{\E_2\cosh Y_2^s}\bigg)\bigg]  - \frac{\beta^2}{2}(Q_2-Q_1).
}
When $p = q_*$, we have $Y_2^s = Y_1^s$ and $Q_2 = Q_1$.
In particular, $Y_2^s$ has no dependence on $\eta_2$, and so the above expression reduces to $V(q_*) = 0$.
This proves claim (a).

The next step is to take partial derivatives with respect to the $p^t$.
First, we use \eqref{gaussian_calculation} to make the simple calculation
\eeq{ 
\E_2 \cosh Y_2^s = \exp\Big(\frac{\beta^2}{2}(Q_2^s-Q_1^s)\Big)\cosh Y_1^s. \label{Y_switch}
}
Since $Y_1^s$ has no dependence on $p$, and
\eeq{ \label{Q_deriv}
\frac{\partial}{\partial p^t}(Q_2^s-Q_1^s) = \frac{\partial}{\partial p^t}Q_2^s = 2\Delta_{st}^2\lambda_t,
}
we find
\eeq{ \label{deriv_1}
\frac{\partial}{\partial p^t} \E_1\log\E_2\cosh Y_2^s 
=\frac{\partial}{\partial p^t}\Big[\frac{\beta^2}{2}(Q_2^s-Q_1^s)+\E_1\log\cosh Y_1^s\Big]
= \beta^2\Delta_{st}^2\lambda_t.
}
Next we observe that for any twice differentiable function whose derivatives have at most exponential growth at infinity, \eqref{Q_deriv} and Gaussian integration by parts together give
\eeq{ \label{f_deriv}
\frac{\partial}{\partial p^t} \E_2 f(Y_2^s) 
=  \beta\Delta_{st}^2\lambda_t\E_2\bigg(f'(Y_2^s)\frac{\eta_2}{\sqrt{Q_2^s-Q_1^s}}\bigg)
&= \beta^2\Delta_{st}^2\lambda_t\E_2 f''(Y_2^s).
}
Hence
\eeq{ \label{expectation_formula}
&\frac{\partial}{\partial p^t} \E_1\bigg(\frac{\E_2f(Y_2^s)}{\E_2\cosh Y_2^s}\bigg)
\stackrel{\hspace{2.7ex}\mbox{\footnotesize\eqref{Y_switch}}\hspace{2.7ex}}{=} \frac{\partial}{\partial p^t}\bigg[ \exp\Big(-\frac{\beta^2}{2}(Q_2^s-Q_1^s)\Big)\E_1\bigg(\frac{\E_2 f(Y_2^s)}{\cosh Y_1^s}\bigg)\bigg] \\
&\stackrel{\mbox{\footnotesize\eqref{Q_deriv},\eqref{f_deriv}}}{=} \beta^2\Delta_{st}^2\lambda_t\exp\Big(-\frac{\beta^2}{2}(Q_2^s-Q_1^s)\Big)\bigg[- \E_1\bigg(\frac{\E_2 f(Y_2^s)}{\cosh Y_1^s}\bigg) + \E_1\bigg(\frac{\E_2f''(Y_2^s)}{\cosh Y_1^s}\bigg)\bigg] \\
&\stackrel{\phantom{\mbox{\footnotesize\eqref{Q_deriv},\eqref{f_deriv}}}}{=} \beta^2\Delta_{st}^2\lambda_t\exp\Big(-\frac{\beta^2}{2}(Q_2^s-Q_1^s)\Big)\E_1\bigg(\frac{\E_2[f''(Y_2^s)-f(Y_2^s)]}{\cosh Y_1^s}\bigg) \\
&\stackrel{\hspace{2.7ex}\mbox{\footnotesize\eqref{Y_switch}}\hspace{2.7ex}}{=} \beta^2\Delta_{st}^2\lambda_t\E_1\bigg(\frac{\E_2[f''(Y_2^s)-f(Y_2^s)]}{\E_2\cosh Y_2^s}\bigg).
}
We now apply \eqref{expectation_formula} to $f(x) = \log(\cosh x)\cosh x$, for which
\eq{
f''(x)-f(x) &= \cosh x + \sinh x \tanh x,
}
to obtain
\eeq{ \label{deriv_2}
\frac{\partial}{\partial p^t}\E_1\bigg(\frac{\E_2 \log(\cosh Y_2^s) \cosh Y_2^s}{\E_2\cosh Y_2^s}\bigg)\bigg] 
= \beta^2\Delta_{st}^2\lambda_t\bigg[1 + \E_1\bigg(\frac{\E_2 \sinh Y_2^s \tanh Y_2^s}{\E_2 \cosh Y_2^s}\bigg)\bigg].
}
Finally, we have
\eeq{ \label{deriv_3}
\frac{\partial}{\partial p^t}(Q_2 - Q_1) = \frac{\partial}{\partial p^t}Q_2 = 2\sum_s\Delta_{st}^2\lambda_s\lambda_t p^s. 
}
Using \eqref{deriv_1}, \eqref{deriv_2}, and \eqref{deriv_3} in \eqref{pre_deriv}, we arrive at
\eq{
\frac{\partial}{\partial p^t}V(p) = \beta^2\lambda_t\sum_s \Delta_{st}^2\lambda_s\bigg[\E_1\bigg(\frac{\E_2 \sinh Y_2^s \tanh Y_2^s}{\E_2 \cosh Y_2^s}\bigg) - p^s\bigg].
}
Once more, if $p = q_*$, then $Y_2^s = Y_1^s$ has no dependence on $\eta_2$, in which case
\eq{
\E_1\bigg(\frac{\E_2 \sinh Y_2^s \tanh Y_2^s}{\cosh Y_1^s}\bigg) - p^s
= \E_1(\tanh^2 Y_1^s) - q_*^s = q_*^s - q_*^s = 0 \quad \text{for all $s$.}
}
Consequently, claim (b) holds: $\nabla V(q_* ) = 0$.

Our final step is to compute the Hessian of $V$.
We have
\eq{
\frac{\partial^2}{\partial p^{t'}\partial p^t}V(p) = \beta^2\lambda_t\sum_{s}\Delta^2_{st}\lambda_s\bigg[\frac{\partial}{\partial p^{t'}} \E_1\bigg(\frac{\E_2\sinh Y_2^s \tanh Y_2^s}{\E_2\cosh Y_2^s}\bigg) - \delta_{st'}\bigg],
}
where $\delta_{st'} = 1$ if $s = t'$ and $0$ zero otherwise.
To determine the derivative of the expectation, we apply \eqref{expectation_formula} with $f(x) = \sinh x \tanh x$, for which 
\eq{
f''(x)-f(x) &= 2\mathrm{sech}^3\,x.
}
After doing so, we arrive at
\eq{
\frac{\partial^2}{\partial p^{t'}\partial p^{t}}V(p) &= \beta^2\lambda_t\sum_s \Delta^2_{st}\lambda_s\bigg[2\beta^2\Delta_{st'}^2\lambda_{t'}\E_1\bigg(\frac{\E_2\sech^3 Y_2^s}{\E_2\cosh Y_2^s}\bigg) - \delta_{st'}\bigg].
\intertext{As before, the expression simplifies when $p=q^*$, since then $Y_2^s=Y_1^s$ has no dependence on $\eta_2$.
Namely,}
\frac{\partial^2}{\partial p^{t'}\partial p^{t}}V(q_*)
&= \beta^2\lambda_t\sum_{s}\Delta_{st}^2\lambda_s\big[2\beta^2\Delta_{st'}^2\lambda_{t'}\E_1\, \mathrm{sech}^4\, Y_1^s - \delta_{st'}\big] \\
&= \bigg(2\beta^4\lambda_t\lambda_{t'}\sum_s \Delta_{st}^2\Delta_{st'}^2\lambda_s \E_1\, \mathrm{sech}^4\, Y_1^s\bigg) - \beta^2\lambda_t\lambda_{t'}\Delta_{tt'}^2.
}
Rewriting the expression in terms of matrices yields claim (c).

    
    \chapter{Fluctuations in planar random growth models} \label{fluctuations}
    
    \section{Introduction}
Even after years of study on random growth models, such as first- and last-passage percolation and directed polymers, much remains mysterious or out of reach technically.
For instance, beyond the fundamental shape theorems guaranteeing linear growth rates for the passage times/free energy, there are sublinear fluctuations whose asymptotics are not established.
Even in the planar setting, for which the conjectural picture is clear, general tools are far from making it rigorous.
This is in stark contrast with integrable models, for which fluctuation exponents are only a fraction of what has been proved.
In this chapter we consider three widely studied random growth models: first-passage percolation (FPP), last-passage percolation (LPP), and directed polymers in random environment.
While the models differ in how growth is measured, they each possess a law of large numbers that says the rate of growth is asymptotically linear.
More mysterious, however, are the sublinear fluctuations.
In their two-dimensional versions, these models are believed to belong to the Kardar--Parisi--Zhang universality class \cite{corwin12}, and in particular that growth fluctuations are of order $n^{1/3}$.
Except in exceptional cases of LPP and directed polymers having exact solvability properties, rigorous results are far from this goal, or in some cases non-existent.

The goal of this chapter is two-fold.
First, we describe a general strategy for proving lower bounds on the order of fluctuations for a sequence of random variables (defined precisely in Definition \ref{fluctuations_def}). The approach is an adaptation of techniques developed recently by the second author in \cite{chatterjee19II}. It is general in that it can be used in a wide variety of problems consisting of i.i.d.~random variables, where no assumptions are made on the common distribution of these variables.
Second, we apply the method to study fluctuations in the growth of planar FPP, LPP, and directed polymers.
In all three cases, we are able to prove a lower bound of order $\sqrt{\log n}$ fluctuations.
In addition, for FPP we extend the shape fluctuation lower bound of $n^{1/8-\delta}$ to almost all distributions for which it should be true.
Although still far from $n^{1/3}$, which by all accounts is the correct order (e.g.~see \cite{sosoe18} and references therein), our results require almost no assumptions on the underlying weight distribution.

The chapter is structured as follows.
The general method mentioned above for establishing fluctuation lower bounds is outlined in Section \ref{fluctuations_sec}, and some necessary lemmas are proved.
The random growth models under consideration are introduced in Section \ref{fluctuations_results}, where the main results are also stated.
Finally, Section \ref{proofs} sees the method put into action to prove these results.

\section{General method for lower bounds on fluctuations} \label{fluctuations_sec}

\subsection{Definitions}

Let us begin by precisely stating what is meant by a lower bound on fluctuations.

\begin{defn} \label{fluctuations_def}
Let $(X_n)_{n\geq1}$ be a sequence of random variables, and let $(\delta_n)_{n\geq1}$ be a sequence of positive real numbers.
We will say that $X_n$ has fluctuations of order at least $\delta_n$ if there are positive constants $c_1$ and $c_2$ such that for all large $n$, and for all $-\infty < a \leq b < \infty$ with $b-a \leq c_1\delta_n$, one has $\P(a\leq X_n\leq b) \leq 1 - c_2$.
\end{defn}

In other words, fluctuations are of order at least $\delta_n$ if no sequence of intervals $I_n$ of length $o(\delta_n)$ satisfies $\P(X_n\in I_n) \to 1$.
Note that if fluctuations are at least of order $\delta_n$, then so is $\sqrt{\Var(X_n)}$.
The converse, however, is not true in general, necessitating alternative approaches even when a lower bound on variance is known.
On the other hand, if a variance lower bound is accompanied by an upper bound of the same order, then fluctuations must be of that order.
One can see this from a second moment argument, for instance using the Paley--Zygmund inequality.
In the absence of matching variance bounds, one must work with Definition \ref{fluctuations_def} directly.
For this reason, the following simple lemma is useful.
Here $d_{\tv}(\nu_1,\nu_2)$ is the \textit{total variation distance} between probability measures $\nu_1,\nu_2$ on the same measurable space $(\Omega,\FF)$, defined as
\eq{
d_{\tv}(\nu_1,\nu_2) \coloneqq \sup_{A \in \FF} |\nu_1(A) - \nu_2(A)|.
}

\begin{lemma}[{\cite[Lemma 1.2]{chatterjee19II}}] \label{fluctuation_lemma}
Let $X$ and $Y$ be random variables defined on the same probability space.
For any $-\infty < a \leq b < \infty$,
\eq{
\P(a \leq X \leq b) \leq \frac{1}{2}\Big(1 + \P(|X-Y| \leq b-a) + d_{\tv}(\LL_X,\LL_{Y})\Big),
} 
where $\LL_X$ and $\LL_Y$ denote the laws of $X$ and $Y$, respectively.
\end{lemma}

\begin{proof}
We have
\eq{
\P(a\leq X \leq b) &= \P(a\leq X,Y\leq b) + \P(a\leq X\leq b, Y\notin[a,b]) \\
&\leq \P(|X-Y| \leq b - a) + \P(Y\notin[a,b]) \\
&\leq \P(|X-Y|\leq b-a) + \P(X\notin[a,b]) + d_{\tv}(\LL_X,\LL_Y) \\
&= \P(|X-Y|\leq b-a) + 1 - \P(X\in[a,b]) + d_{\tv}(\LL_X,\LL_Y).
}
Now rearranging terms yields the claim.
\end{proof}

A simple but important fact is that total variation distance can be related to the \textit{Hellinger affinity} between $\mu$ and $\wt\mu$,
\eeq{ \label{hellinger_def}
\rho(\nu_1,\nu_2) \coloneqq \int_\Omega \sqrt{fg}\ \dd\nu_0,
}
where $\nu_0$ is any probability measure on $(\Omega,\FF)$ with respect to which both $\nu_1$ and $\nu_2$ are absolutely continuous, and $f$ and $g$ are their respective densities.
Since
\eq{
d_{\tv}(\nu_1,\nu_2) = \frac{1}{2}\int_\Omega |f-g|\ \dd\nu_0,
}
the following upper bound follows from the Cauchy--Schwarz inequality:
\eeq{ \label{tv_hellinger}
d_{\tv}(\nu_1,\nu_2) \leq \sqrt{1 - \rho(\nu_1,\nu_2)^2}.
}

\subsection{The general method}
To produce a lower bound on the order of fluctuations using Lemma \ref{fluctuation_lemma}, the basic idea is to introduce a coupling $(X,Y)$ such that $|X-Y|$ is large with substantial probability while $d_{\tv}(\LL_X,\LL_Y)$ is small.
A general approach formalizing this idea was initiated in \cite{chatterjee19II}, in which the couplings are obtained from multiplicative perturbations inspired by the Mermin--Wagner theorem of statistical mechanics~\cite{mermin-wagner66}.
Such couplings only work, however, for a certain class of random variables, namely those with \eeq{ \label{previous_assumption}
&\text{density proportional to $\e^{-V}$, where $V \in C^\infty(\R)$, such that} \\
&\text{$V$ and its derivatives of all orders have at most polynomial growth, and} \\
&\text{$\e^V$ grows faster than any polynomial.}
}
We now propose a different type of coupling that allows for the approach of \cite{chatterjee19II} to be extended to any distribution.
Although the couplings we will use to prove the main theorems of this chapter are more specific, we present here the most general setup in hopes that the method might be useful in other settings.

Consider a real-valued random variable $X$ defined on some probability space $(\Omega,\FF,\P)$.
Let $\LL_X$ denote the law of $X$.
Suppose $X'$ is another random variable defined on the same probability space, such that
$\LL_{X'}$ is absolutely continuous with respect to $\LL_X$ and has bounded density.
Given $\eps \in (0,1)$, let $Y$ be a Bernoulli($\eps$) random variable independent of $X$ and $X'$.
Finally, set
\eeq{ \label{new_X}
\wt X = \begin{cases}
X' &\text{if }Y=1, \\
X &\text{if }Y=0.
\end{cases}
}

\begin{lemma} \label{hellinger_lemma}
The Hellinger affinity between $\LL_X$ and $\LL_{\wt X}$ satisfies the lower bound
\eq{
\rho(\LL_X,\LL_{\wt X}) \geq 1 - C\eps^2,
}
where $C$ is a constant depending only on $\LL_X$ and $\LL_{X'}$. 
\end{lemma}

\begin{proof}
Let us denote the density of $\LL_{X'}$ with respect to $\LL_{X}$ by $f(t)$, which we assume to be bounded;
say $f(t) \leq M$.
It is easy to see that $\eps f(t) + 1-\eps$ is the density of $\LL_{\wt X}$ with respect to $\LL_{X}$, and so
\eq{
\rho(\LL_X,\LL_{\wt X}) = \int_\R \sqrt{\eps f(t) + 1-\eps}\ \LL_X(\dd t).
}
For $\eps < 1/M$, we can write the Taylor expansion
\eq{
\sqrt{1-\eps[1-f(t)]} = 1 - \frac{\eps}{2}[1-f(t)] - \frac{\eps^2}{8}[1-f(t)]^2 + \eps^3r(t),
}
where $r(t)$ is bounded.
In fact, the entire right-hand side above is bounded, and so there is no problem in writing
\eq{
\rho(\LL_X,\LL_{\wt X}) &= \int_\R\Big(1 - \frac{\eps}{2}[1-f(t)] - \frac{\eps^2}{8}[1-f(t)]^2 + \eps^3r(t)\Big)\ \LL_X(\dd t).
}
Using the fact that $\int_\R f(t)\, \LL_X(\dd t) = 1$, we find
\eq{
\rho(\LL_X,\LL_{\wt X}) &= 1 - \frac{\eps^2}{8}\int_\R[1-f(t)]^2\ \LL_X(\dd t) + O(\eps^3) \geq 1 - C\eps^2,
}
where $C$ depends only on $\LL_X$ and $\LL_{X'}$.
Replacing $C$ by $\max(C,M^2)$ allows the statement to also hold trivially for $\eps \geq 1/M$.
\end{proof}

When the same type of coupling is applied to several i.i.d.~variables, we get the following bound which can be used in Lemma \ref{fluctuation_lemma}.

\begin{lemma} \label{tv_lemma}
Let $X_1,\dots,X_n$ be i.i.d.~random variables with law $\LL_X$, and $X_1',\dots,X_n'$ be $i.i.d.$~random variables with law $\LL_{X'}$. 
Assume $\LL_{X'}$ is absolutely continuous with respect to $\LL_X$ with bounded density.
For each $i = 1,\dots,n$, let $Y_i$ be a $\mathrm{Bernoulli}(\eps_i)$ random variable independent of everything else, and define $\wt X_i$ as in \eqref{new_X} with $\eps = \eps_i$.
Then
\eq{
d_{\tv}(\LL_{(X_1,\dots,X_n)},\LL_{(\wt X_1,\dots,\wt X_n)}) \leq C\bigg(\sum_{i=1}^n \eps_i^2\bigg)^{1/2},
}
where $C$ is a constant depending only on $\LL_X$ and $\LL_{X'}$.
\end{lemma}

\begin{proof}
By properties of product measures, it is clear from the definition \eqref{hellinger_def} that
\eeq{ \label{product_hellinger}
\rho(\LL_{(X_1,\dots,X_n)},\LL_{(\wt X_1,\dots,\wt X_n)}) = \prod_{i=1}^n \rho(\LL_{X_i},\LL_{\wt X_i}).
}
Now let $C_0$ be the constant from Lemma \ref{hellinger_lemma}.
From \eqref{tv_hellinger}, \eqref{product_hellinger}, and Lemma \ref{hellinger_lemma}, we deduce
\eq{
d_{\tv}(\LL_{(X_1,\dots,X_n)},\LL_{(\wt X_1,\dots,\wt X_n)}) \leq \bigg(1 - \prod_{i=1}^n (1-C_0\eps_i^2)^2\bigg)^{1/2}.
}
The desired bound is now obtained by iteratively applying the inequality $(1-x)(1-y) \geq 1-x-y$ for $x,y\geq0$.
\end{proof}

\subsection{Choice of coupling}
Naturally there are many measures $\LL_{X'}$ that are absolutely continuous to $\LL_X$, but we look for one which can be naturally coupled to $\LL_X$ in such a way that $X'$ deviates from $X$ by as much as possible.
Without further assumptions on $\LL_X$, the possibilities can be rather limited.
Two choices that are always available, however, are
\eeq{ \label{our_coupling}
X' = \min(X,X^{(1)},\dots,X^{(m)}) \quad \text{or} \quad X' = \max(X,X^{(1)},\dots,X^{(m)}),
}
where $X^{(1)},\dots,X^{(m)}$ are independent copies of $X$.
Indeed, these are the two couplings we will use to prove results on fluctuations in planar random growth models.
It is easy to check that the bounded density condition from Lemma \ref{tv_lemma} is satisfied.

\begin{lemma}
For any law $\LL_X$ and any $m\geq1$, the law $\LL_{X'}$ of $X'$ given by \eqref{our_coupling} is absolutely continuous with respect to $\LL_X$, and has bounded density.
\end{lemma}

\begin{proof}
For any Borel set $A\subset\R$,
\eq{
\P(X'\in A) &\leq\P\bigg(\{X\in A\} \cup \bigcup_{j=1}^m \{X^{(j)}\in A\}\bigg) 
\leq \P(X\in A) + \sum_{j=1}^m \P(X^{(j)}\in A)
= (m+1)\P(X\in A).
}
It follows that $\P(X'\in A) = 0$ whenever $\P(X\in A)=0$, and that the density of $\LL_{X'}$ with respect to $\LL_{X}$ is bounded by $m+1$.
\end{proof}

For a specific distribution $\LL_X$, other couplings might also be useful and easier to work with.
For instance, if $X$ is a uniform random variable on $[0,1]$, one could take $X' = aX$ for any $a\in(0,1)$.
If 
$\P(X=0) > 0$, one could simply take $X' = 0$.
For $X$ that is geometrically distributed, $X' = X + a$ is also valid for any positive integer $a$.

\section[Random growth models: definitions, background, and results]{Planar random growth models: definitions, background, and results} \label{fluctuations_results}

\subsection{Two-dimensional first-passage percolation} \label{fpp_sec}

Let $ E(\Z^2)$ denote the edge set of $\Z^2$.
Let $(X_e)_{e\in E(\Z^2)}$ be an i.i.d.~family of nonnegative, non-degenerate random variables.
Along a nearest-neighbor path $\gamma = (\gamma_0,\gamma_1,\dots,\gamma_n)$, the \textit{passage time} is 
\eq{
T(\gamma) \coloneqq \sum_{i=1}^n X_{(\gamma_{i-1},\gamma_i)},
}
where $(\gamma_{i-1},\gamma_i)$ denotes the (undirected) edge between $\gamma_{i-1}$ and $\gamma_i$.
For $x,y\in\Z^2$, denote by $T(x,y)$ the minimum passage time of a path connecting $x$ and $y$; that is,
\eq{
T(x,y) \coloneqq \inf\{T(\gamma)\ :\ \gamma_0 = x, \gamma_n = y\}.
}
The quantity $T(x,y)$ is called the \textit{(first) passage time} between $x$ and $y$, and any path achieving this time will be called a \textit{(finite) geodesic}. 
For a recent survey on first-passage percolation, we refer the reader to \cite{auffinger-damron-hanson17}.

We are interested in the fluctuations of $T(x,y)$ when $x$ and $y$ are separated by a distance of order $n$.
In dimensions three and higher, there is actually no known lower bound other than the trivial observation that fluctuations are at least of order $1$.
In the planar setting considered here, order $\sqrt{\log n}$ fluctuations (in the sense of Definition \ref{fluctuations_def}) were established by Pemantle and Peres \cite{pemantle-peres94} when $X_e$ is exponentially distributed.
In \cite[Theorem 2.6]{chatterjee19II}, this lower bound was extended to the family of passage time distributions described in Section \ref{fluctuations_sec}, satisfying \eqref{previous_assumption}.
Our result below expands the result to optimal generality (cf.~Remark \ref{assumption_remark}).

Let $p_\cc(\Z^d)$ and $\vec{p}_\cc(\Z^d)$ denote the critical values for undirected and directed bond percolation on $\Z^d$.
When $d = 2$, we have $p_\cc(\Z^2) = 1/2$ and $\vec{p}_\cc(\Z^2) \approx 0.6445$ \cite[Chapter 6]{bollobas-riordan06}.
In order to have a rigorous upper bound, we cite the result of \cite{balister-bollobas-stacey94} which guarantees
\eeq{ \label{oriented_bond_upper}
\vec{p}_\cc(\Z^2) \leq 0.6735.
}

\begin{thm} \label{fpp_thm}
With $s \coloneqq \essinf X_e \in [0,\infty)$, assume
\eeq{
\P(X_e=s) < p_\cc(\Z^2). \label{fpp_assumption_1}
}
Let $y_n$ be any sequence in $\Z^2$ such that $\|y_n\|_1 \geq n$ for every $n$.
Then the fluctuations of $T(0,y_n)$ are at least of order $\sqrt{\log n}$.
\end{thm}

\begin{remark} \label{assumption_remark}
The above result is optimal in the following sense.
If $s = 0$ and $\P(X_e = 0) > p_\cc(\Z^d)$, then $T(0,y_n)$ is tight because there is an infinite cluster of zero-weight edges extending in every direction \cite{zhang-zhang84,zhang95}.
\end{remark}

When $s > 0$, we can relax \eqref{fpp_assumption_1} upon adding a weak moment condition \eqref{fpp_assumption_2b}.
This condition is standard in planar FPP and is equivalent to the limit shape having nonempty interior (see \eqref{limit_shape} and the discussion that follows).

\begin{thm} \label{fpp_thm_1}
With $s \coloneqq \essinf X_e \in [0,\infty)$, assume
\begin{subequations} \label{fpp_assumption_2}
\begin{align}
s > 0, \quad \P(X_e=s) < \vec{p}_\cc(\Z^2), \label{fpp_assumption_2a}
\intertext{and}
\E\min(X^{(1)},X^{(2)},X^{(3)},X^{(4)})^2 < \infty,
\label{fpp_assumption_2b}
\end{align}
\end{subequations}
where the $X^{(i)}$'s are independent copies of $X_e$.
Let $y_n$ be any sequence in $\Z^2$ such that $\|y_n\|_1 \geq n$ for every $n$.
Then the fluctuations of $T(0,y_n)$ are at least of order $\sqrt{\log n}$.
\end{thm}

\begin{remark}
As similarly mentioned in Remark \ref{assumption_remark}, the above result is optimal in the following sense.
If $s > 0$ and $\P(X_e = s) > \vec{p}_\cc(\Z^d)$, then $T(0,y_n) - n\|y_n\|_1$ is tight so long as $y_n$ is in or at the edge of the oriented percolation cone \cite[Remark 7]{zhang08} (c.f.~\cite{durrett84} for a description of this cone).
An independent work of Damron, Hanson, Houdr{\'e}, and Xu \cite{damron-hanson-houdre-xu20}, which uses different methods and was posted shortly after \cite{bates-chatterjee20III}, shows that Theorem \ref{fpp_thm_1} holds even if one assumes \eqref{fpp_assumption_2a} without \eqref{fpp_assumption_2b};  their Lemma 6 is the key innovation needed to remove this moment condition.
They also prove a statement equivalent to Theorem \ref{fpp_thm}.
\end{remark}

One should compare Theorems  \ref{fpp_thm} and \ref{fpp_thm_1} with the results of Newman and Piza \cite{newman-piza95}.
Under \eqref{fpp_assumption_1} or \eqref{fpp_assumption_2a}, and the additional assumption that $\E(X_e^2)$ is finite --- which is slightly stronger than \eqref{fpp_assumption_2b} --- they show $\Var(T(0,y_n))\geq C\log n$.
Zhang \cite[Theorem 2]{zhang08} shows the same for $y_n = (n,0)$ assuming only $\P(X_e = 0) < p_\cc(\Z^2)$, and Auffinger and Damron \cite[Corollary 2]{auffinger-damron13} extend this result to any direction outside the percolation cone (see also \cite[Corollary 1.3]{kubota15}).
Unfortunately, these lower bounds on variance give no information on the true size of fluctuations, hence the need for Theorems \ref{fpp_thm} and \ref{fpp_thm_1}.
Indeed, one cannot expect a matching upper bound since $\Var(T(0,y_n))$ should be of order $n^{2/3}$ in the standard cases.

The best known variance upper bound is $Cn/\log n$, proved in general dimensions for progressively more general distributions by Benjamini, Kalai, and Schramm \cite{benjamini-kalai-schramm03}, Bena\"\i m and Rossignol \cite{benaim-rossignol08}, and Damron, Hanson, and Sosoe \cite{damron-hanson-sosoe15,damron-hanson-sosoe14}.
One notable exception to the $n/\log n$ barrier comes from a simplified FPP model introduced by Sepp{\"a}l{\"a}inen \cite{seppalainen98}, for which Johansson \cite[Theorem 5.3]{johansson01} proves that the passage time fluctuations, when rescaled by a suitable factor of $n^{1/3}$, converge to the GUE Tracy--Widom distribution \cite{tracy-widom94}. 

Interestingly, in the critical case $\P(X_e=0)=1/2$ with $\P(0<X_e<\eps) = 0$, fluctuations are of order exactly $\sqrt{\log n}$.
Kesten and Zhang \cite{kesten-zhang97} prove a central limit theorem on this scale, and in the binary case $\P(X_e=1) = 1/2$, Chayes, Chayes, and Durrett \cite[Theorem 3.3]{chayes-chayes-durrett86} establish the expected asymptotic $\E(T(0,n\mathrm{e}_1)) = \Theta(\log n)$.
More delicate critical cases are examined in \cite{zhang99,damron-lam-wang17}.

Next we turn our attention to the related shape fluctuations.
For $x\in\R^2$, let $[x]$ be the unique element of $\Z^2$ such that $x\in[x]+[0,1)^2$.
For each $t>0$, define
\eeq{ \label{limit_shape}
B(t) \coloneqq \{x \in \R^2 : T(0,[x]) \leq t\},
}
which encodes the set of points reachable by a path of length at most $t$.
Sharpened from a result of Richardson \cite{richardson73}, the Cox--Durrett shape theorem  \cite[Theorem 3]{cox-durrett81} says that if (and only if) $\P(X_e = 0) < p_\cc(\Z^2)$ and \eqref{fpp_assumption_2b} holds,
then there
exists a deterministic, convex, compact set $\BB\subset\R^2$, having the symmetries of $\Z^2$ and nonempty interior, such that for any $\eps > 0$, almost surely
\eq{
(1 - \eps)\BB \subset \frac{1}{t}B(t) \subset (1 + \eps)\BB \quad \text{for all large $t$}.
}
More specifically, for every $x\in\R^2$, there is a positive, finite constant $\nu(x)$ such that
\eeq{ \label{mu_def}
\lim_{n\to\infty} \frac{T(0,[nx])}{n} = \nu(x) \quad \mathrm{a.s.},
}
and
\eq{
\BB = \{x \in \R^2 : \nu(x) \leq 1\}.
}
Moreover, $\nu$ is a norm on $\R^2$, and so $\BB$ is the unit ball under this norm.

The question remains as to how far $B(t)$ typically is from $t \BB$.
One way to pose this problem precisely is to ask for the value of
\eeq{ \label{chi_prime_def}
\chi' \coloneqq \inf\Big\{\nu : \P\big((t-t^\nu)\BB \subset B(t) \subset (t+t^\nu)\BB \text{ for all large $t$}\big) = 1\Big\}.
}
Another possible quantity to consider is $\chi \coloneqq \sup_{\|x\|_2=1} \chi_x$, where
\eq{
\chi_x \coloneqq \sup\{\gamma\geq0 : \exists\, C>0,\Var T(0,[nx]) \geq Cn^{2\gamma} \text{ for all $n$}\}.
}
Although it is conjectured that $\chi_x = \chi = \chi' = \frac{1}{3}$, even relating $\chi$ and $\chi'$ is  challenging because a variance lower bound does not by itself guarantee anything about fluctuations.
Assuming $\E(X_e^2) < \infty$ and either \eqref{fpp_assumption_1} or \eqref{fpp_assumption_2a}, Newman and Piza \cite[Theorem 7]{newman-piza95} prove $\max(\chi,\chi') \geq 1/5$.
Furthermore, they show $\chi_x \geq 1/8$ if $x$ is a direction of curvature for $\BB$, a notion defined in \cite{newman-piza95} and recalled here.

\begin{defn} \label{curvature_def}
Let $x\in\R^2$ be a unit vector, and $z\in\partial \BB$ the boundary point of $\BB$ in the direction $x$.
We say $x$ is a direction of curvature for $\BB$ if there exists a Euclidean ball $\SS$ (with any center and positive radius) such that $\SS\supset\BB$ and $z\in\partial\SS$.
\end{defn}

Since $\BB$ must have at least one direction of curvature (e.g.~take a large ball $\SS$ containing $\BB$, and then translate $\SS$ until it first intersects $\partial\BB$), one has $\chi \geq 1/8$ in the setting of \cite{newman-piza95}.
Unfortunately, this result does not imply order $n^{1/8}$ fluctuations without a matching upper bound on the variance.

The first work addressing typical shape fluctuations is due to Zhang \cite{zhang06}, who shows they are at least of order $\sqrt{\log n}$ in a certain sense for Bernoulli weights and general dimension.
Nakajima \cite{nakajima20} extends this result to general distributions.
In the first result proving $\chi' > 0$, Chatterjee \cite[Theorem 2.8]{chatterjee19II} shows that if for some direction of curvature $x$, $T(0,[nx])$ has fluctuations of order $n^{1/8-\delta}$ for any $\delta > 0$ in the sense of Definition \ref{fluctuations_def}, then $\chi' \geq 1/8$.
It is then shown in \cite[Theorem 2.7]{chatterjee19II} that the hypothesis of the previous sentence is true if the weight distribution satisfies \eqref{previous_assumption}.
Here we are able to replace that assumption with a small moment condition needed to use Alexander's shape theorem \cite{alexander97}, as refined by Damron and Kubota \cite{damron-kubota16}.

\begin{thm} \label{fpp_thm_2}
Assume $\P(X_e=0) < p_\cc(\Z^2)$ and $\E(X_e^\lambda)<\infty$ for some $\lambda > 3/2$.
If $x$ is a direction of curvature for $\BB$, then $T(0,[nx])$ has fluctuations of order at least $n^{1/8-\delta}$ for any $\delta > 0$.
\end{thm}

By the argument of \cite[Theorem 2.8]{chatterjee19II}, we obtain the following lower bound on the shape fluctuation exponent.

\begin{cor}
Assume the setting of Theorem \ref{fpp_thm_2}.
Then the shape fluctuation exponent defined by \eqref{chi_prime_def} satisfies $\chi' \geq \frac{1}{8}$.
\end{cor}

\subsection{Corner growth model} \label{lpp_sec}

In its planar form, LPP is often called the \textit{corner growth model}.
It is similar to FPP, the main differences being that only directed paths are considered (i.e.~coordinates never decrease), and the passage time $T$ is defined by time-maximizing paths rather than minimizing ones.
Furthermore, by convention we place the weights on the vertices instead of the edges, but this difference is more technical than conceptual.
We will now make this setup precise.

Let $\Z^2_+$ denote the first quadrant of the square lattice, that is the set of all $v = (a,b) \in \Z^2$ with $a,b\geq 0$.
We will write the standard basis vectors as $\mathbf{e}_1 = (1,0)$ and $\mathbf{e}_2 = (0,1)$.
Let $(X_v)_{v\in\Z^2_+}$ be an i.i.d.~family of non-degenerate random variables; because of the directedness, no assumption of nonnegativity is needed.
A \textit{directed} path $\vec\gamma = (\gamma_0,\gamma_1,\dots,\gamma_n)$ is one in which each increment $\gamma_i - \gamma_{i-1}$ is equal to $\mathbf{e}_1$ or $\mathbf{e}_2$.
The \textit{passage time} of such a path is
\eq{
T(\vec\gamma) \coloneqq \sum_{i=1}^n X_{\gamma_i}.
}
Let $T(u,v)$ be the maximum passage time of a directed path from $u$ to $v$, called the \textit{(last) passage time},
\eq{
T(u,v) \coloneqq \sup\{T(\vec\gamma)\ |\ \gamma_0 = u, \gamma_n = v\}.
}
We will again refer to any path achieving this time as a \textit{(finite) geodesic}.
Once more $T$ satisfies a shape theorem under mild assumptions on $\LL_{X}$, which we will not discuss. 
For further background, the reader is directed to \cite{martin06,quastel-remenik14,rassoul18}.

The directed structure advantages this model because of correspondences with problems in queueing networks, interacting particle systems, combinatorics, and random matrices.
Remarkable progress has been made by leveraging these connections in specific cases, leading to rigorous proofs of order $n^{1/3}$ passage time fluctuations converging to Tracy--Widom distributions upon rescaling.
This has been successfully carried out by Johansson \cite{johansson00} when the $X_v$'s are geometrically or exponentially distributed, building on work of Baik, Deift, and Johansson \cite{baik-deift-johansson99} connected to a continuum version of LPP.
The results extend to point-to-line passage times \cite{borodin-ferrari-prahofer-sasamoto07}.
Purely probabilistic techniques for accessing fluctuation exponents appear in \cite{cator-groeneboom06, balazs-cator-seppalainen06}.
The fluctuation exponent of $1/3$ is also present in a model known as Brownian LPP, for which the connection to Tracy--Widom laws is more explicit \cite{oconnell03}.

Away from exactly solvable settings, Chatterjee \cite[Theorem 8.1]{chatterjee08} proves that when the vertex weights are Gaussian, the point-to-line passage time has variance at most $Cn/\log n$.
Graham \cite{graham12} extends this result to general dimensions, also discussing uniform and gamma distributions.
To our knowledge, no general lower bound on fluctuations has been written for LPP.
It is worth mentioning, however, that the results in \cite{newman-piza95} are also stated for \textit{directed} FPP.
It is natural to suspect that many of results mentioned for FPP could be naturally translated to the LPP setting.
Indeed, as we now discuss, Theorem \ref{fpp_thm} carries over with little modification.

Let $\vec{p}_{c,\,\mathrm{site}}(\Z^2)$ be the critical value of directed site percolation on $\Z^2$.
It is clear that $\vec{p}_{c,\,\mathrm{site}}(\Z^2)$ is at least as large as its undirected counterpart ${p}_{c,\,\mathrm{site}}(\Z^2)$, which in turn satisfies ${p}_{c,\,\mathrm{site}}(\Z^2) > p_\cc(\Z^2) = 1/2$ \cite{grimmett-stacey98}.
In the way of upper bounds, it is known from \cite{balister-bollobas-stacey94,liggett95} that $\vec{p}_{c,\,\mathrm{site}}(\Z^2) \leq 3/4$. 
Let $S \coloneqq \esssup X_v \in (-\infty,\infty]$.
The assumption analogous to \eqref{fpp_assumption_1} or \eqref{fpp_assumption_2a} is
\eeq{ \label{lpp_assumption}
\P(X_v = S) < \vec{p}_{c,\,\mathrm{site}}(\Z^2).
}

\begin{thm} \label{lpp_thm}
Assume \eqref{lpp_assumption}.
Let $v_n$ be any sequence in $\Z^2_+$ such that $\|v_n\|_1 \geq n$ for every $n$.
Then the fluctuations of $T(0,v_n)$ are at least of order $\sqrt{\log n}$.
\end{thm}

In the case $v_n = n\mathbf{e}_1$, the passage time $T(0,n\mathbf{e}_1)$ is just the sum of $n$ i.i.d.~random variables and thus fluctuates on the scale of $n^{1/2}$.
The $n^{1/3}$ scaling should manifest when the two coordinates of $v_n$ are both of order $n$.
Interpolating between these two regimes, it is expected that if $v_n = (n,\floor{n^{a}})$ for $a \in (0,1)$, then $T(0,v_n)$ has fluctuations of order $n^{1/2-a/6}$. 
Such a result is proved, along with rescaled convergence to the GUE Tracy--Widom distribution, for $a < 3/7$ \cite{baik-suidan05,bodineau-martin05}.

\setcounter{footnote}{21}
\subsection{Directed polymers in $1+1$ dimensions} \label{dp_sec}

The model of directed polymers in random environment is a positive-temperature version of LPP.
That is, instead of examining only maximal paths, we consider the softer model of defining a Gibbs measure on paths, with those of greater passage time receiving a higher probability.
This is the same model as in Chapters \ref{endpoint} and \ref{replica}, although here specialized to the $1+1$-dimensional case.\footnote{In this $1+1$-dimensional case, instead of defining the model on $\{0,1,\dots,\}\times \Z$, it is customary to ``rotate by 45 degrees" to obtain the equivalent model on $\Z^2_+$.}

With $\Z^2_+$ as before, we again take $(X_v)_{v\in\Z^2_+}$ to be an i.i.d.~family of non-degenerate random variables, called the \textit{random environment}.
Let $\vec\Gamma_n$ denote the set of directed paths $\vec\gamma = (v_0,v_1,\dots,v_n)$ of length $n$ starting at the origin $v_0 = 0$.
Given an inverse temperature $\beta > 0$, define a Gibbs measure $\mu_n^\beta$ on $\vec\Gamma_n$ by
\eq{
\mu_n^\beta(\vec\gamma) \coloneqq \frac{\e^{\beta H_n(\vec\gamma)}}{Z_n(\beta)}, \qquad H_n(\vec\gamma) \coloneqq \sum_{i=1}^n X_{v_i}, \quad \vec\gamma \in \vec\Gamma_n, 
}
where now the object of interest is the \textit{partition function},
\eq{
Z_n(\beta) \coloneqq \sum_{\vec\gamma\in\vec\Gamma_n} \e^{\beta H_n(\vec\gamma)}.
}
Since $Z_n(\beta)$ grows exponentially in $n$, the proper linear quantity to consider is the \textit{free energy}, $\log Z_n(\beta)$.
Strictly speaking, the following result is not the exact analog of Theorems \ref{fpp_thm} and \ref{lpp_thm}, since we have not fixed the endpoint.
Nevertheless, the same argument goes through for point-to-point free energies (see \eqref{p2p}).

\begin{thm} \label{dp_thm}
Assume \eqref{lpp_assumption}.
Then the fluctuations of $\log Z_n(\beta)$ are at least of order $\sqrt{\log n}$ for any $\beta > 0$.
\end{thm}

As in LPP, there are several exactly solvable models of $(1+1)$-dimensional directed polymers for which free energy fluctuations on the order of $n^{1/3}$ can be calculated, beginning with the inverse-gamma (or log-gamma) polymer introduced by Sepp{\"a}l{\"a}inen \cite{seppalainen12}.
There are now three other solvable models: the strict-weak polymer \cite{corwin-seppalainen-shen15,oconnell-ortmann15}, the Beta RWRE \cite{barraquand-corwin17}, and the inverse-beta polymer \cite{thiery-doussal15}.
Chaumont and Noack show in \cite{chaumont-noack18I} that these are the only possible models possessing a certain stationarity property, and in \cite{chaumont-noack18II} provide a unified approach to calculating their fluctuation exponents.
We also mention the positive temperature version of Brownian LPP, introduced by O'Connell and Yor \cite{oconnell-yor01}, for which order $n^{1/3}$ energy fluctuations have been established \cite{seppalainen-valko10,borodin-corwin14,borodin-corwin-ferrari14}.

For the general model considered here, the situation is much the same as for FPP.
In the way of upper bounds, Alexander and Zygouras \cite{alexander-zygouras13} prove exponential concentration of $\log Z_n(\beta) - \E \log Z_n(\beta)$ on the scale of $\sqrt{n/\log n}$, in analogy with works mentioned earlier \cite{benjamini-kalai-schramm03,benaim-rossignol08,damron-hanson-sosoe15,damron-hanson-sosoe14,chatterjee08,graham12}.
Their results hold in general dimensions and for a wide range of distributions.
As for lower bounds, Piza \cite{piza97} proves $\Var[\log Z_n(\beta)] \geq C\log n$ for non-positive weights with finite variance, as well as weaker versions of the shape theorem results from \cite{newman-piza95}.

Although Theorem \ref{dp_thm} does not even prove a positive fluctuation exponent, simply knowing that free energy fluctuations diverge may be significant in understanding the phenomenon of polymer localization.
One way of defining this phenomenon is to say the polymer measure is \textit{localized} if its endpoint distribution has atoms:
\eeq{ \label{localization}
\limsup_{n\to\infty} \max_{\|v\|_1=n} \mu_n^\beta(\gamma_n = v) > 0 \quad \mathrm{a.s.}
}
It is known \cite[Proposition 2.4]{comets-shiga-yoshida03} that \eqref{localization} occurs for any $\beta > 0$ in $1+1$ and $1+2$ dimensions, and for sufficiently large $\beta$ in higher dimensions, depending on the law of the $X_v$'s.
What is unclear, however, is whether the atoms or ``favorite endpoints" are typically close to one another or far apart.
From the solvable case \cite{seppalainen12}, there is evidence suggesting the former is true at least in $1+1$ dimensions \cite{comets-nguyen16}.
In general dimensions, the same was shown along random subsequences in Chapter~\ref{endpoint}. 
These subsequences also exist for polymers on trees, but in that setting, the favorite sites more frequently appear far apart \cite{barral-rhodes-vargas12}; this behavior is thus difficult to rule out in high-dimensional lattices.
It is interesting, then, that for both polymers on trees and for high-temperature lattice polymers in dimensions $1+3$ and higher, the fluctuations of $\log Z_n^\beta$ are order $1$.
On the lattice, this fact is easy to deduce from a martingale argument; see \cite[Chapter 5]{comets17}.
For the tree case, see \cite[Section 5]{derrida-spohn88}.

\section{Proofs of main results} \label{proofs}

The proofs follow a general strategy, which we outline below.
For clarity, we will break each proof into two parts:
\begin{itemize}[leftmargin=0.75in]
\item[{\textbf{Part 1.}}] Use the coupling \eqref{our_coupling} with large enough $m$ to show that in all relevant paths, there is a high frequency of weights where $X'$ is far away from $X$.
\item[{\textbf{Part 2.}}] Show the same is true when $X'$ is replaced by $\wt X$ defined by \eqref{new_X}, provided we make good choices for $\eps$.
This step uses Part 1, as well as the independence of $Y$ from $X$ and $X'$.
Conclude that the passage time (or free energy) has, with positive probability independent of $n$, changed by an amount of the desired order.
\end{itemize}

\subsection{Proof of Theorems \ref{fpp_thm} and \ref{fpp_thm_1}}

Recall the notation
\eq{
s = \essinf X_e.
}
Before proceeding with the main argument, we begin with a lemma meant to guarantee that geodesics contain many edges with weights far from $s$.
Preempting a technical concern, we note that with probability $1$, geodesics do exist between all pairs of points in $\Z^2$ without any assumptions on the distribution of $X_e$ \cite{wierman-reh78}.
We will use the notation $B_n(x) \coloneqq \{y\in\Z^2 : \|x-y\|_1\leq n$\} and $\partial B_n(x) \coloneqq \{y \in \Z^2 : \|x-y\|_1 = n\}$ for $n\geq1$.

\begin{lemma} \label{fpp_lemma}
Given $\delta>0$ and $\rho \in (0,1)$, let $E_n^x$ be the event that there exists a geodesic $\gamma = (\gamma_0,\gamma_1,\dots,\gamma_N)$ from $x\in\Z^2$ to some $y\in \partial B_n(x)$ such that 
\eeq{ \label{bad_proportion}
\#\{1 \leq i \leq N : X_{(\gamma_{i-1},\gamma_i)} \geq s + 2\delta\} < \rho n.
}
If \eqref{fpp_assumption_1} or \eqref{fpp_assumption_2} holds, then there are $\delta$ and $\rho$ sufficiently small that
\eeq{ \label{sum_n_bd_1}
\sum_{n = 1}^\infty \P(E_n^0) < \infty.
}
Furthermore, for some sequence $(n_k)_{k=1}^\infty$ satisfying $2^{k-1}< n_k\leq 2^{k}$,
\eeq{ \label{sum_n_bd_2}
\sum_{k=1}^\infty\sum_{\|x\|_1 = n_k} \P(E_{n_k}^x) < \infty.
}
\end{lemma}

\begin{remark}
As will be seen in the proof, the restriction of Lemma \ref{fpp_lemma} to geodesics is only necessary when assuming \eqref{fpp_assumption_2} without \eqref{fpp_assumption_1}.
\end{remark}

We will need two results from the literature.
The first theorem below was originally established by van den Berg and Kesten \cite{vandenberg-kesten93} when $y = (1,0)$, and later generalized by Marchand \cite{marchand02}.

\begin{thm}[{Marchand \cite[Theorem 1.5(ii)]{marchand02}}] \label{marchand_thm}
Let $(X_e)_{e\in E(\Z^2)}$ and $(\hat X_e)_{e\in E(\Z^2)}$ be two i.i.d.\linebreak families of nonnegative random variables, such that $\hat X_e$ stochastically dominates $X_e$.
Let $\nu$ and $\hat \nu$ be the respective limiting norms, given by \eqref{mu_def}.
If $\P(X_e=s) < \vec{p}_\cc(\Z^2)$, then $\nu(y) < \hat \nu(y)$ for all $y \neq 0$.
\end{thm}

The next theorem demonstrates why \eqref{fpp_assumption_2b} is necessary when \eqref{fpp_assumption_1} is not assumed.
The version stated in \cite{ahlberg15} uses $\|\cdot\|_2$ in place of $\|\cdot\|_1$, but this makes no difference because all norms on $\R^2$ are equivalent.

\begin{thm}[{Ahlberg \cite[Theorem 1]{ahlberg15}}] \label{ahlberg_thm}
For every $\alpha, \eps>0$,
\eq{
\E\min(X^{(1)},X^{(2)},X^{(3)},X^{(4)})^\alpha<\infty \quad \iff \quad
\sum_{y\in\Z^2}\|y\|_1^{\alpha-2}\P(|T(0,y) - \nu(y)| > \eps\|y\|_1) < \infty,
}
where the $X^{(i)}$'s are independent copies of $X_e$.

\end{thm}

\begin{proof}[Proof of Lemma \ref{fpp_lemma}]
We handle the cases of \eqref{fpp_assumption_1} and \eqref{fpp_assumption_2} separately.

\textbf{Case 1: Assuming \eqref{fpp_assumption_1}.}
Choose $\delta > 0$ small enough that $\P(X_e < s+ 2\delta) < p_\cc(\Z^2) = 1/2$.
Consider the first-passage percolation when each $X_e$ is replaced by
\eq{
\hat X_e \coloneqq 
\begin{cases}
0 &\text{if $X_{e} < s+2\delta$,} \\
1 &\text{ otherwise.}
\end{cases}
}
Let $\hat T$ be the associated passage time, so that $\hat T(x,y)$ is simply the minimum number of edges $e$ satisfying $X_e \geq s+2\delta$ in a path from $x$ to $y$.
By \cite[Theorem 1]{kesten80}, there exists $\rho$ small enough that with probability tending to $1$ exponentially quickly in $n$, every self-avoiding path $\gamma$ starting at the origin that has length at least $n$ --- not just those terminating at $\partial B_n(0)$ --- has $\hat T(\gamma) \geq \rho n$.
That is, $\P(E_n^0) \leq a\e^{-bn}$ for some $a,b>0$, which easily gives
\eq{
\sum_{n=1}^\infty n\P(E_n^0) < \infty.
}
In particular, \eqref{sum_n_bd_1} is true, and  \eqref{sum_n_bd_2} holds for any increasing sequence $n_k\to\infty$, since $|\partial B_n(0)| = 4n$ for every $n\geq1$.

\textbf{Case 2: Assuming \eqref{fpp_assumption_2}.}
Recall that \eqref{fpp_assumption_2b} implies the existence of the finite limit \eqref{mu_def} for every $x\in\R^2$.
By \eqref{fpp_assumption_2a}, we can choose $\delta > 0$ small enough that $\P(X_e < s+2\delta) < \vec{p}_\cc(\Z^2)$.
Next we choose $M$ large enough that $\P(s+2\delta \leq X_e < s+2\delta+M) \geq 1/4$, which is possible because of \eqref{oriented_bond_upper}.
Consider the first-passage percolation model where each $X_e$ is replaced by
\eq{
\hat X_e \coloneqq \begin{cases}
s+2\delta + M &\text{if } s+2\delta\leq X_e< s+2\delta+M, \\
X_e &\text{otherwise.}
\end{cases}
}
Let $\hat T$ and $\hat \nu$ be the associated passage time and limiting norm.
We also define
\eq{
\nu_\mathrm{min} &\coloneqq \min\{\nu(y) : y\in\R^2, \|y\|_1=1\},
}
which is positive because $s>0$, 
and finite because of \eqref{fpp_assumption_2b}.
Because of our choice of $\delta$ and $M$, Theorem \ref{marchand_thm}
guarantees $\nu(y) < \hat \nu(y)$ for every nonzero $y \in \R^2$.
By compactness and continuity of $\nu$ and $\hat\nu$,
there is $\eps_1 > 0$ such that $\nu(y)(1 + \eps_1) < \hat \nu(y)(1-2\eps_1)$ for every $y$ with $\|y\|_1 = 1$.
By scaling, the same inequality holds for all $y\neq 0$.
Therefore, if we set $\eps_2 \coloneqq \eps_1\min(\nu_\mathrm{min},1)$,
then for all $y\in \partial B_n(0)$,
\eeq{ \label{scaling_ineq}
\nu(y)(1+\eps_1) + \eps_2 n
&< \hat\nu(y)(1-2\eps_1) + \eps_2 \|y\|_1 \\
&\leq \hat\nu(y)(1-2\eps_1) + \eps_1\nu(y) \\
&< \hat\nu(y)(1-2\eps_1) + \eps_1\hat \nu(y)
= \hat\nu(y)(1-\eps_1).
}
Finally, choose $\rho \in (0,1)$ such that $\rho M < \eps_2$.

Now consider any $y\in \partial B_n(0)$.
If there exists a geodesic $\gamma$ (with respect to $T$) from $0$ to $y$ such that \eqref{bad_proportion} holds, then $\gamma$ contains fewer than $\rho n$ edges $e$ such that $\hat X_e \neq X_e$.
Moreover, for each such edge, we have $\hat X_e\leq X_e + M$.
Therefore,
\eq{
\hat T(0,y) \leq \sum_{i=1}^N \hat X_{(\gamma_{i-1},\gamma_i)} \leq T(0,y) + n\rho M < T(0,y)+\eps_2 n.
}
But in light of \eqref{scaling_ineq},
\eq{
\{T(0,y) \leq \nu(y)(1+\eps_1)\}\cap\{\hat\nu(y)(1-\eps_1) \leq \hat T(0,y)\} \subset \{T(0,y) + \eps_2n \leq \hat T(0,y)\}.
}
From these observations, we see
\eeq{ \label{complement_containment}
E_n^0 \subset
\bigcup_{\|y\|_1=n} \{T(0,y) > \nu(y)(1+\eps_1)\}\cup\{\hat T(0,y) < \hat\nu(y)(1-\eps_1)\},
}
and hence 
\eq{
\P(E_n^0) &\leq \sum_{\|y\|_1=n} \big[\P\big(T(0,y) - \nu(y) > \eps_1\nu(y)\big) + \P\big(\hat T(0,y) - \hat\nu(y) < -\eps_1\hat\nu(y)\big)\big] \\
&\leq \sum_{\|y\|_1=n} \big[\P\big(T(0,y) - \nu(y) > \eps_2\|y\|_1\big) + \P\big(\hat T(0,y) - \hat\nu(y) < -\eps_2\|y\|_1\big)\big].
}
By Theorem \ref{ahlberg_thm} with $\alpha=2$, 
\eqref{fpp_assumption_2b} gives
\eq{
\sum_{y\in\Z^2} \big[\P\big(T(0,y) - \nu(y) > \eps_2\|y\|_1\big) + \P\big(\hat T(0,y) - \hat\nu(y) < -\eps_2\|y\|_1\big)\big] < \infty.
}
Now \eqref{sum_n_bd_1} follows from the previous two displays.
To conclude \eqref{sum_n_bd_2}, we take
\eq{
n_k \coloneqq \argmin_{2^{k-1}<n\leq2^k} \P(E_n^0).
}
Note that by translation invariance, $\P(E_n^x) = \P(E_n^0)$ for all $x\in\Z^2$.
Again using the fact that $|\partial B_n(0)| = 4n$ for all $n\geq1$, we have 
\eq{
\sum_{k=1}^\infty\sum_{\|x\|_1 = n_k} \P(E_{n_k}^x) = 4\sum_{k=1}^\infty n_k\P(E_{n_k}^0)
&\leq 8\sum_{k=1}^\infty2^{k-1}\P(E_{n_k}^0) \\
&\leq 8\sum_{k=1}^\infty \sum_{n=2^{k-1}+1}^{2^{k}} \P(E_n^0)
= 8\sum_{n=2}^\infty \P(E_n^0) < \infty.
}
\end{proof}

\begin{proof}[Proof of Theorems \ref{fpp_thm} and \ref{fpp_thm_1}] \hspace{1in}\\[-0.5\baselineskip]

\textbf{Part 1.} Let $T_n = T(0,y_n)$.
From Lemma \ref{fpp_lemma}, take $\delta > 0$, $\rho \in (0,1)$, and $(n_k)_{k=1}^\infty$ satisfying $2^{k-1}<n_k\leq2^k$, such that \eqref{sum_n_bd_2} holds.
Then choose $k_0$ large enough that
\eq{
\sum_{k=k_0}^{\infty} \sum_{\|x\|_1=n_k} \P(E_{n_k}^x) \leq \frac{1}{7}.
}
Define the event 
\eeq{ \label{G_0_def}
G_0 \coloneqq \bigcap_{k=k_0}^{\infty} \bigcap_{\|x\|_1=n_k} (E^x_{n_k})^\mathrm{c},
}
so that
\eeq{
\P(G_0) \geq \frac{6}{7}. \label{G_bd}
}
Finally, choose $m$ large enough that if $X_e^{(1)},\dots,X_e^{(m)}$ are independent copies of $X_e$, then 
\eeq{
\P(\min(X_e^{(1)},\dots,X_e^{(m)})\leq s+\delta) > 1 - \Big(\frac{1}{3}\Big)^{1/\rho}. \label{m_choice}
}
Throughout the rest of the proof, $C$ will denote a constant that may depend on $m$ and $\LL_{X}$, but nothing else.
Its value may change from line to line or within the same line.
To condense notation, we will also define 
\eeq{ \label{prime_definitions}
X_e' \coloneqq \min(X_e,X_e^{(1)},\dots,X_e^{(m)}), \qquad Z_e \coloneqq X_e - X_e', \qquad W_e \coloneqq 1 - \e^{-Z_e},
}
where $(X_e^{(j)})_{e\in E(\Z^2)}$, $1\leq j\leq m$, are independent copies of the i.i.d.~edge weights.

Given any realization of the percolation, the subgraph of $\Z^2$ induced by the geodesics between all pairs of points in $B_{2n}(0)$ is finite and connected.
Therefore, we can choose one of its spanning trees according to some arbitrary, deterministic rule.
From that tree we have a distinguished geodesic for each $x,y\in B_{2n}(0)$.
Moreover, if $x'$ and $y'$ lie along the geodesic from $x$ to $y$, then the distinguished geodesic from $x'$ to $y'$ is the relevant subpath. 

Given $c>0$ to be chosen later, consider the event $F_n$ that there exist $x\in  \partial B_n(0)$ and $y\in\partial B_{2n}(0)$ whose distinguished geodesic --- which we denote by its edges $(e_1,\dots,e_N)$ in a slight abuse of notation --- satisfies
\eeq{ \label{defining_F}
\sum_{i=1}^N W_{e_i} \leq cn.
}
For a given $x\in B_n(0)$, if $E_n^x$ does not occur, then any geodesic from $x$ to any $y\in\partial B_n(x)$ contains at least $\rho n$ edges satisfying $X_{e_i} \geq s + 2\delta$.
Furthermore, because $\|x\|_1=n$, every geodesic from $x$ to $\partial B_{2n}(0)$ must pass through $\partial B_n(x)$.
Therefore, if $E_n^x$ does not occur, then any geodesic from $x$ to $\partial B_{2n}(0)$ contains at least $\rho n$ edges satisfying $X_{e_i} \geq s + 2\delta$.

It will be convenient to define
\eq{
U_{e} \coloneqq \begin{cases}
1 &\text{if } \min(X_e^{(1)},\dots,X_e^{(m)}) \leq s + \delta, \\
0 &\text{otherwise.}
\end{cases}
}
The reason for doing so is that now the $U_e$'s are mutually independent and independent of $\sigma(X)$, the $\sigma$-algebra generated by the $X_e$'s.
In addition, if $(e_1,\dots,e_N)$ is the distinguished geodesic between some fixed $x\in\partial B_n(0)$ and $y\in\partial B_{2n}(0)$, then from the observation
\eq{
X_{e}\geq s+2\delta,\ U_{e} = 1 \quad \implies \quad Z_{e} \geq \delta \quad \implies \quad W_{e} \geq 1-\e^{-\delta},
}
we see
\eq{
\sum_{i=1}^N \one_{\{X_{e_i}\geq s+2\delta\}}\one_{\{U_{e_i}=1\}}(1-\e^{-\delta}) \leq \sum_{i=1}^N W_{e_i}.
}
By the discussion of the previous paragraph, if $E_n^x$ does not occur, then there is a subsequence $1\leq i_1<i_2<\cdots<i_{\ceil{\rho n}}\leq N$ such that $X_{e_{i_\ell}}\geq s+2\delta$ for each $\ell=1,\dots,\ceil{\rho n}$.
With this notation, we have
\eq{
\P\givenp[\bigg]{\sum_{i=1}^N W_{e_i} \leq cn}{\sigma(X)}\one_{(E_n^x)^\mathrm{c}}
&\leq \P\givenp[\bigg]{\sum_{i=1}^N \one_{\{X_{e_i}\geq s+2\delta\}}\one_{\{U_{e_i}=1\}}(1-\e^{-\delta})\leq cn}{\sigma(X)}\one_{(E_n^x)^\mathrm{c}} \\
&\leq \P\givenp[\bigg]{\sum_{\ell=1}^{\ceil{\rho n}} \one_{\{U_{e_{i_\ell}}=1\}}(1-\e^{-\delta})\leq cn}{\sigma(X)}\one_{(E_n^x)^\mathrm{c}} \\
&\leq \phi(t)^{\ceil{\rho n}}\exp\Big\{\frac{cnt}{1-\e^{-\delta}}\Big\} \quad \text{for any $t > 0$,}
}
where
\eq{
\phi(t) \coloneqq \E(\e^{-t U_e}) = \P(U_e=0) + \e^{-t}\P(U_e=1) 
\stackrel{\mbox{\footnotesize\eqref{m_choice}}}{<} \Big(\frac{1}{3}\Big)^{1/\rho} +  \e^{-t}\Big(1-\frac{1}{3^{1/\rho}}\Big). 
}
We can choose $t$ sufficiently large that $\phi(t)^\rho \leq 1/3$.
Then setting $c = (1-\e^{-\delta})t^{-1}$, we have 
\eq{
\P\givenp[\bigg]{\sum_{i=1}^N W_{e_i} \leq cn}{\sigma(X)}\one_{(E_n^x)^\mathrm{c}} \leq \frac{\e^n}{3^n}.
}
We now use this estimate to bound the conditional probability of the event $F_n$ defined via \eqref{defining_F}.
Since $|\partial B_n(0)| = 4n$ and $|\partial B_{2n}(0)| = 8n$, a union bound gives
\eeq{ \label{conditional_F_bd}
\P\givenp{F_n}{\sigma(X)}\one_{\{\bigcap_{\|x\|_1= n}(E_n^x)^\mathrm{c}\}} \leq \frac{32n^2\e^n}{3^n} \quad \text{for all $n\geq1$.}
}
Now we choose an even integer $k_1\geq k_0$ sufficiently large that
\eeq{ \label{sum_with_n_k}
32\sum_{k=k_1}^\infty \frac{n_k^2\e^{n_k}}{3^{n_k}} \leq \frac{1}{8},
}
and define the event
\eeq{ \label{Gdef}
G \coloneqq \bigcap_{k=k_1}^{\infty} F_{n_k}^\mathrm{c}.
} 
Recall the event $G_0 \in \sigma(X)$ defined in \eqref{G_0_def}.
The above discussion yields
\eq{ 
\P\givenp{G}{\sigma(X)}\one_{G_0} 
&\stackrel{\phantom{{\mbox{\footnotesize{\eqref{conditional_F_bd}}}}}}{\geq} \bigg(1-\sum_{k=k_1}^{\infty} \P\givenp{F_{n_k}}{\sigma(X)} \bigg)\one_{G_0} \\
&\stackrel{\phantom{{\mbox{\footnotesize{\eqref{conditional_F_bd}}}}}}{=} \bigg(1-\sum_{k=k_1}^{\infty} \P\givenp{F_{n_k}}{\sigma(X)} \bigg)\prod_{k=k_0}\one_{\{\bigcap_{\|x\|_1=n_k} (E_{n_k}^x)^\mathrm{c}\}} \\
&\stackrel{{\mbox{\footnotesize{\eqref{conditional_F_bd}}}}}{\geq} \bigg(1-32\sum_{k=k_1}^\infty \frac{n_k^{2}\e^{n_k}}{3^{n_k}}\bigg)\prod_{k=k_0}\one_{\{\bigcap_{\|x\|_1=n_k} (E_{n_k}^x)^\mathrm{c}\}} 
\stackrel{{\mbox{\footnotesize{\eqref{sum_with_n_k}}}}}{\geq} \frac{7}{8}\one_{G_0}.
}
It now follows from \eqref{G_bd} that
\eeq{ \label{Gprime_bd}
\P(G) \geq \frac{7}{8}\P(G_0) \geq \frac{3}{4}.
}
Having chosen $k_1$, we will assume $n$ satisfies
\eeq{ \label{n_vs_k}
\floor{(\log_2 n)/2} \geq k_1+1.
}

\textbf{Part 2.} For each edge $e$, let $\|e\|$ denote its distance from the origin, i.e.~the graph distance from $0$ to the closest endpoint of $e$.
For each $e$ with $\|e\| \leq n$, set
\eeq{ \label{edge_eps_def}
\eps_e \coloneqq \frac{\alpha}{(\|e\|+1)\sqrt{\log n}},
}
where $\alpha$ is a constant to be chosen below.
For each such $e$, define $\wt X_e$ as in \eqref{new_X} with $\eps = \eps_e$ and
$X_e'$ given in \eqref{prime_definitions}.
Let $\wt T_n = \wt T(0,y_n)$ be the passage time if $X_e$ is replaced by $\wt X_e$ whenever $\|e\|\leq n$.
Because there are at most $C(i+1)$ edges $e$ with $\|e\| = i$, we have
\eq{
\sum_{\|e\| \leq n} \eps_e^2 = \frac{\alpha^2}{\log n}\sum_{i=1}^n\sum_{\|e\|=i} \frac{1}{(i+1)^2}
\leq \frac{\alpha^2}{\log n} \sum_{i=1}^n \frac{C}{i+1} \leq C\alpha^2.
}
Hence, by Lemma \ref{tv_lemma},
\eq{
d_{\tv}(\LL_{T_n},\LL_{\wt T_n}) \leq C\alpha.
}
Choose $\alpha$ so that
\eeq{
d_{\tv}(\LL_{T_n},\LL_{\wt T_n}) \leq \frac{1}{4}.\label{FPP_final_1}
}
Now we aim to show that with sufficiently large probability,
$T_n-\wt T_n$ is of order $\sqrt{\log n}$.
Let $e_1,\dots,e_N$ be a geodesic from $0$ to $y_n$, chosen according to same deterministic rule as before. 
Note that necessarily $N\geq \|y\|_1 \geq n$.
We will use the notation $e_i = (x_{i-1},x_i)$ to denote endpoints of $e_i$ in the order traversed by the geodesic.
For each $k = 1,\dots,\floor{(\log_2 n)/2}$, let $i_k$ be the first index such that $\|x_{i_k}\|_1 = n_{2k}$, where the $n_k$'s were chosen in Part 1 and satisfy $2^{k-1}<n_k\leq2^k$.
Observe that
\eeq{ \label{edge_distance_bd}
\|e_i\| \leq n_{2k}-1\leq 4^k-1 \quad \text{for every $i\leq i_{k}$}.
}
Furthermore, $(e_{{i_k}+1},\dots,e_{i_{k+1}})$ is a geodesic from $x_{i_k} \in B_{n_{2k}}(0)$ to $x_{i_{k+1}} \in B_{n_{2k+2}}(0)$, where $n_{2k+2} > 2^{2k+1} \geq 2n_k$.
Therefore, on the event $G$ defined in \eqref{Gdef},
\eq{
\sum_{i=i_k+1}^{i_{k+1}} W_{e_i} 
\geq cn_{2k} > c2^{2k-1} \quad \text{for all } k = k_1/2,\dots,\floor{(\log_2 n)/2}-1.
}
which implies
\eeq{ \label{G_consequence}
\sum_{i=1}^N \eps_{e_i}W_{e_i} 
\geq \sum_{i=i_{k_1/2}+1}^N \eps_{e_i}W_{e_i} 
&\stackrel{\phantom{\mbox{\footnotesize\eqref{edge_distance_bd}}}}{\geq} \sum_{k=k_1/2}^{\floor{(\log_2 n)/2}-1} \sum_{i=i_k+1}^{i_{k+1}} \eps_{e_i}W_{e_i} \\
&\stackrel{\mbox{\footnotesize\eqref{edge_distance_bd}}}{\geq} \sum_{k=k_1/2}^{\floor{(\log_2 n)/2}-1} \frac{\alpha}{4^{k+1}\sqrt{\log n}} \sum_{i=i_k+1}^{i_{k+1}} W_{e_i} \\
&\stackrel{\phantom{\mbox{\footnotesize\eqref{edge_distance_bd}}}}{\geq} \frac{\alpha}{\sqrt{\log n}}\sum_{k=k_1/2}^{\floor{(\log_2 n)/2}-1}\frac{c2^{2k-1}}{4^{k+1}} \\
&\stackrel{\phantom{\mbox{\footnotesize\eqref{edge_distance_bd}}}}{=} \frac{\alpha c(\floor{(\log_2 n)/2}-k_1/2)}{8\sqrt{\log n}}\\
&\stackrel{\mbox{\footnotesize\eqref{n_vs_k}}}{\geq} \frac{\alpha c\sqrt{\log n}}{16\log 2} \eqqcolon \theta\sqrt{\log n}. \raisetag{5\baselineskip}
}
Denote by $\sigma(X,X^{(1)},\dots,X^{(m)})$ the $\sigma$-algebra generated by the $X_e$'s and $X_e^{(j)}$'s, $1\leq j\leq m$.
Recall that each $\wt X_{e_i}$ is equal to $\min(X_{e_i},X^{(1)}_{e_i},\dots,X^{(m)}_{e_i}) = X_{e_i}-Z_{e_i}$ independently with probability $\eps_{e_i}$, and equal to $X_{e_i}$ otherwise.
In the former case, the value of $\wt T_n$ is lowered relative to $T_n$ by at least $Z_{e_i}$; in the latter case, no change occurs.
Therefore,
\eq{
T_n - \wt T_n\geq \sum_{i=1}^N \one_{\{Y_{e_i}=1\}}Z_{e_i} \eqqcolon D,
}
where the $Y_{e_i}$'s are Bernoulli($\eps_{e_i}$) variables independent of each other and of $X,X^{(1)},\dots, X^{(m)}$.
It follows that for any $t\geq0$,
\eq{
\P\givenp[\big]{D\leq t}{\sigma(X,X^{(1)},\dots,X^{(m)})}
&\leq \e^{t}\E\givenp[\big]{\e^{-D}}{\sigma(X,X^{(1)},\dots,X^{(m)})} \\
&= \e^{t}\prod_{i=1}^N (1-\eps_{e_i} + \eps_{e_i}\e^{-Z_{e_i}}) \\
&\leq \e^{t}\prod_{i=1}^N \exp\{-\eps_{e_i}(1-\e^{-Z_{e_i}})\} 
= \e^{t}\exp\bigg\{-\sum_{i=1}^N \eps_{e_i}W_{e_i}\bigg\}.
}
Therefore, on the event $G$, \eqref{G_consequence} shows
\eq{
\P\givenp[\Big]{D\leq\frac{\theta}{2}\sqrt{\log n}}{\sigma(X,X^{(1)},\dots,X^{(m)})}\one_G \leq \e^{-\frac{\theta}{2}\sqrt{\log n}}.
}
Assuming $n$ is large enough that
\eeq{ \label{n_easy}
\e^{-\frac{\theta}{2}\sqrt{\log n}} \leq \frac{1}{2},
}
we have
\eeq{ \label{single_path}
\P\givenp[\Big]{D>\frac{\theta}{2}\sqrt{\log n}}{\sigma(X,X^{(1)},\dots,X^{(m)})} \geq \frac{1}{2}\one_G,
}
and thus
\eeq{ \label{FPP_final_2}
\P\Big(T_n - \wt T_n > \frac{\theta}{2}\sqrt{\log n}\Big) \geq \P\Big(D>\frac{\theta}{2}\sqrt{\log n}\Big)\geq \frac{1}{2}\P(G)\stackrel{\mbox{\footnotesize{\eqref{Gprime_bd}}}}{\geq}\frac{3}{8}.
}
Using \eqref{FPP_final_1} and \eqref{FPP_final_2} in Lemma \ref{fluctuation_lemma}, we see that $T_n$ has fluctuations of order at least $\sqrt{\log n}$.
\end{proof}

\subsection{Proof of Theorem \ref{fpp_thm_2}}
Recall Definition \ref{curvature_def} for a direction of curvature, as well as the exponent $\chi'$ from \eqref{chi_prime_def}.

\textbf{Part 1.}
Fix any unit vector $x$ that is a direction of curvature for $\BB$, and fix any $\delta > 0$.
We will write $T_n = T(0,[nx])$, where $[y]$ denotes the unique element of $\Z^2$ such that $y\in[y]+[0,1)^d$.
Let $L$ be the line passing through $0$ and $x$, and let $\Lambda_n$ be the cylinder of width $n^{3/4+\delta}$ centered about~$L$:
\eq{
\Lambda_n \coloneqq \{z \in \Z^2: d(z,L) \leq n^{3/4+\delta}\},
}
where $d(z,L) = \inf\{\|z-y\|_2 : y\in L\}$.
Under the given assumptions, \cite[Theorem 1.2]{damron-kubota16} guarantees $\chi' \leq 1/2$.
It then follows from \cite[Theorem 6 and (2.21)]{newman-piza95} that there exists $q_0 \in (0,1]$ such that with probability at least $q_0$, the following event, which we call $G_1$, is true:
For all large $n$, all geodesics from the origin to $[nx]$ lie entirely inside~$\Lambda_n$.

We would like to replace $\Lambda_n$ with a finite set.
To do so, we let $L_n$ be the line segment connecting $0$ and $nx$, and then introduce
\eq{
V_n \coloneqq \{z \in \Z^2 : d(z,L_n) \leq n^{3/4+2\delta}\}.
}
Suppose toward a contradiction that $G_1$ occurs but there exists a geodesic from $0$ to $[nx]$ that remains inside $\Lambda_n$ but not $V_n$.
Observe that from any $z \in \Lambda_n \setminus V_n$, the closest point on $L_n$ is either $0$ or $[nx]$.
Consequently, it follows from our supposition that from one of the endpoints of $L_n$ (say $0$, for concreteness), there are points $z_1$ within distance $n^{3/4+\delta}$ and $z_2$ at distance at least $n^{3/4+2\delta}$, such that $T(0,z_1) \geq T(0,z_2)$; see Figure \ref{fpp_fig}.
By the shape theorem, this inequality can only happen for finitely many $n$.
From this argument we conclude that with probability at least $q_0$, the following event, which we call $G_2$, is true:
For all large $n$, all geodesics from the origin to $[nx]$ lie entirely inside $V_n$.

\begin{figure}[h!]
\centering
\includegraphics[width=0.8\textwidth]{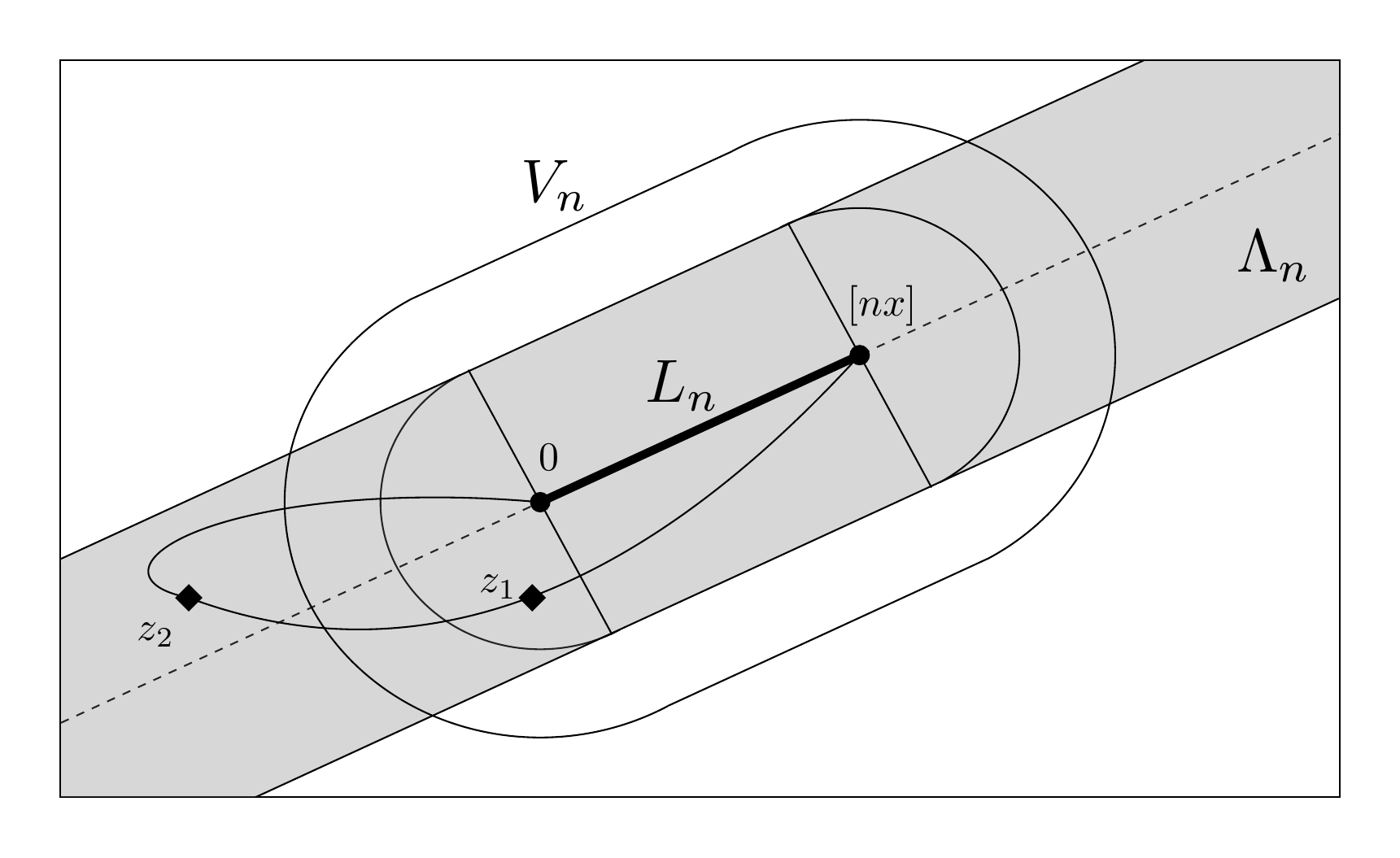}
\caption[Graphic accompanying the proof of Theorem \ref{fpp_thm_2}]{The geodesic connecting $0$ and $[nx]$ remains inside $\Lambda_n$ but exits and re-enters $V_n$.
The point $z_2$ is outside $V_n$ but has a shorter passage time to $0$ than does $z_1$, which is within distance $n^{3/4+\delta}$ of $0$.} \label{fpp_fig}
\end{figure}

Note that \eqref{fpp_assumption_2b} is implied by $\E(X_e^{1/2}) < \infty$ and thus also by $\E(X_e^\lambda)<\infty$ for $\lambda>3/2$.
From Lemma \ref{fpp_lemma} we can find $\delta > 0$ and $\rho \in (0,1)$ such that \eqref{sum_n_bd_1} holds.
As in the proof of Theorem \ref{fpp_thm}, for each edge $e \in  E(\Z^2)$ we define 
\eq{
X_e' \coloneqq \min(X_e,X_e^{(1)},\dots,X_e^{(m)}), \qquad Z_e \coloneqq X_e-X_e', \qquad W_e \coloneqq 1-\e^{-Z_e}.
}
When considering geodesics between $0$ and $[nx]$, we always choose a distinguished geodesic \linebreak $(e_1,\dots,e_N)$ according some deterministic rule.
As in the proof of Theorem~\ref{fpp_thm}, we take $m$ large enough and $c>0$ small enough that 
\eq{
\P\givenp[\bigg]{\sum_{i=1}^N W_{e_i} \leq cn}{\sigma(X)}\one_{(E_n^0)^\mathrm{c}} \leq \frac{\e^n}{3^n} \quad \text{for all $n\geq1$.}
}
Let $F_n$ be the event that $\sum_{i=1}^N W_{e_i} \leq cn$ (here we have fixed the endpoints, and so this event is different from the $F_n$ considered in the proof of Theorem \ref{fpp_thm}). 
By the above display and \eqref{sum_n_bd_1}, there is $n_0$ such that
\eeq{ \label{n_assumption_1}
\P(F_n) \leq \frac{q_0}{2} \quad \text{for all $n\geq n_0$}.
}

\textbf{Part 2.}
Now we set
\eq{
\eps \coloneqq \alpha n^{-7/8-\delta},
}
where $\alpha$ will be chosen below, and define the perturbed edge weights as in \eqref{new_X}:
For each edge $e$ with both endpoints in $V_n$, we let
\eq{
\wt X_e = \begin{cases}
X_e' &\text{if } Y_e = 1, \\
X_e &\text{if } Y_e = 0,
\end{cases}
\quad \text{where} \quad
Y_e \stackrel{\text{i.i.d.}}{\sim} \mathrm{Bernoulli}(\eps).
}
Denote by $\wt T_n$ be the passage time from $0$ to $[nx]$ if $X_e$ is replaced by $\wt X_e$ whenever $e$ has both endpoints in $V_n$.
Before proceeding, let us note that by Lemma \ref{tv_lemma},
\eq{
d_{\tv}(\LL_T,\LL_{\wt T}) \leq C\alpha n^{-7/8-\delta}\sqrt{\#(\text{edges in $V_n$})} \leq C\alpha n^{-7/8-\delta}\sqrt{Cn^{7/4+2\delta}} = C\alpha,
}
where $C$ depends only on $\LL_X$ and $m$.
We can then take $\alpha$ sufficiently small that
\eeq{ \label{alpha_choice_fluc}
d_{\tv}(\LL_T,\LL_{T'}) \leq \frac{q_0}{8}.
}
We will also assume
\eeq{ \label{n_assumption_2}
\frac{\alpha c}{2}n^{1/8-\delta} \geq -\log\Big(\frac{q_0}{4}\Big).
}
Let $(e_1,\dots,e_N)$ be the distinguished geodesic from $0$ to $[nx]$, which lies entirely inside $V_n$ for all large $n$ provided $G_2$ occurs.
In this case, as in the proof of Theorem~\ref{fpp_thm},
\eq{
T_n - \wt T_n\geq \sum_{i=1}^N \one_{\{Y_{e_i}=1\}}Z_{e_i} \eqqcolon D,
}
where the $Y_{e_i}$'s are i.i.d.~Bernoulli($\eps$) random variables that are independent of $\sigma(X,X^{(1)},\dots,X^{(m)})$.
So on the event $F_n^\mathrm{c} \cap G_2$, for any $t>0$,
\eq{
\P\givenp{D\leq tn^{1/8-\delta}}{\sigma(X,X^{(1)},\dots,X^{(m)})}\one_{F_n^\mathrm{c} \cap G_2}
&\leq \e^{tn^{1/8-\delta}}\E\givenp{\e^{-D}}{\sigma(X,X^{(1)},\dots,X^{(m)})}\one_{F_n^\mathrm{c} \cap G_2} \\
&= \one_{F_n^\mathrm{c} \cap G_2}\e^{tn^{1/8-\delta}}\prod_{i=1}^N (1-\eps + \eps \e^{-Z_{e_i}}) \\
&\leq \one_{F_n^\mathrm{c} \cap G_2}\e^{tn^{1/8-\delta}}\prod_{i=1}^N \exp\{-\eps(1-\e^{-Z_{e_i}})\} \\
&=\one_{F_n^\mathrm{c} \cap G_2}\e^{tn^{1/8-\delta}}\exp\bigg\{-\eps\sum_{i=1}^N W_{e_i}\bigg\} \\
&\leq \e^{tn^{1/8-\delta}-\alpha c n^{1/8-\delta}}.
}
Choosing $t = \alpha c/2$, we find that
\eq{
\P\Big(T_n - \wt T_n \leq \frac{\alpha c}{2} n^{1/8-\delta}\Big) &\stackrel{\phantom{\mbox{\footnotesize{\eqref{n_assumption_1},\eqref{n_assumption_2}}}}}{\leq} \P(F_n \cup G_2^\mathrm{c})+\e^{-\frac{\alpha c}{2} n^{1/8-\delta}} \\
&\stackrel{\mbox{\footnotesize{\eqref{n_assumption_1},\eqref{n_assumption_2}}}}{\leq} \frac{q_0}{2} + 1 - q_0 + \frac{q_0}{4}
= 1 - \frac{q_0}{4}.
}
Together with \eqref{alpha_choice_fluc} and Lemma \ref{fluctuation_lemma}, this completes the proof.

\subsection{Proof of Theorem \ref{lpp_thm}}  We begin with a lemma that will serve a similar purpose as Lemma \ref{fpp_lemma} did in the proof of Theorem \ref{fpp_thm}.

\begin{lemma} \label{lpp_lemma}
Consider directed site percolation on $\Z^2_+$ in which each site is open independently with probability $p < \vec p_{\cc,\,\mathrm{site}}(\Z^2)$.
Given $\rho > 0$, let $E_n$ be the event that exists a directed path $(v_0,v_1,\dots,v_n)$ with $\|v_0\|_1 \leq n$, such that
\eq{
\#\{1 \leq i \leq n : v_i \mathrm{\ closed}\} < \rho n.
}
Then there is $\rho$ sufficiently small that for some $a,b>0$,
\eq{
\P(E_n) \leq a\e^{-bn} \quad \text{for all $n\geq1$.}
}
\end{lemma}

\begin{proof}
First observe that by a union bound,
\eq{
\P(E_n) \leq \frac{(n+1)(n+2)}{2}\P(E_n^0),
}
where $E_n^0$ is the event that there exists a directed path of length $n$ starting at the origin and passing through fewer than $\rho n$ closed sites.
If we can prove $\P(E_n^0) \leq a\e^{-bn}$ for some $a,b>0$, then it will follow that $\P(E_n) \leq a'\e^{-b'n}$ for some $a',b'>0$.
Therefore, we henceforth concern ourselves only with the event $E_n^0$.

For a directed path $\vec\gamma = (\gamma(0),\gamma(1),\dots,\gamma(\ell))$, let $|\vec\gamma| = \ell$ denotes its length.
Let $A_k$ be the event that there exists an open directed path of length $k$ starting at the origin.
Since $p < \vec p_{\cc,\,\mathrm{site}}(\Z^2)$, \cite[Theorem 7]{griffeath80} (see also \cite[Theorem 14]{durrett-liggett81}) guarantees the existence of $c_1,c_2>0$ such that
\eq{
\P(A_k) \leq c_1\e^{-c_2k} \quad \text{for all $k\geq1$.}
}
Choose $k$ large enough that
\eeq{ \label{choosing_k}
\P(A_k) \leq \frac{1}{36(k+1)^2},
}
and then set $\rho \coloneqq 1/(4k)$. 
Let $F_n$ be the event that some directed path of length $nk$ starting at the origin passes through fewer than $n/2$ closed sites.
Since $\rho(n+1)k = (n+1)/4 \leq n/2$ for any $n\geq 1$, we have the following containments for $n\geq1$ and $0\leq j<k$:
\eq{
E_{nk+j}^0 &= \{\exists\ \vec\gamma,\ \vec\gamma(0)=0,\ |\vec\gamma| = nk+j, \text{ with fewer than $\rho(nk+j)$ closed sites}\} \\
&\subset \{\exists\ \vec\gamma,\ \vec\gamma(0)=0,\ |\vec\gamma| = nk, \text{ with fewer than $\rho(n+1)k$ closed sites}\} \\
&\subset \{\exists\ \vec\gamma,\ \vec\gamma(0)=0,\ |\vec\gamma| = nk, \text{ with fewer than $n/2$ closed sites}\} = F_n.
}
It suffices, then, to obtain a bound of the form $\P(F_n) \leq a\e^{-bn}$.
The remainder of the proof is to achieve such an estimate.

Consider the set
\eq{
\Lambda_n \coloneqq \{\vc w = (w_0=0,w_1,\cdots,w_n)\ |\ \forall\ i=1,\dots,n,\ \exists\ \vec \gamma : w_{i-1}\to w_i \text{ with } |\vec\gamma| = k\}.
}
In words, $\Lambda_n$ is the set of all $(n+1)$-tuples whose $i^\text{th}$ coordinate is $ik$ steps from the origin, and for which there exists a directed path passing through all its coordinates.
Since a directed path of length $\ell$ starting at a fixed position must terminate at one of exactly $\ell+1$ vertices, the cardinality of $\Lambda_n$ is
\eeq{ \label{counting_Lambda}
|\Lambda_n| = (k+1)^n.
}
Recall that $\vec\Gamma_{nk}$ denotes the set of directed paths of length $nk$ starting at the origin.
For each $\vc w \in \Lambda_n$, let $\vec\Gamma_{\vc w}$ denote the subset of those paths traversing the coordinates of $\vc w$:
\eq{
\vec\Gamma_{\vc w} \coloneqq \{\vec\gamma \in \vec\Gamma_{nk} : \vec\gamma(ik) = w_i, 1\leq i\leq n\}.
}
From the definitions, we have $\vec\Gamma_{nk} = \bigcup_{\vc w \in \Lambda_n} \vec\Gamma_{\vc w}$.
Moreover, if we define $F_{\vc w}$ to be the event that some $\vec\gamma\in\vec\Gamma_{\vc w}$ has fewer than $n/2$ closed sites, then
\eeq{ \label{event_partition}
F_n = \bigcup_{\vc w \in\Lambda_n} F_{\vc w}.
}
Fix any $\vc w \in \Gamma_n$.
For $1\leq i\leq n$, let $X_i$ denote the minimum number of closed sites in a directed path of length $k$ starting at $w_{i-1}$.
It is immediate from translation invariance that $\P(X_i \geq 1) = 1-\P(A_k)$.
We thus have the estimate
\eq{
\P(F_{\vc w}) &\leq \P(X_1 + \cdots + X_n \leq n/2) \\
&\leq \P(\one_{\{X_1 \geq 1\}} + \cdots + \one_{\{X_n \geq 1\}} \leq n/2) \\
&= \sum_{i = 0}^{\floor{n/2}} {n \choose i} \big(1-\P(A_k)\big)^{i}\P(A_k)^{n-i} \\
&\leq \frac{n}{2} {n \choose \floor{n/2}} \P(A_k)^{n/2} \leq C\sqrt{\frac{n}{2\pi}}\Big(2\sqrt{\P(A_k)}\Big)^n,
}
where the final inequality holds for some $C>0$ by Stirling's approximation.
It now follows from \eqref{counting_Lambda}, \eqref{event_partition}, and \eqref{choosing_k} that
\eq{
\P(F_n) \leq C\sqrt{\frac{n}{2\pi}}\Big(2(k+1)\sqrt{\P(A_k)}\Big)^n \leq C\sqrt{\frac{n}{2\pi}}3^{-n} \leq a2^{-n}
}
for some $a>0$. 
\end{proof}

\begin{proof}[Proof of Theorem \ref{lpp_thm}] \hspace{1in}\\[-0.5\baselineskip]

\textbf{Part 1.}
For each $v \in \Z^2_+\setminus\{0\}$, define
\eq{
X_v' \coloneqq \max(X_v,X_v^{(1)},\dots,X_v^{(m)}), \qquad Z_v \coloneqq X_v' - X_v, \qquad W_v \coloneqq 1 - \e^{-Z_v},
}
where $m$ is chosen below, and $(X_v^{(j)})_{v\in\Z^2_+}$, $1\leq j\leq m$ are independent copies of the i.i.d.~vertex weights.
Recall that $S = \esssup X_v$.
If $S = \infty$, take $\delta = 1$ and choose $S'$ sufficiently large that
\eq{ 
\P(X_v \geq S' - 2\delta) < \vec p_{\cc,\,\mathrm{site}}(\Z^2).
}
If $S < \infty$, set $S' = S$ and choose $\delta > 0$ sufficiently small that the above display holds.
In either case, we can find $m$ sufficiently large that
\eq{
\P(\max(X_v^{(1)},\dots,X_v^{(m)}) < S'-\delta) < \vec p_{\cc,\,\mathrm{site}}(\Z^2) - \P(X_v \geq S' - 2\delta),
}
so that
\eq{
\P(Z_v < \delta)
&\leq \P(\max(X_v^{(1)},\dots,X_v^{(m)})< S'-\delta) + \P(X_v \geq S-2\delta) < \vec p_{\cc,\,\mathrm{site}}(\Z^2).
}
By Lemma \ref{lpp_lemma}, there is $\rho \in (0,1)$ and $a,b>0$ so that with probability at least $1 - a\e^{-b2^k}$, every directed path $(v_0,v_1,\dots,v_{2^k})$ of length $2^k$ with $\|v_0\|_1 = 2^k$ satisfies
\eq{
\sum_{i=1}^{2^k} W_{v_i} \geq \rho (1-\e^{-\delta})2^k.
}
Let $G$ be the event that this is the case for every $k \geq k_1$, where $k_1$ is chosen large enough that
\eeq{ \label{Gprime_bd_lpp}
\P(G) \geq 3/4.
}
We will assume $n$ is large enough to satisfy \eqref{n_vs_k}.

\textbf{Part 2.}
Similarly to \eqref{edge_eps_def}, we will take
\eq{
\eps_v \coloneqq \frac{\alpha}{\|v\|_1\sqrt{\log n}}, \quad v \in \Z^2_+ \setminus \{0\},
}
and define $\wt X_v$ as in \eqref{new_X} with $\eps = \eps_v$.
Let $T_n = T(0,y_n)$ be the passage time with the $X_v$'s as the vertex weights, and let $\wt T_n = \wt T(0,y_n)$ be the passage time with the $\wt X_v$'s.
The constant $\alpha > 0$ is taken small enough that \eqref{FPP_final_1} holds.

On the event $G$, every directed path $(0=v_0,v_1,\dots,v_n)$ of length $n$ satisfies
\eeq{ \label{Gprime_consequence}
\sum_{i=1}^n \eps_{v_i}W_{v_i} 
\geq \sum_{i=2^{k_1}}^n \eps_{v_i}W_{v_i} 
&\stackrel{\phantom{\mbox{\footnotesize\eqref{n_vs_k}}}}{\geq} \sum_{k=k_1}^{\floor{\log_2 n}-1} \sum_{i=2^{k}+1}^{2^{k+1}} \eps_{v_i}W_{v_i} \\
&\stackrel{\phantom{\mbox{\footnotesize\eqref{n_vs_k}}}}{\geq} \sum_{k=k_1}^{\floor{\log_2 n}-1} \frac{\alpha}{2^{k+1}\sqrt{\log n}} \sum_{i=2^{k}+1}^{2^{k+1}} W_{v_i} \\
&\stackrel{\phantom{\mbox{\footnotesize\eqref{n_vs_k}}}}{\geq} \frac{\alpha}{\sqrt{\log n}}\sum_{k=k_1}^{\floor{\log_2 n}-1}\frac{\rho(1-\e^{-\delta})2^k}{2^{k+1}} \\
&\stackrel{\phantom{\mbox{\footnotesize\eqref{n_vs_k}}}}{=} \frac{\alpha \rho(1-\e^{-\delta})(\floor{\log_2 n}-k_1)}{2\sqrt{\log n}} \\
&\stackrel{\mbox{\footnotesize\eqref{n_vs_k}}}{\geq} \frac{\alpha \rho(1-\e^{-\delta})\sqrt{\log n}}{4\log 2} \eqqcolon \theta\sqrt{\log n}. \raisetag{5\baselineskip}
}
The argument is now completed by proceeding exactly as in the proof of Theorem \ref{fpp_thm} following \eqref{G_consequence}, where \eqref{Gprime_bd} and \eqref{G_consequence} are replaced by \eqref{Gprime_bd_lpp} and \eqref{Gprime_consequence}, respectively.
\end{proof}

\subsection{Proof of Theorem \ref{dp_thm}}
We will absorb the inverse temperature $\beta$ into the $X_v$'s and then work in the case $\beta = 1$.
Let the notation be as in the proof of Theorem \ref{lpp_thm}.
In addition, let $\wt H_n$ and $\wt Z_n$ be the Hamiltonian and partition function, respectively, in the environment formed by the $\wt X_v$'s.
Now \eqref{FPP_final_1} reads as
\eeq{
d_{\tv}(\LL_{\log Z_n},\LL_{\log \wt Z_n}) \leq \frac{1}{4}.\label{dp_final_1}
}
We repeat all steps of the proof of Theorem \ref{lpp_thm} and take $n$ sufficiently large that on the event $G$ defined therein, 
\eeq{ \label{good_pathwise}
\P\givenp[\Big]{\wt H_n(\vec\gamma)-H_n(\vec\gamma) \geq \frac{\theta}{2}\sqrt{\log n}}{\sigma(X,X^{(1)},\dots,X^{(m)})} \geq \frac{3}{4}\one_G
}
for \textit{every} $\vec\gamma\in\vec\Gamma_n$.
(This is in analogy with \eqref{single_path}, but for $n$ satisfying a more restrictive lower bound than \eqref{n_easy}.)
The remainder of the argument must be slightly modified to account for the fact that all paths contribute to the free energy, not just those with maximum weight.

For each $\vec\gamma\in\Gamma_n$, define
\eq{
D_{\vec\gamma} \coloneqq \begin{cases} 
\frac{\theta}{2}\sqrt{\log n} &\text{if } \wt H_n(\vec\gamma)-H_n(\vec\gamma) \geq \frac{\theta}{2}\sqrt{\log n}, \\
0 &\text{otherwise}.
\end{cases}
}
From Jensen's inequality, It is immediate that
\eq{
\log \wt Z_n - \log Z_n = \log \sum_{\vec\gamma\in\vec\Gamma_n} \frac{\e^{H_n(\vec\gamma)}}{Z_n} \e^{\wt H_n(\vec\gamma)-H_n(\vec\gamma)}
&\geq \log \sum_{\vec\gamma\in\vec\Gamma_n} \frac{\e^{H_n(\vec\gamma)}}{Z_n} \e^{D_{\vec\gamma}} 
\geq \sum_{\vec\gamma\in\vec\Gamma_n} \frac{\e^{H_n(\vec\gamma)}}{Z_n} D_{\vec\gamma}.
}
On one hand,
\eeq{ \label{lower_bound_expectation}
\E\bigg[\sum_{\vec\gamma\in\vec\Gamma_n} \frac{\e^{H_n(\vec\gamma)}}{Z_n} D_{\vec\gamma}\bigg] 
&\stackrel{\phantom{\mbox{\footnotesize\eqref{good_pathwise}}}}{=} \E\bigg[\sum_{\vec\gamma\in\vec\Gamma_n} \frac{\e^{H_n(\vec\gamma)}}{Z_n} \E\givenp[\big]{D_{\vec\gamma}}{\sigma(X,X^{(1)},\dots,X^{(m)})}\bigg] \\
&\stackrel{\mbox{\footnotesize\eqref{good_pathwise}}}{\geq} \E\bigg[\sum_{\vec\gamma\in\vec\Gamma_n} \frac{\e^{H_n(\vec\gamma)}}{Z_n} \one_G\frac{3\theta}{8}\sqrt{\log n}\bigg] \\
&\stackrel{\phantom{\mbox{\footnotesize\eqref{good_pathwise}}}}{=} \frac{3\theta}{8}\sqrt{\log n}\, \P(G) \stackrel{\mbox{\footnotesize{\eqref{Gprime_bd_lpp}}}}{\geq} \frac{9\theta}{32}\sqrt{\log n}.
}
On the other hand, we have the deterministic upper bound
\eq{
\sum_{\vec\gamma\in\vec\Gamma_n} \frac{\e^{H_n(\vec\gamma)}}{Z_n} D_{\vec\gamma} \leq \frac{\theta}{2}\sqrt{\log n}.
}
Therefore, the lower bound \eqref{lower_bound_expectation} can only hold if
\eq{
\P\Big(\log \wt Z_n - \log Z_n \geq \frac{\theta}{16}\sqrt{\log n}\Big) 
&\geq \P\bigg(\sum_{\vec\gamma\in\vec\Gamma_n} \frac{\e^{H_n(\vec\gamma)}}{Z_n} D_{\vec\gamma} \geq \frac{\theta}{16}\sqrt{\log n}\bigg) 
\geq \frac{1}{2}.
}
Together with \eqref{dp_final_1} and Lemma \ref{fluctuation_lemma}, this completes the proof.

    \appendix
    \chapter[MATLAB code for directed polymers]{MATLAB code for directed polymers in random environment}
    
The following MATLAB script will perform the following (for the relevant notation, see Chapters \ref{endpoint} and \ref{replica}):

\begin{enumerate}

\item Sample a $(1+2)$-dimensional random environment up to a specified time $N$; the code below will create an environment of i.i.d.~standard normal random variables, while options for Bernoulli and uniform environments are commented out.

\item Generate the $i^\text{th}$-point distribution $\mu_N^\beta(\sigma_i = \cdot)$ for all $i=1,2,\dots,N$, for a specified value of $\beta$ and reference walk $P = $ SRW.

\item Compute the expected overlap $\langle R_{1,2}\rangle_\beta$ of two paths sampled independently from $\mu_N^\beta$.

\item Create a AVI video file---currently titled \texttt{endpoint\_video}---showing the evolution of the endpoint distribution $\mu_i^\beta(\sigma_i = \cdot)$ as $i$ varies over $1,2,\dots,N$ (a Markov chain).
For comparison purposes, the evolution is displayed alongside the (deterministic) distribution for the location of a simple random walk after $i$ steps.

\item Create a AVI video file---currently titled \texttt{ipoint\_video}---showing the evolution of the $i^\text{th}$-point distribution $\mu_N^\beta(\sigma_i = \cdot)$ as $i$ varies over $1,2,\dots,N$.
For comparison purposes, the evolution is again displayed alongside the distribution for the location of a simple random walk after $i$ steps.

\end{enumerate}

Note that the simulation can be rerun using the same environment but a different inverse temperature by setting \texttt{beta} equal to a different value, and then running the script from the third section onward (beginning at ``Determine the effect of the past/present...").

\noindent \hrulefill

\begin{verbatim}
%% Set parameters

N = 100; % length of polymer
beta = 1; % inverse temperature

%% Random environment
% Coordinates are (time,space1,space2)

% Gaussian
X = randn(N,2*N+1,2*N+1);

% Bernoulli
%X = randi([0,1],[N,2*N+1,2*N+1]);

% uniform
%X = rand(N,2*N+1,2*N+1);

% Create environment
E = exp(beta*X);

%% Determine the effect of the past/present on the i-th point distribution

% preallocate
P = zeros(N,2*N+1,2*N+1);

% initialization at time 1
P(1,N,N+1) = E(1,N,N+1);
P(1,N+2,N+1) = E(1,N+2,N+1);
P(1,N+1,N) = E(1,N+1,N);
P(1,N+1,N+2) = E(1,N+1,N+2);
P(1,:,:) = P(1,:,:)/sum(sum(P(1,:,:)));

% effect on i-th point is calculated from effect on (i-1)-th point
% neighbors together with the environment at time i at current location
for i = 2:N
    left = cat(2,P(i-1,2:end,:),zeros(1,1,2*N+1));
    right = cat(2,zeros(1,1,2*N+1),P(i-1,1:end-1,:));
    down = cat(3,P(i-1,:,2:end),zeros(1,2*N+1,1));
    up = cat(3,zeros(1,2*N+1,1),P(i-1,:,1:end-1));
    P(i,:,:) = (left + right + up + down) .* E(i,:,:);
    
    % normalize for numerical stability
    P(i,:,:) = P(i,:,:)/sum(sum(P(i,:,:)));
end

%% Determine the effect of the future on the i-th point distribution

% preallocate
F = zeros(N,2*N+1,2*N+1);

% initialization at time N
F(end,:,:) = ones(1,2*N+1,2*N+1);

% zero out the sites that cannot be accessed at last step
x = 1:2*N+1;
[X,Y] = meshgrid(x,x);
F(end,mod(X+Y,2)~=mod(N,2)) = 0;

for l = 1:2*N+1
    for k = 1:2*N+1
        if  abs(l-N-1)+abs(k-N-1) > N
            F(end,l,k) = 0;
        end
    end
end

% effect on i-th point is calculated from effect on (i+1)-th point
% neighbors together with the environment at time i+1 at those neighbors
for j = 1:N-1
    i = N-j;
    left = cat(2,E(i+1,2:end,:).*F(i+1,2:end,:),zeros(1,1,2*N+1));
    right = cat(2,zeros(1,1,2*N+1),E(i+1,1:end-1,:).*F(i+1,1:end-1,:));
    down = cat(3,E(i+1,:,2:end).*F(i+1,:,2:end),zeros(1,2*N+1,1));
    up = cat(3,zeros(1,2*N+1,1),E(i+1,:,1:end-1).*F(i+1,:,1:end-1));
    F(i,:,:) = (left + right + up + down);
    
    % zero out sites not accessible at time i
    for l = 1:2*N+1
        for k = 1:2*N+1
            if  abs(l-N-1)+abs(k-N-1) > i
                F(i,l,k) = 0;
            end
        end
    end
    
    % normalize for numerical stability
    F(i,:,:) = F(i,:,:)/sum(sum(F(i,:,:)));
end
    
Z = P.*F; % put past/present and future together

%% Simple random walk (for comparison purposes)

% preallocate
S = zeros(N,2*N+1,2*N+1);

% initialization at time 1
S(1,N,N+1) = 1/4;
S(1,N+2,N+1) = 1/4;
S(1,N+1,N) = 1/4;
S(1,N+1,N+2) = 1/4;

% at each step, equal weight distributed to nearest neighbors
for i = 2:N
    left = cat(2,S(i-1,2:end,:),zeros(1,1,2*N+1));
    right = cat(2,zeros(1,1,2*N+1),S(i-1,1:end-1,:));
    down = cat(3,S(i-1,:,2:end),zeros(1,2*N+1,1));
    up = cat(3,zeros(1,2*N+1,1),S(i-1,:,1:end-1));
    S(i,:,:) = (left + right + up + down);
    
    % normalize for numerical stability
    S(i,:,:) = S(i,:,:)/sum(sum(S(i,:,:)));
end

%% Include time 0 as first coordinate (point mass at origin)

point_mass = zeros(1,2*N+1,2*N+1);
point_mass(1,N+1,N+1) = 1;

P = cat(1,point_mass,P); 
% P(i+1,:,:) is the endpoint distribution at time i

Z = cat(1,point_mass,Z); 
% Z(i+1,:,:)/sum(sum(Z(i+1,:,:)) is the i-th point distribution at time N

S = cat(1,point_mass,S); 
% S(i+1,:,:) is the distribution of SRW at time i

%% Endpoint video

% create file for movie to be recorded
v = VideoWriter('endpoint_video');
v.FrameRate = 11;
open(v)

% prepare figure
close all
figure(1)
colormap(bone)
screen_size = get(0,'ScreenSize');
set(gcf,'Position',[0 0 screen_size(3) screen_size(4)]);

% axis coordinates
x = -N:N;
[x,y] = meshgrid(x,x);

for i = 1:N+1
    % left plot
    subplot(1,2,1)
    mesh(x,y,squeeze(P(i,:,:)))
    zlabel(['$i = $ ' num2str(i-1) '\phantom{xxxxxxx}'],...
        'interpreter','latex','FontSize',24)
    set(get(gca,'ZLabel'),'Rotation',0)
    axis([-N/4 N/4 -N/4 N/4 0 0.2])
    title(['Endpoint distribution, $\beta = $ ' num2str(beta)],...
        'interpreter','latex','FontSize',24)
    
    % right plot
    subplot(1,2,2)
    mesh(x,y,squeeze(S(i,:,:)))
    axis([-N/4 N/4 -N/4 N/4 0 0.2])
    title('SRW','interpreter','latex','FontSize',24)
    
    % include delay in order to watch video being created
    %pause(0.05)
    
    % capture movie frame
    writeVideo(v,getframe(gcf))

end

close(v)

%% ith-point video

% create file for movie to be recorded
v = VideoWriter('ipoint_video');
v.FrameRate = 11;
open(v)

% prepare figure
figure(2)
colormap(bone)
screen_size = get(0,'ScreenSize');
set(gcf,'Position',[0 0 screen_size(3) screen_size(4)]);

% axis coordinates
x = -N:N;
[x,y] = meshgrid(x,x);

for i = 1:N+1
    % left plot
    Z(i,:,:) = Z(i,:,:)/sum(sum(Z(i,:,:)));
    subplot(1,2,1)
    mesh(x,y,squeeze(Z(i,:,:)))
    zlabel(['$i = $ ' num2str(i-1) '\phantom{xxxxxxx}'],...
        'interpreter','latex','FontSize',24)
    set(get(gca,'ZLabel'),'Rotation',0)
    axis([-N/4 N/4 -N/4 N/4 0 0.2])
    title(['$i$-th point distribution, $\beta = $ ' num2str(beta) ...
        ', $N =$ ' num2str(N)],'interpreter','latex','FontSize',24)
    
    % right plot
    subplot(1,2,2)
    mesh(x,y,squeeze(S(i,:,:)))
    axis([-N/4 N/4 -N/4 N/4 0 0.2])
    title('SRW','interpreter','latex','FontSize',24)
    
    % include delay in order to watch video being created
    %pause(0.05)
    
    % capture movie frame
    writeVideo(v,getframe(gcf))

end

close(v)
\end{verbatim}

    \bibliographystyle{acm}
    \bibliography{thesis}
    \end{document}